\documentclass[11pt,a4paper]{amsart}
\usepackage[margin=2cm]{geometry}

\usepackage{amsmath, amsthm, amssymb}
\input xy
\xyoption{all}
\usepackage{mathrsfs}

\usepackage{enumerate}
\usepackage{caption}
\usepackage{comment}
\usepackage{subcaption}
\usepackage{graphicx}
\usepackage{color}

\usepackage{tikz}
\usepackage{tikz-cd}
\usepackage{tkz-euclide}

\usepackage{scalerel}

\usepackage[driverfallback=hypertex]{hyperref}
\usepackage{nameref,zref-xr}

\newtheorem{thm}{Theorem}[section]
\newtheorem{prop}[thm]{Proposition}
\newtheorem{pr}[thm]{Proposition}
\newtheorem{lem}[thm]{Lemma}

\newtheorem{lemma}[thm]{Lemma}
\newtheorem{cor}[thm]{Corollary}

\theoremstyle{definition}
\newtheorem{definition}[thm]{Definition}
\newtheorem{nn}[thm]{Notation}

\theoremstyle{remark}
\newtheorem{rmk}[thm]{Remark}

\newtheorem{ex}[thm]{Example}
\newtheorem{obs}[thm]{Observation}

\newcommand{\CI}{{\nu}} 
\newcommand{\CM}{{\mathcal{M}}}

\newcommand{\oCM}{{\overline{\mathcal{M}}}}

\newcommand{\Gammar}{{\Gamma^W_{0,k_1,k_2,1,\{(a_i,b_i)\}_{i\in[l]}}}}

\newcommand{\CC}{{\mathcal{C}}}

\newcommand{\CL}{{\mathbb{L}}}
\newcommand{\Pos}{{H^+}}
\newcommand{\TAU}{{\mathrm{proj}}}
\newcommand{\TTT}{{\mathcal{T}}}
\newcommand{\detach}{\textup{detach}}
\newcommand{\smooth}{{\text{sm}}}
\newcommand{\RS}{{(r,s)}}
\newcommand{\Root}{{\times}}
\newcommand{\alt}{{\text{alt} }}
\newcommand{\xch}{\textup{exchange}}
\newcommand{\XCH}{\textup{Exchange}}
\newcommand{\SIMX}{{\sim_{\XCH} }}
\newcommand{\simx}{{\sim_{\xch} }}

\newcommand{\CRIT}{{\text{Crit}}}
\newcommand{\Critrit}{{\widetilde{\text{Crit}}}}
\newcommand{\INT}{{\text{Balanced}}}
\newcommand{\Inv}{{\mathbb{Q}\textendash\mathrm{Bal}}}
\newcommand{\InvSym}{{\mathbb{Q}\textendash\mathrm{Bal}^{\mathrm{sym}}}}
\newcommand{\OFJRW}{\Omega_{r,s}}
\newcommand{\OFJRWSym}{{\Omega_{r,s}^{\mathrm{sym}}}}
\newcommand{\ChamberIndices}{\Upsilon_{r,s}}
\newcommand{\ChamberIndicesSym}{\Upsilon_{r,s}^{\mathrm{sym}}}
\newcommand{\MAPS}{{\text{Maps}}}

\newcommand{\ess}{\mathbf{s}}
\newcommand{\vecd}{\mathbf{d}}

\newcommand{\op}[1]{\operatorname{#1}}
\newcommand{\Aut}{\mathup{Aut}}

\newcommand{\II}{\mathbb{I}}

\newcommand{\CB}{{\mathcal{B}}}

\newcommand{\CG}{{\mathcal{G}}}

\newcommand{\s}{\mathbf{s}}
\newcommand{\srest}{\mathbf{s}_1}

\renewcommand{\Aut}{\text{Aut}}

\newcommand{\ttt}{t}
\newcommand{\XXX}{w}
\newcommand{\ttts}{\mathbf{t}}
\newcommand{\cS}{\mathcal{S}}
\newcommand{\cJ}{\mathcal{J}}
\newcommand{\cW}{\mathcal{W}}

\newcommand{\eee}{{{\mathbf{e}}}}

\newcommand{\tw}{\text{tw}}
\newcommand{\NN}{\Z_{\ge 0}}
\newcommand{\NNN}{{\nu}}

\newcommand{\maltcross}{\scalerel*{
    \tikz\fill
    (0.02,0.02)    -- (0.2,0.5)   -- (0,0.4)  -- (-0.2,0.5)  --
    (-0.02,0.02)   -- (-0.5,0.2)  -- (-0.4,0) -- (-0.5,-0.2) --
    (-0.02,-0.02)  -- (-0.2,-0.5) -- (0,-0.4) -- (0.2,-0.5)  --
    (0.02,-0.02)   -- (0.5,-0.2)  -- (0.4,0)  -- (0.5,0.2)   --
    cycle;}
        {0}
}

\newcommand{\N}{\mathbb{N}}
\newcommand{\Universe}{\Omega}

\newcommand{\RRR}{{\mathcal{R}}}
\newcommand{\TRAM}{{\mathcal{R}}}
\newcommand{\Ass}{{{\maltcross}}}
\newcommand{\Detach}{{{\textup{Detach}}}}
\newcommand{\Conn}{{{\textup{Conn}}}}
\newcommand{\oPM}{{\overline{\mathcal{PM}}}}
\newcommand{\PM}{{{\mathcal{PM}}}}

\newcommand{\tildooo}{{\tilde{\mathfrak{o}}}}
\newcommand{\hatooo}{{\hat{\mathfrak{o}}}}
\newcommand{\ooo}{{{\mathfrak{o}}}}
\newcommand{\std}{{\text{std}}}

\newcommand{\kk}{{K}}
\newcommand{\rank}{\operatorname{rank}}
\newcommand{\For}{{\textup{For}}}
\newcommand{\forg}{{\textup{for}}}
\newcommand{\spinqi}{{\text{spin}\neq i}}

\newcommand{\Cont}{\textup{Cont}}

\newcommand{\Z}{\ensuremath{\mathbb{Z}}}
\newcommand{\Q}{\ensuremath{\mathbb{Q}}}
\newcommand{\C}{\ensuremath{\mathbb{C}}}
\newcommand{\R}{\ensuremath{\mathbb{R}}}
\renewcommand{\P}{\ensuremath{\mathbb{P}}}
\newcommand{\M}{\ensuremath{\overline{\mathcal{M}}}}
\renewcommand{\O}{\ensuremath{\mathcal{O}}}

\newcommand{\ev}{\ensuremath{\textrm{ev}}}

\renewcommand{\d}{\ensuremath{\partial}}

\newcommand{\eps}{\epsilon}
\newcommand{\<}{\left<}
\renewcommand{\>}{\right>}

\DeclareMathOperator{\re}{Re}

\usepackage{geometry}
\numberwithin{equation}{section}

\begin{document}

\title{Open FJRW Theory and Mirror Symmetry}

\author{Mark Gross}
\address{M. Gross: \newline Department of Pure Mathematics and Mathematical Statistics, University of Cambridge, CB4 0WB, United Kingdom}
\email{mgross@dpmms.cam.ac.uk}

\author{Tyler L. Kelly}
\address{T. L. Kelly: \newline School of Mathematical Sciences, Queen Mary University of London, 327 Mile End Rd, London E1 4NS, UK}
\email{t.l.kelly@qmul.ac.uk}

\author{Ran J. Tessler}
\address{R. J. Tessler:\newline Department of Mathematics, Weizmann Institute of Science, Rehovot 76100, Israel}
\email{ran.tessler@weizmann.ac.il}

\begin{abstract}
We construct an open enumerative theory for the Landau-Ginzburg (LG) model $(\C^2, \mu_r\times \mu_s, x^r+y^s)$. The invariants are defined as integrals of multisections of a Witten bundle with descendents over a moduli space that is a real orbifold with corners. In turn, a generating function for these open invariants yields the mirror LG model and a versal deformation of it with flat coordinates. After establishing an open topological recursion result, we prove an LG/LG open mirror symmetry theorem in dimension two with all descendents. The open invariants we define are not unique but depend on boundary conditions that, when altered, exhibit wall-crossing phenomena for the invariants. We describe an LG wall-crossing group classifying the wall-crossing transformations that can occur. 
\end{abstract}

\maketitle

\setcounter{tocdepth}{2}
\tableofcontents

\section*{Introduction}

In a sequence of papers, Abramovich, Chiodo, Jarvis, Kimura and Vaintrob developed an enumerative theory for the moduli space of $r$-spin curves, that is, stable curves with an $r$th root of the canonical bundle \cite{Witten93, AJ03, Ja1, Ja2, JKV, Chi08}. The intersection numbers of this moduli space can be encapsulated in a generating function related to the Gelfand-Dickey hierarchies. This is predicted by the $r$-spin version of Witten's conjecture, generalizing the original proof of Witten's conjecture by Kontsevich \cite{KontCMP}. Nowadays, their theory can be seen as a special case of FJRW theory, an enumerative theory for a Landau-Ginzburg $A$-model introduced by Fan, Jarvis and Ruan in another sequence of
papers \cite{FJR,FJR2,FJR3}. FJRW theory introduces an $A$-model
enumerative invariant associated to the Landau-Ginzburg model $(W,G)$ where
$W:\C^a\rightarrow \C$ is a weighted homogeneous polynomial
and $G$  is a subgroup of the automorphism group of $W$. FJRW theory is 
equivalent to $r$-spin theory when considering the case $(x^r, \mu_r)$ \cite{FJRrspin}.

Here, there should be a form
of mirror symmetry between Landau-Ginzburg models
\[
(W,G) \leftrightarrow (\check W,\check G).
\]
One such example, which shall be explored in this paper, is the
mirror pair
\begin{equation}
\label{eq:mirror pair}
(x_1^{r_1}+\cdots+x_a^{r_a},\mu_{r_1}\times\cdots\times\mu_{r_a})
\leftrightarrow (x_1^{r_1}+\cdots+x_a^{r_a}, 1).
\end{equation}
This mirror pair at its onset was predicted as a case of Berglund-H\"ubsch duality \cite{BH}, a combinatorial mirror construction that provides mirror pairs of Landau-Ginzburg models. In this paper, we consider the left-hand side as the $A$-model and the right hand side as the $B$-model.  Many aspects of mirror symmetry for this mirror pair  are understood. For example it is known that their corresponding Frobenius algebras are isomorphic \cite{Kra} and the (closed) FJRW theory of $(x_1^{r_1}+\cdots+x_a^{r_a},\mu_{r_1}\times\cdots\times\mu_{r_a})$ corresponds to the (closed) Saito-Givental theory of $(x_1^{r_1}+\cdots+x_a^{r_a}, 1)$ \cite{HeLiShenWebb}.

The closed $A$-model on the left of \eqref{eq:mirror pair} involves moduli spaces of pointed
stable \emph{orbicurves} $$(C,z_1,\ldots,z_n)$$
where $z_1,\ldots,z_n$ are marked points with non-trivial stabilizer,
and the only other points on $C$ with non-trivial stabilizer are
nodes, satisfying the so-called balancing condition.
Further, one also includes data of \emph{spin bundles} $S_1,\ldots,S_a$
on $C$, with the data of an isomorphism $S_i^{\otimes r_i}\cong
\omega_{C,\log}$, the logarithmic dualizing sheaf. Such a curve
is called a \emph{$W$-spin curve}. In this particular (closed) case,
the genus zero FJRW invariants we are interested in are defined
relatively easily from the top Chern class of the \emph{Witten bundle}
on the moduli space, whose fibre at a point in moduli space is
$\bigoplus_i H^1(C,S_i)^{\vee}$. In more general situations, such
as higher genus or more complex potentials, the definition of
FJRW invariants is considerably more difficult, as one might
have $H^0(C,S_i)\not=0$.
Those cases may require either a more in-depth study of analytic properties of the Witten equation as carried out in \cite{FJR3} or the more recent
approaches \cite{FJRGLSM, CFFGKS, FaveroKim} for general algebraic enumerative theories for gauged linear sigma models.
Thus we avoid such situations in this paper.

The closed $B$-model side of the story, at least for the right-hand side of
the correspondence \eqref{eq:mirror pair}, was clarified in work
of Li, Li, Saito and Shen \cite{LLSS}, with mirror symmetry for closed invariants being established in \cite{FJR, HeLiShenWebb, MS}. The $B$-model side is
a Saito-Givental
theory, which, put as simply as possible, involves calculating
oscillatory integrals of the form
\begin{equation}
\label{eq:osc integral}
\int_{\Gamma} e^{W_{\bf t}/\hbar} f(x_1,\ldots,x_a,{\bf t}) dx_1\wedge
\cdots \wedge dx_a.
\end{equation}
Here $x_1,\ldots,x_a$ are coordinates on $\C^a$, $\hbar$ is a coordinate
on an auxiliary $\C^*$, $f$
is a carefully chosen complex-valued regular function, and ${\bf t}=\{t_{i_1\cdots i_a}\,|\,
1\le i_j \le r_j-2\}$ is
a set of coordinates on the parameter space for a universal unfolding
$W_{\bf t}$ of $W$.
Finally $\Gamma$ runs over some suitable non-compact cycles in $\C^a$.
The requirement on $f$ is that the form $f dx_1\wedge\cdots \wedge dx_a$
is a so-called \emph{primitive form} in the sense of Saito-Givental theory.

While Saito-Givental theory in general gives a Frobenius manifold
structure on the parametrizing space of a universal unfolding of $W$,
determining this structure
can be quite difficult. However, experience with mirror
symmetry for toric Fano varieties \cite{ChoOh,GrossP2,FOOO1} suggests that
mirror symmetry becomes much more transparent when a specific perturbation
$W_{\bf t}$ of the original potential is used. In particular,
in the Fano case, the mirror Landau-Ginzburg model $W_{\bf t}$ should be taken to be a generating function
for certain \emph{open} Gromov-Witten invariants. Slightly more precisely,
one considers a Lagrangian torus fibre $F$ of an SYZ fibration on a given
Fano manifold and counts holomorphic disks with interior marked points realizing
some set of constraints and with boundary lying on $F$. The correct choice
of $W_{\bf t}$ is then a generating function for these counts. Importantly,
these open invariants are not well-defined. However, the oscillatory
integrals analogous to \eqref{eq:osc integral} are well-defined,
and furthermore, if properly set up, the function $f$ becomes identically
$1$ and the analogous coordinates ${\bf t}$ become the so-called \emph{flat
coordinates} of Saito-Givental theory. As a consequence, the entirely
well-defined closed Gromov-Witten invariants can be immediately read
off from the oscillatory integrals. Unfortunately, an enumerative theory for a Landau-Ginzburg model does not have a target space, hence there is no associated SYZ fibration; therefore, a new approach is needed to create an analogous story.

We take this as motivation for providing 
an open enumerative theory for the Landau-Ginzburg model $(x_1^{r_1}+x_2^{r_2},\mu_{r_1}\times\mu_{r_2})$. We now summarize the key points of the paper:
\begin{itemize}
\item We define open FJRW invariants 
as integrals of multisections of (descendent) Witten bundles over moduli
spaces which are real orbifolds
with corners. 
\item These open invariants depend on choices of boundary conditions for
sections of the (descendent) Witten bundles on the boundaries of moduli spaces.
However, such choices may be made (Theorem \ref{open FJRW invariants exist}).
\item Some boundary conditions can be fixed via a geometrically natural 
positivity condition. However, there are degrees of freedom in the choice
of boundary conditions which cannot be constrained. This leads to the first example of wall-crossing for an open enumerative theory for a Landau-Ginzburg model.
\item If the boundary conditions not determined by positivity
are chosen with an inductive
structure, certain combinations of open FJRW invariants satisfy a form
of topological recursion: see Theorem~\ref{thm:open TRR intro}. This type of topological recursion is new in a sense as we can only establish recursions between certain polynomials of invariants rather than a single invariant due to the wall-crossing phenomenon.
\item Given a system of open FJRW invariants determined by a family
of boundary
conditions satisfying this inductive structure, we construct
a perturbation of $W$ as a generating function for these open invariants.
The oscillatory integrals \eqref{eq:osc integral} then become
generating functions for closed extended FJRW invariants, see Theorem
\ref{thm:intro mirror theorem}. Besides proving mirror symmetry at
genus zero, we can construct the mirror entirely in terms of the
original open $A$-model, as opposed to relying on previous combinatorial
constructions such as Berglund-H\"ubsch.
\item We then show there is a wall-crossing group
which acts faithfully and transitively  
on (i) the collection of all possible systems of open FJRW invariants; and
 (ii) the set of all perturbations
of the potential $W$ satisfying certain homogeneity constraints for
which Theorem \ref{thm:intro mirror theorem} holds. Thus our method of
choosing boundary conditions is geometrically natural, allowing us to
match the $A$- and $B$-models for FJRW theory on the nose. See 
Theorem~\ref{introthm: torsor}.
\end{itemize} 

\noindent We now give more details of different aspects of this.

\subsection{The open $A$-model invariants}

Our construction of the open $A$-model invariants generalizes
the construction of open $r$-spin invariants of
Buryak, Clader and Tessler \cite{BCT:I,BCT:II}.
This in turn builds to a certain extent on the open descendent theory
of Pandharipande, Solomon and Tessler \cite{PST14}.

Given a potential $W=\sum_{i=1}^a x_i^{r_i}$, we review in
\S\ref{sec:graded surfaces}
the notion of a stable (twisted) $W$-spin curve and give the
definition of a stable (twisted) $W$-spin disk.
In slightly more detail, a $W$-spin curve involves data $(C,\{z_i\},S_1,\ldots, S_a,
\tau_1,\ldots,\tau_a)$.
Here $C$ is a stable marked orbicurve, with all marked points and
nodes having the same stabilizer group $\mu_d$, where $d$ is the
least common multiple of the $r_i$. Further, $\tau_i$ is an isomorphism
\[
\tau_i:S_i^{\otimes r_i} \rightarrow \omega_{C,\log}(-\sum_{i\in I_0} r_iz_i)
\]
where $I_0$ is a subset of indices of the marked points. It is
this twist by $z_i$ which leads to the phrase ``twisted'', and only Ramond
points are allowed to be twisted. Each marked point comes with
\emph{twists}
$\tw_i$ for $1\le i \le a$, where $\tw_i \in \{-1,\ldots,r_i-1\}$
records the action of the stabilizer group of the marked
point on the fibre of the twisted spin 
bundle $S_i$. See \S\ref{subsec:closed graded}
for details
on our conventions for these twists. The case of $\tw_i=-1$ or $r_i-1$ corresponds to the Ramond
case (where the action of the stabilizer on the fibre of $S_i$ is
trivial), with the two possibilities corresponding to the cases that
$i\not\in I_0$ or $i\in I_0$ respectively.

To have a compact moduli space of $W$-spin curves, we must consider $W$-spin curves whose underlying orbicurves $C$ have nodal singularities. If $C$ has a nodal point $n$, then take a partial normalization $\nu:\widetilde C
\rightarrow C$ at the node $n$. This yields two \emph{half-nodes} on $\widetilde C$, the inverse image of $n$, which we call $z_1, z_2$. These half-nodes will have twists, 
with the standard balancing condition
at nodes dictated by Equations \eqref{eq:balanced} and \eqref{eq:balanced action}
implying that these twists satisfy
\[
\tw_i(z_1)+\tw_i(z_2)\equiv r_i-2 \mod r_i.
\]

Closed FJRW invariants, in this setup, may be defined from the Euler class
of the Witten bundle on the moduli space of twisted $W$-spin curves when
the underlying curves have genus $0$ and, for each $1\le i \le n$, there
is at most one marked point with $\tw_i=-1$; otherwise the spin bundle
$S_i$ may have sections and defining FJRW invariants becomes more complicated.
In the $r$-spin case, such invariants were introduced in \cite{JKV}, and are called \emph{closed extended
FJRW invariants}. See \S\ref{subsec:closed extended} for details.
They were often not studied in the subsequent literature as the $\tw_i=-1$ case is not directly involved in statements of closed mirror symmetry; however, the extended theory is required to establish a natural relationship between open and closed extended invariants through open topological recursion as seen here and in \cite{BCT_Closed_Extended, BCT:I, BCT:II}. 

\medskip

A \emph{$W$-spin disk} is a genus zero $W$-spin curve
\[
(C,\{z_i\}_{i\in I}, \{\bar z_i\}_{i\in I}, \{x_i\}_{i\in B},
S_1,\ldots,S_a,\tau_1,\ldots,\tau_a)
\]
equipped with the additional data of an anti-holomorphic involution
$\phi:C\rightarrow C$ along with liftings of $\phi$ to involutions of $S_i$:
see \S\ref{subsec:open objects} for details and more precise notation.
We usually write $\Sigma$ for a choice
of fundamental domain of the anti-holomorphic involution $\phi$,
so that $C= \Sigma \cup \phi(\Sigma)$ and $\partial\Sigma =
\Sigma\cap \phi(\Sigma)$. The marked points $z_i$ are \emph{interior
marked points} indexed by a set $I$, i.e.,
$z_i\in \Sigma\setminus \partial\Sigma$ for $i\in I$, with
complex conjugate marked points $\bar z_i=\phi(z_i)$. The marked points
$x_i$
are \emph{boundary marked points} indexed by a set $B$, i.e., $x_i\in \partial\Sigma$ for $i\in B$.
While the twists of internal marked points are unrestricted,
we place the restriction on boundary marked points that
$\tw_i \in \{0, r_i-2\}$ (see Remark~\ref{remark on twist restriction} for discussion).
We further exclude the possibility that $\tw_i=0$ for all $i$: these
are in fact points which may be forgotten; forgetting such points
will play an important role.
As in \cite{BCT:I,BCT:II},
we also require some additional choice of data called a \emph{grading}
involving the restriction of the spin bundles to $\partial\Sigma$.
While this is technical and we omit further mention of
this data in the introduction, gradings are crucial for imposing the necessary
boundary conditions to define open invariants.
$W$-spin
disks with this additional data and constraints on boundary twists
are called \emph{graded $W$-spin disks}.

In \S\ref{sec:graded surfaces}, we first give the definition of
(closed) $W$-spin curves, followed by $W$-spin disks. This definition is a slight generalization of the definitions of the corresponding
$r$-spin curves and disks of \cite{BCT:I,BCT:II}. We then introduce a language
of dual intersection graphs to describe various ways that graded
$W$-spin curves or disks may degenerate. Unfortunately, there are a large number
of ways $W$-spin disks may degenerate, and this is where
combinatorial complexities appear. However, these generate 
interesting features.
We define the notion of a \emph{graded $W$-spin graph},
which records the combinatorial data of a stable graded $W$-spin curve or
disk. This data consists essentially of its dual intersection graph,
the twists of marked points and nodes, and the ordering (or collection of orderings) of boundary marked points.

In \S\ref{sec: moduli} we commence a study of moduli spaces
of stable graded $W$-spin disks. In Proposition \ref{prop:orbi_w_corners}
we construct a compact moduli space $\oCM^W_{\Gamma}$ where $\Gamma$ denotes
a graded $W$-spin graph as above corresponding to a smooth disk with
specified numbers of boundary and internal marked points along with
their twists.  This compact moduli space includes all
degenerations of smooth disks described by the graph $\Gamma$.

Having constructed these moduli spaces, we may define the Witten
and descendent bundles on these moduli spaces. This is
carried out as in \cite{BCT:I,BCT:II}. For $1\le i \le a$,
there is a bundle ${\mathcal W}_i$ on $\oCM^W_{\Gamma}$ whose fibre at a point
corresponding to a $W$-spin curve $C$ with
anti-holomorphic involution $\phi$ is the eigenspace of the
action of $\phi$ on $H^1(C,S_i)^{\vee}$ with eigenvalue $-1$.
Further, for each internal marked point there is a descendent
bundle $\CL_i$, a rank one complex vector bundle, whose fibre is
the cotangent space at the $i$th marked point. If $\Gamma$ is a
$W$-spin graph corresponding to a smooth disk
with internal marked points indexed by a set $J$, then we may consider
a \emph{descendent vector} $\vecd=(d_j)_{j\in J}\in \NN^{J}$.
We then obtain a real vector bundle
\[
E_\Gamma(\vecd):= \bigoplus_{i=1}^a \mathcal{W}_i \oplus
\bigoplus_{j\in J} \CL_j^{\oplus d_j},
\]
 the \emph{descendent Witten bundle}.
We say $\Gamma$ is \emph{balanced} for $\vecd$ if
\[
\dim \oCM_{\Gamma}^W = \rank E_{\Gamma}(\vecd).
\]
It is precisely this case where we will extract invariants.

The very first step in defining invariants is to orient the moduli
spaces and the descendent Witten bundles. This is carried out in
\S\ref{subsec:or} using results of \cite{PST14, BCT:I}. It is here where for the first time we restrict to the rank $2$ case. There is little to say here about this process,
other than that it is a necessary but unpleasant step.  Because we are restricting to rank two, we now change notation and write
\[
W=x^r+y^s.
\]
Since now a marked point or half-node has two twists $\tw_1,\tw_2$,
we typically write the twist as a pair $(\tw_1,\tw_2)$.

The second step in defining invariants is to have a suitable
set of boundary conditions.
Because $\oCM_{\Gamma}^W$ is an orbifold with corners,
one does not define the invariants simply by integrating Euler classes of the vector
bundle $E_{\Gamma}(\vecd)$. Rather, one has to impose boundary conditions.

The general setup is as follows; see \S\ref{subsubsec:multisections}
for a brief but more detailed review of these notions.
If given a compact orbifold with corners $M$ of dimension $n$
and a rank $n$ orbibundle $E$ on $M$,
along with a nowhere vanishing multisection $\ess$ of $E|_{\partial M}$, the
signed number of zeroes of any extension of $\ess$ to a transverse multisection
$\tilde \ess$ of $E$ over all of $M$ is well-defined, see \cite{PST14},
Theorem~A.14. This is the \emph{relative Euler class} of $E$ with respect
to $M$ and we write this as
\[
\int_{M} e(E; \ess)=\# Z(\tilde\ess).
\]
In \cite{BCT:II}, Appendix A, a variant of this is given, where $M$
need not be compact. Rather, one assumes given an open subset $U\subseteq M$
such that $M\setminus U$ is compact, and a nowhere vanishing multisection $\ess$ of
$E|_{U\cup \partial M}$. Then, the signed number of zeroes of any
transverse extension of $\ess$ is well-defined, again giving the relative
Euler class. We still write this using the notation above.

Thus, to define open invariants, we need to make a choice of boundary
conditions for the descendent Witten bundle for every balanced
$\Gamma$. To be of any use and to have any hope of proving relationships
between the different invariants, such as open topological recursion,
there has to be some inductive structure to this choice of boundary
conditions. Here, the approach to such an inductive structure follows
the approach of \cite{BCT:I,BCT:II, PST14}.

We now give an impressionistic description of how this is done. The
reader should be aware that the definitions given here in the introduction aren't
precisely correct, and should refer to the main text for the complete
definitions.\footnote{For example, we elide the notion of alternating for boundary marked points and nodes, see Definition \ref{def:alt}.}
Consider a $W$-spin graph $\Gamma$ corresponding to a smooth $W$-spin
disk, thus giving a moduli space $\oCM_{\Gamma}^W$ of such disks
and their degenerations. The boundary
of this moduli space is a union of codimension one strata, which arise
in several different ways. See Figure~\ref{fig:node_types} for a pictorial description of the various degenerations possible. We note that strata depicted in Figure~\ref{fig:node_types}(A) have codimension 2, so such strata are not on the boundary $\partial\M_{\Gamma}^W$. Thus we now focus on the cases in Figure~\ref{fig:node_types}(B) and~\ref{fig:node_types}(C). We broadly distinguish here between three
types of boundary strata:
\begin{enumerate}
\item  There are strata where the corresponding
stable disk has a \emph{contracted boundary node}, i.e., $C$ has a node which
is a fixed point of the anti-holomorphic involution.
\item There are strata with a boundary node where the corresponding stable disk
is a union of two disks, with the resulting half-nodes having twists
$(a,b)$, $(r-a-2,s-b-2)$, with either $a \in \{1,\ldots,r-3\}$
or $b\in \{1,\ldots,s-3\}$. Such a stratum is called \emph{positive}.
\item We have a similar degeneration to two disks with a boundary node, but this time
$a\in \{0,r-2\}$ and $b\in \{0,s-2\}$. We call such a stratum
\emph{relevant}.
\end{enumerate}

For case (1), we impose a boundary condition on multisections of the Witten bundle along
boundary strata arising from contracted boundary nodes. This is done
via a positivity condition and
guarantees non-vanishing of these multisections along such boundary
strata.

For case (2), we deal with the positive strata by removing them and imposing conditions
near such strata.
We define a new moduli space $\oPM_{\Gamma}$ obtained from
$\oCM_{\Gamma}^W$ by deleting the positive boundary strata. If $V\subseteq
\oCM_{\Gamma}^W$ is a suitably small open neighbourhood of the union
of positive boundary strata, we take $U=V\cap \oPM_{\Gamma}$. We then
only consider multisections of $E_{\Gamma}(\vecd)$ which satisfy a certain
positivity condition on $U$. These sections are then non-vanishing
on $U$, by definition.

A key point in the construction of the boundary conditions is to define the positivity boundary conditions in a consistent way. A more naive strategy is consistent only in low codimension strata and requires modifications (see Remark~\ref{rmk: why strongly positive}). We continue to use the modification used by Buryak, Clader and Tessler in \cite{BCT:I, BCT:II}, where $\oPM_\Gamma$ is introduced to construct an open $r$-spin intersection theory. The crucial point is that the zero locus of the section of Witten's bundle avoids nodal strata with certain twists at so-called non-alternating marked points in $\oPM_\Gamma$ (see Definition~\ref{def:alt}).

We call those multisections of $E_{\Gamma}(\vecd)$ on $\oPM_{\Gamma}$
which satisfy these positivity
conditions at contracted boundary node strata and near positive
strata \emph{strongly positive multisections}, see
Definition \ref{def: strongly positive}. We note that the notion of
grading, brushed over earlier in this introduction, is necessary to
define positivity. Given such a section $\ess^{\Gamma}$, we
are then able to define the \emph{open $W$-spin invariant}
\[
\left\langle \prod_{i\in I(\Gamma)}\tau^{(a_i,b_i)}_{d_i}\sigma_1^{k_1(\Gamma)}\sigma_2^{k_2(\Gamma)}\sigma_{12}^{k_{12}(\Gamma)} \right\rangle^{\mathbf{s}^{\Gamma},o}:=
\int_{\oPM_\Gamma}e(E ; \mathbf{s}^{\Gamma}|_{U\cup\partial\oPM_\Gamma})=\#Z(\mathbf{s}^{\Gamma})\in\mathbb{Q}.
\]
Here $I(\Gamma)$ denotes the set of tails of $\Gamma$
corresponding to internal marked points of the corresponding disk,
with
$(a_i,b_i)$ denoting the twist at the internal point labelled by $i$
and $d_i$ the component of the descendent vector. On the other hand,
$k_1(\Gamma), k_2(\Gamma)$ and $k_{12}(\Gamma)$ denote the number of
boundary points with twists $(r-2,0)$, $(0,s-2)$ and $(r-2,s-2)$
respectively. We note that we do not label the boundary points, but the invariants may depend on the labels of the internal points and not just their twists.

This construction has now produced a single invariant associated
to a $W$-spin graph $\Gamma$ corresponding to a smooth $W$-spin disk
and a descendent vector $\vecd=(d_i)_{i\in I(\Gamma)}
\in \Z_{\ge 0}^{I(\Gamma)}$. However, the range of boundary conditions allowed by strongly positivity is still too flexible to provide meaningful systems of invariants.

The key punchline here is that, in case (3) above, there is an inductive
structure for the relevant strata. Here, the normalization
of a stable disk in the interior of a relevant stratum is a disjoint
union of two disks, each of which only have boundary points with
twists $(0,0)$, $(r-2,0)$, $(0,s-2)$ or $(r-2,s-2)$.
However, after forgetting those boundary points of twist $(0,0)$,
these are precisely
the types of boundary twists we allow in graded $W$-spin disks. We call points with twist $(r-2,0)$ $r$-points, $(0,s-2)$ $s$-points, and $(r-2,s-2)$ fully twisted points. In particular,
if $\CM$ denotes such a boundary stratum and $\CM_1$, $\CM_2$
are the moduli spaces of the corresponding normalized disks (with
boundary points of twist $(0,0)$ forgotten),
we have a natural map $\CM\rightarrow \CM_1\times\CM_2$. We will call this moduli $\CM_1\times\CM_2$ the \emph{base} of the stratum $\CM$,
and it will be a central object for this paper. Further,
it follows from gluing properties of the Witten bundle (see Proposition
\ref{pr:decomposition}) that
$E_{\Gamma}(\vecd)|_{\CM}$ is the pull-back of a corresponding
descendent Witten bundle on $\CM_1\times \CM_2$, a sum $E_1\boxplus E_2$ of the pullbacks of descendent Witten bundles $E_i$
on $\CM_i$.
As a consequence, if we have already fixed multisections of $E_1$ and $E_2$,
we obtain via pullback under the map $\CM\rightarrow \CM_1\times\CM_2$
a multisection of $E_{\Gamma}(\vecd)|_{\CM}$. This allows us to inductively
determine boundary conditions along relevant strata.

To be useful, and to take into account inductively
defined boundary conditions, we organize this data as follows.
We will fix the set\footnote{We remark this is
a slight simplification of what is necessary in the main text.}
$\N=\{1,2,\ldots\}$ to be the set
of all possible labels of interior marked points, so that every $W$-spin
graph $\Gamma$ comes with an injective marking function $m^I:I(\Gamma)
\rightarrow \N$. We also fix a twist
function
\[
\tw:\N \rightarrow \{0,\ldots,r-1\}\times \{0,\ldots,s-1\}
\]
with the property that $\tw^{-1}(a,b)$ is an infinite set for
any $(a,b)\in \{0,\ldots,r-1\}\times \{0,\ldots,s-1\}$. We often write
$\tw(i)=(a_i,b_i)$ for $i\in\N$. We will thus always assume that an internal
marked point labeled by $i$ has twist $(a_i,b_i)$.
We typically will further fix a finite subset $I\subset \N$
of possible labels and focus attention on those graphs $\Gamma$
with $I(\Gamma)\subseteq I$.
This will help deal with a number of convergence issues in formal
power series rings when we define various  potentials later.

Given the above data,
we define the notion of a \emph{family of
canonical multisections bounded by $I$} (see
Definition \ref{def:family of canonical}). This consists of
a choice of multisection $\mathbf{s}^{\Gamma}$ of $E_{\Gamma}(\vecd)$
for each graph $\Gamma$ associated to a smooth graded $W$-spin disk
and descendent vector $\vecd\in\NN^{I(\Gamma)}$, with $I(\Gamma)\subseteq I$.
Further, these multisections satisfy the inductive boundary conditions
described above.
Note these sections must be defined even in the non-balanced case to be
useful; indeed, in general boundary strata involved in the inductive
definition of canonical families involve moduli spaces which are not
balanced.

Given such a family, we may now assemble a system of invariants. Let
$J\subseteq I$ and $\vecd\in \NN^J$. Denote by $\INT(J,\vecd)$
the set of graded
$W$-spin graphs $\Gamma$ with $I(\Gamma)=J$ which are balanced for
$\vecd$. Then define
\begin{equation}\label{eq: def Inv I}
\Inv(I)=\prod_{J\subseteq I}\prod_{\vecd\in\NN^J} \Q^{\INT(J,\vecd)}.
\end{equation}

A family $\ess$ of canonical multisections bounded by $I$ then produces an
element $\CI$ of $\Inv(I)$ whose entry corresponding to a
graph $\Gamma$ balanced for $\vecd$ is
\begin{equation}
\label{eq:chamber index from multisection}
\CI^{\ess}_{\Gamma,\vecd}=
\left\langle \prod_{i\in I(\Gamma)}\tau^{(a_i,b_i)}_{d_i}\sigma_1^{k_1(\Gamma)}\sigma_2^{k_2(\Gamma)}\sigma_{12}^{k_{12}(\Gamma)} \right\rangle^{\mathbf{s},o}.
\end{equation}

We remark here that we do not assume
our canonical multisections are invariant under a twist-preserving change
of labels of internal marked points. This will make many arguments easier,
and we may, if we wish,
pass to symmetric families by an averaging process, see
\S\ref{subsec:chamber indices}. We say a family $\ess$ of canonical multisections (bounded by $I$) is \emph{symmetric} if it is invariant under the natural lift of a twist-preserving bijection on $I$ to the level of moduli (see \textsection\ref{subsec:chamber indices} for a more precise treatment).
Analogously, we call an element $\CI\in \Inv(I)$
\emph{symmetric} if $\CI$ is invariant under all twist-preserving
automorphisms of the set $I$. Explicitly, let $\sigma:I\rightarrow I$
be a permutation with $\tw\circ\sigma=\tw$, and for a graph $\Gamma$
with internal tails labeled by elements of a set $J\subseteq I$,
let $\sigma(\Gamma)$ denote the same graph with an internal label
$i$ replaced by $\sigma(i)$. Then $\CI$ is symmetric if
\[
\CI_{\Gamma,(d_j)_{j\in J}} =
\CI_{\sigma(\Gamma),(d_{\sigma(j)})_{j\in J}}
\]
for all such $\sigma$. We let $\InvSym(I)\subseteq\Inv(I)$ denote the subset of symmetric
elements. With these definitions in hand, we can state our first theorem, which is a combination of  Theorem~\ref{prop:int_numbers_exist} and Observation~\ref{symmetric invariants exists}.

\begin{thm}\label{open FJRW invariants exist}
Open FJRW invariants for the LG model $(x^r+y^s, \mu_r \times \mu_s)$ exist.  That is, we can construct a family (bounded by $I$) of canonical transverse multisections. Moreover, one can construct a  family (bounded by $I$) of symmetric canonical transverse multisections. 
\end{thm}

We now turn to describing the systems of open FJRW invariants for the LG model  $(x^r+y^s, \mu_r \times \mu_s)$. Define 
\[
\OFJRW(I)= \{\CI^{\ess}\,|\,\hbox{$\ess$ a family of canonical multisections
bounded by $I$}\}\subseteq\Inv(I)
\]
to be the set of all systems of open FJRW invariants arising from
families of canonical multisections. We also define
\[
\OFJRWSym(I)=\OFJRW(I)\cap \InvSym(I)\subseteq\InvSym(I),
\]
which we will see consists of systems of open FJRW invariants
arising from symmetric families of canonical multisections.
A priori, Theorem~\ref{open FJRW invariants exist} only shows the nonemptiness $\OFJRW(I)$ and $\OFJRWSym(I)$.
Thus the real
task is to understand the structure of the sets $\OFJRW(I)$ and $\OFJRWSym(I)$.

Unlike the $r$-spin case of \cite{BCT:I,BCT:II}, it is very important
to keep in mind that these open $W$-spin invariants are dependent on boundary
conditions. It is precisely this dependence which we will explore
in more detail below in \S\ref{subsec:intro wall crossing}.
In general, given two families of canonical
multisections $\ess_0$ and $\ess_1$,
one can construct homotopies between them, and the change of
open invariants is then seen by counting
zeroes escaping off the boundary of moduli space, or coming in from the
boundary of moduli space. In other words, what matters is the number of
zeros of these homotopies restricted to the boundary of the various moduli
spaces. To harness this, we define
\emph{families of canonical homotopies} in Definition \ref{def:canonical homotopy}. It is important for the homotopies to have a rich structure for us to control their contributions to the invariants. In particular, they must be strongly positive and be pulled back from the base when restricted to relevant boundary strata. 
 In Lemma \ref{lem:partial_homotopy}, we show that we can construct families of canonical homotopies in such a way that we can control their non-vanishing on various boundary strata.

We then are able to provide a structure theorem for the sets $\OFJRW(I)$ and $\OFJRWSym(I)$ after taking inspiration from the structure of the $B$-model.

\subsection{The $B$-model}

In \S\ref{sec:B model} we study the period integrals which yield the
$B$-model side of our story. We refer the reader to \S\ref{subsec:state space} and \cite{GKTdim1}
for more details of these period integrals. The key point
is to study integrals of the form \eqref{eq:osc integral}
where $\Gamma$ runs over a so-called \emph{good basis} of cycles
$\Xi_{\mu}$ of unbounded $a$-dimensional cycles in $\C^a$
with the property that $\re(W/\hbar)$ tends to $-\infty$ in the
unbounded directions. Here, $\mu$ runs over tuples
$(\mu_1,\ldots,\mu_a)$ with $0\le \mu_i \le r_i-2$.
This basis is described explicitly in \cite{GKTdim1}.
These cycles have a multi-valued dependence on $\hbar$, hence exhibit 
monodromy as
$\hbar$ varies in $\C^*$.  The good basis is determined by the requirement that
\[
\int_{\Xi_{\mu}} x^{\mu'}e^{W/\hbar}
dx_1\wedge\cdots\wedge dx_a=\delta_{\mu \mu'},
\]
where $\delta_{\mu\mu'}$ is the Kronecker delta function, $x^{\mu'}$ denotes the monomial $\prod_i x_i^{\mu'_i}$,
and $W=\sum_i x_i^{r_i}$.

While the full $B$-model description gives a semi-infinite variation of
Hodge structure and a Frobenius manifold, we will ignore this structure,
other than making a connection between the description of this $B$-model
data in \cite{LLSS, HeLiShenWebb} and the explicit period integrals we consider
here. We briefly review the setup in terms of period integrals, and
send the reader to \S\ref{subsec:state space} for more details.

One considers a universal unfolding $W_{\mathbf{y}}$ of $W$,
where $\mathbf{y}$ denotes a set of coordinates $\{y_{\mu}\}$
on the universal unfolding parameter space. One then would like
to identify a \emph{primitive form} in the sense of Saito, i.e., a form
$f(x_1,\ldots,x_a,\mathbf{y})dx_1\wedge\cdots\wedge dx_a$ with the
property that
\[
\int_{\Xi_{\mu}} e^{W_{\mathbf{y}}/\hbar} fdx_1\wedge\cdots \wedge dx_a
=
\delta_{\mu 0} +t_{\mu}(\mathbf{y}) \hbar^{-1} +O(\hbar^{-2})
\]
for some collection of functions $t_{\mu}$. These functions then
form a system of coordinates on the universal unfolding moduli space
known as \emph{flat coordinates}. Once the above integrals are
written in terms of these flat coordinates, they become generating
functions for closed FJRW invariants. The precise form of this
generating function may be found in the statement of Theorem
\ref{thm:intro mirror theorem}.

To be explicit, we now return to the rank two case with $W = x^r + y^s$.
Fix $I\subseteq \N$ a finite set of allowable markings as in the
previous subsection, and recall the twist function $\tw$ assigning
a specific twist to each label in $I$. Consider the
ring
\[
A_{I,\mathrm{sym}}:=\Q[t_{\alpha,\beta,d}\,|\, 0 \le \alpha\le r-2,
0 \le \beta\le s-2, d\in \NN]/\mathrm{Ideal}(I),
\]
where $\mathrm{Ideal}(I)$ is the ideal generated by monomials of the
form
\[
\prod_d t_{\alpha,\beta,d}^{n_d}
\]
such that $\sum n_d > |(\tw|_I)^{-1}(\alpha,\beta)|$.
Note that the definition of $\mathrm{Ideal}(I)$ is
designed so that in the potential defined below, we only consider graphs
with at most $|(\tw|_I)^{-1}(\alpha,\beta)|$ internal marked points with twist
$(\alpha,\beta)$.

Given a symmetric element $\CI=(\CI_{\Gamma,\vecd})\in \InvSym(I)$, we may then define the
\emph{potential} of $\CI$ to be
\[
W^{\CI,\mathrm{sym}} = \sum_{k_1,k_2 \geq 0, l \geq 0} \, \,
\sum_{A= \{(\alpha_i, \beta_i,d_i)\} \in \mathcal{A}_l}
(-1)^{l-1} \frac{
\CI_{\Gamma_{0,k_1,k_2,1,\{(\alpha_i,\beta_i)\}},\vecd}}{|\Aut(A)|}
(\textstyle{\prod}_{i=1}^l t_{\alpha_i,\beta_i,d_i} ) x^{k_1}y^{k_2}
\in A_{I,\mathrm{sym}}[[x,y]].
\]
Here:
\begin{itemize}
\item $\mathcal{A}_l$ denotes the set of cardinality $l$ multisets of tuples
$(\alpha,\beta,d)$ with $0\le\alpha\le r-2$, $0\le\beta\le s-2$, $d\ge 0$;
\item $\Gamma_{0,k_1,k_2,1,\{(\alpha_i,\beta_i)\}}$ denotes
the graph corresponding to a smooth disk with:
\begin{itemize} 
\item $k_1$ boundary $r$-points,
\item $k_2$ boundary $s$-points, 
\item 1 boundary point that is fully twisted, and
\item $l$ internal marked points with twists $(\alpha_i,\beta_i)$,
$1\le i \le l$;
\end{itemize}
\item  $\Aut(A)$ denotes the automorphism group
of the multiset $A$, i.e., permutations $\sigma$ on $l$ letters such that
$(\alpha_{\sigma(i)},\beta_{\sigma(i)},d_{\sigma(i)})=(\alpha_i,\beta_i,
d_i)$ for $1\le i \le l$; and 
\item $\vecd=(d_i)_{1\le i \le l}$ denotes the
descendent vector.
\end{itemize}

\noindent We take the number $\nu_{\Gamma,\vecd}$ to be
zero if $\Gamma$ is not balanced for
$\vecd$. 

We make three remarks about this potential, assuming that
$\CI=\CI^{\ess}$, i.e., $\nu$ is enumerative, arising from a symmetric canonical family of
multisections $\ess$
via \eqref{eq:chamber index from multisection}. First, in fact we will show
(see Theorem \ref{thm:simple invariants}) that
\begin{equation}
\label{eq:needed cong}
W^{\CI,\mathrm{sym}}\equiv
x^r+y^s+\sum_{\substack{0\le \alpha\le r-2\\ 0\le \beta \le s-2}} t_{\alpha,\beta,0} x^\alpha y^\beta
\end{equation}
modulo all monomials quadratic in the variables $t_{\alpha,\beta,0}$ or linear
in $t_{\alpha,\beta,d}$ for $d>0$. This agrees
with a natural choice of universal unfolding of $W=x^r+y^s$.
However, in general there will be a possibly infinite number of additional
terms which are higher order in the coordinates $t_{\alpha,\beta,d}$.

Second, the condition that a graph $\Gamma$ must be balanced with
respect to $\vecd$ to contribute
to the potential makes $W^{\CI,\mathrm{sym}}$ homogeneous of
degree $rs$ with respect to an Euler vector field
\begin{equation}
\label{eq:Euler}
E:=sx\partial_x+ry\partial_y-\sum_{\stackrel{0\le\alpha\le r-2,\ 0\le\beta\le s-2,}{ d\ge 0}}
(s\beta+r\alpha+rs(d-1))t_{\alpha,\beta,d}\partial_{t_{\alpha,\beta,d}}.
\end{equation}
In other words,
\[
E(W^{\CI,\mathrm{sym}})=rs W^{\CI,\mathrm{sym}}.
\]

Third, integrality considerations for degrees of spin bundle show that
$W^{\CI,\mathrm{sym}}$ is invariant under the action 
$x\mapsto \xi_r x$, $y\mapsto \xi_sy$, $t_{a,b,d}\mapsto 
\xi^{-a}_r \xi^{-b}_s t_{a,b,d}$, where $\xi_r,\xi_s$ are primitive
$r^{th}$ and $s^{th}$ roots of unity respectively. 

The following theorem, which is Corollary \ref{cor:symmetric integal},
is fairly straightforward via integration by parts:

\begin{thm}
\label{thm:osc int}
Fix $\CI\in\InvSym(I)$ such that
\[
W^{\CI,\mathrm{sym}}\equiv x^r+y^s \mod \langle t_{\alpha,\beta,d} \ | \ 0 \le \alpha\le r-2, \  0 \le \beta \le s-2, \ 0 \le d\rangle.
\]
For a multiset $A\in \bigcup_l \mathcal{A}_l$, $A=\{(\alpha_i,\beta_i,d_i)\}$,
define $r(A)$ to be the unique element of $\{0,\ldots,r-1\}$
congruent to $\sum_{i=1}^l \alpha_i$ modulo $r$, and similarly  $s(A)$  the
unique element of $\{0,\ldots,s-1\}$ congruent to $\sum_{i=1}^l \beta_i$ modulo $s$.
Finally, let
\[
d(A)={sr(A)+rs(A)-\sum_{i=1}^l (s\alpha_i+r\beta_i+rs(d_i-1))\over rs}-2.
\]

Then
\begin{equation}\begin{aligned}\label{messy intro equation}
&\int_{\Xi_{(a,b)}} e^{W^{\CI,\mathrm{sym}}/\hbar} dx \wedge dy\\
= {} & \delta_{a,0}\delta_{b,0} +
\sum_{l\ge 1}
\sum_{A=\{(\alpha_i,\beta_i,d_i)\}\in \mathcal{A}_l} (-1)^l
(-\hbar)^{-d(A)-2}
\mathcal{A}(A,\CI)\delta_{r(A), a}\delta_{s(A), b}
\left({\prod_{j=1}^l t_{\alpha_j,\beta_j,d_j}\over |\textup{Aut}(A)|}\right).
\end{aligned}\end{equation}
Here $\mathcal{A}(A, \CI)\in \mathbb{Q}$ is defined in Notation \ref{nn:A(J)} and Definition~\ref{AAD}.
\end{thm}

We do not repeat the definition of $\mathcal{A}(A,\CI)$ in
the introduction as it is slightly involved, but it involves
a sum over products of the $\CI_{\Gamma,\vecd}$ and is obtained quite
naturally from evaluating the integral by integration by parts
after a power series
expansion of $e^{(W^{\CI,\mathrm{sym}}-x^r-y^s)/\hbar}$. Nevertheless, these
invariants play a crucial role in the paper, with our form of open
topological recursion being a recursion involving a combination
of these $\mathcal{A}(A,\CI)$ and closed extended FJRW
invariants.

\subsection{The mirror theorem and open topological recursion}
The key mirror symmetry result in our paper is the following. A priori, if we choose a deformed potential $W^{\CI, \mathrm{sym}}$ with coefficients determined by an arbitrary $\CI \in \InvSym(I)$, we can compute the oscillatory integrals in Theorem~\ref{thm:osc int}, but the quantities $\mathcal{A}(A, \nu)$ in ~\eqref{messy intro equation} contain no enumerative meaning.  However, when we use a system of open FJRW invariants for $\CI$ to construct $W^{\CI, \mathrm{sym}}$, the variables $t_{a,b,0}$ are flat coordinates for the underlying Frobenius manifold and the oscillatory integrals are then generating functions for all closed extended FJRW invariants with descendents with internal markings bounded by $I$. Moreover, this only happens when $\nu$ is enumerative, that is,  a system of open FJRW invariants.

We can state this precisely as follows.
We continue with a choice of finite set $I\subseteq \N$. The following is
Corollary \ref{cor:open MS}.

\begin{thm}[Open Mirror Symmetry for 2-dimensional Fermat polynomials]
\label{thm:intro mirror theorem}
Let $\CI\in \OFJRWSym(I)$. Then
\begin{align*}
&\int_{\Xi_{a,b}} e^{W^{\CI,\mathrm{sym}}/\hbar} dx \wedge dy
= {}
\delta_{a,0}\delta_{b,0} + \sum_{d\ge 0} t_{a,b,d}\hbar^{d-1}
\\
{} & +\sum_{l\ge 2}
\sum_{\substack{A\in \mathcal{A}_l\\
d(A)\ge 0}} (-1)^{l}
(-\hbar)^{-d(A)-2}{\delta_{r(A),a}\delta_{s(A),b}\over |\textup{Aut}(A)|}
 \left\langle \tau_{d(A)}^{(r-r(A)-2,s-s(A)-2)}\prod_{(\alpha_i,\beta_i,d_i)
\in A}\tau_{d_i}^{(\alpha_i,\beta_i)}
\right\rangle^{\textup{ext}}
\prod_{(\alpha_i,\beta_i,d_i)\in A} t_{\alpha_j,\beta_j,d_j}.
\end{align*}
 Here $r(A)$ and $s(A)$ are as defined in Theorem \ref{thm:osc int} and $\< \cdots \>^{\mathrm{ext}}$ denotes the closed extended FJRW
invariants, as $r-r(A)-2$ or $s-s(A)-2$ may be $-1$.

In fact, if one removes the descendent part of the potential $W^{\CI,\mathrm{sym}}$
by working modulo the ideal $\langle t_{\alpha,\beta,d} \,|\, 0 \le \alpha 
\le r-2, \ 
0\le \beta \le s-2, \ 0<d\rangle$, then $dx\wedge dy$ is a primitive form
and the $t_{a,b,0}$ are flat coordinates.
\end{thm}

\begin{rmk}
(1)
The above theorem matches with the philosophy of the open Gromov-Witten 
framework for Fano manifolds. For example, work of Cho--Oh, Gross and
Fukaya--Oh--Ohta--Ono \cite{ChoOh,GrossP2,FOOO1} showed how counts of
Maslov index two disks with boundaries on Lagrangian tori naturally
give rise to the correct mirror potentials. Our framework then provides
an analogue of Maslov index two disks in FJRW theory.

(2) When working modulo the ideal $\langle t_{\alpha,\beta,d} \,|\, 0 \le \alpha \le r-2, \ 0\le \beta \le s-2, \ 0<d\rangle$, we are in the setting of standard Saito-Givental theory. Here we recover the (closed) mirror symmetry results of \cite{LLSS, HeLiShenWebb} in the case of $W=x^r+y^s$. The proofs are fundamentally different. There, the authors consider the integrand $e^{W^{\nu,\mathrm{sym}}}f dx \wedge dy$ where $W^{\nu,\mathrm{sym}}$ is the versal deformation given by~\eqref{eq:needed cong} and perturbatively correct the coordinates through the function $f$ (Theorem 3.7 of \cite{LLSS}) up to quadratic terms and then use reconstruction arguments. Our approach finds a precise versal deformation $W^{\nu,\mathrm{sym}}$ using open FJRW invariants so that $f=1$ and the coordinates for the versal deformation are already flat.

(3) In the context of open mirror symmetry for Fano manifolds, Overholser \cite{Overholser} generalized the oscillatory integral framework to include open descendent invariants for $\P^2$. Theorem~\ref{thm:intro mirror theorem} creates a similar framework for Landau-Ginzburg mirror symmetry, deducing all genus zero FJRW invariants with descendents from the oscillatory integrals.
\end{rmk}

In light of Theorem \ref{thm:osc int}, the proof of Theorem
\ref{thm:intro mirror theorem} follows from
a calculation of the complicated invariants $\mathcal{A}(A,\CI)$,
which are shown in Theorem \ref{thm:A_mod_invs} to be independent
of the choice of $\CI\in \OFJRWSym(I)$. It also provides a large amount of structure for the open FJRW invariants due to the sheer number of constraints given by the relations in Theorem~\ref{thm:intro mirror theorem}. This is summarized in Corollary~\ref{thm:mirrorA}.

\medskip

The key tool for calculating these quantities is open topological recursion. To express it, we consider the case where the open FJRW invariants are defined using a family of canonical multisections bounded by $I$ that need not be symmetric. If $\CI \in \Inv(I)$, we have a expression $\mathcal{A}(I, \vecd, \nu)$ similar to $\mathcal{A}(A, \nu)$ given by the oscillatory integrals, as defined in Notation~\ref{nn:A(J)}. In Theorem~\ref{thm:A_mod_invs}, we show that the expression $\mathcal{A}(I, \vecd, \nu)$ is independent of $\nu$ whenever $\nu \in \OFJRW(I)$. For this reason, we write $\mathcal{A}(I, \vecd):=\mathcal{A}(I, \vecd, \nu)$ for $\nu \in \OFJRW(I)$.  We are now ready to state open topological recursion:

\begin{thm}[Open topological recursion for 2-dimensional Fermat polynomials]
\label{thm:open TRR intro}
Given $J\subseteq I$ non-empty, choose $j_1\in J$
and $\vecd\in \NN^J$. Denote by $\mathbf{e}_1\in \NN^J$ the vector
whose entry corresponding to $j_1$ is $1$ and all other entries zero.
Then the following identities hold.
\begin{enumerate}
\item
If $|J|\ge 1$, then
\begin{align}\label{eq:calculation_1}
\mathcal{A}&(J,\vecd+\eee_1)=\\
&\notag\quad\quad\sum_{\substack{a\in\{-1,\ldots,r-2\}\\b\in\{-1,\ldots,s-2\}}}\sum_{\substack{A \coprod B = J\setminus\{j_1\}\\ A\not=\varnothing}}\left\langle \tau_0^{(a,b)}\tau_{d_{j_1}}^{(a_1,b_1)}\prod_{i \in A}\tau_{d_i}^{(a_i,b_i)}\right\rangle^{\textup{ext}}_0
\mathcal{A}(B\cup\{z_{a,b}\},\vecd)-
\mathcal{A}(J,\vecd).
\end{align}
Here $z_{a,b}$ is an additional internal marking with twist $(r-2-a,s-2-b)$,
with $d_{z_{a,b}} = 0$. 
\item If $|J|\ge 2$, let $j_2\in J\setminus \{j_1\}$. Then
\begin{equation}\begin{aligned}\label{eq:calculation_2}
\mathcal{A}&(J, \vecd + \eee_1)\\
&=\sum_{\substack{a\in\{-1,\ldots,r-2\}\\b\in\{-1,\ldots,s-2\}}}\sum_{\substack{A \coprod B = J\setminus\{j_1\},j_2\in B\\ A\not=\varnothing}}\left\langle \tau_0^{(a,b)}\tau_{d_{j_1}}^{(a_1,b_1)}\prod_{i \in A}\tau_{d_i}^{(a_i,b_i)}\right\rangle^{\textup{ext}}_0 \mathcal{A}(B\cup\{z_{a,b}\},\vecd).
\end{aligned}\end{equation}
where $z_{a,b}$ and $d_{z_{a,b}}$ are as in (1).
\end{enumerate}
\end{thm}

The proof of open topological recursion occupies
\S\ref{subsec:proof of open TRR} and is the heart of the paper.
The proof works by choosing a suitable multisection of $\CL_{j_1}$
whose vanishing locus consists of strata in which closed spheres bubble
off from disks. Most importantly this vanishing locus is easy to
describe. However, this multisection
does not belong to a family of canonical multisections, and hence a detailed study of zeroes arising
in homotoping this multisection to a canonical multisection is required. This
is the hardest part of the proof.

Another key difference in the open topological recursion relations in our paper and the previous ones found in the literature is that there does not exist an open topological recursion relation for a single open invariant. Rather one must use the polynomial $\mathcal{A}(J,\vecd+\eee_1,\CI)$ in open invariants  in order to obtain a topological recursion relation. This is unlike what is found in the open $r$-spin case in \cite{BCT:II} and the  topological recursion structures found by higher Airy structures (see, e.g., \textsection6 of \cite{BBCCN}).

Furthermore, the quantities $\mathcal{A}(I, \vecd)$ are familiar. We can crucially use Theorem~\ref{thm:open TRR intro} to establish in Corollary~\ref{thm:mirrorA}(3) an open-closed correspondence that links polynomials of open invariants to closed extended invariants.
\begin{cor}[Open-Closed Correspondence]
Given $J\subseteq I$ non-empty
and $\vecd\in \NN^J$. Suppose $|J|\ge 2$.
Let $d(J,\vecd)$ be as defined in Definition \ref{def:d(J)}. Then
\[
\mathcal{A}(J,\vecd)=
\begin{cases}
0 & d(J,\vecd)<0\\
(-1)^{d(J,\vecd)-1}\left\langle \tau_{d(J,\vecd)}^{(r-r(I)-2,s-s(I)-2)}\prod_{i\in J}\tau_{d_i}^{(a_i,b_i)}
\right\rangle^{\textup{ext}}&d(J,\vecd)\ge 0.
\end{cases}
\]
\end{cor}

\subsection{Wall-crossing formulae}
\label{sec:wall-crossing-intro}
The remaining question is to understand the subset $\OFJRWSym(I)$
giving all possible sets of symmetric open $W$-spin invariants.
As mentioned earlier in this introduction, given two families
of canonical multisections $\ess_0$ and $\ess_1$, there exists
a canonical family of homotopies $H^*$ between them, and the change
of invariants associated to a $W$-spin graph $\Gamma$
requires understanding the number of zeros $\#Z(H|_{[0,1]\times \partial\oPM^W_{\Gamma}})$.
Crucially, these zeros will occur on boundary strata corresponding
to $W$-spin graphs $\Lambda$ representing stable $W$-spin disks which
are unions of two disks. These disks meet at a node with the twists at the half-nodes satisfying certain constraints. Further, we have the  relation
\[
\dim\oCM^W_{\Lambda}=\rank E_{\Gamma}(\vecd)-1
\]
so that $H^{\Gamma}|_{[0,1]\times \oPM^W_{\Lambda}}$ will have a well-defined
number of zeroes. Such boundary strata are called \emph{critical
boundaries}, and are responsible for wall-crossing phenomena. Given a graph $\Lambda$, we denote by $\Lambda_v$ the subgraph consisting of all half-edges adjacent to the vertex $v$ in $\Lambda$. A graph $\Lambda$ corresponding to a critical boundary will have a vertex $v$ 
such that the subgraph $\Lambda_v$ will be graded, have no fully twisted boundary marked points, and 
\[
\dim\oCM^W_{\Lambda_v}=\rank E_{\Lambda_v}(\vecd)-1.
\]
We will call such a (connected) subgraph $\Lambda_v$ \emph{critical for $\vecd$}, see Definition~\ref{def:critical}(1).

We define the \emph{Landau-Ginzburg wall-crossing group} (Definition~\ref{Wall Crossing Group}) to be the subgroup $G_{A_{I, \mathrm{sym}}}\subseteq \Aut_{A_{I,\mathrm{sym}}}(A_{I,\mathrm{sym}}[[x,y]]) $ of the group of continuous automorphisms of
$A_{I,\mathrm{sym}}[[x,y]]$ as an $A_{I,\mathrm{sym}}$-algebra consisting of elements that (1) are the identity modulo the ideal generated by the $t_{\alpha,\beta,d}$;
(2) preserve $dx\wedge dy$; (3) preserve the ideal generated by $xy$. These properties are natural for our framework. The first property arises as all
critical boundaries involve disks with internal marked points.
The second property ensures that, if $g\in
G_{A_{I,\mathrm{sym}}}$,
then replacing a potential
$W^{\CI,\mathrm{sym}}$ with $g(W^{\CI,\mathrm{sym}})$ does
not change the result of the period integrals. Finally, the third condition is more subtle,
but reflects that for certain boundary strata, we are able to define
positivity of multisections of the Witten bundle to guarantee that
zeroes never occur on these boundary strata. This issue arises in
Definitions \ref{def:special kind of graded graphs}(3) and
\ref{def: strongly positive}(3).

 It turns out that the wall-crossing is governed by a Lie subgroup
\[
G^{r,s}_{A_{I,\mathrm{sym}}} \subset
G_{A_{I,\mathrm{sym}}},
\]
which is specified via its corresponding Lie algebra $\mathfrak{g}^{r,s}_{A_{I,\mathrm{sym}}}$. This Lie algebra is a subalgebra of the Lie algebra
$\mathfrak{g}_{A_{I,\mathrm{sym}}}$ of $G_{A_{I,\mathrm{sym}}}$ consisting
of vector fields on which $[E,\cdot]$ acts by multiplication by $rs$,
where $E$ is as in 
\eqref{eq:Euler}, and are invariant under the action 
$x\mapsto \xi_r x$, $y\mapsto \xi_s, y$, $t_{a,b,d}\mapsto 
\xi^{-a}_r \xi^{-b}_s t_{a,b,d}$, where $\xi_r,\xi_s$ are primitive
$r^{th}$ and $s^{th}$ roots of unity respectively. 

The underlying vector space of the Lie algebra is a direct sum of one-dimensional $\mathbb{Q}$-vector spaces indexed by all critical graphs.

Theorem \ref{thm:G action} shows that the group $G^{r,s}_{A_{I,\mathrm{sym}}}$ acts on $\InvSym(I)$ via its action
on potentials: for $g\in G^{r,s}_{A_{I,\mathrm{sym}}}$ and $\CI\in
\InvSym(I)$, we define $g(\CI)$ by the formula
\[
W^{g(\CI),\mathrm{sym}}=g(W^{\CI,\mathrm{sym}}).
\]
The main wall-crossing theorem (see Corollary \ref{chambers are enumerative})
then states:

\begin{thm}\label{introthm: torsor}
The group $G^{r,s}_{A_{I,\mathrm{sym}}}$ acts faithfully and transitively
on $\OFJRWSym(I)$. Further, $\OFJRWSym(I)\subseteq \InvSym(I)$ is characterized
as the subset of those $\CI\in\InvSym(I)$ satisfying the following conditions:
\begin{enumerate}
\item If $J\subseteq I$  and
$\Gamma\in \INT(J,\mathbf{0})$, then
$\CI_{\Gamma,\mathbf{0}}=1$ if $|J|=1$ and
$\CI_{\Gamma,\mathbf{0}}=-1$ if $|J|=0$.
\item For $I'\subseteq I$ with $|I'|=1$, $\vecd = (d)\in \NN^{I'}$,
we have
\[
\mathcal{A}(I',\vecd,\CI) = (-1)^d.
\]
\item For any $J\subseteq I$, $\mathbf{d}\in \NN^{J}$ with
$d(J,\mathbf{d})<0$, we have $\mathcal{A}(J,\mathbf{d},\CI)=0$.
\end{enumerate}
\end{thm}

In fact, the three conditions in the theorem characterizing
$\OFJRWSym(I)$ are precisely the conditions guaranteeing
that the oscillatory integral $\int e^{W^{\CI,\mathrm{sym}}/\hbar} \Omega$
of Theorem \ref{thm:intro mirror theorem} takes the form given there up to terms of
the form $\hbar^{-n}$, $n\ge 2$.

\begin{rmk}
The Landau-Ginzburg wall crossing group  is very similar in spirit to the tropical vertex group,
a group of automorphisms
of a two-dimensional algebraic torus preserving ${dx\over x}\wedge {dy\over y}$
introduced by Kontsevich and Soibelman in \cite{KontSoib}
and given an enumerative interpretation for Fano manifolds in \cite{GPS}.
Generalizations of the tropical vertex group have been used extensively in the literature
in a large number of wall-crossing situations. For Fano manifolds, elements of the wall-crossing group for a Fano manifold correspond to Maslov index zero disks.  In our case, a $\mathbb{Q}$-summand of automorphisms in $G^{r,s}_{A_{I,\mathrm{sym}}}$ correspond to a critical graph, which is crucial in the proof of Theorem~\ref{introthm: torsor}.
\end{rmk}

\begin{rmk}
For pairs $(r,s) \in \{(n,2), (2,n), (3,3), (3,4), (4,3) \ | \ n \in \mathbb{Z}_{\ge 2}\}$, the potentials $W=x^r + y^s$ are simple singularities $A_{n-1}$, $E_6$, and $E_8$. Here, there are no critical graphs with respect to descendent vector $\vecd= \mathbf{0}$, and thus the open FJRW invariants with no descendents are well-defined. The fact that the primitive form is unique in these cases is well-established \cite{Sai83, NoumiYamada}. However, in these cases there are indeed critical graphs for non-trivial $\vecd$ and wall-crossing for the theory with descendents. We predict that the only open FJRW invariants that do not enjoy wall-crossing phenomena will be the simple singularities when considered with no descendents.
\end{rmk}

\subsection{Context and comparison to existing literature}

Mirror symmetry for open Gromov Witten (OGW) theories with Landau-Ginzburg (LG) mirrors were first considered for Fano manifolds in \cite{ChoOh,GrossP2,FOOO1} Here, the open Gromov-Witten invariants corresponded to compact toric manifolds paired with a Lagrangian torus $T^n$.  The toric $n=1$ case generalizes nicely to some more complicated Lagrangians, essentially to Lagrangians whose rational cohomology is spanned by the unit and the point class, and certain generalizations of this case (see, for example, \cite{Sol16}).  It turns out that such OGW theories \cite{Sol16,Net17,Bur22} give rise to intersection numbers which are independent of choices  and are governed by a universal recursion, the \emph{open WDVV (OWDVV) equation} \cite{Sol07},\cite[Theorem 3]{Sol19}. Moreover, they satisfy other interesting properties such as an extension of the quantum cohomology to relative quantum cohomology, a correspondence between open invariants with a certain internal relative constraint, and open invariants in which this constraint is replaced by a boundary constraint \cite[Theorem 6]{Sol19}. In the case of the target pair $(\mathbb{CP}^1,\mathbb{RP}^1)$ the OWDVV extends also to descendant theory \cite{Bur22}, and such an extension is expected to hold more generally.

The OGW theory in the case $n\geq 2$ is less understood. Here, Gross and Overholser  studied the tropical open $\mathbb{CP}^2$ case with descendants \cite{GrossP2,Overholser}, while Fukaya, Oh, Ohta, and Ono studies the case $n\geq2$ without descendants in detail and obtained mirror symmetry for $n=2$ and under additional assumptions also for $n>2$ targets \cite{FOOO1}. In the general case where $n\geq 2$, invariants are defined only up to wall crossings. These wall crossings are conjectured to be governed by the Kontsevich-Soibelman wall crossing group \cite{KontSoib}, and this conjecture is proven in some tropical or symplectic cases \cite{GrossP2,Overholser,Lin17a,Lin17b}.

\emph{Open FJRW} theory (OFJRW) is a theory which is still under development. The first construction of an OFJRW theory, the open $r$-spin theories, has appeared in \cite{BCT:I,BCT:II}, where the main difficulty was identifying a natural choice of boundary conditions to the \emph{open Witten bundle}. These boundary conditions rely on a hidden notion of positivity in this theory.  Open $r$-spin are precisely the OFJRW analog of the $n=1$ OGW case described above.
As such they have well defined invariants \cite[Theorem 3.17]{BCT:II}, satisfy the OWDVV equations \cite[Section 4]{BCT:II}, involve an extension of the closed state space (where the Ramond states play the role of relative classes) \cite{BCT_Closed_Extended} and a correspondence between open and closed extended intersection numbers.

This work addresses the construction and calculation of the OFJRW analog of the $n=2$ OGW case. The first sections of the paper up to Section~\ref{subsec:Rank 2} provide a geometric construction of the open moduli spaces in all dimension; however, we specialize to the dimension $n=2$ case exactly when the theory diverges from the situations found in $n<2$ and $n>2$. We then geometrically construct the rank $2$ theory, including descendants, characterize all invariant quantities, which this time do not satisfy the OWDVV equations, but a different universal recursion. We then identify the wall crossing group and realize it geometrically. Our geometric construction involves the boundary conditions found in \cite{BCT:II}. Interestingly, there are certain boundaries for which the boundary conditions are not specified. They turn out to be exactly the source for the wall crossing. The calculation of invariants and geometric realization of the wall crossings involve new ideas which are then used to establish mirror symmetry in this case.

\medskip

\emph{Acknowledgements:} The authors would like to thank Alexander Buryak, Emily Clader, Rahul Pandharipande, Yongbin Ruan, and Jake Solomon for discussions relating to this work. The authors would also like to thank the referee for their comments which have improved the paper. The first author acknowledges support from the
EPSRC under Grants EP/N03189X/1, a Royal Society Wolfson Research Merit Award,
and the ERC Advanced Grant MSAG. The second author acknowledges support provided by the National Science Foundation under Award No.\ DMS-1401446, the EPSRC under Grants EP/N004922/1,2 and EP/S03062X/1, and the UK Research and Innovation Future Leaders Fellowship MR/T01783X/1,2 and MR/Y033841/1. The third author, incumbent of the Lillian and George Lyttle Career Development Chair, acknowledges support provided by the ISF grants No. 335/19 and 1729/23, and by a research grant from the Center for New Scientists of Weizmann Institute.
\label{subsec:intro wall crossing}

\section{Graded $W$-spin surfaces}
\label{sec:graded surfaces}

\subsection{Closed graded $W$-structures}
\label{subsec:closed graded}
We begin by defining closed $W$-spin Riemann surfaces. These are needed
for two reasons. First, an open $W$-spin surface is a closed surface
equipped with additional structure, including an anti-holomorphic involution.
Second, even in the study of open surfaces, disks may bubble off spheres
on the boundary of moduli space. Hence moduli of closed $W$-spin
surfaces naturally occur
in the boundary of open $W$-spin surfaces. We first consider the special
case of an $r$-spin surface, reviewing terminology from \cite{BCT:I}.

A \emph{closed marked genus $0$ orbifold Riemann surface $C$}
is a proper, one-dimensional Deligne--Mumford stack whose coarse
space $|C|$ is a genus $0$ Riemann surface with at worst nodal singularities, and such that the morphism $\pi: C \to |C|$ is an
isomorphism away from the finitely many marked points and nodes.
We often allow disconnected surfaces, in which case it must be a disjoint
union of such genus $0$ surfaces.
We collectively call these marked points and nodes the \emph{special points}. Any special point may have a non-trivial stabilizer which is required to be a finite cyclic group. Any node must be \emph{balanced}, i.e.,
the local picture at any node is
\begin{equation}\label{eq:balanced}
\{xy=0\}/\mu_d,
\end{equation}
where $\mu_d$ is the group of $d^{th}$ roots of unity, and generator $\zeta$ acts by
\begin{equation}
\label{eq:balanced action}
\zeta\cdot(x,y)=(\zeta x,\zeta^{-1}y).
\end{equation}
The curve $C$ is \emph{$d$-stable} if all special points have isotropy
group $\mu_d$.

As in \cite{BCT:I}, we introduce a set for markings: let $\Universe$
denote the set of all finite subsets of $\N=\{1,2,3,\ldots\}$.
We identify the element $\{i\}\in\Universe$ with $i$, for $i\in \N$.
A marking of a set $A$ is a function
\[m: A \rightarrow \Universe,\]
such that whenever $a,a'$ are distinct elements of $A\setminus m^{-1}(\emptyset)$,
we have $m(a)\cap m(a')=\emptyset$. Such functions are used in what
follows to label the marked points on the curve. The possibility of marking
some points by $\emptyset$ or with a set is needed for two reasons. First, we use marking by $\emptyset$ to  handle boundary marked points, whose cyclic order will often serve as a substitute for their markings. Second, we use $\emptyset$ and subsets of $\N$ in a consistent way to mark new internal marked
points that arise via normalization of a nodal curve, see Definitions \ref{def:closed normal} and~\ref{def:open normal}. An element of $A$ which is mapped to $\emptyset$ will sometimes be referred as \emph{unlabeled}.

Explicitly, the markings on the curve $C$ are indexed by a set $I$
and are labeled by a marking function $m^I:I\to\Universe$.
We denote the marked points of $C$ by $z_i,~i\in I$. There is another distinction we have for marked points:

\begin{definition}[Anchors]
\label{def:anchors}
We consider a subset of the marked points of $C$
to be \emph{anchors}. We allow at most one anchor in each connected component
of $C$
and require that this point be labeled $\emptyset$.
A connected component of $C$ is \emph{anchored} if it contains an anchor.
\end{definition}

\begin{rmk}
\label{rmk:anchor yoga}
The anchor will serve as a book-keeping mechanism to deal with closed $W$-spin
surfaces which arise as a connected component in a partial normalization
of an open surface.
The anchor should be thought of as the marked point in a closed connected component which, before normalization, was a half-node of an open surface. In this context, such a half-node connected the closed part to the open part or  is a contracted boundary half-node. Note that as we only consider genus zero in this
paper, the anchor must be unique.

We define the slightly complicated labeling above in order to keep track of this information and to avoid adding symmetries when we normalize.
Intersection theory on moduli spaces of $r$-spin Riemann surfaces without boundary is, of course, not affected by the anchor, and is equivalent to the classical one defined in \cite{Witten93}. The same comment holds for closed $W$-spin surfaces, to be defined below.

This yoga of anchors thus requires some care in defining the normalization
of a curve.
\end{rmk}

\begin{definition}
\label{def:closed normal}
Let $q$ be a node of $C$.
The \emph{normalization} of $C$ at $q$ is the usual normalization at $q,$ $\NNN_q:C'\to C,$ with additional care concerning anchors. Both points of
the preimage of $q$ become marked points.
Suppose a connected component $C_1$ of $C$ containing $q$ is separated into
two connected components $C_1',C_1''$ by the normalization. If $C_1$ is anchored and $C'_1$ does not contain the anchor, we define the preimage of $q$ in
$C'_1$ as its anchor and label it $\emptyset$ (See Figure~\ref{fig:norm_anchors}). The other preimage of $q$ will be labeled by the union of labels of the other marked points of $C_1'$. In all other cases, the preimages of $q$ are labeled $\emptyset$. Observe that the resulting labeling is still a marking. We denote by
\[
\NNN:\widehat{C}\to C
\]
the result of the normalization at all nodes.

A \emph{half-node} is a point of the normalization $\widehat{C}$
mapping to a node $q$ in $C$. We often refer to a half-node also as a node $q\in C$
along with a choice of branch of $C$ at $q$.
\end{definition}

\begin{figure}
  \centering
\begin{tikzpicture}[scale=0.6]

%The normalisation

  \draw (-1,7) circle (2cm);
  \draw (-3,7) arc (180:360:2 and 0.6);
  \draw[dashed] (1,7) arc (0:180:2 and 0.6);
    \draw [fill] (1,7) circle [radius=0.15];

  \draw [fill, blue] (-1,6) circle [radius=0.15];

  \draw (5,7) circle (2cm);
  \draw (3,7) arc (180:360:2 and 0.6);
  \draw[dashed] (7,7) arc (0:180:2 and 0.6);
  \draw [fill, blue] (3,7) circle [radius=0.15];

  \node[above] at (2,8) {$\nu_q^{-1}(q)$};
  \draw (2.5,8) to [out=270,in=180] (3,7);

    \draw (1.5,8) to [out=270,in=0] (1,7);

  \draw [->] (2,4.5) -- (2,2.5);
  \node [right] at (2,3.5) {$\nu_q$};

%The nodal curve

  \draw (0,0) circle (2cm);
  \draw (-2,0) arc (180:360:2 and 0.6);
  \draw[dashed] (2,0) arc (0:180:2 and 0.6);

\draw [fill] (2,0) circle [radius=0.15];
\node [below] at (2,-1) {$q$};
    \draw (2,-1) to (2,0);

\draw [fill, blue] (0,-1) circle [radius=0.15];

  \draw (4,0) circle (2cm);
  \draw (2,0) arc (180:360:2 and 0.6);
  \draw[dashed] (6,0) arc (0:180:2 and 0.6);

\end{tikzpicture}

\caption{The normalization $\nu_q: C' \rightarrow C$ of the connected component $C_1$ at a node $q$. Anchors are color-coded as blue.}
\label{fig:norm_anchors}
\end{figure}
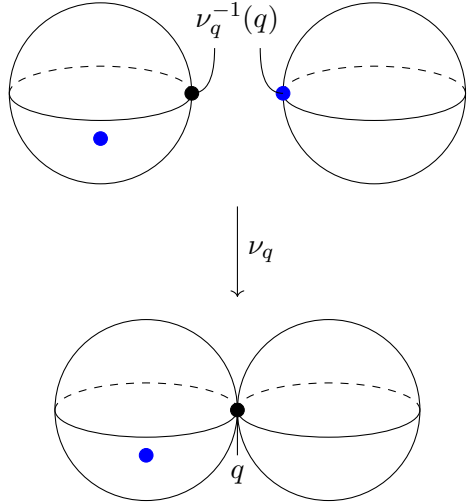

Suppose that $C$ is $d$-stable and $r$ divides $d$.
An \emph{$r$-spin} structure on $C$ is an orbifold line bundle $L\to C$, called the \emph{$r$-spin bundle}, together with an isomorphism
\begin{equation}\label{defn: iso tau L}
\tau:L^{\otimes r}\to\omega_{C,\log}.
\end{equation}
The local structure of $L$ at a marked point $p$ looks like the quotient stack $[\C^2/\mu_d]$ where the generator $\zeta=e^{\frac{2\pi i}{d}}\in\mu_d$ acts by
\begin{equation}\label{eqn: local action}
\zeta\cdot(X,T)=(e^{\frac{2\pi i}{d}} X,e^{m\frac{2\pi i}{d}} T).
\end{equation}
Here $X$ is a local coordinate on the curve, $T$ a coordinate on the
fibres of $L$, and
the \emph{multiplicity} (of $L$) is some integer $m=\text{mult}_p(L)$  modulo $d$. We impose here that $\text{mult}_p(L) \in \{0,\dots, d-1\}$. Moreover, since the multiplicity of $\omega_{C,\log}$ at any point is $0,$ for each marking $p$ it must hold that
\[{\text{mult}_p(L)\over d/r}\in\Z. \]

A \emph{twisted $r$-spin} structure on a closed marked genus $0$ orbifold Riemann surface $C$ is an orbifold bundle of the form
\begin{equation}\label{eq:S}S = L \otimes \O\left(-\sum_{i \in I_0} [z_i] \right),\end{equation}
where $L$ is an $r$-spin bundle and $I_0$ is a subset of the markings of $L$-multiplicity zero which contains
every marking whose $L$-multiplicity is $0,$ except, possibly,
anchors. Note that a priori the subset $I_0$ is a choice and part of the data of the $r$-spin structure, but we will see below that there is a standard choice of $I_0$ determined by the rest of the $r$-spin structure. Here, the twist
$\O([z_i])$ is the orbifold line bundle associated to the degree $1$ divisor
$[z_i]$, i.e., $\O([z_i])$ is the pull-back of a line bundle on the
coarse moduli space.
The \emph{twist} of the $i^{th}$ marking (with respect to $S$) is defined to be \[l_i-1+r\mathbf{1}_{i\in I_0}\] where  $l_i$ is the unique solution of\[\text{mult}_{z_i}(L) = \left(\frac{d}{r}\right)l_i,~~l_i\in\{0,\ldots,r-1\}\]
and $\mathbf{1}_{i\in I_0}$ denotes the characteristic function of
$I_0$. Thus in particular:
\begin{equation}
\label{eq:I0}
\hbox{If a marked point has twist $-1$, it is an anchor.}
\end{equation}

Using~\eqref{defn: iso tau L}, note that there is an isomorphism $\tau$ so that
\begin{equation}\label{defn: iso tau S}
\tau: S^{\otimes r}\simeq\omega_{C,\log}(-\sum_{i\in I_0}r[z_i]).
\end{equation}

\begin{definition}
A \emph{closed twisted genus $0$ $r$-spin surface} is a closed genus $0$
$d$-stable surface together with a twisted $r$-spin structure.
The notion of isomorphism is the standard one.
\end{definition}

Given the normalization $\nu:\widehat C\rightarrow C$,
the orbifold line bundle $\NNN^*L\to\widehat{C}$ is automatically an $r$-spin bundle. However,  $\NNN^*S$ need not be a twisted $r$-spin bundle. Instead,
if we take
\begin{equation}
\label{eq:S Shat def}
\widehat{S}=\NNN^*S\otimes\O\left(-\sum_{q\in\mathcal{R}} [q]\right),
\end{equation}
where
\begin{equation}
\label{eq:R closed}
\mathcal{R}:=\{\hbox{$q$ a half-node of $\widehat{C}$ which is of
multiplicity $0$ but not an anchor}\},
\end{equation}
then $\widehat{S}$ is a twisted $r$-spin bundle, with the set of
marked points where the spin structure is twisted being
\[
\widehat{I}_0=I_0\cup \mathcal{R}.
\]
The \emph{twists} of the half-nodes of $\widehat{C}$ are defined to be the twists of
$\widehat{S}$ at these points. We often use this notion for the twist of a half-node
viewed as a nodal point in $C$ along with a choice
of branch of $C$ through this node.

Note in analogy with \eqref{eq:I0}, we have
\begin{equation}
\label{eq:cR}
\hbox{The only half-nodes of $(\widehat{C},\widehat{S})$
of twist $-1$ are anchors.}
\end{equation}
For twisted $r$-spin structures, we will impose that $I_0$ consists of all markings of $L$-multiplicity zero except anchors.\footnote{Moving forwards, we require a rule for twisted $r$-spin structures in the open case. Here, closed $r$-spin curves are connected components of the normalization of some stable marked genus 0 orbifold Riemann surfaces with boundary. We exclude from $I_0$ the anchors that come from a contracted boundary (see Definition~\ref{def: open RS bdry}).}

\begin{definition}
A special point whose twist is $-1\mod r$ is called \emph{Ramond}, while the other markings are called \emph{Neveu-Schwarz} (with respect to $S$).
\end{definition}

Consider the coarsification map $\pi: C \rightarrow |C|$ and the pushforward $\pi_*L$. Here, we can see that locally at the marked point $p$, the sections of this sheaf will be given by $\mu_d$-invariant sections of $\O_{\mathbb{C}}$ with respect to the action in Equation~\eqref{eqn: local action}. The generator acts locally on a section $f$ by taking $\zeta\cdot(X, f(X)) = (\zeta X, \zeta^m f(X))$. For invariance, we need that $\zeta^m f(X) = f(\zeta X)$, which implies that $f(X) = X^m g(X^d)$ for some polynomial $g$. Note that
there is a local coordinate $x$ on $|C|$ with
$\pi^*(x) = X^d$, hence $\pi_*L$ is a line bundle.  We then have that $(\pi_*L)^{\otimes d}$ near $p$ is generated by $X^{md} = x^m$, which is used to prove the following proposition:

\begin{prop}\label{obs:twists_on_coarse}
For any connected component $C_l$ of $|\widehat{C}|,$ with markings also denoted by $\{z_i\}_{i\in I_l}$ and half-nodes $\{p_h\}_{h\in N_l},$ it holds that $|\widehat{S}|$ is a line bundle and
\begin{equation}\label{eq:nodal_curve_closed}\left(|\widehat{S}|\big|_{C_l}\right)^{\otimes r} \cong \omega_{|C_{l}|} \otimes  \O\left(-\sum_{i \in I_l} a_i [z_i] - \sum_{h \in N_l} c_h [p_h]\right),\end{equation}
\[a_i,c_h\in \{-1,0,\ldots, r-1\},\]
where $a_i,c_h$ are the twists.
\end{prop}
\begin{proof}
The corresponding claim for $|\NNN^* L|$ is well-known and can be seen through the argument in the preceding paragraph. 
Since $\widehat{S}$ differs from $\NNN^*L$ by twisting down at a divisor pulled back from the coarse curve, the twisting and coarsening commute which allows us to deduce the claim.
\end{proof}

We can also define the orbifold line bundle
\[
J:=\omega_C\otimes S^\vee
\]
on $C$, which has a natural pairing with $S$ given by
\begin{equation}
\label{eq:natural_pairing}
\langle -,-\rangle: S\otimes J \rightarrow \omega_C.
\end{equation}

We now require the notion of a \emph{grading}.

\begin{definition}[Grading]\label{def:closed grading}
A \emph{grading} of a twisted closed genus $0$ $r$-spin surface $(C,S)$ consists
of the following additional structure at each anchor $z_j$ with
twist $r-1$.
Recall that we have a map $\tau'$ defined as the composition
\[
\tau':\left(S\otimes \O\left([z_j]\right)\right)^{\otimes r}\big|_{z_j}\rightarrow \omega_{C,\log}\big|_{z_j}\rightarrow\C,
\]
where the second map is an isomorphism given by the residue map. Then
the data of a grading at $z_j$ is, first, the choice of
an $\R$-linear involution $\widetilde{\phi}$ on the fiber $\left(S\otimes \O\left([z_j]\right)\right)_{z_j}$ which is required to satisfy two properties:
\begin{enumerate}
\item For all $v\in\left(S\otimes \O\left([z_j]\right)\right)_{z_j}$,
we have $\tau'(\widetilde{\phi}(v)^{\otimes r})=-\overline{\tau'(v^{\otimes r})}$
where $w\mapsto\overline{w}$ is the standard conjugation.
\item We have
$\{\tau'(v^{\otimes r})\,|\,v\in \left(S\otimes \O\left([z_j]\right)\right)^{\widetilde\phi}_{z_j}\}\supseteq i\R_+$, where $i$ is the root of $-1$ in the upper half plane.
\end{enumerate}
Second, we choose a non-zero vector $v_j$ in the fiber $(S \otimes \O([z_j]))^{\tilde \phi}_{z_j}$ so that $\tau'(v_j^{\otimes r}) \in i\R_+$.
Two gradings $\{v_j\}, \{v_j'\}$ are \emph{equivalent} if $v_j=\lambda_j
v_j'$ for some collection of positive real numbers $\lambda_j$.

We call the connected component of $\left(S\otimes \O\left([z_j]\right)\right)^{\widetilde\phi}_{z_j}\setminus\{0\}$ containing $v_j$ the \emph{positive direction} on $\left(S\otimes \O\left([z_j]\right)\right)^{\widetilde\phi}_{z_j}$ and call any $v$ in that connected component \emph{positive}. Note that if $v$ is positive then $\tau'(v^{\otimes r})\in i\R_+$.
\end{definition}

Since
\[J_{z_j}\otimes \left(S\otimes \O\left([z_j]\right)\right)_{z_j}\simeq \omega_{C,\log}\big|_{z_j},\] $\widetilde\phi$ induces a conjugation on the fiber $J_{z_j},$ also denoted by $\widetilde\phi$, by the requirement that under the identification of $\omega_{C,\log}\big|_{z_j}$ with $\C$, we have
\[\langle \widetilde\phi(w),\widetilde\phi (v)\rangle = -\overline{\langle w, v\rangle}\]
for $w\in J_{z_j}$, $v\in \left(S\otimes \O\left([z_j]\right)\right)_{z_j}$.
Similarly, using the positive direction on $\left(S\otimes \O\left([z_j]\right)\right)^{\widetilde\phi}_{z_j}$, we define a positive direction on $J^{\widetilde\phi}_{z_j}$ in the following way.  We say that $w\in J_{z_j}$ is positive if, for any positive $v\in \left(S\otimes \O\left([z_j]\right)\right)_{z_j}$,  it holds that \[\langle w,v\rangle\in i\R_+.\]

\begin{rmk}\label{rmk: gradings and contracted boundaries}
Gradings arise naturally from the open $r$-spin disk case. Indeed,
as we shall see, such disks may acquire contracted boundary nodes,
i.e., nodes obtained by shrinking the boundary to a point, see Figure
\ref{fig:node_types}. The
stalk of the spin bundle at these points then naturally carry
an involution $\tilde\phi$ and the grading can be seen as the
limiting data of a lifting on the open disks in the sense of Definition
\ref{def:lifting_compatible},(2).
\end{rmk}

\begin{definition}
A \emph{closed genus $0$ graded $r$-spin surface} is a twisted closed genus $0$ $r$-spin surface all of whose connected components are anchored, together with a choice of  grading, up to equivalence. An {\it isomorphism} of closed genus $0$ graded $r$-spin surfaces is an isomorphism of closed genus $0$ twisted $r$-spin surfaces that preserves the anchors and, in the case where an anchor has twist $r-1$, also the involution $\widetilde{\phi}$ and the positive direction.
\end{definition}

\begin{rmk}
\label{rmk:anchor yoga2}
As already mentioned in Remark \ref{rmk:anchor yoga}, once we start to
normalize open $r$-spin surfaces, the
anchor of a sphere component should be thought of as the half-node
at which the component met a disk component, or met a simple
path of sphere components connecting it to a disk component. The additional
data in the grading in the case of a twist $r-1$ anchor arises as in Remark~\ref{rmk: gradings and contracted boundaries}. 
\end{rmk}

The properties of the twists and grading can be summarized in the following well-known observation below.
\begin{obs}\label{obs:closed_constraints}
It follows immediately from
Proposition~\ref{obs:twists_on_coarse} that if $C$ is a smooth $r$-spin curve,
then the degree of the line bundle $|S|\to|C|$ is
\begin{equation}\label{eq:Sr}
\frac{-2-\sum a_i}{r}\in \Z,
\end{equation} where $a_i$ are the twists. This number must be an integer.
Thus, also
\begin{equation}\label{eq:close_rank1_general}
\frac{\sum a_i -(r-2)}{r}\in \Z.
\end{equation}
If $C$ is a $d$-stable $r$-spin curve and
$p$ and $p'$ are the two branches of a node, then
\begin{equation}\label{eq:twists_at_half_nodes}c_p + c_{p'} \equiv r-2 \mod r.\end{equation}
Moreover, in the notation of \eqref{eq:nodal_curve_closed}, for any component $C_l$ of $|\widehat{C}|$
with markings also denoted by $\{z_i\}_{i\in I_l}$ and half-nodes $\{p_h\}_{h\in N_l},$ we have
\[r\Big|\sum_{i \in I_l} a_i + \sum_{h \in N_l} c_h+2-r.\]
When \eqref{eq:close_rank1_general} holds, there is, up to isomorphism,
a unique graded structure on $C$ with the given twists.
\end{obs}

\begin{rmk}
\label{rmk: automorphisms}
The only automorphisms of genus $0$ graded $r$-spin surfaces with injective marking $m$
are the \emph{ghost automorphisms}. There are two types of such automorphisms: (i)  scaling fiberwise the bundle $S$ and (ii)  ``rotating'' $C$ at a node. In type (i), we scale the bundle $S$ fiberwise by an $r^{th}$ root of unity. Note that we can do this independently on each of the connected components; however, as explained
in Proposition 2.15 of \cite{BCT:I}, the extra data of a grading at
anchors of twist $r-1$ eliminates the automorphism given by scaling
the spin bundle by an $r^{th}$ root of unity on the connected component
of $C$ containing the anchor. This provides a factor of $r^{|\pi_0(C)|-A}$ to the group of ghost automorphisms, where $|\pi_0(C)|$ is the number
of connected components of $C$ and $A$ is the number of anchors whose
twist is $r-1$. In type (ii), recall that the stabilizer group is $\mu_d$ at a node, so
this contributes a factor of $\mu_d$ to the automorphism group of $C$.
In terms of the local picture \eqref{eq:balanced}, the action of
$\mu_d$ is given by
\[\zeta\cdot(x,y)=(\zeta x,y),~~\zeta~\text{is an $d^{th}$ root of unity.}\]
One can check as in \cite{CZ} that the pull back of the spin bundle
$S$ via this automorphism gives an isomorphic spin structure, and hence
induces an automorphism. See also \cite{BCT:I}, Remark 2.12 and Proposition
2.15. However, in that reference, $d=r$. Consequently, we obtain a factor of $d^N$ where $N$ is the number of nodes. Putting together, the group of ghost automorphisms is of order $r^{|\pi_0(C)|-A}d^N$.
\end{rmk}

We now generalize to the $W$-spin case.
Let $W$ be any Fermat polynomial in $n$ variables,
\begin{equation}
\label{eq:Fermat W}
W(x_1,\ldots, x_a)=\sum_{i=1}^a x_i^{r_i},
\end{equation}
and let $d={\mathrm{lcm}}(r_1,\ldots,r_a)$.

\begin{definition}
A \emph{closed graded genus $0$ $W$-spin Riemann surface} is a tuple
\[
(C,S_1,\ldots, S_a;\tau_1,\ldots,\tau_a)
\]
where $C$ is $d$-stable marked genus $0$ orbifold Riemann surface, all of whose components are anchored, and for all $i\in[a]:=
\{1,\ldots,a\}$, $(S_i,\tau_i)$ is a graded twisted $r_i$-spin structure on $C$.

The \emph{twist} of a marking or a half-node is the $a$-tuple of its twists with respect to each $S_i$. We write $\tw(q)$ for the twist of $q,$ and $\tw_i(q)$ for its $i^{th}$ component, $i\in[a]$. A special point $q$ is \emph{broad} if it is Ramond with respect to one of the bundles $S_i,$ otherwise $q$ is \emph{narrow}.
\end{definition}

\begin{rmk}\label{rmk:NegTwist}
Recall from~\eqref{eq:cR} that the only half-nodes of twist $-1$ are anchors. In (closed) genus $0$ $W$-spin theory, it is possible to allow at most one marked point with twist $-1$ for each $r_i$-spin bundle, $i=1,\ldots,a$. This is observed in \cite{JKV2}, and studied in detail in \cite{BCT_Closed_Extended} for a single spin bundle. So long as all other marked points have non-negative twist, the degree of the spin bundle is negative which implies that the Witten bundle defined in Section~\ref{subsec:WittenBundle} is indeed a bundle.
\end{rmk}

\begin{obs}\label{obs:automorphism}
The number of automorphisms is $d^N\prod_{i=1}^a r_i^{|\pi_0(C)|-A_i}$, where $A_i$ is the number of anchors at which the bundle $S_i$ has twist $r_i -1$
and $N$ is the number of nodes.
\end{obs}

\subsection{Open $r$-spin structures}\label{subsec:open objects}

\begin{definition}\label{def: open RS bdry}
\begin{enumerate}
\item
A {\it connected marked genus $0$ orbifold Riemann surface with boundary} or \emph{disk} is a tuple
\[(C, \phi, \Sigma, \{z_i\}_{i \in I}, \{x_j\}_{j \in B}, m^I, m^B),\]
in which:
\begin{enumerate}[(i)]
\item $C$ is a connected genus $0$ orbifold Riemann surface.
\item $\phi: C \rightarrow C$ is an anti-holomorphic involution ({\it conjugation}).
\item $\Sigma$ is a fundamental domain for the induced conjugation
on $|C|$. We also write $\phi:|C|\rightarrow |C|$ for this induced
conjugation. Here $\Sigma$ is a Riemann surface, usually, but not
always, with boundary.
\item $z_i \in C$ are a collection of distinct points
 (the {\it internal marked points}) labeled by the set $I$, whose images in $|C|$ lie in $\Sigma \setminus \d\Sigma$, with {\it conjugate marked points} $\bar{z}_i:= \phi(z_i)$.
\item $x_j \in \text{Fix}(\phi)$ are a collection of distinct points (the {\it boundary marked points}) labeled by the set $B$, whose images in $|C|$ lie in $\d\Sigma$.
\item $m^I:I\to\Universe$ and $m^B:B\to\Universe$ are markings.
\end{enumerate}
A typical such surface with the various marked points labeled can be seen in Figure~\ref{fig:smooth_orbRS_boundary}. We shall often denote the object only by $C$ or by the preferred half $\Sigma$.

Note there can be at most one isolated $\phi$-invariant point. If it exists, it is called a \emph{contracted boundary node} and it is regarded as an \emph{anchor}. In this case, $\Sigma$ has no boundary. In all other cases, $\Sigma$ is a \emph{nodal marked disk} and has no anchor. Those $\phi$-invariant nodes which are not a contracted boundary are called \emph{boundary nodes}. The remaining nodes whose image under the coarsification map falls in the interior of $\Sigma$ are the \emph{internal nodes}. Their conjugate points are the \emph{conjugate nodes}. We call $\partial\Sigma$ the \emph{boundary of the open surface}. It is oriented as the boundary of $\Sigma$. See Figure \ref{fig:node_types}.
\item
A {\it marked genus $0$ orbifold Riemann surface with boundary} is a finite union of connected marked genus $0$ orbifold Riemann surfaces with boundary and connected marked genus $0$ orbifold Riemann surfaces without boundary. It is \emph{smooth} if there are no nodes. See Figure~\ref{fig:smooth_orbRS_boundary}.

\item
A marked genus $0$ orbifold Riemann surface $C$ (with or without boundary) is \emph{stable} if each connected component of the normalization has at least three marked points. We say that $C$ is \emph{partially stable} if $C$ is connected and open, and either (i) $|I|=1$ and $|B|=0$ or (ii) $|I|=0$ and $|B| = 2$.
\item
We usually suppress $\phi$ from the notation and write $\overline{w}$ for $\phi(w)$ when $w$ lies in the preimage of $\Sigma$ in $C$. In what follows we usually identify a special point with its image in $\Sigma$.
\item
An {\it isomorphism} of marked orbifold Riemann surfaces with boundary
\[(C_1, \phi_1, \Sigma_1, \{z_{1,i}\}_{i \in I_1}, \{x_{1,j}\}_{j \in B_1}, m_1^{I_1}, m_1^{B_1}) \cong (C_2, \phi_2, \Sigma_2, \{z_{2,i}\}_{i \in I_2}, \{x_{2,j}\}_{j \in B_2}, m_2^{I_2}, m_2^{B_2})\]
consists of an isomorphism $s: C_1\rightarrow C_2$ and bijections $f^I: I_1 \rightarrow I_2$ and $f^B: B_1 \rightarrow B_2$ such that
\begin{enumerate}[(i)]
\item $s \circ \phi_1 = \phi_2 \circ s$,
\item $s(\Sigma_1) = \Sigma_2$,
\item $s(x_{1,j}) = x_{2,f^B(j)}$ for all $j \in B_1$ and $s(z_{1,i}) = z_{2,f^I(i)}$ for all $i \in I_1$, and
\item $m_1^{I_1} = m_2^{I_2} \circ f^I$ and $m_1^{B_1} = m_2^{B_2} \circ f^B$.
\end{enumerate}
Note that one can write
\[|C| = \left(\Sigma \coprod_{\d \Sigma} \overline{\Sigma}\right)/\sim_{CB},\]
where $\bar{\Sigma}$ is obtained from $\Sigma$ by reversing the complex structure, and $\sim_{CB}$ is the equivalence relation on $\Sigma \cup_{\d \Sigma} \overline{\Sigma}$ defined by $x\sim_{CB}y$ precisely if $x=\phi(y)$ and one of them is a contracted boundary node.
\end{enumerate}
\end{definition}

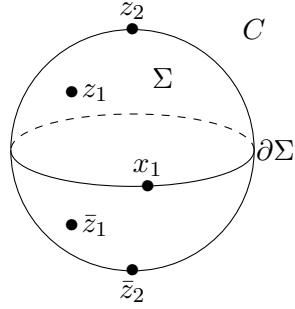
\begin{figure}
  \centering

\begin{tikzpicture}[scale=0.8]
  \draw (0,0) circle (2cm);
  \draw (-2,0) arc (180:360:2 and 0.6);
  \draw[dashed] (2,0) arc (0:180:2 and 0.6);

 \node at (2, 2) {$C$};
 \node at (2.35,0) {$\partial\Sigma$};
\node (c) at (-1,.95) {$\bullet$};
\node[right] at (-1,.95) {$z_1$};

\node at (-1,-1.25) {$\bullet$};
\node[right] at (-1,-1.25) {$\bar z_1$};

\node at (.25, -.6) {$\bullet$};
\node[above] at (.25,-.6) {$x_1$};

\node (a) at (.5, 1.2) { $\Sigma$};
\node at (0,2) {$\bullet$};
\node[above] at (0,2) {$z_2$};
\node (c) at (0, -2) {$\bullet$};
\node[below] at (0,-2) {$\bar z_2$};
\end{tikzpicture}

  \caption{A smooth connected marked genus 0 orbifold Riemann surface with boundary with two internal markings $z_1, z_2$ and a boundary marking $x_1$.}
  \label{fig:smooth_orbRS_boundary}
\end{figure}

We next take care in defining normalization of open spin surfaces.
As already mentioned in Remarks \ref{rmk:anchor yoga} and \ref{rmk:anchor yoga2},
anchors play a key role in bookkeeping in the normalization process.
Precisely:

\begin{definition}
\label{def:open normal}
Let $(C, \phi, \Sigma, \{z_i\}_{i \in I}, \{x_j\}_{j \in B}, m^I, m^B)$
be a marked genus $0$ orbifold Riemann surface with boundary.
The \emph{partial normalization} $\NNN_q:C'\to C$ at the node $q$ is defined by
taking the usual partial normalization $\NNN_q:C''\rightarrow C$ at $q$ and
its conjugate node
$\phi(q)$ and then taking $C'\subseteq C''$ to be the union of
connected components whose image in $|C|$ intersects $\Sigma$ in more
than one point. For example, in the examples in
Figure \ref{fig:node_types}, the light gray spheres are discarded in
cases (A) and (C), if the normalization occurs at the internal node
and the contracted boundary node respectively. However, nothing is
discarded if one normalizes (B) at the boundary node.

We then define $\NNN_q:C'\rightarrow C$ to be $\NNN_q|_{C'}$. In the normalization, we label the resulting half-nodes and choose
anchors, generalizing Definition \ref{def:closed normal}, as follows:
\begin{enumerate}
\item If $q$ is an internal node, suppose without loss of generality that its image in $|C|$ belongs to $\Sigma$. In this case $C'$ will have two connected components: one component will be closed and one component will be open or have
a contracted boundary node. The preimage $q_c$ of $q$ in the closed component is declared to be the anchor, and is marked by $\emptyset\in\Universe$. The other half-node is marked by the union of the labels of the remaining points in the component of $q_c$.
\item If $q$ is a boundary node, then we label each of its half-nodes by $\emptyset
\in \Universe$. Note that both connected components after the partial normalization are still nodal marked disks and consequently do not have anchors.
\item If $q$ is a contracted boundary node, we declare its preimage in $C'$
to be an anchor, and we mark it by $\emptyset\in\Universe$.
\end{enumerate}
\begin{figure}
\centering
\begin{subfigure}{.3\textwidth}
  \centering

\begin{tikzpicture}[scale=0.3]
  \shade[ball color = gray, opacity = 0.5] (0,0) circle (2cm);
  \draw (0,0) circle (2cm);
  \draw (-2,0) arc (180:360:2 and 0.6);
  \draw[dashed] (2,0) arc (0:180:2 and 0.6);

  \shade[ball color = gray, opacity = 0.1] (0,-4) circle (2cm);
  \draw (0,-4) circle (2cm);
  \draw (-2,-4) arc (180:360:2 and 0.6);
  \draw[dashed] (2,-4) arc (0:180:2 and 0.6);
   \shade[ball color = gray, opacity = 0.6] (-2,-4) arc (180:360:2 and 0.6) arc (0:180:2);

    \shade[ball color = gray, opacity = 0.1] (0,-8) circle (2cm);
  \draw (0,-8) circle (2cm);
  \draw (-2,-8) arc (180:360:2 and 0.6);
  \draw[dashed] (2,-8) arc (0:180:2 and 0.6);
\end{tikzpicture}

  \caption{Internal node}
\end{subfigure}
\begin{subfigure}{.3\textwidth}
  \centering
\vspace{0.9cm}
\begin{tikzpicture}[scale=0.4]
  \shade[ball color = gray, opacity = 0.1] (0,0) circle (2cm);
  \draw (0,0) circle (2cm);
  \draw (-2,0) arc (180:360:2 and 0.6);
  \draw[dashed] (2,0) arc (0:180:2 and 0.6);
  \shade[ball color = gray, opacity = 0.6] (-2,0) arc (180:360:2 and 0.6) arc (0:180:2);

  \shade[ball color = gray, opacity = 0.1] (4,0) circle (2cm);
  \draw (4,0) circle (2cm);
  \draw (2,0) arc (180:360:2 and 0.6);
  \draw[dashed] (6,0) arc (0:180:2 and 0.6);
  \shade[ball color = gray, opacity = 0.6] (2,0) arc (180:360:2 and 0.6) arc (0:180:2);
\end{tikzpicture}
\vspace{0.9cm}

  \caption{Boundary node}
\end{subfigure}
\begin{subfigure}{.3\textwidth}
  \centering

\begin{tikzpicture}[scale=0.4]
\vspace{0.15cm}
  \shade[ball color = gray, opacity = 0.6] (0,0) circle (2cm);
  \draw (0,0) circle (2cm);
  \draw (-2,0) arc (180:360:2 and 0.6);
  \draw[dashed] (2,0) arc (0:180:2 and 0.6);

  \shade[ball color = gray, opacity = 0.1] (0,-4) circle (2cm);
  \draw (0,-4) circle (2cm);
  \draw (-2,-4) arc (180:360:2 and 0.6);
  \draw[dashed] (2,-4) arc (0:180:2 and 0.6);
\end{tikzpicture}
\vspace{0.15cm}

  \caption{Contracted boundary}
\end{subfigure}

\caption{The closed curve $|C|$, with the nodal disk $\Sigma$ shaded.}
\label{fig:node_types}
\end{figure}
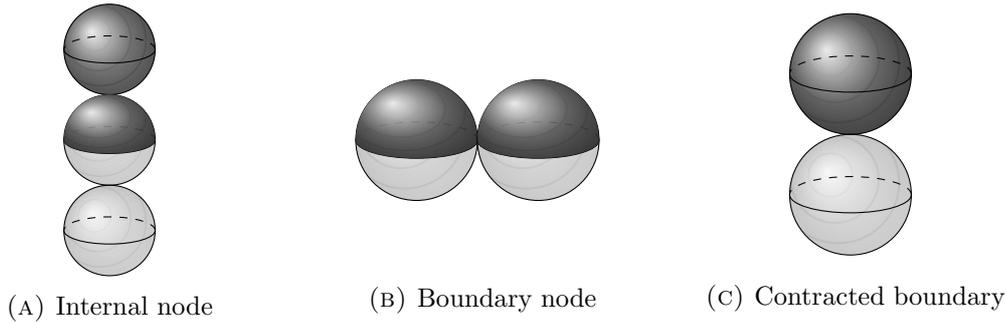

The normalization $\NNN:\widehat{C}\to C$ for a marked genus 0 orbifold Riemann surface with boundary $C$ is similarly defined by iterating the partial normalization definition over all nodes.
\end{definition}

\begin{definition}
Let $C$ be a $d$-stable marked genus 0 orbifold Riemann surface with boundary. Let $r$ divide $d$. An \emph{$r$-spin} (resp. \emph{twisted $r$-spin}) structure on $C$ is an $r$-spin structure $L$ (resp. twisted $r$-spin structure $S$) on $C$ together with an involution  $\widetilde{\phi}$ on $L$ (resp. $S$) lifting
$\phi$ and such that $\widetilde{\phi}^{\otimes r}$ agrees under $\tau$
with the involution of $\omega_{C,\log}$ induced by $\phi$.
\end{definition}

Multiplicities of marked points and half-nodes, as well as twists of marked points, are defined as in the closed case. Note that this implies that, for any $r$-spin structure on a $d$-stable marked genus 0 orbifold Riemann surface with boundary, the multiplicities and twists at an internal marked point $z_i$ (called the \emph{internal twist} $a_i$) must be the same at its conjugate marked point $\bar z_i$. We note that we analogously call the twist $b_i$ at a boundary special point the \emph{boundary twist}.

\begin{obs}
\label{obs:open_rank1_general}
There is a constraint on the allowable twists of the marked points,
arising from Observation \ref{obs:closed_constraints}.  Indeed,
let $S$ be a twisted $r$-spin structure with internal twists $a_i$ and boundary twists $b_j$. Then
\begin{equation}\label{eq:open_rank1_general}
\frac{2\sum_{i\in I} a_i + \sum_{j\in B} b_j-(r-2)}{r}\in \Z.
\end{equation}
\end{obs}

As in the closed case, $\NNN^*L\to\widehat{C}$ is itself an $r$-spin structure. We define the twisted $r$-spin structure $\widehat{S}\to\widehat{C}$ as
\begin{equation} \label{eq: normalised twisted spin structure}
\widehat{S}=\NNN^*S\otimes \O\left(-\sum_{q\in\mathcal{R}} [q]\right).
\end{equation}
Here
\begin{align}
\label{eq:R open}
\begin{split}
\mathcal{R}:=&\{\hbox{$q$ is a contracted boundary node of multiplicity $0$}\}\\
&\cup\{\hbox{$q$ is a half-node of $\widehat{C}$ which is of
multiplicity $0$ but not an anchor}\}.
\end{split}
\end{align}
We recall that $\O([q])$ denotes the pull-back of the corresponding
line bundle from $|\widehat C|$.
The twists of half-nodes with respect to $S$ are defined as the twists of the corresponding half-nodes with respect to $\widehat{S}$. The notions of Ramond and Neveu-Schwarz points are as in the closed case.

We remark that \eqref{eq:open_rank1_general} also induces constraints
on the twists of half-nodes. Indeed, when applied to a connected
component of the normalization $\widehat{C}$, one must sum over twists
of the half-nodes as well as the twists of the marked points.
Note that the definition of $\mathcal{R}$ determines the twists of
all half-nodes.

\begin{rmk}
\label{rem:CB Ramond}
In the course of the proof of \cite{BCT:I}, Proposition 2.15, it is shown that
every contracted boundary node is Ramond.
\end{rmk}

\subsubsection{Lifting and alternations}
\begin{definition}\label{def:lifting_compatible}
Let $S$ be a twisted $r$-spin structure on a connected marked genus 0 orbifold Riemann surface with boundary $C$.

\begin{enumerate}

\item Suppose that $C$ has no contracted boundary node. Let $A$ be the complement in $C^\phi$ of the set of special points. In this case a \emph{lifting of $S$} is a choice (if it exists) of an orientation for the real line bundle
\[S^{\widetilde\phi}\to A\]
 satisfying the following property.  If we take any vector $v \in S^{\widetilde\phi}_a$ at any $a\in A$ such that $v$ is positive with respect to the orientation specified by the choice of lifting, then its image under the map $v \mapsto \tau(v^{\otimes r})$ is positive with respect to the natural orientation of $\omega_{C}^{\phi}|_{A}$. Two sections $v$ and $v'$ of $S|_A$ are \emph{equivalent} if $v = c v'$ for a continuous function $c: A \rightarrow \R^+$. We write $[v]$ for the equivalence class of $v$.  When a section $v$ is positive with respect to the lifting, we denote a lifting of $S$ over $A$ by $[v]$.

A {\it lifting of $J$ over $A$} is an orientation of the real line bundle
\[  J^{\widetilde\phi}\to A,\]
 satisfying the following property. If we take any vector $w\in J^{\widetilde\phi}_a$ at any $a\in A$ such that $w$ is positive with respect to the orientation, then there exists a vector $v\in S^{\widetilde\phi}_a$ with $\tau(v^{\otimes r})$ and $\langle w,v\rangle$ positive with respect to the natural orientation of $\omega_{C}^{\phi}|_{A}$, where $\langle-,-\rangle$ is as in \eqref{eq:natural_pairing}. We similarly define an equivalence class $[w]$ for sections of $J$ over $A$ as above.

 \item Suppose that $C$ has a contracted boundary node $q$. In this case a \emph{lifting of $S$} is a choice (if it exists) of an orientation for the real line $(S_q)^{\tilde{\phi}}$
 satisfying a different set of properties as follows. 
The fiber $\omega_{C,q}$ is canonically identified with $\C$ via the residue, and the involution $\phi$ is sent, under this identification, to the involution $z\to -\bar{z},$ whose fixed points are the purely imaginary numbers.\footnote{The residue of a conjugation invariant form $\zeta$ can be calculated as $\frac{1}{2\pi i} \oint_L\zeta $ where $L\subset\Sigma$ is a small loop surrounding $q$ whose orientation is such that $q$ is to the left of $L$. The behaviour under conjugation shows that the residue of an invariant section is purely imaginary.} 
The choice of orientation must satisfy the property that for any vector $v\in S_q^{\tilde\phi}$ positive with respect to the orientation, the image of $v^{\otimes r}$ under the map
\[
S^{\otimes r}_q \rightarrow \omega_{C,q}
\]
is \emph{positive imaginary}, meaning that it lies in $i\R_+$. We then write a lifting of $S$ at $q$ as an equivalence class $[v]$ of such positive elements $v$ under the equivalence relation of multiplication by a positive real number. In the contracted boundary case,
there always exists a $\widetilde{\phi}$-invariant $w \in J|_q$ such that $\langle w , v\rangle$ is positive imaginary, and we refer to such $w$ as a {\it grading at $q$}.
 \end{enumerate}

We say that a twisted $r$-spin structure $S$ on a marked genus 0 orbifold Riemann surface with boundary $C$ is \emph{compatible} if each connected component of $C$ has a lifting. A \emph{lifting on $C$} is a choice of lifting on each connected component.
\end{definition}

\begin{rmk}
Consider the case where $C$ has a contracted boundary node. Recall by Remark~\ref{rem:CB Ramond} that $q$ is Ramond and by Observation~\ref{obs:open_rank1_general} that the twist at $q$ on the partial normalization at $q$ will be $r-1$. Then the choice of a lifting $q$ will induce a grading (Definition~\ref{def:closed grading}) on the closed genus 0 $r$-spin surface given by the partial normalization at $q$, as $q$ will become the anchor (Definition~\ref{def:open normal}(3)).
\end{rmk} 

In the case where there are no contracted boundaries, observe that a twisted $r$-spin structure admits a lifting of $S$ over $A$ precisely if it admits a lifting of $J$ over $A$.  Moreover, there is a bijection between equivalence classes of liftings of $S$ and of $J$, in which $[v]$ corresponds to $[w]$ if $\langle w,v \rangle$ is everywhere positive for all representatives $v$ and $w$ of $[v]$ and $[w]$.

Now, a lifting of $J$ induces a lifting of $\widehat{J}:=\omega_{\widehat{C}}\otimes \widehat{S}^\vee$. Take $q$ to be a boundary marked point or boundary half-node of $\widehat{C}$. Then let $U_q$ be some neighborhood of $|\NNN^{-1}(q)|$ in $\partial \widehat{\Sigma}$, where $\widehat{\Sigma}$ is the normalisation of $\Sigma$. The lifting of $J$ induces an orientation for the real line bundle $|\widehat{J}|^{\widetilde{\phi}}|_{U_q \setminus|\nu^{-1}(q)|}$. Here, we use the notation $\widetilde\phi$ also for the involutions on $\widehat{J}$ and its coarsening $|\widehat{J}|$.

\begin{definition}[Alternating nodes or marked boundary points]
\label{def:alt}
Let $S$ be a compatible twisted $r$-spin structure on a connected marked genus $0$ orbifold Riemann surface with boundary. Suppose further that $C$ has no contracted boundary node. Take $q$ to be a boundary marked point or boundary half-node of $\widehat{C}$ and take $U_q$ to be some neighborhood of $|\NNN^{-1}(q)|$ as above. We say that the \emph{lifting alternates in $q$} and that $q$ is \emph{alternating} with respect to the lifting if the orientation \emph{cannot be extended} to an orientation of $|\widehat{J}|^{\widetilde\phi}\big|_{U_q}$. Otherwise the lifting \emph{does not alternate} or is \emph{non-alternating} and $q$ is \emph{non-alternating}.

We write $\alt(q)=1$ if the lifting alternates in $q$ and $\alt(q)=0$ if it does not alternate.
\end{definition}

\begin{rmk}
In \cite{BCT:I}, the terms \emph{legal} and \emph{illegal} are used
instead of alternating and non-alternating, respectively.
\end{rmk}

We summarize the properties of $r$-spin structures with a lifting in the following propositions. We start with Proposition 2.5 of \cite{BCT:I}:

\begin{prop}\label{prop:graded_r_spin_prop}
Suppose that $(C,S)$ is a smooth connected genus $0$ twisted $r$-spin Riemann surface with boundary.
\begin{enumerate}
\item\label{it:compatibility_odd} When $r$ is odd, any twisted $r$-spin structure is compatible, and there is a unique equivalence class of liftings.
\item\label{it:compatibility_even} When $r$ is even, the boundary twists $b_j$ in a compatible twisted $r$-spin structure must be even.  Whenever the boundary twists are even, either the $r$-spin structure is compatible or it becomes compatible after replacing $\widetilde\phi$ by $\xi \circ \widetilde\phi \circ \xi^{-1}$ for $\xi$ an $r^{th}$ root of $-1$, which yields an isomorphic $r$-spin structure.
\item\label{it:lifting and parity_odd} Suppose $r$ is odd and $v$ is a lifting. Then $x_j$ is alternating if and only if its twist is odd.
\item\label{it:lifting and parity_even}
Suppose $r$ is even.  If a lifting alternates precisely at a subset $D \subset \{x_j\}_{j \in B}$, then
\begin{equation}
\label{parity}
\frac{2\sum a_i + \sum b_j+2}{r} \equiv |D| \mod 2.
\end{equation} If~\eqref{parity} holds, then there exist exactly two liftings up to equivalence, one the negative of the other, which alternate precisely at $D \subset \{x_j\}_{j \in B}$.
\end{enumerate}
\end{prop}

In particular, in the nodal case the analogous claims hold for the connected components of $(\widehat{C},\widehat{S})$.

The following is a combination of \cite{BCT:I}, Propositions 2.11 and 2.15,
with a slight change to the number of automorphisms as indicated in
Remark \ref{rmk: automorphisms}.

\begin{prop}\label{prop:existence_stable}
Suppose $(C,\phi,\Sigma, \{z_1, \ldots, z_l\}, \{x_1, \ldots, x_k\})$ is a connected genus $0$ marked orbifold Riemann surface with boundary, and
\[
a_1,\ldots,a_l \in \{0, 1, \ldots, r-1\},\quad b_1,\ldots,b_k\in
\{0,1,\ldots,r-2\}
\]
are such that \eqref{eq:open_rank1_general} holds. If $r$ is even, suppose the $b_i$ are even for all $i$ and let $D\subseteq[k]$ be an arbitrary set for which \eqref{parity} holds. If $r$ is odd, let $D=\{i\in[k]\,|\,~2\nmid b_i\}$. Then, up to isomorphism,
there exists a unique twisted $r$-spin structure with a lifting on this disk with internal twists given by the integers $a_i,$ boundary twists by the integers $b_j$, and $D$ being the set of alternating boundary marked points. The automorphism group of the $r$-spin disk with lifting is of size $d^{N}$,
where $N$ is the number of internal nodes.
\end{prop}

\begin{rmk}
Note that if $r$ is even, the two choices of liftings
in Proposition \ref{prop:graded_r_spin_prop} are related by an isomorphism
which is the identity on the underlying Riemann surface and rescales the
spin bundle by $-1$.
\end{rmk}

\begin{definition}\label{def:open_W_graded}
Consider a Fermat polynomial $W$ as in \eqref{eq:Fermat W}, with $d=\mathrm{lcm}(r_1,\ldots,r_a)$.
A \emph{$W$-spin Riemann surface with boundary} is a tuple
\[(C,S_1,\ldots, S_a;\tau_1,\ldots,\tau_a;\phi_1,\ldots,\phi_a)\]
where
\begin{enumerate}
\item $C$ is an orbifold marked Riemann surface with boundary such that:
\begin{enumerate}
\item
all of the special points of $C$ have isotropy group $\mu_d$;
\item all the closed connected components of $C$ are anchored;
\end{enumerate}
\item each $(S_i,\tau_i,\phi_i)$ are twisted $r_i$-spin structures.
\end{enumerate}
The $W$-spin structure on $C$ is \emph{compatible} if each $r_i$-spin structure is.

A \emph{lifting} for a $W$-spin Riemann surface with boundary is an $a$-tuple $([v_1],\ldots,[v_a])$ of liftings for each $J_i=\omega_C\otimes S_i^\vee$ (or equivalently liftings of $S_i$). For any boundary special marking or half-node $q$ we write $\alt(q)$ for the $a$-tuple whose $i^{th}$ component is $\alt_i(q),$ the alternation with respect to the $i^{th}$ lifting.
\end{definition}

We now consider restrictions on the behaviour of boundary marked points.
The kinds of open surfaces considered below will be the ones for which
we define open $W$-spin invariants, as Euler classes of certain bundles
on their moduli spaces:

\begin{definition}
\label{def:open_W_graded2}
\begin{enumerate}
\item
A \emph{connected pre-graded open $W$-spin surface} is a connected, compatible, $W$-spin Riemann surface with boundary, together with a lifting which satisfies
\begin{enumerate}
\item
For any boundary marking $p$ and any $i\in[a]$, either \[\tw_i(p)=r_i-2,\alt_i(p)=1~~\text{or }\tw_i(p)=\alt_i(p)=0.\]
\item
If $q$ is a boundary node and $i\in [a]$ where $q$ is Neveu-Schwarz
with respect to $S_i$, then the lifting alternates at precisely one of
the half-nodes of $q$.
\end{enumerate}
We call the lifting the \emph{grading}.
\item
A \emph{connected graded $W$-spin surface} is a connected pre-graded
$W$-spin surface such that in addition
there is no boundary marking $p$ with $\tw(p)=\alt(p)=\vec{0}$.
For $T\subseteq[a]$ denote by $k_T$ the number of boundary points $p$ with
\[(\tw_i(p),\alt_i(p))=(r_i-2,1)\Leftrightarrow i\in T.\] $k_{\{j\}}$ is also denoted by $k_j$.
\item A boundary marked point $p$ of a graded connected $W$-spin surface is \emph{singly twisted} if there exists an $i\in[a]$ such that
$\tw_j(p)=(r_j-2)\delta_{ij}$. In this case $p$ is alternating only for the $i^{th}$ bundle.
\item
A \emph{rooted connected $W$-spin surface} is a graded connected $W$-spin
surface such that all boundary marked points but one are singly twisted, and the
remaining point, called the \emph{root}, is \emph{fully twisted}, i.e., $\tw(p)=(r_1-2,\ldots, r_a-2)$.
\end{enumerate}
\end{definition}

\begin{rmk}\label{remark on twist restriction}
The imposition of the twists and alternation dictated in Definition~\ref{def:open_W_graded2} is required to define open $W$-spin invariants in the way we do below. 
The existence of a lifting was used in \cite{BCT:I} to define canonical relative orientation for the Witten bundle. It was used further in \cite{BCT:II} to define the notion of \emph{positivity} for boundary conditions of the Witten bundle, 
which is a crucial condition in their open $r$-spin construction (and our generalization), see \textsection\ref{subsec: positivity}. It is nontrivial to ensure that multisections which satisfy the positivity boundary conditions even exist. For this one needs in the $r$-spin case the boundary twists to be $r-2$, see \cite[Proposition 3.20]{BCT:II}. 

Moreover, by ensuring positivity, the graphs that will be important to define open $W$-spin intersection numbers will be precisely the pre-graded ones and their degenerations.  The boundary strata of the moduli space corresponding to \emph{relevant} graphs (see Definition~\ref{def:special kind of graded graphs}) will correspond to pre-graded open $W$-spin surfaces, after normalization. Here, we will be able to forget the (untwisted) boundary markings $p$ with $\tw(p)=\alt(p)=\vec{0}$, and build boundary conditions recursively from lower-dimensional graded $W$-spin surfaces, whose boundary twists are again as above. This is used to establish open topological recursion \cite[Theorem 4.1]{BCT:II}, which can be thought of as an open $r$-spin, descendent, analog of Solomon's Open WDVV \cite{Sol07}. 

Algebraically, the enumerative results from imposing the choice of boundary twist being $r-2$ and the modulo 2 condition fit naturally into other enumerative structures, including the nonzero coefficients of the Gelfand-Dickey wave function \cite{BCT:II} or of the perturbation of the mirror superpotential \cite{GKTdim1}.

Recent work \cite{TZ1,TZ2} constructs a new collection of open $r$-spin theories which still require a lifting, but allow more general boundary twists. It is not yet known how these theories are related to mirror symmetry.
\end{rmk}

\begin{rmk}
The grading at a contracted boundary, defined in Definition \ref{def:lifting_compatible}, is a limit case of the grading in Definition~\ref{def:open_W_graded2} that is obtained when there are no boundary markings or nodes and the boundary itself contracts.
\end{rmk}
\begin{rmk}
In order to obtain compact moduli spaces of connected
pre-graded open $W$-spin surfaces, we need to know that
given a degeneration of a smooth connected pre-graded
$W$-spin surface, condition (1)(b) in the above definition
holds for the degenerate surface.
When $r$ is odd, it follows immediately from \eqref{eq:twists_at_half_nodes}
and Proposition \ref{prop:graded_r_spin_prop}, (3) that
condition (1)(b) holds for any compatible
$W$-spin surface with boundary.
For $r$ even, a priori we have that $(2\sum_i a_i + \sum_j b_j + 2)/ r \equiv |D| \pmod 2$ as in Proposition \ref{prop:graded_r_spin_prop}, (4).
When we degenerate, we obtain twists $t_1$ and $t_2$ at a newly
acquired boundary node. If we normalize at the node, the corresponding
half-nodes lie in different
connected components of the normalization. Then the internal and boundary marked points are partitioned into the two different components. Thus we have
partitions $I = I_1 \sqcup I_2$ and $B = B_1 \sqcup B_2$. We then have
\begin{align*}
&\frac{2\sum_{i \in I_1} a_i + \sum_{j \in B_1} b_j + t_1 + 2 }{r} + \frac{2\sum_{i \in I_2} a_i + \sum_{j \in B_2} b_j + t_2 + 2 }{r}\\
= {} & \frac{2\sum_{i \in I} a_i + \sum_{j \in B} b_j + 2 + t_1 +t_2 + 2 }{r}\equiv |D| + 1 \pmod 2,
\end{align*}
as $t_1+t_2\equiv r-2 \pmod 2$.
Thus, if the two half-nodes are either both non-alternating or both alternating,
we will obtain a contradiction.
Hence one must be alternating and the other non-alternating.
\end{rmk}
In the disconnected case, we define:

\begin{definition}
A \emph{pre-graded  genus $0$ $W$-spin surface} is the disjoint union of several closed graded and open connected pre-graded $W$-spin surfaces. It is \emph{graded/open/smooth/rooted} if all components are.
\end{definition}

\begin{rmk}
In \cite{ST1} the case of graded $2$-spin surfaces is considered. There, another equivalent definition of open graded surfaces is given purely in terms of the half $\Sigma,$ considered as an \emph{orbifold surface with corners}. This definition is shown to be equivalent to the definition in terms of the doubled surface $C$. We shall not take this path, but we should note that the definition in the language of orbifold surfaces with corners can be generalized to the general $r$-spin, or even $W$-spin case.
\end{rmk}
By combining Observation~\ref{obs:open_rank1_general} and Proposition \ref{prop:graded_r_spin_prop}, we obtain the following.
\begin{obs}
\label{obs:open_rank_pre_graded_W}
Suppose given a smooth pre-graded $W$-spin genus $0$ marked orbifold Riemann
surface with boundary with:
\begin{itemize}
\item  internal marked points $z_1,\ldots,z_l$ with
twist
$\vec{a}_j=(a_{j,1},\ldots,a_{j,a}),~j\in[l]$;
\item the numbers $k_T,~T\subseteq[a]$ governing the twists of the
boundary points.
\end{itemize}
Then the following congruence conditions are satisfied:
\begin{equation}\label{eq:open_rank1_W}
e_i:=\frac{2\sum_j a_{j,i} + (-1+\sum_{T: i\in T}k_T )(r_i-2)}{r_i}\in \Z,~~i=1,\ldots,a,
\end{equation}
\begin{equation}\label{eq:open_rank2_W}
e_i\equiv -1+\sum_{T: i\in T}k_T~(\text{mod}~2).
\end{equation}
Here \eqref{eq:open_rank2_W} follows from \eqref{eq:open_rank1_W} if
$r_i$ is odd and from \eqref{parity} if $r_i$ is even.
\end{obs}

We now define an equivalence relation on graded genus $0$ $W$-spin surfaces
which will simplify some definitions and computations in what follows.

\begin{definition}\label{def:SIMX}
Let $C$ be a graded genus $0$ $W$-spin surface.
Suppose $C_1,\ldots,C_m$ are a collection of irreducible components of $C$
such that each $C_i$ is a disk with three special points, namely a boundary
 half-node and two boundary marked points $x_i,x'_i$, with
\[
\tw_k(x_i)=(r_k-2)\mathbf{1}_{k\in T_i},~\alt_k(x_i)=\mathbf{1}_{k\in T_i}\quad \tw_k(x'_i)=(r_k-2)\mathbf{1}_{k\in T'_i},~\alt_k(x'_i)=\mathbf{1}_{k\in T'_i},
\]
for some arbitrary disjoint sets $T_i,T'_i\subseteq[a]$. We obtain a new
graded $W$-spin curve $C'$ by replacing
$C_i$ with the disk $C_i'$ which simply transposes $x_i$ and $x_i'$
in the cyclic order (see Figure~\ref{fig:exchange}). In this case, we write $C\SIMX C'$, and this
is an equivalence relation on graded genus $0$ $W$-spin disks.

Intuitively, equivalence classes of this relation can be thought of as graded
genus $0$ $W$-spin surfaces
whose representatives share pairs of boundary points which have collided.
Equivalently, $C'$ is being obtained from $C$ by replacing
each $C_i$ with its complex conjugate, thus reversing the boundary orientation.
We then use Proposition \ref{prop:graded_r_spin_prop} to possibly modify
the involution on the spin bundles to obtain a grading.
\end{definition}

\begin{figure}
\begin{subfigure}{.45\textwidth}
  \centering
  \begin{tikzpicture}[scale=0.6]

  \draw (0,0) circle (2cm);
  \draw (4,0) circle (2cm);

\node at (1.35, 0) {$n_h$};
\node at (2.9, 0) {$n_{\sigma_1 h}$};
\node at (2,0) {$\bullet$};
\node at (-1, 1.732) {$\bullet$};
\node at (-1, -1.732) {$\bullet$};
\node[above] at (-1, 1.732) {$p_1$};
\node[below] at (-1, -1.732) {$p_2$};\node at (4,2) {$\bullet$};
\node at (4,-2) {$\bullet$};
\node at (6,0) {$\bullet$};

\node at (0,0) {$C_1$};

  \end{tikzpicture}
  \caption{Fundamental domain of $C$ with $C_1$ \newline the irreducible component on the left.}
 \end{subfigure}
 \begin{subfigure}{.45\textwidth}
  \centering
  \begin{tikzpicture}[scale=0.6]

  \draw (0,0) circle (2cm);
  \draw (4,0) circle (2cm);

\node at (1.35, 0) {$n_h$};
\node at (2.9, 0) {$n_{\sigma_1 h}$};
\node at (2,0) {$\bullet$};
\node at (-1, 1.732) {$\bullet$};
\node at (-1, -1.732) {$\bullet$};
\node[above] at (-1, 1.732) {$p_2$};
\node[below] at (-1, -1.732) {$p_1$};
\node at (4,2) {$\bullet$};
\node at (4,-2) {$\bullet$};
\node at (6,0) {$\bullet$};

\node at (0,0) {$C_1'$};

  \end{tikzpicture}
  \caption{Fundamental domain of $C'$ with $C_1'$ \newline the irreducible component on the left.}
 \end{subfigure}

 \caption{Two graded genus 0 $W$-spin surfaces $C$ and $C'$ where $C \SIMX C'$. Only the fundamental domains $\Sigma$ and $\Sigma'$ are drawn.}
\label{fig:exchange}
\end{figure}
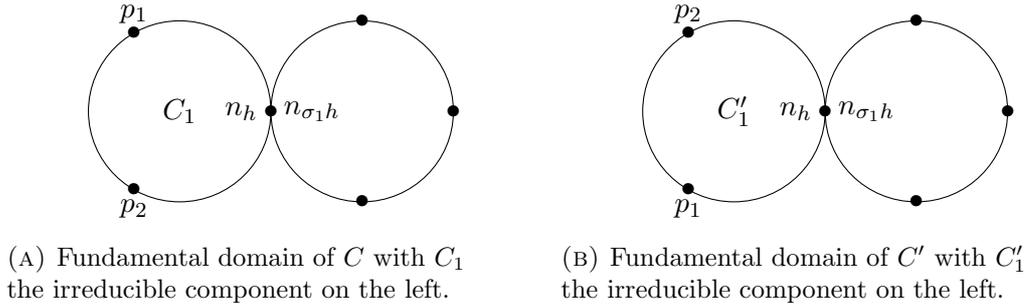

\subsection{Graded and rooted graphs}

Here we consider the combinatorial data which will be used to keep track of
boundary strata of moduli spaces of graded $W$-spin disks.
Of course, there is a standard dual graph construction associated
to a stable curve with marked points. In our case, we have
more structure which we encode
in decorated dual graphs. Our graphs generalize
the $r$-spin graphs of \cite{BCT:I}, with one difference: our graphs will also carry a collection of cyclic orders of boundary tails.

\begin{definition}\label{def:non_spin_graph}
A {\it genus zero, anchored, pre-stable dual graph} is a tuple
\[\Gamma = (V, H, \sigma_0, \sim, \hat\Pi, H^{CB}, T^{\mathrm{anch}},m),\]
in which
\begin{enumerate}[(i)]
\item $V$ is a finite set (the {\it vertices}) equipped with a decomposition $V = V^O \sqcup V^C$ into {\it open} and {\it closed vertices};
\item $H$ is a finite set (the {\it half-edges}) equipped with a decomposition $H = H^B \cup H^I$ into {\it boundary} and {\it internal half-edges};
\item $\sigma_0: H \rightarrow V$ is a function, viewed as associating to each half-edge the vertex from which it emanates;
\item $\sim$ is an equivalence relation on $H$, which decomposes as a pair of equivalence relations $\sim_B$ on $H^B$ and $\sim_I$ on $H^I$.  The equivalence classes are required to be of size $1$ or $2$, and those of size $1$ are referred to as {\it tails}.  We denote by $T^B \subseteq H^B$ and $T^I \subseteq H^I$ the sets of equivalence classes of size $1$ in $H^B$ and $H^I$, respectively;
\item $\hat\Pi=(\hat\Pi_v)_{v\in V^O},$  where $\hat\Pi_v$ is a collection of cyclic orders on $\sigma_0^{-1}(v)\cap H^B$.
\item $H^{CB}\subseteq T^{\mathrm{anch}}\subseteq T^I$. Here $H^{CB}$ is the set of {\it contracted boundary tails} and $T^{\mathrm{anch}}$ is the set of \emph{anchors};
\item $m$ is a function given by
\[m = m^B \sqcup m^I: T^B \sqcup (T^I \setminus H^{CB}) \rightarrow \Universe,\]
where $m^B$ and $m^I$ (the \emph{boundary} and \emph{internal markings}) satisfy the definition of a marking when restricted to any connected component of $\Gamma$. 
\end{enumerate}
Note that $(V,H, \sigma_0,\sim)$ defines a graph, and that we do not require this graph to be connected; denote its connected components by $\{\Gamma_i\}$. We require the above data to satisfy the following conditions:
\begin{enumerate}
\item For each boundary half-edge $h \in H^B$, we have $\sigma_0(h) \in V^O$;
\item For each $\Gamma_i$, the obvious topological realisation of the
graph $\Gamma_i$ has first Betti number $0$;
\item Each \emph{open component} $\Gamma_i$ (a component which has open vertices) has no half-edge in $T^{\mathrm{anch}}$. Any \emph{closed component} ($V^O(\Gamma_i)=\emptyset$) contains exactly one half-edge in $T^{\mathrm{anch}}$.
\item For each $\Gamma_i$, the sub-graph formed by its open vertices (if any exist) and their incident boundary edges is connected.
\end{enumerate}
\end{definition}

\begin{ex}
In Figure~\ref{fig:graded_disk_and_graph}, we provide a graded $W$-spin graph and its corresponding dual graph $\Gamma$. Here, $V^O = \{v_1, v_2\}$, $V^C = \{v_3\}$, $H^B$ contains all the $x_i$ and the two half-edges corresponding to the edge $q_2$, and $H^I$ contains all the half-edges $z_i$ and the two half-edges corresponding to the edge $q_1$.
\end{ex}

\begin{figure}
\begin{subfigure}{.4\textwidth}
  \centering

\begin{tikzpicture}[scale=0.6]
\vspace{0.15cm}

  \draw (0,4) circle (2cm);
   \draw (-2,4) arc (180:360:2 and 0.6);
  \draw[dashed] (2,4) arc (0:180:2 and 0.6);
  \draw (-2,0) arc (180:360:2 and 0.6);
  \draw[dashed] (2,0) arc (0:180:2 and 0.6);
  \draw (2,0) arc (0:180:2 and 2);
    \draw (2,0) arc (180:360:2 and 0.6);
  \draw[dashed] (6,0) arc (0:180:2 and 0.6);
  \draw (6,0) arc (0:180:2 and 2);

  \node (c) at (-1,.95) {$\bullet$};
\node[right] at (-1,.95) {$z_1$};

  \node (c) at (2,0) {$\bullet$};
\node[below] at (2,0) {$q_2$};

\node at (.25, -.6) {$\bullet$};
\node[above] at (.25,-.6) {$x_1$};

\node at (4.25, -.6) {$\bullet$};
\node[above] at (4.25,-.6) {$x_2$};

\node at (0,2) {$\bullet$};
\node[above] at (0,2) {$q_1$};
  \node (c) at (3,.95) {$\bullet$};
\node[right] at (3,.95) {$z_2$};

\node at (.25, 3.6) {$\bullet$};
\node[above] at (.25,3.6) {$z_4$};
  \node (c) at (-1,4.95) {$\bullet$};
\node[right] at (-1,4.95) {$z_3$};

\end{tikzpicture}
\vspace{0.15cm}

  \caption{Graded $W$-spin disk}
\end{subfigure}
\begin{subfigure}{.4\textwidth}
  \centering
\vspace{.24cm}
\begin{tikzpicture}[scale=0.4]

	\draw (0,0) circle (.2cm);
	\draw (0,4)[black, fill = black] circle (.2cm);
	\draw (4,0) circle (.2cm);
	
	\draw (0,0.2) -- (0,4);
	\draw[dashed] (0.2,0) -- (3.8,0);
	\node[below] at (2,0) {$e_{q_2}$};
	\node[right] at (0,2) {$e_{q_1}$};
	
	\draw (4.1414,.1414) -- (5.6414, 1.6414);
	\draw[dashed] (4.1414,-.1414) -- (5.6414, -1.6414);
	
	\draw (-.1414,.1414) -- (-1.6414, 1.6414);
	\draw[dashed] (-.1414,-.1414) -- (-1.6414, -1.6414);
	
	\node[above] at (5.6414, 1.6414) {$t_{z_2}$};
	\node[below] at (5.6414, -1.6414) {$t_{x_2}$};
	
	\node[above] at (-1.6414, 1.6414) {$t_{z_1}$};
	\node[below] at (-1.6414, -1.6414) {$t_{x_1}$};
	
	\draw (-.1414,4.1414) -- (-1.6414, 5.6414);
	\draw (.1414,4.1414) -- (1.6414, 5.6414);
	\node[above] at (-1.6414, 5.6414) {$t_{z_3}$};
	\node[above] at (1.6414, 5.6414) {$t_{z_4}$};
	
	\node[below] at (0,0) {$v_1$};
	\node[below] at (4,0) {$v_2$};
	\node[right] at (0,4) {$v_3$};
\end{tikzpicture}
\vspace{0.15cm}

  \caption{Dual graph $\Gamma$}
\end{subfigure}
\caption{A graded $W$-spin disk and its corresponding dual graph $\Gamma$}
\label{fig:graded_disk_and_graph}
\end{figure}
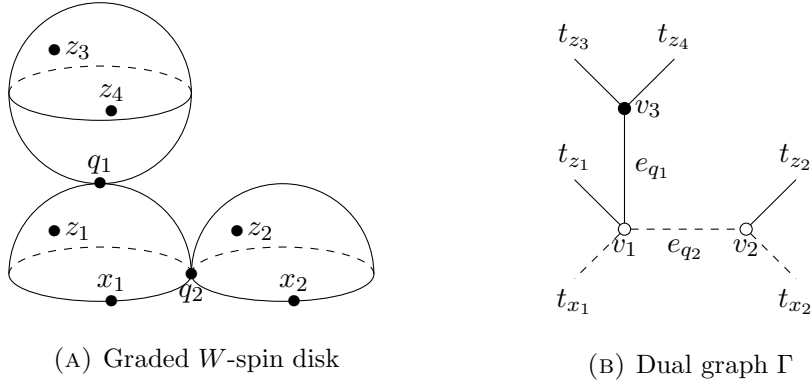

\begin{nn}
(1) We refer to elements of $T^B$ as {\it boundary tails} and elements of $T^I \setminus H^{CB}$ as {\it internal tails}, and we denote
\[T := T^I \sqcup T^B.\]
Note that $\sim$ induces a fixed-point-free involution on $H \setminus T$, which we denote by $\sigma_1$.  We denote
\[E^B := (H^B \setminus T^B)/\sim_B,~~E^I := (H^I \setminus T^I)/\sim_I\]
and refer to these as {\it boundary edges} and {\it internal edges}.
The set of {\it edges} is
\[E := E^B \sqcup E^I.\]
Denote by $\sigma_0^B$ the restriction of $\sigma_0$ to $H^B$, and similarly for $\sigma_0^I$.

For each vertex $v$, set
\[k(v) := |(\sigma_0^B)^{-1}(v)|,~~B(v)=m^B(T^B\cap (\sigma_0^B)^{-1}(v)).\]
Analogously, set
\[l(v):= |(\sigma_0^I)^{-1}(v)|,~~I(v)= m^I((T^I\setminus H^{CB})\cap (\sigma_0^I)^{-1}(v) ).\]
We say that an open vertex $v \in V^O$ is {\it stable} if $k(v) + 2l(v) > 2$, and we say that a closed vertex $v \in V^C$ is {\it stable} if $l(v) > 2$. An open vertex is \emph{partially stable} if $k(v)=0,l(v)=1$ or $k(v) = 2, l(v) = 0$,
i.e., $k(v)+2l(v)=2$.

\smallskip

(2) A genus zero pre-stable dual graph is said to be {\it connected} if it is connected as a graph, and we write $\Conn(\Gamma)$ for the set of connected components of $\Gamma$. The graph is said to be {\it stable} ({\it partially stable}) if all of its vertices are stable (stable or partially stable), and it is said to be {\it closed} if $V^O = \emptyset$. The graph is \emph{smooth} if there are no edges or contracted boundary tails. An \emph{isomorphism} of genus $0$ pre-stable dual graphs is a graph isomorphism which respects all additional structure. We denote by $\text{Aut}(\Gamma)$ the group of automorphisms of $\Gamma$.

\smallskip

(3) When each set $\hat\Pi_v$ is a singleton $\{\hat\pi_v\}$, we denote $\hat\Pi,\hat\Pi_v$ by $\hat\pi,\hat\pi_v,$ respectively. When $\hat\Pi$ is omitted from the notation we mean that for all $v\in V^O,~\hat\Pi_v$ consists of all possible cyclic orders of the boundary half edges of $v$.

Observe that $\hat\Pi$ induces a set of cyclic orders of the boundary tails of each connected component $\Lambda$ of $\Gamma$ obtained as follows.
For each vertex $v$ of $\Lambda$, choose $\hat\pi_v\in \hat\Pi_v$. We may then
follow the boundary tails of $\Lambda$ in the unique cyclic way compatible with $\hat\pi_v$. We denote the set of all cyclic orders of the
boundary tails of $\Gamma$ obtained in this way by $\hat\Pi_\Gamma,$ or,
in case of a singleton, $\hat\pi_\Gamma$. The data of $\hat\Pi_\Gamma$ and $\hat\Pi$ is equivalent and will be referred to as the cyclic orderings of $\Gamma$.
\end{nn}

While pre-stable dual graphs encode the discrete data of an anchored orbifold Riemann surface with boundary, the following decorated version encodes the additional data of a graded $W$-spin structure. We require many properties
of the data, but they all arise from corresponding properties of
surfaces. For each condition, we will indicate the corresponding property of
surfaces to motivate the definition.

\begin{definition}
\label{def:graph}
A {\it genus zero $W$-spin dual graph with a lifting} is a genus zero pre-stable anchored dual graph $\Gamma$ as above, together with two $a$-tuples of maps,
\[\tw = (\tw_i)_{i\in [a]},~~\tw_i: H \rightarrow \{-1, 0, 1, \ldots, r_i-1\}\]
(the {\it twist}) and
\[\alt = (\alt_i)_{i\in[a]},~~\alt_i: H^B \rightarrow \Z/2\Z,\]
satisfying the following conditions:
\begin{enumerate}[(i)]
\item%[Partially stable:]
Any connected component of $\Gamma$ is either: (1) stable, i.e., all
vertices are stable, or (2) consists of a single partially stable vertex.

\item%[Anchor:]
All tails $t$ with $\tw_i(t) =-1$ for some $i$ belong to $T^{\mathrm{anch}}$:
this arises from \eqref{eq:I0}.
Recalling $H^{CB}\subseteq T^{\mathrm{anch}}$, any element of $T^{\mathrm{anch}}\setminus H^{CB}$ is marked $\emptyset$ and is the only tail marked $\emptyset$ in its connected component: this condition arises from the behaviour of labeling of anchors under
normalization, Definition \ref{def:open normal}.

\item%[$-1$ twist:]
\label{c1}
For any vertex $v$, there exists at most one incident half-edge $h$ with
either $h \in T^{\mathrm{anch}}$ or $\tw_i(h) = -1$ for some $i$.  Note that the latter case only consists of half-edges that, after normalization, will become an anchor.

\item%[$H^{CB}$:]
For any contracted boundary tail $t \in H^{CB}$, we have $\tw(t) =(r_1-1,\ldots,r_a-1)$. This condition arises from Remark \ref{rem:CB Ramond}.

\item%[Open vertices]
 \label{item} For any open vertex $v \in V^O,~i\in[a]$,
\[2\sum_{h \in (\sigma_0^I)^{-1}(v)} \tw_i(h) + \sum_{h \in (\sigma_0^B)^{-1}(v)} \tw_i(h) \equiv r_i-2 \mod r_i\]
and
\[\frac{2\sum_{h \in (\sigma_0^I)^{-1}(v)}\ \tw_i(h) + \sum_{h \in (\sigma_0^B)^{-1}(v)} \tw_i(h) + 2}{r_i} \equiv \sum_{h \in (\sigma_0^B)^{-1}(v)} \alt_i(h) \mod 2.\]
These conditions arise from \eqref{eq:open_rank1_general} and
Proposition \ref{prop:graded_r_spin_prop}, (3), (4) respectively.

\item%[Closed vertices:]
For any closed vertex $v \in V^C,~i\in[a]$,
\[\sum_{h \in \sigma_0^{-1}(v)} \tw_i(h) \equiv r_i-2 \mod r_i.\]
This condition arises from \eqref{eq:close_rank1_general}.

\item%[Edge twist:]
\label{c2}
\begin{enumerate}
\item
For any half-edge $h \in H \setminus T,~i\in[a]$, we have
\[\tw_i(h) + \tw_i(\sigma_1(h)) \equiv r_i -2 \mod r_i,\]
and
at most one of $\tw_i(h)$ and $\tw_i(\sigma_1(h))$ equals $-1$.
This condition arises from \eqref{eq:twists_at_half_nodes} and \eqref{eq:cR}.
\item No boundary half-edge $h$ satisfies $\tw_i(h)=-1$, as all anchors
are internal.
\item In case $h \in H^I \setminus T^I$ satisfies $\tw_i(h) \equiv -1\mod~r_i,$
then $\tw_i(h)=r_i-1$ precisely when the corresponding half-node does not become an anchor in the partial normalisation following the rules given in Definition~\ref{def:open normal}. More precisely, in terms of graphs, this can determined algorithmically as follows. 
Remove the edge of $\Gamma$ determined by $h$ and replace it
with two tails, $h$ and $\sigma(h)$, to obtain a graph
$\Gamma'$ with one more connected component than $\Gamma$.\footnote{This will be called the
detaching of $\Gamma$ at $e$: see
Definition \ref{def:detach}.}  Note that
one of these connected components can be viewed as containing
$h$. Then either (a) 
$h$ belongs to a connected component of $\Gamma'$ containing an anchor of $\Gamma$,
or (b)  $h$
belongs to a connected connected component of $\Gamma'$ containing
an open vertex. 
\end{enumerate}

\item%[$\alt:$]
\label{it:-1}For any boundary half-edge $h \in H^B \setminus T^B$, if $\tw_i(h) \neq r_i-1 $ we have
\[\alt_i(h) + \alt_i(\sigma_1(h)) = 1\]
and if $\tw_i(h) =r_i -1$ we have
\[\alt_i(h) = \alt_i(\sigma_1(h)) = 0.\]
The first condition arises from Definition \ref{def:open_W_graded2}(1)(b), and the
second from Proposition~\ref{prop:graded_r_spin_prop}(2) and (3). Note that this proposition also implies that the second case only occurs for odd $r$.
\item%[Parity of $\alt$:]
If $r_i$ is odd, then for any $h \in H^B$,
\[\alt_i(h) \equiv \tw_i(h) \mod 2,\]
and if $r_i$ is even, then for any $h \in H^B$,
\[2|\tw_i(h).\]
These conditions arise from Proposition \ref{prop:graded_r_spin_prop},
(3) and (2) respectively.
\end{enumerate}
\end{definition}
We define $\text{for}_{\text{spin}}(\Gamma)$ to be the dual graph obtained from $\Gamma$ by forgetting the additional data $\tw,\alt$. We similarly define $\text{for}_{\text{spin}\neq i}(\Gamma),$ to be the $W_i$-dual graph with a lifting obtained by forgetting all $\tw_j,\alt_j$ for $j\neq i,$ where $W_i=x_i^{r_i}$.

A genus $0$ $W$-spin surface with boundary and lifting $\Sigma$  induces
a genus $0$ $W$-spin dual graph with a lifting, denoted by $\Gamma(\Sigma)$,
in the standard way, see Figure \ref{fig:smooth_and_detach}.

In what follows we shall be interested only in pre-graded, graded and
rooted graphs
analogous to Definition \ref{def:open_W_graded2},
which we now define. More general graphs with a lifting will appear only
as boundary strata of graded graphs.

\begin{definition}\label{def: graded W spin graph}
\begin{enumerate}
\item
A \emph{pre-graded $W$-spin graph} is a genus zero $W$-spin dual graph with a lifting, with the additional requirement that
\begin{enumerate}
\item
For any boundary marking $p,$ and any $i\in[a]$, either \[\tw_i(p)=r_i-2,\alt_i(p)=1,~~\text{or }\tw_i(p)=\alt_i(p)=0.\]
\item
If $e$ is a boundary edge with half-edges $h, h'$  and $i\in [a]$ with
$\tw_i(h)\in \{0,\ldots,r_i-2\}$,
then the lifting alternates at precisely one of $h,h'$.
\end{enumerate}
\item
A \emph{graded $W$-spin graph} is a pre-graded $W$-spin graph,
such that in addition
there is no boundary tail $t$ with $\tw(t)=\alt(t)=\vec{0}$.
\item
A \emph{rooted $W$-spin graph} is a graded $W$-spin graph
such that for every
open connected component all boundary tails are singly twisted,
i.e., there exists an $i\in[a]$ such that $\tw_j(p)=(r_j-2)\delta_{ij}$,
except for one which has twist $(r_1-2,\ldots,r_a-2)$.
\end{enumerate}
\end{definition}

\begin{definition}
Let $\Gamma$ be a genus zero $W$-spin dual graph with a lifting.
\begin{enumerate}
\item Boundary half-edges $h$ with $\alt_i(h) = 1$ are called {\it alternating for the $i^{th}$ bundle}, and those with $\alt_i(h) = 0$ are called {\it non-alternating for the $i^{th}$ bundle}.
\item
Half-edges $h$ with $\tw_i(h) \in \{-1, r_i-1\}$ are called {\it Ramond for the $i^{th}$ bundle}, and those with $\tw_i(h) \in \{0, \ldots, r_i-2\}$ are called {\it Neveu--Schwarz for the $i^{th}$ bundle}. An edge is Ramond or
Neveu--Schwarz with respect to the $i^{th}$ bundle if its half-edges
are.\footnote{We note that if a half-edge is Ramond for some
$i\in [a]$, it is standard in the literature to refer to this as a
\emph{broad half-edge}, and otherwise it is called \emph{narrow}. The
same terminology holds for edges.}
\item A half-edge $h$ is \emph{singly twisted} if there exists an $i\in [a]$ such that
$\tw_j(h)=(r_j-2)\delta_{ij}$ for all $j\in [a]$.
\end{enumerate}
\end{definition}

\begin{definition}
We say that a genus zero pre-graded $W$-spin graph is {\it stable} if the underlying dual graph is stable, in the sense specified above.  An {\it isomorphism} between pre-graded $W$-spin dual graphs with liftings consists of an isomorphism as dual graphs that respects $\tw,~\alt$.

In the special case $W=x^r$ we call the (pre-)graded $W$-spin graph a \emph{(pre-)graded $r$-spin graph}.
\end{definition}

\begin{nn}
When $\Gamma=\Gamma(\Sigma),$ we write $n_h=n_h(\Sigma)$ for the half-node in the normalization of $\Sigma$ that corresponds to the half-edge $h\in (H\setminus T)\cup H^{CB}$. We write $n_e$ for the node in $\Sigma$ that corresponds to the edge $e$.  When $h$ is an internal half-edge or a contracted boundary tail, we sometimes write $z_h$ instead of $n_h$.  When $h$ is a boundary half-edge, we sometimes write $x_h$ instead of $n_h$.
\end{nn}
We end this subsection by providing notation for some special graphs that will appear throughout the remainder of the paper.

\begin{nn}\label{nn:repeating graphs}
We denote by $\Gamma_{0,B,I}$ the genus zero, anchored, pre-stable dual graph
$$
\Gamma_{0,B,I} = (V, H, \sigma_0, \sim, \hat\Pi, H^{CB},T^{\mathrm{anch}},m)
$$
where
\begin{itemize}
\item $V$ consists of a single open vertex $v$;
\item $H$ is a finite set of half-edges $H = B\sqcup I$, where $H^B = B$ are the boundary half-edges and $H^I = I$ are the interior half-edges;
\item $\sigma_0: H\to V$ is the constant map to $v$;
\item $\sim$ is the trivial equivalence relation;
\item $\hat\Pi = \hat\Pi_v$ is the collection of all possible cyclic orders of the set of boundary half-edges $B$;
\item $H^{CB} =T^{\mathrm{anch}} = \emptyset$;
\item the marking function $m=m^B \sqcup m^I: B \sqcup I \to \Omega$ is defined in such a way so that $m^B$ is the constant map mapping all elements to $\emptyset \in \Omega$ and $m^I$ is a bijection with $\{1, \ldots, |I|\} \subset \Omega$.
\end{itemize}
Note that all half-edges in $\Gamma_{0,B,I}$ are also tails. We write  $\Gamma_{0,B,I}^{\text{labeled}}$ for the same graph, except we replace $m^B$ with a marking function $m^{B, \text{labeled}}$ that is a bijection with $\{1,\ldots, |B|\}\subset\Omega$.
In the special case where $B=[k]$ we write $\Gamma_{0,k,I}$. If we further have that $I=[l]$, we write $\Gamma_{0,k,l}$. In the labeled version, we write $\Gamma_{0,k,I}^{\text{labeled}}$ and $\Gamma_{0,k,l}^{\text{labeled}}$ for the graphs with the same alteration of the marking function.\footnote{Frequently we want to choose
some master index set $I\subseteq\Omega$, typically $I=[l]$, and then
consider graphs whose internal tails are marked by subsets $J\subseteq I$.
The choice of index set is chosen depending on what marking is natural in the context.}

We also write
$\Gamma_{0,\{B_J\}_{J\subseteq{[a]}}, \{\vec{a}_j\}_{j\in [l]}}^{W}$ for the (smooth, connected) pre-graded $W$-spin graph given by taking $\Gamma = \Gamma_{0, B, I}$ where $B = \sqcup_{J\subseteq [a]} B_J$ and $I =[l]$ with the $a$-tuples of maps $\tw$ and $\alt$ defined as follows. For the internal tails, enumerated from $1$ to $l$, the $j$th internal tail point has twist $\vec{a}_j$. For the set of unlabeled boundary tails $t_j$ in the set $B_J$ for some $J \subseteq [a]$, we have that
\[(\tw_i(t_j),\alt_i(t_j))=\begin{cases} (r_i-2,1) & \text{ if $i\in J$}\\ (0,0) & \text{ if $i\notin J$.}\end{cases} \]

We write $\Gamma_{0,\{B_J\}_{J\subseteq{[a]}}, \{\vec{a}_j\}_{j\in [l]}}^{W,\text{labeled}}$ for the labeled analogue with the aforementioned change to the marking function. We write $\Gamma_{0,\{k_J\}_{J\subseteq{[a]}}, \{\vec{a}_j\}_{j\in [l]}}^{W}$ for the case $B_J=[k_J]$, noting that $B$ is the disjoint union of the sets $B_J$.

Note that when $B_\emptyset=\emptyset$, the graph is graded. In this case, we omit $B_\emptyset$ from the notations.
\end{nn}

\begin{ex}
The dual graph to the smooth connected marked genus 0 Riemann surface with $k$ boundary markings and $l$ internal markings is $\Gamma_{0,k,l}$.
\end{ex}

\begin{definition}\label{defn:subordinate}
Let $\Gamma= (V, H, \sigma_0, \sim, \hat\Pi, H^{CB}, T^{\mathrm{anch}},m)$ be a genus zero, anchored, pre-stable dual graph. Take another genus zero, anchored, pre-stable dual graph $\Lambda$. We say that $\Lambda\equiv \Gamma$ \emph{modulo cyclic orders} if $\Lambda$ can be written in the form
\[\Lambda = (V, H, \sigma_0, \sim, \hat\Pi_{\Lambda}, H^{CB}, T^{\mathrm{anch}},m)
\] for some new set of cyclic orders $\hat\Pi_{\Lambda}$. If, moreover, we have the inclusion $\hat\Pi_{\Lambda} \subset \hat\Pi$ then we say $\Lambda$ is a \emph{subordinate graph} of $\Gamma$ and denote this relation by $\Lambda \prec \Gamma$.
\end{definition}
\begin{rmk}
Geometrically, the relation $\Lambda \prec \Gamma$ will later mean that the moduli space of disks corresponding to $\Lambda$ consists of certain connected components of the moduli space of disks corresponding to $\Gamma$.
\end{rmk}

\subsection{Graphs associated to non-smooth graded disks}

Several graph operations play a role in what follows.

\begin{definition}\label{def:smoothingraph}
Let $\Gamma$ be a pre-graded $W$-spin graph.
\begin{enumerate}
\item
Suppose $\Gamma$ has an edge $e$ connecting vertices $v_1$ and $v_2$. The {\it smoothing} of $\Gamma$ along $e$ is the graph $\smooth_e\,\Gamma$
obtained by contracting $e$ and replacing the two vertices $v_1$ and $v_2$ with a single vertex $v_{12}$. This vertex is closed if and only if both $v_1$ and $v_2$ are closed. If $v_1,v_2$ are both open,
we obtain a collection of cyclic orderings for $v_{12}$ as follows.
Take any $\hat{\pi}_{v_1}\in \hat{\Pi}_{v_1}, \hat{\pi}_{v_2}\in \hat{\Pi}_{v_2}$. Let $e=\{h_1,h_2\}$ where $h_i$ is attached to $v_i$.  Construct the ordering induced by the cyclic ordering $\hat{\pi}_{v_2}\big|_{\sigma_0^{-1}(v_2)\cap H^B\setminus\{h_2\}}$, obtained by starting at the half-edge that directly follows $h_2$. When one reaches $h_2$ again, we now concatenate with the analogous ordering induced from the cyclic ordering $\hat{\pi}_{v_1}\big|_{\sigma_0^{-1}(v_1)\cap H^B\setminus\{h_1\}}$. The set $\hat{\Pi}_{v_{12}}$ is the
collection of cyclic orders that can be obtained this way.
\item
Suppose $\Gamma$ has a contracted boundary tail $h \in H^{CB}$.
The {\it smoothing} of $\Gamma$ along $h$
is the graph $\smooth_h\Gamma$ obtained by erasing $h$ and moving the vertex $v = \sigma_0^{-1}(h)$ from $V^C$ to $V^O$.
\item
For a set $F$ of edges and contracted boundary tails, one can perform a sequence of smoothings, and the graph obtained is independent of the order in which those smoothings are performed; denote the result by $\smooth_F\Gamma$. Write also $\smooth\,\Gamma$ for the smooth graph $\smooth_{E(\Gamma)\cup H^{CB}(\Gamma)}\Gamma$ which we will call the \emph{smoothing of $\Gamma$}.
\item
We say a pre-graded $W$-spin graph $\Lambda$ is a \emph{degeneration} of $\Gamma$ if both of the following hold:
\begin{enumerate}
\item $\smooth_{F}\Lambda \prec \Gamma$ for some set $F$ of edges and contracted boundary tails of $\Lambda$.
\item If $\Lambda \prec \Lambda'$ and $\smooth_{F}\Lambda' \prec \Gamma$, then $\Lambda = \Lambda'$, i.e., $\Lambda$ is maximal among orderings satisfying
(a).
\end{enumerate}
\end{enumerate}
\end{definition}

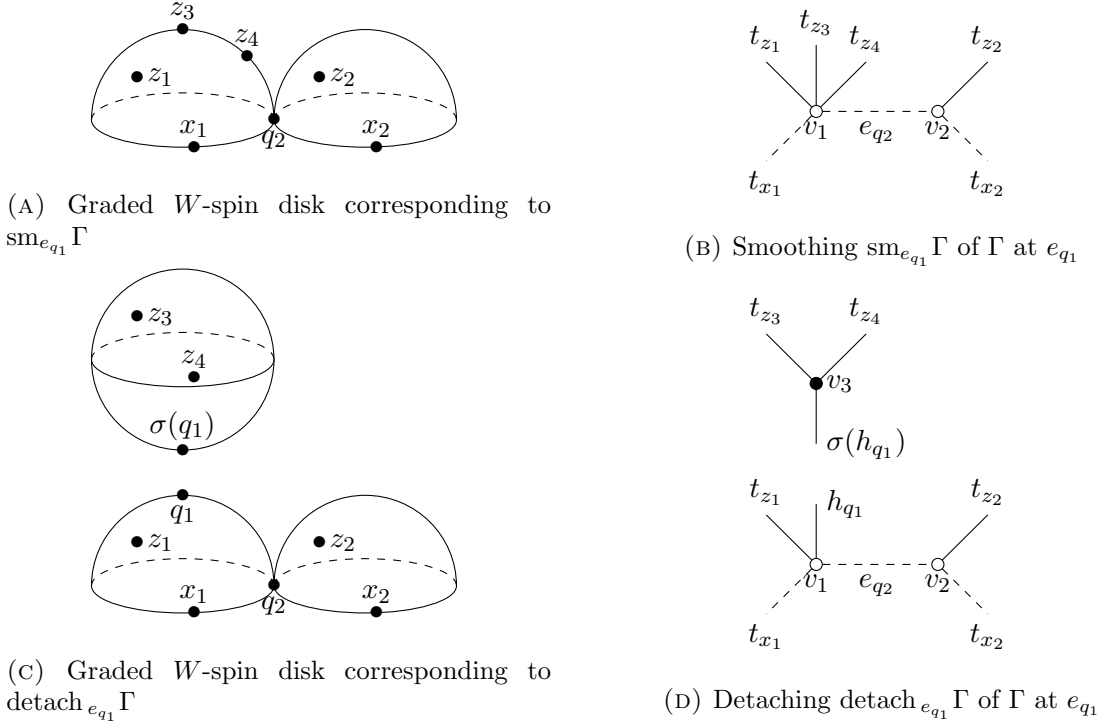
\begin{figure}
\begin{subfigure}{.45\textwidth}
  \centering

\begin{tikzpicture}[scale=0.6]
\vspace{0.15cm}

  \draw (-2,0) arc (180:360:2 and 0.6);
  \draw[dashed] (2,0) arc (0:180:2 and 0.6);
  \draw (2,0) arc (0:180:2 and 2);
    \draw (2,0) arc (180:360:2 and 0.6);
  \draw[dashed] (6,0) arc (0:180:2 and 0.6);
  \draw (6,0) arc (0:180:2 and 2);

  \node (c) at (-1,.95) {$\bullet$};
\node[right] at (-1,.95) {$z_1$};

  \node (c) at (2,0) {$\bullet$};
\node[below] at (2,0) {$q_2$};

\node at (.25, -.6) {$\bullet$};
\node[above] at (.25,-.6) {$x_1$};

\node at (4.25, -.6) {$\bullet$};
\node[above] at (4.25,-.6) {$x_2$};

  \node (c) at (3,.95) {$\bullet$};
\node[right] at (3,.95) {$z_2$};

\node at (1.414, 1.414) {$\bullet$};
\node[above] at (1.414,1.414) {$z_4$};
  \node (c) at (0,2) {$\bullet$};
\node[above] at (0,2) {$z_3$};
\end{tikzpicture}
\vspace{0.15cm}
  \caption{Graded $W$-spin disk corresponding to $\smooth_{e_{q_1}}\Gamma$}
\end{subfigure}\qquad
\begin{subfigure}{.45\textwidth}
  \centering
\vspace{.24cm}
\begin{tikzpicture}[scale=.4]

	\draw (0,0) circle (.2cm);
	\draw (4,0) circle (.2cm);
	
	\draw[dashed] (0.2,0) -- (3.8,0);
	\node[below] at (2,0) {$e_{q_2}$};
	
	\draw (4.1414,.1414) -- (5.6414, 1.6414);
	\draw[dashed] (4.1414,-.1414) -- (5.6414, -1.6414);
	
	\draw (-.1414,.1414) -- (-1.6414, 1.6414);
	\draw[dashed] (-.1414,-.1414) -- (-1.6414, -1.6414);
	
	\node[above] at (5.6414, 1.6414) {$t_{z_2}$};
	\node[below] at (5.6414, -1.6414) {$t_{x_2}$};
	
	\node[above] at (-1.6414, 1.6414) {$t_{z_1}$};
	\node[below] at (-1.6414, -1.6414) {$t_{x_1}$};
	
	\draw (0,.2) -- (0, 2.2);
	\draw (.1414,.1414) -- (1.6414, 1.6414);
	\node[above] at (0, 2.2) {$t_{z_3}$};
	\node[above] at (1.6414, 1.6414) {$t_{z_4}$};
	
	\node[below] at (0,0) {$v_1$};
	\node[below] at (4,0) {$v_2$};
\end{tikzpicture}
\vspace{0.15cm}

  \caption{Smoothing $\smooth_{e_{q_1}}\Gamma$ of $\Gamma$ at $e_{q_1}$}
\end{subfigure}

\begin{subfigure}{.45\textwidth}
  \centering

\begin{tikzpicture}[scale=0.6]
\vspace{0.15cm}

  \draw (0,5) circle (2cm);
   \draw (-2,5) arc (180:360:2 and 0.6);
  \draw[dashed] (2,5) arc (0:180:2 and 0.6);
  \draw (-2,0) arc (180:360:2 and 0.6);
  \draw[dashed] (2,0) arc (0:180:2 and 0.6);
  \draw (2,0) arc (0:180:2 and 2);
    \draw (2,0) arc (180:360:2 and 0.6);
  \draw[dashed] (6,0) arc (0:180:2 and 0.6);
  \draw (6,0) arc (0:180:2 and 2);

  \node (c) at (-1,.95) {$\bullet$};
\node[right] at (-1,.95) {$z_1$};

  \node (c) at (2,0) {$\bullet$};
\node[below] at (2,0) {$q_2$};

\node at (.25, -.6) {$\bullet$};
\node[above] at (.25,-.6) {$x_1$};

\node at (4.25, -.6) {$\bullet$};
\node[above] at (4.25,-.6) {$x_2$};

\node at (0,2) {$\bullet$};
\node[below] at (0,2) {$q_1$};

\node at (0,3) {$\bullet$};
\node[above] at (0,3) {$\sigma(q_1)$};
  \node (c) at (3,.95) {$\bullet$};
\node[right] at (3,.95) {$z_2$};

\node at (.25, 4.6) {$\bullet$};
\node[above] at (.25,4.6) {$z_4$};
  \node (c) at (-1,5.95) {$\bullet$};
\node[right] at (-1,5.95) {$z_3$};

\end{tikzpicture}
\vspace{0.15cm}

  \caption{Graded $W$-spin disk corresponding to $\detach_{e_{q_1}}\Gamma$}
\end{subfigure}\qquad
\begin{subfigure}{.45\textwidth}
  \centering
\vspace{.24cm}
\begin{tikzpicture}[scale=0.4]

	\draw (0,0) circle (.2cm);
	\draw (0,6)[black, fill = black] circle (.2cm);
	\draw (4,0) circle (.2cm);
	
	\draw (0,0.2) -- (0,2);
	\draw (0,4) -- (0,6);
	\draw[dashed] (0.2,0) -- (3.8,0);
	\node[below] at (2,0) {$e_{q_2}$};
	\node[right] at (0,2) {$h_{q_1}$};
	\node[right] at (0,4) {$\sigma(h_{q_1})$};

	\draw (4.1414,.1414) -- (5.6414, 1.6414);
	\draw[dashed] (4.1414,-.1414) -- (5.6414, -1.6414);
	
	\draw (-.1414,.1414) -- (-1.6414, 1.6414);
	\draw[dashed] (-.1414,-.1414) -- (-1.6414, -1.6414);
	
	\node[above] at (5.6414, 1.6414) {$t_{z_2}$};
	\node[below] at (5.6414, -1.6414) {$t_{x_2}$};
	
	\node[above] at (-1.6414, 1.6414) {$t_{z_1}$};
	\node[below] at (-1.6414, -1.6414) {$t_{x_1}$};
	
	\draw (-.1414,6.1414) -- (-1.6414, 7.6414);
	\draw (.1414,6.1414) -- (1.6414, 7.6414);
	\node[above] at (-1.6414, 7.6414) {$t_{z_3}$};
	\node[above] at (1.6414, 7.6414) {$t_{z_4}$};
	
	\node[below] at (0,0) {$v_1$};
	\node[below] at (4,0) {$v_2$};
	\node[right] at (0,6) {$v_3$};
\end{tikzpicture}
\vspace{0.15cm}

  \caption{Detaching $\detach_{e_{q_1}}\Gamma$ of $\Gamma$ at $e_{q_1}$}
\end{subfigure}

\caption{Smoothings and detachings of the dual graph $\Gamma$ from Figure~\ref{fig:graded_disk_and_graph}}
\label{fig:smooth_and_detach}
\end{figure}

\begin{ex}
In Figure~\ref{fig:smooth_and_detach}, we take the smoothing at one of the edges $e_{q_1}$ in the dual graph $\Gamma$ in Figure~\ref{fig:graded_disk_and_graph}. This graded $W$-spin disk corresponding to $\smooth_{e_{q_1}}\Gamma$ has two open irreducible components. We shall see in the next section that the moduli of graded $W$-spin disks with dual graph $\Gamma$ will be a boundary component for the moduli space of graded $W$-spin disks with dual graph $\smooth_{e_{q_1}}\Gamma$. For this reason, we call $\Gamma$ a degeneration of $\smooth_{e_{q_1}}\Gamma$.
\end{ex}

\begin{definition}
We set
\begin{align}
\label{eq:d!}\begin{split}
 \d^!\Gamma &= \{\Lambda \; | \; \Lambda \text{ is a degeneration of $\Gamma$} \},\\
\d \Gamma &= \d^!\Gamma \setminus \{\Gamma\}.
 \end{split}\end{align}
\end{definition}

Here $\d^!\Gamma$ is the set of all graphs that are degenerations of $\Gamma$ and $\d\Gamma$ is the set of all graphs that are non-trivial degenerations of $\Gamma$.

\begin{definition}
\label{def:detach}
Let $\Gamma$ be a genus zero pre-graded $W$-spin graph, and let $e$ be an edge of $\Gamma$ or a contracted boundary tail.  Then the {\it detaching} of $\Gamma$ at $e$ is the graph $\detach_e(\Gamma)=\Gamma(\Sigma_n)$ where
$\Sigma$ is any pre-graded $W$-spin disk with $\Gamma=\Gamma(\Sigma)$
and $\Sigma_n$ is the partial normalization of $\Sigma$ at the node $n$
corresponding to $e$.\footnote{The detaching operation can also be stated purely in terms of graphs, in the expected way. For example, when $e$ is an edge, the graph $\detach_e(\Gamma)$ is obtained from $\Gamma$ by cutting $e$ into the two half edges which form it, and then updating both $\sim$ and the marking data. Note that the half-edges that were originally in the edge $e$ are now tails.}

There is a canonical identification of $E(\Gamma) \setminus \{e\}$ with the edges of $\detach_e(\Gamma)$ when $e$ is an edge of $\Gamma$.
Similarly there is a canonical identification of $E(\Gamma)$ with the edges of $\detach_t(\Gamma)$ when $t$ is a contracted boundary tail. As the set of boundary half-edges incident to any vertex $v$ remains unchanged, the detaching does not change the set $\hat\Pi$ of cyclic orderings. One can also iterate this detaching process.  For any subset $N \subset E(\Gamma) \cup H^{CB}(\Gamma)$, we denote by $\detach_{N}(\Gamma)$ the graph obtained by performing $\detach_f$ for each element $f \in N$; the result is independent of the order in which the detachings are performed. When we write $\detach(\Gamma)$ without any subscript, we mean $\detach_{E(\Gamma)}(\Gamma)$.
\end{definition}

\begin{definition}\label{def: positive with respect to grading}
Let $\Gamma$ be a genus zero pre-graded $W$-spin graph with a lifting.
A boundary half-edge $h$ with $\alt_i(h)=0,\tw_i(h)>0$ for $i\in [a]$ will be called \emph{positive for the $i^{th}$ grading}. A boundary half-edge which is positive for one of the gradings is called \emph{positive}. Write $\Pos_i(\Gamma)$ for the set of half-edges $h$ of the graph $\Gamma$ such that either $h$ or $\sigma_1(h)$ is positive for the $i^{th}$ grading. We set
\[
\Pos(\Gamma)=\bigcup_{i\in [a]}\Pos_i(\Gamma).
\]
A half-node in a stable pre-graded $W$-spin surface with a lifting which corresponds to a positive half-edge will be called positive as well. 
\end{definition}

\begin{definition}\label{def:graphBestiary}
Let $\Gamma$ be a pre-graded $W$-spin graph. Let $\partial^{+}\Gamma$ be the subset of $\partial^{!}\Gamma$ with at least one positive half-edge.
Write $\partial^{\xch}\Gamma$ for the subset of $\partial\Gamma\setminus\partial^+\Gamma$ made of graphs $\Lambda$ such that $\Lambda$ has at least one boundary half-edge $h$ with the following property. Let $\Xi$ be the subgraph of $\detach_{\mathrm{edge}(h)}\Lambda$
which contains $h$. Then $\Xi$ is the graph with precisely one vertex and three boundary tails $t$, $t'$, and $h$, such that
\[
\tw_k(t)=(r_k-2)\mathbf{1}_{k\in I},\quad \tw_k(t')=(r_k-2)\mathbf{1}_{k\in J},~~\tw_k(h)=(r_k-2)\mathbf{1}_{k\notin I\cup J},
\]
for some arbitrary disjoint sets $I,J\subseteq[a]$. We call such graphs \emph{exchangeable}, and if $v\in V^O$ is the vertex which after detaching forms the graph $\Xi,$ then $v$ is called an \emph{exchangeable vertex}.

Define the equivalence relation $\simx$ on $\partial\Gamma\setminus\partial^+\Gamma$ by saying that $\Lambda\simx\Lambda'$ if $V^O(\Lambda)$ contains exchangeable vertices $v_1,\ldots,v_m$ whose sets of cyclic orders are singletons, $\hat\pi_{v_1},\ldots,\hat{\pi}_{v_m}$, such that the graph obtained by reversing those cyclic orders is isomorphic to $\Lambda'$. If  $\Gamma$ is an exchangeable graph and $v$ an exchangeable vertex, we denote by $\xch_v(\Gamma)$ the graph obtained from $\Gamma$ by replacing the cyclic orders in $\hat\Pi_v$ with their opposite cyclic orders. If $\Gamma$ has a single exchangeable vertex $v$ we write $\xch(\Gamma)$ for $\xch_v(\Gamma)$. This arises from Definition~\ref{def:SIMX}.
\end{definition}

\begin{rmk}Note that in rank $a=2$, a graph in $\partial\Gamma\setminus\partial^+\Gamma$ will be exchangeable if there exist $h,\Xi,t,t'$ as in the previous definition such that either (i) one of $t,t'$ has $(\tw,\alt)=(\vec{0},\vec{0})$. or (ii) both are singly twisted and one of them has $(\tw_1,\alt_1)=(0,0)$ and the other has $(\tw_2,\alt_2)=(0,0)$. Note that if the graph is graded,
case (i) does not occur.
\end{rmk}

\begin{definition}\label{partial 0 graph bestiary}
Let $\partial^0\Gamma $
be the subset of all graphs $\Lambda \in \partial\Gamma\setminus\left(\partial^+\Gamma\cup\partial^\xch\Gamma\right)$ satisfying one of the following:
\begin{enumerate}
\item $\Lambda$ contains a contracted boundary half-edge.
\item $\Lambda$ contains a boundary edge.
\end{enumerate}
\end{definition}

The motivation for these definitions and the upcoming definition
of irrelevant graphs is as follows.
The reader may skip this upon first reading if they do not enjoy spoilers.
In Section~\ref{sec: moduli}, we will define a moduli space $\oCM^W_{\Gamma}$
compactifying the moduli space $\CM^W_{\Gamma}$ parameterizing all open $W$-spin curves with dual graph $\Gamma$ (Definition~\ref{def:moduliGraph}).
Graphs $\Lambda \in \partial \Gamma$ parameterize strata of $\oCM^W_{\Gamma}$.
At the same time, $\oCM^W_{\Gamma}$ carries the \emph{Witten bundle} $\cW$
arising from the spin structure, also defined in Section~\ref{sec: moduli}.

The open FJRW invariants we define here will
be counts of zeroes of sections of $\cW$ satisfying certain boundary
conditions. How these conditions are treated depends on the type
of boundary stratum. Strata corresponding to $\Lambda\in\partial^+\Gamma$
are simply removed from $\oCM^W_{\Gamma},$ giving a partial
compactification $\oPM_\Gamma\subset \oCM^W_{\Gamma}$.

On the other hand, the graphs of $\partial^0\Gamma\cup\partial^\xch\Gamma$
precisely correspond to strata of $\partial \oPM_{\Gamma}$, see
the discussion following Definition \ref{def:positive_edge,node,PM}. Among these graphs, boundary components labeled by graphs that are $\simx$-equivalent but distinct will have similar
boundary counditions; however, an analysis of orientations will
show the contribution of such boundary strata to the zero count
cancels. 
One can think of this as if we were gluing exchangeable boundaries of $\oPM_\Gamma$, and counting the number of zeroes in the glued space. This is indeed possible, and equivalent to the approach we take.

The \emph{irrelevant} graphs $\Lambda$ defined below
are called irrelevant as they are irrelevant to the wall-crossing and enumerative invariants in the following sense. We will
require (Definition~\ref{def: strongly positive}) that all sections of the
Witten bundle we consider satisfy a notion of positivity either on
or near the strata corresponding to irrelevant graphs.
Hence as we vary sections, zeroes
cannot escape off of $\oCM^W_{\Gamma}$ through these boundary strata.
The sections of $\cW$ we consider must satisfy an inductive boundary
condition over the remaining strata indexed by elements of $\partial^0\Gamma$ which
are relevant. This is encoded in the notion of a
family of canonical multisections, see Definition \ref{def:family of canonical}.

\begin{definition}\label{def:special kind of graded graphs}
Let $\Gamma$ be a pre-graded $W$-spin graph. It is \emph{irrelevant}
if one of the following holds:
\begin{enumerate}
\item
\label{irrelevant:item 1}
$H^{CB}(\Gamma)\neq \emptyset$;
\item
\label{irrelevant:item 2}
$\Gamma$ has a positive boundary half-edge;
\item
\label{irrelevant:item 3}
$\Gamma$ has a vertex $v$ such that there exists an $i\in[a]$ such that, for any boundary half-edge $h$ attached to the vertex $v$, one has that \[\tw_i(h)=\alt_i(h)=0.\]
\end{enumerate}
Otherwise the graph is \emph{relevant}.
\end{definition}

\begin{figure}

\begin{subfigure}{.45\textwidth}
  \centering
\vspace{.24cm}
\begin{tikzpicture}[scale=.4]

	\draw (0,0) circle (.2cm);
	
	\draw[dashed] (0,-0.2) -- (0,-2.2);
    \node[below] at (0,-2.2) {$t_{x_2}$};
	
	\draw (-.1414,.1414) -- (-1.6414, 1.6414);
	\draw[dashed] (-.1414,-.1414) -- (-1.6414, -1.6414);
    \draw[dashed] (.1414,-.1414) -- (1.6414, -1.6414);
	
	%\node[above] at (5.6414, 1.6414) {$t_{x_1}$};
	\node[below] at (1.6414, -1.6414) {$t_{x_1}$};
	
	\node[above] at (-1.6414, 1.6414) {$t_{z_1}$};
	\node[below] at (-1.6414, -1.6414) {$t_{x_3}$};
	
	\draw (0,.2) -- (0, 2.2);
	\draw (.1414,.1414) -- (1.6414, 1.6414);
	\node[above] at (0, 2.2) {$t_{z_2}$};
	\node[above] at (1.6414, 1.6414) {$t_{z_3}$};
	
	\node[left] at (0,0) {$v$};
\end{tikzpicture}
\vspace{0.15cm}

  \caption{The $W$-spin graph  $\Gamma$}
\end{subfigure}
\qquad
\begin{subfigure}{.45\textwidth}
  \centering
\vspace{.24cm}
\begin{tikzpicture}[scale=.4]

	\draw (0,0) circle (.2cm);
	\draw (4,0) circle (.2cm);
	
	\draw[dashed] (0.2,0) -- (3.8,0);
	\node[below] at (2,0) {$e_{q}$};
	
	\draw[dashed] (4,-.2) -- (4, -2.2);
	\draw[dashed] (4.1414,-.1414) -- (5.6414, -1.6414);
	
	\draw (-.1414,.1414) -- (-1.6414, 1.6414);
	\draw[dashed] (-.1414,-.1414) -- (-1.6414, -1.6414);
	
	\node[below] at (4, -2.2) {$t_{x_2}$};
	\node[below] at (5.6414, -1.6414) {$t_{x_1}$};
	
	\node[above] at (-1.6414, 1.6414) {$t_{z_1}$};
	\node[below] at (-1.6414, -1.6414) {$t_{x_3}$};
	
	\draw (0,.2) -- (0, 2.2);
	\draw (.1414,.1414) -- (1.6414, 1.6414);
	\node[above] at (0, 2.2) {$t_{z_2}$};
	\node[above] at (1.6414, 1.6414) {$t_{z_3}$};
	
	\node[left] at (0,0) {$v_1$};
	\node[right] at (4,0) {$v_2$};
\end{tikzpicture}
\vspace{0.15cm}

  \caption{Exchangeable graph $\Lambda_1 \in \partial^{\xch}\Gamma$}
\end{subfigure}
\qquad
\begin{subfigure}{.45\textwidth}
  \centering
\vspace{.24cm}
\begin{tikzpicture}[scale=0.4]

	\draw (0,0) circle (.2cm);

	\draw (4,0) circle (.2cm);
	
	\draw (4,.2) -- (4,2);

	\draw[dashed] (0.2,0) -- (3.8,0);
	\node[below] at (2,0) {$e_{q}$};
	\node[above] at (4,2) {$t_{z_2}$};

	\draw[dashed] (.1414,-.1414) -- (1.6414, -1.6414);
    \node[below] at (1.6414, -1.6414) {$t_{x_3}$};
	
	\draw (.1414,.1414) -- (1.6414, 1.6414);
	\draw[dashed] (4.1414,-.1414) -- (5.6414, -1.6414);
	
	\draw (-.1414,.1414) -- (-1.6414, 1.6414);
	\draw[dashed] (-.1414,-.1414) -- (-1.6414, -1.6414);
	
	\node[above] at (1.6414, 1.6414) {$t_{z_3}$};
	\node[below] at (5.6414, -1.6414) {$t_{x_2}$};
	
	\node[above] at (-1.6414, 1.6414) {$t_{z_1}$};
	\node[below] at (-1.6414, -1.6414) {$t_{x_1}$};
	
	% \draw (-.1414,6.1414) -- (-1.6414, 7.6414);
	% \draw (.1414,6.1414) -- (1.6414, 7.6414);
	% \node[above] at (-1.6414, 7.6414) {$t_{z_3}$};
	% \node[above] at (1.6414, 7.6414) {$t_{z_4}$};
	
	\node[below] at (0,0) {$v_1$};
	\node[below] at (4,0) {$v_2$};
\end{tikzpicture}
\vspace{0.15cm}

  \caption{Positive graph $\Lambda_2\in \partial^+\Gamma$}
\end{subfigure}
\qquad
\begin{subfigure}{.45\textwidth}
  \centering
\vspace{.24cm}
\begin{tikzpicture}[scale=0.4]

	\draw (0,0) circle (.2cm);

	\draw (4,0) circle (.2cm);
	
	\draw (4,.2) -- (4,2);

	\draw[dashed] (0.2,0) -- (3.8,0);
	\node[below] at (2,0) {$e_{q}$};
	\node[above] at (4,2) {$t_{z_2}$};

	\draw[dashed] (.1414,-.1414) -- (1.6414, -1.6414);
    \node[below] at (1.6414, -1.6414) {$t_{x_3}$};
	
	\draw (4.1414,.1414) -- (5.6414, 1.6414);
	\draw[dashed] (4.1414,-.1414) -- (5.6414, -1.6414);
	
	\draw (-.1414,.1414) -- (-1.6414, 1.6414);
	\draw[dashed] (-.1414,-.1414) -- (-1.6414, -1.6414);
	
	\node[above] at (5.6414, 1.6414) {$t_{z_3}$};
	\node[below] at (5.6414, -1.6414) {$t_{x_2}$};
	
	\node[above] at (-1.6414, 1.6414) {$t_{z_1}$};
	\node[below] at (-1.6414, -1.6414) {$t_{x_1}$};
	
	% \draw (-.1414,6.1414) -- (-1.6414, 7.6414);
	% \draw (.1414,6.1414) -- (1.6414, 7.6414);
	% \node[above] at (-1.6414, 7.6414) {$t_{z_3}$};
	% \node[above] at (1.6414, 7.6414) {$t_{z_4}$};
	
	\node[below] at (0,0) {$v_1$};
	\node[below] at (4,0) {$v_2$};
\end{tikzpicture}
\vspace{0.15cm}

  \caption{Relevant graph $\Lambda_3\in \partial^0\Gamma$}
\end{subfigure}

\caption{Graphs from Example~\ref{exmple: exchangeable and positive}}
\label{fig:exchange_and positive}
\end{figure}

\begin{ex}\label{exmple: exchangeable and positive}
We give two examples to help the reader understand positive graphs and exchangeable graphs. Consider the LG model $(x_1^4 + x_2^5, \mu_4 \times \mu_5)$. We consider the smooth connected $W$-spin graph $\Gamma$ depicted in Figure~\ref{fig:exchange_and positive}(A) with three internal tails $t_{z_1}, t_{z_2}, t_{z_3}$ and three boundary tails $t_{x_1}, t_{x_2}, t_{x_3}$ adjacent to an open vertex $v$. We will assume that $t_{x_1}$ is singly twisted with respect to $1$, $t_{x_2}$ is singly twisted with respect to $2$, and $t_{x_3}$ is a root (i.e., fully twisted). We suppose that the twist vectors for $t_{z_1}, t_{z_2},$ and $t_{z_3}$ are $(1,1)$, $(2,2)$, and $(2,3)$, respectively. One can check that $\Gamma$ is a graded $W$-spin graph as in Definition~\ref{def: graded W spin graph}(2).

One can then see in Figure~\ref{fig:exchange_and positive}(B) an example of an exchangeable graph $\Lambda_1 \in \partial^{\xch}\Gamma$. Lastly, one can check in Figure~\ref{fig:exchange_and positive}(C) that $\Lambda_2 \in \partial^+\Gamma$ by computing $(\tw(q_i), \alt(q_i))$ at the two half-nodes, as required in Definition~\ref{def:graph} to ensure that $\Lambda_2$ is a genus zero $W$-spin dual graph with a lifting. Indeed, the half-node $q_1$ adjacent to $v_1$ has $\tw(q_1) = (0,2)$ and $\alt(q_1) = (1,0)$ and $q_2$ adjacent to $v_2$ has $\tw(q_2) = (2,1)$ and $\alt(q_2) = (0,1)$. Thus, from Definition~\ref{def: positive with respect to grading}, we see that $q_1$ is positive for the $2^{nd}$ grading and $q_2$ is positive for the $1^{st}$ grading.

Lastly, we have in Figure~\ref{fig:exchange_and positive}(D) an example of a relevant graph $\Lambda_3 \in \partial^0\Gamma$. One can check, in similar notation above that the half-node $q_1$ adjacent to $v_1$ has $\tw(q_1) = (0,3)$ and $\alt(q_1) = (0,1)$ and $q_2$ adjacent to $v_2$ has $\tw(q_2) = (2,0)$ and $\alt(q_2) = (1,0)$.
\end{ex}

\section{Moduli, bundles and orientation}\label{sec: moduli}

\subsection{The moduli}
This subsection follows \S3.2 in \cite{BCT:I}.

Fan, Jarvis, and Ruan constructed a moduli space $\M_{g,n}^{\text{FJR}, W}$ consisting of compact stable $W$-spin orbicurves, for which they then provide an enumerative
theory. This moduli space does not precisely coincide with our moduli space $\M_{g,n}^W$ of closed $W$-spin curves. This is due to the treatment of anchors and gradings
which rigidify the curves, as well as choice of stabilizer group at orbifold points (see below). The following theorem is essentially carrying out the proof of Theorem 2.2.6 of \cite{FJR} in our context of
$W=\sum_{i=1}^n x_i^{r_i}$ with the maximal symmetry group, with some
slight differences which arise, for the most part, because of different
automorphism groups.

\begin{thm}
$\M_{g,n}^{W}$ is a smooth Deligne-Mumford stack with projective coarse moduli. Moreover, if we let $\operatorname{st}: \M_{g,n}^{W}\rightarrow \M_{g,n}$ be the morphism given by forgetting the spin bundles and orbifold structure, then $\operatorname{st}$ is a flat, proper, and quasi-finite (but not representable) morphism.
\end{thm}

 We point out the differences from our moduli space and the one outlined in \cite{FJR}. First, with our definitions, stabilizer groups of special points are always $\mu_d$, while Fan, Jarvis and Ruan choose a stabilizer group $\mu_e$ at a special point $x$
so that the representation of $\mu_e$ on the direct sum of the fibres of the spin bundles at $x$ is faithful. This changes the universal curve over $\M_{g,n}^W$. In addition, it changes automorphism groups of nodal curves which lie at the boundary of the moduli space (see Observation~\ref{obs:automorphism}), so that our moduli space is a root stack over the moduli space of \cite{FJR}.

However, there is a further distinction arising from twisted spin structures. Our universal twisted spin bundles are twists of the universal spin bundles of \cite{FJR}. Further, there may be an additional change in automorphism group arising due to the notion of grading (Definition \ref{def:closed grading}) at anchors with twist $r-1$, as the data of a grading partially rigidifies the curve. In this case, our moduli space is an unramified $r$-fold cover in the orbifold sense (but does not change the coarse moduli spaces) of the moduli space of \cite{FJR}, at least away from the boundary.

By stratifying the moduli space $\M_{g,n}^W$ by twists in the spin structure, it decomposes into open and closed substacks,
\begin{equation}
\label{eq:closedmoduli}
\M_{g,n}^W = \bigsqcup_{( (b_{11},\ldots,b_{1a}),\ldots,(b_{n1},\ldots,b_{na}))}
\M_{g, ((b_{11},\ldots,b_{1a}) \ldots, (b_{n1},\ldots,b_{na}))}^W,
\end{equation}
where $b_{ij}\in \{-1,0,\ldots, r_j-1\}$ for each $i,j$ and $\M_{g, ((b_{11},\ldots,b_{1a}) \ldots, (b_{n1},\ldots,b_{na}))}^W$ denotes the substack of graded closed
genus $g$ $W$-spin curves with $W$-spin structures having twist $(b_{i1}, \ldots,b_{ia})$ at the $i$th marked point for all $i$.  For any choice of
$\{(b_{i1},\ldots,b_{ia})\}_{i\in[n]}$ such that Equation~\eqref{eq:close_rank1_general} holds for all $j\in [a]$, we know that  the moduli space of compact  stable genus 0 $W$-spin orbicurves
$\M_{0, ((b_{11},\ldots,b_{1a}) \ldots, (b_{n1},\ldots,b_{na}))}^W$ with the corresponding prescribed twists is non-empty, by the results mentioned in Proposition~\ref{prop:existence_stable}. Moreover, in this case, the moduli space has coarse moduli isomorphic to $\M_{0,n}$ and generic additional isotropy $\prod_{i\in [a]} \Z/r_i\Z$. The isomorphism of coarse moduli is given by the smooth map $\text{For}_{\text{\text{spin}}}$ that forgets the spin structures.

To generalize the construction of the moduli space to the open setting, we introduce some new notation.
\begin{definition}\label{ModSpaceOfMarkedDiscs}
For a genus zero, anchored, pre-stable, smooth dual graph $\Gamma$, we denote by $\oCM_\Gamma$ the set of isomorphism classes of smooth marked disks whose dual graph $\Gamma'$ after smoothing gives $\Gamma$. We denote by $\oCM_\Gamma^W$ the set of all isomorphism classes of pre-graded $W$-spin disks whose corresponding dual graph $\Gamma'$ satisfies  $\smooth(\text{for}_{\text{spin}}(\Gamma'))=\Gamma$.  For a smooth pre-graded $W$-spin dual graph $\Gamma$ we denote by $\oCM_\Gamma^W$ the set of all isomorphism classes of graded $W$-spin disks whose corresponding graph, after smoothing, is $\Gamma$.

For some common cases, we introduce special notation. Recall Notation \ref{nn:repeating graphs}.
We write
\[\oCM_{0,B,I}=
\oCM_{\Gamma_{0,B,I}},\quad\oCM^W_{0,B,I}=
\oCM^W_{\Gamma_{0,B,I}},\quad\oCM_{0,B,I}^{\text{labeled}}=\oCM_{\Gamma_{0,B,I}^{\text{labeled}}},\quad \oCM_{0,B,I}^{W,\text{labeled}}=\oCM^W_{\Gamma_{0,B,I}^{\text{labeled}}}.\]
As before, we add the superscript `labeled' since in most of the article we shall assume that the boundary points are marked by $\emptyset$, but it is sometimes more convenient to assume that they are marked bijectively by $[k]\subset\Omega$.
As in Notation \ref{nn:repeating graphs}, when $B=[k],I=[l]$ we similarly define $\oCM_{0,k,l}$, $\oCM_{0,k,l}^{\text{labeled}}$,
$\oCM_{0,k,l}^W$, and $\oCM_{0,k,l}^{W,\text{labeled}}$.
We also write
$\M_{0,\{B_J\}_{J\subseteq{[a]}}, \{\vec{a}_j\}_{j\in [l]}}^{W},~\M_{0,\{k_J\}_{J\subseteq{[a]}}, \{\vec{a}_j\}_{j\in [l]}}^{W}$ for $\oCM_{\Gamma_{0,\{B_J\}_{J\subseteq{[a]}}, \{\vec{a}_j\}_{j\in [l]}}^W}^W,~\oCM_{\Gamma_{0,\{k_J\}_{J\subseteq{[a]}}, \{\vec{a}_j\}_{j\in [l]}}^W}^W,$ respectively, and $\M_{0,\{B_J\}_{J\subseteq{[a]}}, \{\vec{a}_j\}_{j\in [l]}}^{W,\text{labeled}},~\M_{0,\{k_J\}_{J\subseteq{[a]}}, \{\vec{a}_j\}_{j\in [l]}}^{W,\text{labeled}}$ for their labeled versions. When $W=x^r$ we write the superscript `$1/r$' instead of `$W$.' When $B_\emptyset=\emptyset$ or $k_\emptyset=0$, we omit it from the notation.
\end{definition}

\begin{rmk} The labeled versions of the moduli spaces will only be used in this section, mainly in Proposition \ref{prop:orbi_w_corners} and in
\S\ref{subsec:or} to create a consistent family of orientations on our moduli space. \end{rmk}

The moduli space $\M_{0,k,l}^{\textup{labeled}}$ of stable marked disks is considered in \cite{PST14}.  It is a smooth orientable manifold with corners in the sense of \cite{Joyce}, and its dimension is
\begin{equation}\label{real dim of open Moduli of discs}
\dim_{\R}(\M_{0,k,l}^{\textup{labeled}}) = \dim_{\R}(\M_{0,k,l})=k+2l-3.
\end{equation}
There is a set-theoretic decomposition analogous to \eqref{eq:closedmoduli},
\begin{equation}\label{DecompositionOfModuliOfDiscs}
\M_{0,k,l}^{W,\textup{labeled}} = \bigsqcup_{\substack{\{\vec{a}_j\}_{j\in [l]}, \ \{k_J\}_{J\subseteq {[a]}},\\\sum_{J\subseteq{[a]}}k_J=k}}\M_{0,\{k_J\}_{J\subseteq{[a]}}, \{\vec{a}_j\}_{j\in [l]}}^{W,\textup{labeled}}.
\end{equation}

If we set $k=\sum k_T$, then there is a forgetful map
\[\text{For}_{\text{\text{spin}}}: \M_{0,\{k_J\}_{J\subseteq{[a]}}, (\vec{a}_j)_{j\in [l]}}^{W,\textup{labeled}} \rightarrow \M_{0,k,l}^{\textup{labeled}},\]
or more generally
\[\text{For}_{\text{\text{spin}}}: \M_\Gamma^W \rightarrow \M_{\text{for}_{\text{spin}}(\Gamma)},\]
which forgets the spin structures and removes the orbifold structure at
the special points. By
Proposition~\ref{prop:existence_stable},
if the domain is non-empty, this map is a bijection at the level of
sets. We  use these bijections to give
the coarse moduli space of $\M_{0,k,l}^{W,\textup{labeled}}$ the structure of a manifold with corners.

This describes the underlying coarse moduli space of 
$\M_\Gamma^W$.  We now explain a procedure that defines an orbifold with corners structure on $\M_\Gamma^{W}$, in the sense of \cite[Section 3]{Zernik}.

\begin{prop}\label{prop:orbi_w_corners}
Let $\Gamma$ be a smooth pre-graded $W$-spin graph.
There exists a compact moduli space $\M_\Gamma^{W}$ of stable pre-graded $W$-spin disks whose dual graph smooths to $\Gamma$.
It is a smooth orbifold with corners of real dimension $|T^B(\Gamma)|+2|T^I(\Gamma)|-3|\Conn(\Gamma)|$. Its universal bundle admits a universal $r_i$-spin grading for each $i\in[a]$.

\end{prop}
\begin{proof}
We first establish the result in the case of $\M_{0,k,l}^{W,\text{labeled}}$.
What we do is completely analogous to \cite{BCT:I}, \S3.2 (which, in turn, follows the procedure performed in \cite[Section 2]{Zernik}, and sketched also in Section 1 of \cite{Zernik}). We have the following sequence of maps, with arrows
labeled by the step in which they are defined below:
\begin{equation}
\label{eq:OWCsequence}
\begin{tikzcd}[row sep=scriptsize, column sep = small]
\M_{0,k,l}^{W,\textup{labeled}} \arrow{r}{(E)}\arrow{d}{(F)} &	\widehat{\mathcal{M}}_{0,k,l}^{W,\textup{labeled}}\arrow{r}{(D)}  & 	\widetilde{\mathcal{M}}_{0,k,l}^{W,\textup{labeled}} 	\arrow{r}{(C)}  & \widetilde{\mathcal{M}}_{0,k,l}^{W,\Z_2,\textup{labeled}}\arrow{r}{(B)}  &
\overline{\mathcal{M}}_{0,k+2l}^{W,\Z_2,\textup{labeled}}\arrow{r}{(A)} &\overline{\mathcal{M}}_{0,k+2l}^{'W,\textup{labeled}}  \\
\M_{0,k,l}^{W} & &&&&&
\end{tikzcd}
\end{equation}
  Let us now define the moduli spaces appearing in \eqref{eq:OWCsequence}. We will go through them in reverse order.

{\bf Step (A):} The morphism $\overline{\mathcal{M}}_{0,k+2l}^{W,\Z_2,\textup{labeled}}\rightarrow\overline{\mathcal{M}}_{0,k+2l}^{'W,\textup{labeled}}$. The space $\overline{\mathcal{M}}_{0,k+2l}^{'W,\textup{labeled}}$ is the sub-orbifold of $\M_{0,k+2l}^{W}$ given by the conditions
\begin{enumerate}
\item The first $k$ markings, $w_1,\ldots,w_k$, have twists which satisfy
\[\forall j\in[k],i\in[a],~\tw_i(w_j)\in\{0,r_i-2\}.\]
\item For every $i\in[a]$
\[
\frac{\sum_{j\in[k+2l]}\tw_i(w_j)-(r_i-2)}{r_i}\equiv-1+\sum_{T:i\in T} k_T\mod 2.\]
\end{enumerate}
The latter condition comes from \eqref{eq:open_rank2_W}.
Inside this space, $\M_{0,k+2l}^{W, \Z_2,\textup{labeled}}$ is the fixed locus of the involution defined by
\[(C;w_1, \ldots, w_{k+2l}, \{S_i\}) \mapsto (\overline{C}; w_1, \ldots, w_k, w_{k+l+1},\ldots, w_{k+2l},w_{k+1},\ldots,w_{k+l}, \{\overline{S}_i\}),\]
where $\overline{C}$, $\overline{S}_i,~i\in[a]$ are the same as $C$, $S_i,~i\in[a]$, respectively, but with the conjugate complex structure.
As the moduli space $\M_{0,k+2l}^{W, \Z_2,\textup{labeled}}$ is
the fixed locus of the above anti-holomorphic involution, it
 has the structure of a real orbifold. A point in the fixed locus comes equipped with an involution $\phi: C \rightarrow C$ given by conjugation which is covered by involutions $\tilde \phi_i: S_i \rightarrow S_i$ for all $i$.
This moduli space parameterizes isomorphism types of marked spin spheres with a real structure, involutions $\phi$, $\tilde\phi_i,~i\in[a]$, and the prescribed twists. The moduli space $\M_{0,k+2l}^{W, \Z_2,\textup{labeled}}$ then maps to $\M_{0,k+2l}^{'W,\textup{labeled}}$, so it inherits a universal curve via pullback.
Note that in general  $\M_{0,k+2l}^{W, \Z_2,\textup{labeled}}$ is not a sub-orbifold of $\M_{0,k+2l}^{'W,\textup{labeled}}$ as isotropy is lost: We now cannot scale the $r$-spin structure by an arbitrary $r$th root of unity as its action must be the same under conjugation. This means that the only scaling that can happen is by $\pm1$ in the case where $r$ is even.

{\bf Step (B):} The morphism $\widetilde{\mathcal{M}}_{0,k,l}^{W,\Z_2,\textup{labeled}}\rightarrow\overline{\mathcal{M}}_{0,k+2l}^{W,\Z_2,\textup{labeled}}$. 
Take $N\hookrightarrow \overline{\mathcal{M}}_{0,k+2l}^{W,\Z_2,\textup{labeled}}$ to be the real simple normal crossing divisor consisting of
 curves with at least one boundary node.
Via the real hyperplane blowup of
\cite{Zernik}, we cut $\overline{\mathcal{M}}_{0,k+2l}^{W,\Z_2,\textup{labeled}}$ along $N$, yielding an orbifold with corners $\widetilde{\mathcal{M}}_{0,k,l}^{W, \Z_2,\textup{labeled}}$. The morphism in this step is then constructed by gluing the cuts described here.

{\bf Step (C):} The morphism $\widetilde{\mathcal{M}}_{0,k,l}^{W,\textup{labeled}}
\rightarrow\widetilde{\mathcal{M}}_{0,k,l}^{W,\Z_2,\textup{labeled}}$. From here, we define $\widetilde{\mathcal{M}}_{0,k,l}^{W,\textup{labeled}}$ to be the disconnected $2$-to-$1$ cover of $\widetilde{\mathcal{M}}_{0,k,l}^{W,\Z_2,\textup{labeled}}$. The generic point of the moduli space $\widetilde{\mathcal{M}}_{0,k,l}^{W,\textup{labeled}}$ corresponds to a smooth marked real spin sphere with a choice of a distinguished connected disk component of $C\setminus C^\phi$.  Equivalently, in the generic (smooth) situation, we are choosing an orientation for $C^{\phi}$.  It is important to note, however, that this choice can be uniquely continuously extended to nodal points, see \cite{Zernik}, \S2.6,
as opposed to being chosen independently for each boundary component.

{\bf Step (D):} The morphism $\widehat{\mathcal{M}}_{0,k,l}^{W,\textup{labeled}}\hookrightarrow\widetilde{\mathcal{M}}_{0,k,l}^{W,\textup{labeled}}$. Inside $\widetilde{\mathcal{M}}_{0,k,l}^{W,\textup{labeled}}$, we denote by $\widehat{\mathcal{M}}_{0,k,l}^{W,\textup{labeled}}$ the union of connected components such that the marked points points $w_{k+1},\ldots, w_{k+l}$ lie in the distinguished stable disk and, for even $r_i,$ the $i^{th}$ spin structure is compatible in the sense of Definition~\ref{def:lifting_compatible}. The morphism here is inclusion.

{\bf Step (E):} The morphism $\M_{0,k,l}^{W,\textup{labeled}}\rightarrow\widehat{\mathcal{M}}_{0,k,l}^{W,\textup{labeled}}$.  Here, $\M_{0,k,l}^{W,\textup{labeled}}$ is the cover of $\widehat{\mathcal{M}}_{0,k,l}^{W,\textup{labeled}}$ given by a choice of lifting. By Proposition~\ref{prop:graded_r_spin_prop}, the cover is of degree $2^e$, where $e$ is the number of even $r_i$.  Thus we have given the moduli space $\M_{0,k,l}^{W,\textup{labeled}}$ defined
in Definition~\ref{ModSpaceOfMarkedDiscs} the structure of an orbifold with corners. The proof of this fact is similar to the proof of the analogous claim, Theorem 2, in \cite{Zernik}, and will be omitted.

Over $\M_{0,k,l}^{W,\textup{labeled}}$, there is a universal curve whose fibers are compatible stable $W$-spin disks. Take a component from the decomposition given in Equation~\eqref{DecompositionOfModuliOfDiscs}, say $\oCM_{0,\{k_J\}_{J\subseteq[a]},\{\vec{a}_j\}_{j\in[l]} }^{W,\textup{labeled}}$. It has forgetful maps, for
each $i\in [a]$,
\begin{equation}
\label{eq:forget all but i}
\text{For}_{\text{spin }\neq i}:
\oCM_{0,\{k_J\}_{J\subseteq[a]},\{\vec{a}_j\}_{j\in[l]} }^{W,\textup{labeled}}
\rightarrow
\oCM^{1/{r_i},\textup{labeled}}_{0,\sum_{J|i\in J}k_J,\{a_{j,i}\}_{j\in[l]}}
\end{equation}
which forget all spin structures but the $i^{th}$ one, while also forgetting the consequently untwisted boundary points. We define the \emph{$r_i$-spin lifting} for $i\in[a]$ on $\M_{0,k,l}^{W,\textup{labeled}}$ as the pullbacks with respect to
these maps of the lifting on $\oCM^{1/{r_i},\textup{labeled}}_{0,\sum_{J|i\in J}k_J,\{a_{j,i}\}_{j\in[l]}}$, as in \cite{BCT:I}, Theorem 3.4.

{\bf Step (F):} The morphism $\overline{\mathcal{M}}_{0,k,l}^{W,\textup{labeled}}\rightarrow\overline{\mathcal{M}}_{0,k,l}^W$.
Finally, for any smooth pre-graded $W$-spin graph $\Gamma$, let $\Gamma'$ be the graph obtained from $\Gamma$ by allowing all possible cyclic orders at open vertices, and let $\widetilde\Gamma$ be the graph obtained from $\Gamma'$ by replacing the marking functions by injective marking functions. Then $\oCM^W_{\widetilde\Gamma}$ is the product of moduli spaces of the form $\oCM_{0,k,l}^{W,\text{labeled}}$. Finally, the moduli space $\oCM_\Gamma^W$ is isomorphic to the union of those connected components of $\oCM_{\widetilde{\Gamma}}^W/\text{Aut}(\Gamma')$ whose boundary cyclic orders correspond to those in $\Gamma$. The universal bundles and gradings are defined by pulling back from $\oCM_{\widetilde{\Gamma}}^W/\text{Aut}(\Gamma')$.
\end{proof}

\begin{rmk}
Henceforth, we usually denote a pre-graded or graded $W$-spin disk simply by $\Sigma$, meaning the preferred half along with the spin structure and marked points, suppressing most of the notation. We will write $\Sigma\in\M_{0,k,l}^{W}$ for the point $\Sigma$ represents in $\M_{0,k,l}^{W}$.
\end{rmk}
\subsubsection{Moduli associated to dual graphs}
We also want moduli spaces associated to dual graphs in which we keep track of cyclic orders
on boundary marked points. We define them as follows.

Given a connected stable genus zero twisted $W$-spin pre-graded dual graph $\Gamma$, recall the smoothing $\smooth\,\Gamma$ of $\Gamma$ defined in Definition \ref{def:smoothingraph}. As $
\smooth\,\Gamma$ is a smooth connected graph with one vertex, it is associated to a moduli space $\M_{\smooth\Gamma}^W$ as in Proposition \ref{prop:orbi_w_corners}.

\begin{definition}\label{def:moduliGraph}
Given a connected stable genus zero twisted $W$-spin pre-graded dual graph $\Gamma$, there is a closed embedding
\begin{equation}\label{def:iotaGamma}
\iota_{\Gamma}: \M^W_{\Gamma}\hookrightarrow \M^W_{\smooth\Gamma},\end{equation} where the general point $\Sigma$ of $\M^W_{\Gamma}$ is a graded open $W$-spin surface with dual graph $\Gamma=\Gamma(\Sigma)$. If $\Gamma$ is disconnected, then $\M^W_{\Gamma}$ is defined as the product of the moduli spaces $\M^W_{\Gamma_i}$ associated to its connected components $\Gamma_i$.
\end{definition}

We have that $\M^W_{\Gamma}$ is a  closed sub-orbifold with corners of $\M^W_{\smooth\Gamma}$, of (real) codimension \begin{equation}\label{codim computation} 2|E^I(\Gamma)|+|E^B(\Gamma)|+|H^{CB}(\Gamma)|.\end{equation}  Take $\CM^W_\Gamma$ to be the open subspace of $\M^W_{\Gamma}$ consisting of graded $W$-spin disks whose dual graph is $\Gamma$.

\begin{ex}\label{ex: hexagon}
Consider the moduli space of disks with one internal point $z_1$ and three marked boundary points $x_1, x_2, x_3$. By \eqref{real dim of open Moduli of discs}, we know that the real dimension of $\oCM_{0,3,1}$ and $\oCM^{\mathrm{labeled}}_{0,3,1}$
is $2$. The labeled moduli has two connected components, depending on the
cyclic order of the boundary marked points. Each connected
component is a hexagon and has
 six codimension 1 boundary strata and six codimension 2 boundary strata. We can see that there is no way to have a contracted boundary or internal node in this case. The dimension of each strata, computed using \eqref{codim computation}, is $3$ minus the number of irreducible components. We draw one of the connected
components of the labeled moduli space in Figure~\ref{fig:moduliExample}. In this figure, each stratum is accompanied by a depiction of a representative stable marked disk $\Sigma$. Note the unlabeled moduli space is a quotient of
the labeled moduli space by the symmetric group on three letters. This
identifies the two hexagons and further quotients to produce a bigon.
We suppress twists in the example, but they must satisfy the conditions outlined in Definition~\ref{def:graph} and satisfy Equations~\eqref{eq:open_rank1_W} and~\eqref{eq:open_rank2_W}.

\begin{figure}

  \centering
  \begin{tikzpicture}[scale=0.7]

\newdimen\Rad
   \Rad=5cm
   \draw[line width=0.25mm] (0:\Rad) \foreach \x in {60,120,...,360} {  -- (\x:\Rad) };
   \draw (5,0)[black, fill = black] circle (.1cm);
   \draw (-5,0)[black, fill = black] circle (.1cm);
   \draw (2.5, 4.33)[black, fill = black] circle (.1cm);
      \draw (-2.5, 4.33)[black, fill = black] circle (.1cm);
         \draw (2.5, -4.33)[black, fill = black] circle (.1cm);
            \draw (-2.5, -4.33)[black, fill = black] circle (.1cm);

 %Generic
  \draw (0,0) circle (1cm);
  \node at (0, 0) {$\bullet$};
\node[above] at (0,0) {$z_1$};
  \node at (1, 0) {$\bullet$};
\node[right] at (1,0) {$x_1$};
\node at (-.5, .866) {$\bullet$};
\node[above] at (-.5,.866) {$x_2$};
\node at (-.5, -.866) {$\bullet$};
\node[below] at (-.5,-.866) {$x_3$};

%First codim 1
  \draw (0,6) circle (1cm);
    \draw (0,8) circle (1cm);

  \node at (0, 6) {$\bullet$};
\node[above] at (0,6) {$z_1$};
  \node at (0, 5) {$\bullet$};
\node[below] at (0,5) {$x_1$};
\node at (.866, 8.5) {$\bullet$};
\node[right] at (.866,8.5) {$x_2$};
\node at (-.866, 8.5) {$\bullet$};
\node[left] at (-.866,8.5) {$x_3$};

%Second codim 1

  \draw (5.196,3) circle (1cm);
    \draw (6.928,4) circle (1cm);

  \node at (5.196,3) {$\bullet$};
\node[above] at (5.196,3) {$z_1$};
  \node at (6.428, 4.866) {$\bullet$};
\node[above] at (6.428, 4.866) {$x_3$};
\node at (7.79423, 4.5) {$\bullet$};
\node[right] at (7.79423, 4.5) {$x_2$};
\node at (7.428, 3.134) {$\bullet$};
\node[below] at (7.428, 3.134) {$x_1$};

%Third codim 1

  \draw (5.196,-3) circle (1cm);
    \draw (6.928,-4) circle (1cm);

  \node at (5.196,-3) {$\bullet$};
\node[above] at (5.196,-3) {$z_1$};
  \node at (7.794, -3.5) {$\bullet$};
\node[right] at (7.794, -3.5) {$x_2$};
  \node at (4.330, -2.5) {$\bullet$};
\node[above] at (4.330, -2.5) {$x_3$};
\node at (7, -5) {$\bullet$};
\node[below] at (7, -5) {$x_1$};

%Fourth codim 1

  \draw (0,-6) circle (1cm);
    \draw (0,-8) circle (1cm);

  \node at (0, -6) {$\bullet$};
\node[above] at (0,-6) {$z_1$};
  \node at (1, -8) {$\bullet$};
\node[right] at (1,-8) {$x_2$};
\node at (0, -9) {$\bullet$};
\node[below] at (0,-9) {$x_1$};
\node at (-1, -8) {$\bullet$};
\node[left] at (-1,-8) {$x_3$};
%Fifth codim 1

  \draw (-5.196,-3) circle (1cm);
    \draw (-6.928,-4) circle (1cm);

  \node at (-5.196,-3) {$\bullet$};
\node[above] at (-5.196,-3) {$z_1$};
  \node at (-7.794, -3.5) {$\bullet$};
\node[left] at (-7.794, -3.5) {$x_3$};
  \node at (-4.330, -2.5) {$\bullet$};
\node[above] at (-4.330, -2.5) {$x_2$};
\node at (-7, -5) {$\bullet$};
\node[below] at (-7, -5) {$x_1$};

%Sixth codim 1

  \draw (-5.196,3) circle (1cm);
    \draw (-6.928,4) circle (1cm);

  \node at (-5.196,3) {$\bullet$};
\node[above] at (-5.196,3) {$z_1$};
  \node at (-6.428, 4.866) {$\bullet$};
\node[above] at (-6.428, 4.866) {$x_2$};
\node at (-7.79423, 4.5) {$\bullet$};
\node[left] at (-7.79423, 4.5) {$x_3$};
\node at (-7.428, 3.134) {$\bullet$};
\node[below] at (-7.428, 3.134) {$x_1$};

%First Codim 2 %This will be between 1 and 2 above and continue clockwise

  \draw (3.25,5.629) circle (1cm);
    \draw (4.25,7.361) circle (1cm);
        \draw (5.25,9.093) circle (1cm);

            \node at (3.25,5.629) {$\bullet$};
\node[below] at (3.25,5.629) {$z_1$};

  \node at (5.116,6.861) {$\bullet$};
\node[right] at (5.116,6.861) {$x_1$};
\node at (6.25,9.093 ) {$\bullet$};
\node[right] at (6.25,9.093) {$x_2$};
\node at (4.75, 9.959) {$\bullet$};
\node[left] at (4.75, 9.959) {$x_3$};

%Second Codim 2

  \draw (6.5,0) circle (1cm);
    \draw (8.5,0) circle (1cm);
        \draw (10.5,0) circle (1cm);

          \node at (6.5,0) {$\bullet$};
\node[above] at (6.5,0) {$z_1$};
  \node at (8.5,1) {$\bullet$};
\node[above] at (8.5,1) {$x_3$};
\node at (11,.866 ) {$\bullet$};
\node[right] at (11,.866) {$x_2$};
\node at (11,-.866) {$\bullet$};
\node[right] at (11,-.866) {$x_1$};

%Third Codim 2

  \draw (3.25,-5.629) circle (1cm);
    \draw (4.25,-7.361) circle (1cm);
        \draw (5.25,-9.093) circle (1cm);

            \node at (3.25,-5.629) {$\bullet$};
\node[below] at (3.25,-5.629) {$z_1$};

  \node at (3.384,-7.861) {$\bullet$};
\node[left] at (3.384,-7.861) {$x_3$};
\node at (6.25,-9.093 ) {$\bullet$};
\node[right] at (6.25,-9.093) {$x_2$};
\node at (4.75, -9.959) {$\bullet$};
\node[left] at (4.75, -9.959) {$x_1$};

%Fourth Codim 2

  \draw (-3.25,-5.629) circle (1cm);
    \draw (-4.25,-7.361) circle (1cm);
        \draw (-5.25,-9.093) circle (1cm);

            \node at (-3.25,-5.629) {$\bullet$};
\node[below] at (-3.25,-5.629) {$z_1$};

  \node at (-3.384,-7.861) {$\bullet$};
\node[right] at (-3.384,-7.861) {$x_2$};
\node at (-6.25,-9.093 ) {$\bullet$};
\node[left] at (-6.25,-9.093) {$x_3$};
\node at (-4.75, -9.959) {$\bullet$};
\node[right] at (-4.75, -9.959) {$x_1$};

%Fifth Codim 2

  \draw (-6.5,0) circle (1cm);
    \draw (-8.5,0) circle (1cm);
        \draw (-10.5,0) circle (1cm);

                  \node at (-6.5,0) {$\bullet$};
\node[above] at (-6.5,0) {$z_1$};
  \node at (-8.5,1) {$\bullet$};
\node[above] at (-8.5,1) {$x_2$};
\node at (-11,.866 ) {$\bullet$};
\node[left] at (-11,.866) {$x_3$};
\node at (-11,-.866) {$\bullet$};
\node[left] at (-11,-.866) {$x_1$};

%Sixth Codim 2

  \draw ( -3.25,5.629) circle (1cm);
    \draw (-4.25,7.361) circle (1cm);
        \draw (-5.25,9.093) circle (1cm);

            \node at (-3.25,5.629) {$\bullet$};
\node[below] at (-3.25,5.629) {$z_1$};

  \node at (-5.116,6.861) {$\bullet$};
\node[left] at (-5.116,6.861) {$x_1$};
\node at (-6.25,9.093 ) {$\bullet$};
\node[left] at (-6.25,9.093) {$x_3$};
\node at (-4.75, 9.959) {$\bullet$};
\node[right] at (-4.75, 9.959) {$x_2$};

\end{tikzpicture}
 \caption{Moduli space of stable disks with 3 boundary markings and one internal marking.}
\label{fig:moduliExample}
\end{figure}
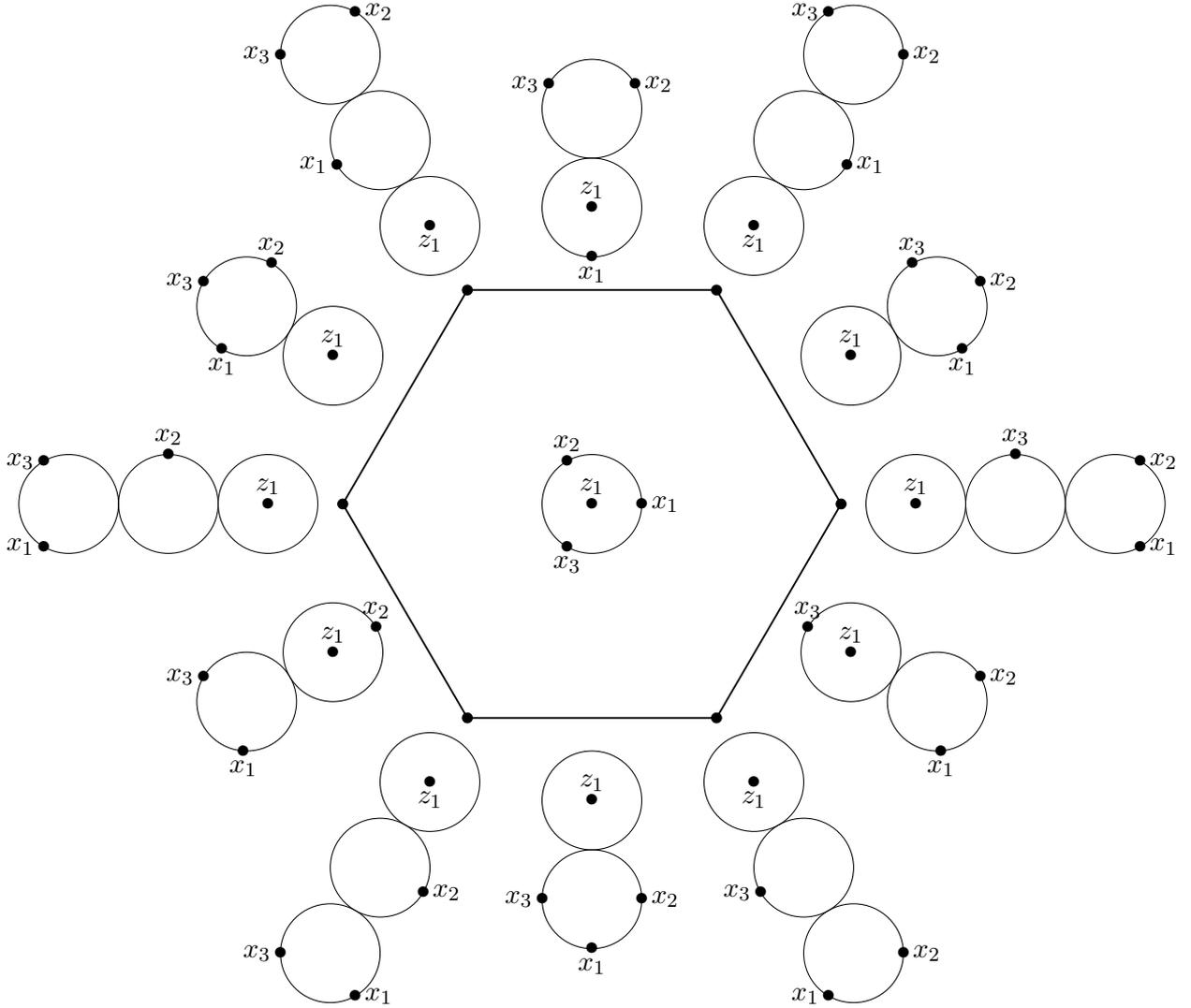
\end{ex}

There are forgetful maps between some of the $\oCM_{\Gamma}^W$. Indeed, we note that marked points can only be forgotten if all their twists are zero, since \eqref{eq:open_rank1_general} must be preserved. Moreover, boundary marked points can only be forgotten if they are in addition non-alternating, since the grading must descend to the moduli space with fewer marked points.  With this in mind, we define:

\begin{definition}
\label{def:for non-alt}
Given a (not necessarily connected) stable genus zero twisted $W$-spin pre-graded dual graph, we define
\[\text{For}_{\text{non-alt}}: \M_{\Gamma}^W \rightarrow \M_{\Gamma'}^W\]
by forgetting all non-alternating boundary marked points with $\tw=\vec{0}$. However, this may create unstable irreducible components. We repeatedly contract any unstable irreducible component which is not  partially stable with an internal
marked point.
This process might create new boundary marked points which were formerly boundary half-nodes. These new boundary marked points may be non-alternating with $\tw=\vec{0}$. In this case, we repeat the process. If the process ends with some unstable connected components which are not partially stable, then we remove them. We update the cyclic orders data in the natural way.

In the case where the graph $\Gamma'$ is partially stable, we define its moduli to be the moduli of the graph obtained by removing the partially stable irreducible
components. If the resulting graph is empty, we take the moduli space to be a point. If $\Gamma$ is a twisted $W$-spin dual graph with a lifting, we analogously define the graph $\Gamma'=\text{for}_{\text{non-alt}}(\Gamma)$ (which, again, may include partially stable components).
\end{definition}

\begin{definition}
\label{def:for spin}
Let $\Gamma$
be a pre-graded $W$-spin graph. The graded $r_i$-spin graph $\text{for}_{\text{spin}\neq i}\Gamma$ is an $r_i$-spin graph with a lifting obtained from $\Gamma$ by forgetting the additional $\tw_j,\alt_j$ data, for $j\neq i$ and from forgetting all tails $t$ with $(\tw_i(t),\alt_i(t))=(0,0)$.
We obtain a map
\[\text{For}_{\text{spin}\neq i}:\oCM_{\Gamma}^W\to\oCM^{1/r_i}_{\text{for}_{\spinqi}\Gamma}\]
induced by the map \eqref{eq:forget all but i}.
\end{definition}

\subsection{The bundles}
\subsubsection{The Witten bundles and their properties}\label{subsec:WittenBundle}
For graded closed $\Gamma$ and $i \in [a]$, we define the Witten bundle \[\cW_{i}=\cW_{i,\Gamma} = (R^1\pi_*\mathcal{S}_i)^{\vee},\] where $\pi: \mathcal{C} \rightarrow \M_{0,n}^{W}$ is the universal family and $\mathcal{S}_i \rightarrow \mathcal{C}$ are the universal twisted spin bundles.

\begin{obs}
\label{obs:twist-1}
As stated in Remark~\ref{rmk:NegTwist}, we can have at most one marked point with twist $-1$ with respect to each $r_i$-spin structure, and as a consequence
the degree of each $\mathcal{S}_i$ is negative. Thus
$R^0\pi_*\mathcal{S}_i = 0$ and hence $(R^1\pi_*\mathcal{S}_i)^{\vee}$
is a bundle.
This observation, which is crucial for the normalization operation defined below, explains our choice of allowable twist in conditions \eqref{c1} and \eqref{c2} of Definition~\ref{def:graph}.  
We note that if such a point exists, it must be the anchor, and whether or not
the point has twist $-1$  for the $i^{th}$ spin bundle will depend on $i$.
\end{obs}

These are complex vector bundles with fibers
\[H^1(C, S_i)^{\vee} \cong H^0(C, J_i),\]
of complex ranks \begin{equation}\label{eq:graded_closed_rank}\frac{\sum_{j\in I} \tw_i(z_j) - (r_i-2)}{r_i}.\end{equation}

Equivalently, the bundle $\cW_i$ can be constructed by pulling back the Witten bundle with respect to the $r_i$-spin structure via the map $\text{For}_{\text{spin}\neq i}$ from Definition~\ref{def:for spin}. 
Write \[\cW=\cW_\Gamma=\bigoplus_{i\in[a]}\cW_{i,\Gamma}.\]

Following \cite{BCT:I}, we now define the open Witten bundles. First, recalling the sequence of maps in \eqref{eq:OWCsequence}, we can construct the corresponding unlabeled versions by quotienting by the appropriate automorphisms, as dictated in Step (F) of Proposition \ref{prop:orbi_w_corners}:

\begin{equation}\label{unlabelled moduli}\begin{tikzcd}[row sep=scriptsize, column sep = small]
\M_{0,k,l}^{W,\textup{labeled}} \arrow{r}{}\arrow{d}{} &	\widehat{\mathcal{M}}_{0,k,l}^{W,\textup{labeled}}\arrow{r}{}\arrow{d}{}  & 	\widetilde{\mathcal{M}}_{0,k,l}^{W,\textup{labeled}} 	\arrow{r}{}\arrow{d}{}  & \widetilde{\mathcal{M}}_{0,k,l}^{W,\Z_2,\textup{labeled}}\arrow{r}{} \arrow{d}{} &
\overline{\mathcal{M}}_{0,k+2l}^{W,\Z_2,\textup{labeled}}\arrow{r}{}\arrow{d}{}  &\overline{\mathcal{M}}_{0,k+2l}^{'W,\textup{labeled}} \arrow{d}{}  \\
\M_{0,k,l}^{W} \arrow{r}{} &	\widehat{\mathcal{M}}_{0,k,l}^{W}\arrow{r}{} & 	\widetilde{\mathcal{M}}_{0,k,l}^{W} 	\arrow{r}{} & \widetilde{\mathcal{M}}_{0,k,l}^{W,\Z_2}\arrow{r}{} &
\overline{\mathcal{M}}_{0,k+2l}^{W,\Z_2}\arrow{r}{} &\overline{\mathcal{M}}_{0,k+2l}^{'W}
\end{tikzcd}
\end{equation}

We denote by $\varpi:\M_{0,k,l}^{W} \to \overline{\mathcal{M}}_{0,k+2l}^{W,\Z_2}$ the map given by composing the appropriate maps in~\eqref{unlabelled moduli}.  Abusing notation somewhat, denote by $\pi: \mathcal{C} \rightarrow \M^{W,\Z_2}_{0,k+2l}$ the universal curve over the moduli space of real spin spheres defined above (after setting the markings of the first $k$ points to $\emptyset$), and by $\cS_i \rightarrow \mathcal{C},~i\in[a],$ the universal spin bundles.  There are universal involutions
\[\phi: \mathcal{C} \rightarrow \mathcal{C}\]
and
\[\widetilde{\phi}_i: \cS_i \rightarrow \cS_i,\]
which induce an involution on $R^1\pi_*\mathcal{S}_i$.
Set $\mathcal{J}_i:= \mathcal{S}^{\vee}_i \otimes \omega_{\pi}$. Define
\[\mathcal{W}^{\text{pre}}_i := (R^0\pi_*\mathcal{J}_i)_+ =(R^1\pi_*\mathcal{S}_i)^{\vee}_-\]
to be the real vector bundle of $\widetilde{\phi}_i$-invariant sections of $J_i$. Note here the second equality uses Serre duality, under which invariant sections become anti-invariant.
The real rank of $\mathcal{W}^{\text{pre}}_i$ is seen in Equation (4.2) of \cite{BCT:I} to be
\begin{equation}\label{eq:graded_open_rank}
\frac{\sum_{j=1}^{2l} \tw_i(z_j)+(-1+\sum_{T:i\in T}k_T)(r_i-2)}{r_i}.
\end{equation}

\begin{definition}
\label{def:open Witten bundles}
The \emph{open Witten bundles} are defined on $\M^{W}_{0,k,l}$ by $\cW_i=\varpi^*\mathcal{W}^{\text{pre}}_i$. 
For an open $\Gamma,$ let $\iota_\Gamma:\oCM_\Gamma\to\oCM_{\smooth\Gamma}$
 be the embedding as described in Definition~\ref{def:moduliGraph}. We write $\cW_{i,\Gamma}=\iota_\Gamma^*\cW_{i,\smooth\Gamma}$ and 
\[
\cW_\Gamma=\bigoplus_{i\in[a]}\cW_{i,\Gamma}.
\]
 \end{definition}

Let $\Gamma$
be a pre-graded $W$-spin graph.
There is an identification of Witten bundles under the map $\text{For}_{\text{non-alt}}$ that forgets all non-alternating boundary marked points with $\tw = \vec{0}$. We will use the following observation extensively:

\begin{obs}\label{obs:forgetful}
If $\Gamma'$ is obtained from $\Gamma$ by forgetting all tails $t$ with $(\tw(t),\alt(t))=(\vec{0},\vec{0}),$
then
\begin{equation}\label{eq:forgetful_witten_0}\text{For}_{\text{non-alt}}^*(\cW_{\Gamma'})\backsimeq \cW_{\Gamma}\end{equation}
where $\text{For}_{\text{non-alt}}:\oCM^W_{\Gamma}\to\oCM^W_{\Gamma'}$ is
defined in Definition \ref{def:for non-alt}. See \cite{BCT:I}, (4.3).

Applying this in the $r_i$-spin case, we thus have a canonical isomorphism
\[\text{For}_{\text{spin}\neq i}^*(\cW_{i,\text{for}_{\spinqi}\Gamma})\backsimeq \cW_{i,\Gamma}\]
where $\text{For}_{\text{spin}\neq i}$ is defined in Definition
\ref{def:for spin}. 
Here, the graph $\text{for}_{\spinqi}\Gamma$ is the graded $r_i$-spin graph obtained from $\Gamma$ by
first forgetting the data $\tw_j,\alt_j$ for $j\neq i,$ and second forgetting the boundary tails $h$ with $\tw_i(h)=\alt_i(h)=0$.
\end{obs}

Another important property is the behaviour of the Witten bundles under restriction to strata. We follow \S4.2 of \cite{BCT:I} (although these properties are well-known already from Witten's original work \cite{Witten93}).

Let $\Gamma$ be a pre-graded $W$-spin graph, and let $\widehat{\Gamma} = \detach_{N}(\Gamma)$ for some set $N \subset E(\Gamma) \cup H^{CB}(\Gamma)$ of edges and contracted boundary tails, see Definition \ref{def:detach}.
Unlike the moduli space of disks, the moduli space of $W$-spin disks does not always have a gluing map $\M_{\widehat\Gamma}^{W} \rightarrow \M_{0,k,l}^{W}$, because there is no canonical way to glue the fibers of $S_i,~i\in[a]$ at the internal nodes.  Instead, we consider the following diagram of morphisms:
\begin{equation}\begin{tikzcd}\label{eq:instead_of_product}
\M_{\forg_{\text{spin}}(\widehat\Gamma)} \times_{\M_{\forg_{\text{spin}}(\Gamma)}} \M_{\Gamma}^{W}  \arrow{d}{q} \arrow{r}{\mu} &  \M_{\Gamma}^{W} \arrow{r}{\iota_\Gamma} & \M_{0,k,l}^{W}  \\
\M_{\widehat\Gamma}^{W} & &
\end{tikzcd}
\end{equation}
Here, $\M_{\forg_{\text{spin}}(\Gamma)} \subset \M_{0,k,l}$ is the moduli space of marked disks
with dual graph $\Gamma$, and $\M_{\forg_{\text{spin}}(\widehat\Gamma)}$ is analogously defined.  Recalling~\eqref{eq: normalised twisted spin structure}, we can define the morphism $q$ to send the spin structure $S_j,~j\in[a]$ to
\begin{equation}
\label{eq:hatS}
\widehat{S}_j = \NNN^*S_j \otimes \O\left(-\sum_{p \in \RRR_j} [p]\right),
\end{equation}
where $\NNN: \widehat{C} \rightarrow C$ is the normalization map and $\RRR_j$ is the subset of all half-nodes $p$ of multiplicity $0$ with respect to the $r_j$-spin structure which are not anchors in the normalization. 
Note that $\widehat{S}_j$ then has the correct twists for $\M_{\widehat\Gamma}^{W}$. In particular, it has $\tw_j=r_j-1$ at the Ramond half-nodes corresponding to half-edges in $\RRR_j$ and $\tw_j=-1$ at the other Ramond half-nodes.

Since the map $\M_{\forg_{\text{spin}}(\hat\Gamma)} \to \M_{\forg_{\text{spin}}(\Gamma)}$ is an isomorphism, the projection map $\mu$ is an isomorphism. We will distinguish however between the domain and codomain of $\mu$ because they have different universal objects. We note that $q$ is not an isomorphism, in general. Indeed, the automorphism groups of $\M_\Gamma^W$ and $\M_{\hat\Gamma}^W$ differ when one computes using Observation~\ref{obs:automorphism}.  However, in the strata of $\M_{\Gamma}$ where $\Gamma$ has no internal edges, we have that $q$ is an isomorphism when restricted to the strata (see \cite[Remark 4.5]{BCT:I}). It may be of some comfort to the reader for us to remark that this is an isomorphism on the level of coarse moduli spaces.

\begin{nn}\label{nn:Detach}
For $\Gamma,N$ as above, we denote the map $q\circ \mu^{-1}:\M_{\Gamma}^{W}\to\M_{\widehat\Gamma}^{W}$ by $\Detach_{N}$. When $N=\{e\}$ is a singleton, we denote this map by $\Detach_{e}$.
\end{nn}

There are two natural universal curves over the fiber product $\M_{\forg_{\text{spin}}(\widehat\Gamma)} \times_{\M_{\forg_{\text{spin}}(\Gamma)}} \M_{\Gamma}^{W}$: we define $\mathcal{C}_{\Gamma}$ by the fiber diagram

\begin{equation*}\begin{tikzcd}
\mathcal{C}_{\Gamma} \arrow{r}{} \arrow{d}{\pi} & \mathcal{C}\arrow{d}{} \\
\M_{\forg_{\text{spin}}(\widehat\Gamma)} \times_{\M_{\forg_{\text{spin}}(\Gamma)}} \M_{\Gamma}^{W} \arrow{r}{\iota_\Gamma \circ \mu} & \M_{0,k,l}^{W}
\end{tikzcd}
\end{equation*}
and $\mathcal{C}_{\widehat{\Gamma}}$ by the fiber diagram
\begin{equation*}\begin{tikzcd}
\mathcal{C}_{\widehat{\Gamma}} \arrow{r}{} \arrow{d}{\hat{\pi}} & \widehat {\mathcal{C}}\arrow{d}{} \\
\M_{\forg_{\text{spin}}(\widehat\Gamma)} \times_{\M_{\forg_{\text{spin}}(\Gamma)}} \M_{\Gamma}^{W} \arrow{r}{q} & \M^W_{\widehat{\Gamma}}
\end{tikzcd}
\end{equation*}
in which $\widehat{\mathcal{C}}$ is the universal curve over $\M_{\widehat\Gamma}^{W}$.  There are universal bundles $\mathcal{S}_j$ and $\widehat{\mathcal{S}}_j,~j\in[a]$ on these two universal curves, and a universal normalization morphism
\[\NNN: \mathcal{C}_{\widehat\Gamma} \rightarrow \mathcal{C}_{\Gamma}\]
relating them via \eqref{eq:S Shat def}. Let $\mathcal{J}_i$ be the universal twisted spin bundle corresponding to the twisted spin bundle  $J_i$ for $i\in [a]$.

We can now state the decomposition properties of the Witten bundle.  We state the properties in the case where $N = \{e\}$ for a single edge $e$, but all can be readily generalized to the setting where more than one edge is detached.

\begin{prop}
\label{pr:decomposition}
Let $\Gamma$ be a pre-graded $W$-spin graph.
Suppose that $\Gamma$ has a single edge $e$, so the general point of $\M_{\Gamma}^{W}$ is a stable $W$-spin disk with two components $C_1$ and $C_2$ meeting at a node $p$.  Let $\widehat{\Gamma} = \detach_e(\Gamma)$.  Let $\cW_i$ and $\widehat\cW_i$ denote the $i^{th}$ Witten bundle summand on $\M_{0,k,l}^{W}$ and $\M_{\widehat\Gamma}^{W}$, respectively.

The Witten bundle decomposes as follows along the node $p$:
\begin{enumerate}[(i)]
\item\label{it:NS} If $e$ is Neveu--Schwarz for the $i^{th}$ bundle $S_i$,
then $\mu^*\iota_{\Gamma}^*\cW_i = q^*\widehat\cW_i$.
\item\label{it:Ramondbdryedge} If $e$ is a Ramond boundary edge for the $i^{th}$ bundle $S_i$, then there is an exact sequence
\begin{equation}
\label{eq:decompses}0 \rightarrow \mu^*\iota_{\Gamma}^*\mathcal{W}_i \rightarrow q^*\widehat{\mathcal{W}}_i \rightarrow \TTT_+ \rightarrow  0,
\end{equation}
where $\TTT_+$ is a trivial real line bundle.
\item\label{it:internalRamondWitten} If $e$ is a Ramond internal edge for the $i^{th}$ bundle $S_i$, then it connects two vertices corresponding to components $\mathcal{C}_1$ and $\mathcal{C}_2$. Either both components are closed or exactly one is open. Without loss of generality, let $\mathcal{C}_1$ be the component which is open or the closed component containing the anchor of $\Gamma$. Write $q^*\widehat\cW_i = \widehat\cW_i^1 \boxplus \widehat\cW_i^2$, where $\widehat\cW_i^1$ is the $i^{th}$ Witten summand on the component $\mathcal{C}_1$, defined using the spin bundle $\widehat\cS_i|_{\mathcal{C}_1}$, and $\widehat\cW_i^2$ is the $i^{th}$ Witten summand on the component $\mathcal{C}_2$.  Then:
\begin{enumerate}
\item There is an exact sequence
\begin{equation}
\label{eq:decompses2}
0 \rightarrow \widehat\cW_i^2 \rightarrow \mu^*\iota_{\Gamma}^*\cW_i \rightarrow \widehat\cW_i^1 \rightarrow 0.
\end{equation}
\item
If $\widehat\Gamma'$ is defined to agree with $\widehat\Gamma$ except that the twist at each tail which corresponds to a half-edge of
$e$ is $r_i-1$,
and $q': \M_{\forg_{\textup{spin}}(\widehat{\Gamma})} \times_{\M_{\forg_{\textup{spin}}(\Gamma)}} \M_{\Gamma}^{W} \rightarrow \M_{\widehat\Gamma'}^{W}$ is defined analogously to $q$, then there is an exact sequence
\begin{equation}
\label{eq:decompses3}
0 \rightarrow \mu^*\iota_{\Gamma}^*\mathcal{W}_i \rightarrow (q')^*\widehat{\mathcal{W}}'_i \rightarrow \TTT \rightarrow 0,
\end{equation}
where $\widehat{\cW}'_i$ is the $i^{\text{th}}$ Witten bundle associated to $\widehat\Gamma'$ and $\TTT$ is  a line bundle whose $r^{\text{th}}$ power is trivial.
\end{enumerate}
\item\label{it:cont_bdry_tail} Suppose that $\Gamma$ has a single vertex, no edges, and a contracted boundary tail $t$, and let $\widehat{\Gamma} = \detach _t(\Gamma)$.  If $\cW_i$ and $\widehat\cW_i$ denote the $i^{th}$ Witten summands on $\M_{0,k,l}^{W}$ and $\M_{\widehat\Gamma}^{W}$, respectively, then the sequence \eqref{eq:decompses} also holds in this case.
\end{enumerate}
\end{prop}
This is Proposition 4.7 in \cite{BCT:I}. It should be stressed that in the presence of internal edges these decomposition properties do not respect the automorphism group actions on the Witten bundles. Despite this subtlety, a multisection on $\widehat{\cW}$ canonically induces a multisection of $\cW$, see Remark 2.7 in \cite{BCT:II}.

\subsubsection{The tautological lines and direct sums of bundles}

For each $i \in I$, a cotangent line bundle $\mathbb{L}_i$ is defined on the moduli space of stable marked disks as the line bundle whose fiber over $(C, \phi, \Sigma, \{z_i\}, \{x_j\}, m^I, m^B)$ is the cotangent line $T^*_{z_i}\Sigma$ at the interior marked point $z_i$.
We define cotangent line bundles $\mathbb{L}_i$ on $\M^W_{0,k,l}$ by pullback under the morphism forgetting the graded spin structure, and, for any graded graph $\Gamma$, we let $\mathbb{L}_i^{\Gamma}$ be the pullback of $\mathbb{L}_i$ to $\M_{\Gamma}$.

The following important observations are discussed further in \cite[Section 3.5]{PST14}:
\begin{obs}\label{isomorphism forgetting non-alt for descendents}
Let $\Gamma$ be a pre-graded $W$-spin graph.
\begin{enumerate}[(i)]
\item If $i$ is a marking of an internal tail of $\Lambda\in\Conn(\detach(\Gamma))$,
\[\mathbb{L}_i^{\Gamma} = \pi^*\mathbb{L}_i^{\Lambda},\]
where $\pi: \oCM^W_{\Gamma} \rightarrow \oCM^W_{\Lambda}$ is the
composition of the map $\Detach_{E(\Gamma)}:\oCM^W_{\Gamma}
\rightarrow \oCM^W_{\detach(\Gamma)}$ of Notation \ref{nn:Detach}
with the projection to the factor $\oCM^W_{\Lambda}$.
\item If $\Gamma' = \text{for}_{\text{non-alt}}(\Gamma)$, then there exists a canonical morphism
\[t_{\Gamma} : \text{For}_{\text{non-alt}}^*\mathbb{L}_i^{\Gamma'} \rightarrow \mathbb{L}_i^{\Gamma}.\]
This morphism vanishes identically on the strata
where the component containing $z_i$ is contracted by the forgetful map. Away
from these strata, this morphism is an isomorphism. See Observations 3.31 and 3.32 of \cite{PST14}.
\item $\mathbb{L}_i^{\Gamma}$ is canonically oriented as a complex orbifold line bundle.
\end{enumerate}
\end{obs}

\begin{definition}\label{nn:Direct sums of bundles}
Let $\vecd=(d_i)_{i\in I}\in\NN^I$ be a vector labeled by a set $I\subseteq\Universe$ and let $\Gamma$ be a graded graph with $I(\Gamma)\subseteq I$. We write
\[
E_\Gamma(\vecd):=
\cW_\Gamma\oplus\bigoplus_{i\in I(\Gamma)}\CL_i^{\oplus d_i}.
\]
We call this the \emph{descendent-Witten bundle with respect to $\vecd$}.
We omit $\Gamma$ from the notation if it is clear from context. We call the $\cW_{i,\Gamma}$ the \emph{Witten bundle components}, while the other summands are called the \emph{tautological line components}.\footnote{We remark that
$\vecd$ contains data which may be irrelevant for the definition
of $E_{\Gamma}(\vecd)$, i.e., $d_i$ for $i\in I\setminus I(\Gamma)$. However,
this will prove to be convenient to avoid even heavier notation in proofs.}
\end{definition}

The decomposition properties of the bundles $E_\Gamma(\vecd)$ are easily deduced from the decomposition properties of the Witten bundles and of the tautological line given in Proposition~\ref{pr:decomposition} and Observation~\ref{isomorphism forgetting non-alt for descendents}.

\subsubsection{Multisections and the relative Euler class}\label{multisection primer}
\label{subsubsec:multisections}
We start by recalling some notation involving multisections of
orbifold bundles. Let $E$ be an orbifold vector bundle over an orbifold with corners $M$. We refer the reader to Appendix A of \cite{BCT:II} for more details.

\begin{definition}[Definition A.1 of \cite{BCT:II}]
A \emph{multisection} of $E \to M$ is a function
$$
\ess: E \to \Q_{\ge 0}
$$
that satisfies the following:
\begin{enumerate}
\item the function $\ess$ is well-defined with respect to the orbifold structure on $E$.
\item for each $x\in M$, there is an (\'etale) open neighborhood $U$, a nonempty finite set of smooth local sections $s_i: U \to E, i = 1, \dots, N$ and a collection of positive rational numbers $\mu_1, \ldots, \mu_N$ so that for all $y \in U$,
$$
\sum_{v \in E_y} \ess(y, v) = 1.
$$
and for all $v \in E_y$
$$
 \ess(y, v) = \sum_{i, s_i(y) = v} \mu_i.
$$
We call $s_i$ the \emph{local branches} and $\mu_i$ the \emph{weights}.
\end{enumerate}
\end{definition}

Note that in a chart where $N=1$, then the weight is $\mu_1 = 1$ and the support of $\ess(y,v)$ is the graph of the local branch $s_1$ as a section in $E$.

\begin{nn}\label{not: restriction notation}
Given $E$ and $M$ as above, we denote by $C^\infty_m(M, E)$ the space of smooth multisections. When the base space $M$ is understood from the context, we will drop it and write $C^\infty_m(E)$. If $U\subseteq M$, we will abuse notation and write $C^\infty_m(U, E)$ or $C^\infty_m(E|_U)$, rather than a pullback. In words, when we discuss a multisection  on a subspace of the base then we mean the total space of the bundle restricted to the subspace of the base. When we do this, the bundle should be obvious from the context. For example, if we have a multisection $\ess \in C_m^\infty(M, E)$ and $U\subseteq M$, we will write $\ess|_{U}$ for $\ess|_{E|_U}$ to compress notation.
\end{nn}

There are natural operations for addition of multisections and a multiplication of a multisection by a function. Let $\lambda\in C^\infty(M, \mathbb{R})$ be a real-valued smooth function and $\ess,\ess'\in C_m^\infty(M,E)$ be multisections given locally by $(s_i,\mu_i)_{i=1}^N,~~(s'_i,\mu'_i)_{i=1}^{N'},$ respectively.
The multisection $\lambda\ess$ is obtained from $\ess$ by replacing each $s_i$ with $\lambda s_i$, and keeping $\mu_i$ the same.
The multisection $\ess+\ess'$ is given locally by the collection of smooth local sections $\tilde{s}_{ij}=s_i+s'_j$ and $\tilde{\mu}_{ij}=\mu_i\mu_j$ (see \cite[Definition A.9]{PST14} or \cite[Appendix A]{BCT:II}).

\begin{definition}\label{def:union_of_sections}

Let $\ess_1,\ldots, \ess_m$ be multisections of an orbifold bundle $E\to M$ and take positive integers $a_1,\ldots, a_m \in \mathbb{Z}_{>0}$.
Define the multisection $\uplus_{i=1}^m a_i\ess_i(x,v)$ by
\[\uplus_{i=1}^m a_i\ess_i(x,v) = \frac{1}{\sum_{i=1}^m a_i}\sum a_i\ess_i(x,v).\]
\end{definition}

We note that the support of $\uplus_{i=1}^m \ess_i(x)$ is the properly weighted union of the supports of $\ess_i$. Observe  (or recall from Equation (A.2) of \cite{BCT:II}) that
\begin{equation}\label{eq:union_of_sections}
\uplus_{i=1}^m \ess=\ess.
\end{equation}

Let $p \in M$ and consider the smooth local sections $s_i$ of $\ess$ near $p$. If $p \in Z(s_i)$ for some $i$, then we say that $\ess$ \emph{vanishes} at $p$. Otherwise, we say $\ess$ is nonvanishing at $p$. Given a multisection $\ess$ the \emph{zero locus} is defined to be $Z(\ess) := \{ p \in M \ | \ \ess \text{ vanishes at $p$} \}$. A multisection $\ess$ is \emph{transverse to zero} if, for any point $x$ in the zero locus of $\ess$, there is a neighborhood $U$ so that  every local branch $s_i$ of $\ess$ in $U$ is transverse to the zero section, denoted $\ess \pitchfork 0$.

Given a global multisection $\ess$ of $E$ on $M$ with isolated zeros, one can define the weighted cardinality of its zero set as follows. In a neighborhood of a point $p\in M$, consider the local branches $s_i$ with weights $\mu_i$ for $i = 1, \ldots, N$ of the multisection $\ess$. Let $\deg_p(s_i)$ be the degree of vanishing of $s_i$ at $p$ (see Appendix A of \cite{BCT:II} for a precise definition). Note that this is zero if $s_i(p) \ne 0$. We then can say the weight of $\ess$ at $p$ is given by
\begin{equation}\label{def: weight}
\eps_p := \frac{1}{|G|} \sum_{i=1}^N \mu_i \deg_p(s_i)
\end{equation}
where $G$ is the isotropy group of the orbifold at $p$. Then we can define the weighted cardinality of the zero locus $\#Z(\ess)$ of $\ess$ to be
\begin{equation}\label{def: weighted cardinality}
\#Z(\ess) = \sum_{p \in M} \eps_p.
\end{equation}
We emphasize that both the zero locus (and the weighted cardinality of the zero locus) are (weighted) subsets of the base space $M$.

We now define the relative Euler class as constructed in Appendix A and Notation 3.2 of \cite{PST14}.
By Theorem A.14 of \cite{PST14}, if (i) $\dim M = \rank E$,
(ii) $\ess$ a nowhere vanishing multisection of $E|_{\partial M}$, and
(iii) $\tilde \ess \in C_m^\infty(E)$ is a transverse extension of $\ess$ to all of $M$, then
the weighted cardinality of the zero set $\#Z(\tilde{\ess})$ depends only on $\ess$ and not on the choice of $\tilde \ess$. For this reason, we will often conflate notation and denote a choice of extension $\tilde{\ess}$ just by $\ess$.  Moreover, this means that the homology class $[Z(\tilde \ess)] \in H_0(M, \Q)$ only depends on $E$ and $\ess$.
Note that this can be extended to a special case of noncompact orbifolds $M$ in the sense that one can replace $\partial M$ with $U \cup \partial M$ where $M\setminus (U \cup \partial M)$ is compact (see Appendix A of \cite{BCT:II}).

 \begin{definition}\label{def: relative Euler class}
 Let $E$ be an orbifold vector bundle over an orbifold with corners $M$. Take an open sub-orbifold $U \subseteq M$ such that $M \setminus U$ is a compact orbifold with corners.  Given a nowhere vanishing smooth multisection $\mathbf{s} \in C^\infty_m(E|_{\partial M \cup U})$, we denote by $e(E; \mathbf{s}) \in H^*(M, \partial M\cup U)$  the Poincar\'e dual of the zero locus $[Z(\tilde \ess)] \in H_0(M)$ of a transverse extension $\tilde \ess$ of $\mathbf{s}$ to $M$. We call  $e(E; \mathbf{s})$ \emph{the relative Euler class} of $E$ and $\ess$.
 \end{definition}

\begin{nn}
\label{not:pushforward}
Suppose given a covering map $F:M\rightarrow M'$ 
of orbifolds with corners  with (positive) finite degree $\deg F$. Suppose
also given
orbifold vector bundles $E, E'$ on $M$ and $M'$ respectively with
$E = F^*E'$. Then for a multisection $\ess\in C^{\infty}_m(E)$, we may
define a multisection $F_*\ess$ given by
\[
(F_*\ess)(x',v') =  {1\over \deg F}\sum_{x\in F^{-1}(x')} \ess(x,v_x)
\]
where $v_x \in E_x$ is the point of the fibre $E_x$ mapping to
$v'\in E'_{x'}$. 
Note this needs to be interpreted in an orbifold sense. For example, if
$M\rightarrow M'$ is a quotient by a finite group $G$ acting trivially
on $M$ but possibly non-trivially on $E$, then $F_*\ess$ coincides with
$\biguplus_{g\in G} g^*\ess$.

If $\ess$ is transverse to the zero-section of $E$, then it is easy to see
that $F_*\ess$ is transverse to the zero-section of $E'$, and if $\ess$ has
a finite number of zeros,
\begin{equation}
\label{eq:pushforward zeros}
\# Z(F_*\ess) = {1\over \deg F} \# Z(\ess), \quad \# Z(F^*F_* \ess) = 
\# Z(\ess).
\end{equation}
Here, the pull-back of a multisection is given
by the standard notion of pull-back of sections of vector bundles.
\end{nn}

We finish this subsection with the following useful tool.
 \begin{lemma}\label{lem:zero diff as homotopy}
Let $E \to M$ be an orbifold vector bundle over an orbifold with corners, with $\rank(E) = \dim(M)$, and let $U\subseteq M$ be an open set with $M\setminus U$
compact.
Let $\ess_0$ and $\ess_1$ be nowhere-vanishing smooth multisections on $U\cup \partial M$. Denote by $p : [0,1] \times M\to M$ the projection, and let
\[
H \in C_m^\infty([0,1]\times (U\cup \partial M),p^*E)
\]
where $U\subseteq M$ is an open set with $M\setminus U$ compact.
Suppose $H|_{\{i\}\times E}=\ess_i$, $H\pitchfork 0$ and $H$ is nowhere vanishing in $[0,1]\times U$. 
Then
\[
\int_M e(E ; \ess_1) - \int_M e(E ; \ess_0) = \# Z(H)=\# Z(H|_{[0,1] \times
\partial M}).
\]
\end{lemma}

This lemma is the non-compact orbifold analog of \cite[Lemma 3.55]{PST14} and the proof is verbatim.

\begin{rmk}
Note that we orient $[0,1]$ by its natural orientation and the boundary of $M$ by its induced orientation. We write the unit interval first to denote that its orientation is first in the appropriate ordering of the orientations of the product. This will become important in Subsection~\ref{subsec:or} (see, e.g.,  Observation~\ref{rmk:orientation and exact sequence}).
\end{rmk}

\subsection{Coherent multisections and their assembling}\label{subsec:coherent}

This subsection is technical and a reader trying to  understand only the main thrust of the paper may skip it.
As indicated in the introduction, the boundary conditions of interest
will be inductively defined, and lead to the definition of canonical
families of multisections. However, the key topological recursion result
of the paper, Theorem \ref{thm:open TRR}, requires understanding the
vanishing of certain sections along strata of the moduli spaces where the
disks have one internal node. Such strata are not actually contained
in the topological boundary of the moduli space, so the behaviour of canonical
sections along these strata is not particularly controlled.
Thus it will be important in the
proof of open topological recursion to choose canonical
multisections also satisfying an inductive structure along these strata:
these sections will be called \emph{special canonical multisections},
see Definition \ref{def: special canonical}. Ideally one would
like to specify an element of the fibre of the Witten bundle at a point
of the moduli space corresponding to a disk with an internal node
by gluing sections of $\widehat{J}_1, \widehat{J}_2$ on the two curves
obtained by normalizing
the internal node. However, in the case of Ramond nodes, the gluing
properties of the Witten bundle as stated in Lemma \ref{pr:decomposition}
are subtle.
This leads to the technical complication addressed in this subsection.

In fact this notion of special canonical only plays a role in the
proof of Lemma \ref{lem:closed_contribution}, whose proof follows
the proof of the analogous result in \cite{BCT:II}. However, there are
some slight complications in the generalization of the definitions
of coherent multisections from \cite{BCT:II}, and hence we must devote
some space to this concept here, even though the use of this concept
will be largely invisible for the remainder of the paper.

\subsubsection{Coherent multisections}
Consider a (possibly disconnected)  stable graded
$W$-spin dual graph $\Gamma$.
For any twisted $r_i$-spin structure on a $W$-spin surface
with dual graph $\Gamma$, we set
\[J'_i := J_i \otimes \O\left(\sum_t [z_t]\right),\]
where the sum is over all anchors $t$ with twist $\tw_i(t)=-1$.
Effectively, we are changing the twists of the anchors to $r-1$,
as described in Proposition~\ref{pr:decomposition}(iii)(b).

Similarly, on the universal curve over $\oCM_{\Gamma}^W$, we take a modified universal line bundle
\[\mathcal{J}'_i := \mathcal{J}_i \otimes \O \left( \sum_t \Delta_{z_t}\right),\]
where $\Delta_{z_t}$ is the divisor in the universal curve corresponding to $z_t$.
Define an orbifold vector bundle $\TRAM_{\Gamma}$ on $\M_{\Gamma}^W$ by
\[\TRAM_{\Gamma} = \bigoplus_{i=1}^a\bigoplus_{\substack{\text{anchors $t$,} \\ \tw_i(t)= -1}} \sigma_{z_t}^*\mathcal{J}'_i,\]
where $\sigma_{z_t}$ is the corresponding section of the universal curve. The second summation only runs over anchors $t$ that have $\tw_i(t)= -1$  for the given $i$.\footnote{Consider the special case that all anchors are either (1) tails
which are Neveu-Schwarz
for all $i\in [a]$ or (2) contracted boundary tails. Then $\TRAM_\Gamma$ is the rank 0 vector bundle on $\oCM_\Gamma^{W}$.
If $\Gamma$ is connected, then $\mathcal{R}_{\Gamma}$ contains at most
one summand for each $i$.
Throughout what follows (except in the current subsection), the notation $\TRAM_\Gamma$ will not be used for graphs that are not closed, to avoid confusion.}
We use the notation $\TRAM_{\Gamma}$ both for the bundle and for its total space, and we denote by
\[\TAU: \TRAM_{\Gamma} \rightarrow \M_{\Gamma}^{W}\]
the projection.
For any degeneration $\Gamma' \in \d\Gamma$ of the graph $\Gamma$, there is a map $\TRAM_{\Gamma'}\rightarrow \TRAM_\Gamma$ between total spaces, which is an embedding on the underlying coarse level.

We denote
by $\cW_i'$ the bundle
\[{\cW}_i' = (R^0\pi_*\mathcal{J}_i')_+\]
on $\oCM_\Gamma^W$, and by $\boldsymbol{\cW}_i'$ the pullback of the bundle $(R^0\pi_*\mathcal{J}_i')_+$ to $\TRAM_\Gamma$ via $\TAU$.

\begin{definition}
\label{def:coherent}
Let $\Gamma$ be a connected graded $W$-spin dual graph, and let $s$ be a multisection of $\boldsymbol{\mathcal{W}}'_i$ over a subset $U\subset\TRAM_{\Gamma}$ for $i\in[a]$.  We say that $\ess$ is {\it coherent} if either:
\begin{enumerate}
\item the anchor of $\Gamma$ does not have $\tw_i=-1$; or
\item take any point $\zeta=(\Sigma, \vec{u}) \in U$. This consists of
data of a graded $W$-spin surface $\Sigma$ with bundles
$J_i'$ for each $i$ and an element $\vec{u}$ of the fibre
of $\TRAM_{\Gamma}$ over the corresponding point of $\oCM_{\Gamma}^W$.
The choice of $\vec{u}$ is equivalent to the data
$\vec{u}_t \in \bigoplus_{i\big|\tw_i(z_t)=-1}(J'_i)_{z_t}$ for
each anchor $t$ which has twist $-1$ for some $i$. Then, for any local branch $s_j$ of the multisection $\ess$, the element $s_j(\zeta)\in H^0(\Sigma,J'_i)$
satisfies
\[\ev_{z_t}s_j(\zeta) = (u_t)_i.\]
Here $\ev_{z_t}$ denotes the value of a section of the bundle
$J_i'$ at the point $z_t$, and $(u_t)_i$ denotes
the component of $\vec{u}_t$ which corresponds to the $i^{th}$ bundle.
\end{enumerate}

A \emph{coherent multisection} of $\boldsymbol{\cW}':=\bigoplus_{i=1}^a \boldsymbol{\cW}'_i$ is a direct
sum of coherent multisections of $\boldsymbol{\cW}'_i,i=1,\ldots,a$.

If $\Gamma$ is instead a disconnected graded $W$-spin graph, we say that a multisection $\ess$ of $\boldsymbol{\cW}'_i$ or $\boldsymbol{\cW}'$ over $U \subset \TRAM_{\Gamma}$ is coherent if it can be written as the restriction to $U$ of
\[\boxplus_{\Lambda\in\Conn(\Gamma)}\ess_\Lambda,\]
where each $\ess_\Lambda$ is a coherent multisection of $(\boldsymbol{\cW}'_i)_\Lambda \rightarrow U_\Lambda$ or $\boldsymbol{\cW}'_\Lambda \rightarrow U_\Lambda,$ respectively, for some $U_\Lambda\subseteq\TRAM_{\Lambda}$.
\end{definition}

Note that, if $\ess$ is a coherent multisection of $\boldsymbol{\cW}'_i \rightarrow \TRAM_{\Gamma}$ and $\zeta=(\Sigma, \vec{0}) \in\TRAM_{\Gamma}$, the multisection $\ess(\zeta)$ of $J'_i$ vanishes at each anchor $z_t$ with twist $-1$, and thus it is induced by a multisection of $J_i$.   In other words, the restriction of a coherent multisection $\ess$ to $\M_{\Gamma}^{W}$ is canonically identified with a multisection of $\cW_i \rightarrow \M_{\Gamma}^{W}$, which we will call $\overline{\ess}$.  If the graded $W$-spin dual graph $\Gamma$ has no anchor $z_t$ with twist $\tw_i(t) =-1$ for some $i \in [a]$, then $\overline{
\ess} = \ess$.

We note, also, that a multisection of $\TAU^*(\cW\to\oCM_\Gamma^{W})$ induces a multisection of $\boldsymbol{\cW}' \rightarrow \TRAM_{\Gamma}$ which vanishes along the $z_t$. Adding this multisection to a coherent multisection of $\boldsymbol{\cW}' \rightarrow \TRAM_{\Gamma}$ thus yields another coherent multisection.

\subsubsection{The assembling operation} 
In \S4.1.3 of \cite{BCT:II}, the authors define the \emph{assembling operation} $\Ass$. In the $r$-spin case, if $E'$ is a subset of internal edges of
$\Gamma$, the assembling operation takes a coherent multisection of $\ess$ of $\boldsymbol{\cW}'$ on  $\TRAM_{\detach_{E'}\Gamma}$ and yields a coherent multisection $\Ass_{\Gamma,E'}(\ess)$ of $\boldsymbol{\cW}'$ on $\TRAM_{\Gamma}$.

In our situation, this operation generalizes to coherent multisections of $\boldsymbol{\cW}'$
with an analogous construction.
We refer the reader to \S4.1.3 in \cite{BCT:II} for details, only sketching the construction and the slight differences needed to generalize from the $r$-spin
to the $W$-spin case, and then mentioning
the relevant properties for this paper.

We first discuss the assembling operation when
$\Gamma$ is a connected graded $W$-spin graph and
$E'=\{e\}$ a single internal edge.
Let $\ess=\oplus_{i\in [a]} \ess_i$ be a coherent multisection of
$\boldsymbol{\cW}'\to\TRAM_{\detach_{e}\Gamma}$, with $\ess_i$ a coherent
multisection
of $\boldsymbol{\cW}'_i$. Let $\Gamma_1$ and $\Gamma_2$ be the connected components of $\detach_{e}\Gamma$, with half-edges $h_1$ and $h_2$ respectively
corresponding to the edge $e$ of $\Gamma$. Without loss of generality, we choose $\Gamma_1$ so that it contains the open vertices or the anchor of $\Gamma$.
Thus $\Gamma_2$ will contain an additional anchor, the half-edge $h_2$.

A  point of $\TRAM_{\Gamma}$ consists of a
$W$-spin surface $\Sigma$ with dual graph lying in $\partial^!\Gamma$
along with a choice of some additional data $\vec{u}_t$ at the anchor $t$
of $\Sigma$, if it has one, indexed by the tail $t$.
We obtain a corresponding normalization
$\widehat\Sigma = \Sigma_1\sqcup \Sigma_2$. Then $(\Sigma_1, \vec{u}_t)$
determines a point of $\TRAM_{\Gamma_1}$ and $\Sigma_2$ determines
a point of $\oCM^W_{\Gamma_2}$. In particular, there is a natural morphism
$\TRAM_{\Gamma}\rightarrow \TRAM_{\Gamma_1}\times \oCM_{\Gamma_2}^W$.

Let $z_1$, $z_2$ be the half-nodes of
$\Sigma_1\cup\Sigma_2$ corresponding to the half-edges $h_1$, $h_2$ of $\Gamma$.
For each $i$ such that $\tw_i(h_2)=-1$,
we choose an identification $\rho_i$ of $(\mathcal{J}'_{1,i})_{z_1}$ and
$(\mathcal{J}_{2,i}')_{z_2}$; there are $r_i$ possible choices which differ by
$r_i^{th}$ roots of unity.\footnote{One can see this choice does not affect
the result, as the different choices of identification are related by
automorphisms of the $W$-spin curve $\Sigma_2$ given by rescaling the
spin bundles, see Remark~\ref{rmk: automorphisms}.}
In any event, we may write $\ess_i = \ess_i^1\boxplus \ess_i^2$, where
$\ess_i^1$ and $\ess_i^2$ are coherent multisections over $\TRAM_{\Gamma_1}$
and $\TRAM_{\Gamma_2}$. This allows us to lift the morphism
$\TRAM_{\Gamma}\rightarrow \TRAM_{\Gamma_1}\times \oCM^W_{\Gamma_2}$
to a morphism $\Detach_e:\TRAM_{\Gamma}\rightarrow \TRAM_{\Gamma_1}
\times \TRAM_{\Gamma_2}$ by taking $(\Sigma,\vec{u}_t)$ to the
pair
$\big((\Sigma_1,\vec{u}_t),(\Sigma_2,(\rho_i(\ess_i^1(z_1)))_{i|\tw_i(z_2)=-1})
\big)$.

The
$i^{th}$ component of the multisection $\Ass_{\Gamma,e}(\ess)$ is constructed in the following way. We have two cases:
\begin{itemize}
\item
The edge $e$ is Neveu-Schwarz for the $i^{th}$ Witten bundle summand.
We can use the identification of the Witten bundle given in Proposition~\ref{pr:decomposition}(\ref{it:NS}). The
$i^{th}$ component of the multisection $\Ass_{\Gamma,e}(\ess)$ is then,
under this identification, the pullback
$\Detach_e^*(\ess^1_i\boxplus \ess_i^2)$.
\item
The edge $e$ is Ramond for the $i^{th}$ Witten bundle summand. We are now in the case of Proposition \ref{pr:decomposition}(\ref{it:internalRamondWitten}).
In this case $\Detach_e^*(\ess^1_i\boxplus \ess_i^2)$ still defines
a multisection of $\boldsymbol{\cW}_i$ using
Proposition~\ref{pr:decomposition}(iii)(b), which follows from the fact
that $J'_i$ is obtained by gluing $J'_{1,i}$ and $J'_{2,i}$ via the
identification $\rho_i$. In particular, $\ess_i^1(\Sigma_1,\vec{u}_t)$
and $\ess^2_i\big(\Sigma_2,(\rho_i(\ess_i^1(z_1)))_{i|\tw_i(z_2)=-1})\big)$
glue at the node to give a section of $J'_i$. This gives the 
$i^{th}$ component of the multisection $\Ass_{\Gamma,e}(\ess)$.
\end{itemize}

In the Ramond case,
 $\Ass_{\Gamma,e}(\ess)$ has the property that when it is projected to
$\boldsymbol{\cW}'_i\to\TRAM_{\Gamma_1}$ via the exact sequence \eqref{eq:decompses2}, one obtains $\ess_1$.
In the locus where $\ess_1=0,$ the multisection $\Ass_{\Gamma,e}(\ess)$ is the image of $\overline{\ess}_2$ via the inclusion given by the exact sequence.

It is also worth mentioning that using the assembling operation is
associative (see Observation 4.3 of \cite{BCT:II}). That is, let $\Gamma$ be a graded $W$-spin dual graph. Let $E_1$ and $E_2$ be disjoint sets of edges in $\Gamma$. Take $\Gamma_i$ be the graph obtained from detaching edges $E_i$ from $\Gamma$ and $\Gamma_{12}$ be the graph from detaching $E_1\cup E_2$.
If $\ess$ is a coherent multisection of $\boldsymbol{\cW}' \rightarrow \TRAM_{\Gamma_{12}}$, then
\[\Ass_{\Gamma,E_1}(\Ass_{\Gamma_1,E_2}(\ess)) = \Ass_{\Gamma,E_1\cup E_2}(\ess).\]

\subsection{Rank two}\label{subsec:Rank 2}
In the special case $W=x^r+y^s$, we call a $W$-spin disk an $\RS$-disk and a $W$-spin graph an $(r,s)$-graph.
In this case, we write the twist vector as $(a,b)$ instead of $(a_1, a_2)$.

Recall that we have the set $\Universe$ used to mark points, containing
a copy of $\N$. We will, from now on, assume given a function
\[
\tw:\N \rightarrow
\{0,\ldots,r-1\}\times\{0,\ldots,s-1\}
\]
associating to each $i\in\N\subseteq
\Universe$ a fixed choice of twist $(a_i,b_i)$.
We assume that for each $(a,b)\in \{0,\ldots,r-1\}\times \{0,\ldots,s-1\}$,
the inverse image $\tw^{-1}(a,b)$ is infinite.

For a pre-graded $\RS$-disk,
we call a boundary marked point with $\tw=(r-2,0),\alt=(1,0)$ an $r$-point. We similarly define an $s$-point. The $r$-points and $s$-points are the two types
of singly twisted (see Definition \ref{def:open_W_graded2})
points on a pre-graded $\RS$-spin disk. We shall denote the $r$-points by $x_i$ and the $s$-points by $y_j$.
We use the same terminology for boundary tails or boundary half-nodes
($r$-tails et cetera), or conflate the terminology.

As in Definition \ref{def:open_W_graded2}(2), we write for a pre-graded open $\RS$-spin surface:
\begin{itemize}
\item  $k_1=k_{\{1\}}$, the number of $(r-2,0)$-twisted  boundary points,
\item $k_2=k_{\{2\}}$, the number of $(0,s-2)$-twisted boundary points, and
\item  $k_{12}=k_{\{1,2\}}$, the number of $(r-2,s-2)$-twisted points.
\end{itemize}

\begin{obs}\label{obs:closed}
It follows from \eqref{eq:close_rank1_general}
that the twists $(a_i,b_i)$ for a twisted $\RS$-spin structure on a closed, smooth marked genus $0$ orbifold Riemann surface satisfy
\begin{equation}\label{eq:close_rank1}
e_1=\frac{\sum a_i -(r-2)}{r}\in \Z,~~e_2=\frac{\sum b_i -(s-2)}{s}\in \Z.
\end{equation}
We have that the complex ranks of the two Witten bundle components in the closed case are $e_1,e_2$.
\end{obs}
Similarly, Observation \ref{obs:open_rank_pre_graded_W} gives the following:
\begin{obs}
\label{obs:open_rank1}
The internal twists $(a_i,b_i)$ for a pre-graded $\RS$-spin structure on a smooth genus $0$ marked orbifold Riemann surface with boundary satisfy:
\begin{equation}\label{eq:open_rank1}
e_1=\frac{2\sum a_i + (k_1+k_{12}-1)(r-2)}{r}\in \Z,~~e_2=\frac{2\sum b_i + (k_2+k_{12}-1)(s-2)}{s}\in \Z,
\end{equation}
and
\begin{equation}\label{eq:open_rank2}
e_1\equiv k_1+k_{12}-1~(\text{mod}~2),~~e_2\equiv k_2+k_{12}-1~(\text{mod}~2).
\end{equation}
Here, $e_1,e_2$ are the ranks of the real Witten bundles, see \eqref{eq:graded_open_rank}.
\end{obs}

For the benefit of the reader we now specialize Notation \ref{nn:repeating graphs} and Definition \ref{ModSpaceOfMarkedDiscs} to the rank $2$ case.
\begin{nn}\label{balanced rank 2 gammas}
For a finite set $I\subseteq \N\subseteq \Universe$,
write $\Gamma=\Gamma^W_{0,k_1,k_2, k_{12},\{(a_i,b_i)\}_{i \in I}}$ for the graded graph with
\begin{itemize}
\item one vertex $v$;
\item $|I|$ internal tails marked bijectively to $I$, with the $i$th
tail having twist $\operatorname{tw} = (a_i,b_i)$;
\item $k_1$ boundary $r$-tails and $k_2$ boundary $s$-tails marked by $\emptyset$ (that is, unlabeled);
\item $k_{12}$ boundary tails with $\tw=(r-2,s-2)$ and $\alt=(1,1)$  marked by $\emptyset$;
\item $\Pi = (\Pi_v)$ consisting of all possible cyclic orders for boundary tails.
\end{itemize}
We define $\oCM_{0,k_1,k_2, k_{12},\{(a_i,b_i)\}_{i \in I}}^{W} := \oCM_\Gamma^W,$ and $\oCM_{0,k_1,k_2, k_{12},\{(a_i,b_i)\}_{i \in I}}^{W,\text{labeled}}$ for its labeled version.
\end{nn}

\begin{rmk}\label{rmk:cyclic_orders}
There is an important difference between the conventions for our moduli and those of \cite{PST14,BCT:I, BCT:II}. In those references, boundary markings are labeled; however, we do not label them. This corresponds to taking some quotients in the moduli defined in those papers, see Step (F) of the proof of
Proposition \ref{prop:orbi_w_corners}.

In the case where $k_{12}=1$, the labelled moduli space $\oCM_{0,k_1,k_2, k_{12},\{(a_i,b_i)\}_{i \in I}}^{W,\text{labeled}}$ has $(k_1+k_2)!$ connected components. This can be understood by computing the number of cyclic orders of $k_1+k_2+k_{12}$ boundary markings. However, in the unlabeled case, there are additional automorphisms of the disks which permute the boundary markings that preserve the twists.  Without loss of generality, we can view any cyclic ordering as an ordering that starts at the root. There are  $\binom{k_1+k_2}{k_1}$ orderings starting with the root that preserve the relative ordering of $r$-points and $s$-points. Each of these orderings correspond to a connected component of the moduli $\oCM_{0,k_1,k_2, k_{12},\{(a_i,b_i)\}_{i \in I}}^{W}$.
\end{rmk}

\begin{ex}
Consider Example~\ref{ex: hexagon} where we have one internal marked point and three boundary marked points. Suppose further that the boundary marked points consist of one root, one $r$-point, and one $s$-point.
Here, the moduli $\oCM_{0,1,1, 1,\{(1,1)\}}^{W}$ has two connected components, depending on the cyclic orderings of the three boundary markings.

On the other hand, if we suppose that the boundary marked points consist of one root and two $r$-points, then the moduli $\oCM_{0,2,0, 1,\{(2,0)\}}^{W}$ has only one connected component. Indeed, the two $r$-points will become indistinguishable and there will exist an isomorphism for the $W$-spin disks that will identify the two hexagons. This can be viewed as quotienting the labeled moduli space $\oCM_{0,2,0, 1,\{(2,0)\}}^{W,\text{labeled}}$ (consisting of two connected components) by this extra involution, so the moduli $\oCM_{0,2,0, 1,\{(2,0)\}}^{W}$ is the hexagon in Figure~\ref{fig:moduliExample}.
\end{ex}

\begin{nn}\label{nn:r(I),s(I),m(I)}
Let $I$ index a set of internal markings with twists $(a_i,b_i)$ and $\vecd=(d_i)_{i\in I}$ be a vector of descendents at each internal marking. Write $r(I) \in \{0,\dots, r-1\}$ and $s(I) \in \{0,\dots, s-1\}$ for the unique numbers so that 
\begin{align}\label{eq:r(I),s(I),m(I)}
\begin{split}
r(I) \equiv {} & \sum_{j\in I}a_j~(\text{mod}~r),\\
s(I)\equiv {} & \sum_{j\in I}b_j~(\text{mod}~s), \text{ and }\\
m(I,\vecd):= {} &rs + \sum_{j\in I} \left(sa_j+rb_j+rs(d_j-1)\right).
\end{split}
\end{align}
If $J\subseteq I$, we write $m(J,\vecd):= m(J, \vecd|_J)$.
\end{nn}

\begin{prop}\label{prop:balanced}
Consider a smooth rooted connected $\RS$-graph $\Gamma:=\Gamma_{0,k_1,k_2,1,\{(a_i,b_i)\}_{i\in I}}$ with a set of internal tails $I$.  Then we have the following:
\begin{enumerate}[(a)]
\item  $k_1\equiv r(I)~(\text{mod} ~r)$ and $k_2 \equiv s(I)~(\text{mod} ~s)$ if and only if the space $\oCM^W_{\Gamma}$ is nonempty.
\item Assume that the space $\oCM^W_{\Gamma}$ is nonempty. The equation $sk_1+rk_2=m(I,\vecd)$ holds if and only if
\begin{equation}\label{balancedModuli}
\rank(E_{\Gamma}(\vecd))=\dim\oCM^W_{\Gamma}.
\end{equation}
\end{enumerate}
\end{prop}

\begin{proof}
(a) The backwards implication is immediate from \eqref{eq:open_rank1} and \eqref{eq:open_rank2}. We know, since $k_{12}(\Gamma) =1$, that
$$
e_1=k_1 + \frac{2(\sum a_i - k_1)}{r}\in \Z, \qquad e_1\equiv k_1 \pmod{2}.
$$
This implies that $ \frac{2(\sum a_i - k_1)}{r}\in 2\Z$, which requires that $r$ divides $(\sum a_i - k_1)$. The analogue for the $s(I)$ follows similarly. The forwards implication follows from Proposition~\ref{prop:existence_stable}.

(b) We start by noting that Equation~\eqref{balancedModuli} is equivalent to the equation
\begin{equation}\begin{aligned}\label{balancedequationstandard} \sum_{i\in I}2d_i+\frac{2\sum_{i\in I(\Gamma)}a_i+(k_1(\Gamma)+k_{12}(\Gamma)-1)(r-2)}{r}&+\frac{2\sum_{i\in I(\Gamma)}b_i+(k_2(\Gamma)+k_{12}(\Gamma)-1)(s-2)}{s} \\&=2|I|+k_1(\Gamma)+k_2(\Gamma)+k_{12}(\Gamma)-3.\end{aligned}\end{equation}

By multiplying by $rs$, using that $k_{12}(\Gamma) = 1$, and subtracting $rs(k_1(\Gamma)+k_2(\Gamma))$ from both sides, we have that the above equation is equivalent to:
\begin{equation}\begin{aligned}\sum_{i\in I}2rsd_i+2\sum_{i\in I}sa_i-2sk_1(\Gamma)+2\sum_{i\in I}rb_i-2rk_2(\Gamma) & =2rs|I|-2rs.\end{aligned}\end{equation}
We then can divide by two and isolate $sk_1(\Gamma)+rk_2(\Gamma)$ on one side of the equation to see:
\begin{equation}\begin{aligned}
sk_1(\Gamma)+rk_2(\Gamma) &= \sum_{i\in I}rsd_i+\sum_{i\in I}sa_i+\sum_{i\in I}rb_i - rs|I| +rs \\
	&= rs + \sum_{i\in I}(rs (d_i -1) + sa_i + rb_i)\\
	&= m(I, \vecd). 
\end{aligned}\end{equation}
\end{proof}
Proposition~\ref{prop:balanced}(b) provides a combinatorial description on the level of graphs for when~\eqref{balancedModuli} holds, and we will consequently have non-trivial enumerative invariants. We give the following definition for such graphs:

\begin{definition}\label{def:balanced,critical,etc}
Let $I$ be a set of markings with twists $(a_i,b_i)$, $i\in I$,
and $\vecd=(d_i)_{i\in I}\in\NN^I$ a vector of descendents.
A smooth rooted connected graded $\RS$-graph $\Gamma$ with $I(\Gamma)\subseteq I$ is \emph{balanced} for $\vecd$ if $\oCM^W_\Gamma$ is nonempty and 
\[
\rank E_{\Gamma}(\vecd)
=\dim\oCM^W_{\Gamma}.
\]
For $J\subseteq I$, denote by $\INT(J,\vecd)$  the collection of connected rooted  graphs with internal markings $J$ that are balanced for $\vecd$. If instead $\Gamma$ is a disconnected smooth rooted
$\RS$-graph, it is balanced for $\vecd$ if all its connected components are balanced for $\vecd$.
\end{definition}

We also can define critical graphs, which will correspond to irreducible components of a codimension one boundary strata of $\oCM^W_{0,k_1,k_2,1,\{(a_i,b_i)\}}$ that yield wall-crossing (see Theorem~\ref{thm: group actions only}
and \S\ref{sec:WC}). Recall the definition of a relevant graph given in Definition~\ref{def:special kind of graded graphs}.

\begin{definition}\label{def:critical}
Let $I$ be a set of internal markings and $\vecd=(d_i)_{i\in I} \in \NN^I$.
\begin{enumerate}
\item
A smooth connected relevant graded $\RS$-graph $\Gamma$ with $I(\Gamma)\subseteq
I$ and without fully twisted boundary tails is \emph{critical} for $\vecd$ if
\begin{equation}
\label{eq:crit equality}
\rank(E_{\Gamma}(\vecd))=
\dim\oCM^W_{\Gamma}+1.
\end{equation}
For $J\subseteq I$, we denote by $\CRIT(J,\vecd)$ the collection of critical  graphs for $\vecd$ with internal markings $J$.
\item
A (connected) graded graph $\Lambda\in\partial\Gamma$, for $\Gamma\in\INT(J,\vecd)$, is a \emph{critical boundary graph} if it has a unique edge which is a boundary edge and after detaching this edge the graph is separated into two vertices $v_0,v_1$ such that $v_0\in\CRIT(I(v_0),\vecd)$ and $v_1\in\INT(I(v_1),\vecd)$.
\item A graded graph $\Lambda\in\partial\Gamma$ where $\Gamma$ is
balanced is a \emph{critical boundary graph} if one connected component
is a critical boundary graph and the remaining components are balanced.
\item If $\Gamma$ is balanced, a graph $\Lambda\in \partial\Gamma$
is an \emph{exchangeable critical boundary} if
$\Lambda \in\partial^\xch\Gamma$ and $\Lambda$ is a critical boundary graph.
\end{enumerate}
\end{definition}

\begin{nn}
\label{subsets for Bal and Crit}
When $\vecd$ is the zero vector we do not add the phrase ``for $\vecd$" and omit $\vecd$ from the notations of $\CRIT$ and $\INT$.
We write
\[
\CRIT(\subsetneq I,\vecd) = \bigcup_{J\subsetneq I}\CRIT(J,\vecd), \quad \CRIT(\subseteq I,\vecd)=\CRIT(\subsetneq I,\vecd)\cup \CRIT(I,\vecd),
\]
and similarly for $\INT(\subsetneq I,\vecd)$ and $\INT(\subseteq I,\vecd)$.
\end{nn}

\begin{definition}
A stratum $\oCM_\Lambda$ in the moduli space $\oCM_{\Gamma}$ is relevant, critical etc.\ if the corresponding graph $\Lambda \in \partial \Gamma$ is.
\end{definition}

\begin{definition}\label{top boundary, subsets for Bal and Crit}
With the same notation as in Definition~\ref{def:critical}(2), a \emph{top boundary} for $\Gamma\in\INT(I,\vecd)$ is a critical boundary of $\Gamma$ such that $I(v_1)=\emptyset$.
\end{definition}

\begin{rmk}\label{rmk:crit}
We can rewrite \eqref{eq:crit equality} as an equation analogous to
\eqref{balancedequationstandard} in Proposition~\ref{prop:balanced}(b) as:
\begin{align*}\sum_{i\in I}2d_i+\frac{2\sum_{i\in I(\Gamma)}a_i+(k_1(\Gamma)-1)(r-2)}{r}&+\frac{2\sum_{i\in I(\Gamma)}b_i+(k_2(\Gamma)-1)(s-2)}{s} \\&=2|I|+k_1(\Gamma)+k_2(\Gamma)-2.\end{align*}
\end{rmk}

\begin{figure}

\begin{subfigure}{.45\textwidth}
  \centering
\vspace{.24cm}
\begin{tikzpicture}[scale=.4]

	\draw (0,0) circle (.2cm);
	
%	\draw[dashed] (0,-0.2) -- (0,-2.2);
%    \node[below] at (0,-2.2) {$t_{x_2}$};
	
	\draw (-.1414,.1414) -- (-1.6414, 1.6414);
    \draw (.1414,.1414) -- (1.6414, 1.6414);
    \node[above] at (-1.6414, 1.6414) {$t_{z_1}$};
    \node[above] at (1.6414, 1.6414) {$t_{z_2}$};
    
    \draw[dashed] (.1732, -.1) -- (1.732, -1);
    \draw[dashed] (-.1732, -.1) -- (-1.732, -1);
    \draw[dashed] (-.1, -.1732) -- (-1, -1.732);
    \draw[dashed] (.1, -.1732) -- (1, -1.732);
    \draw[dashed] (-.2, 0) -- (-2, 0);

    \node[left] at (-2,0) {$\times$};
    \node[below] at (1.732, -1) {$r$};
    \node[below] at (-1.732, -1) {$r$};
    \node[below] at (-1, -1.732) {$r$}; 
    \node[below] at (1, -1.732) {$r$}; 
	%\draw[dashed] (-.1414,-.1414) -- (-1.6414, -1.6414);
   % \draw[dashed] (.1414,-.1414) -- (1.6414, -1.6414);
	
	%\node[above] at (5.6414, 1.6414) {$t_{x_1}$};
	%\node[below] at (1.6414, -1.6414) {$t_{x_1}$};
	
	%\node[below] at (-1.6414, -1.6414) {$t_{x_3}$};

	%\node[left] at (0,0) {$v$};
\end{tikzpicture}
\vspace{0.15cm}

  \caption{The balanced $W$-spin graph  $\Gamma_1$}
\end{subfigure}
\qquad
\begin{subfigure}{.45\textwidth}
  \centering
\vspace{.24cm}
\begin{tikzpicture}[scale=.4]

	\draw (0,0) circle (.2cm);
	\draw (4,0) circle (.2cm);
	
	\draw[dashed] (0.2,0) -- (3.8,0);
	\node[above] at (2,0) {$e_{q}$};

	\draw (3.8586,.1414) -- (2.3586, 1.6414);
    \draw (4.1414,.1414) -- (5.6414, 1.6414);
    \node[above] at (2.3586, 1.6414) {$t_{z_1}$};
    \node[above] at (5.6414, 1.6414) {$t_{z_2}$};
    
    \draw[dashed] (4.1732, -.1) -- (5.732, -1);
    \draw[dashed] (-.1732, -.1) -- (-1.732, -1);
    \draw[dashed] (-.1, -.1732) -- (-1, -1.732);
    \draw[dashed] (.1, -.1732) -- (1, -1.732);
    \draw[dashed] (-.2, 0) -- (-2, 0);

    \node[left] at (-2,0) {$\times$};
    \node[below] at (5.732, -1) {$r$};
    \node[below] at (-1.732, -1) {$r$};
    \node[below] at (-1, -1.732) {$r$}; 
    \node[below] at (1, -1.732) {$r$};

	% \draw[dashed] (4,-.2) -- (4, -2.2);
	% \draw[dashed] (4.1414,-.1414) -- (5.6414, -1.6414);
	
	% \draw (-.1414,.1414) -- (-1.6414, 1.6414);
	% \draw[dashed] (-.1414,-.1414) -- (-1.6414, -1.6414);
	
	% \node[below] at (4, -2.2) {$t_{x_2}$};
	% \node[below] at (5.6414, -1.6414) {$t_{x_1}$};
	
	% \node[above] at (-1.6414, 1.6414) {$t_{z_1}$};
	% \node[below] at (-1.6414, -1.6414) {$t_{x_3}$};
	
	% \draw (0,.2) -- (0, 2.2);
	% \draw (.1414,.1414) -- (1.6414, 1.6414);
	% \node[above] at (0, 2.2) {$t_{z_2}$};
	% \node[above] at (1.6414, 1.6414) {$t_{z_3}$};
	
	 \node[above] at (0,0) {$v_1$};
	 \node[below] at (4,0) {$v_2$};
\end{tikzpicture}
\vspace{0.15cm}

  \caption{Critical boundary graph $\Lambda_1$}
\end{subfigure}
\qquad
\begin{subfigure}{.45\textwidth}
  \centering
\vspace{.24cm}
\begin{tikzpicture}[scale=0.4]

	\draw (0,0) circle (.2cm);
	
%	\draw[dashed] (0,-0.2) -- (0,-2.2);
%    \node[below] at (0,-2.2) {$t_{x_2}$};
	
	\draw (-.1414,.1414) -- (-1.6414, 1.6414);
    \draw (.1414,.1414) -- (1.6414, 1.6414);
    \node[above] at (-1.6414, 1.6414) {$t_{z_1}$};
    \node[above] at (1.6414, 1.6414) {$t_{z_2}$};
    
    \draw[dashed] (.1732, -.1) -- (1.732, -1);
    \draw[dashed] (-.1732, -.1) -- (-1.732, -1);
    \draw[dashed] (-.1, -.1732) -- (-1, -1.732);
    \draw[dashed] (.1, -.1732) -- (1, -1.732);
    \draw[dashed] (-.2, 0) -- (-2, 0);
    \draw[dashed] (0,-.2) -- (0,-2);

    \node[left] at (-2,0) {$\times$};
    \node[below] at (1.732, -1) {$s$};
    \node[below] at (-1.732, -1) {$s$};
    \node[below] at (-1, -1.732) {$s$}; 
    \node[below] at (1, -1.732) {$s$}; 
    \node[below] at (0,-2) {$s$};
	%\draw[dashed] (-.1414,-.1414) -- (-1.6414, -1.6414);
   % \draw[dashed] (.1414,-.1414) -- (1.6414, -1.6414);
	
	%\node[above] at (5.6414, 1.6414) {$t_{x_1}$};
	%\node[below] at (1.6414, -1.6414) {$t_{x_1}$};
	
	%\node[below] at (-1.6414, -1.6414) {$t_{x_3}$};

	%\node[left] at (0,0) {$v$};
\end{tikzpicture}
\vspace{0.15cm}

  \caption{Balanced $W$-spin graph $\Gamma_2$}
\end{subfigure}
\qquad
\begin{subfigure}{.45\textwidth}
  \centering
\vspace{.24cm}
\begin{tikzpicture}[scale=0.4]
	\draw (0,0) circle (.2cm);
	\draw (4,0) circle (.2cm);
	
	\draw[dashed] (0.2,0) -- (3.8,0);
	\node[above] at (2,0) {$e_{q}$};

	\draw (3.8586,.1414) -- (2.3586, 1.6414);
    \draw (4.1414,.1414) -- (5.6414, 1.6414);
    \node[above] at (2.3586, 1.6414) {$t_{z_1}$};
    \node[above] at (5.6414, 1.6414) {$t_{z_2}$};
    
    \draw[dashed] (4.1732, -.1) -- (5.732, -1);
    \draw[dashed] (-.1732, -.1) -- (-1.732, -1);
    \draw[dashed] (-.1, -.1732) -- (-1, -1.732);
    \draw[dashed] (.1, -.1732) -- (1, -1.732);
    \draw[dashed] (-.2, 0) -- (-2, 0);

    \node[left] at (-2,0) {$\times$};
    \node[below] at (5.732, -1) {$s$};
    \node[below] at (-1.732, -1) {$s$};
    \node[below] at (-1, -1.732) {$s$}; 
    \node[below] at (1, -1.732) {$s$}; 
    \draw[dashed] (0,-.2) -- (0,-2);
    \node[below] at (0,-2) {$s$};   
	
	% \draw[dashed] (4,-.2) -- (4, -2.2);
	% \draw[dashed] (4.1414,-.1414) -- (5.6414, -1.6414);
	
	% \draw (-.1414,.1414) -- (-1.6414, 1.6414);
	% \draw[dashed] (-.1414,-.1414) -- (-1.6414, -1.6414);
	
	% \node[below] at (4, -2.2) {$t_{x_2}$};
	% \node[below] at (5.6414, -1.6414) {$t_{x_1}$};
	
	% \node[above] at (-1.6414, 1.6414) {$t_{z_1}$};
	% \node[below] at (-1.6414, -1.6414) {$t_{x_3}$};
	
	% \draw (0,.2) -- (0, 2.2);
	% \draw (.1414,.1414) -- (1.6414, 1.6414);
	% \node[above] at (0, 2.2) {$t_{z_2}$};
	% \node[above] at (1.6414, 1.6414) {$t_{z_3}$};
	
	 \node[above] at (0,0) {$v_1$};
	 \node[below] at (4,0) {$v_2$};
\end{tikzpicture}
\vspace{0.15cm}

  \caption{Critical boundary graph $\Lambda_2$}
\end{subfigure}

\caption{Balanced and critical boundary graphs from Example~\ref{exmpl: bal and crit}.}
\label{fig:balanced and crit graphs}
\end{figure}

\begin{ex}\label{exmpl: bal and crit}
Consider the LG model $(x^4 + y^5, \mu_4 \times \mu_5)$. We consider the smooth connected rooted graded $W$-spin graph $\Gamma_1$ with 4 $r$-tails and two internal tails $t_{z_1}$ and $t_{z_2}$ with twist vectors $(2,2)$ and $(2,3)$. Then one can check that $\Gamma_1$ is balanced for $\mathbf{d} = \mathbf{0}$. See Figure~\ref{fig:balanced and crit graphs}(A), where $\times$ denotes a root and $r$ denotes a (boundary) $r$-tail.
Consider the boundary graph $\Lambda_1$ as seen in Figure~\ref{fig:balanced and crit graphs}(B). One can see that the vertex $v_1$ is balanced and $v_2$ is critical, so $\Lambda_1$ is a critical boundary graph. Moreover, it is a top boundary for $\Gamma_1$.

Secondly, one can consider the smooth connected rooted graded $W$-spin graph $\Gamma_2$ in Figure~\ref{fig:balanced and crit graphs}(C) with 5 $s$-tails and two internal tails $t_{z_1}$ and $t_{z_2}$ with twist vectors $(2,2)$ and $(2,3)$. Then one can check that $\Gamma_2$ is balanced for $\mathbf{d} = \mathbf{0}$. Moreover, the graph $\Lambda_2$ as seen in Figure~\ref{fig:balanced and crit graphs}(D) is top boundary for $\Gamma_2$ in a similar way. One can notice that these are two instances where the same critical graph appears in two distinct critical boundaries. 

Indeed, the vertex $v_2$ has an $r$-point and an $s$-point. In $\Lambda_1$, the $s$-point is a boundary half-edge that corresponds to a half-node in the corresponding $W$-spin disk. Meanwhile the roles of the boundary special points reverse for $\Lambda_2$. 

\end{ex}

\begin{prop}\label{prop:critical m(I)}
Let $\Gamma = \Gamma_{0,k_1+1, k_2+1, 0, \{(a_i,b_i)\}_{i\in I}}$ be a smooth, connected graph. Then:
\begin{enumerate}[(a)]
\item $k_1 \equiv r(I) \pmod r$ and $k_2 \equiv s(I)  \pmod s$ if and only if the space $\oCM^W_{\Gamma}$ is nonempty.
\item Assume that the space $\oCM^W_{\Gamma}$ is nonempty. The equation $sk_1 + rk_2 = m(I, \vecd)- rs$ holds if and only if $\Gamma$ is critical for $\vecd$.
\end{enumerate}
\end{prop}

\begin{proof}
This proof follows the same computations as in Proposition~\ref{prop:balanced}.
\end{proof}

Lastly, we can classify all balanced and critical graphs with a fixed set of internal markings $J \subseteq I$:

\begin{nn}\label{Concrete bal and crit graphs with prescribed internals}
Suppose $\Gamma \in \INT(J, \vecd)$. Using Notation~\ref{nn:r(I),s(I),m(I)}, we have that
\[\sum_{i\in J} a_i= r(J) + \ell_1 r \text{ and } \sum_{i\in J} b_i = s(J) + \ell_2 s\]
for some $\ell_1, \ell_2 \in \mathbb{Z}_{\geq 0}$.   Since $\Gamma$ is a smooth rooted graded graph $\Gamma$ with $I(\Gamma)=J$ and
balanced with respect to $\vecd \in\NN^J$, one can compute using Proposition \ref{prop:balanced}(b), that
\begin{equation}\label{balanced equivalent}
k_1(\Gamma) = r(J) + pr \text{ and } k_2(\Gamma) = s(J) + (\ell_1+\ell_2 - |J|+1+\sum_{j\in J} d_j - p)s,
\end{equation}
for some $0 \leq p \leq  \ell_1+\ell_2 - |J|+1 + \sum_{j\in J} d_j$.
Taking $N := \ell_1+\ell_2 - |J|+1 + \sum_{j\in J} d_j$, we can enumerate
$$\INT(J,\vecd) = \{ \Gamma_{J,0}, \ldots, \Gamma_{J,p}, \ldots, \Gamma_{J,N}\},$$
 where $\Gamma_{J,p}$ is the unique balanced graph above that has internal markings $J$ and $k_1(\Gamma_{J,p}) = r(J) + pr$.

Similarly, suppose $\Lambda \in \CRIT(J, \vecd)$. We can also compute using Proposition~\ref{prop:critical m(I)} that 
\begin{equation}\label{balanced equivalent}
k_1(\Gamma) = r(J) + (p-1)r +1 \text{ and } k_2(\Gamma) = s(J) + (\ell_1+\ell_2 - |J| - p +1 + \sum_{j\in J} d_j)s +1,
\end{equation}
for some  $1 \leq p \leq  N$. Enumerate the set of critical graphs $$
\CRIT(J,\vecd)=\{\Lambda_{J,1},\ldots,\Lambda_{J,N}\},$$ with the same $N$ as above, where $\Lambda_{J,p}$ is the unique critical graph above with internal markings $J$ and $k_1(\Lambda_{J,p}) = r(J) + (p-1)r +1$.
\end{nn}

\subsection{Orientation}\label{subsec:or}
In this subsection, we describe the canonical relative orientation of the Witten bundle $\cW$ over the various moduli spaces considered here, as well as the properties of these orientations. For that, we need some notation and background. We start by recalling the following.

\begin{obs}\label{rmk:orientation and exact sequence}
If $0\to A\xrightarrow{f} B\to C\to 0$ is an exact sequence of vector bundles, then there is a canonical isomorphism
\begin{equation}\label{eqn:ses orient}
\det(A)\otimes \det(C)\simeq \det(B)
\end{equation}
given by
\[(a_1\wedge\cdots\wedge a_n)\otimes(c_1\wedge\cdots\wedge c_m)\longmapsto \bigwedge_{i=1}^n f(a_i) \wedge \bigwedge_{i=1}^m c'_i, \]
where $\{a_i\}$ are $n:=\dim(A)$ elements of $A$, $\{c_i\}$ are $m:=\dim(C)$ elements of $C$, and $c'_i$ is an arbitrary element of $B$ mapped to $c_i$.  This isomorphism induces, by writing the elements of $C$ first, an isomorphism
\begin{equation}\label{eqn:ses orient reverse}
\det(C)\otimes\det(A)\simeq \det(B).
\end{equation}
\end{obs}

We will often use this observation in order to decompose orientations using a fibration structure. We provide a few examples now:

\begin{ex}\label{ex:internal edge example}
Let $\Gamma\in\d\Gamma_{0,k_1,k_2,1, \{(a_i, b_i)\}_{i\in [l]}}^{W,\textup{labeled}}$
 be a labeled $\RS$-graph consisting of two vertices $v_1$ and $v_2$ connected by a single edge $e$. Let $N$ be the normal bundle of  the inclusion $\iota_{\Gamma}: \oCM_{\Gamma}^W \hookrightarrow \oCM_{0,k_1,k_2,1,\{(a_i,b_i)\}_{i\in[l]}}^{W,\text{labeled}}$. We then have the isomorphism of line bundles using~\eqref{eqn:ses orient reverse}
\begin{equation}
\label{eq:TMdec}
\iota_{\Gamma}^* \det(T\oCM_{0,k_1,k_2,1,\{(a_i,b_i)\}_{i\in[l]}}^{W,\textup{labeled}}) \cong \det(N)  \otimes \det(T\oCM^W_{\detach_e(\Gamma)}),
\end{equation}
where $\iota_\Gamma$ is the inclusion map defined in Equation~\eqref{def:iotaGamma}.

Consider a point in the moduli space corresponding to $\Sigma = \Sigma_1 \cup_q \Sigma_2$, with $\Sigma_i$ corresponding to the vertex $v_i$. Let $T_q\Sigma_i$ be the tangent line at the node $q$. The fiber of the normal line bundle $N$ at $\Sigma$ can be identified with the tensor product $T_q\Sigma_1 \otimes T_q \Sigma_2$. These two tangent lines are canonically oriented complex lines when $e$ is internal, and are orientable real lines invariant under complex conjugation when $e$ is boundary.
In both cases, $N$ carries a canonical orientation, where in the second case we use the convention of orienting $N$ by the outward-pointing normal.
\end{ex}

\begin{ex}\label{ex:forget internal}
Consider the forgetful map $\text{For}_{\varnothing, l+1}:\CM_{0,k,l+1}^{\textup{labeled}} \to \CM_{0,k,l}^{\textup{labeled}}$ that forgets the $(l+1)^{\textup{st}}$ internal marked point. Its fiber is a disk with $l$ punctures, which carries a canonical (complex) orientation. For $k \ge 1$, we can also consider the forgetful map $\text{For}_{k+1, \varnothing}: \CM_{0,k+1,l}^{\textup{labeled}} \to \CM_{0,k,l}^{\textup{labeled}}$, which has a fiber that is a union of open intervals, canonically oriented as a boundary of an oriented disk.  We will denote these canonical orientations by $o_{\text{For}_{\varnothing, l+1}^{-1}(\Sigma)}$ and $o_{\text{For}_{k+1, \varnothing}^{-1}(\Sigma)}$ respectively. Note that by Observation~\ref{rmk:orientation and exact sequence} this gives two potential decompositions of the orientation $\CM_{0,k,l}^{\textup{labeled}}$ via the two different forgetful maps.
\end{ex}

\subsubsection{Background in rank one}\label{or:background in rank one}

In \cite{BCT:I}, the authors constructed families of orientations for the moduli space $\oCM_{0,B,I}$ of marked disks and the Witten bundles over the $r$-spin moduli space $\oCM^{1/r,\text{labeled}}_{0,B,\{a_i\}_{i\in I}}$. In this subsection, we will recall their results and the properties of these orientations.

We note that our orientations will depend on the cyclic ordering of the boundary markings, hence we will have to discuss each connected component of the moduli spaces individually. To do so, we introduce the following notation. Let $\Gamma$ be a genus zero, anchored, pre-stable dual graph. This comes with the set $\hat \Pi$ of cyclic orders. We denote by $\oCM_{\Gamma}^{\text{labeled}, \hat\pi}$  the connected component whose disks have cyclic order $\hat\pi \in \hat \Pi$. Similarly, if $\Gamma$ is a (pre-)graded $W$-spin graph, then we will denote by $\oCM_{\Gamma}^{W, \text{labeled}, \hat\pi}$ the  connected component associated to the cyclic order $\hat \pi$. Note that if we want to consider the unlabeled case, then there is a set of cyclic orderings that correspond to the same connected component. Here, we will continue to  use the notation $\oCM_{\Gamma}^{\hat\pi}$ and $\oCM_{\Gamma}^{W, \hat \pi}$ for the corresponding connected component (which may be represented by multiple $\hat \pi \in \hat \Pi$).

Crucially, the families of orientations also must extend over the boundary in a way that is compatible with the lower dimensional moduli spaces. One key case for this is when one has exactly one boundary edge. We will now give the notation for this key case.

\begin{nn}\label{nn: boundary orientation case}
Let $\Gamma$ be a graph with two open vertices $v_1$ and $v_2$, connected by an edge $e$. Suppose that each vertex $v_i$ has internal tails $I_i$, boundary tails $B_i$, and a unique boundary half-edge $h_i$ which is part of the edge $e$. Take $B = B_1 \sqcup B_2$ and $I = I_1 \sqcup I_2$ and assume that the marking functions $m^I$ and $m^B$ are injective.

We will also denote by $v_i$ the graph $\Gamma_{v_i}$ given by taking the subgraph of $\detach_e\Gamma$ which is the connected component containing the vertex $v_i$. We write $v_i'$ for the graph $\text{For}_{h_i}(v_i)$.

We can construct an ordering $\pi$ on $B$ from the orderings $\pi_i$ on the two vertices $v_1$ and $v_2$ in a unique way. Intuitively, this is done by starting at $h_1$ and following the arrow in 
Figure~\ref{fig:boundary edge orient}(A). Algorithmically, it is done as follows.

For any $\Sigma \in \CM_{\Gamma}$, with normalization $\Sigma_1 \sqcup \Sigma_2$ (where $\Sigma_i$ corresponds to $v_i$), we have an ordering $\hat{\pi}_{v_i}$ on the boundary half-edges, induced by the natural orientation on $\partial\Sigma_i$. Using $\hat{\pi}_{v_1}$ we can define an ordering $\pi_1$ on the half-edges of $v_1$ by starting after the half-edge $h_1$ and following the natural orientation on $\partial\Sigma_1$, keeping $h_1$ at the end of the ordering. Similarly, we denote by $\pi_2$ the ordering on the half-edges of $v_2$ given by starting at $h_2$ and following the natural orientation. We can restrict the orderings $\pi_i$ to the boundary tails by omitting the half-edges; in this case, we will still denote this ordering by $\pi_i$. 

Lastly, we will denote by $\pi$ the unique ordering of $B$ such that, for any $\Sigma\in \CM_\Gamma$ as above,
\begin{itemize}
\item under $\pi$, the marked points of $\Sigma_1$ appear before those of $\Sigma_2$;
\item when $\pi$ is restricted to the marked points of $\Sigma_i$, it agrees with the ordering $\pi_i$ on the boundary tails. 
\end{itemize}
 Denote by $\hat \pi$ a cyclic ordering on $B$ induced by $\pi$.
\end{nn}

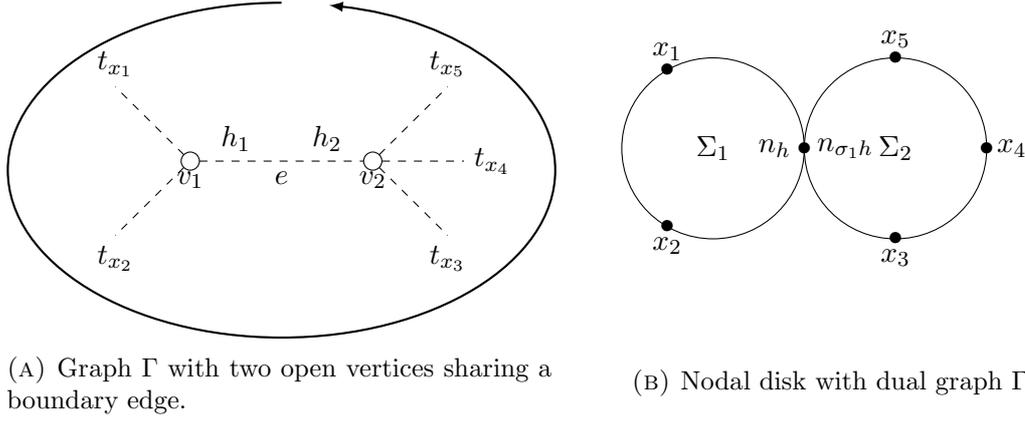
\begin{figure}
 \begin{subfigure}{.45\textwidth}
  \centering
  \begin{tikzpicture}[scale=0.6]

	\draw (0,0) circle (.2cm);
	\draw (4,0) circle (.2cm);
	
	\draw[dashed] (0.2,0) -- (3.8,0);
	\node[below] at (2,0) {$e$};
	
	\draw[dashed] (4.1414,.1414) -- (5.6414, 1.6414);
	\draw[dashed] (4.1414,-.1414) -- (5.6414, -1.6414);
	
	\draw[dashed] (-.1414,.1414) -- (-1.6414, 1.6414);
	\draw[dashed] (-.1414,-.1414) -- (-1.6414, -1.6414);
	
	\draw[dashed] (4.2,0) -- (6,0);
	
	\node[above] at (5.6414, 1.6414) {$t_{x_5}$};
	\node[below] at (5.6414, -1.6414) {$t_{x_3}$};
	\node[right] at (6,0) {$t_{x_4}$};
	
	\node[above] at (-1.6414, 1.6414) {$t_{x_1}$};
	\node[below] at (-1.6414, -1.6414) {$t_{x_2}$};
	
	\node[above] at (1,0) {$h_1$};
	\node[above] at (3,0) {$h_2$};
	
	\draw [-latex, thick, rotate=0] (2,3.5) arc [start angle=-270, end angle=80, x radius=6cm, y radius=3.7cm];

	\node[below] at (0,0) {$v_1$};
	\node[below] at (4,0) {$v_2$};
  \end{tikzpicture}
  \caption{Graph $\Gamma$ with two open vertices sharing a boundary edge. }
 \end{subfigure}
\begin{subfigure}{.45\textwidth}
  \centering
  \begin{tikzpicture}[scale=0.6]

  \draw (0,0) circle (2cm);
  \draw (4,0) circle (2cm);

\node at (1.35, 0) {$n_h$};
\node at (2.9, 0) {$n_{\sigma_1 h}$};
\node at (2,0) {$\bullet$};
\node at (-1, 1.732) {$\bullet$};
\node at (-1, -1.732) {$\bullet$};
\node[above] at (-1, 1.732) {$x_1$};
\node[below] at (-1, -1.732) {$x_2$};
\node at (4,2) {$\bullet$};

\node at (4,-2) {$\bullet$};
\node at (6,0) {$\bullet$};
\node[right] at (6,0) {$x_4$};
\node[below] at (4,-2) {$x_3$};
\node[above] at (4,2) {$x_5$};

\node at (0,0) {$\Sigma_1$};
\node at (4,0) {$\Sigma_2$};

  \end{tikzpicture}
  \vspace{1cm}
  \caption{Nodal disk with dual graph $\Gamma$.}
 \end{subfigure}

 \caption{An example of a graph $\Gamma$ and its corresponding nodal disk from Notation~\ref{nn: boundary orientation case}. The arrow  in (A) corresponds to the unique ordering $\pi$ of boundary markings by just following the arrow around (that is, $\pi$ is the ordering $(t_{x_1} \to t_{x_2} \to t_{x_3} \to t_{x_4} \to t_{x_5})$).}
 \label{fig:boundary edge orient}
 \end{figure}

We now define and then construct families of orientations $\{\tilde{\mathfrak{o}}^{\pi}_{0,B,I}\}$ and $\{\ooo^\pi_{0,B, \{a_i\}_{i\in I}}\}$ on the moduli spaces $\oCM_{0,B,I}^{\text{labeled},\hat{\pi}}$ and the Witten bundle over $\oCM^{1/r,\text{labeled},\hat{\pi}}_{0,B,\{a_i\}_{i\in I}}$, respectively.

\begin{definition}
(i) A \emph{family of orientations} $\{\tildooo_{0,B,I}^{\pi}\}$ on the moduli spaces $\oCM_{0,B,I}^{\text{labeled},\hat{\pi}}$ consists of the data of an orientation $\tildooo_{0,B,I}^{\pi}$ on the moduli space $\oCM_{0,B,I}^{\text{labeled},\hat{\pi}}$ for each tuple $(B,I, \pi, \hat\pi)$ where $|B| + 2|I| \ge 3$. Here, each tuple $(B,I, \pi, \hat\pi)$ consists of: $B$ a set of boundary marked points, $I$ a set of internal marked points, $\hat\pi$ a cyclic ordering of $B$, and $\pi$ an ordering of $B$ that induces $\hat\pi$.

(ii) A \emph{family of orientations} $\{\ooo^\pi_{0,B, \{a_i\}_{i\in I}}\}$ on the Witten bundles  $\cW\to\oCM^{1/r,\text{labeled},\hat{\pi}}_{0,B,\{a_i\}_{i\in I}}$  over the moduli space of $r$-spin disks consists of the data of an orientation $\ooo^\pi_{0,B, \{a_i\}_{i\in I}}$ on the Witten bundle over the moduli space $\oCM^{1/r,\text{labeled},\hat{\pi}}_{0,B,\{a_i\}_{i\in I}}$ of $r$-spin disks, for each tuple $(B, \{a_i\}_{i\in I}, \pi, \hat\pi)$ where $|B| + 2|I| \ge 3$. Here,  each tuple $(B, \{a_i\}_{i\in I}, \pi, \hat\pi)$ consists of: $B$ the set of boundary marked points, $\{a_i\}$ the set of twists on the $|I|$ internal marked points, $\hat\pi$ a cyclic ordering of the boundary marked points, and $\pi$ an ordering of $B$ that induces $\hat\pi$.
\end{definition}

\begin{lemma}\label{lem:or_mod}
There exists a unique family of orientations $\{\tilde{\mathfrak{o}}^{\pi}_{0,B,I}\}$
for the spaces $\oCM^{\text{labeled},\hat{\pi}}_{0,B,I}$ with the following properties:
\begin{enumerate}
\item\label{it:or_mod_zero}
In the zero-dimensional case where $k = l=1$, the orientation is positive, while when $k = 3$ and $l=0$, the orientations are negative.
\item\label{it:or_mod_cov}
If $f^B:B\to B'$ and $f^I:I\to I'$ are bijections and $F:\oCM_{0,B,I}^{\textup{labeled}}\to\oCM_{0,B',I'}^{\textup{labeled}}$ is the induced map, then $\tilde{\mathfrak{o}}_{0,B,I}^\pi=F^* \tilde{\mathfrak{o}}_{0,B',I'}^{f^B\circ \pi}$.
In particular,
the action of any $g\in \text{Sym}(I)$ preserves the orientation on $\oCM_{0,B,I}^{\textup{labeled}}$.
\item\label{it:or_mod_perm}
Fix an integer $h$, and let $\pi$ be an order of $B,~g\in \text{Sym}(B)$
be such that $g$ sends $x_{\pi(i)}$ to $x_{\pi(i+h)}$ cyclically.  Then $g$ preserves the orientation of $\oCM_{0,B,I}^{\textup{labeled},\hat\pi}$ if and only if $h(|B|-1)$ is even.
\item\label{it:or_mod_forget_int}
The orientation $\tilde{\mathfrak{o}}^{\pi}_{0,k,l+1}$ agrees with the orientation induced from $\tilde{\mathfrak{o}}^{\pi}_{0,k,l}$ by the fibration $\text{For}_{\varnothing, l+1}:\oCM_{0,k,l+1}^{\textup{labeled}} \to \oCM_{0,k,l}^{\textup{labeled}}$ and the complex orientation on the fiber. (See Example
\ref{ex:forget internal} for the forgetful map.)
\item\label{it:or_mod_forget_bdry}
On $\oCM_{0,k+1,l}^{\textup{labeled},(1\to2\to\cdots \to k+1\to 1)}$, we have $\tilde{\mathfrak{o}}_{0,k+1,l}^{(1\to2\to\cdots \to k+1)}=o_{\text{For}_{k+1, \varnothing}^{-1}(\Sigma)}\otimes \text{For}_{k+1, \varnothing}^*\tilde{\mathfrak{o}}_{0,k,l}^{(1\to2\to\cdots \to k)}$.
\item\label{it:or_mod_bdry} Let $\Gamma$ be a graph with injective boundary marking, two open vertices, $v_1$ and $v_2$, connected by an edge $e$, as in Notation~\ref{nn: boundary orientation case}.
We have \[\det(T\oCM_{0,B,I}^{\textup{labeled},\hat\pi})|_{\CM_\Gamma} = \det(N)\otimes \det(Tf_{h_1})\otimes(\det(T\CM_{v'_1})\boxtimes\det(T\CM_{v_2})),\]
where $N$ is the normal bundle and $Tf_{h_1}$ is the tangent bundle of the fiber of the map which forgets the tail $h_1$ (see Example~\ref{ex:internal edge example}). If we write $o_N$ for the canonical orientation of $N$ and $o_{h_1}$ for the canonical orientation of $Tf_{h_1}$ described above, then \[\iota_\Gamma^*\tilde{\mathfrak{o}}^\pi=(-1)^{(|B_1|-1)|B_2|}o_N\otimes o_{h_1}\otimes
(\tilde{\mathfrak{o}}^{\pi_1}_{0,B_1,I_1}\boxtimes\tilde{\mathfrak{o}}^{\pi_2}_{0,\{h_2\}\cup B_2,I_2}).
\]
Moreover, using Item \eqref{it:or_mod_forget_bdry}, this equation can also be written as
\[\iota_\Gamma^*\tilde{\mathfrak{o}}^\pi=(-1)^{(|B_1|-1)|B_2|}o_N\otimes
(\tilde{\mathfrak{o}}^{\pi_1}_{0,B_1\cup\{h_1\},I_1}\boxtimes\tilde{\mathfrak{o}}^{\pi_2}_{0,\{h_2\}\cup B_2,I_2}).\]
\item\label{it:or_mod_int}
Let $\Gamma$ be a connected graph with injective boundary marking, two vertices, an open vertex $v^o$ and a closed vertex $v^c$.  We have
\[
\det(T\CM_\Gamma) = \det(N)\otimes ( \det(T\CM_{v^c})\boxtimes\det(T\CM_{v^o})),
\] where $N$ is again the normal bundle.  Then, for any order $\pi,$
\begin{equation}\label{eq:induced_or_interior_moduli}
\iota_\Gamma^*\tilde{\mathfrak{o}}^\pi = o_N\otimes (\tilde{\mathfrak{o}}^\pi_{v^o}\boxtimes \tilde{\mathfrak{o}}_{v^c}),
\end{equation}
where $o_N$ and $\tilde{\mathfrak{o}}_{v^c}$ are the canonical complex orientations.
\end{enumerate}
\end{lemma}
\begin{proof}
Parts (1), (3), (4), (5) are parts (1), (2), (3), and (4) of Proposition 3.12 of \cite{BCT:I}.  Part (2) corresponds to the covariant property (Definition 3.11 of loc.\ cit.) for the unique family of orientations constructed there. Lastly, parts (6) and (7) are Lemma 3.15(1) and (2) of loc.\ cit., respectively.
\end{proof}

\begin{rmk}
To clean up notation, we will often instead use restriction notation $|_{\oCM_\Gamma}$ to denote pulling back under an inclusion map $\iota_\Gamma$, viewing the moduli space $\oCM_{\Gamma}$ as a boundary strata of $\oCM_{\smooth{\Gamma}}$. For example, in Equation~\eqref{eq:induced_or_interior_moduli} above, we write $\tilde{\mathfrak{o}}^\pi |_{\oCM_\Gamma}$ instead of $\iota_{\Gamma}^*\tilde{\mathfrak{o}}^\pi $.
\end{rmk}

We now turn to the family of orientations of the Witten bundles.
\begin{lemma}\label{lem:or_single_Witten}
There exists a family of orientations ${\mathfrak{o}}^{\pi}_{0,B,\{a_i\}_{i\in I}}$
for the Witten bundles $\cW\to\oCM^{1/r,\text{labeled},\hat{\pi}}_{0,B,\{a_i\}_{i\in I}}$ with the following properties:
\begin{enumerate}
\item\label{it:or_Witten_zero}
The orientation is positive if the rank is $0$ and $I$ is non-empty ($|B|=|I|=1,a_i=0$).
\item\label{it:or_Witten_cov}
Assume $f^B:B\to B'$ and $f^I:I\to I'$ are bijections and \[F:(\cW\to\oCM_{0,B,I}^{1/r,\textup{labeled}})\to(\cW\to\oCM_{0,B',I'}^{1/r,\textup{labeled}})\] is the induced map, lifted to the Witten bundle. Then ${\mathfrak{o}}_{0,B,\{a_i\}_{i\in I}}^\pi=F^* {\mathfrak{o}}_{0,B',\{a_i\}_{i\in I'}}^{f^B\circ \pi}$. In particular, the action of any $g\in \text{Sym}(I)$ permuting the internal marked points is orientation preserving for $\cW\to\oCM_{0,B,I}^{1/r,\textup{labeled}}$.
\item\label{it:or_Witten_perm}
Fix an integer $h$, and let $\pi$ be an order of $B$. Take $g\in \text{Sym}(B)$ so that $g$ sends $x_{\pi(i)}$ to $x_{\pi(i+h)}$ cyclically.  The action of $g$ on the Witten bundle preserves the orientation  ${\mathfrak{o}}^{\pi}_{0,B,\{a_i\}_{i\in I}}$ of $\cW\to\oCM_{0,B,I}^{1/r,\textup{labeled},\hat\pi}$ if and only if $h(|B|-1)$ is even.
\item\label{it:or_Witten_bdry} Let $\Gamma$ be a graded $r$-spin graph, whose underlying pre-stable dual graph is that in Notation~\ref{nn: boundary orientation case}. Assume that $\tw(h_1) = \alt(h_1) = 0$ and recall that $v_1'$ is the vertex obtained when one takes $v_1$ and forgets the half-edge $h_1$. Then
 $\mathfrak{o}^\pi$ agrees with
$\mathfrak{o}^{\pi_1}_{0,B_1,I_1}\boxtimes\mathfrak{o}^{\pi_2}_{0,\{h_2\}\cup B_2,I_2}$ under the isomorphism
\[\det(\cW_{\Gamma}) = \det(\cW_{v_1}) \boxtimes \det(\cW_{v_2}) = \For_{h_1}^*\det(\cW_{v'_1}) \boxtimes \det(\cW_{v_2}).\]
\item\label{it:or_Witten_int}
Let $\Gamma$ be a connected graph with injective boundary markings, two vertices, an open vertex $v^o,$ whose tails are cyclically ordered by $\pi,$ and a closed vertex $v^c$.  Then for any order $\pi$ whose induced cyclic order is $\hat\pi,$ the pull-back orientation $\iota_\Gamma^*\mathfrak{o}^\pi$ on the Witten bundle agrees with $\mathfrak{o}^\pi_{v^o}\boxtimes\mathfrak{o}_{v^c}$ under the isomorphism \[\det(\cW_{\Gamma})\cong \Detach^* \bigg( \det(\cW_{v^o})\boxtimes \det(\cW_{v^c}) \bigg).\]
\end{enumerate}
\end{lemma}
\begin{proof}
Item (1) is induced by the choice given in Definition 5.10 in \cite{BCT:I}. Item (2) is Theorem 5.2(i) of loc. cit. Item (3) is Lemma 5.6 of loc cit. Item (4) is Lemma 5.13 and the discussion above it in loc. cit. Lastly, Item (5) is Lemma 5.14 of loc. cit.
Here, in Item (5), we note that we have used Observation \ref{rmk:orientation and exact sequence}: In case $e$ is Neveu-Schwarz, the Witten bundle decomposes as a direct sum by  Proposition \ref{pr:decomposition}(\ref{it:NS}), but when $e$ is Ramond, we use Proposition~\ref{pr:decomposition}(\ref{it:internalRamondWitten})  along with Observation \ref{rmk:orientation and exact sequence} in order to write $\iota_\Gamma^*\ooo^\pi = \ooo_{v^o}^\pi \boxtimes \ooo_{v^c}$.
\end{proof}

\begin{rmk}\label{rmk:uniqueness}
In fact, there are precisely two families $\{\mathfrak{o}^{\pi}_{0,B,\{a_i\}_{i\in I}}\},\{\mathfrak{o}^{'\pi}_{0,B,\{a_i\}_{i\in I}}\}$ of orientations of the Witten bundle which satisfy Lemma \ref{lem:or_single_Witten}. They are related by $(-1)^{|B|-1}\mathfrak{o}^{\pi}_{0,B,\{a_i\}_{i\in I}}=\mathfrak{o}^{'\pi}_{0,B,\{a_i\}_{i\in I}}$. In \cite[Section 5]{BCT:I}, the family of orientations
$\mathfrak{o}^{\pi}$ under which $\langle\tau^1_0\sigma^2\rangle^{1/r,o}=1$ was chosen, and we continue with this choice.
\end{rmk}

We require another useful result concerning the induced orientations on the boundary that involves changing the cyclic ordering of the boundary marked points.
\begin{cor}\label{cor:different_starting_pt}
Suppose $\Gamma$ is an $r$-spin boundary graph as described in Notation~\ref{nn: boundary orientation case}, using the unique ordering $\pi$ stated there. Assume further that $\tw(h_1) = \alt(h_i) = 0$. Let $\pi'$ be a different ordering given by shifting $\pi$ by $m$ steps, i.e., $\pi'(i)=\pi(m+i\pmod{|B|})$.
\begin{enumerate}
\item\label{it:SameVertexOrderShift} If $0\leq m<|B_1|$, then we have that
\begin{equation}\label{eq:ChangeOf b}
\tilde{\mathfrak{o}}^{\pi'}|_{\CM_\Gamma}=(-1)^{(|B_1|-1)|B_2|+m|B_2|}o_N\otimes o_{h_1}\otimes
(\tilde{\mathfrak{o}}^{\pi'_1}_{0,B_1,I_1}\boxtimes\tilde{\mathfrak{o}}^{\pi'_2}_{0,\{h_2\}\cup B_2,I_2}),
\end{equation}
\begin{equation}\label{orientation shift order with half-edge included}
\tilde{\mathfrak{o}}^{\pi'}|_{\CM_\Gamma}= (-1)^{(|B_1|-1)|B_2|+m(|B_2|-1)}o_N\otimes
(\tilde{\mathfrak{o}}^{\pi'_1}_{0,\{h_1\}\cup B_1,I_1}\boxtimes\tilde{\mathfrak{o}}^{\pi'_2}_{0,\{h_2\}\cup B_2,I_2}),
\end{equation}
and
\begin{equation}\label{eq:changeOrderOnTwoVertices}
\mathfrak{o}^{\pi'}=(-1)^{m|B_2|}\mathfrak{o}^{\pi'_1}_{0,B_1,I_1}\boxtimes\mathfrak{o}^{\pi'_2}_{0,\{h_2\}\cup B_2,I_2}.
\end{equation}

\item\label{it:DifferentVertexOrderShift} If $|B_1| \leq m <|B|$, then we have that
\[\ooo^{\pi'}=(-1)^{m(|B_1|-1)}\mathfrak{o}^{\pi'_2}_{0,\{h_2\}\cup B_2,I_2}\boxtimes\mathfrak{o}^{\pi'_1}_{0,B_1,I_1}.\]
Moreover, if $m=|B_1|$ so  $\pi_2''$ is the cyclic order that takes $h_2$ to be the last tail, then we have that
\begin{equation}
\label{eq:Witten-orientation-special-case}
\ooo^{\pi'}=\mathfrak{o}^{\pi''_2}_{0,B_2\cup\{h_2\},I_2}\boxtimes\mathfrak{o}^{\pi'_1}_{0,B_1,I_1}.
\end{equation}
\end{enumerate}
\end{cor}
\begin{proof}
First, we consider the case where $0 \leq m<|B_1|$. We define $\pi_1'$ as the order of $B_1$ induced by restricting $\pi'$, meaning $\pi'_1(i)=\pi_1(m+i~(\text{mod} |B_1|))$. Note that the order of $B_2$ induced by restricting $\pi'$ stays the same, i.e.,  $\pi'_2=\pi_2$. By applying Lemma \ref{lem:or_mod}(\ref{it:or_mod_perm}) and (\ref{it:or_mod_bdry}), we have that
\begin{equation}\begin{aligned}\label{orientation shift order}
\tilde{\mathfrak{o}}^{\pi'}|_{\CM_\Gamma}& = (-1)^{m(|B| -1)} \tilde{\mathfrak{o}}^{\pi}|_{\CM_\Gamma}\\
	&= (-1)^{m(|B|-1)} \left((-1)^{(|B_1|-1)|B_2|}o_N\otimes o_{h_1}\otimes
(\tilde{\mathfrak{o}}^{\pi_1}_{0,B_1,I_1}\boxtimes\tilde{\mathfrak{o}}^{\pi_2}_{0,\{h_2\}\cup B_2,I_2})\right) \\
	&= (-1)^{(|B_1|-1)|B_2|+m|B_2|}o_N\otimes o_{h_1}\otimes
(\tilde{\mathfrak{o}}^{\pi'_1}_{0,B_1,I_1}\boxtimes\tilde{\mathfrak{o}}^{\pi'_2}_{0,\{h_2\}\cup B_2,I_2}).
\end{aligned}\end{equation}
Note that if we want to instead replace $o_{h_1}\otimes
\tilde{\mathfrak{o}}^{\pi'_1}_{0,B_1,I_1}$ with $\tilde{\mathfrak{o}}^{\pi'_1}_{0,\{h_1\} \cup B_1,I_1}$, we have the analogous computation but using the last line of Lemma \ref{lem:or_mod}(\ref{it:or_mod_bdry}) and the fact that $\tilde{\mathfrak{o}}^{\pi_1}_{0,B_1\cup\{h_1\},I_1} = (-1)^{m|B_1|}\tilde{\mathfrak{o}}^{\pi_1'}_{0,B_1\cup\{h_1\},I_1}$. In this fashion, we obtain \eqref{orientation shift order with half-edge included}.
Similarly, we apply Lemma \ref{lem:or_single_Witten}(\ref{it:or_Witten_perm}) and (\ref{it:or_Witten_bdry}) to obtain that $\mathfrak{o}^{\pi'}=(-1)^{m|B_2|}\mathfrak{o}^{\pi'_1}_{0,B_1,I_1}\boxtimes\mathfrak{o}^{\pi'_2}_{0,\{h_2\}\cup B_2,I_2}$.

Next, we consider the case $|B_1| \leq m <|B|$. Here, we can see that the induced order on $B_1$ is the same, i.e., $\pi'_1 = \pi_1$; however, we compute that $\pi'_2(i)=\pi_2(i+m+1-|B_1|~\pmod{|B_2|+1})$. By Lemma \ref{lem:or_single_Witten}(\ref{it:or_Witten_perm}), we then have the following relations using the new orderings:
$$
\ooo^{\pi'}=(-1)^{m(|B|-1)}\ooo^\pi; \qquad \ooo^{\pi'_2}=(-1)^{(1+m+|B_1|)|B_2|}\ooo^{\pi_2}.
$$
Recall that, by using Observation~\ref{obs:open_rank1}, we can see that the parities of the rank of the bundles $\cW_{v_1'}$ and $\cW_{v_2}$ are $|B_1|-1$ and $|B_2|$, respectively. We then can see that, by using Lemma \ref{lem:or_single_Witten}(\ref{it:or_Witten_bdry}), we have that
\begin{equation}\begin{aligned}
\ooo^{\pi'} &= (-1)^{m(|B|-1)} \ooo^\pi \\
	&= (-1)^{m(|B|-1)}  \mathfrak{o}^{\pi_1}_{0,B_1,I_1}\boxtimes\mathfrak{o}^{\pi_2}_{0,\{h_2\}\cup B_2,I_2} \\
	&= (-1)^{m(|B|-1) + (1+m+|B_1|)|B_2|}\mathfrak{o}^{\pi_1'}_{0,B_1,I_1}\boxtimes\mathfrak{o}^{\pi_2'}_{0,\{h_2\}\cup B_2,I_2} \\
	&= (-1)^{m(|B_1| - 1)}\mathfrak{o}^{\pi_2'}_{0,\{h_2\}\cup B_2,I_2}  \boxtimes \mathfrak{o}^{\pi_1'}_{0,B_1,I_1}.
\end{aligned}\end{equation}
\end{proof}

\subsubsection{Orientations in Rank two} We now turn to the rank two case to develop a family of orientations on $\oCM_{\Gamma}^W$ for graded graphs $\Gamma$ and $W= x^r+y^s$. 

\begin{nn}\label{nn:N_pi}
Let $\pi$ be an ordering of the boundary markings $B = B_1\sqcup B_2\sqcup B_{12}$,
where $B_1$ (resp.\ $B_2, B_{12}$) is the set of boundary $r$-tails
(resp.\ boundary $s$-tails, fully twisted tails).
Write $\hat\pi$ for the induced cyclic order. 
We write $x_1,\ldots,x_{k_1}$ and $y_1,\ldots,y_{k_2}$ for the boundary marked points on a disk with twists $(r-2,0)$ and $(0,s-2)$, respectively. Let
\[
N(\pi)=\#\{(x_i,y_j)\,|\, \hbox{the $s$-point $y_j$ precedes the $r$-point $x_i$
in the $\pi$-order}\}.
\]
Denote by $\pi|_{r}$ the restriction of $\pi$ to $B_1\sqcup B_{12},$ and by $\pi|_{s}$ the restriction to $B_2\sqcup B_{12}$.

Given a rooted connected smooth graded 
graph $\Gamma$ and a cyclic ordering $\hat\pi\in\hat\Pi_\Gamma$, let $\pi^\std$ be the unique ordering of boundary tails which starts from the root and induces $\hat \pi$. Equivalently, on the level of a disk in the moduli space, this is the order of markings along the orientation of the boundary, starting from the root.
\end{nn}

\begin{definition}\label{def:or_and_rel_or1}
Given a smooth graded graph $\Gamma$ with bijective boundary marking, a cyclic order $\hat\pi\in\hat\Pi_\Gamma$, and an order $\pi$ which induces the cyclic order $\hat\pi$, we define the orientation
\[
\hatooo^\pi_\Gamma = (-1)^{N(\pi)}\tildooo^\pi,
\]
on the underlying moduli space of disks $\CM^{\text{labeled},\hat\pi}_{\Gamma}$, where $N(\pi)$ is defined in Notation \ref{nn:N_pi}. Recall that, by Observation \ref{obs:forgetful}, we have an isomorphism of vector bundles over $\oCM^W_\Gamma$
 \[\cW_1\simeq \text{For}_{\text{spin}\neq 1}^*(\cW_1),\]
 where  $\text{For}_{\text{spin}\neq 1}:\oCM^W_\Gamma\to\oCM^{1/r}_{\Gamma_r}$
is as defined in Definition \ref{def:for spin}, with
the graph $\Gamma_r=\text{for}_{\text{spin}\neq 1}(\Gamma)$.

The ordering given by the restriction of $\pi$ to unforgotten tails, denoted by $\pi|_r$ as seen in Notation \ref{nn:N_pi}, induces an orientation $\ooo^{\pi|_r}_{\Gamma_{r}}$ on $\cW_1\to\oCM^{1/r}_{\Gamma_r}$ given in Lemma~\ref{lem:or_single_Witten}. We then use the isomorphism of vector bundles under $\text{For}_{\text{spin}\neq 1}$ to pull back this orientation to obtain the orientation on the bundle $\cW_1\to\oCM^W_\Gamma$, which we will denote by $\ooo_r^{\pi|_{r}}=\ooo_{r,\Gamma}^{\pi|_{r}}$. Define $\ooo_s^{\pi|_s}=\ooo_{s,\Gamma}^{\pi|_s}$ similarly for $\cW_2$. Finally, we can combine these three orientations to obtain a relative orientation
\[
o_{\Gamma}^\pi := \hatooo^\pi_\Gamma \otimes\ooo_r^{\pi|_{r}}\otimes\ooo_s^{\pi|_s},
\]
of the full Witten bundle $\cW$, or equivalently an orientation on the total space of $\cW\to\oCM^W_{\Gamma}$.

When $\Gamma = \Gamma_{0,k_1,k_2,1,\{(a_i,b_i)\}_{i\in I}}^{\text{labeled}}$, we write \[\hatooo_{0,k_1,k_2,1,\{(a_i,b_i)\}_{i\in I}}=\hatooo^{\pi^\std}_\Gamma,~o_{0,k_1,k_2,1,\{(a_i,b_i)\}_{i\in I}}=o_\Gamma^{\pi^\std},\]
where, on each connected component of the moduli, $\pi^\std$ is the order of markings along the orientation of the boundary, starting from the root as defined in Notation~\ref{nn:N_pi}. When $\Gamma = \Gamma_{0,k_1+1,k_2+1,0,\{(a_i,b_i)\}_{i\in I}}^{\text{labeled}}$
, we write \[\hatooo_{0,k_1+1,k_2+1,0,\{(a_i,b_i)\}_{i\in I}}=\hatooo^{\pi}_\Gamma,~o_{0,k_1+1,k_2+1,0,\{(a_i,b_i)\}_{i\in I}}=o_\Gamma^{\pi},\]
where $\pi$ is an arbitrary order which agrees with the cyclic order of the points.
\end{definition}

This latter orientation a priori depends on the choice of $\pi$, but we have the following proposition.

\begin{prop}\label{obs:orders_well_def_and_r_s_change}
Suppose $\Gamma$ is a graph that can be written either in the form
$\Gamma_{0,k_1,k_2,1,\{(a_i,b_i)\}_{i\in I}}^{\text{labeled}}$ or $\Gamma_{0,k_1+1,k_2+1,0,\{(a_i,b_i)\}_{i\in I}}^{\text{labeled}}$.
The following hold:
\begin{enumerate}
\item
The relative orientation 
 $o_{0,k_1+1,k_2+1,0,\{(a_i,b_i)\}_{i\in I}}$
is independent of the choice of $\pi$.
\item
If we change the roles of the $r$- and $s$-bundles in the definitions of $\hatooo$ and $o$ then the orientations $o_{0,k_1+1,k_2+1,0,\{(a_i,b_i)\}_{i\in I}}$ changes by $(-1)^{k_1+k_2+1}$ while $o_{0,k_1,k_2,1,\{(a_i,b_i)\}_{i\in I}}$ does
not change.
\item The relative orientations
\[
\hbox{$o_{0,k_1,k_2,1,\{(a_i,b_i)\}_{i\in I}}$ and $o_{0,k_1+1,k_2+1,0,\{(a_i,b_i)\}_{i\in I}}$}
\]
descend to relative orientations of the Witten bundles over the \emph{unlabeled} moduli
spaces $\oCM^W_{0,k_1,k_2,1,\{(a_i,b_i)\}_{i\in I}}$ and
$\oCM^W_{0,k_1+1,k_2+1,0,\{(a_i,b_i)\}_{i\in I}}$.
\item The orientation
$\hatooo_{0,k_1,k_2,1,\{(a_i,b_i)\}_{i\in I}}$ descends to an orientation of $\oCM^W_{0,k_1,k_2,1,\{(a_i,b_i)\}_{i\in I}}$.
\end{enumerate}
\end{prop}

\begin{proof}
(1) Consider the case $\Gamma = \Gamma_{0,k_1+1,k_2+1,0,\{(a_i,b_i)\}_{i\in I}}^{\hat\pi}$. Change $\pi$ to the ordering $\pi'$ which induces the same cyclic order but moves the first point of $\pi$ to the end. Then, if the first point is an $r$-point, we have the following changes:
 $\tildooo$ changes by $(-1)^{k_1+k_2+1}$ by Lemma~\ref{lem:or_mod}(\ref{it:or_mod_perm}),
 $N(\pi')=N(\pi)+k_2+1$, and $\ooo^{\pi'|_r}=(-1)^{k_1}\ooo^{\pi|_r}$ and $\ooo^{\pi'|_s}=\ooo^{\pi|_s}$, by Lemma~\ref{lem:or_single_Witten}(\ref{it:or_Witten_perm}).
Thus, $o$ does not change. Similar analysis shows it does not change if the $\pi$-first point is an $s$-point.

(2) Changing the roles of $r$ and $s$ changes $N(\pi)$ to $(k_1+1)(k_2+1)-N(\pi)$ and we can compute that swapping $\cW_r,~\cW_s$ adds a sign of $k_1k_2$ as well using Equation~\eqref{eq:open_rank2}. Thus, we have that the relative orientation changes by $(-1)^{(k_1+1)(k_2+1)+k_1k_2} = (-1)^{k_1 + k_2 +1}$.

In the case where $\Gamma = \Gamma_{0,k_1,k_2,1,\{(a_i,b_i)\}_{i\in I}}^{\hat\pi}$, exchanging the roles of $r$ and $s$ changes $N(\pi)$ to $k_1k_2-N(\pi)$ and swapping $\cW_r,~\cW_s$ adds a sign of $k_1k_2$ as well. Thus, the total expression does not change.

(3) and (4).
Let $G$ be the group of automorphisms of the boundary labelings $B$ which
preserve twists. Then from Lemmas \ref{lem:or_mod}(2) and \ref{lem:or_single_Witten}(2), it follows that $G$ preserves the orientations on the bundles
$\cW\to\oCM_{0,k_1,k_2,1,\{(a_i,b_i)\}_{i\in I}}^{W,\textup{labeled}}$ and $\cW\to\oCM_{0,k_1+1,k_2+1,0,\{(a_i,b_i)\}_{i\in I}}^{W,\textup{labeled}}$, and also on the moduli space
${{\overline{\mathcal{M}}}_{0,k_1,k_2,1,\{(a_i,b_i)\}_{i\in I}}^{W,\textup{labeled}}}$.
Thus, the orientations descend to the quotients by $G$.
\end{proof}

In light of the previous proposition, we may define:
\begin{definition}\label{def:or_and_rel_or2}
We denote by $o_{0,k_1,k_2,1,\{(a_i,b_i)\}_{i\in I}}$ and $o_{0,k_1+1,k_2+1,0,\{(a_i,b_i)\}_{i\in I}}$ the relative orientations induced on the Witten bundles
over the unlabeled moduli spaces $\oCM^W_{0,k_1,k_2,1,\{(a_i,b_i)\}_{i\in I}}$ and $
\oCM^W_{0,k_1+1,k_2+1,0,\{(a_i,b_i)\}_{i\in I}}$ respectively and call them the \emph{canonical relative orientations}. Denote by $\hatooo_{0,k_1,k_2,1,\{(a_i,b_i)\}_{i\in I}}$ also the induced orientation on $\oCM^W_{0,k_1,k_2,1,\{(a_i,b_i)\}_{i\in I}}$ and call it the \emph{canonical orientation} of
$\oCM^W_{0,k_1,k_2,1,\{(a_i,b_i)\}_{i\in I}}$.
If $\Gamma$ is a smooth rooted graded $\RS$-graph, or a smooth graded $\RS$-graph with at most one unrooted component we define the \emph{canonical relative orientation}
\[
o_\Gamma=\boxtimes_{\Lambda\in\Conn(\Gamma)}o_\Lambda.
\]
\end{definition}
Observe that $o_\Gamma$ is well-defined with respect to the order of connected components. Indeed, we can compute that the dimension of the total space of the Witten bundle over a smooth rooted graded $W$-spin disk is even by using \eqref{eq:open_rank2} and the fact that the moduli space has dimension $k_1+k_2 +k_{12}+ 1 \pmod 2$.

The key point of this definition is that it constructs relative orientations
for moduli spaces associated to balanced graphs and to critcal graphs.

\subsubsection{Behavior of the canonical orientations at the boundary strata} Consider a graded $W$-spin graph that can be written in the form
$\Gamma_{0,k_1,k_2,1,\{(a_i,b_i)\}_{i\in I}}^W$. We now describe the behavior of the canonical relative orientations for the following degenerations of the smooth graph:
\begin{itemize}
\item The graphs $\Lambda \in \partial\Gamma^W_{0,k_1,k_2,1,\{(a_i,b_i)\}_{i\in I}}$ with two vertices and  a single internal edge.
\item The graphs $\Lambda \in \partial^\xch\Gamma^W_{0,k_1,k_2,1,\{(a_i,b_i)\}_{i\in I}}$ with two vertices.
\item The graphs $\Lambda \in \partial^0\Gamma^W_{0,k_1,k_2,1,\{(a_i,b_i)\}_{i\in I}}$ with two vertices and a single boundary edge, where the half-node $h$ that is adjacent to the rooted vertex is not fully twisted.
\end{itemize}
Understanding the orientation at these boundaries will be necessary to prove the subsequent results for topological recursion and wall-crossing for open FJRW invariants.

 First, consider the boundary graphs that have a single internal edge and no boundary edges.

\begin{cor}\label{thm:or_and_induced internal}
Let $\Lambda\in\partial\Gamma^W_{0,k_1,k_2,1,\{(a_i,b_i)\}_{i\in I}}$ be a graded $\RS$-graph with two vertices, $v_1$ and $v_2$. Assume that $v_1$ is open and $v_2$ is closed. Then, under the identification
\[\iota_{\Lambda}^*\det(T\oCM^W_{0,k_1,k_2,1,\{(a_i,b_i)\}_{i\in I}}) = \det(N) \otimes \left(\det(T\oCM^W_{v_1}) \boxtimes \det(T \oCM^W_{v_2})\right),\]
we have
\begin{equation}\label{eq:ind_int}\iota_{\Lambda}^*o_{0,k_1,k_2,1,\{(a_i,b_i)\}_{i\in I}}= o_N \otimes \left(o_{v_1} \boxtimes o_{v_2}\right),\end{equation}
where $o_{v_1},o_{v_2}$ are the orientations on the total space associated to the two vertices, $o_{v_2}$ is the standard complex one.
\end{cor}

\begin{proof}
This is a direct consequence of the behavior under internal degenerations of $\tildooo,\ooo_r,\ooo_s$ described in Lemma \ref{lem:or_mod}~\eqref{it:or_mod_int}, and Lemma \ref{lem:or_single_Witten}\eqref{it:or_Witten_int}.
\end{proof}

Next, we consider the case where $\Lambda \in \partial^{\xch}\Gamma$ is a graded $W$-spin graph. We will use the setup where the underlying dual graph of $\Lambda$ is that of $\Gamma$ as described in Notation~\ref{nn: boundary orientation case}. Assume further that the vertex $v_2$ is the exchangeable vertex. Recall from Definition~\ref{def:graphBestiary} that there exists another graph $\Lambda' := \xch(\Lambda) \in \partial^{\xch}\Gamma$, found by swapping the $r$- and $s$-points adjacent to the exchangeable vertex. 

Moreover, for any $W$-spin curve $C$ representing a point in $\oCM^W_{\Lambda}$ there exists a $W$-spin curve $C'$ representing a point in $\oCM^W_{\Lambda'}$ so that $C \sim_{\XCH} C'$ where $\sim_{\XCH}$ is as defined in Definition~\ref{def:SIMX}.  Hence we can define the map
\begin{equation}\label{exchange moduli}
\XCH_v: \oCM^W_{\Lambda} \to \oCM^W_{\Lambda'}
\end{equation}
 given by taking the point represented by $C$ to the point represented by $C'$.

\begin{prop}\label{prop: orient exchange}
Suppose $\Gamma$ is a graph that can be written in the form
$\Gamma_{0,k_1,k_2,1,\{(a_i,b_i)\}_{i\in I}}^{W,\textup{labeled}}$.
Let $\Lambda, \Lambda' \in\partial^{\xch}\Gamma$ so that $\Lambda\sim_{\XCH} \Lambda'$ and both have exactly one boundary edge. 
Then $\XCH_v: \oCM^W_{\Lambda}\to\oCM^W_{\Lambda'}$ reverses the orientation $o_\Gamma$ when restricted to the boundary strata $\oCM^W_{\Lambda}$ and $\oCM^W_{\Lambda'}$, i.e.,
$$
\iota_{\Lambda}^* o_{\Gamma} = -\XCH_v^* \iota^*_{\Lambda'} o_{\Gamma}.
$$
This orientation reversing isomorphism descends to the unlabeled moduli
spaces.
\end{prop}

\begin{proof}
Without loss of generality, we may assume that $\Lambda$ has a single cyclic order $\hat\pi$ of its vertices. Let $\hat\pi'$ be the corresponding cyclic order on the boundary tails of $\Lambda'$. Recall the definition of standard orderings from Notation~\ref{nn:N_pi} and let $\pi^{\std}$ and $(\pi^{\std})'$ be the standard orderings of boundary tails which lift the cyclic orderings $\hat \pi$ and $\hat\pi'$.  Note that, just like everything else in this case, these orderings differ by transposing an $r$-tail and an $s$-tail.

Note that the orientations induced by $\tildooo^{\pi^{\std}},\tildooo^{(\pi^{\std})'}$ on the boundary components $\oCM^W_{\Lambda},\oCM^W_{\Lambda'}$ are the same, by applying Lemma \ref{lem:or_mod}\eqref{it:or_mod_bdry} and then using \ref{lem:or_mod}\eqref{it:or_mod_zero} in either case.  Moreover, the orientations of the $r$-spin and $s$-spin Witten bundles ignore the relative order of the $r$-points and $s$-points as they forget one or the other. As $N(\pi^{\std})=N((\pi^{\std})')+1 \pmod 2$, the identification $\XCH_v$ takes the orientation $\hat\ooo^{\pi^{\std}}$ to $-\hat\ooo^{(\pi^{\std})'}$, implying the result.
The descent to the unlabeled moduli space follows from Theorem
\ref{obs:orders_well_def_and_r_s_change}(3).
\end{proof}

Lastly, consider the boundary graphs that have no internal edges and a single boundary edge. 

\begin{nn}
\label{nn:K1 K2}
For an open vertex $v$ in a graded graph $\Gamma$, we write
\[
\kk_1(v)=k_1(v)+k_{12}(v)-1,~\kk_2(v)=k_2(v)+k_{12}(v)-1.
\]%includes half edges
Note that $\kk_1(v)$ and $\kk_2(v)$ include half-nodes.
\end{nn}

\begin{thm}\label{thm:or_and_induced boundary}
Let $\Lambda\in\partial\Gamma^{W, \textup{labeled}}_{0,k_1,k_2,1,\{(a_i,b_i)\}_{i\in I}}$ be a graded $\RS$-graph where the underlying dual graph is as in Notation~\ref{nn: boundary orientation case}. Suppose that the vertex $v_1$ contains the root and that the sets $\hat\Pi_{v_i}$ contain all cyclic orderings for both vertices $v_i$. Take $\iota_\Lambda$ as in Definition~\ref{def:moduliGraph}.

\begin{enumerate}
\item Suppose that $h_2$ is fully twisted and consequently $\tw(h_1)=\alt(h_1)=(0,0)$.  Then
\[\iota_{\Lambda}^*o_{0,k_1,k_2,1,\{(a_i,b_i)\}_{i\in I}}
= o_N \otimes o_{h_1} \otimes F_\Gamma^*(o_{v'_1} \boxtimes o_{v_2}),\]
where \[
F_{\Lambda}:\oCM^W_{\Gamma} \rightarrow \oCM^W_{v_1}\times\oCM^W_{v_2}\rightarrow\oCM^W_{v'_1}\times\oCM^W_{v_2}
\]
is the composed map $F_\Lambda:=\For_{h_1}\circ\Detach_e$ and $o_N$ and $o_{h_1}$ are the orientations of the bundles $N$ and $Tf_{h_1}$, respectively, as in Lemma \ref{lem:or_mod}(6).
\item\label{part 2: orientation boundary} Suppose that $h_2$ is singly twisted (and consequently so is $h_1$).
\begin{enumerate}
\item If $h_1$ is an $r$-point, then
\begin{equation}\label{eq:ind_WC_1}
\iota_{\Lambda}^*o_{0,k_1,k_2,1,\{(a_i,b_i)\}_{i\in I}}= (-1)^{\kk_2(v_2)}o_N \otimes (o_{v_1} \boxtimes o_{v_2}),
\end{equation}
\item If $h_1$ is an $s$-point, then
\begin{equation}\label{eq:ind_WC_2}
\iota_{\Lambda}^*o_{0,k_1,k_2,1,\{(a_i,b_i)\}_{i\in I}}= (-1)^{\kk_2(v_2)+1}o_N \otimes (o_{v_1} \boxtimes o_{v_2}).
\end{equation}
\end{enumerate}
\end{enumerate}
\end{thm}
We note that in the statement of this theorem, we use the identification
\begin{equation}
\label{eq:Wdec}
\iota_{\Lambda}^*\mathcal{W} = F_\Lambda^*(\mathcal{W}_{v_1'} \boxplus \mathcal{W}_{v_2}).
\end{equation}

\begin{proof}For the proofs below, we will use some new notation. First, recall Definition \ref{def:or_and_rel_or1}. For any $\Lambda$, we write
\[
\ooo^{\pi}_{r,s,\Lambda}=\ooo_{r,\Lambda}^{\pi|_r}\otimes
\ooo_{s,\Lambda}^{\pi|_s}.
\]
We have (suppressing the subscript $\Gamma$ everywhere)
\[
o^\pi = \hat\ooo^{\pi} \otimes \ooo_{r,s}^{\pi} = (-1)^{N(\pi)} \tildooo^{\pi} \otimes \ooo_{r,s}^{\pi}.
\]

{\bf Proof of (1).}
Consider first the connected component $\CM$ of $\CM^W_\Lambda$
with the property that for any $\Sigma\in \CM$,
the root directly follows the node $x_{h_1}$ with respect to the cyclic
order of boundary points on its $v_1$-component. In this case, the standard ordering $\pi^\std$ from Notation~\ref{nn:N_pi} coincides with the ordering $\pi$ given in Notation~\ref{nn: boundary orientation case}.

By Lemma \ref{lem:or_single_Witten}\eqref{it:or_Witten_bdry} applied to
the two Witten bundles separately, (noting that the the cyclic order
in that lemma coincides with the one used here because of the location of
the root), we have
\[
\ooo_r^{\pi^\std|_r}=\ooo_{r}^{\pi^\std_1|_{r}}\boxtimes\ooo_{r}^{\pi_2|_{r}},
\quad\ooo_s^{\pi^\std|_s}=\ooo_{s}^{\pi^\std_1|_{s}}\boxtimes \ooo_{s}^{\pi_2|_{s}},
\]
where $\pi_1^\std$ is the restriction of $\pi^\std$ to the tails of $v_1'$ as in Lemma~\ref{lem:or_single_Witten} and
$\pi_2$ is the restriction of $\pi^\std$ to the tails of $v_2$ and extending by adding $h_2$ in the beginning. We denote by $\ooo_r^{\pi^\std_1|_r}$ the orientation of the $r$-Witten bundle over $\oCM^W_{v_1}$ which corresponds to the order $\pi_1^\std$ and denote the other orientations similarly.
Thus
\begin{align*}
\ooo_{r,s}^{\pi}
= {} &
\ooo_{r}^{\pi_1^\std|_{r}}\boxtimes\ooo_{r}^{\pi_2|_{r}}
\boxtimes
\ooo_{s}^{\pi_1^\std|_{s}}\boxtimes \ooo_{s}^{\pi_2|_{s}}\\
= {} &
(-1)^{a}
\ooo_{r}^{\pi_1^\std|_{r}}\boxtimes\ooo_{s}^{\pi_1^\std|_{s}}
\boxtimes
\ooo_{r}^{\pi_2|_{r}}\boxtimes \ooo_{s}^{\pi_2|_{s}} \\
= {} & (-1)^{a}  \ooo_{r,s,v_1'}^{\pi_1^\std} \boxtimes\ooo_{r,s,v_2}^{\pi_2},
\end{align*}
where $a = \rank(\cW_{v_2,r})\rank(\cW_{v_1,s})=K_1(v_2)K_2(v_1)\pmod 2$.

Similarly, by Lemma \ref{lem:or_mod}\eqref{it:or_mod_bdry},
\[\tildooo^{\pi^\std}|_{\CM} = (-1)^{b}o_N\otimes
o_{h_1}\otimes(\tildooo_{v_1'}^{\pi^\std_1}\boxtimes
\tildooo_{v_2}^{\pi_2}),\quad b=(\kk_1(v_1)+\kk_2(v_1))(\kk_1(v_2)+\kk_2(v_2))\]
where $\pi^\std_1$ is the standard order for the vertex $v_1'$ and $\pi_2$ is the induced order of boundary markings on the vertex $v_2$, as in Notation~\ref{nn: boundary orientation case}, starting from $h_2$. We remark that $\pi_2 = \pi_2^{\std}$, as $h_2$ is a root for $v_2$ in the detaching in this case.

In addition,
\[
N(\pi^\std)=N(\pi^\std_1)+N(\pi_2)+c, \quad c = K_1(v_2)K_2(v_1).
\]
Note here that, in this special case,  $c\equiv a\mod 2$.
We then see that on the Witten bundle over $\CM$ the orientation is
\begin{align*}
o^\pi = {} & (-1)^{N(\pi^\std)} \tildooo^{\pi} \otimes \ooo_{r,s}^{\pi}
	\\
	%= {} &(-1)^{N(\pi^\std)+b}(o_N\otimes o_{h_1}\otimes \tildooo_{v_1'}^{\pi_1^\std} \boxtimes \tildooo_{v_2}^{\pi_2}) \otimes \ooo_{r,s}^\pi \\
	={} & (-1)^{N(\pi^\std)+a+ b}(o_N\otimes o_{h_1}\otimes \tildooo_{v_1'}^{\pi_1^\std} \boxtimes \tildooo_{v_2}^{\pi_2}) \otimes(\ooo_{r,s,v_1'}^{\pi_1^\std} \boxtimes\ooo_{r,s,v_2}^{\pi_2}) \\
	={} & (-1)^{N(\pi^\std) + a + b + b} o_N\otimes o_{h_1}\otimes (\tildooo_{v_1'}^{\pi_1^\std}   \otimes\ooo_{r,s,v_1'}^{\pi_1^\std})\boxtimes (\tildooo_{v_2}^{\pi_2} \otimes\ooo_{r,s,v_2}^{\pi_2}) \\
	={} & (-1)^{(N(\pi^\std_1)+N(\pi_2)+a)+ a} o_N\otimes o_{h_1}\otimes (\tildooo_{v_1'}^{\pi_1^\std}   \otimes\ooo_{r,s,v_1'}^{\pi_1^\std})\boxtimes (\tildooo_{v_2}^{\pi_2} \otimes\ooo_{r,s,v_2}^{\pi_2}) \\
	={} & o_N\otimes o_{h_1}\otimes ((-1)^{N(\pi^\std_1)}\tildooo_{v_1'}^{\pi_1^\std}   \otimes\ooo_{r,s,v_1'}^{\pi_1^\std})\boxtimes ((-1)^{N(\pi_2)}\tildooo_{v_2}^{\pi_2} \otimes\ooo_{r,s,v_2}^{\pi_2}) \\
	= {} & o_N\otimes o_{h_1}\otimes o_{v_1'}^{\pi_1^\std}  \boxtimes o_{v_2}^{\pi_2}.
	%={} & o_N\otimes o_{h_1}\otimes ((-1)^{N(\pi^\std_1)}\tildooo_{v_1'}^{\pi_1^\std}   \otimes\ooo_{v_1}^{\pi_1^\std})\boxtimes ((-1)^{N(\pi_2)}\tildooo_{v_2}^{\pi_2} \otimes\ooo_{v_2}^{\pi_2})
%= {} & (-1)^{a+b+(\kk_1(v_1)+\kk_2(v_1))(\kk_1(v_2)+\kk_2(v_2))+c}o_N\otimes o_{h_1}\otimes \left((-1)^{N(\pi^\std_1)}\tildooo^{\pi^\std_1}\otimes\ooo^{\pi^\std_1}\right)\left(\tildooo^{\pi_2}\otimes\ooo^{\pi_2}\right) \\
%= {} & o_N\otimes o_{h_1}\otimes \left((-1)^{N(\pi^\std_1)}\tildooo^{\pi^\std_1}\otimes\ooo^{\pi^\std_1}\right)\left(\tildooo^{\pi_2}\otimes\ooo^{\pi_2}\right).
\end{align*}
Note the additional $(-1)^{b}$ in the third line came from swapping $\tildooo_{v_2}^{\pi_2} $ and $\ooo_{r,s,v_1'}^{\pi_1^\std}$.

We now tackle the general case. Consider a component $\CM$ of $\CM^W_\Lambda$  where, for any $\Sigma\in\CM$, if we follow the cyclic order of the boundary points on its $v_1$-component from the node $x_{h_1}$ to the root, then we pass through $m_1$ boundary points of twist $(r-2,0)$ and $m_2$ of twist $(0,s-2)$.\footnote{Note that, for our hypothesis in this case to continue to hold, the root must stay in vertex $v_1$.} Now use Corollary \ref{cor:different_starting_pt}(\ref{it:SameVertexOrderShift}) to compute how $a$ and $b$ change with respect to this new ordering. Here, we see that $a$ changes by $\kk_1(v_2)m_1+\kk_2(v_2)m_2$ by Equation~\eqref{eq:changeOrderOnTwoVertices}, after applying it to both spin bundles. The value of $b$ changes by $(m_1+m_2)(\kk_1(v_2)+\kk_2(v_2))$ by Equation~\eqref{eq:ChangeOf b}. Finally, $c$ changes by $m_2\kk_1(v_2)+m_1\kk_2(v_2)$ by examining $N(\pi^\std)$ with the shifted position of the root. Combining these alterations of the constants, we see that there is no change of sign.

\medskip 

{\bf Proof of (2a).} As before, consider first the component $\CM$ of $\CM^W_\Lambda$ with the property that, for any $\Sigma\in \CM$, the root directly follows the node $x_{h_1}$ with respect to the cyclic order of boundary points on its $v_1$-component.   Again, the standard ordering $\pi^\std$ coincides with the ordering $\pi$ given in Notation~\ref{nn: boundary orientation case}. We write $\pi_1^\std$ for the standard ordering on the half-edges of $v_1,$ starting from the root, and $\pi_2$ for the ordering of the half-edges of $v_2$ which extends the orientation on tails by putting $h_2$ first.

For the orientation $\ooo_r^{\pi^\std|_r}$, we apply Lemma~ \ref{lem:or_single_Witten}\eqref{it:or_Witten_bdry} in order to decompose the orientation into a tensor product of the orientations $\ooo^{\pi_1^\std}_{r,v_1}$ and $\ooo^{\pi_2}_{r,v_2}$;
however, to apply the lemma, we must swap the roles of $v_1$ and $v_2$ since the twist at $h_2$ is $(0,s-2)$, requiring some care with the cyclic orderings.  By Equation ~\eqref{eq:Witten-orientation-special-case} in Corollary \ref{cor:different_starting_pt}(\ref{it:DifferentVertexOrderShift}), 
we get that
\[
\ooo_r^{\pi^\std|_r} =\ooo_{r}^{\pi_1^\std|_{r}}\boxtimes\ooo_{r}^{\pi_2|_{r}}.
\]
We can directly apply Lemma~ \ref{lem:or_single_Witten}\eqref{it:or_Witten_bdry} for the $s$-spin Witten bundles, giving the decomposition $\ooo_s^{\pi^\std|_s}=\ooo_{s}^{\pi_1^\std|_{s}}\boxtimes \ooo_{s}^{\pi_2|_{s}}$. Putting this all together, we obtain
\begin{equation}\begin{aligned}\label{Case2a change in a}
\ooo_{r,s}^{\pi^\std} = {} & \ooo_{r}^{\pi_1^\std|_{r}}\boxtimes\ooo_{r}^{\pi_2|_{r}} \boxtimes \ooo_{s}^{\pi_1^\std|_{s}}\boxtimes \ooo_{s}^{\pi_2|_{s}} \\
	={} &  (-1)^{a} \ooo_{r}^{\pi_1^\std|_{r}}\boxtimes \ooo_{s}^{\pi_1^\std|_{s}}\boxtimes\ooo_{r}^{\pi_2|_{r}} \boxtimes \ooo_{s}^{\pi_2|_{s}},
\end{aligned}\end{equation}
where 
$a:= K_1(v_2)K_2(v_1)$.

Secondly, by Lemma \ref{lem:or_mod}\eqref{it:or_mod_bdry}, we have that
\begin{equation}\label{Case2b original b} \tildooo^{\pi^\std}|_{\CM} = (-1)^{b}o_N\otimes (\tildooo_{v_1}^{\pi^\std_1}\boxtimes \tildooo_{v_2}^{\pi_2}), \quad b=(\kk_1(v_1)+\kk_2(v_1)+1)(\kk_1(v_2)+\kk_2(v_2)+1).
\end{equation}
We remark to the reader that we include the half-edge ${h_2}$ in the count $\kk_2(v_2),$ and the half-edge ${h_1}$ in $\kk_1(v_1)$; however, the sets $B_1$ and $B_2$ in the statement of the lemma do not include half-edges.

Thirdly, we decompose $N(\pi^\std)$
into summands again:
\[N(\pi^\std)=N(\pi^\std_1)+N(\pi_2)+c,~c:=
K_1(v_2)(K_2(v_1)-1)-1.
\]
Indeed, we find $c$ in the following way. Note that the $r$-point $h_1$ added at the end of $\pi_1^\std$ adds $K_2(v_1)$ such pairs in the vertex $v_1$, the $s$-point $h_2$ at the beginning of $\pi_2$ adds $k_1(v_2) = K_1(v_2) + 1$ such pairs in the vertex $v_2$ and that there are $K_2(v_1)(K_1(v_2)+1)$ pairs where the $s$-point is in $v_1$ and the $r$-point is in $v_2$. Putting together, we compute
$c = -K_2(v_1)-K_1(v_2)-1+K_2(v_1)(K_1(v_2)+1)$.

Finally, we obtain that: 
\begin{align*}
o^\pi = {} & (-1)^{N(\pi^\std)} \tildooo^{\pi^\std} \otimes \ooo_{r,s}^{\pi^\std}
	\\
	%= {} &(-1)^{N(\pi^\std)+b}(o_N\otimes o_{h_1}\otimes \tildooo_{v_1'}^{\pi_1^\std} \boxtimes \tildooo_{v_2}^{\pi_2}) \otimes \ooo_{r,s}^\pi \\
	={} & (-1)^{N(\pi^\std)+a+ b}(o_N\otimes \tildooo_{v_1}^{\pi_1^\std} \boxtimes \tildooo_{v_2}^{\pi_2}) \otimes(\ooo_{r,s,v_1}^{\pi_1^\std} \boxtimes\ooo_{r,s,v_2}^{\pi_2}) \\
	={} & (-1)^{N(\pi^\std) + a + b + (\kk_1(v_1)+\kk_2(v_1))(\kk_1(v_2)+\kk_2(v_2)+1)} o_N\otimes (\tildooo_{v_1}^{\pi_1^\std}   \otimes\ooo_{r,s,v_1}^{\pi_1^\std})\boxtimes (\tildooo_{v_2}^{\pi_2} \otimes\ooo_{r,s,v_2}^{\pi_2}) \\
	={} & (-1)^{(N(\pi^\std_1)+N(\pi_2)+c) + a + b + (\kk_1(v_1)+\kk_2(v_1))(\kk_1(v_2)+\kk_2(v_2)+1)} o_N\otimes (\tildooo_{v_1}^{\pi_1^\std}   \otimes\ooo_{r,s,v_1}^{\pi_1^\std})\boxtimes (\tildooo_{v_2}^{\pi_2} \otimes\ooo_{r,s,v_2}^{\pi_2}) \\
	={} & (-1)^{a + b+ c + (\kk_1(v_1)+\kk_2(v_1))(\kk_1(v_2)+\kk_2(v_2)+1) }o_N\otimes ((-1)^{N(\pi^\std_1)}\tildooo_{v_1}^{\pi_1^\std}   \otimes\ooo_{r,s,v_1}^{\pi_1^\std})\boxtimes ((-1)^{N(\pi_2)}\tildooo_{v_2}^{\pi_2} \otimes\ooo_{r,s,v_2}^{\pi_2}) \\
	= {} & (-1)^{K_2(v_2)} o_N\otimes o_{v_1}^{\pi_1^\std}  \boxtimes o_{v_2}^{\pi_2}.
	%={} & o_N\otimes o_{h_1}\otimes ((-1)^{N(\pi^\std_1)}\tildooo_{v_1'}^{\pi_1^\std}   \otimes\ooo_{v_1}^{\pi_1^\std})\boxtimes ((-1)^{N(\pi_2)}\tildooo_{v_2}^{\pi_2} \otimes\ooo_{v_2}^{\pi_2})
%= {} & (-1)^{a+b+(\kk_1(v_1)+\kk_2(v_1))(\kk_1(v_2)+\kk_2(v_2))+c}o_N\otimes o_{h_1}\otimes \left((-1)^{N(\pi^\std_1)}\tildooo^{\pi^\std_1}\otimes\ooo^{\pi^\std_1}\right)\left(\tildooo^{\pi_2}\otimes\ooo^{\pi_2}\right) \\
%= {} & o_N\otimes o_{h_1}\otimes \left((-1)^{N(\pi^\std_1)}\tildooo^{\pi^\std_1}\otimes\ooo^{\pi^\std_1}\right)\left(\tildooo^{\pi_2}\otimes\ooo^{\pi_2}\right).
\end{align*}

We now consider the general choice of component $\CM$
of $\CM^W_{\Lambda}$. As we did in the proof of Item (1), consider a component $\CM$ of $\CM^W_\Lambda$
where, for any $\Sigma\in\CM$, if we follow the cyclic order of the boundary points on its $v_1$-component from the node $x_{h_1}$ to the root, then we pass through $m_1$ boundary points of twist $(r-2,0)$ and $m_2$ of twist $(0,s-2)$.
We may now use Corollary \ref{cor:different_starting_pt} twice in
order to understand
the analogous version of Equation~\eqref{Case2a change in a}.
First, considering the $r$-spin bundle, note that the roles of the two vertices are reversed since the $h_1$ in our context has non-trivial twist, so we are in
the case of Corollary \ref{cor:different_starting_pt}(2) with $m= m_1 + |B_1|$. Here,  $B_1$ defined in the Corollary in the current context is the set of boundary marked points (excluding half-nodes) on $v_2$ with non-trivial twist $\tw_1$.
As $|B_1| = K_1(v_2) +1$, the sign changes by $(m_1 + K_1(v_2)+1)K_1(v_2)$. For the $s$-spin bundle, the roles of the two vertices are the same in both contexts, so we can straightforwardly apply Corollary \ref{cor:different_starting_pt}(1) with $m= m_2$ and $|B_2| = K_2(v_2)$. So we have that the change in $a$ from that stated in ~\eqref{Case2a change in a} is given by
 $$
 (m_1 + K_1(v_2)+1)K_1(v_2) + m_2K_2(v_2).
 $$
 Next, we provide the analogous analysis for $b$ in~\eqref{Case2b original b}. Here, we apply Equation~\eqref{orientation shift order with half-edge included} with $m = m_1+m_2$ and $|B_2| = K_1(v_2) + K_2(v_2) + 1$ to obtain that $b$ changes by
 $$
 (m_1+m_2)( K_1(v_2) + K_2(v_2) ).
 $$

 Lastly, we compute $c$. We can do this by re-computing the discrepancy between $N(\pi^\std)$ and $N(\pi^\std_1)+N(\pi_2)$. There are four different cases where one or the other counts a pair but the other quantity does not:

 \begin{enumerate}
 \item When an $s$-tail of $v_1$ is after the root but before the half-edge, then we need to count the pairs of that $s$-tail and any $r$-tail in $v_2$. This is counted in $N(\pi^\std)$ but not in $N(\pi^\std_1)+N(\pi_2)$. There are $(K_2(v_1) - m_2)( K_1(v_2) +1)$ such pairs.
 \item When an $r$-tail of $v_1$ is after the half-edge, then we need to count the pairs with that $r$-tail and any $s$-tail in $v_2$. Again, this is counted in  $N(\pi^\std)$ but not in $N(\pi^\std_1)+N(\pi_2)$. There are $m_1K_2(v_2)$ such pairs.
 \item When we take the half-edge $h_1$ and take any $s$-tail in $v_1$ that is before the half-edge and after the root. This is counted in  $N(\pi^\std_1)+N(\pi_2)$ but not in $N(\pi^\std)$. There are $( K_2(v_1) - m_2)$ such pairs.
 \item When we take the half-edge $h_2$ and any $r$-tail in $v_2$, we get a contribution as $\pi_2$ starts at the half-edge. This is counted in  $N(\pi^\std_1)+N(\pi_2)$ but not in $N(\pi^\std)$. There are $(K_1(v_2) +1)$ such pairs.
 \end{enumerate}

 This computes that
 $$
c= (K_2(v_1) - m_2)( K_1(v_2) +1) + m_1K_2(v_2) - ( K_2(v_1) - m_2) -(K_1(v_2) +1),
 $$
 so $c$ changes by
 $$
 -m_2( K_1(v_2) +1) + m_1K_2(v_2) + m_2 = -m_2K_1(v_2) + m_1 K_2(v_2).
 $$
 Combining all these changes, we therefore get that the changes of $a$, $b$ and $c$ cancel out modulo $2$, so the sign does not change.

{\bf Proof of (2b).} Suppose now that $\Gamma$ is as described in (2b). Denote by $o'_{0,k_1+1,k_2+1,0,\{(a_i,b_i)\}_{i\in I}}$ and $o'_{0,k_1,k_2,1,\{(a_i,b_i)\}_{i\in I}}$ the relative orientations of the Witten bundles with the roles of $r$ and $s$ interchanged. 
By Proposition \ref{obs:orders_well_def_and_r_s_change}, we have that $o'_{0,k_1,k_2,1,\{(a_i,b_i)\}_{i\in I}}= o_{0,k_1,k_2,1,\{(a_i,b_i)\}_{i\in I}}$, hence the orientation $o'_{v_1}=o_{v_1}$. Secondly, by the same proposition, we have that
\[
o'_{0,k_1+1,k_2+1,0,\{(a_i,b_i)\}_{i\in I}} = (-1)^{k_1+k_2+1} o_{0,k_1+1,k_2+1,0,\{(a_i,b_i)\}_{i\in I}},
\]
hence $o_{v_2}=(-1)^{\kk_1(v_2)+\kk_2(v_2)+1}o'_{v_2}$.
By symmetry we would get, using the previous case, that
\[\iota_{\Gamma}^*o'_{0,k_1+1,k_2+1,0,\{(a_i,b_i)\}_{i\in I}}= (-1)^{\kk_1(v_2)}o'_N \otimes (o'_{v_1} \boxtimes o'_{v_2}).\]
Putting together, we obtain \eqref{eq:ind_WC_2}.
\end{proof}

\section{Boundary conditions and intersection numbers}\label{sec: boundary conditions and intersection numbers}

In this section, we will define open intersection numbers associated to $W$-spin disks in the case of $W=x^r+y^s$. They will be defined as counts of zeros of
multisections of the Witten bundle on an open subset of our moduli spaces $\oCM^W_{\Gamma}$ satisfying certain boundary conditions, which we shall now explain.

\subsection{Boundary strata and intervals}

\begin{nn}\label{defn:PerturbedModuli}
For an open graded graph $\Gamma$, we define $\oPM_\Gamma$ to be the orbifold
\[
\oPM_\Gamma:=\oCM^W_\Gamma\setminus\left(\bigcup_{\substack{\Lambda\in\partial\Gamma\text{ such that }\\\Pos(\Lambda) \neq \varnothing}}\oCM^W_\Lambda\right),\]
where $\Pos(\Lambda)$, as defined in Definition~\ref{def: positive with respect to grading}, is the set of half-edges $h$ of the graph $\Lambda$ such that either $h$ or $\sigma_1(h)$ is positive for the $i^{th}$ grading for some $i$.
\end{nn}
Recall the subset of graphs $\partial^{+}\Gamma,~\partial^0\Gamma$ and $\partial^\xch\Gamma$, and the relation $\simx$, from Definitions~\ref{def:graphBestiary} and~\ref{partial 0 graph bestiary}.

\begin{definition}\label{def:positive_edge,node,PM}
Set
\begin{equation*}\begin{aligned}
\partial^+\oCM^W_\Gamma&=\bigcup_{\Lambda\in\partial^+\Gamma}\CM^W_\Lambda
\subseteq \oCM^W_{\Gamma}, \\
\partial^0\oCM^W_\Gamma&=\bigcup_{\Lambda\in\partial^0\Gamma}\CM^W_\Lambda
\subseteq \oCM^W_{\Gamma},\\
\partial^\XCH\oCM^W_\Gamma&=\bigcup_{\Lambda\in\partial^\xch\Gamma}\CM^W_\Lambda
\subseteq \oCM^W_{\Gamma}.
\end{aligned}\end{equation*}
\end{definition}
Note that $\oPM_\Gamma =\oCM^W_\Gamma\setminus \partial^+\oCM^W_\Gamma$ and that $\partial \oPM_\Gamma = \partial^0\oCM^W_\Gamma\cup\partial^\XCH\oCM^W_\Gamma$.
Indeed, for $\CM^W_{\Lambda}$ to lie in $\partial\oCM^W_{\Gamma}$, it must
either have a contracted boundary edge or a boundary edge, so $\Lambda$ must be contained in the union $\partial^+\oCM^W_\Gamma \cup  \partial^0\oCM^W_\Gamma\cup\partial^\XCH\oCM^W_\Gamma$. We remark that if $\partial^+\oCM^W_\Gamma$ is nonempty, then $\oPM_\Gamma$ is not compact. 

\begin{definition}
\label{def:PM Gamma}
Define
\[
\PM_{\Gamma}:=\oPM_{\Gamma}\setminus\partial\oPM_{\Gamma}.
\]
\end{definition}

Note that even if $\Gamma$ is smooth, there may be points of
$\PM_{\Gamma}$ corresponding to singular $W$-spin curves with
internal nodes.

\begin{ex}\label{ex:emptyboundary}
For the cases where $\Gamma=\Gamma^W_{0,m_1,m_2,1,\{(m_1,m_2)\}}$, with $m_1 < r-1,~m_2 < s-1$, and $\Gamma^W_{0,r,0,1,\emptyset}, \Gamma^W_{0,0,s,1,\emptyset}$, we now show that $\partial^0\Gamma=\emptyset$.

Consider the case where $\Gamma = \Gamma^W_{0,r,0,1,\emptyset}$.  Suppose that $\Lambda \in \partial^0\Gamma$. Note that $\Lambda$ cannot have any closed
vertices since there are no internal marked points. Thus
if $e$ is an edge in $\Lambda$ then $\smooth_e \Lambda \in \partial^0\Gamma \cup\{\Gamma\}$. So, it suffices to show that there are no $\Lambda \in \partial^0\Gamma$ such that $\Lambda$ has exactly two vertices. If such a $\Lambda$ exists, then note that one of the two vertices must have the root, which we call $v_1$. Using \eqref{eq:open_rank1} and \eqref{eq:open_rank2} in the same manner we did in Proposition~\ref{prop:balanced}(a), we can see that \[k_1(v_1) + k_{12}(v_1) - 1 \equiv r(I(v_1)) = 0 \pmod r.\] That is, there must be either $1$ or $r+1$ boundary half-edges on the vertex $v_1$ so that $(\tw_1, \alt_1) = (r-2, 1)$. Thus, the vertex $v_1$ contains at most $r$ $r$-points and one boundary half-edge, hence there are three cases: (i) $v_1$ has $r$ $r$-points and the half-edge does not have the $r$-twisting, (ii) $v_1$ has $r-1$ $r$-points and the half-edge does have the $r$-twisting, or (iii) $v_1$ has no $r$-points and the half-edge does not have the $r$-twisting. However, none of these cases will have two stable vertices, hence $\partial^0\Gamma = \emptyset$. Similarly, we have the analogous result for $\Gamma = \Gamma^W_{0,0,s,1,\emptyset}$.

Now consider the case $\Gamma = \Gamma^W_{0,m_1,m_2,1,\{(m_1,m_2)\}}$. As before, suppose that there exists a $\Lambda \in \partial^0\Gamma$ with exactly two vertices. Let $v_1$ be the vertex containing the root. Note that, as before using \eqref{eq:open_rank1} and \eqref{eq:open_rank2}, we have that 
\begin{align*}k_1(v_i) + k_{12}(v_i) - 1 &\equiv r(I(v_i)) \pmod r, \\ k_2(v_i) + k_{12}(v_i) - 1 &\equiv s(I(v_i)) \pmod s.
\end{align*} If $I(v_1) = \emptyset$ then $k_1(v_1) + k_{12}(v_1)  \equiv 1 \pmod r$ and $k_2(v_1) + k_{12}(v_1)  \equiv 1 \pmod s$. Since $m_1 < r-1$ and $m_2 < s-1$, we must then have that $k_1(v_1) = k_2(v_1) = 0$. However, this implies that the vertex $v_1$ is not stable and hence $\Lambda$ is not a boundary graph. Thus $v_1$ must have the internal marking. In this case, however, we then have that $k_1(v_1) + k_{12}(v_1) - 1 \equiv m_1 \pmod r$ and $k_2(v_1) + k_{12}(v_1) - 1 \equiv m_2 \pmod s$. We then have exactly one case that has two stable vertices, which is that of an exchangeable graph.
\end{ex}

\subsubsection{The base}
\label{subsec:base}
We now introduce one of the key concepts of this paper, the base of
a graph. This allows for the inductive structure
of boundary conditions for the canonical multisections to be introduced
in \S\ref{subsec:canonical multisections}.

\begin{definition}\label{def:base}
Let $\Gamma$ be a graded graph. We take $E^0(\Gamma)$ to be the set of boundary edges such that for each $i$, one of its half-edges has
$\tw_i=\alt_i=0$. Recall from Definition~\ref{def:for non-alt} that $\text{for}_{\text{non-alt}}$ is the operation that deletes all tails $h$ where $\tw_i = \alt_i = 0$ for all $i$.
We define the {\it base} of $\Gamma$ to be
\[\CB \Gamma = \text{for}_{\text{non-alt}} (\detach_{E^0(\Gamma)}(\Gamma)).\]

We call the moduli $\oCM^W_{\CB\Gamma}$ the \emph{base moduli} of $\oCM^W_{\Gamma}$.
We write
\[F_{\Gamma}:= (\text{For}_{\text{non-alt}}\circ\Detach_{E^0(\Gamma)})
:\CM^W_{\Gamma}\to\CM^W_{\CB\Gamma}.\]
\end{definition}
On the level of surfaces, the morphism $F_{\Gamma}$ normalizes all boundary half-edges in $E^0(\Gamma)$ and forgets all boundary marked points which have twist
zero and are non-alternating with respect to both spin structures, i.e.,  $\tw=\alt=({0},0)$.
Note that the forgetful map may create unstable connected components. This happens, for example, for exchangeable graphs. We consider the moduli of an unstable connected component to be a point.
\begin{nn}\label{nn:fGamma}
We can extend $F_{\Gamma}$ continuously to a map $\oCM^W_{\Gamma}\to\oCM^W_{\CB\Gamma}$. We use the same notation for the extended map.
If we take $\Lambda\in \partial^!\Gamma$, it  maps  $\CM^W_{\Lambda}$ under $F_{\Gamma}$ to $\CM^W_{\Lambda'}$ for some $\Lambda' \in \d^!\CB\Gamma$. We can define a function \[f_{\Gamma}:\partial^!\Gamma\to \d^!\CB\Gamma\]
so that $f_{\Gamma}(\Lambda) = \Lambda'$.
Explicitly, $\Lambda'=\text{for}_{\text{non-alt}} ( \detach_{E^0(\Gamma)}(\Lambda))$, where we use the identification of the edges of $\Gamma$ as edges of $\Lambda$.
\end{nn}

\begin{rmk}
To motivate the definition of the base moduli,
assume we are given a graded graph $\Gamma$ indexing a boundary
stratum of some moduli space $\oCM^W_{\Gamma'}$ for $\Gamma'$ a smooth
graded graph. If $\widehat\Gamma=\detach_E \Gamma$ for some subset
of edges $E$ of $\Gamma$, then the detaching map $\Detach_E:\oCM^W_{\Gamma}
\rightarrow \oCM^W_{\widehat{\Gamma}}$ allows us to relate the former
moduli space to the latter, which is a product of moduli spaces
associated to the connected components of $\Gamma$. In general,
 these connected components will not yield moduli spaces
of the sort we consider here: we wish to consider only those moduli
spaces associated to graphs whose boundary tails have twists
$(r-2,0)$, $(0,s-2)$ or $(r-2,s-2)$. 

However,
if $\Gamma\in \partial^0\Gamma'$, then all boundary half-edges
of $\Gamma$ necessarily have twists being either one of the above
three possibilities or twist $(0,0)$. Thus if we detach $\Gamma$
along all such boundary half-edges, and forget the newly acquired
boundary tails with twist $(0,0)$, obtaining $\CB\Gamma$,
the connected components of $\CB\Gamma$ are precisely graded graphs with
boundary tails of the type we are considering in this paper.
Further, as we will see in Observation \ref{obs:Witten_pulled_back_perm}, the
Witten and descendent bundles on $\oPM_{\Gamma}$
pull back from those on $\oPM_{\CB\Gamma}$, thus enabling the
construction of inductive boundary conditions. For the remainder
of the boundary conditions near the strata we delete, we rely on positivity conditions.
\end{rmk}

The following observations are straightforward but useful for later.  First, analogously to \cite[Observations 3.14 and 3.28]{PST14}, we have the following observation, which serves as a compatibility relation for boundary conditions:

\begin{obs}\label{obs:for_comp}
Let $\Gamma$ be a graded graph and take $\Gamma'\in\partial\Gamma$.  Then $F_{\Gamma}$ takes $\oCM^W_{\Gamma'}$ to $\oCM^W_{f_{\Gamma}(\Gamma')},$ and $\CB f_{\Gamma}(\Gamma')=\CB\Gamma'$.
Moreover, we have
\[F_{f_{\Gamma}(\Gamma')}\circ F_{\Gamma}|_{\oCM^W_{\Gamma'}}=F_{\Gamma'}.\]
This follows directly from the definitions.
\end{obs}

Recall that when $v$ is an exchangeable vertex of an exchangeable graph $\Gamma$, we have defined the graph $\xch_v(\Gamma)$ to be the graph obtained by reversing the cyclic orders at $v$ and a corresponding map $\XCH_v$ given in~\eqref{exchange moduli} at the level of moduli. When $\Gamma$ has a single exchangeable vertex $v$, we sometimes write $\XCH$ instead of $\XCH_v$.

Another straightforward observation is the following:
\begin{obs}\label{obs:exchangeable_and_base}
Graphs in the same $\simx$-equivalence class have the same base. Moreover, if $\Gamma$ is exchangeable and $v\in V^O(\Gamma)$ is an exchangeable vertex, then
\[F_{\xch_v(\Gamma)}\circ \XCH_v=F_\Gamma.\]
\end{obs}

\begin{definition}\label{def:abs_vertex}
Let $\mathcal{V}(\Gamma)$ be the collection of dual graphs consisting of one vertex that may appear as a connected component of $ \textup{for}_{\textup{non-alt}}(\detach_{E(\Lambda)}(\Lambda))$ for some $\Lambda\in\partial^!\Gamma\setminus\partial^+\Gamma$.
We call the elements of $\mathcal{V}(\Gamma)$ \emph{abstract vertices}.
Let $i\in I(\Gamma)$ be a labeling for an internal tail. Let $\mathcal{V}^i(\Gamma)\subset\mathcal{V}(\Gamma)$  be the collection of abstract vertices which have an internal tail labeled $i$.
For a graph $\Gamma$ with $i\in I(\Gamma)$,
denote by $v_i^*(\Gamma)$ the connected component of the graph \[ \text{for}_{\text{non-alt}}(\detach_{E(\Gamma)}(\Gamma)) \] which contains $i$ as an internal marking. As we have detached at all edges, this connected component only has a single vertex, hence $v_i^*(\Gamma) \in \mathcal{V}^i(\Gamma)$.

If $v$ is a connected component of the graph
$ \text{for}_{\text{non-alt}}(\detach_{E(\Gamma)}(\Gamma))$,
denote by $\Phi_{\Gamma,v}:\oCM^W_{\Gamma}\to\oCM^W_{v}$ the natural map,
which is the composition of
\[\text{For}_{\text{non-alt}}\circ \text{\Detach}_{E(\Gamma)}\]
with the projection to $\oCM^W_{v}$.
Write $\Phi_{\Gamma,i}$ for $\Phi_{\Gamma,v_i^*(\Gamma)}$.
\end{definition}

The key observation for the inductive structure of the Witten and
descendent bundles is then:

\begin{obs}\label{obs:Witten_pulled_back_perm}
Let $\Gamma$ be a pre-graded graph. There is a canonical isomorphism between the Witten bundle $\cW \to \oPM_{\Gamma}$ and the pullback bundle $F_{\Gamma}^*\cW$.  To construct it, we first note that over the locus of $\oPM_{\Gamma}$, all boundary half-edges are Neveu-Schwarz, hence we can apply Proposition \ref{pr:decomposition}(1).
Further, in forgetting those resulting boundary tails
with $\tw=\alt=(0,0)$, we can apply \cite[Equation 4.3]{BCT:I}.

Similarly, by Observation 3.32 of \cite{PST14}, if the component which contains the internal tails labeled $i$ in $\CB\Gamma$ is stable then $\CL_i\to\oCM^W_\Gamma$ is canonically $F_\Gamma^*\left(\CL_i\to \oCM^W_{v^*_i(\Gamma)}\right)$.
\end{obs}

In the rest of the paper, we will construct multisections of $E_\Gamma(\vecd)$
which, using the above observation,  are pulled back from the bundle
$E_{\CB\Gamma}(\vecd)$. For this reason,
it is important to understand the relationship between the dimensions of $\oPM_{\CB\Gamma}$
and $\oPM_{\Gamma}$ on the one hand and $E_{\Gamma}(\vecd)$ on the
other. We now calculate.

\begin{obs}
\label{obs:base dim yoga}
Fix a smooth graded $W$-spin graph $\Gamma$, and consider
$\Lambda \in \partial^0\Gamma\cup\partial^\xch\Gamma$ where $\Lambda$
is a relevant graph with no internal edges. Let
$\vecd=(d_i)_{i\in I(\Gamma)}$ be a descendent vector, and
$E_{\Gamma}(\vecd)$ the descendent-Witten bundle on
$\oPM_{\Gamma}$ with respect to $\vecd$. We define the following integers:
\begin{enumerate}
\item $\alpha=\rank E_{\Gamma}(\vecd)-\dim \oPM_{\Gamma}$.
\item $\nu$ is the number of (open) vertices of $\Lambda$.
\item  $\beta$ is the number of half-edges with $\tw=\alt=(0,0)$ forgotten
in passing from $\detach_{E^0(\Lambda)}\Lambda$ to $\CB\Lambda$.
\item
$\sigma$ is the number of partially stable connected components of
$\CB\Lambda$.
\end{enumerate}
Note
\[
\sigma \le \beta < \nu.
\]
Since we are assuming $\Lambda$ is relevant,
necessarily any such partially stable connected component
has two boundary tails and no interior tails. Then
\begin{equation}
\label{eq:strata dim}
\dim \oPM_{\Lambda}=\dim\oPM_{\Gamma}-(\nu-1)
\end{equation}
and
\begin{equation}
\label{eq:base dim}
\dim\oPM_{\CB\Lambda}=\dim\oPM_{\Lambda}-\beta+\sigma,
\end{equation}
bearing in mind that the moduli space associated to a partially
stable component is a point, hence one dimension higher than expected.

Denote the connected components of $\CB\Lambda$ as $\Xi_1,\ldots,\Xi_\nu$,
with the last $\sigma$ being the partially stable components.
By Observation \ref{obs:Witten_pulled_back_perm},
\[
\rank E_{\Gamma}(\vecd)= \sum_{i=1}^{\nu}\rank E_{\Xi_i}(\vecd).
\]
Combining the previous three equations and using the fact that
$\rank E_{\Xi_i}(\vecd)=0$ for any partially stable component $\Xi_i$,
we obtain
\begin{equation}
\label{eq:dim rank comparison}
\sum_{i=1}^{\nu-\sigma}\dim\oPM_{\Xi_i} = \sum_{i=1}^{\nu-\sigma} \rank E_{\Xi_i}(\vecd)
-\alpha-\beta-(\nu-\sigma)+1.
\end{equation}
\end{obs}

\subsubsection{Positivity constraints}\label{subsec: positivity}

Recall that the moduli space $\oPM_{\Gamma}$ in general is not compact; indeed, we have that $\partial^+\M_{\Gamma}^W = \M_{\Gamma}^W \setminus \oPM_{\Gamma}$. In order to apply Definition~\ref{def: relative Euler class} to define open $W$-spin invariants, we will require non-vanishing of any global multisection of the Witten bundle in a region ``close'' to $\partial^+\M_{\Gamma}^W$ in $\oPM_{\Gamma}$. To this end, we proceed with the following definitions that aim to provide such multisections which we will call strongly positive.

\begin{definition}\label{def:positive_neighborhoods1}
Let $\Gamma$ be a graded $W$-spin graph and $\Lambda \in \partial^+\Gamma$. A \emph{$\Lambda$-set with respect to $\Gamma$} is an open set $U\subseteq \oCM^W_\Gamma$ whose closure intersects the strata $\CM^W_\Xi$ precisely for those graphs $\Xi\in \partial^!\Gamma$ with $\Lambda\in \partial^!\Xi$.

Now further suppose $u\in \CM^W_\Lambda\subseteq\partial^+\oCM^W_\Gamma$. A \emph{$\Lambda$-neighborhood with respect to $\Gamma$} of $u\in\oCM^W_\Gamma$ is a neighborhood $U\subseteq\oCM^W_\Gamma$ of $u$ which is a $\Lambda$-set with respect to $\Gamma$. We remark that if $\Gamma$ is smooth, we will omit the addition ``with respect to $\Gamma$" as, for any $\Lambda$, there is a unique smooth graph $\Gamma$ for which $\Lambda\in\partial^!\Gamma$.
\end{definition}

\begin{ex} Consider the moduli space given in Example~\ref{ex: hexagon} with one internal marking and three boundary markings. Let $\Gamma$ be the smooth graded graph and let $\Lambda$ be a codimension one boundary strata. A $\Lambda$-set with respect to $\Gamma$ is any open set whose closure does not intersect any other boundary strata. In Figure~\ref{Lambda Set Example}(A), we have a $\Lambda$-neighbourhood with respect to $\Gamma$ of the point $u \in \CM^W_{\Lambda}$, while in Figure~\ref{Lambda Set Example}(B), the neighbourhood's closure would intersect the codimension two strata, so it is not a $\Lambda$-neighbourhood with respect to $\Gamma$ of $u$.

\end{ex}

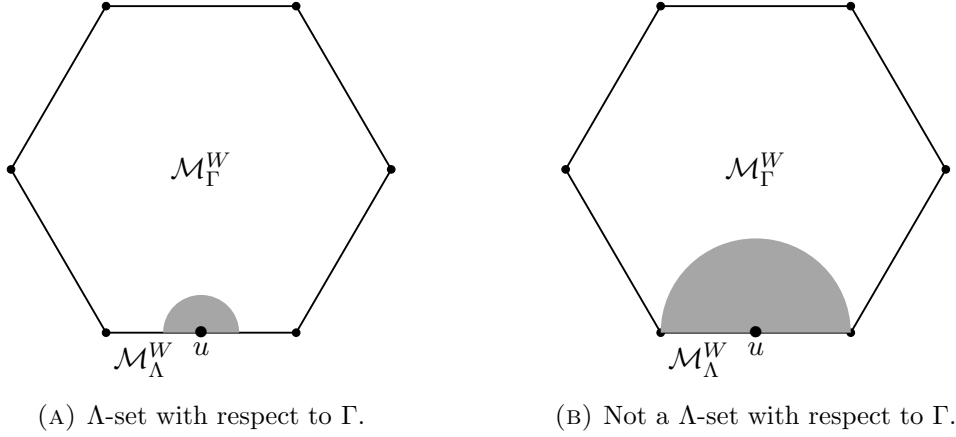
\begin{figure}
\begin{subfigure}{.45\textwidth}
  \centering
  \begin{tikzpicture}[scale=0.5]

 \newdimen\Rad
   \Rad=5cm
   \draw[line width=0.25mm] (0:\Rad) \foreach \x in {60,120,...,360} {  -- (\x:\Rad) };
   \draw (5,0)[black, fill = black] circle (.1cm);
   \draw (-5,0)[black, fill = black] circle (.1cm);
   \draw (2.5, 4.33)[black, fill = black] circle (.1cm);
      \draw (-2.5, 4.33)[black, fill = black] circle (.1cm);
         \draw (2.5, -4.33)[black, fill = black] circle (.1cm);
            \draw (-2.5, -4.33)[black, fill = black] circle (.1cm);

            \fill[gray!70]  (-1,-4.33) arc (180:0:1);
\node at (0,-4.33) {$\bullet$};
 \node[below] at (0,-4.33) {$u$};

 \node[below] at (-1.5,-4.33) {$\CM^W_{\Lambda}$};
  \node at (0,0) {$\CM^W_{\Gamma}$};
  \end{tikzpicture}
  \caption{$\Lambda$-set with respect to $\Gamma$.}
 \end{subfigure}
 \begin{subfigure}{.45\textwidth}
  \centering
  \begin{tikzpicture}[scale=0.5]

 \newdimen\Rad
   \Rad=5cm
   \draw[line width=0.25mm] (0:\Rad) \foreach \x in {60,120,...,360} {  -- (\x:\Rad) };
   \draw (5,0)[black, fill = black] circle (.1cm);
   \draw (-5,0)[black, fill = black] circle (.1cm);
   \draw (2.5, 4.33)[black, fill = black] circle (.1cm);
      \draw (-2.5, 4.33)[black, fill = black] circle (.1cm);
         \draw (2.5, -4.33)[black, fill = black] circle (.1cm);
            \draw (-2.5, -4.33)[black, fill = black] circle (.1cm);

            \fill[gray!70]  (-2.5,-4.33) arc (180:0:2.5);

            \node at (0,-4.33) {$\bullet$};
 \node[below] at (0,-4.33) {$u$};
  \node[below] at (-1.5,-4.33) {$\CM^W_{\Lambda}$};
  \node at (0,0) {$\CM^W_{\Gamma}$};
  \end{tikzpicture}
  \caption{Not a $\Lambda$-set with respect to $\Gamma$.}
 \end{subfigure}

 \caption{Two neighbourhoods of $u \in \CM^W_\Lambda$ in $\CM^W_\Gamma$.}
 \label{Lambda Set Example}
\end{figure}

\begin{definition}\label{def:interval}
Let $\Sigma$ be a (possibly nodal) graded marked disk. Write the oriented boundary $\partial \Sigma$ as a quotient space $S/{\sim}$, where $S$ is the oriented manifold $S^1$ and $\sim$ is an equivalence relation whose equivalence classes are all of size $1$ except for a finite number of classes of size $2$ and take $q:S\to \partial \Sigma$ to be the quotient map. An \emph{interval} $\II$ is the image of a connected open set in S under $q$. Note that the inverse image $q^{-1}(\II)$ is the union of an open set with a finite number of isolated points.

Suppose $\Sigma\in\CM^W_\Lambda$ and consider the boundary half-node $n_h=n_h(\Sigma)$, which corresponds to the half-edge $h\in H^B(\Lambda)$. We say that $n_h$ \emph{belongs} to the interval $\II$ if

\begin{itemize}
\item The corresponding node $n_{\mathrm{edge}(h)}$ is contained in $\II$.
\item Consider the boundary $\partial\hat{\Sigma}$ of the normalization $\hat\Sigma = \NNN^{-1}(\Sigma)$. Then we can write $\NNN^{-1}(\II)=\II_1\cup \II_2$, where $\II_1$ is a half-open interval containing boundary point $n_h$ and $\II_2$ is a half-open interval containing boundary point $n_{\sigma_1h}$. According to the canonical orientation of $\hat\Sigma$, we then have that the half-node $n_h$ is the starting point of $\II_1$ when viewed as an interval on the (oriented) real line and, consequently, the half-node $n_{\sigma_1h}$ is the endpoint of $\II_2$.
\end{itemize}
\end{definition}

We refer the reader to Figure~\ref{fig:interval} for a pictorial description of this definition.

\begin{figure}
\begin{subfigure}{.45\textwidth}
  \centering
  \begin{tikzpicture}[scale=0.6]

  \draw (0,0) circle (2cm);
  \draw (4,0) circle (2cm);
\draw [line width = 2pt,domain=0:90] plot ({2*cos(\x)}, {2*sin(\x)});
\draw [line width = 2pt,domain=90:180] plot ({2*cos(\x)+4}, {2*sin(\x)});

\node at (1.35, 0) {$n_h$};
\node at (2.9, 0) {$n_{\sigma_1 h}$};
\node at (2,0) {$\bullet$};
\node at (2,1.35) {$\II$};

  \end{tikzpicture}
  \caption{Interval $\II$ on a nodal disk $\Sigma$}
 \end{subfigure}
 \begin{subfigure}{.45\textwidth}
  \centering
  \begin{tikzpicture}[scale=0.6]

  \draw (0,0) circle (2cm);
  \draw (5,0) circle (2cm);
\draw [line width = 2pt,domain=0:90] plot ({2*cos(\x)}, {2*sin(\x)});
\draw [line width = 2pt,domain=90:180] plot ({2*cos(\x)+5}, {2*sin(\x)});

\node at (1.35, 0) {$n_h$};
\node at (3.9, 0) {$n_{\sigma_1 h}$};
\node at (2,0) {$\bullet$};
\node at (3,0) {$\bullet$};
\node at (3.3, 1.8) {$\II_2$};
\node at (1.7, 1.8) {$\II_1$};
  \end{tikzpicture}
  \caption{Preimage $\NNN^{-1}(\II) = \II_1 \cup \II_2$ under normalization}
 \end{subfigure}

 \begin{subfigure}{.45\textwidth}
  \centering
  \begin{tikzpicture}[scale=0.6]
\draw [domain=90:270] plot ({2*cos(\x)}, {2*sin(\x)});
\draw [domain=270:450] plot ({2*cos(\x)+4}, {2*sin(\x)});
\draw [domain= 0:4] plot ({\x}, {-.75*cos(90*\x)-1.25});
\draw [line width = 2pt,domain= 0:4] plot ({\x}, {.75*cos(90*\x)+1.25});
\node at (2,1.2) {$\II$};
  \end{tikzpicture}
  \caption{Interval $\II$ considered after smoothing}
 \end{subfigure}
\begin{subfigure}{.45\textwidth}
  \centering
  \begin{tikzpicture}[scale=0.6]

  \draw (0,0) circle (2cm);
  \draw (4,0) circle (2cm);
\draw [line width = 2pt,domain=0:90] plot ({2*cos(\x)}, {2*sin(\x)});
\draw [line width = 2pt,domain=90:180] plot ({2*cos(\x)+4}, {2*sin(\x)});

\draw [line width = 2pt,domain=270:360] plot ({2*cos(\x)}, {2*sin(\x)});
\draw [line width = 2pt,domain=180:270] plot ({2*cos(\x)+4}, {2*sin(\x)});

\node at (1.35, 0) {$n_h$};
\node at (2.9, 0) {$n_{\sigma_1 h}$};
\node at (2,0) {$\bullet$};
\node at (2,1.35) {$\II_{h}$};
\node at (2.1,-1.45) {$\II_{h'}$};

  \end{tikzpicture}
  \caption{Intervals $\II_h$ and $\II_{h'}$ as in Definition~\ref{family of intervals}.2(b)}
 \end{subfigure}

 \caption{Interval on boundary $\partial \Sigma$ on nodal disk $\Sigma$ that belongs to half-node $n_h$. Here, the boundary is oriented anti-clockwise.}
\label{fig:interval}
\end{figure}
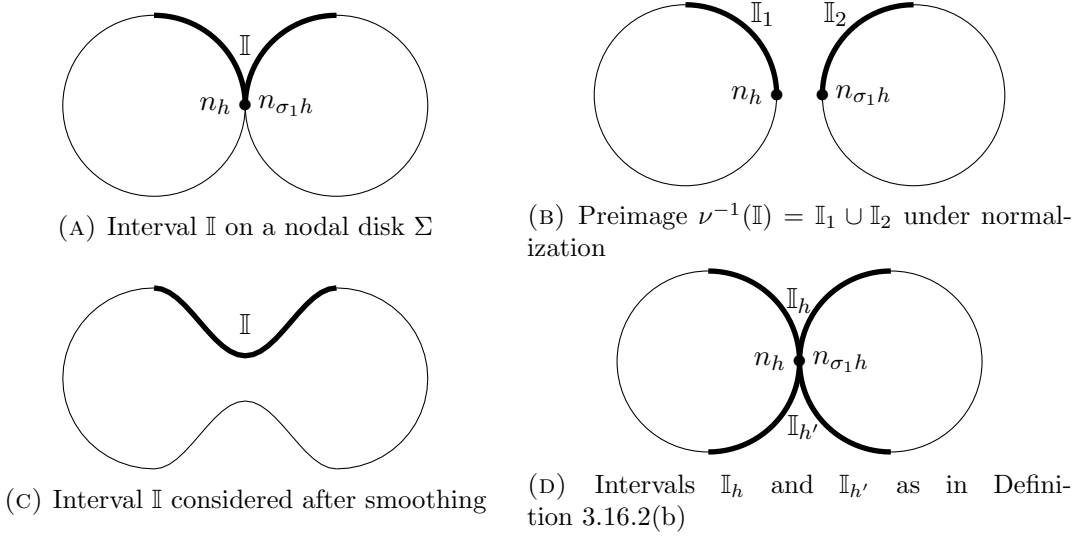

\begin{definition}\label{family of intervals}
Let $\Gamma$ be a graded $W$-spin graph and $\Lambda \in \partial^+\Gamma$. Suppose $U$ is contained in some $\Lambda$-set with respect to
$\Gamma$. A \emph{$\Lambda$-family of intervals} for $U$ is a family $\{\II_{h}(u)\}_{h\in\Pos(\Lambda),u\in U}$ such that the following hold:
\begin{enumerate}
\item
Each $\II_{h}(u)$ is an interval in $\partial\Sigma_u,$ where $\Sigma_u=\pi^{-1}(u)$ is the fiber of the universal curve $\pi:\CC\to U$ at $u$, or equivalently the graded surface which corresponds to $u$. The endpoints of each $\II_h(u)$ vary smoothly with respect to the smooth structure of the universal curve restricted to $U$.
\item Suppose $h$ and $h'$ are two half-edges of $\Lambda$.
\begin{enumerate}
\item If $h \notin  \{h',\sigma_1h'\}$, then $\II_h(u)\cap \II_{h'}(u)=\emptyset$ for any $u\in U$.
\item If $h'=\sigma_1h$, then (i) for any $u \in U\cap \CM^W_{\Xi}$ where $h\in\Pos(\Xi)$, we have $\II_h(u)\cap \II_{h'}(u) = n_{\textup{edge}(h)}$, the node in $\partial\Sigma$ corresponding to the half-edges $h$ and $\sigma_1h$ (so that their union forms an $X$-shape near that node, see Figure \ref{fig:interval}(D)); (ii) for any $u \in \CM^W_{\Xi}$ where $h\notin\Pos(\Xi)$, $\II_h(u)\cap \II_{h'}(u)=\emptyset$. Here, we use the locally defined injection $\iota:H(\Xi)\to H(\Lambda)$. 
\end{enumerate}
\item
$\II_h(u)$ contains no marked points, and
\item the  half-node $n_h$ is the only half-node which belongs to $\II_h(u)$, whenever $u\in\CM^W_\Xi\cap U$ for some partial smoothing $\Xi=\smooth_E(\Lambda)$ for some $E \subsetneq E(\Lambda)$, where $h\in\Pos(\Xi)$.
\end{enumerate}
When the point $u$ corresponds to a stable graded disk $\Sigma$, we sometimes write $\II_h(\Sigma)$ for $\II_h(u)$. 
\end{definition}

\begin{definition}\label{def:positive_section}\begin{enumerate}
\item Let $A$ be a subset of $\partial\Sigma$ not containing any alternating special points. An element $w\in({\cW_i})_\Sigma~(i=1,2)$ \emph{evaluates positively at $A$}, if for every $x\in A$ it holds that $ev_x(w)>0,$ i.e., $ev_x(w)\in (\cJ_i)^{\tilde\phi}_x$ agrees with the positive direction of the $i^{th}$-grading.

\item Let $\Gamma$ be any graded graph. Let $U$ be a $\Lambda$-set with respect to $\Gamma$ and $\mathbb{I}:=\{\mathbb{I}_h\}_{h}$ a $\Lambda$-family of intervals for $U$. A multisection $\ess$ of $\cW$ defined in a subset of $\oPM_\Gamma$ which includes $U\cap\oPM_\Gamma$ is \emph{$(U,\mathbb{I})$-positive} (with respect to $\Gamma$) if, for any $u\in U\cap \oPM_\Gamma$, $j\in \{1,2\}$ and $h\in H^+_j(\Lambda)$, the $\cW_j$-component of any branch $s_i$ of $\ess$ evaluates positively at each $\mathbb{I}_h(u)$.\footnote{Note that if a positive edge of
$\Lambda$ is also an edge of $\Gamma$, then $\oPM_{\Gamma}$ is empty.
Thus this positivity is required to hold only along intervals disjoint from
nodes, and hence makes sense.}

\item Consider a neighborhood $V\subseteq\oCM^W_{\Gamma}$ of $u\in\oCM^W_\Lambda$.
A multisection defined in $V\cap\oPM_\Gamma$ is \emph{positive near} $u$ (with respect to $\Gamma$) if there exists a $\Lambda$-neighborhood $U\subseteq V$ of $u$ and a $\Lambda$-family of intervals $\mathbb{I}$ for $U$ such that the multisection is $(U,\mathbb{I})$-positive.
\end{enumerate}
\end{definition}

Let $V$ be a neighbourhood of $\partial^+\oCM^W_{\Gamma}$ and set
\begin{equation}\label{eq:U_+}U_+=\left(V\cap\oPM_\Gamma\right)\cup \bigcup_{\Lambda\in \partial^!\Gamma~\text{is irrelevant}}{\oPM_\Lambda}
\end{equation}
(recalling the definition of an irrelevant graph from Definition \ref{def:special kind of graded graphs}). We remind the reader that if $\Lambda$ is positive, then $\oPM_{\Lambda}$
is in fact empty.

\begin{definition}\label{def: strongly positive} A multisection $\ess$ defined in a set which contains a $U_+$ as in~\eqref{eq:U_+}  is
 \emph{strongly positive (with respect to $\Gamma$)} if the following hold:
\begin{enumerate}
\item
Let $\Lambda\in\partial^!\Gamma$ satisfy condition \eqref{irrelevant:item 1} of the
definition of irrelevant, Definition \ref{def:special kind of graded graphs}.
Then  $\ess(\Sigma)$ evaluates positively with respect to the two gradings corresponding to the $r$-spin and $s$-spin structures
at every contracted boundary node of $\Sigma$ as described in Definition~\ref{def:closed grading}.
\item The multisection is positive near each point of $\partial^+\oCM^W_\Gamma$.

\item \label{strongly positive:item 3}
Let $\Lambda\in\partial^!\Gamma$ satisfy condition \eqref{irrelevant:item 3} of
Definition \ref{def:special kind of graded graphs} with respect to $i \in \{1,2\}$
and let $\Xi$ be the component of
$\detach_{E(\Lambda)}(\Lambda)$ with vertex $v$ as in that condition.
Let $\Sigma\in\CM^W_\Lambda,$ with normalization
$\NNN:\widehat{\Sigma}\to\Sigma$. Let $\widehat{\Sigma}^{\Xi}$ denote
the connected component of $\widehat{\Sigma}$ corresponding to $\Xi$.
Then:
\begin{enumerate}
\item
The $\cW_i$-component of each branch of $\NNN^*(\ess(\Sigma))$
evaluates positively on some non-empty subset of the boundary of
$\widehat{\Sigma}^{\Xi}$ (note that by ~\eqref{eq:open_rank2}, the rank of $\cW_i$ is nonzero).
\item If $\Xi$ becomes partially stable with no boundary half-edges after forgetting all boundary half-edges $t$ with $\alt(t)=\tw(t)=(0,0)$, then $\NNN^*(\ess(\Sigma))$ evaluates positively with respect to the
gradings of the two spin bundles at all of the boundary
of $\widehat{\Sigma}^{\Xi}$.
\end{enumerate}
\end{enumerate}
Again, we omit the extension ``with respect to $\Gamma$" whenever $\Gamma$ is smooth.
\end{definition}

\begin{rmk}\label{rmk: why strongly positive}
One possible alternative definition for positivity would be to require that, for all points $\Sigma$ of a boundary stratum $\CM_\Lambda$ for $\Lambda\in\partial^+\Gamma$, the evaluation of the section of the Witten bundle at any non-alternating boundary half-node with positive twist is positive. While in low codimension this idea works, Example 3.23 in \cite{BCT:II} shows that it does not hold in general. In fact, it would imply vanishing of the section at certain points in the moduli space. In \cite{BCT:II}, this issue is solved by working on $\PM_\Gamma$ rather than on $\oCM_\Gamma$. 

We note that other versions of positivity are also possible. For example, one may continue to work on $\oCM_\Gamma$ while imposing certain positivity constraints on the derivative of the sections of the Witten bundle. At any rate, both approaches will establish equivalent open analogues of the virtual class of the Witten bundle in our setting, while avoiding boundary strata with non-alternating boundary nodes of positive twists.
\end{rmk} 

\subsection{Canonical multisections}
\label{subsec:canonical multisections}   Intuitively, canonical multisections are strongly positive multisections
whose restriction to relevant (i.e., not irrelevant in
the sense of Definition \ref{def:special kind of graded graphs}) boundary strata are pulled back from the base moduli of these boundary strata. With the same intuition, we would like special canonical multisections to satisfy a more restrictive property. For any boundary stratum, the detaching of the dual graph associated to the stratum
is comprised of abstract vertices $v$ given in Definition~\ref{def:abs_vertex}.
We require that the restriction of the multisection to this boundary stratum
is pulled back from an assembling of multisections of Witten bundles on these abstract vertices. 
Below, we will make this precise.

\begin{definition}\label{def:canonical for Witten}
Let $\Gamma$ be a smooth graded $\RS$-graph with only open vertices. 
Let $U$ be an open set of $\oCM^W_{\Gamma}$ so that it contains $\partial\oPM_\Gamma\cup U_+$, where $U_+$ is a subset of $\oPM_{\Gamma}$ as defined in~\eqref{eq:U_+}. A smooth multisection $\ess$ of $\cW\to U$ is \emph{canonical} if the following both hold:
\begin{enumerate}
\item The multisection $\ess$ is strongly positive as in Definition~\ref{def:positive_section};
\item For any open $\Xi\in\mathcal{V}(\Gamma)$, there exists a multisection $\ess^\Xi\in C_m^\infty(\oPM_\Xi,\cW_\Xi)$ such that, for any $\Lambda\in
\partial^!\Gamma\setminus\partial^+\Gamma$ with only open
vertices,
\begin{equation}\label{eq:pb_Witten}
\ess|_{\oPM_\Lambda}=F_\Lambda^*(\boxplus_{\Xi\in\Conn(\CB\Lambda)}\ess^\Xi),
\end{equation}
using Notation~\ref{not: restriction notation}.
\end{enumerate}
In the case where the set $U$ is the entire moduli $\oPM_\Gamma$, we say that $\ess$ is a \emph{global canonical multisection}.
\end{definition}
Note that the requirement \eqref{eq:pb_Witten} makes sense by Observation
\ref{obs:Witten_pulled_back_perm} and is automatic if
$\Lambda=\Gamma$ and $\Gamma$ is connected.

\begin{definition}\label{def:rs graph}
Let $I\subseteq\Omega$. We define $\RS\text{-Graph}(I)$ to be the collection of
graded open $\RS$-graphs with all boundary markings being
$\emptyset$ and $I(\Gamma)\subseteq I$. When $I$ is understood, we will often omit $I$.
\end{definition}

\begin{definition}
\label{def:family of canonical}
A \emph{family of canonical multisections for $\cW$ bounded by $I$} is a collection of (global) canonical multisections $(\ess^\Gamma)_{\Gamma\in\RS\text{-Graph}(I)}$
 with $\ess^\Gamma\in C_m^\infty(\oPM_\Gamma,\cW)$, such that
for all $\Gamma\in \RS\text{-Graph}(I)$ and
$\Lambda\in \partial^0\Gamma \cup\partial^\xch\Gamma$, \eqref{eq:pb_Witten} holds with
$\ess=\ess^{\Gamma}$. If $I=\Omega$, then we will omit the phrase
``bounded by $I$."
\end{definition}

Note that since there are cases (when $v$ is closed and the anchor is Ramond) that the Witten bundle does not decompose as a direct sum, we must use coherent multisections and the assembling map $\Ass$ of \S\ref{subsec:coherent} (instead of usual multisections and direct sums).

\begin{definition}\label{def: special canonical}
A multisection $\ess$ of $\cW$ over $\oPM_{\Gamma}$
is \emph{special canonical} if $\ess$ is strongly positive and, for each abstract vertex $v\in \mathcal{V}(\Gamma)$ as defined in Definition~\ref{def:abs_vertex}, there exists data $\ess^v$ as follows:
\begin{enumerate}
\item If $v$ is a closed vertex, then $\ess^v$ is a coherent multisection in
$C_m^\infty(\TRAM_v,\cW)$;
\item If $v$ is an open vertex, then $\ess^v\in C_m^\infty(\oPM_v,\cW)$.
\end{enumerate}
This data must satisfy the decomposition formula
\[\ess|_{\oPM_\Lambda} = F_{\Lambda}^*\circ\Ass_{\CB\Lambda,E(\CB\Lambda)}((\ess^v)_{v\in \Conn(\detach(\CB\Lambda))}),\]
for any $\Lambda\in \partial^!\Gamma\setminus\partial^+\Gamma$.
\end{definition}

We can also define the analogous notions of canonical and special canonical for the cotangent line bundles $\CL_i$.

\begin{definition}\label{def:canonical for L_i}
A smooth multisection $\ess$ of $\CL_i\to\partial\oPM_{\Gamma}$,
where $\Gamma$ is smooth and graded, is called \emph{canonical} if
for any stable open $\Xi\in\mathcal{V}^i(\Gamma)$ there exists a multisection $\ess^\Xi\in C_m^\infty(\oPM_\Xi,\CL_i)$ such that for any $\Lambda\in \partial^0\Gamma\cup\partial^\xch\Gamma$ with $v_i^*(\Lambda)=\Xi$, we have
\begin{equation}\label{eq:pb}\ess|_{\oPM_{\Lambda}}=\Phi_{\Lambda,i}^*(\ess^{\Xi}),\end{equation}
where $\Phi_{\Lambda,i}$ is as defined in Definition \ref{def:abs_vertex}.
A global multisection for $\CL_i$ is canonical if its restriction to $\partial\oPM_{\Gamma}$ is canonical.

A multisection $\ess\in C_m^\infty(\oPM_\Gamma,\CL_i)$ is \emph{special canonical} if for every stable $v\in \mathcal{V}^i(\Gamma)$ (not necessarily open), there exists a multisection $\ess^v$ of $\CL_i\to\oCM^W_{v}$ such that for every $\Lambda\in\d\Gamma$ with $v=v^*_i(\Lambda)$, we have
\begin{equation}\label{cotangent special canonical} \ess|_{\oPM_{\Lambda}}=\Phi_{\Lambda,i}^*\ess^v.\end{equation}
\end{definition}
Note that equalities ~\eqref{eq:pb} and ~\eqref{cotangent special canonical} use the identification of cotangent line bundles over $\CM^W_\Gamma$ and $\CM^W_{\CB\Gamma}$ outlined in Observation \ref{obs:Witten_pulled_back_perm}. Now, as we have done above for the Witten bundle, we can define a family of canonical multisections for the cotangent line bundle as follows:

\begin{definition}\label{def: cotangent canonical}
A \emph{family of canonical multisections for $\CL_i$ bounded by $I$} is a family of canonical multisections $(\ess^\Gamma)_{\Gamma}$, where $\Gamma$ runs over all  $\Gamma \in\RS\text{-Graph}$ with $i \in I(\Gamma)$ and
$\ess^\Gamma\in C_m^\infty(\oPM_\Gamma,\CL_i),$ such that~\eqref{eq:pb} holds. If $I=\Omega$, we omit the suffix ``bounded by $I$."
\end{definition}

\begin{rmk}
It is straightforward to observe that a special canonical multisection for a Witten bundle or for a tautological line $\CL_i$ is canonical.
\end{rmk}
With the definitions of canonical and special canonical outlined for the Witten bundle and for the cotangent line bundles, we have the following for the descendent Witten bundles $E_\Gamma(\vecd)$ as in Definition~\ref{nn:Direct sums of bundles}:
\begin{definition}
\label{def:special canonical Witten descendent}
A multisection of the descendent Witten bundle $E_\Gamma(\vecd)$ is \emph{canonical} (respectively, \emph{special canonical}) if the multisection when restricted to each of its direct summands are canonical (respectively, special canonical).
\end{definition}

\begin{definition}
A \emph{family of canonical multisections for descendent Witten bundles
bounded by $I$} is a collection of multisections $(\ess^{\Gamma,\vecd})$,
where $\Gamma$ runs over elements of $\RS\text{-Graph}$ and $\vecd$
runs over elements of $\NN^{I(\Gamma)}$, satisfying the following.
We have $\ess^{\Gamma,\vecd}$ a canonical multisection of $E_{\Gamma}(\vecd)$ such that for all $\Gamma\in \RS\text{-Graph}$, $\vecd \in \NN^{I(\Gamma)}$, and
$\Lambda\in \partial^0\Gamma \cup\partial^\xch\Gamma$,
\[
\ess^{\Gamma,\vecd  }|_{\oPM_\Lambda}=F_\Lambda^*(\boxplus_{\Xi\in\Conn(\CB\Lambda)}\ess^{\Xi,
\vecd|_{I(\Xi)}}).
\]

We say $(\ess^{\Gamma,\vecd})$ is a \emph{family of special canonical multisections for descendent Witten bundles bounded by $I$} if it is a family of canonical multisections for descendent Witten bundles bounded by $I$ where each multisection $\ess^{\Gamma,\vecd}$ is special canonical.

\end{definition}

\begin{thm}\label{prop:int_numbers_exist}
\begin{enumerate}
\item
Consider $\Gamma$ a smooth graded $W$-spin graph and a descendent Witten bundle $E=E_\Gamma(\vecd)$, where $\vecd  =(d_i)_{i\in I}$ is a vector of non-negative integers and $I(\Gamma)\subseteq I$. Then one can construct a transverse special
canonical multisection for $E\to\oPM_\Gamma$.
\item
For any $I\subseteq \Universe$, one can define a family bounded by 
$I$ of transverse special canonical multisections for $E$. In particular,
families of transverse canonical multisections exist.
\item 
We may choose the family of (2) to satisfy
the following stronger transversality condition: Consider any
$\Gamma$ with $I(\Gamma)\subseteq I$, any $v\in \mathcal{V}(\Gamma)$ and any 
subbundle $E'\subseteq E_{\Gamma}(\vecd)$ of the form
\[
E'= \mathcal{W}\oplus \bigoplus_{i\in I(v),~j\in D_i}\CL_i,
\]
where, for each $i\in I(v)$, we have $D_i\subseteq[d_i]$.
Then we may require that the multisection $s^{E',v}$ obtained from $s^v$ 
 via the projection $E_{\Gamma}(\vecd)\to E'$ is transverse to $0$.
\end{enumerate}
\end{thm}

\begin{proof}
This theorem is similar to Lemma 4.10 of \cite{BCT:II}, which is used to show the existence part of Theorem 3.17 in \cite{BCT:II}, and the proof is completely analogous. Here we claim, in addition, that special canonical families can be found. Although it was not claimed in Theorem 3.17 and Lemma 4.10 of \cite{BCT:II}, the inductive construction in the proof of that lemma gives rise automatically to a ``special canonical family" in our sense, and the same inductive construction extends automatically to give canonical families in rank $2$ as well. The details will be omitted.
\end{proof}

\begin{rmk} 
Note when we specify a canonical multisection $\ess$ on $\oPM_\Gamma$, we have also determined uniquely a multisection $\ess^{\CB\Lambda}$ over $\oPM_{\CB\Lambda}$ for any graph $\Lambda \in \partial^0 \Gamma\cup\partial^\xch\Gamma$.
Similarly, for any special canonical multisection, we have moreover determined uniquely multisections $\ess^v$ for $v\in\mathcal{V}(\Gamma)$ or $v\in\mathcal{V}^i(\Gamma)$ which satisfy the constraints in Definitions~\ref{def: special canonical} and~\ref{def:canonical for L_i}. For this reason, when we state that we have the information of a canonical or special canonical multisection $\ess$, we will, in turn, take this to mean we have also the data of the multisections  $\ess^{\CB\Lambda}$ and $\ess^v$ above.
\end{rmk}

\subsection{Open and closed extended FJRW invariants} 
\subsubsection{Closed extended $W$-spin theory}
\label{subsec:closed extended}
When we perform calculations of open invariants, a generalization of the usual closed $W$-spin intersection numbers naturally emerges.
For any Fermat polynomial $W=\sum_{i=1}^a x_i^{r_i}$,
this generalization is defined by allowing the anchors to have twists
$\tw_i=-1$ for $i$ in a
subset of $[a]$, as seen in Observation~\ref{obs:twist-1}. The remaining twists must be non-negative. In this situation (as in the closed situation with only non-negative twists), one can define the class $c_W$ analogously as in \cite[Section 3.1]{BCT_Closed_Extended} by the formula
\[c_W =  e\left(\left(\bigoplus_{i\in[a]}R^1\pi_*\mathcal{S}_i\right)^{\vee}\oplus \bigoplus_{j =1}^n \CL_j^{d_j}\right).\] We define {\it closed extended genus 0 $W$-spin intersection numbers} to be given by the formula
\begin{equation}\label{defn: closed ext W spin invariant}
\<\tau^{\vec{a}_1}_{d_1}\cdots\tau^{\vec{a}_n}_{d_n}\>^{\text{ext},W}_0:=
\left(\prod_{i\in[a]}r_i\right)\int_{\M^{W}_{0,\{\vec{a}_1, \ldots, \vec{a}_n\}}} \hspace{-1cm} c_W.
\end{equation}
In particular, in rank $2$ for $W=x^r+y^s$, for an index set $I$ with
descendent vector $\vecd  =(d_i)_{i\in I}$, we have
\[
c_W = e\left((R^1\pi_*\mathcal{S}_1)^{\vee}\oplus(R^1\pi_*\mathcal{S}_2)^\vee\oplus \bigoplus_{i\in I} \CL_i^{d_i}\right),
\]
and
\begin{equation}
\label{eq:closed rank two}
\<\prod_{i\in I}\tau^{(a_i,b_i)}_{d_i}\>^{\text{ext}}:=rs\int_{\M^{W}_{0,\{(a_i,b_i)\}}} \hspace{-1cm} c_W
\end{equation}
are the {\it closed extended (genus 0) $\RS$-spin intersection numbers}.

We make several observations which will be needed later.

\begin{obs}
\label{obs:closed extended}
(1) The extended invariant is non-vanishing when the complex rank of the
descendent
Witten bundle agrees with the complex dimension of the moduli space, i.e.,
in the rank two case,
\[
|I|-3= {\sum_{i\in I} a_i - (r-2)\over r}+{\sum_{i\in I} b_i - (s-2)\over
s} + \sum_{i\in I} d_i,
\]
which is equivalent to
\[
m(I,\vecd  )+2s+2r=0.
\]

(2) If we are not in the extended case, i.e., if $a_i\not=-1$ for any $i$
and $b_i\not=-1$ for any $i$, and if furthermore some $a_i=r-1$ or
$b_i=r-1$, then the intersection number \eqref{eq:closed rank two}
vanishes. This is known as Ramond vanishing, and the argument is identical
to the proof in the rank one case in \cite{BCT_Closed_Extended}, Remark 2.1.

(3) The extended topological recursion relation of \cite{BCT_Closed_Extended},
Lemma 3.6 also holds for these higher rank invariants via an identical
proof. In the rank two case, this becomes, for $|I|\ge 3$ with $1,2,3 \in I$,
\begin{align}
\label{eq:closed TRR}
\begin{split}
\<\tau_{d_1+1}^{(a_1,b_1)}\prod_{i\in I\setminus \{1\}}\tau_{d_i}^{(a_i,b_i)}\>^{\mathrm{ext}}
&=\\
\sum_{\substack{a\in\{-1,\ldots,r-1\},\\b\in\{-1,\ldots,s-1\}}}
&\sum_{\substack{A \coprod B = I\setminus\{1\}\\2,3\in A}}
\<\tau_0^{(r-2-a,s-2-b)}\prod_{i\in A} \tau_{d_i}^{(a_i,b_i)}\>^{\text{ext}}\<\tau_0^{(a,b)}\tau_{d_1}^{(a_1,b_1)}\prod_{i \in B}\tau_{d_i}^{(a_i,b_i)}\>^{\text{ext}}.
\end{split}
\end{align}
Note that if none of the $a_i$ or $b_i$ are $-1$, then by the Ramond
vanishing of (2), we may sum over only $a\in \{0,\ldots,r-2\}$
and $b\in \{0,\ldots,s-2\}$ and we obtain the usual topological recursion
relation for FJRW invariants.
\end{obs}

\subsubsection{Open intersection numbers}\label{subsec:open int numbers}

 Using Definition~\ref{def: relative Euler class} and Theorem~\ref{prop:int_numbers_exist}, the following definition makes sense:
 \begin{definition}
Consider the situation of Theorem~\ref{prop:int_numbers_exist} above. The \emph{open intersection number} relative to the (transverse, canonical) multisection $\mathbf{s}^\Gamma$ for a graded $W$-spin graph $\Gamma$ is
\begin{equation}\label{eqdef:int_number}\left\langle \prod_{i\in I(\Gamma)}\tau^{(a_i,b_i)}_{d_i}\sigma_1^{k_1(\Gamma)}\sigma_2^{k_2(\Gamma)}\sigma_{12}^{k_{12}(\Gamma)} \right\rangle^{\mathbf{s}^\Gamma,o}:=
\int_{\oPM_\Gamma}e(E ; \mathbf{s}^\Gamma|_{U_+\cup\partial^0\oPM_\Gamma})=\#Z(\mathbf{s}^\Gamma)\in\mathbb{Q},\end{equation}
where $e(E ; \mathbf{s}^\Gamma|_{U_+\cup\partial^0\oPM_\Gamma})$ is defined as in Definition~\ref{def: relative Euler class} and $\#Z(\mathbf{s}^\Gamma)$ as discussed in \textsection\ref{subsubsec:multisections}. We emphasize that this does not just depend on the twists but also the graph $\Gamma$ (e.g., this is defined for disconnected graphs).
In the special case when $\Gamma=\Gamma_{0,k_1, k_2, 1, \{(a_i, b_i)\}_{i \in I}}^W$, we write
\[\left\langle \prod_{i\in I}\tau^{(a_i,b_i)}_{d_i}\sigma_1^{k_1}\sigma_2^{k_2}\sigma_{12} \right\rangle^{\mathbf{s}^\Gamma,o}:=
\int_{\oPM_{0,k_1, k_2, 1, \{(a_i, b_i)\}_{i \in I}}}e(E ; \mathbf{s}^\Gamma|_{U_+\cup\partial^0\oPM_{0,k_1, k_2, 1, \{(a_i, b_i)\}_{i \in I}}}).\]
When we are in the case where $\text{rank}(E)\neq\dim\oCM^W_\Gamma$, the integral is defined to be zero.
 \end{definition}

 \begin{rmk}
The open intersection number $\left\langle \prod_{i\in I(\Gamma)}\tau^{(a_i,b_i)}_{d_i}\sigma_1^{k_1(\Gamma)}\sigma_2^{k_2(\Gamma)}\sigma_{12}^{k_{12}(\Gamma)} \right\rangle^{\mathbf{s},o}$ is independent of the choice of $U_+$.
 \end{rmk}

If there are no tautological line bundles associated to a given internal marking, we will sometimes omit the subscript $0$ in the notation $\tau^{(a_i,b_i)}_0$
inside the brackets $\<\>^{\mathbf{s},o}$. If $\ess:=(\mathbf{s}^\Gamma)$ is a canonical family of transverse multisections, then we use the notation
\begin{align*}
&\left\langle \prod_{i\in I(\Gamma)}\tau^{(a_i,b_i)}_{d_i}\sigma^{k(\Gamma)} \right\rangle^{\mathbf{s},o}:=
\left\langle \prod_{i\in I(\Gamma)}\tau^{(a_i,b_i)}_{d_i}\sigma^{k(\Gamma)} \right\rangle^{\mathbf{s}^\Gamma,o}:=\left\langle \prod_{i\in I(\Gamma)}\tau^{(a_i,b_i)}_{d_i}\sigma_1^{k_1(\Gamma)}\sigma_2^{k_2(\Gamma)}\sigma_{12}^{k_{12}(\Gamma)} \right\rangle^{\mathbf{s}^\Gamma,o},\\&
\left\langle \prod_{i\in [l]}\tau^{(a_i,b_i)}_{d_i}\sigma_1^{k_1}\sigma_2^{k_2}\sigma_{12} \right\rangle^{\mathbf{s},o}:=
\left\langle \prod_{i\in [l]}\tau^{(a_i,b_i)}_{d_i}\sigma_1^{k_1}\sigma_2^{k_2}\sigma_{12} \right\rangle^{\mathbf{s}^{\Gammar},o}.
\end{align*}
\begin{obs}\label{obs:int_dec_over_conn} The intersection numbers decompose over the connected components of the graph $\Gamma$, i.e.,
\[
\left\langle \prod_{i\in I(\Gamma)}\tau^{(a_i,b_i)}_{d_i}\sigma^{k(\Gamma)} \right\rangle^{\mathbf{s},o}=
\prod_{\Lambda\in\Conn(\Gamma)}\left\langle \prod_{i\in I(\Lambda)}\tau^{(a_i,b_i)}_{d_i}\sigma^{k(\Lambda)} \right\rangle^{\mathbf{s}^\Lambda,o}.\]
\end{obs}

We now note that the open invariants only depend on the the boundary conditions given by the canonical multisection.

\begin{pr}\label{invariants are independent of extension of boundary}
Let $\mathbf{s}^*$ and $(\mathbf{s}')^*$ be two families of canonical multisections bounded by $I$. Suppose that $\mathbf{s}^{\Gamma}|_{\oPM_\Lambda} = (\mathbf{s}')^{\Gamma}|_{\oPM_\Lambda}$ for all boundary strata $\Lambda \in \partial^0\Gamma \cup \partial^{\XCH}\Gamma$. Then
$$
\left \langle\prod_{i \in I(\Gamma)}\tau^{(a_i,b_i)}_{d_i}\sigma_1^{k_1(\Gamma)}\sigma_2^{k_2(\Gamma)}\sigma_{12}\right\rangle^{\mathbf{s}^\Gamma,o} = \left \langle\prod_{i \in I(\Gamma)}\tau^{(a_i,b_i)}_{d_i}\sigma_1^{k_1(\Gamma)}\sigma_2^{k_2(\Gamma)}\sigma_{12}\right\rangle^{(\mathbf{s}')^\Gamma,o}.
 $$
\end{pr}

\begin{proof}
By viewing $\ess^\Gamma$ and $(\ess')^{\Gamma}$ as multisections from $[0,1] \times \oPM_{\Gamma}$ via pullback, we can consider the linear homotopy
 $$
  H:= t(\mathbf{s})^\Gamma + (1-t)(\mathbf{s}')^{\Gamma}.
 $$
 Note that the homotopy $H|_{\partial([0,1]\times \oPM_{\Gamma}) }$, as defined in  Notation~\ref{not: restriction notation}, is independent in time as $(\mathbf{s})^\Gamma|_{\partial \oPM_{\Gamma}}= (\mathbf{s}')^\Gamma|_{\partial \oPM_{\Gamma}}$ and non-vanishing. Moreover, since both $\mathbf{s}^{\Gamma}$ and $(\mathbf{s}')^\Gamma$ are canonical, they are also strongly positive. Using Definition~\ref{def: strongly positive}(2), we have an open neighborhood of the form in~\eqref{eq:U_+} where both $(\mathbf{s})^\Gamma$ and $(\mathbf{s}')^\Gamma$ are positive, hence there exists an open neighborhood $V$ of $\partial^+\CM_\Gamma $ in $\CM_\Gamma $ so that $H$ is non-vanishing on $[0,1]\times U $ where
 $$
 U := (V \cap \oPM_{\Gamma}) \cup \partial\oPM_\Gamma
 $$
 Perturb $H$ inside $[0,1]\times (\oPM_\Gamma\setminus U)$ to obtain a homotopy $H'$ so that $H'$ is transverse and does not vanish in $[0,1]\times U$. By applying Lemma \ref{lem:zero diff as homotopy}, the claim follows.
\end{proof}

\subsection{Canonical homotopies}

Recall Definition~\ref{nn:Direct sums of bundles} for the descendent Witten bundle $E_{\Gamma}(\vecd  )$ with respect to the descendent vector
$\vecd   = (d_i)_{i\in I}$ where $I (\Gamma)\subset I$.
We now define a canonical homotopy between two multisections of
$E_{\Gamma}(\vecd  )$.
The definition essentially emulates Definition~\ref{def:canonical for Witten} of a canonical multisection above, but with the unit interval added.

\begin{definition}\label{def:canonical homotopy}
A \emph{canonical homotopy} between two multisections $s_0$ and $s_1$ of $E_\Gamma(\vecd  )$ is a multisection
\[H\in C_m^\infty([0,1]\times \oPM_\Gamma,\pi^*E_\Gamma(\vecd  )),~\text{such that } H(0,-)=s_0,~H(1,-)=s_1,\]
where $\pi:[0,1]\times \oPM_\Gamma\to \oPM_\Gamma$ is the projection to 
the second factor with the following properties:
\begin{enumerate}
\item For any $t\in [0,1]$, the multisection $H(t,-)$ is strongly positive.
\item\label{it:2_for_hom} For any $\Lambda\in\partial^!\Gamma\setminus\partial^+\Gamma
$ with only open vertices, there exists a positive integer $N_\Lambda$ and homotopies $H_1^{\Xi,\Lambda}, \ldots, H_{N_\Lambda}^{\Xi,\Lambda} \in
C_m^\infty([0,1]\times\oPM_{\Xi},\pi^*E_\Xi(\vecd  ))$ for each $\Xi\in\Conn(\CB\Lambda)$, such that
\begin{equation}\label{eq:pb_Witten_H}H|_{[0,1]\times\oPM_\Lambda}=F_\Lambda^*\left(\biguplus_{i=1}^{N_\Lambda}\left(\boxplus_{\Xi\in\Conn(\CB\Lambda)}H_i^{\Xi,\Lambda}\right)\right).\end{equation}
\end{enumerate}

We say that a canonical homotopy $H$  between two multisections $s_0$ and $s_1$ is \emph{transverse} if it is transverse as a multisection  in $C_m^\infty([0,1]
\times\oPM_\Gamma,\pi^*E_\Gamma(\vecd  ))$.
The homotopy is said to be \emph{simple} if $N_\Lambda=1$ for all $\Lambda$ and the multisections $H_1^{\Xi,\Lambda}$ depend only on $\Xi$ and not on $\Lambda$. In this case, we denote $H_i^{\Xi,\Lambda}$ by $H^\Xi$.
\end{definition}
\begin{definition}\label{def:family of canonical homotopy}
A \emph{family of simple canonical homotopies} (bounded by $I\subseteq \Universe$) is a collection of simple canonical homotopies $H^\Lambda\in C_m^\infty([0,1]
\times\oPM_\Lambda,\pi^* E_\Lambda(\vecd  ))$ for every graded smooth graph $\Lambda$ (with $I(\Lambda)\subseteq I$) such that all $H_i^{\Xi,\Lambda}$ appearing in \eqref{eq:pb_Witten_H} equal $H^\Xi$.
\end{definition}

First we note that there exists a canonical homotopy between two given
canonical sections.

\begin{lemma}\label{lem:transverse_homotopy}
Transverse canonical homotopies between any transverse canonical
multisections $\ess_0,\ess_1\in C_m^\infty(\oPM_\Gamma,E_\Gamma(\vecd))$ exist.
\end{lemma}

\begin{proof}
The proof is similar to the proof of Lemma 4.11 in \cite{BCT:II} 
with $E_1=0,~E_2=E_\Gamma(\vecd),~A=\partial^0\Gamma$. While the bundle $E_\Gamma(\vecd)$ is different in our context rather than the context in \cite{BCT:II}, one can still use Proposition 6.2(3) of \cite{BCT:II} to obtain a finite set of multisections $w_j$ of $E_2$ that span each fibre as desired in the proof. The rest of the proof is the same and hence will be omitted.
\end{proof}

We will now show that there exists a family of \emph{simple} canonical homotopies that, under certain dimensionality and boundary marking constraints, is transverse to the zero multisection of the bundle $\pi^*E_{\Lambda}(\vecd)$. We will also be able to impose other properties on various boundary strata that will help us compute open invariants in later sections. Before we start the lemma, we will require the following observation regarding parities of the Witten bundle and the dimension of the moduli space when one has few roots and no internal tails.

\begin{lemma}\label{obs:ind=-1andk=0}
Suppose $\Lambda$ is a stable, smooth and graded $W$-spin graph without internal tails and with at most $3$ boundary tails with twist $(r-2,s-2)$. Then $\rank\cW_\Lambda\neq \dim\oPM_\Lambda+1$.
\end{lemma}
\begin{proof}
As usual, we let $k_{12} = k_{12}(\Lambda)$ and $k_i=k_i(\Lambda)$ (see Definition~\ref{def:open_W_graded2}(2) or the start of Subsection~\ref{subsec:Rank 2}). Then, by using ~\eqref{eq:open_rank1} and~\eqref{eq:open_rank2} in a similar way as done in Proposition~\ref{prop:balanced}(a), we have that
\begin{equation}\label{eqn: obs on parity of with small number of roots}
k_1+k_{12} = rm_1 + 1; \qquad k_2 + k_{12} = sm_2 + 1
\end{equation}
for some $m_1, m_2 \in \Z_{\geq 0}$ and that
 \[\rank\cW_\Lambda = (r-2)m_1+(s-2)m_2.\]
Recall that the dimension formula for the open moduli space in Equation~\eqref{real dim of open Moduli of discs} gives us that \[\dim \oPM_\Lambda = k_1+k_2+k_{12}-3=rm_1+sm_2-k_{12}-1.\]
Suppose for the sake of contradiction that $\rank\cW_\Lambda= \dim\oPM_\Lambda+1$. Then
\[k_{12}= 2m_1+2m_2.\] This implies that $k_{12}$ is even and, since we have assumed that $k_{12}\leq 3$, we then have that it is either $0$ or $2$. If $k_{12}=0$ then $m_1=m_2=0$, hence $\Lambda$ is not a stable graph. If $k_{12}=2$, then either $m_1$ or $m_2$ is $0$; however, if $m_i=0$, then $k_i+k_{12}\geq k_{12}=2>1$, contradicting Equation~\eqref{eqn: obs on parity of with small number of roots}.
\end{proof}

\begin{lemma}\label{lem:partial_homotopy}
Fix $I\subseteq \Universe$. Suppose   $\ess_0$ and $\ess_1$ are two families of transverse canonical multisections bounded by $I$. 
Then one can construct a family $H^*$ of simple canonical homotopies bounded by $I$
and for each $A\subseteq I$ a distinct choice of time $t_A\in (0,1)$ between $\ess_0$ and $\ess_1$
with the following properties. 

Let $\Gamma$ be any smooth, graded
$W$-spin graph with only open vertices and $I(\Gamma)\subseteq I$. Fix $\vecd  =(d_i)_{i\in I}$ a vector of descendents.
\begin{enumerate}
\item
Suppose that $\Gamma$ satisfies one of the following hypotheses:
\begin{enumerate}
\item $\Gamma$ is irrelevant;
\item $\Gamma$ is connected and
$\rank E_\Gamma(\vecd  ) \ge \dim \oPM_\Gamma + 2$;
\item $\Gamma$ contains two connected components $\Xi_1, \Xi_2$ with
distinct internal marking sets, i.e., $I(\Xi_1) \neq I(\Xi_2)$, such that $\rank E_{\Xi_i}(\vecd) = \dim \oPM_{\Xi_i} + 1$ and $k_{12}(\Xi_i)\le 2$.
\end{enumerate}
Then $H^\Gamma$ is nowhere vanishing.
\item\label{it:hom_invariants}
If $\Gamma$ is  connected and one of the following holds:
\begin{enumerate}
\item  $\rank(E_\Gamma(\vecd))= \dim\oPM_\Gamma+1$ and $k_{12}(\Gamma)\leq 2,$ or
\item $\rank(E_\Gamma(\vecd))= \dim\oPM_\Gamma$ and $k_{12}(\Gamma)\leq 1$,
\end{enumerate}
then $H^\Gamma$ is transverse  when restricted to any stratum of $[0,1]\times
\oPM_\Gamma$, that is, $H^{\Gamma}|_{[0,1]\times\oPM_\Lambda }$ is transverse to the zero multisection of $\pi^*E_{\Lambda}(\vecd  ) \to [0,1]\times\oPM_{\Lambda}$ for any $\Lambda \in \partial^0 \Gamma \cup\partial^\xch\Gamma$. 
\item\label{it:hom_WC}
If $\Gamma$ is connected, 
we have:
\begin{enumerate}
\item
If $\rank(E_\Gamma(\vecd  ))= \dim\oPM_\Gamma+1$ and $k_{12}(\Gamma)\leq 2$ then $H^\Gamma$ will be non-zero in 
$([0,1]\setminus\{t_{I(\Gamma)}\})\times\oPM_\Gamma$. 
Note by Lemma \ref{obs:ind=-1andk=0}, if $I(\Gamma)=\emptyset$,
this case does not occur.
\item
If $\rank(E_\Gamma(\vecd  ))= \dim\oPM_\Gamma$ and $k_{12}(\Gamma)\leq 1$, then for any $A\subseteq I$ we have
\[
H^\Gamma|_{\{t_A\} \times (\oPM_\Gamma\setminus\partial\oPM_\Gamma)}
\pitchfork 0
\]
when considered as a multisection of
\[
E_\Gamma(\vecd)\to \{t_A\} \times
(\oPM_\Gamma\setminus\partial\oPM_\Gamma)
\simeq
\oPM_\Gamma\setminus\partial\oPM_{\Gamma}.
\]
(Note that the point here is that while the multisection $H^{\Gamma}$
is transversal to the zero-section of $E_{\Gamma}(\vecd  )$ over
$[0,1]\times \oPM_{\Gamma}$, there will be a finite number of times
$t$ for which the restriction of $H^{\Gamma}$ to $\{t\}\times
\oPM_{\Gamma}$
is not transversal to the zero-section, and we want the set of such  times to be disjoint from the set $\{t_A\}_{A\subseteq I}$.)
\end{enumerate} 
\end{enumerate}

\end{lemma}

\begin{proof}
Let $\ess_0,\ess_1$ be as in the statement. We will construct a family of simple canonical homotopies ${H}^*$ which satisfies the requirements of the various items.

Note that for disconnected $\Gamma$, the homotopies $H^{\Gamma}$ are
determined by \eqref{eq:pb_Witten_H} (using simplicity)
and the corresponding homotopies
$H^{\Lambda}$ for $\Lambda\in\Conn(\Gamma)$.  Thus we will construct homotopies
${H}^\Gamma$ for smooth connected graded $\Gamma$ with $I(\Gamma)\subseteq I$ by induction on $n=\dim\oPM_\Gamma$. If Items (2) and (3) are established,
note that (1)(c) then follows from (3)(a).

The case $n=-1$ is trivial.
Suppose we have constructed a homotopy for any $\Gamma$ such that
$\dim\oPM_\Gamma<n$.
Now take $\Gamma$ satisfying $\dim \oPM_{\Gamma}=n$.
We will analyze boundary strata of $\oPM_{\Gamma}$ indexed by
$\Lambda\in\partial^0\Gamma\cup \partial^{\xch}\Gamma$ with no internal
edges.
If $\Lambda$ has a contracted boundary edge, we define $H^\Lambda$ as the positive, linear homotopy \[H^\Lambda(t,-)=(1-t)\ess^{\Lambda}_0+
t\ess^{\Lambda}_1.\]
If $\Lambda$ does not have a contracted boundary edge, the simple canonical homotopy $H^\Lambda$ is determined
inductively by \eqref{eq:pb_Witten_H} to be
\begin{equation}
\label{eq:H lambda def inductive}
{H}^\Lambda=F_\Lambda^*\big(\boxplus_{i=1}^v H^{\Xi_i}\big),
\end{equation}
where $\Xi_1,\ldots,\Xi_v$ are the components of $\CB\Lambda$, so that $\oPM_{\CB\Lambda} = \prod_{i=1}^v\oPM_{\Xi_i}$.
Thus, the homotopy we are trying to construct is determined at times $t=0$ and
$1$ and on the boundary strata of $\oPM_{\Gamma}$. Hence, we now can define a multisection $\widetilde H^{\Gamma}$ on the boundary \[\partial([0,1]\times
\oPM_\Gamma) =\big(\{0,1\}\times \oPM_\Gamma \big) \cup 
\big([0,1]\times\partial(\oPM_\Gamma) \big)\] via:
\[\widetilde{H}^\Gamma|_{\{\eps\}\times\oPM_\Gamma}(\eps,-)=\ess_\eps^\Gamma,~\epsilon\in \{0,1\},~\widetilde{H}^\Gamma|_{[0,1]\times\oPM_\Lambda}=H^{\Lambda},~\Lambda\in\partial^0\Gamma.\]
Extend $\widetilde H^{\Gamma}$ to a set $U_+$ of the form
\begin{equation}
\label{eq:U plus homotopy extension}
U_+:=\big(V\cap([0,1]\times \oPM_{\Gamma})\big) \cup \partial([0,1]\times\oPM_{\Gamma}),
\end{equation}
where
$V$ is a neighborhood of $[0,1]\times\partial^+\oCM^W_\Gamma$ in $[0,1]\times
\oCM^W_\Gamma$.\footnote{This $U_+$ is analogous to that in Equation~\eqref{eq:U_+} but is not the same, due to the interval for the homotopy.}
This extension can be done by the same argument of Lemma 6.5 in \cite{BCT:II}, steps 1-3.
We may do this in such a way so
that $\widetilde{H}^{\Gamma}(t,-)$ is strongly positive.

Now that we have established a provisional function $\widetilde H^\Gamma$ on the set $U_+$ containing the boundary of $\oPM_{\Gamma}\times [0,1]$,
 we will adapt $\widetilde H^{\Gamma}$ on the interior of $\oPM_\Gamma\times [0,1]$ to obtain a homotopy $H^\Gamma$ that will satisfy the properties described in the Lemma. This will require a case-by-case analysis depending on the properties of $\Gamma$.

\medskip

{\bf Case 1: $\Gamma$ irrelevant.} Suppose $\Gamma$ is irrelevant.
Since $\Gamma$ is smooth and connected, cases
\eqref{irrelevant:item 1} and~\eqref{irrelevant:item 2} of
Definition~\ref{def:special kind of graded graphs} do not apply
and we are thus in case~\eqref{irrelevant:item 3}. Thus $\Gamma$ consists of a single vertex $v$ and
there is an $i\in \{1,2\}$ such that $\tw_i=\alt_i=0$ for all boundary tails
adjacent to $v$.

If $\Lambda \in \partial^0\Gamma$, then each edge
of $\Lambda$ has exactly one half-edge with $\tw_i=\alt_i=0$.
Thus $\Lambda$  will have one fewer half-edge with $\tw_i\not=0$ than
it has vertices.
So by the pigeon-hole principle, $\Lambda$ must contain a vertex
with all adjacent half-edges having $\tw_i=\alt_i=0$, and $\Lambda$
is also irrelevant.
Hence there will be some connected component of $\CB\Lambda$ such that each boundary tail has $\tw_i = \alt_i = 0$ for some $i \in \{1,2\}$. By the condition of Definition~\ref{def: strongly positive}\eqref{strongly positive:item 3} of strongly positive,
we have that for any $(t,\Sigma)\in U_+$ the $\cW_i$-component of $\widetilde{H}^\Gamma(t,\Sigma)$ evaluates positively on at least one boundary point of $\partial\Sigma$.

We can extend $\widetilde H^\Gamma$ continuously from $U_+$
to an open neighbourhood $U$ of $U_+$
which contains a neighborhood $W$ of $\partial([0,1]\times \oPM_{\Gamma})$.
Further, by continuity,
we may do this in such a way so that
$\widetilde H^\Gamma(t,\Sigma)$ evaluates positively on at least one boundary point of $\partial\Sigma$ for each $(t,\Sigma)\in U$,
after possibly shrinking $W$.

Now let $U'$ be another open neighborhood of $U_+$ whose closure is contained in $U$, and take two non-negative functions $\rho_0$ and $\rho_1$ over the moduli $\oPM_\Gamma$ that sum to the constant function $1$ and so that
 \begin{equation} \begin{aligned}
 \rho_0|_{U'}&=1 \text{ and has support }\text{Supp}(\rho_0)\subseteq U \text{ while } \\
 \rho_1|_{([0,1]\times \oPM_\Gamma)\setminus U}&=1 \text{ and has support }\text{Supp}(\rho_1)\subseteq ([0,1]\times\oPM_\Gamma)\setminus U'.
 \end{aligned}\end{equation}

By \cite{BCT:II}, Observation 7.6, we may find a section
$s_+$ of $C^\infty(\oPM_\Gamma,\cW_i)$ which is nowhere vanishing
on $\PM_{\Gamma}$ (see Definition \ref{def:PM Gamma}),
and with the property that for $\Sigma\in \PM_{\Gamma}$,
$s_+(\Sigma)$ evaluates positively on $\partial\Sigma$.\footnote{In \cite{BCT:II}, such a section was constructed on the moduli
of smooth disks, but there is no difficulty in constructing the section
on the whole moduli away from the boundary, i.e., extending the
construction across singular disks with only internal nodes.}
Viewing $s_+$ as a section of $C^\infty(\oPM_\Gamma,E_\Gamma(\vecd  ))$, we may pull it back to a section $\tilde s_+$ of $C^{\infty}([0,1] \times \oPM_\Gamma, E_{\Gamma}(\vecd))$ and define the homotopy \[H^\Gamma = \rho_0 \widetilde{H}^\Gamma+\rho_1\tilde{s}_+.
\]
This homotopy satisfies our requirements, namely, it is nowhere vanishing.
In particular, statement (1)(a) of the lemma is satisfied
if $\Gamma$ is irrelevant and connected. If $\Gamma$ is irrelevant
but not connected, then one of its connected components is irrelevant and (1)(a)
still holds.

{\bf Case 2: $\Gamma$ relevant.} Suppose now that $\Gamma$ is relevant.
Recall that we are proceeding by induction on $\dim\oPM_\Gamma$. We
take special care in three cases for which the statement of the lemma
requires additional structure, namely
\begin{enumerate}[(a)]
\item $\rank E_\Gamma(\vecd  ) \ge \dim\oPM_\Gamma+2$;
\item $\rank E_\Gamma(\vecd  )= \dim\oPM_\Gamma + 1$ and $k_{12}(\Gamma)\leq 2$;
\item $\rank E_\Gamma(\vecd  )=\dim\oPM_\Gamma$ and $k_{12}(\Gamma)\leq 1$.
\end{enumerate}
For all other relevant graphs $\Gamma$,
we will extend $\widetilde{H}^\Gamma$ to an arbitrary homotopy $H^\Gamma$
on $[0,1]\times\oPM_{\Gamma}$ at the end of each induction step.

We will go through the three cases consecutively, extending $\widetilde H^\Gamma$ to a homotopy $H^\Gamma$ that satisfies the required properties given in the statement of the lemma. In doing so, we will be considering boundary
strata corresponding to
$\Lambda \in \partial^0\Gamma\cup\partial^\xch\Gamma$ with no internal
edges, and
study the vanishing properties of $\widetilde H^{\Gamma}|_{[0,1]\times
\oPM_{\Lambda}}$.
If $\Lambda$ is irrelevant, then it follows inductively from (1)(a)
that $\widetilde H^{\Gamma}$ does not vanish on $\oPM_{\Lambda}$. We note that the property that the homotopy is nonvanishing is a stronger condition than any other condition that we require in the statement of the lemma. Hence
we will be able to restrict attention below to the case
that $\Lambda$ is relevant.   In particular,
Observation \ref{obs:base dim yoga} applies, and we use the notation $\alpha,
\beta,\nu$ and $\sigma$ of that observation.

\begin{enumerate}[(a)]
\item
$\alpha\ge 2$. In this case, we  extend $\widetilde H^\Gamma$
to $H^{\Gamma}$ so that  $H^\Gamma$ is nonvanishing: 
 from \eqref{eq:dim rank comparison} and
the pigeonhole principle, as $\alpha+\beta-1\ge 1$,
there must be a connected component
$\Xi$ of $\CB\Lambda$ such that $\dim\oPM_\Xi\leq\rank E_{\Xi}(\vecd  )-2$.
By the inductive hypothesis of (1)(b) of the statement of the lemma,
$H^\Xi$ is nowhere vanishing, and therefore also $\widetilde{H}^\Gamma$ will not vanish at $[0,1]\times \oPM_{\Lambda}$. Extend $\widetilde{H}^\Gamma$ transversally to all $[0,1]\times\oPM_\Gamma,$ and call the result $H^\Gamma$. This extension
may be constructed as in Lemma 3.54 of \cite{PST14} or Lemma 4.11 of
\cite{BCT:II}.
Note that, in this case, since $\dim ([0,1]\times\oPM_\Gamma) < \rank E_\Gamma(\vecd  )$, the transversality assures that $H^\Gamma$ is nowhere vanishing, proving the inductive step for the case given in (1)(b) of the statement of the lemma.

\item $\alpha=1$ and $k_{12}(\Gamma)\leq 2$.  Again consider a relevant
$\Lambda\in \partial^0\Gamma\cup \partial^{\xch}\Gamma$. Suppose first
that $\beta\ge 1$. Then $\alpha+\beta-1\ge 1$ and
\eqref{eq:dim rank comparison} shows again there is a stable component
$\Xi_i$ of $\CB\Lambda$ with $\dim\oPM_{\Xi_i}\le \rank E_{\Xi_i}(\vecd  )-2$.
Thus $\widetilde H_{\Gamma}$ again inductively does not vanish on
$\oPM_{\Lambda}$.

Next, assume $\beta=0$, so that necessarily $\sigma=0$. There are now two possible
subcases:
(i) $\dim\oPM_{\Xi_i}\leq\rank E_\Xi(\vecd  )-2$ for some $i$;
(ii) $\dim\oPM_{\Xi_i}=\rank E_{\Xi_i}(\vecd  )-1$ for all $i$.
In case (i), then inductively as before $\widetilde H^{\Gamma}$ is
non-vanishing on $\oPM_{\Lambda}$.

In case (ii), first note that in any event $\beta=0$ implies
that no half-node of $\Lambda$ has twist
$(r-2,s-2)$, i.e., is fully twisted.
Indeed, if such a half-edge $h$ exists, then the opposite half-edge
$\sigma_1(h)$ has $\tw=\alt=(0,0)$,
and hence $\sigma_1(h)$ is forgotten in $\CB\Lambda$, i.e., $\beta\ge 1$.

By \eqref{eq:open_rank2}, we have for each $i$,
\begin{align}\label{homotopy parity computation}
\begin{split}
k_1(\Xi_i)+k_2(\Xi_i) +k_{12}(\Xi_i) - 3 \equiv {} & \dim \oPM_{\Xi_i}\\
 \equiv  {} &
\rank E_{\Xi_i}(\vecd  )
 -1\\
 \equiv {} & k_1(\Xi_i) + k_2(\Xi_i) - 1 \pmod 2.
\end{split}
\end{align}
As $k_{12}(\Gamma)$ is assumed to be at
most $2$, we thus see that each component of $\CB\Lambda$ has either
zero or two fully twisted boundary points and at most one component can have two fully twisted boundary points.
Note by Lemma \ref{obs:ind=-1andk=0} that
$I(\Xi_i)\not=\emptyset$ for any $i$, and necessarily the $I(\Xi_i)$
are disjoint. Thus (1)(c) applies inductively to conclude that
$H^{\Lambda}$ is nowhere vanishing in this case also.

Thus over $U_+$, we have that the function $\widetilde{H}^\Gamma$ is transverse to the zero section, so we may then extend it as before to a global transverse homotopy. It will have a finite number of zeroes.
By slightly perturbing the homotopy, we may assume that there is no $x\in\oPM_\Gamma$ with two different times $t_1,t_2$ such that (local branches of)
the homotopy vanishes at $(t_1,x),(t_2,x)$.

We further modify the homotopy to guarantee that all zeroes occur
at time $t_{I(\Gamma)}$. Let $\pi_i$, $i=1,2$ be the projection
of $[0,1]\times \oPM_{\Gamma}$ onto the $i^{th}$ factor.
Let $S$ be the collection of zeroes of branches where the homotopy vanishes. Choose a diffeomorphism $\psi:[0,1]\times \oPM_\Gamma\to [0,1]\times
\oPM_\Gamma$ which satisfies
\begin{enumerate}[(i)]
\item $\pi_2\circ\psi=\pi_2$;
\item $\psi|_{([0,1]\times \partial\oPM_\Gamma)\cup U_+}=\textup{id}$;
\item $\pi_1(\psi(S))=t_{I(\Gamma)}$.
\end{enumerate}
We define ${H}^\Gamma = \widetilde{H}^\Gamma\circ \psi^{-1},$ where we use the canonical trivialization of $E_\Gamma(\vecd  )|_{[0,1]\times x},$ for any $x\in\oPM_\Gamma$. This proves the inductive steps for the cases given in (2)(a) and (3)(a).

\item
$\rank E_\Gamma(\vecd  )=
\dim\oPM_\Gamma$
 and $k_{12}(\Gamma)\leq 1$. By the analogous parity computation as in~\eqref{homotopy parity computation}, we have that $k_{12}(\Gamma) =1$.
Now let $\Lambda\in\partial^0\Gamma\cup\partial^\xch\Gamma$ be relevant
as before. Again, if $\dim\oPM_{\Xi_i}\le \rank E_{\Xi_i}(\vecd  )-2$
for some connected component $\Xi_i$ of $\CB\Lambda$, then necessarily $\widetilde H_{\Gamma}$
is non-vanishing on $\oPM_{\Lambda}$. So assume now this does not
occur. Then we have a further subdivision of cases, with notation as
in \eqref{eq:dim rank comparison}.
\begin{enumerate}[(i)]
\item $\beta=0$;
\item $\beta\ge 1$.
\end{enumerate}

In case (i), necessarily $\sigma=0$, and \eqref{eq:dim rank comparison} then
gives
$\dim \oPM_\Lambda = \rank E_{\Gamma}(\vecd  ) - (\nu-1)$. Without loss of generality, we assume that $k_{12}(\Xi_1) = 1$ and $k_{12}(\Xi_j ) = 0$ for all $ 2 \le j \le \nu$. By a parity computation similar to ~\eqref{homotopy parity computation} above, we have that
\begin{equation}
\label{eq:another dim relation}
\dim \oPM_{\Xi_j} = \rank E_{\Xi_j}(\vecd  ) + k_{12}(\Xi_j) - 1 +2p_j
\end{equation}
for some integers $p_j \in \Z$ and with the constraint that $\sum_j p_j = 0$.
If any of these $p_j$ are nonzero, then necessarily we have
$\dim \oPM_{\Xi_i}\le \rank E_{\Xi_i}(\vecd  )-2$ for some $i$, which we
assumed did not occur.
Thus we may assume that $p_j = 0$ for all $j$.
If $\nu\ge 3$, then $\Xi_2,\Xi_3$ satisfy
$\dim \oPM_{\Xi_j} + 1 = \rank E_{\Xi_j}(\vecd  )$.
Thus we see inductively from statement (1)(c) of the lemma
that $\widetilde H^{\Gamma}$ restricted to $[0,1]\times \oPM_{\Lambda}$ is
nowhere vanishing, hence again transversal. Next, suppose $\nu=2$.
We know that $H^{\Xi_2}$ on $[0,1]\times \oPM_{\Xi_2}$ 
is only zero when $t=t_A$, where $A=I(\Xi_2)$.
Thus when $t\not= t_A$, the homotopy $\widetilde H^{\Gamma}$
restricted to $[0,1]\times\oPM_{\Lambda}$ is non-vanishing,
hence again transversal.
On the other hand, the homotopy $H^{\Xi_1}$ on $[0,1]\times\oPM_{\Xi_1}$, 
restricted to $\{t_A\}\times\oPM_{\Xi_1}$, is transverse to the
zero section.
Thus the homotopy $\widetilde H^{\Gamma}$ restricted to $[0,1]\times\oPM_{\Lambda}$ 
is transverse to the zero section at time $t_A$,
and so is transverse at all times.

In case (ii),
$\beta=\sum_{j=1}^{\nu-\sigma} k_{12}(\Xi_j) - 1$
(bearing in mind that $k_{12}(\Gamma)=1$ and none of the partially
stable components contain a root).
Note that \eqref{eq:another dim relation} still holds for
$1\le j\le \nu-\sigma$ 
for some integers $p_j \in \Z$.
Thus by \eqref{eq:dim rank comparison} 
we obtain $\sum_{j=1}^{\nu-\sigma} (k_{12}(\Xi_j) -1 + 2p_j) =  -\beta
 - (\nu-\sigma)+1$.
Simplifying, we have that
\[
\sum_{j=1}^{v-\sigma} (k_{12}(\Xi_j) + 2p_j) = -\beta+1.
\]
Consequently, if $\beta>1$, by the pigeonhole principle,  there must exist a
$j$ such that $\dim \oPM_{\Xi_j} \le  \rank E_{\Xi_j}(\vecd  )  -2$,
again contradicting our assumption.
If $\beta=1$ and $\nu-\sigma\ge 2$,
we must have either again
$\dim \oPM_{\Xi_j}\le \rank E_{\Xi_j}(\vecd  )-2$
for some $j$ or $\dim \oPM_{\Xi_j}= \rank E_{\Xi_j}(\vecd  )-1$
for two choices of $j$. Thus by (1)(b) or (c)
again the restriction of $\widetilde H^{\Gamma}$ to $[0,1]\times\oPM_{\Lambda}$
will not vanish. Finally, if $\beta=1$ and $\nu-\sigma=1$, we necessarily
have $k_{12}(\Xi_1)=2$ and
$\dim \oPM_{\Xi_1} = \rank E_{\Xi_1}(\vecd  )-1$, so by Item (2)(a)
of the statement of the lemma applied inductively, $\widetilde H^{\Gamma}$
is transverse to the zero section on $[0,1]\times\oPM_{\Lambda}$ in this case also.
Thus, $\widetilde{H}^\Gamma$ is transverse on $\partial([0,1]
\times\oPM_\Gamma)$, showing item (2)(b) of the lemma.

Finally, to ensure (3)(b) of the lemma, we proceed as follows.
For any \[t\in T:=
\{t_{I(\Gamma)} \ | \  I(\Gamma) \subseteq I, I(\Gamma) \neq \emptyset\}\] and a boundary point $(t,x)$ in which $\widetilde H^{\Gamma}$ vanishes, extend $\widetilde{H}_\Gamma$ to a neighborhood $U_{t,x}$ of $(t,x)$ in
$[0,1]\times\oPM_{\Gamma}$
in a transverse way.
There is no difficulty in choosing the extension such that the
restriction of this extension to
$U_{t,x}\cap(\{t\}\times
(\oPM_\Gamma\setminus\partial\oPM_\Gamma))$ is transverse.
Now, extend $\widetilde{H}_\Gamma$ to each ${\{t\}\times\oPM_\Gamma}$
for $t\in T$ so that, as a multisection of $E_{\Gamma}(\vecd  )$
on $\{t\}\times\oPM_{\Gamma}$, it is transversal on $\{t\}\times\CM^W_{\Gamma}$.
Again extend $\widetilde{H}^\Gamma$ transversally to the whole space 
$[0,1]\times\oPM_\Gamma$. We call the result $H^\Gamma$.
\end{enumerate}

Lastly, as stated above, for the remaining graded, connected, relevant graphs $\Gamma$ that do not fit in the above cases, we extend $\widetilde{H}^\Gamma$ to an arbitrary homotopy $H^\Gamma$, as the claim requires no additional structure for them.
\end{proof}

Finally, we note that there is a family of transverse simple canonical homotopies between multisections of tautological bundles.

\begin{lemma}\label{lem:hom_trr}
Suppose $d_1>0$ and let $\vecd'=\vecd-(1,0\ldots,0)$ be a vector of descendents. Take $\Lambda\in\partial^0\Gamma\cup\partial^\xch\Gamma$ to be a graph associated to a boundary strata. Let $\ess$ be a family of canonical transverse multisections for $E_\Lambda(\vecd')\to\oPM_\Lambda$, and let $(\ess_0^\Lambda)_\Lambda,(\ess_1^\Lambda)_\Lambda$ be families of transverse multisections of $\CL_1\to\oPM_\Lambda$. Assume that $\ess_0^*$ is a canonical family. Then one can find another family $\rho^*$ of
canonical multisections of $\CL_1$
such that the family of homotopies $H^*$ between $\ess^*\boxplus \ess_0^*$ and $\ess^*\boxplus \ess_1^*$
\begin{equation}\label{eq:form_of_hom}H^\Gamma=\ess^\Gamma\boxplus\left((1-t)\ess_0^\Gamma+t(1-t)\rho^\Gamma+t\ess_1^\Gamma\right)\end{equation} is transverse when restricted to any stratum of $[0,1]\times\oPM_\Gamma$ with $I(\Gamma)\subseteq I,~\rank E_\Gamma(\vecd)\geq\dim([0,1]\times\oPM_\Gamma)$.
\end{lemma}

\begin{proof}
The proof here is identical to the proof of Lemma 4.11 in \cite{BCT:II} in the case where $E_1= E_\Gamma(\vecd'), E_2=\CL_1$, and $A=\partial^0\Gamma\cup\partial^\xch\Gamma$. There, the claim only involved a single graph and did not involve families, but adapting the proof for families adds no difficulty.
\end{proof}
\subsection{Chamber Indices and Symmetric canonical multisections}
\label{subsec:chamber indices}

In this subsection, we give the notation necessary to discuss the
set of all possible tuples of
open FJRW invariants. In particular, we first recall
Definition~\ref{def:balanced,critical,etc} that $\INT(I,\vecd)$ is defined to be the collection of connected, rooted,  graphs with internal markings $I$ balanced with respect to $\vecd$. We recall from
\eqref{eq: def Inv I} that for
$I\subseteq\Universe$ we set
\[
\Inv(I)=\prod_{J\subseteq I} \prod_{\vecd\in \NN^J} \Q^{\INT(J,\vecd)}.
\]
We may now define the crucial expression $\mathcal{A}(J,\vecd,\CI)$.
While it may appear intimidating, it arises naturally from the period
integrals discussed in the next section.

\begin{nn}
\label{nn:A(J)}
Let  $\CI:=(\CI_{\Gamma,\vecd}) \in \Inv(J)$.
We define $\mathcal{A}(J,\vecd,\CI)$ to be
\begin{equation*}
\sum_{h=1}^{|J|}\frac{1}{h!}\sum_{J_1\sqcup\cdots\sqcup J_h=J}\sum_{\substack{\{k_1(i)\}_{i=1}^{h},~k_1(i)=r(J_i)~(mod~r),\\
\{k_2(i)\}_{i=1}^{h},~k_2(i)=s(J_i)~(mod~s),\\sk_1(i)+rk_2(i)=m(J_i,\vecd)}}
\frac{\Gamma(\frac{1+\sum_{i=1}^hk_1(i)}{r}) \Gamma(\frac{1+\sum_{i=1}^h k_2(i)}{s})}{\Gamma(\frac{1+r(J)}{r}) \Gamma(\frac{1+s(J)}{s})}
\prod_{i=1}^h \CI_{\Gamma_{0, k_1(i), k_2(i), 1, \{(a_i,b_i)\}_{i\in J_j},\vecd|_{J_i}}},
\end{equation*}
where the numbers $r(J),s(J)$ and $m(J,\vecd)$ are as defined in Notation \ref{nn:r(I),s(I),m(I)}, the graph $\Gamma_{0, k_1(i), k_2(i), 1, J_i}$ is as defined in Notation~\ref{balanced rank 2 gammas}, and $\Gamma(\cdot)$ denotes the Gamma function.
\end{nn}

In general for an element $\CI\in \Inv(I)$, $\CI_{\Gamma,\vecd}$
may depend on the individual labels of the internal legs of $\Gamma$
and not just their twists. However, let $\sigma:I\rightarrow I$ be a
twist-preserving bijection, i.e., satisfying $\tw\circ \sigma=\tw$.
We define for any graph $\Gamma$ with
$I(\Gamma) \subseteq I$ the graph $\sigma(\Gamma)$ to be the same as $\Gamma$
but with the internal label $i$ replaced with $\sigma(i)$ for all $i \in
I(\Gamma)$. We  define $\vecd \circ \sigma$ to be the analogous
action on the descendent vectors in $\Z^{I(\Gamma)}_{\ge 0}$.
This induces a permutation transformation
\[
\sigma^*:\Inv(I)\rightarrow\Inv(I)
\]
via $\sigma^*(\CI)_{\Gamma,\vecd}=\CI_{\sigma(\Gamma),\vecd\circ\sigma}$.
We then define
\[
\InvSym(I)\subseteq \Inv(I)
\]
to be the subset of elements fixed under $\sigma^*$ for all
twist-preserving bijections $\sigma:I\rightarrow I$.

For any family of canonical multisections $\ess$ bounded by $I$,
we obtain $\CI^{\ess}\in \Inv(I)$ via
\begin{equation}
\label{eq:CI s def}
\CI^{\ess}_{\Gamma,\vecd}=
\left\langle \prod_{i\in I(\Gamma)}\tau^{(a_i,b_i)}_{d_i}\sigma_1^{k_1(\Gamma)}\sigma_2^{k_2(\Gamma)}\sigma_{12}^{k_{12}(\Gamma)} \right\rangle^{\mathbf{s},o}.
\end{equation}

\begin{definition}
A family $\ess$ of canonical multisections bounded by $I$ is \emph{symmetric}
if, for any $\sigma:I\rightarrow I$ a twist preserving bijection, and
$\Gamma$ with $I(\Gamma)\subseteq I$, $\vecd\in\NN^{I(\Gamma)}$,
the multisections $\ess^\Gamma$ and $\ess^{\sigma(\Gamma)}$ of
$E_{\Gamma}(\vecd)$ and $E_{\sigma(\Gamma)}(\vecd\circ\sigma)$
agree under the natural isomorphism $\Psi_{\sigma}:
\oPM_{\Gamma}\cong \oPM_{\sigma(\Gamma)}$.
\end{definition}

\noindent Note that if $\ess$ is a symmetric canonical family bounded by $I$, then
$\CI^{\ess}\in \InvSym(I)$.

It is elementary and useful to know that symmetric families exist:

\begin{obs}\label{symmetric invariants exists}
Given any family $\widetilde{\mathbf{s}}^*$ of canonical multisections
bounded by a finite set $I$, define a family of multisections as follows. For
$\Gamma\in \INT(J,\vecd)$ for $J\subseteq I$, define the multisection
$\ess^{\Gamma}$ of $E_{\Gamma}(\vecd)$ via
\[
\mathbf{s}^\Gamma := \biguplus_{\sigma \in \MAPS(I,I)} \Psi_{\sigma}^*\widetilde{\mathbf{s}}^{\Gamma}
\]
where $\MAPS(I,I)$ is the set of twist-preserving bijections.
Then $\ess$ is a symmetric family of canonical multisections.
\end{obs}

We can now define several key subsets of $\Inv(I)$ and $\InvSym(I)$.

\begin{definition}
We define
\begin{align*}
\OFJRW(I):={} & \{\CI^{\ess}\,|\, \hbox{$\ess$ is a family of canonical
multisections bounded by $I$}\},\\
\OFJRWSym(I):={} & \{\CI^{\ess}\,|\, \hbox{$\ess$ is a symmetric
family of canonical multisections bounded by $I$}\}.
\end{align*}
\end{definition}

\begin{prop}\label{symmetrised chamber indices}
When $I$ is finite, we have that $\OFJRWSym(I)=\OFJRW(I)\cap\InvSym(I)$.
\end{prop}

\begin{proof}
The forward inclusion is trivial. For the reverse inclusion,
let $\ess$ be a canonical family of multisections with $\CI^{\ess}
\in \OFJRW(I)\cap\InvSym(I)$. Let $\widetilde\ess$ be the symmetric
family constructed in Observation \ref{symmetric invariants exists}.
Then it is immediate
that $\CI^{\ess}=\CI^{\widetilde{\ess}}$, hence the result.
\end{proof}

One of the key results of the paper (Corollary~\ref{chambers are enumerative}) will be a characterization of the sets $\OFJRW(I)$, $\OFJRWSym(I)$. They will coincide with the set of
chamber indices:

\begin{definition}\label{def:chamberIndex}
A \emph{chamber index} is a point
\[
\CI=(\CI_{\Gamma,\vecd})\in\Inv(I)
\]
 where the following hold:
\begin{enumerate}
\item $\CI_{\Gamma,\mathbf{0}} = 1$ if $|I(\Gamma)| = 1$  and $\CI_{\Gamma,\mathbf{0}} = -1$ if $|I(\Gamma)| = 0$ for all $\Gamma \in \INT(\subseteq I)$.
\item For $I'\subseteq I$ with $|I'|=1$, $\vecd = (d)\in \NN^{I'}$,
we have
\[
\mathcal{A}(I',\vecd,\CI) = (-1)^d.
\]
\item For any $I'\subseteq I$ with $|I'|\ge 2$, and $\vecd\in \NN^{I'}$ with $m(I',\vecd)\geq sr(I')+rs(I')$, we have
\[
\mathcal{A}(I',\vecd,\CI)=0.
\]
\end{enumerate}
We define
\begin{align*}
\ChamberIndices(I):={} &
\{\CI\in \Inv(I)\,|\, \hbox{$\CI$ is a chamber index}\}\\
\ChamberIndicesSym(I):={} & \ChamberIndices(I)\cap \InvSym(I).
\end{align*}
We call a chamber index $\nu \in \ChamberIndicesSym(I)$ \emph{symmetric}.
\end{definition}

\section{The $B$-Model}
\label{sec:B model}
In this section, we describe the enumerative theory associated to a Landau-Ginzburg $B$-model, due to Saito and Givental. We use the description given  in the case of the $B$-model of FJRW theory 
by He, Li, Li, Saito, Shen and Webb in \cite{LLSS} and
\cite{HeLiShenWebb}. We describe the $B$-model in terms
of oscillatory integrals, reviewing some of the details from
\cite{GKTdim1}, and show how to interpret these oscillatory
integrals in terms of the $\mathcal{A}(J,\vecd,\CI)$ of
Notation \ref{nn:A(J)}. In doing so, we will see how the
set $\ChamberIndices(I)$ defined in Definition \ref{def:chamberIndex}
plays the role of tuples of mirror $B$-model invariants.
We will then explore the wall-crossing group and show that
$\ChamberIndices(I)$ is a torsor for this group, hence giving a clear
structure to this set.

\subsection{The state space}
\label{subsec:state space}
We consider the Landau-Ginzburg model
\[
\left(X=\C^n, W=\sum_{i=1}^n x_i^{r_i}\right).
\]
 We do not allow for a group of symmetries. We shall quickly review
Saito-Givental theory in this context, sending the reader to
\cite{GKTdim1} for more details. For a much more in-depth
exposition, see \cite{kansas}, Chapter 2 and references therein.

A principal object of study is the twisted de Rham complex
\begin{equation}
\label{eq:standardLG}
(\Omega_X^\bullet, d + \hbar^{-1}dW\wedge - ),
\end{equation}
where $\Omega_X^i$ is the sheaf of algebraic
$i$-forms on $X$ and $\hbar\in \C^*$ is an auxiliary parameter.

We then have the following calculation of the hypercohomology of
the twisted de Rham complex (see \cite{GKTdim1}, Propositon 2.1):

\begin{prop}\label{hypercoh for fermat}
Consider the Landau-Ginzburg model $(X,W)$ in \eqref{eq:standardLG}. Then the hypercohomology group $\mathbb{H}^n(X, (\Omega_X^\bullet, d + \hbar^{-1}dW\wedge -))$ has dimension $\prod_{i=1}^n (r_i-1)$ and is generated by the basis
$$
M = \left\{\prod_{i=1}^n x_i^{a_i} d\mathbf{x} \ \middle| \ 0 \leq a_i \leq r_i - 2\right\},
$$
where $ d\mathbf{x} = dx_1 \wedge \cdots \wedge dx_n$.
\end{prop}

There is a homology group dual to the hypercohomology group  $\mathbb{H}^n(X, (\Omega_X^\bullet,  d+\hbar^{-1}dW\wedge-))$. Consider the relative homology
$$
H_n(X, \op{Re} W/ \hbar \ll 0; \C).
$$
There is a natural perfect pairing
\begin{equation}\begin{aligned}
H_n(X, \op{Re} W/ \hbar \ll 0; \C) \times \mathbb{H}^n(X, (\Omega_X^\bullet,
d+\hbar^{-1} dW\wedge-)) &\longrightarrow \C, \\
(\Xi, \omega) &\longmapsto \int_{\Xi} e^{W / \hbar} \omega.
\end{aligned}\end{equation}

Thus there must be a dual basis for $H_n(X, \op{Re} W/ \hbar \ll 0; \C)$ for any basis of the hypercohomology group.
In fact, in the case of $W=x^r$, Li, Li, Saito and Shen in \cite{LLSS} define a \emph{good basis} for
$H_n(X, \op{Re} W/ \hbar \ll 0; \C)$ as a basis of cycles
\[
\{\Xi_{\mu}\,|\, \text{$0\le \mu \le r-2$}\}
\]
satisfying the condition
\[
\int_{\Xi_{\mu}} x^{\mu'} e^{x^r/\hbar} d\mathbf{x}
=\delta_{\mu\mu'} \, \text{ for all $ 0 \le \mu, \mu' \le r-2$},
\]
where $\delta$ denotes the Kronecker delta function. In
\cite{GKTdim1}, we give a concrete description of this good basis
in the case where $W=x^r$, but only its existence is relevant here.
If instead $W=\sum_i x_i^{r_i}$, and $\Xi^{r_i}_0,\ldots,\Xi^{r_i}_{r_i-2}$
is the good basis for $x_i^{r_i}$, then one can bootstrap from the good basis given in \cite{LLSS}, as is done in \cite{HeLiShenWebb}. Consider the following basis for $H_n(X, \op{Re} W/ \hbar \ll 0; \C)$:
\[
\{\Xi_{\mu}:=
\Xi^{r_1}_{\mu_1}\times\cdots\times\Xi^{r_n}_{\mu_n}\,|\, \mu\in D\}.
\]
 where
\[
D=\{\mu=(\mu_1,\ldots,\mu_n)\,|\, \hbox{$0\le \mu_i\le r_i-2$ for
$1\le i \le n$}\}.
\]
This is a good basis for $W$. We remark that the basis $\Xi_{\mu}$ varies with $\hbar$.
Note that there is a clear bijection between $D$ and $M$ given by mapping $\mu$ to $x^\mu d\mathbf{x}$ where $x^\mu:=\prod_i x_i^{\mu_i}$ for $\mu\in D$. We then
have the identity
\begin{equation}\label{dual bases oscillatory integrals}
\int_{\Xi_{\mu}} x^{\mu'} e^{(x_1^{r_1} + \cdots + x_n^{r_n}) / \hbar} d\mathbf{x} = \delta_{\mu \mu'}.
\end{equation}

The following more general integrals will also be important for us,
see also \cite{GKTdim1}, Lemma 2.4:

\begin{lemma}
\label{lem:int by parts}
For any cycle $\Xi\in H_1(\C,\re x^r/\hbar;\C)$, and
for all $r,n \in \mathbb{N}$, $k \in \NN$, we have:
\begin{align}\label{r spin integration by parts}
\begin{split}
\int_{\Xi}  x^{nr+k} e^{x^r/\hbar}dx = {} &
(-1)^n\hbar^n \left(\prod_{i=1}^n (i -1+ \frac{k+1}{r})\right) \int_{\Xi}  x^{k} e^{x^r/\hbar}dx\\
= {} & (-1)^n\hbar^n \frac{\Gamma\left(n+\frac{k+1}{r}\right)}{
\Gamma\left(\frac{k+1}{r}\right)}\int_{\Xi}  x^{k} e^{x^r/\hbar}dx.
\end{split}
\end{align}
\end{lemma}

\begin{proof}
Integration by parts.
\end{proof}

Now write $W_0=\sum_i x_i^{r_i}$.
Recall that the elements $x^{\mu}$ for $\mu \in D$ are a basis for the Milnor ring viewed as a vector space. We introduce coordinates $y_\mu$ on the universal unfolding of $W_0$, parameterized
by a germ $\mathcal{M}$ of the origin in a vector space with basis $\{
x^{\mu}\,|\,\mu\in D\}$.
The versal deformation $W$ for $W_0$ on $\mathcal{M} \times X$ is then
given by
$$
W =  x_1^{r_1} + \cdots + x_n^{r_n} + \sum_{\mu \in D} y_{\mu} x^{\mu}.
$$

Given $f\in \C[x_1, \dots, x_n][[y_\mu]]$,
we can consider the following collection of oscillatory integrals, each of
which we view as a formal power series:
$$
\int_{\Xi_{\mu}}  e^{W / \hbar}f d\mathbf{x}  = \sum_{j=-\infty}^\infty \varphi_{\mu, j}(\mathbf{y}) \hbar^{-j}
$$
where $\varphi_{\mu,j}(\mathbf{y}) \in \C [[y_\mu]]$.

The following is an oversimplification of Saito's theory of primitive
forms, but will be sufficient for our purposes:

\begin{definition}
If $\varphi_{\mu,j}\equiv 0$ for all $j<0$ and
$\varphi_{\mu,0}=\delta_{\mu\mathbf{0}}$, then we say
that $fd\mathbf{x}$ is a \emph{primitive form}. Further, in this case,
$\varphi_{\mu,1}$ form a set of coordinates on the universal unfolding
$\mathcal{M}$ called \emph{flat coordinates}.
\end{definition}

For more details and explicit examples in the one-dimensional case,
see \cite{GKTdim1}.

\begin{nn}
We typically will use the variables $t_\mu := \varphi_{\mu, 1}$ for the flat coordinates.
\end{nn}

\begin{rmk}
In recent papers proving Landau-Ginzburg mirror symmetry
\cite{LLSS, HeLiShenWebb}, the authors use
Saito's general framework for constructing Frobenius manifold structures
on the universal unfoldings of potentials
(see \cite{Sai81, Sai83} or Section III.8 of \cite{Manin}) to construct
the $B$-model $W$-spin Frobenius manifold. Our definition of flat
coordinates coincides with the definition given in
Equation (13) in \S2.2.2 of \cite{HeLiShenWebb}, with somewhat
different notation. For explicit examples in the one-dimensional case,
see \cite{GKTdim1}.
\end{rmk}

\subsection{Period integrals in the rank $2$ case}

In the prequel \cite{GKTdim1}, we consider period integrals in the $r$-spin case.
Here we  consider period integrals in the rank two case. We will first
consider the not-necessarily symmetric case.

\begin{definition}
Fix a finite
marking set for internal points $I\subseteq\Omega$, and define the ring
\[
A_I:=\Q[\{u_{i,j}\,|\, i\in I, j\in \Z_{\ge 0}\}]/\langle u_{i,j}u_{i,j'}
\,|\,i\in I, j,j'\in \Z_{\ge 0}\rangle.
\]
Let $\CI=(\CI_{\Gamma,\vecd})\in \Inv(I)$.  Then we define the potential with
descendents associated to $\CI$ to be
\[
W^{\CI}:=
\sum_{\substack{J\subseteq I,\, \vecd\in \Z_{\ge 0}^J\\
\,\,\Gamma\in
\INT(J, \vecd)}} (-1)^{|J|-1}
\CI_{\Gamma,\vecd}
x^{k_1(\Gamma)}y^{k_2(\Gamma)}\prod_{i\in J} u_{i,d_i}\in A_I[[x,y]].
\]
Given a canonical family of multisections $\mathbf{s}$ bounded by $I$,
we define
\[
W^{\mathbf{s}}:=W^{\CI^{\ess}},
\]
where $\CI^{\ess}$ is defined in \eqref{eq:CI s def}.
\end{definition}

\begin{rmk}
\label{rmk:formal potential}
We remark that for a fixed $k_1, k_2$, it follows
from Proposition \ref{prop:balanced}(b) and the definition of $m(J,\vecd)$
that there are only a finite number of $J\subseteq I$ and $\vecd\in \NN^J$
such that the set $\{\Gamma\in\INT(J,\vecd)\,|\,k_i(\Gamma)=k_i, i=1,2\}$
is finite. Hence the coefficient of
$x^{k_1}y^{k_2}$ is a polynomial, so the potential is indeed
an element of $A_I[[x,y]]$.
\end{rmk}

\begin{definition}
\label{def:d(J)}
Let $J\subseteq I$ be a finite index set, $\vecd\in\NN^J$. Then we
define
\[
d(J,\vecd):={sr(J)+rs(J)-m(J,\vecd)\over rs}-1.
\]
\end{definition}

We now have a direct calculation of the period integral, which motivates
the definition of $\mathcal{A}(J,\vecd,\CI)$ from Notation~\ref{nn:A(J)}:

\begin{thm}
\label{thm:period integral}
If $W^{\CI}=x^r+y^s\mod \langle u_{i,j}\,|\,i\in I, j\in \Z_{\ge 0}
\rangle$, then for $0\le a \le r-2$, $0\le b \le s-2$,
\[
\int_{\Xi_{(a,b)}} e^{W^{\CI}/\hbar}d\mathbf{x}
=\delta_{a,0}\delta_{b,0}+
\sum_{\substack{J\subseteq I\\J\not=\varnothing}}\sum_{\vecd\in \NN^J}
(-1)^{|J|}(-\hbar)^{-d(J,\vecd)-2}
\mathcal{A}(J,\vecd,\CI)\delta_{r(J), a}\delta_{s(J), b}
\prod_{j\in J} u_{j,d_j}.
\]
\end{thm}

\noindent We remark that this equality holds when both sides are viewed as Laurent series in the $\hbar$ variable with coefficients in $A_I$.

\begin{proof}
Let $J\subseteq I$ and $\vecd\in \NN^J$.
Using Notation \ref{nn:r(I),s(I),m(I)}, Proposition \ref{prop:balanced} shows
that the graph $\Gamma = \Gamma_{0, k_1, k_2, 1, \{(a_j, b_j)\}_{j\in J}}$
is balanced with respect to $\vecd$ if and only if
\begin{equation}\label{original balance}
k_1 \equiv r(J) \pmod r, \quad k_2 \equiv s(J) \pmod s, \text{ and } sk_1 + rk_2 = m(J, \vecd).
\end{equation}
By definition of $r(J),s(J)$, there exists nonnegative integers $\ell_1(J), \ell_2(J)
\in \Z$ so that
\[
\sum_{j\in J} a_j = r(J) + \ell_{1}(J)r,\quad \sum_{j \in J} b_j= s(J) +
\ell_{2}(J)s.
\]

We may write
\begin{equation}\label{Exponentiating the perturbation 2 dim}
\int_{\Xi_{(a,b)}}  e^{W^{\CI} / \hbar}d\mathbf{x}  = \int_{\Xi_{(a,b)}}  \left(\sum_{h\geq 0} \frac{ (W^{\CI} -x^r-y^s)^h}{h!\hbar^h}\right) e^{(x^r+y^s) / \hbar} d\mathbf{x}.
\end{equation}

Noting by assumption that
$W^{\CI}-x^r-y^s \equiv 0 \mod \langle u_{i,j}\rangle$,
 we may now expand the summation.
We can see that if $J\not=\varnothing$, then
the coefficient of $u_{J,\vecd} := \prod_{j \in J} u_{j,d_j}$
in \eqref{Exponentiating the perturbation 2 dim} is
\begin{align}\label{lambda uj}
\begin{split}
\Lambda_{J,\vecd}
:= {} & \int_{\Xi_{(a,b)}} \sum_{h=1}^{|J|} \frac{1}{h!} \sum_{J_1\sqcup
\cdots\sqcup J_h=J}
\quad\sum_{\substack{\{k_1(j)\}_{j=1}^{h},~k_1(j)=r(J_j)~(mod~r),\\
\{k_2(j)\}_{j=1}^{h},~k_2(j)=s(J_j)~(mod~s),\\sk_1(j)+rk_2(j)=m(J_j,\vecd)}}\\
&\Bigg(
 x^{\sum_{j} k_1(j)}y^{\sum_j k_2(j)}
\hbar^{-h} \prod_{j=1}^h (-1)^{|J_j|-1}
\CI_{\Gamma_{0,k_1(j),k_2(j),1,\{(a_i,b_i)\}_{i\in J_j}},\vecd|_{J_j}}
\Bigg)e^{(x^r+y^s)/\hbar}d\mathbf{x}.
\end{split}
\end{align}

Performing integration by parts using
Lemma \ref{lem:int by parts} and using \eqref{dual bases oscillatory integrals},
we can see that
the integral is $0$ unless $(a,b)=(r(J),s(J))$. As the monomial in
$x,y$ occurring in the integrand is $x^{\sum_j k_1(j)}y^{\sum_j k_2(j)}$,
the integration by parts produces a power $\hbar^{n_1+n_2}$, where
\[
\sum_j k_1(j)=r(J)+n_1r, \quad \sum_j k_2(j)=s(J)+n_2s.
\]
Putting this together with Lemma \ref{lem:int by parts} and
Equation~\eqref{dual bases oscillatory integrals}, the coefficient
is given by
\begin{align}
\label{eq:yet another integral expression}
\begin{split}
\Lambda_{J,\vecd}  = {} & \sum_{h=1}^{|J|} \frac{1}{h!}
\sum_{J_1\sqcup\cdots\sqcup J_h=J}\quad
\sum_{\substack{\{k_1(j)\}_{j=1}^{h},~k_1(j)=r(J_j)~(mod~r),\\
\{k_2(j)\}_{j=1}^{h},~k_2(j)=s(J_j)~(mod~s),\\sk_1(j)+rk_2(j)=m(J_j,\vecd)}}
(-1)^{|J|-h+n_1+n_2}
\\
&\Bigg(\frac{\Gamma(\tfrac{1+ \sum_{j} k_1(j)}{r}) \Gamma(\tfrac{1 + \sum_j k_2(j)}{s})}{\Gamma(\tfrac{1+r(J)}{r}) \Gamma(\tfrac{1+s(J)}{s})} \hbar^{n_1 + n_2 - h}
\prod_{j=1}^h
\CI_{\Gamma_{0,k_1(j),k_2(j),1,\{(a_i,b_i)\}_{i\in J_j}},\vecd|_{J_j}}
\Bigg)
\delta_{r(J), a}\delta_{s(J), b}
\end{split}
\end{align}

Using \eqref{eq:r(I),s(I),m(I)}, note that
\[
\sum_{j=1}^h m(J_j, \vecd) = (h-1) rs + m(J, \vecd).
\]
From $sk_1(j) + rk_2(j) = m(J_j, \vecd)$, we obtain
\begin{align*}
n_1+n_2 = & \frac{\sum_j k_1(j) - r(J)}{r} + \frac{\sum_j k_2(j)-s(J)}{s} = \frac{(\sum_j sk_1(j) + rk_2(j)) - sr(J) - rs(J)  } { rs}\\
= & \frac{\sum_j m(J_j,\vecd) -sr(J)-rs(J)}{rs} = \frac{(h-1)rs + m(J, \vecd)  -sr(J)-rs(J)}{rs}.
\end{align*}
Hence
\begin{equation}\label{the hbar numbers in terms of the internals}
(n_1+n_2 - h + 1)rs =  m(J, \vecd) -sr(J) - rs(J).
\end{equation}
In particular, $n_1+n_2-h$ only depends on $J$, and we may
rewrite \eqref{eq:yet another integral expression} as
\begin{equation}\label{relating lambda uj with chamber index}
\Lambda_{J, \vecd}=(-1)^{|J|} (-\hbar)^{(m(J,\vecd)-sr(J)-rs(J)-rs)/rs}
\mathcal{A}(J,\vecd,\CI)\delta_{r(J), a}\delta_{s(J), b}.
\end{equation}
This agrees with the stated result via the definition of $d(J,\vecd)$.
On the other hand, the contribution to the integral from the case
$J=\varnothing$ follows directly  the definition of a good basis to be
$\delta_{(a,b),(0,0)}$. This gives the desired result.
\end{proof}

We next consider the symmetric case.

\begin{definition}
Let $I\subseteq \N\subseteq\Omega$ be a finite set of labels.
Set
\[
A_{I,\mathrm{sym}}:=\Q[t_{\alpha,\beta,d}\,|\,0\le \alpha\le r-2, 0\le
\beta \le s-2,
d\in \Z_{\ge 0}]/\mathrm{Ideal}(I),
\]
where $\mathrm{Ideal}(I)$ is the ideal generated by monomials of the form
\[
\prod_d t_{\alpha,\beta,d}^{n_d}
\]
such that $\sum n_d > |(\tw|_I)^{-1}(\alpha,\beta)|$.
\end{definition}

\begin{definition}\label{AAD}
For each $l\ge 0$, let $\mathcal{A}_l$ denote the collection of
multi-sets of size $l$ of the form
\[
\{(\alpha_1,\beta_1,d_1),\ldots, (\alpha_l,\beta_l,d_l)\}
\]
with $0\le \alpha_i\le r-2$, $0\le \beta_i\le s-2$, and $d_i\ge 0$.
For $A\in \mathcal{A}_l$,
denote by $\Aut(A)$ the group of permutations $\sigma:[l]\rightarrow [l]$
such that $(\alpha_i,\beta_i,d_i)=(\alpha_{\sigma(i)},
\beta_{\sigma(i)},d_{\sigma(i)})$ for all $i\in [l]$.

For $A=\{(\alpha_i,\beta_i,d_i)\}\in \mathcal{A}_l$ such that there exists 
distinct $j_1,\ldots,j_l\in I$ with $\tw(j_i)=(\alpha_i,\beta_i)$, we
write for $\CI\in\InvSym(I)$
\[
\mathcal{A}(A,\CI) := \mathcal{A}(J,\vecd,\CI),\quad\quad d(A):=d(J,\vecd),
\]
where $J=\{j_1,\ldots,j_l\}$ and $\vecd=(d_i)_{1\le i \le l}$. Note
that since $\CI$ is symmetric, the first quantity is independent of 
the choice of $j_i$, while the second is independent by definition.
\end{definition}

\begin{definition}
Let $\CI=(\CI_{\Gamma,\vecd})\in \InvSym(I)$.
We define the potential with descendents
\begin{equation}\label{def:flat family in 2 dim}
W^{\CI,\mathrm{sym}} = \sum_{k_1,k_2 \geq 0, l \geq 0} \, \,
\sum_{A= \{(\alpha_i, \beta_i,d_i)\} \in \mathcal{A}_l}
(-1)^{l-1} \frac{
\CI_{\Gamma_{0,k_1,k_2,1,\{(\alpha_i,\beta_i)\}},(d_i)}
}{|\Aut(A)|}
(\textstyle{\prod}_{i=1}^l t_{\alpha_i,\beta_i,d_i} ) x^{k_1}y^{k_2}
\in A_{I,\mathrm{sym}}[[x,y]].
\end{equation}
Given a symmetric canonical family of multisections $\mathbf{s}$,
we define
\[
W^{\mathbf{s},\mathrm{sym}}=W^{\CI^{\ess},\mathrm{sym}}.
\]
\end{definition}

Again, similarly as in Remark \ref{rmk:formal potential}, there are only a finite
number of non-zero terms containing a fixed monomial $x^{k_1}y^{k_2}$, so that indeed $W^{\CI,\mathrm{sym}} 
\in A_{I,\mathrm{sym}}[[x,y]].$

\begin{definition}
Given a finite subset $I\subseteq\N$ we define a map
\begin{align*}
\psi_I:A_{I,\mathrm{sym}}\rightarrow& A_I\\
t_{\alpha,\beta,d} \mapsto & \sum_{i\in I:(a_i,b_i)=(\alpha,\beta)} u_{i,d}.
\end{align*}
This also induces a map
\[
\psi_I:A_{I,\mathrm{sym}}[[x,y]] \rightarrow A_I[[x,y]].
\]
\end{definition}

\begin{lemma}
The homomorphism $\psi_I$ is well-defined and injective. Further,
let $A\in\mathcal{A}_l$ with $A=\{(\alpha_i,\beta_i,d_i)\}$. Then
\begin{equation}
\label{eq:psi I formula}
\psi_I\left({\prod_{i=1}^l t_{\alpha_i,\beta_i,d_i}\over |\textup{Aut}(A)|}\right)
=
\sum_{\substack{J\subseteq I,\\ \vecd'\in S}} \prod_{j\in J} u_{j,d'_j},
\end{equation}
where the sum is over all subsets $J\subseteq I$ and $\vecd'\in S$ where $S \subseteq
\NN^J$ is the subset of descendent vectors $\vecd' := (d_i')$ for which there exists a bijection $\phi:J\rightarrow [l]$ with
$(a_j,b_j)=(\alpha_{\phi(j)},\beta_{\phi(j)})$ and $d'_j=d_{\phi(j)}$.
\end{lemma}

\begin{proof}
We begin with the second statement by viewing $\psi_I$ as
a clearly well-defined map $\Q[\{t_{\alpha,\beta,d}\}]\rightarrow A_I$.
We have
\begin{align*}
\psi_I\left({\prod_{i=1}^l t_{\alpha_i,\beta_i,d_i}\over |\Aut(A)|}\right)
= {} & |\Aut(A)|^{-1} \prod_{i=1}^l\sum_{j\in I:(a_j,b_j)=(\alpha_i,\beta_i)}
u_{j,d_i}\\
= {} & |\Aut(A)|^{-1}\sum_{\substack{J\subseteq I\\
\phi:J\rightarrow [l]
}}
\prod_{j\in J} u_{j,d_{\phi(j)}}
\end{align*}
where we sum over all bijections $\phi:J\rightarrow [l]$ such that
for all $j\in J$, $(a_j,b_j)=(\alpha_{\phi(j)}, \beta_{\phi(j)})$; here
we use the fact that $u_{i,j}u_{i,j'}=0$ in $A_I$ for all
$i\in I, j,j'\in \Z_{>0}$. Given such a bijection $\phi$, we may
define $\vecd'\in \NN^J$ by $d'_j=d_{\phi(j)}$. Thus we may rewrite
the above expression as
\[
|\Aut(A)|^{-1}\sum_{\substack{J\subseteq I, \vecd'\in\NN^J\\
\phi:J\rightarrow [l]
}}
\prod_{j\in J} u_{j,d'_{j}},
\]
where now the sum is over bijections $\phi$ satisfying
$(a_j,b_j)=(\alpha_{\phi(j)},\beta_{\phi(j)})$, $d'_j=d_{\phi(j)}$. Now
assuming such a bijection exists, there are precisely $|\Aut(A)|$
such bijections. Thus summing over all such bijections gives the desired
formula \eqref{eq:psi I formula}.

To finish the proof, we just need to show $\ker\psi_I=\mathrm{Ideal}(I)
\subseteq \Q[\{t_{\alpha,\beta,d}\}]$. Note that if
$M=\prod_{i=1}^l t_{\alpha_i,\beta_i,d_i}$, then $\psi_I(M)=0$ if and only
if the sum in the right-hand side of ~\eqref{eq:psi I formula} is empty, i.e., there is no
$J\subseteq I$ and bijection
$\phi:J\rightarrow [l]$ with $(a_j,b_j)=(\alpha_{\phi(j)},
\beta_{\phi(j)})$. However, this is the case if and only if there
is some $i\in [l]$ such that $|\{j\in [l]\,|\, (\alpha_i,\beta_i)=(\alpha_j,\beta_j)\}|
> |(\tw|_I)^{-1}(\alpha_i,\beta_i)|$. Put another way, this occurs
if $I$ does not contain enough elements $j$ with $\tw(j)=(\alpha_i,\beta_i)$.
However, by definition of $\mathrm{Ideal}(I)$, this occurs if and only if
$M\in \mathrm{Ideal}(I)$. In particular $\mathrm{Ideal}(I)
\subseteq\ker {\psi_I}$.

We now note that given two distinct monomials $M_1, M_2$ in the
$t_{\alpha_i,\beta_i,d_i}$, the sets of monomials appearing on the
right-hand side of \eqref{eq:psi I formula} are disjoint. Thus
if $f=\sum a_i M_i\in A_{I,\mathrm{sym}}$ with $a_i \in \Q$, $M_i$ monomials,
and $\psi_I(f)=0$,
we must have $\psi_I(M_i)=0$ for each $i$. But then each $M_i$, hence $f$,
lies in $\mathrm{Ideal}(I)$. This proves $\ker\psi_I=\mathrm{Ideal}(I)$.
\end{proof}

The following two corollaries are then immediate consequences of the
above lemma, the formulas for $W^{\CI,\mathrm{sym}}$,
$W^{\CI}$, and the formula of Theorem \ref{thm:period integral}.

\begin{cor}
\label{cor:sym to asym}
Let $\CI=(\CI_{\Gamma,\vecd})\in \InvSym(I)$. Then
\[
\psi_I(W^{\CI,\mathrm{sym}})=W^{\CI}.
\]
\end{cor}

\begin{cor}
\label{cor:symmetric integal}
Let $\CI$ be as in Corollary \ref{cor:sym to asym}.
If 
\[
W^{\CI,\mathrm{sym}}=x^r+y^s\mod
\langle t_{\alpha,\beta,d}\,|\,0\le \alpha \le r-2, 0\le \beta\le s-2, d\in \Z_{\ge 0} \rangle,
\]
then
\begin{align*}
&\int_{\Xi_{(a,b)}} e^{W^{\CI,\mathrm{sym}}/\hbar}dx\wedge dy \\
= {} & \delta_{a,0}\delta_{b,0} +
\sum_{l\ge 1}
\sum_{A=\{(\alpha_i,\beta_i,d_i)\}\in \mathcal{A}_l}(-1)^l
(-\hbar)^{-d(A)-2}
\mathcal{A}(A,\CI)\delta_{r(A), a}\delta_{s(A), b}
\left({\prod_{j=1}^l t_{\alpha_j,\beta_j,d_j}\over |\mathrm{Aut}(A)|}\right).
\end{align*}
\end{cor}

We now can see immediately that  a symmetric chamber index $\nu$ from Definition~\ref{def:chamberIndex} provides flat coordinates for the Landau-Ginzburg model $W$. This will be used in proving the last paragraph of Theorem~\ref{thm:intro mirror theorem}.

\begin{cor}
\label{cor:flat coordinates}
Let $\CI$ be a symmetric chamber index bounded by $I$. Then we have
\[
\int_{\Xi_{(a,b)}} e^{W^{\CI,\mathrm{sym}}/\hbar}dx\wedge dy =
\delta_{(a,b),(0,0)} +
\hbar^{-1} t_{a,b,0} + O(\hbar^{-2})
\mod \langle t_{(\alpha,\beta,d)}\,|\, 0\le \alpha\le r-2, 0\le
\beta\le s-2, d>0\rangle.
\]
In particular, $dx\wedge dy$ is a primitive form and
the $t_{a,b,0}$ are flat coordinates.
\end{cor}

\begin{proof}
We first note that by Condition (1) of Definition \ref{def:chamberIndex}
that the hypothesis on $W^{\CI,\mathrm{sym}}$ of
Corollary \ref{cor:symmetric integal} holds, and
hence we may use the formula for the integral given there. 
Further, by Condition (3)
of Definition \ref{def:chamberIndex}, there is no contribution
of the form $(-\hbar)^{-d(A)-2}\mathcal{A}(A,\CI)$ with 
$d(A)<0$ and $|A|\ge 2$. On the other hand, if $A=\{(a,b,0)\}$, 
then it follows from Condition (1) of Definition \ref{def:chamberIndex}
and the formula for $\mathcal{A}(A,\CI)$ that the coefficient of
$\hbar^{-1}$ is the stated one.
\end{proof}

\begin{ex}
Let $W = x^4 + y^5$. Working modulo any cubic term in the $t_{(\alpha, \beta, 0)}$ variables and any $t_{(\alpha, \beta, d)}$ with $d > 0$, we can write a deformation $W^{\CI, \mathrm{sym}}$. To compress notation, we write $t_{ij} := t_{(i,j,0)}$. From the constraints given in Definition~\ref{def:chamberIndex}, we can see that

\begin{align*}
   W^{\CI, \mathrm{sym}} &= x^4 + y^5 + t_{23} x^2y^3 + t_{22} x^2y^2 + t_{21} x^2y + t_{20}x^2 + t_{13} xy^3 + t_{12} xy^2 + t_{11} xy + t_{10}x \\ &\qquad + t_{03} y^3 + t_{02} y^2 + t_{01} y + t_{00}+ {c_x} t_{22}t_{23}x^4 + {c_y} t_{22}t_{23}y^5 + d_x t_{23}^2 x^4y + d_y t_{23}^2 y^6 \\
   &\qquad + \tfrac14 t_{23}t_{21} y^4 + \tfrac14 t_{23}t_{20} y^3 + \tfrac25 t_{23}t_{13}x^3y + \tfrac15 t_{23}t_{12}x^3 + \tfrac25 t_{23}t_{03} x^2y + \tfrac15 t_{23}t_{02} x^2 \\
   &\qquad + \tfrac18 t_{22}^2 y^4 + \tfrac14 t_{22}t_{21}y^3 + \tfrac14 t_{22}t_{20} y^2 + \tfrac15 t_{22}t_{13}x^3 + \tfrac15 t_{22}t_{03} x^2  + \tfrac18 t_{21}^2 y^2 + \tfrac14 t_{21}t_{20} y \\
   &\qquad + \tfrac18 t_{20}^2 + \tfrac15 t_{13}^2 x^2y + \tfrac15 t_{13}t_{12} x^2 + \tfrac25t_{13}t_{03}xy + \tfrac15 t_{13}t_{02} x + \tfrac15 t_{12}t_{03} x  + \tfrac{2}{5}t_{03}^2 y + \tfrac15t_{03}t_{02},
\end{align*}
subject to the constraints 
\begin{equation}\label{eq: W12 constraints}
\tfrac{c_x}{4} + \tfrac{c_y}{5} = \tfrac{1}{20} \text{ and } \tfrac{d_x}{4} + \tfrac{2d_y}{5} = \tfrac{1}{20}; \quad \text{ with $c_x, c_y, d_x, d_y \in \Q$}.
\end{equation}
Note that any choice of $\Q$-solutions to~\eqref{eq: W12 constraints} will give a chamber index. 

Comparing with \cite[Appendix A, Type $W_{12}$]{LLSS}, the authors start with a versal deformation, and then find a primitive form that in our notation is of the form
$$
\zeta = dx\wedge dy (1 - \tfrac{1}{20}t_{22}t_{23} - \tfrac{1}{20}t_{23}^2 y + \cdots).
$$
One can apply Lemma~\ref{r spin integration by parts} and the constraint~\eqref{eq: W12 constraints} to see that
\begin{equation*}\begin{aligned}
\int_{\Xi_{(a,b)}} &e^{W^{\CI,\mathrm{sym}}/\hbar}dx\wedge dy= \\& \int_{\Xi_{(a,b)}} e^{(W^{\CI,\mathrm{sym}}- ({c_x} t_{22}t_{23}x^4 + {c_y} t_{22}t_{23}y^5 + d_x t_{23}^2 x^4y + d_y t_{23}^2 y^6))/\hbar}( 1 - \tfrac{1}{20}t_{22}t_{23} - \tfrac{1}{20}t_{23}^3 y + \ldots) dx\wedge dy
\end{aligned}\end{equation*}
since 
$$
e^{({c_x} t_{22}t_{23}x^4 + {c_y} t_{22}t_{23}y^5 + d_x t_{23}^2 x^4y + d_y t_{23}^2 y^6)/\hbar} dx\wedge dy = (1 - \tfrac{1}{20}t_{22}t_{23} - \tfrac{1}{20}t_{23}^3 y + \ldots) dx\wedge dy
$$
in the hypercohomology $\mathbb{H}^n(X, (\Omega_X^\bullet,
d+\hbar^{-1} dW\wedge-))$
We note that the potential
$$
W^{\CI,\mathrm{sym}}- ({c_x} t_{22}t_{23}x^4 + {c_y} t_{22}t_{23}y^5 + d_x t_{23}^2 x^4y + d_y t_{23}^2 y^6)
$$
would be equal to the \emph{flat} deformation (up to order 2) found by the approach in \cite{LLSS} once viewed in $\mathbb{H}^n(X, (\Omega_X^\bullet,
d+\hbar^{-1} dW\wedge-))$. Essentially, our choice of deformation of $W$ has yielded a less complicated primitive form. 
\end{ex}

\subsection{The Landau-Ginzburg wall-crossing Group}
\label{subsec:LG wall crossing}
To motivate the discussion of this subsection, let us consider
a rank $n$ Fermat potential $W=\sum_{i=1}^n x_i^{r_i}$, and let
$A$ be some coefficient ring parameterizing a perturbation
of $W$. Fix a maximal ideal $\mathfrak{m}\subseteq A$ and
let $W'\in A[[x_1,\ldots,x_n]]$ be such that $W\equiv W'\mod \mathfrak{m}$.
We are interested in changes to $W'$ which do not affect the
oscillatory integrals $\int e^{W'/\hbar} dx_1
\wedge\cdots\wedge dx_n$. Indeed, take an automorphism
$\psi \in \text{Aut}_{A} (A[[x_1,\ldots, x_n]])$
which is the identity modulo the ideal $\mathfrak{m}$ and satisfies the relation
\[
\psi^*(dx_1 \wedge \cdots \wedge dx_n)
=
dx_1 \wedge \cdots \wedge dx_n.
\]
Then
\begin{equation}
\label{eq:invariance of osc int}
\int_{\Xi} e^{W'/\hbar} dx_1\wedge\cdots\wedge dx_n
=\int_{\Xi} \psi^*\left(e^{W'/\hbar} dx_1\wedge\cdots\wedge dx_n\right)
=\int_{\Xi} e^{\psi(W')/\hbar} dx_1\wedge\cdots\wedge dx_n.
\end{equation}
Note that we have not changed the cycle of integration $\Xi$ and in
fact should use $\psi^{-1}(\Xi)$ in the second and third integrals,
but because $\psi$ is just an infinitesimal extension of the identity,
this has no effect on the integral.
Thus we may change $W'$ by applying $\psi$ without changing
the integrals.

We now study the Lie algebra of the group of such $\psi$. Some care
is necessary to handle the completion, so let
\[
R_k:=A[x_1,\ldots,x_n]/(x_1,\ldots,x_n)^{k+1},
\]
so that $A[[x_1,\ldots,x_n]]$ is the inverse limit of the $R_k$.

Consider the
module $\Theta_k$ of derivations of
$R_k$ over $A$:
\[
\Theta_k:=\bigoplus_{i=1}^n R_k \partial_{x_i}.
\]
This module comes equipped with the usual Lie bracket, i.e.,
\begin{equation}
\label{eq:Lie bracket}
\left[\left(\prod_i x_i^{n_i} \right)\partial_{x_j},
\left(\prod_i x_i^{n_i'}\right) \partial_{x_k}\right]
=
\left(n_j' \prod_{i=1}^n x_i^{n_i+n_i'-\delta_{ij}}\right)\partial_{x_k}
- \left(n_k \prod_{i=1}^n x_i^{n_i+n_i'-\delta_{ik}}\right)\partial_{x_j}.
\end{equation}
We next define a linear subspace of $\Theta_k$,
\begin{align}
\begin{split}
\mathfrak{g}_{A,k}:= {} &
\bigoplus_{(a_1, \ldots,a_n)} \mathfrak{g}_{A,(a_1, \ldots, a_n)}\\
:= {} &
\bigoplus_{\substack{(a_1, \ldots, a_n) \in \NN^n \setminus \{0\}\\
\sum a_i \le k}} \sum_{i=1}^n \mathfrak{m} \cdot (x_1^{a_1}\cdots x_n^{a_n} ((a_i+1) x_1 \partial_{x_1} - (a_1+1) x_i\partial_{x_i})).
\end{split}
\end{align}
To show $\mathfrak{g}_{A,k}$ is closed under Lie bracket,
one tediously calculates
\begin{align}
\label{eq:tedious commutator}
\begin{split}
&\left[\left(\prod_{i=1}^n x_i^{a_i}\right)((a_j+1)x_1\partial_{x_1}-(a_1+1)
x_j\partial_{x_j}),
\left(\prod_{i=1}^n x_i^{b_i}\right)((b_\ell+1)x_1\partial_{x_1}-(b_1+1)
x_\ell\partial_{x_\ell})\right]\\
= {} &
\left(\prod_{i=1}^n x_i^{a_i+b_i}\right)
\big((a_j+b_j+1)\alpha+(a_\ell+b_\ell+1)\beta\big)x_1\partial_{x_1}
-(a_1+b_1+1)(\alpha x_j\partial_{x_j}+\beta x_\ell\partial_{x_\ell}),
\end{split}
\end{align}
with
\[
\alpha= {-(b_\ell+1)(a_1+1)a_1+(b_1+1)(a_1+1)a_\ell \over a_1+b_1+1},
\quad
\beta = {(a_j+1)(b_1+1)b_1-(a_1+1)(b_1+1)b_j \over a_1+b_1+1}.
\]
Further,
if $v\in \mathfrak{g}_{A,k}$,
then the Lie derivative ${\mathcal L}_v(dx_1\wedge\cdots\wedge dx_n)=0$.
In detail, since $dx_1\wedge\cdots \wedge dx_n$ is a closed form, we have
\[
{\mathcal L}_v(dx_1\wedge\cdots\wedge dx_n)=
d(\iota(v)(dx_1\wedge\cdots\wedge dx_n)).
\]
Taking $v=x_1^{a_1}\cdots x_n^{a_n}((a_i+1)x_1\partial_{x_1}
-(a_1+1)x_i\partial_{x_i})$, we see this latter quantity is
\begin{align*}
&d\left(x_1^{a_1}\cdots x_n^{a_n}((a_i+1)x_1 dx_2\wedge\cdots \wedge dx_n-(-1)^{i-1}(a_1+1)x_idx_1\wedge \cdots\wedge
\widehat{dx_i}\wedge\cdots\wedge dx_n\right)\\
= {} & x_1^{a_1}\cdots x_n^{a_n}
((a_1+1)(a_i+1)-(a_i+1)(a_1+1))dx_1\wedge\cdots \wedge dx_n\\
= {} & 0.
\end{align*}
Thus vector fields in $\mathfrak{g}_{A,k}$ indeed preserve $dx_1\wedge
\cdots \wedge dx_n$.

We note that this is not the algebra of all vector fields preserving
$dx_1\wedge\cdots \wedge dx_n$. Indeed, vector fields of the form
$f(x_1,\ldots,\widehat{x_i},\ldots,x_n)\partial_{x_i}$ also preserve
$dx_1\wedge\cdots \wedge dx_n$. But such a vector field will not play a role in our theory.
This is a non-trivial point, which in the rank two case is in fact
enforced by Condition (3)(a) of the definition of strongly positive,
Definition \ref{def: strongly positive}.

\begin{definition}\label{Wall Crossing Group}
As $\mathfrak{g}_{A,k}$ is a nilpotent Lie algebra, we may define a group
$G_{A,k}:=\exp(\mathfrak{g}_{A,k})$. This is a group
whose underlying set is $\mathfrak{g}_{A,k}$ but where
multiplication is given by the Baker-Campbell-Hausdorff formula.
We may then define $G_A$ as the pro-nilpotent Lie group given as
the inverse limit of the $G_{A,k}$, identified as a set with
$\mathfrak{g}_A$, the inverse limit of the $\mathfrak{g}_{A,k}$.
We call $G_A$ the \emph{Landau-Ginzburg wall-crossing group of rank $n$
(over $A$)}.
\end{definition}

We note that $G_A$ acts by automorphisms of
$A[[x_1,\ldots,x_n]]$ which preserve $dx_1\wedge\cdots
\wedge dx_n$ via the exponential: i.e., if $v\in \mathfrak{g}_A$, then
$\exp(v)$ acts on elements of $A[[x_1,\ldots,x_n]]$ by
\begin{equation}
\label{eq:action}
f\mapsto \sum_{n=0}^{\infty} {v^n(f)\over n!},
\end{equation}
where $v^n$ denotes differentiating with respect to the vector field
$n$ times.

\begin{rmk}
A similar group of automorphisms of $A[x_1^{\pm 1},\ldots, x_n^{\pm 1}]$
preserving $d\log x_1\wedge\cdots\wedge d\log x_n$
or a symplectic holomorphic form
has been studied extensively in the literature and applied
in many different settings, with initial study in
the two-dimensional case by Kontsevich and Soibelman in \cite{KontSoib} and
in higher dimension by Gross and Siebert in \cite{GSAnnals}. The particular use
here follows closely the application of the wall-crossing group
by Gross in \cite{GrossP2}.
\end{rmk}

Restricting now to the two-dimensional case, we may write
\begin{equation}
\mathfrak{g}_{A,k} = \bigoplus_{\substack{(a,b) \in \NN^2 \setminus \{(0,0)\}
\\ a+b\le k}} \mathfrak{m} \cdot (x^ay^b ((b+1)x \partial_x - (a+1)y\partial_y)) = \bigoplus_{ \substack{(a,b) \in \NN^2 \setminus \{(0,0)\}
\\ a+b\le k}} \mathfrak{g}_{A,(a,b)}.
\end{equation}
 The action of $G_A$ on $A[[x,y]]$ can be described explicitly through the following computation:

\begin{lem}\label{WC:closedform}
For $(a,b)\in \NN^2\setminus \{(0,0)\}$, consider
\[
v=g x^ay^b((b+1)x\partial_x-(a+1)y\partial_y)\in
\mathfrak{g}_{A,(a,b)},
\]
with $g\in \mathfrak{m} \subseteq A$.
\begin{enumerate}
\item
If $a\not=b$,
then $\exp(v)$ acts on $A[[x,y]]$ via
\begin{equation}
\label{eq:exp v}
\begin{aligned}
x &\longmapsto x(1+(b-a) g x^ay^b)^{(b+1)/(b-a)}\\
y &\longmapsto y(1+(b-a) g x^ay^b)^{(a+1)/(a-b)}.
\end{aligned}\end{equation}
\item If $a=b$, then $\exp(v)$ acts on $A[[x,y]]$ via
\begin{equation}
\label{eq:exp v equal case}
\begin{aligned}
x &\longmapsto x\exp((a+1) g x^ay^b)\\
y &\longmapsto y\exp(-(a+1) g x^ay^b).
\end{aligned}\end{equation}
\end{enumerate}
\end{lem}

\begin{proof}
We first note the calculation of (2) is straightforward, using the
fact that $v$ vanishes on $(xy)^n$ for any $n$.
For (1), note that given the description \eqref{eq:action}
for the action of $\exp(v)$,
we can rewrite the automorphism $\psi= \exp(v)$ on the elements $x,y \in A[[x,y]]$ as
\begin{align*}
\psi(x)= {} & x\left(1+\sum_{n=1}^{\infty} f_{x,n}(x^ay^b)^n\right)
=xf_x\\
\psi(y)= {} & y\left(1+\sum_{n=1}^{\infty} f_{y,n}(x^ay^b)^n\right)
=yf_y
\end{align*}
for some $f_{x,n}, f_{y,n}\in A$ defining $f_x,f_y\in A[[x^ay^b]]$
with $f_x,f_y\equiv 1 \mod \mathfrak{m}$. Further, since $v(x^{a+1}y^{b+1})=0$,
$\psi(x^{a+1}y^{b+1})=x^{a+1}y^{b+1}$, so $f_x^{a+1}=f_y^{-(b+1)}$. Thus by taking roots as formal
power series, there is a unique function $f\in
A[[q]]$ so
that $f_x=f(x^ay^b)^{b+1}$, $f_y=f(x^ay^b)^{-(a+1)}$ and $f\equiv 1
\mod q$.
Now recall that $\psi^*(dx\wedge dy)=dx\wedge dy$. 
Writing $f'$ for the derivative of $f$ with respect to $q$, and
substituting $q=x^ay^b$, we can compute that
\[
dx \wedge dy = \psi^*(dx\wedge dy)=
(f^{b-a}+(a-b)x^ay^bf'f^{b-a-1})dx\wedge dy.
\]
\begin{comment}
\begin{align*}
dx\wedge dy = {} & \psi^*(dx\wedge dy)\\
= {} &
d(x f^{b+1})\wedge
d(y f^{-(a+1)})\\
= {} &\big((f^{b+1}+a(b+1)xf^bf'x^{a-1}y^b)dx
+((b+1)bxf^bf'x^ay^{b-1})dy\big)\\
& \wedge
\big((-(a+1)ayf^{-a-2}f'x^{a-1}y^b)dx
+(f^{-(a+1)}-(a+1)byf^{-a-2}f'x^ay^{b-1})dy\big)\\
= {} & (f^{b+1}+a(b+1)xf^bf'x^{a-1}y^b)
(f^{-(a+1)}-(a+1)byf^{-a-2}f'x^ay^{b-1})dx\wedge dy\\
&+ab(a+1)(b+1)f^{-a+b-2}(f')^2x^{2a}y^{2b}dx\wedge dy\\
= {} &
(f^{b-a} + x^ay^b((b+1)a-(a+1)b)f' f^{b-a-1})dx\wedge dy\\
= {} &
(f^{b-a}+(a-b)x^ay^bf'f^{b-a-1})dx\wedge dy.
\end{align*}
\end{comment}
Thus, with $q=x^ay^b$, we require that
\[
f^{b-a}+(a-b)qf' f^{b-a-1}=1.
\]
If $a=b$, we obtain no restriction on $f$, and hence the explicit
calculation is necessary. However, if $a\not=b$, then any solution
to this differential equation is of the form
$f(q)=(1+C q)^{1/(b-a)}$ for $C$ a constant, which we may
view as an element of $A$. This now gives the
form \eqref{eq:exp v} up to determining the constant $C$. However,
we may expand $\psi(x)$ as a Taylor series in $C$ to first order
to obtain
\[
\psi(x)=x(1+{(b+1)C\over b-a}x^ay^b+\cdots).
\]
Noting that
\[
\psi(x)=x+v(x)+\cdots=x+(b+1) g x^{a+1} y^b+\cdots,
\]
we see we must have $C= (b-a) g$, as desired.
\end{proof}

We now fix a Landau-Ginzburg model $W= x^r + y^s$. We will
fix a finite $I\subseteq \Omega$ and take $A=A_I$
or $A=A_{I,\mathrm{sym}}$. In both cases, we take $\mathfrak{m}$ to
be the ideal generated by the variables (the $u_{i,d}$
or $t_{\alpha,\beta,d}$ in the two cases). We thus obtain the groups
$G_{A_I}$ or $G_{A_{I,\mathrm{sym}}}$.

\begin{rmk}
The morphism $\psi_I:A_{I,\mathrm{sym}}\rightarrow A_I$ induces a
homomorphism of nilpotent Lie algebras $\widetilde\psi_I:
\mathfrak{g}_{A_{I,\mathrm{sym}},k} \rightarrow\mathfrak{g}_{A_I,k}$
and a homomorphism of
groups $\widetilde\psi_I:G_{A_{I,\mathrm{sym}}}
\rightarrow G_{A_I}$ which is compatible with the action of these
two groups on $A_{I,\mathrm{sym}}[[x,y]]$ and $A_I[[x,y]]$ respectively,
i.e., for $g\in G_{A_{I,\mathrm{sym}}}$, $W\in A_{I,\mathrm{sym}}[[x,y]]$,
we have
\begin{equation}
\label{eq:intertwining}
\psi_I(g(W))=\widetilde\psi_I(g)(\psi_I(W)).
\end{equation}
\end{rmk}

The choice of $W$ determines natural subgroups as follows. 
We want to emphasize that these subgroups consist of direct summands 
that are parameterised by critical graphs.

\begin{definition}
Fix $I\subseteq \Universe$ a finite set of markings.
\begin{enumerate}
\item
For $J\subseteq I$, $\vecd\in \NN^J$,
we define
\[
u_{J,\vecd}:=\prod_{j\in J}u_{j,d_j}.
\]
Define for $k_1,k_2\ge 0$,
$(k_1,k_2)\not= (0,0)$,
\begin{equation}
\label{eq:special subalgebra}
\mathfrak{g}_{A_I,(k_1,k_2)}^{r,s}:=\bigoplus_{\substack{\varnothing\not=J\subseteq I,
\vecd\in\NN^J\\
k_1=r(J)\text{ mod }r,\quad k_2=s(J)\text{ mod } s\\
sk_1+rk_2=m(J,\vecd)-rs}}
u_{J,\vecd}x^{k_1}y^{k_2}((k_2+1)x\partial_x-(k_1+1)y\partial_y)\Q
\subseteq \mathfrak{g}_{(k_1,k_2)}.
\end{equation}
Set
\[
\mathfrak{g}_{A_I,k}^{r,s}
=\bigoplus_{\substack{k_1,k_2\ge 0\\ 0< k_1+k_2 \le k}}
\mathfrak{g}^{r,s}_{(k_1,k_2)}.
\]

\item
Define
\[
\mathfrak{g}_{A_{I,\mathrm{sym}},(k_1,k_2)}^{r,s}:=
\widetilde\psi_I^{-1}(\mathfrak{g}^{r,s}_{A_I,(k_1,k_2)})
\]
and
\[
\mathfrak{g}_{A_{I,\mathrm{sym}},k}^{r,s}
:=\widetilde\psi_I^{-1}(\mathfrak{g}^{r,s}_{A_I,k}).
\]
\end{enumerate}
\end{definition}

\begin{rmk}
The introduction gave different descriptions of the groups
$G_{A_{I, \mathrm{sym}}}$ and $G_{A_{I,\mathrm{sym}}}^{r,s}$. 
It is not difficult to see that the explicit constructions
of these groups here agree with the characterizations given in
\S\ref{sec:wall-crossing-intro}. Indeed, for $G_{A_{I, \mathrm{sym}}}$,
automorphisms which are
the identity modulo $\mathfrak{m}\subseteq A_{I,\mathrm{sym}}$
are given as exponentials of derivations which are zero modulo
$\mathfrak{m}$. A computation of the Lie bracket
$\mathcal{L}_v(dx\wedge dy)$ shows that any derivation preserving
$dx\wedge dy$ must be a linear combination of derivations of
the form $x^ay^b((b+1)x\partial_x
-(a+1)y\partial_y)$ with $a,b\ge -1$, $(a,b)\not= (-1,-1)$. However
the requirement that the derivation take $xy$ into the ideal generated by
$xy$ forces $a,b\ge 0$. Thus these three conditions yield the group
$G_{A_{I,\mathrm{sym}}}$. The additional restrictions for 
$\mathfrak{g}^{r,s}_{A_{I,
\mathrm{sym}},k}$ arise as follows. The condition that $[E,\cdot]$
acts by multiplication by $rs$, where $E$ is the Euler 
field of \eqref{eq:Euler}, imposes the condition that $sk_1+sk_2=
m(J,\vecd)-rs$. Finally invariance under the action 
$x\mapsto \xi_r x$, $y\mapsto \xi_s, y$, $t_{a,b,d}\mapsto 
\xi^{-a}_r \xi^{-b}_s t_{a,b,d}$ imposes the condition
$k_1=r(J)\mod r$, $k_2=s(J)\mod r$.

We omit the details as
we will not need the description given in the introduction.
\end{rmk}

\begin{prop}
The subalgebra $\mathfrak{g}^{r,s}_{A,k}$ is closed under Lie bracket in
either case $A=A_I$ or $A_{I,\mathrm{sym}}$.
\end{prop}
\begin{proof}
It is sufficient to
check this in the case where $A=A_I$. Using \eqref{eq:tedious commutator},
we have, for $J,J'\subseteq I$ disjoint, $k_1,k_2, \vecd, \vecd'$
satisfying the conditions in \eqref{eq:special subalgebra},
\begin{align*}
& [u_{J,\vecd} x^{k_1}y^{k_2}((k_2+1)x\partial_x-(k_1+1)y\partial_y,
u_{J',\vecd'} x^{k'_1}y^{k'_2}((k'_2+1)x\partial_x-(k'_1+1)y\partial_y]\\
= {} & u_{J,\vecd}u_{J',\vecd'}x^{k_1+k_1'}x^{k_2+k_2'} C
((k_2+k_2'+1)x_1\partial_{x_1}-(k_1+k_1'+1)x_2\partial_{x_2})
\end{align*}
for some $C\in \Q$.
If $\vecd''\in \NN^{J\cup J'}$ denotes the
descendent vector agreeing with $\vecd$ on $J$ and $\vecd'$ on $J'$,
then we just need to check (1) $k_1+k_1'=r(J\cup J') \mod r$,
$k_2+k_2'=s(J\cup J') \mod s$, which is automatic, and (2)
$s(k_1+k_1')+r(k_2+k_2')=m(J\cup J',\vecd'')-rs$. But this follows
immediately from the definition of $m$.
\end{proof}

\begin{definition}
We write $G_A^{r,s}\subseteq G_A$ for the subgroups obtained from
the system of Lie algebras $\mathfrak{g}^{r,s}_{A,k}$ for $A=A_I$
or $A_{I,\mathrm{sym}}$.
\end{definition}

The key observation is then:

\begin{thm}
\label{thm:G action}
Let $I\subseteq\Omega$ be finite.
There is a well-defined action of the group $G_{A_I}^{r,s}$ on
the set of chamber indices
$\ChamberIndices(I)$ (see Definition \ref{def:chamberIndex}) given by
\[
g(W^{\CI})=W^{g(\CI)}.
\]
This action is faithful and transitive. Similarly, there is a well-defined
action of the group $G^{r,s}_{A_{I,\mathrm{sym}}}$ on
$\ChamberIndicesSym(I)$ defined via the same formula. This action is
faithful and transitive.
\end{thm}

\begin{proof}
{\bf Step 1. Existence of the action}.
We first consider the non-symmetric case.
There is an injection from $\ChamberIndices(I)$ into
$A_I[[x,y]]$ given by $\CI\mapsto W^{\CI}$.
Given $g\in G_{A_I}^{r,s}$, we thus obtain $g(W^{\CI})\in A_I[[x,y]]$. To
obtain the desired action, we need to show that there exists a chamber index
$\CI'$ such that $W^{\CI'}=g(W^{\CI})$.

Since elements of the form $\exp(v)$ for $v\in
\mathfrak{g}^{r,s}_{A_I,(k_1,k_2)}$ generate the group $G^{r,s}_{A_I}$,
it is enough to check this statement for
\begin{equation}
\label{eq:our automorphism}
g=\exp(Cu_{J,\vecd}
x^{k_1}y^{k_2}((k_2+1)x\partial_x-(k_1+1)y\partial_y)), \quad
C\in \Q,
\end{equation}
the exponential of a generator of $\mathfrak{g}^{r,s}_{A_I,(k_1,k_2)}$. Note that since $g$ corresponds to a generator of $\mathfrak{g}^{r,s}_{A_I,(k_1,k_2)}$, we have
\begin{equation}\label{local conditions of critical}
k_1=r(J) \pmod r,\quad k_2=s(J) \pmod s, \text{ and }sk_1+rk_2=m(J,\vecd)-rs.
\end{equation}
So let $C' u_{J',\vecd'}x^{k_1'}y^{k_2'}$ be a term in $W^{\CI}$.
By the definition of chamber index, this term is non-zero only if
$\Gamma_{0,k_1',k_2',1,J'}$ is balanced with respect to $\vecd'$,
i.e.,
\begin{equation}\label{local conditions of balanced}
k_1'=r(J') \pmod r,\quad  k_2'=s(J') \pmod s, \text{ and }sk_1'+rk_2'=m(J',\vecd').
\end{equation}
Applying
$g$ to this term produces no new terms if $J\cap J'\not=\varnothing$,
as then $u_{J,\vecd}u_{J',\vecd'}=0$. Otherwise,
we see that,
by applying Equation~\eqref{eq:action} and noting $u_{J,\vecd}^2=0$,
we obtain a new term of the form
$$C''u_{J\cup J',\vecd''}x^{k_1+k_1'}y^{k_2+k_2'}$$ for some $C''\in \Q$. Now
$k_1+k_1'=r(J)+r(J')=r(J\cup J') \pmod r$ and
$k_2+k_2'=s(J)+s(J')=s(J\cup J') \pmod s$. Furthermore, using~\eqref{local conditions of critical} and~\eqref{local conditions of balanced}
\[
m(J\cup J',\vecd'')=m(J,\vecd)+m(J',\vecd')-rs=
sk_1+rk_2+rs + sk_1'+rk_2' - rs = s(k_1+k_1')+r(k_2+k_2'),
\]
showing that $\Gamma_{0,k_1+k_1',k_2+k_2',1,J\cup J'}$ is balanced with
respect to $\vecd''$.
This shows that $g(W^{\CI})$ arises from an element $g(\CI)=(g(\CI)_{\Gamma,\vecd})
\in \Inv(I)$.

We now only need to show that $g(\CI)$
satisfies the conditions of Definition \ref{def:chamberIndex}, still
with $g$ as given in \eqref{eq:our automorphism}.
First we show Condition (1) of the definition.
Suppose that $\vecd=0$. Then necessarily $|J|\ge 2$. Indeed,
using Notation \ref{Concrete bal and crit graphs with prescribed internals}, there are $\ell_1+\ell_2-|J|+1$
choices of $k_1,k_2\ge 0$ satisfying the conditions in~\eqref{local conditions of critical}. However, if $J=\{j\}$, then
$r(J)=a_j, s(J)=b_j$, and thus $\ell_1=\ell_2=0$ and there are no such
$k_1,k_2$. Thus in this case,
$g$ is the identity modulo $\mathfrak{m}^2$, and
we have $g(\CI)_{\Gamma,\mathbf{0}}=\CI_{\Gamma,\mathbf{0}}$ when
$|I(\Gamma)|=0$ or $1$, as desired.

For Conditions (2) and (3) of Definition \ref{def:chamberIndex}, consider the two potentials $W^\nu$ and $W^{g(\nu)}$. For any $I' \subseteq I$ and $\vecd'\in \NN^{I'}$, it follows from \eqref{eq:invariance of osc int} that
\[
\int_{\Xi_{r(I'), s(I')}} e^{W^{\nu}/\hbar}dx\wedge dy = \int_{\Xi_{r(I'), s(I')}} e^{W^{g(\nu)}/\hbar}dx \wedge dy,
\]
which implies by Theorem \ref{thm:period integral} that
$\mathcal{A}(I',\vecd',\CI) = \mathcal{A}(I',\vecd', g(\CI))$,
and hence the required value of $\mathcal{A}(I',\vecd',g(\CI))$
follows from the same requirement for $\mathcal{A}(I',\vecd',\CI)$.

For the symmetric case, we note that if $g\in G^{r,s}_{A_{I,\mathrm{sym}}}$,
then $\widetilde\psi_I(g)$ acts on $\CI\in \InvSym(I)$ as above
giving $\widetilde\psi_I(g)(\CI) \in \Inv(I)$. However, since
$\widetilde\psi_I(g)(\psi_I(W^{\CI}))=\psi_I(g(W^{\CI}))$ by
\eqref{eq:intertwining}, we see in fact that $\widetilde\psi_I(g)(\CI)
\in \InvSym(I)$. As we have already seen above this action takes
chamber indices to chamber indices, we get the desired action on
$\ChamberIndicesSym(I)$.

\medskip

{\bf Step 2. Faithfulness of the action}.
Take $A=A_I$ or $A_{I,\mathrm{sym}}$, and suppose that $\CI\in
\ChamberIndices(I)$ or $\ChamberIndicesSym(I)$ in the two cases. Suppose
$g=\exp(v)\in G_A^{r,s}$ satisfies $g(\CI)=\CI$. Let $k\ge 2$ be the smallest
integer such that $v \equiv 0 \mod \mathfrak{m}^k$ but
$v\not\equiv 0 \mod \mathfrak{m}^{k+1}$. Then working modulo
$\mathfrak{m}^{k+1}$, we have
$g(W^{\CI})=W^{\CI}+v(W^{\CI})=W^{\CI}+v(x^r+y^s)$.
Thus it suffices to check that $v(x^r+y^s)=0$
implies $v =0$, still modulo $\mathfrak{m}^{k+1}$.

Now $v$ is a linear combination
of vector fields of the form given in \eqref{eq:special subalgebra}.
Let $k_1$ be minimal so that $v$ has a non-zero summand in
$\mathfrak{g}^{r,s}_{A,(k_1,k_2)}$ for some $k_2$,
and let $k_2$ be minimal among such choices
of $k_2$ fixing $k_1$. In other words, we consider the minimal pair
$(k_1,k_2)$ occuring in $v \mod \mathfrak{m}^{k+1}$ with respect to the
lexicographic ordering. Thus we may write the
$\mathfrak{g}^{r,s}_{A,(k_1,k_2)}$ summand of $v$ as
$v'=Cx^{k_1}y^{k_2} ((k_2+1)x\partial_x-(k_1+1)y\partial_y)$
for some $C\in \mathfrak{m}^k\setminus \mathfrak{m}^{k+1}$.
Now
\[
v'(x^r+y^s)= C\left(r(k_2+1)x^{k_1+r}y^{k_2}
-s(k_1+1) x^{k_1} y^{k_2+s}\right).
\]
Note that neither of the two terms in this expression is zero.
On the other hand,
because of the minimality of the pair $(k_1,k_2)$, this expression
does not cancel with any other terms appearing in $v(x^r+y^s)$. This
shows $C=0$, hence $v=0\mod \mathfrak{m}^{k+1}$, contradicting the
definition of $k$. Hence $v=0$.

\medskip

{\bf Step 3. Transitivity of the action.}
We will consider the case $A=A_I$, the symmetric case being similar.
Let $\CI$, $\CI'\in \ChamberIndices(I)$.
We proceed inductively, showing that
for any positive integer $k$, there exists an element $g\in G^{r,s}_{A_I}$
such that $g(\CI)_{\Gamma,\vecd}=\CI'_{\Gamma,\vecd}$ for any $\Gamma\in\INT(J,\vecd)$ with $sk_1(\Gamma)+rk_2(\Gamma)<k$.
Note that the base case of $k=0$ is automatic,
as we may take $g$ to be the identity. Define an ideal $\mathcal{I}_k := \langle x^{k_1}y^{k_2} \in A_I[[x,y]] \ | \ sk_1 + rk_2 \ge k \rangle$.

Now assume we have found a $g$ which works for a given $k$. By
replacing $\CI$ with $g(\CI)$, we may assume that $W^{\CI}=W^{\CI'}
\mod \mathcal{I}_k$. Now consider $J\subseteq I, \vecd\in \NN^J$
such that there exists a
$\Gamma\in\INT(J,\vecd)$ with $sk_1(\Gamma)+rk_2(\Gamma)=k$.
Consider the coefficient of $u_{J,\vecd}$ in $W^{\CI'}-W^{\CI}$, namely
\[
(-1)^{|J|-1}\sum_{\Gamma=\Gamma_{0,k_1,k_2,1,J}\in \INT(J,\vecd)}
(\CI'_{\Gamma,\vecd}-\CI_{\Gamma,\vecd}) x^{k_1}y^{k_2}
\]
Recall from Notation~\ref{Concrete bal and crit graphs with prescribed internals}
that we have a complete classification of graphs $\Gamma_{J, p} \in
\INT(J,\vecd)$ where $\Gamma_{J,p}$ is the unique balanced graph with $k_1(\Gamma_{J,p}) = r(J) + pr$ and $k_2(\Gamma_{J,p}) = s(J) + (N-p) s$, and $I(\Gamma_{J,p}) = J$ for $p = 0, \ldots, N = \frac{\sum_{i\in J} a_i -r(J)}{r} + \frac{\sum_{i\in J} b_i - s(J) }{s} -|J| + 1 + \sum_{j\in J} d_j $.
Note that $sk_1(\Gamma_{J,p})+rk_2(\Gamma_{J,p})$ is independent of $p$,
and hence is always $k$.

In particular, we may write
the $u_{J,\vecd}$ coefficient of $W^{\CI'}-W^{\CI}$ as
\[
(-1)^{|J|-1}\sum_{p=0}^N (\CI'_{\Gamma_{J,p},\vecd}-\CI_{\Gamma_{J,p},\vecd})
x^{r(J)+pr}
y^{s(J)+(N-p)s}.
\]
Let $p'\ge 0$ be the smallest $p$ such that $\CI'_{\Gamma_{J,p},\vecd}\not=
\CI_{\Gamma_{J,p}, \vecd}$.
Now consider
\begin{align*}
v=(-1)^{|J|}u_{J,\vecd}{\CI'_{\Gamma_{J,p'},\vecd}-\CI_{\Gamma_{J,p'},
\vecd}
\over s(r(J)+p'r+1)}
&
x^{r(J)+p'r}y^{s(J)+(N-p'-1)s}\\
&((s(J)+(N-p'-1)s+1)x\partial_x
-(r(J)+p'r+1)y\partial_y).
\end{align*}
One checks easily that this lies in $\mathfrak{g}^{r,s}_{A_I}$ and corresponds to the critical graph $\Lambda_{J, p'+1}$ in Notation~\ref{Concrete bal and crit graphs with prescribed internals}.
Modulo $\mathcal{I}_{k+1}$, $\exp(v)(W^{\CI})$ differs from
$W^{\CI}$ only in the coefficients of $u_{J,\vecd}x^{r(J)+p'r}y^{s(J)+(N-p')s}$
and $u_{J,\vecd}x^{r(J)+(p'+1)r}y^{s(J)+(n-p'-1)s}$, and further
the coefficient of $u_{J,\vecd}x^{r(J)+p'r}y^{s(J)+(N-p')s}$  in $\exp(v)(W^{\CI})$ agrees with
the coefficient of the same monomial in $W^{\CI'}$. Thus by replacing
$W^{\CI}$ with $\exp(v)(W^{\CI})$ and repeating this process inductively on $p'$, we may assume that
the coefficient of $u_{J,\vecd}$ in $W^{\CI'}-W^{\CI}$ is
\[
(-1)^k(\CI'_{\Gamma_{J,N},\vecd}-\CI_{\Gamma_{J,N},\vecd}) x^{r(J)+Nr}
y^{s(J)},
\]
i.e., there is at most one $\Gamma$ with $I(\Gamma)=J$ with
$\CI_{\Gamma,\vecd}\not=\CI'_{\Gamma,\vecd}$.
Further, inductively we
have assumed $\CI_{\Gamma',\vecd'}=\CI'_{\Gamma',\vecd'}$
whenever $I(\Gamma')\subsetneq J$.
Since $\Gamma$ is balanced for $\vecd$, we have $m(J,\vecd)=
sk_1(\Gamma)+r k_2(\Gamma)=sr(J)+rs(J)+rsN \ge sr(J)+rs(J)$,
so if $|J|\ge 2$,
the inequality of Condition (3) of Definition \ref{def:chamberIndex}
holds and $\mathcal{A}(J,\vecd,\CI)=\mathcal{A}(J,\vecd,\CI')=0$. On
the other hand, if $|J|=1$, Condition (2) similarly implies that $\mathcal{A}(J,\vecd,\CI)=\mathcal{A}(J,\vecd,\CI')$.
Looking at the form of these expressions, one notes they agree
except for the contribution from the terms $\CI_{\Gamma,\vecd}$
and $\CI'_{\Gamma,\vecd}$,
and that implies equality of these two terms.

Repeating this for all $J,\vecd$ such that there exists
a $\Gamma\in\INT(J,\vecd)$ with $sk_1(\Gamma)+rk_2(\Gamma)=k$ then gives
the induction step. Note there are a finite number of $J\subseteq I$ and graphs $\Gamma$ with $I(\Gamma) =J$  with $sk_1(\Gamma) + rk_2(\Gamma) \le k$ and balanced with respect to some descendent vector $\vecd$.
Thus by working modulo $\mathcal{I}_{k+1}$ the automorphism we build only involves a
composition of a finite number of automorphisms.
\end{proof}

We end this section with a result that will be of use in \S\ref{sec:WC} and clarifies the computational result from Corollary~\ref{cor:symmetric integal}.

\begin{cor}\label{cor:CIs same result}
Let $I$ be a finite set of $\Omega$. Take $\CI$ and $\CI'$ to be two chamber indices bounded by $I$. Then,
\[
\mathcal{A}(J,\vecd,\CI) = \mathcal{A}(J,\vecd,\CI')
\]
for all $J \subseteq I$ and $\vecd \in \NN^J$.
\end{cor}
\begin{proof}
By Theorem~\ref{thm:G action}, there exists an element $g \in G_{A_I}^{r,s}$ so that $g\cdot \CI = \CI'$. Next, note that, by~\eqref{eq:invariance of osc int}, since $\CI$ and $\CI'$ are in the same orbit of the wall-crossing group action we have that
$$
\int_{\Xi} e^{W^{\CI}/\hbar} dx \wedge dy = \int_{\Xi} e^{W^{\CI'}/\hbar} dx \wedge dy.
$$
The result then follows from comparing coefficients in the result from Corollary~\ref{cor:symmetric integal}.
\end{proof}

\section{Invariance, wall-crossing, mirror symmetry, and open topological recursion}
\label{sec:invariants}

\subsection{Invariance}
Unlike the open $r$-spin invariants defined in \cite{BCT:I, BCT:II}, the open $W$-spin invariants for $W = x^r + y^s$ defined here \emph{do depend on the choice of the canonical multisection}. However, some simple invariants are independent of these choices:

\begin{thm}\label{thm:simple invariants}
We have
\[
\langle\sigma_1^{r}\sigma_{12}\rangle^{\mathbf{s},o}=\langle\sigma_2^{s}\sigma_{12}\rangle^{\mathbf{s},o}=-1.
\]
For $0 \leq m_1\leq r-1, 0 \leq m_2\leq s-1$, we have
\[
\langle\tau^{(m_1,m_2)}_0\sigma_1^{m_1}\sigma_2^{m_2}\sigma_{12}\rangle^{\mathbf{s},o}=1.
\]
In particular these numbers are independent of
the choice of $\mathbf{s}$ and 
$W^{\CI,\mathrm{sym}}$ satisfies \eqref{eq:needed cong}.
\end{thm}

\begin{rmk}
We note that this justifies item (1) in the definition of chamber
index, see Definition~\ref{def:chamberIndex}. 
\end{rmk}

\begin{proof}[Proof of Theorem \ref{thm:simple invariants}.]
Let $\mathbf{s}_0,\mathbf{s}_1$ be two canonical multisections for $E\to\oPM_\Gamma,$ for a smooth rooted graded (possibly disconnected) $\Gamma$ that
is balanced with respect to a descendent vector $\vecd$.
By Lemma \ref{lem:partial_homotopy}\eqref{it:hom_invariants} we can find a canonical transverse homotopy $H\in C_m^\infty([0,1]\times\oPM_\Gamma,\pi^*E)$ with $H(\eps,-)=\mathbf{s}_\eps,$ for $\eps=0,1$, and it holds, by Lemma \ref{lem:zero diff as homotopy}, that
\begin{equation}\label{eq:diff_of_can}\left\langle \prod_{i\in I(\Gamma)}
\tau^{(a_i,b_i)}_{d_i}\sigma^{k(\Gamma)} \right\rangle^{\mathbf{s}_1^\Gamma,o}-\left\langle \prod_{i\in I(\Gamma)}\tau^{(a_i,b_i)}_{d_i}\sigma^{k(\Gamma)} \right\rangle^{\mathbf{s}_0^\Gamma,o}=\#Z(H|_{([0,1]\times\partial^0\oPM_\Gamma)\cup([0,1]\times\partial^\XCH\oPM_\Gamma)}).\end{equation}
Recall that we showed in Example~\ref{ex:emptyboundary} that if $\Gamma\in\{\Gamma_{0,m_1,m_2,1,\{(m_1,m_2)\}},\Gamma_{0,r,0,1,\emptyset},\Gamma_{0,0,s,1,\emptyset}\}$, then $\partial^0\Gamma=\emptyset$. Here, the set $\partial^\XCH\oPM_\Gamma$ need not be empty; however, by Lemma~\ref{lem:exchange vanishes} below, we have that $\#Z(H|_{[0,1]\times\partial^\XCH\oPM_\Gamma}) = 0$. 
Thus, the quantity in \eqref{eq:diff_of_can} vanishes and we conclude that the intersection number is independent of the choice of the canonical multisection.

When $\Gamma=\Gamma_{0,m_1,m_2,1,\{(m_1,m_2)\}}$, let
$\Gamma_i=\text{for}_{\text{spin}\neq i}(\Gamma)$.
Note that $\Gamma_1$ is the graded $r$-spin graph with a single internal tail with twist $m_1$ and $m_1+1$ boundary tails, and $\Gamma_2$ is the similarly defined graded $s$-spin graph.
We may then take $\mathbf{s}=\text{For}_{\text{spin}\neq 1}^*\zeta_1\oplus
\text{For}_{\text{spin}\neq 2}^*\zeta_2$,
where $\zeta_i$ is a transverse global canonical multisection of the Witten bundle over $\oPM_{\Gamma_i}$. By transversality we may assume that it vanishes only in $\CM^W_{\Gamma_i}$.
Using \cite[Appendix A]{BCT:II}, and specifically Notation A.4 there, we can write
\[Z(\zeta_i)=\sum_{p_j^i\in Z(\zeta_i)}\varepsilon_j^i p_j^i,\] where $\varepsilon_j^i$ is the weight of $p_j^i$ in the zero locus of $\zeta_i.$
Hence,
\[\#Z(\mathbf{s})=\sum_{p_{j_1}^1\in Z(\zeta_1)}\sum_{p_{j_2}^2\in Z(\zeta_2)}\varepsilon_{j_1}^1\varepsilon_{j_2}^2\#\left(\pi_1^{-1}(p_{j_1}^1)\cap \pi_2^{-1}(p_{j_2}^2)\right).
\]
For smooth moduli points $p^i\in \CM^W_{\Gamma_i}$, $i=1,2$,
$\pi_1^{-1}(p^1)\cap \pi_2^{-1}(p^2)$ is a singleton set, the unique point of $\CM^W_\Gamma$ whose projection to each $\CM^W_{\Gamma_i}$ is $p^i$.\footnote{As both projections share a common internal and boundary marking given by the unique internal marking and the root, we may represent each point $p^1$ and $p^2$ as the unit disk model with the internal marking at $0$ and root marking at $1$.  This kills off all automorphisms given by the action of $PSL(2, \mathbb{R})$. Thus, knowing the points $p^i\in\oPM_{\Gamma_i}$, $i=1,2$ is sufficient to
identify the point $p\in \oPM_{\Gamma}$ with $\pi_i(p)=p^i$, $i=1,2$.}
Hence
\[\#Z(\mathbf{s})=
\sum_{p_{j_1}^1\in Z(\zeta_1)}\sum_{p_{j_2}^2\in Z(\zeta_2)}\varepsilon_{j_1}^1\varepsilon_{j_2}^2 =
\frac{\langle\tau_0^{m_1}\sigma^{m_1+1}\rangle^{1/r,o}}{m_1!}\frac{\langle\tau_0^{m_2}\sigma^{m_2+1}\rangle^{1/s,o}}{m_2!}.
\]
We note that the division by $m_1!m_2!$ is because of our convention that we do not mark the boundary points, unlike \cite{BCT:I,BCT:II}. By Theorem~1.2
of \cite{BCT:II},
\[\#Z(\mathbf{s})=1.\]

When $\Gamma=\Gamma_{0,r,0,1,\emptyset}$, the $s$-Witten bundle is
rank $0$.
Hence, the multisection $\mathbf{s}$ is a pull back of a canonical multisection of $\cW\to\oPM_{\Gamma_1},$ where $\Gamma_1=\text{for}_{\text{spin}\neq 1}(\Gamma)$
is the smooth graded $r$-spin graph with no internal tails and $r+1$ boundary tails. Consequently,
\[\langle\sigma_1^r\sigma_{12}\rangle^{\mathbf{s},o}=\langle\sigma^{r+1}\rangle^{1/r,o}/r!=-1,\]
where the last equality used Theorem~1.2 of \cite{BCT:II}. The same argument shows $\langle\sigma_2^s\sigma_{12}\rangle^{\mathbf{s},o}=-1$.
\end{proof}

The reason open $W$-spin invariants depend on the choice of canonical section is that homotopies between these canonical sections may have zeroes on boundary strata. However, we can now check that this does not occur at exchangeable strata.

\begin{lemma}\label{lem:exchange vanishes}
Let $\Gamma=\Gamma_{0,k_1,k_2,k_{12},\{(a_i,b_i)\}_{i\in I}}$ be a smooth, connected, graded $W$-spin graph which is
balanced with respect to
$\vecd$.  Let $\mathbf{s}_0,\mathbf{s}_1$ be two canonical multisections for $E_{\Gamma}(\vecd)\to\oPM_\Gamma$. Take $H$ to be a canonical homotopy with $H(\epsilon,-) = \mathbf{s}_\epsilon$ for $\epsilon \in \{0,1\}$ that is a part of a family of canonical homotopies as built in Lemma~\ref{lem:partial_homotopy}. Then
$$
\#Z(H|_{[0,1]\times\partial^\xch\oPM_\Gamma}) = 0.
$$
\end{lemma}

\begin{proof}
Suppose $\Xi \in \partial^\xch \Gamma$ and $\dim \oPM_\Xi \leq \dim \oPM_\Gamma - 2$. Then by the pigeon-hole principle and \eqref{eq:dim rank comparison},
there either exists a connected component $\Lambda \in \Conn(\CB\Xi)$ so that $\dim \CM^W_\Lambda
\le  \rank E_{\Lambda}(\vecd) - 2$, or two connected components
$\Lambda_1,\Lambda_2\in \Conn(\CB\Xi)$ so that $\dim\CM^W_{\Lambda_i}
=\rank E_{\Lambda_i}(\vecd)-1$ for $i=1,2$. In either case,
the canonical homotopy $H$ will not vanish on the strata $[0,1]\times\oPM_{\CB\Xi}$ by Lemma~\ref{lem:partial_homotopy}(1)(b) or (c). In turn, by the canonicity of $H$, it will also not vanish on $\oPM_\Xi$. Thus, in order for $H$ to vanish, we must have that $\oPM_\Xi$ has codimension one. 
Thus, we can write
\[\#Z(H|_{[0,1]\times\partial^\XCH\oPM_\Gamma})=\sum_{\substack{\Xi\in\partial^\xch\Gamma \\ \dim \oPM_{\Xi} = \dim \oPM_{\Gamma} - 1}}\#Z(H|_{[0,1]\times
\oPM_\Xi}).\]

Recall that $\Gamma$ contains all cyclic orderings of the vertex, and
that the graphs $\Xi\in\partial^\xch\Gamma$ of codimension one have two open vertices: a rooted vertex and an exchangeable vertex $v$. This exchangeable vertex has one half-edge $h$ and two boundary tails $t_1$ and $t_2$. 
 Note that $v$ has two possible cyclic orderings on its half-edges, namely $\pi_{1}:=(h \to t_1 \to t_2)$ and $\pi_{2} := (h \to t_2 \to t_1)$, which correspond to two collections of connected components of $\oPM_{\Xi}$.

We now make two new $W$-spin graphs $\Xi_{1}$ and $\Xi_2$ that are the same as $\Xi$, but the only cyclic ordering on $v$ for $\Xi_i$ is $\pi_i$.
The moduli $\oPM_{\Xi_1}$ and $\oPM_{\Xi_2}$ are related by the (involutive) map $\XCH_v$. Since $H$ is a canonical homotopy, we have by Observation~\ref{obs:exchangeable_and_base} that \[|\#Z( H_{[0,1]\times\oPM_{\Xi_1}]})|=|\#Z(H_{[0,1]\times \oPM_{\XCH_v(\Xi_1)}})| =|\#Z([H_{[0,1]\times\oPM_{\Xi_2}})| ,\]while Proposition \ref{prop: orient exchange} shows that their orientations as boundary strata in $\oPM_\Gamma$ are opposite. Thus, for any $\Xi \in\partial^\xch\Gamma$ of codimension one, we have that
\[\#Z(H|_{[0,1]\times\oPM_\Xi})=\#Z(H|_{[0,1]\times\oPM_{\Xi_1}})+\#Z(H|_{[0,1]\times\oPM_{\Xi_2}})=0.\]
\end{proof}

The main theorem of this subsection is:

\begin{thm}\label{thm:A_mod_invs}
$\mathcal{A}(I,\vecd,\CI^{\ess})$ is independent of the choice of family
of canonical multisections $\mathbf{s}$.
\end{thm}

We will now assemble some of the key notation for keeping track of
terms in ${\mathcal A}(I,\vecd,\CI^{\ess})$
and critical boundary graphs which
play a role in changes of the individual open invariants appearing
in ${\mathcal A}(I,\vecd,\CI^{\ess})$. 
(See Definition~\ref{def:critical} for critical boundary graphs).
This notation will be used both in
the proof of the above theorem and in the proof of the open topological
recursion formula of the next subsection. 

\begin{nn}\label{P sets of balanced}
Let $I$ be a set of internal markings with $|I|=l$.
For $h\le l$, write $\mathcal{P}_h(I,\vecd)$ for the collection of sets
of the form
\[
\{(I_1,k_1(1),k_2(1)),\ldots, (I_h,k_1(h),k_2(h))\}
\]
where $I_1,\ldots, I_h$ is an ordered partition of $I$ into $h$ non-empty pieces, $k_1(i),k_2(i)\geq 0$ for all $i$ and
\begin{equation}
\label{eq:balancing requirements}
k_1(i) = r(I_i)~(\text{mod}~r),~k_2(i) = s(I_i)~(\text{mod}~s),~\text{and } sk_1(i)+rk_2(i)=m(I_i,\vecd).
\end{equation}
 Write $\mathcal{P}(I,\vecd)=\bigcup_{h\in[l]}\mathcal{P}_h(I,\vecd)$.
For $P={\{(I_1,k_1(1),k_2(1)),\ldots, (I_h,k_1(h),k_2(h))\}}\in\mathcal{P}_h(I,\vecd),$ let $\Gamma_P$ be the following disjoint union of graded rooted
graphs:
\begin{equation}\label{def: GammaP}
\Gamma_P:=
\bigsqcup_{j=1}^h \Gamma_{P,j}:=
\bigsqcup_{j=1}^h \Gamma_{0,k_1(j),k_2(j),1, \{(a_i,b_i)\}_{i\in I_j}}.
\end{equation}
\end{nn}

The graphs $\Gamma_P$ for $P\in \mathcal{P}(I,\vecd)$ are precisely the smooth graphs which correspond to the product of open invariants in
Notation \ref{nn:A(J)} for a given summand. With notation as in the first paragraph of the proof of Theorem
\ref{thm:simple invariants}, note that $E=E_{\Gamma_P}(\vecd)$
naturally decomposes
as $E=\boxplus_{\Lambda\in\Conn(\Gamma_P)}E_\Lambda(\vecd)$. 
The requirement \eqref{eq:balancing requirements} then guarantees that
for each $\Lambda\in\Conn(\Gamma_P)$,
$\rank E_\Lambda(\vecd)=\dim\oPM_\Lambda$, i.e., $\Lambda$ is balanced,
see Proposition \ref{prop:balanced}.

\begin{nn}
\label{not:Gamma mess of subscripts}
We now consider notation for accounting for critical boundaries of
graphs $\Gamma_P$.
Let $P \in \mathcal{P}_h(I,\vecd)$ and take $(j, \eps, I_0, k_1(0), k_2(0), R_0, S_0)$ so that $1 \le j \le h$,  $I_0 \subseteq I_j$,  $\eps \in \{1,2\}$, and
\begin{equation}\label{eq:index constraints}\begin{aligned}
0 &\le k_1(0) \le k_1(j), &0 &\le k_2(0) \le k_2(j),
\\ 0 &\le R_0 \le k_1(j) - k_1(0), &0&\le S_0 \le k_2(j) - k_2(0).
\end{aligned} \end{equation}
We will construct a boundary graph $\Gamma_{P, j, \epsilon, I_0, k_1(0),k_2(0),R_0,S_0}\in\partial^0\Gamma_P\cup \partial^{\xch}\Gamma_P$.
This graph keeps all but the $j$th connected component of $\Gamma_P$ the same, but replaces $\Gamma_{P,j}$ with a graph
consisting of
\begin{itemize}
\item two open vertices $v_0$ and $v_j$, where $v_j$ is rooted;
\item one boundary edge $e$ consisting of two half-edges $h_0$ and $h_j$ belonging to the vertices $v_0$ and $v_j$ respectively;
\item the vertex $v_0$ has $k_1(0) + 1$ $r$-points and $k_2(0)+1$ $s$-points, including the half-edge $h_0$ above; where
\begin{itemize}
\item if $\epsilon =1$, the half-edge $h_0$ has $\tw(h_0) = (0,s-2)$ and $\alt(h_0)=(0,1)$; 
\item if $\epsilon =2$, the half-edge $h_0$ has $\tw(h_0) = (r-2,0)$ and $\alt(h_0)=(1,0)$;

\end{itemize}
\item the vertex $v_j$ has $(k_1(j)- k_1(0))$ $r$-points and $(k_2(j)-k_2(0))$ $s$-points, including the half-edge $h_j$;
\item the set $\hat\Pi$ of cyclic orderings consists of all cyclic orderings so that the following two conditions hold:
\begin{itemize}
\item the number of $r$-points in $v_0$ following the root but before the half-edge $h_0$ is $R_0$;
\item  the number of $s$-points in $v_0$ following the root but before the half-edge $h_0$ is $S_0$.
\end{itemize}
\end{itemize}
\end{nn}

If $\Gamma_{P, j, \epsilon, I_0,k_1(0),k_2(0), R_0,S_0}$ is a critical
boundary graph in the sense of Definition~\ref{def:critical},
then it corresponds to a codimension one boundary strata of the moduli space $\oPM_{\Gamma_P}$ where a canonical homotopy between two multisections may vanish.
This will correspond to wall-crossing phenomena for the open invariants
discussed throughout the rest of the paper.

\begin{figure}

  \centering

\begin{tikzpicture}[scale=0.5]
  \draw (0,0) circle (2cm);
  \draw (4,0) circle (2cm);
  \node (c) at (-2,0) {$\times$};
  \node (a) at (2,0) {$\bullet$};
  \node (b) at (1.5, 0) {\tiny$h_j$};
  \node (d) at (2.5, 0) {\tiny$h_0$};
  \node (0) at (-2, 2) {\tiny $v_j$};
  \node (j) at (6,2) {\tiny $v_0$};
  \node (ij) at (0,0) {\tiny $I_j\setminus I_0$};
  \node (i0) at (4,0) {\tiny $I_0$};
  \node (r) at (-2.9,0) {\tiny Root};

  \node (r1) at (5.73,-1) {\tiny $\bullet$};
    \node (r2) at (5.29,-1.532) {\tiny $\bullet$};
   \node (r1) at (4.347,-1.9696) {\tiny $\bullet$};
   \node (r1) at (5.36,1.46) {\tiny $\bullet$};
   \node (r1) at (2.71,-1.532) {\tiny $\bullet$};

   \node (r1) at (-0.347,-1.9696) {\tiny $\bullet$};
   \node (r1) at (0,2) {\tiny $\bullet$};

\end{tikzpicture}

  \caption{Stable disk corresponding to the $j$th component of the critical boundary graph $\Gamma_{P, j, \epsilon, I_0, k_1(0),k_2(0), R_0,S_0}$.}
\end{figure}
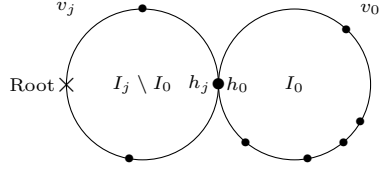

\begin{nn}
\label{nn:Q notation}
We next introduce an analogue of the set $\mathcal{P}_h(I, \vecd)$ that corresponds to bases of critical boundaries.  For $h\ge 0$, let $\mathcal{Q}_h(I, \vecd)$ be the collection of sets
$$
\{(I_0, k_1(0)+1,k_2(0)+1), \ldots, (I_h, k_1(h),k_2(h))\}
$$
so that the following hold:
\begin{itemize}
\item $\{I_0, \ldots, I_h\}$ is an ordered partition of $[h]$ such that $I_0$ is nonempty and at most one $I_j$ is empty;
\item $k_1(i) \equiv r(I_i) \pmod r$ for all $i$;
\item $k_2(i) \equiv s(I_i) \pmod s$ for all $i$; and
\item $sk_1(i)+rk_2(i) = m(I_i, \vecd) - \delta_{0i} rs$ for all
$i$.
\end{itemize}
These conditions arise from Propositions~\ref{prop:balanced} and~\ref{prop:critical m(I)}. Set
$$
\mathcal{Q}(I,\vecd) = \bigcup_{h=1}^l \mathcal{Q}_h(I, \vecd)
$$
with $l=|I|$.
Given a $Q = \{(I_0, k_1(0)+1,k_2(0)+1), \ldots, (I_h, k_1(h),k_2(h))\}\in \mathcal{Q}(I,\vecd)$, we can construct a graph $\Lambda_Q$ with $h+1$ connected components so that
\begin{itemize}
\item one connected component is the critical graph
\[
\Lambda_{Q,0}:=\Gamma_{0, k_1(0)+1, k_2(0)+1, 0, \{(a_i,b_i)\}_{i \in I_0}};
\]
\item $h$ connected components are the balanced graphs
\[
\Lambda_{Q,j}:=\Gamma_{0, k_1(j), k_2(j),1,\{(a_i,b_i)\}_{i \in I_j}}
\]
for $1\leq j\leq h$.
\end{itemize}
These graphs are the graphs corresponding to the \emph{bases of critical boundaries} for some graph $\Gamma_P$ associated to $P\in \mathcal{P}_h(I,\vecd)$.
In other words, $\CB{\Gamma_{P, j, \epsilon, I_0,k_1(0),k_2(0), R_0,S_0}}$
is a critical boundary graph of $\Gamma_P$ if and only if
$\Lambda_Q \succ \CB{\Gamma_{P, j, \epsilon, I_0,k_1(0),k_2(0), R_0,S_0}}$
for some $Q$,
using the notation in Definition~\ref{defn:subordinate}.
We remark that there is not necessarily a unique $P$ which satisfies this property, and indeed, it is the combinatorics of how the same base graph
$\Lambda_Q$ contributes to critical boundaries of different
$\oPM_{\Gamma_P}$ that is key to many of the arguments in the rest of
this paper. This kind of argument first appears in Part (B) of the proof of
Theorem \ref{thm:A_mod_invs}.

Consider $Q\in \mathcal{Q}_h(I,\vecd)$ with
$I_j=\emptyset$ for some $1\le j \le h$. Then the graph $\Lambda_{Q,j}$ is rooted and has either $k_1(j) = r, k_2(j) = 0$ or $k_1(j) =0, k_2(j) = s$, by Proposition~\ref{prop:balanced}(a) and (b).  Thus, we can partition the set $\mathcal{Q}_h(I,\vecd)$ into three subsets $\mathcal{Q}^{r}_h(I,\vecd)\cup \mathcal{Q}^{s}_h(I,\vecd)\cup \mathcal{Q}^{\neq\emptyset}_h(I,\vecd)$
defined as follows:
\begin{enumerate}
\item  $\mathcal{Q}^{r}_h(I,\vecd)$ is made of partitions with some $1\le j\le h$ so that $I_j = \emptyset$, $k_1(j)=r,  k_2(j) = 0$,
\item $\mathcal{Q}^{s}_h(I,\vecd)$ is made of partitions which contain some empty $I_j$ with $k_1(j) =0, k_2(j)=s$, and
\item $\mathcal{Q}^{\neq\emptyset}_h(I,\vecd)$ is made of partitions where $I_j$ is nonempty for all $j$.
\end{enumerate}
We define
\[
{\mathcal Q}^{\not=\emptyset}(I,\vecd) = \bigcup_{i=1}^l {\mathcal Q}_i^{\not=
\emptyset}(I,\vecd), \qquad
{\mathcal Q}^{\not=\emptyset}_{\le h}(I,\vecd) = \bigcup_{i=1}^h {\mathcal Q}_i^{\not=
\emptyset}(I,\vecd),
\]
and similar notation for ${\mathcal Q}^r(I,\vecd)$ and
${\mathcal Q}^s(I,\vecd)$.

We then also have various maps between these sets. First, we have
\begin{align}
\label{eq:Q r Q s}
\begin{split}
\mathcal{Q}_h^{\not=\emptyset}(I,\vecd)
\rightarrow \mathcal{Q}_{h+1}^r(I,\vecd),\quad\quad &
Q\mapsto Q^{+r}\\
\mathcal{Q}_h^{\not=\emptyset}(I,\vecd)
\rightarrow \mathcal{Q}_{h+1}^s(I,\vecd),\quad\quad&
Q\mapsto Q^{+s}
\end{split}
\end{align}
with
\begin{equation}\begin{aligned}\label{+r and +s definition}
{Q}^{+r} &=Q\cup \{(\emptyset,r,0)\} \\
{Q}^{+s} &=Q\cup \{(\emptyset,0,s)\},
\end{aligned}\end{equation}
with the convention on ordering the elements of $Q^{+r}$ or $Q^{+s}$
that the additional element is added at the end. We may define the inverse
maps
\[
Q^r_{h+1}(I,\vecd)\rightarrow \mathcal{Q}_h^{\not=\emptyset}(I,\vecd),
\quad\quad Q\mapsto Q^-
\]
via, if $Q\in\mathcal{Q}^{r}_{h+1}(I,\vecd)$ with $I_i=\emptyset$,
the element $Q^-$ is given by
\begin{equation}
\label{def:Q minus}
(I^{Q^{-}}_j,k_1^{Q^{-}}(j),k_2^{Q^{-}}(j))=\begin{cases}(I^{Q}_j,k_1^{Q}(j),k_2^{Q}(j)),~~\text{if $j<i$,}\\
(I^{Q}_{j+1},k_1^{Q}(j+1),k_2^{Q}(j+1)),~~\text{otherwise}.
\end{cases}
\end{equation}
Similarly we obtain
\[
Q^s_{h+1}(I,\vecd)\rightarrow \mathcal{Q}_h^{\not=\emptyset}(I,\vecd),
\quad\quad Q\mapsto Q^-.
\]

\end{nn}

\begin{ex}
    Consider the LG model $(x^4 + y^5, \mu_4 \times \mu_5)$. Take $|I| = 2$ where $(a_1, b_1) = (2,2)$ and $(a_2, b_2) = (2,3)$. Then we have examples where 
    \[ 
    \{(I, 1,1), (\emptyset, 4,0)\} \in \mathcal{Q}^{r}_1(I,\mathbf{0}); \qquad \{(I, 1,1), (\emptyset, 0,5)\} \in \mathcal{Q}^{s}_1(I,\mathbf{0}).
    \]
    The graphs corresponding to these partitions are the detachings of the graphs found in Figure~\ref{fig:balanced and crit graphs}(B) and (D). If we instead take $|I| = 3$ where $(a_1, b_1) = (1,1)$, $(a_2, b_2) = (2,2)$, and $(a_3, b_3) = (2,3)$, then we have 
    \[
    \{ (\{2,3\}, 1,1), (\{1\}, 1, 1) \} \in \mathcal{Q}^{\ne0}_1(I,\mathbf{0}).
    \]
    This graph corresponding to this partition is the detaching of the graph found in Figure~\ref{fig:exchange_and positive}(D). We remark that the vertex $v_2$ in the figure above serves as $v_0$ in the present notation.
\end{ex}

\begin{proof}[Proof of Theorem \ref{thm:A_mod_invs}]
We outline the proof by explaining the five steps we will take:
\begin{enumerate}[(A)]
\item We identify the strata on which a canonical homotopy may vanish: these are precisely those strata which correspond to critical boundary graphs in Definition~\ref{def:critical}. We then show these coincide with the graphs introduced
in Notation \ref{not:Gamma mess of subscripts}.
\item  We classify all the ways that we can glue the bases of critical boundaries to create critical boundaries so that they can be smoothed to a balanced graph.
\item We compute how much an introduction of a zero by a homotopy for a critical graph will change the invariants associated to such a moduli problem.
\item We provide closed formulas for how the invariants change with respect to the number of zeros in a given homotopy.
\item We show that, for any introduction of a zero through the homotopy for a critical graph, $\mathcal{A}(I,\vecd,\CI^{\ess})$ remains invariant.
\end{enumerate}

{\bf Part (A). Identification of strata where the homotopy may vanish
and classification of critical boundary graphs.}
Given two families of canonical multisections $\mathbf{s}_1$ and $\mathbf{s}_2$, we may compare them via the family of simple canonical homotopies $H$ constructed in Lemma~\ref{lem:partial_homotopy}. In this context,
as we will use Lemma \ref{lem:zero diff as homotopy} to understand
changes in open invariants appearing in
$\mathcal{A}(I,\vecd,\CI^{\ess})$, we identify boundary strata where $H$
may vanish.

To this end, start with a smooth balanced connected graph $\Gamma$
with respect to the descendent vector $\vecd$.
We now identify all boundary graphs $\Lambda\in \partial^0\Gamma$ such that the homotopy $H^{\Gamma}$ between $\mathbf{s}_1^{\Gamma}$ and
$\mathbf{s}_2^{\Gamma}$ may vanish in $[0,1]\times\CM^W_\Lambda $.

First, if $\Lambda \in \partial^0\Gamma$ is irrelevant, then
as in the first paragraph of Case 1 of the proof of Lemma~\ref{lem:partial_homotopy} a connected component of $\CB\Lambda$ is irrelevant as well. By Lemma~\ref{lem:partial_homotopy}(1)(a) and canonicity, $H^{\Gamma}$ will then not vanish on $[0,1]\times
\oPM_{\Lambda}$.

Second, suppose $\Lambda\in\partial^0\Gamma$ is relevant.
We claim that, for $H^{\Gamma}|_{[0,1]\times\CM^W_{\Lambda}}$ to vanish,
we must have that $\dim \oPM_\Lambda = \dim \oPM_\Gamma - 1$. Indeed, if $\dim \oPM_\Lambda < \dim \oPM_\Gamma - 1$, then in the notation of
Observation \ref{obs:base dim yoga}, $\nu\ge 3$. By the fact that
$\Gamma$ is balanced for $\vecd$, we also always have $\alpha=0$ and $\beta\ge\sigma$ in \eqref{eq:dim rank comparison}, and hence by
the pigeon-hole principal,
there exists a connected component $\Xi$ of $\CB\Lambda$ such that
$\dim \CM^W_\Xi \le  \rank E_{\Xi}(\vecd) - 2$ or there exists
two connected components $\Xi_1,\Xi_2$ of $\CB\Lambda$ such that
$\dim\CM^W_{\Xi_i}=\rank E_{\Xi_i}(\vecd)-1$ for $i=1,2$. In either case,
$H^{\CB\Lambda}$ does
not vanish on $\oPM_{\CB\Lambda}$ by Lemma~\ref{lem:partial_homotopy}(1)(b)
or (c),
and, in turn, $H^{\Gamma}$ does not vanish on $\oPM_\Lambda$.

Note that the only codimension one relevant boundary strata $\CM^W_{\Lambda}$ of $\oPM_{\Gamma}$ are those with a single boundary node.
Hence if $\Lambda\in\partial^0\Gamma$ is such that $H^{\Gamma}$ vanishes somewhere on $[0,1]\times\CM^W_{\Lambda}$, then $\Lambda$ consists of two open vertices connected by a boundary edge, and a number of boundary tails. Take $v_1$ to be the rooted vertex with corresponding boundary half-edge $h_1$ and let $v_2$ be the additional vertex with half-edge $h_2$.

Since $\Lambda$ is relevant we only have the following possibilities for the twist at the boundary edge:
\begin{enumerate}[(i)]
\item $\tw(h_1)= (r-2,0), \alt(h_1) = (1,0), \tw(h_2) = (0,s-2),$ and $\alt(h_2)=(0,1)$;
\item $\tw(h_1)= (0,s-2), \alt(h_1) = (0,1), \tw(h_2) = (r-2,0),$ and $\alt(h_2)=(1,0)$;
\item $\tw(h_1)= (r-2,s-2), \alt(h_1) = (1,1), \tw(h_2) = (0,0),$ and $\alt(h_2)=(0,0)$;
\item $\tw(h_1)= (0,0), \alt(h_1) = (0,0), \tw(h_2) = (r-2,s-2),$ and $\alt(h_2)=(1,1)$.
\end{enumerate}

When we pass to the base $\CB\Lambda$ in case (iii) (resp.\ (iv)), we forget the half-edge $h_2$ (resp.\ $h_1$), so $\beta=1$. Here we have two cases, depending on whether forgetting the half-edge will keep the corresponding vertex stable or make it partially stable, i.e., if $\sigma=0$ or $1$ in the notation of
Observation \ref{obs:base dim yoga}.
If the corresponding vertex remains stable, so that $\sigma=0$ in
\eqref{eq:dim rank comparison}, we again have one connected component $\Xi$ of $\CB\Lambda$ with
$\dim \CM^W_\Xi \le  \rank E_{\Xi}(\vecd) - 2$ or two connected components $\Xi_i$ of $\CB\Lambda$ with $\dim \CM^W_{\Xi_i} \le  \rank E_{\Xi_i}(\vecd) - 1$. Hence again by Lemma~\ref{lem:partial_homotopy}(1)(b)
or (c), 
$H^{\Gamma}$ is nowhere vanishing on $\oPM_{\Lambda}$.
If on the other hand $\sigma=1$, we may describe the partially stable vertex. Indeed, if it had two $r$-points or two $s$-points then the constraints given for a pre-graded $\RS$-spin structure in Observation~\ref{obs:open_rank1} do not hold. Consequently, it must have exactly one $r$-point and one $s$-point, which implies that 
$\Lambda$ is an exchangeable critical boundary graph. In this case, we have by Lemma~\ref{lem:exchange vanishes} that while $H^{\Gamma}$ may vanish on
$\oPM_{\Lambda}$, this vanishing will not contribute to a change
of the total number of zeroes between $\mathbf{s}_1^{\Gamma}$ and
$\mathbf{s}_2^{\Gamma}$. Thus we may ignore any contributions from these strata.\footnote{One could view these strata as irrelevant relevant boundary strata.}

It is left to analyze cases (i) and (ii) above. Indeed, both these cases can happen, but we can deduce further restrictions on the vertices $v_1$ and $v_2$, since $v_1$ is the vertex which has a root.
Using~\eqref{eq:open_rank2} we obtain
\[\rank\, E_{v_1}(\vecd)\equiv \dim\CM^W_{v_1} \pmod 2,~\rank\, E_{v_2}(\vecd)\equiv\dim\CM^W_{v_2}+1 \pmod 2.\]
 By Lemma \ref{lem:partial_homotopy}, this requires that
\[\rank \, E_{v_1}(\vecd)=\dim\CM^W_{v_1},~\rank\, E_{v_2}(\vecd)=\dim\CM^W_{v_2}+1\]
as otherwise $H$ will not vanish on $[0,1]\times\oPM_{\Lambda}$.

We summarize the above discussion. Suppose $H^{\Gamma}$ vanishes on
$[0,1]\times\oPM_{\Lambda}$ with $\Lambda$ not an exchangeable graph. 
Take the two connected components $\Xi_1$ and $\Xi_2$ of $\CB\Lambda$
corresponding to the vertices $v_1$ and $v_2$. Then $\Xi_1$ is a balanced graph and $\Xi_2$ is a critical graph, so $\Lambda$ is a critical boundary graph in the terminology of Definitions~\ref{def:balanced,critical,etc} and \ref{def:critical}. 

\medskip

We now remove the connectedness assumption on $\Gamma$. In this case we may
write $\Gamma = \Gamma_P$ for some $P \in\mathcal{P}_h(I,\vecd)$ as defined in Equation~\eqref{def: GammaP} in Notation~\ref{P sets of balanced}.
Let $\Lambda$ be a boundary graph for $\Gamma$. In order for $H^{\Gamma}$ to vanish on $[0,1]\times\oPM_\Lambda$, each connected component of $\Lambda$ must be either smooth or a critical boundary graph as deduced above. Moreover,
if there exists two connected components $\Lambda_1,\Lambda_2$ of $\Lambda$ that are critical boundary graphs, then by Lemma~\ref{lem:partial_homotopy}(1)(c),
$H^{\Gamma}$ is non-vanishing on $[0,1]\times\oPM_{\Lambda}$. 
Hence $\Lambda$ has at most one critical boundary component, thus
has only one boundary edge. This implies that, if the graph $\Lambda$ is not exchangeable and $H$ vanishes on $[0,1]\times\oPM_\Lambda$, then the graph $\CB \Lambda$ consists of $h$ graded, balanced, rooted vertices and one critical graded open vertex, and all the boundary half-edges of $\Lambda$ which are not tails are singly twisted.

From this description, it is now immediately clear that, given
$P\in \mathcal{P}_h(I,\vecd)$, any non-exchangeable critical boundary graph
for $P$ is of the form
$\Gamma_{P,j,\epsilon,I_0,k_1(0),k_2(0),R_0,S_0}$ of Notation
\ref{not:Gamma mess of subscripts}
with indices as given subject to the constraints of
\eqref{eq:index constraints}.\footnote{We remark that $k_i(0)$ must be non-negative as the graph $\Gamma_{P,j,\epsilon,I_0,k_1(0),k_2(0),R_0,S_0}$ must be relevant. Otherwise, the graph would satisfy condition (3) of Definition~\ref{def:special kind of graded graphs} and be irrelevant. } Further, in order for such a graph
to be a critical boundary graph, the vertex $v_0$ must be critical
and the vertex $v_j$ must be balanced. In particular, we then
have $\Lambda_Q\succ\CB\Gamma_{P,j,\epsilon,I_0,k_1(0),k_2(0),R_0,S_0}$
where, using the notation of Notation~\ref{nn:Q notation}, if
$P=\{(I_1,k_1(1),k_2(1)),\ldots,(I_h,k_1(h),k_2(h))\}$,
\begin{align}
\label{eq:Q mess of subscripts}
\begin{split}
Q=Q_{P,j,I_0,k_1(0),k_2(0)}:=\{(I_0,&k_1(0)+1,k_2(0)+1),(I_1,k_1(1),k_2(1)),\ldots,\\
&(I_j\setminus I_0, k_1(j)-k_1(0),k_2(j)-k_2(0)),\ldots,(I_h,k_1(h),k_2(h))\}
\in \mathcal{Q}_h(I,\vecd).
\end{split}
\end{align}

We write
\begin{equation}
\label{eq:crit def}
\Critrit(P):=\{\Lambda\in \partial\Gamma_P\,|\,\hbox{$\Lambda$ is a critical
boundary graph}\}.
\end{equation}
Elements of $\Critrit(P)$ are all of the form $\Gamma_{P, j, \epsilon, I_0,k_1(0),k_2(0), R_0, S_0}$ . We then define
\[
\Critrit(I, \vecd) := \bigcup_{P \in \mathcal{P}(I, \vecd)} \Critrit(P)
\]
to be the set of all critical boundary graphs that have the set $I$ of internal markings with descendent vector $\vecd$.

\medskip

{\bf Part (B). Classification of smooth graded balanced graphs whose bases of critical boundaries include a fixed critical graph.}
We now classify, given $Q\in\mathcal{Q}_h(I,\vecd)$, the critical boundary graphs with base $\Lambda_Q$. In particular, define
\[
A^\epsilon_Q:=\{\Gamma_{P, j, \epsilon,k_1(0),k_2(0), I_0, R_0, S_0}\,|\,
\CB\Gamma_{P, j, \epsilon, I_0,k_1(0),k_2(0), R_0, S_0} \prec \Lambda_Q\}.
\]
Note that $P\in{\mathcal P}_h(I,\vecd)$ is not fixed here. We may describe $A^{\epsilon}_Q$
depending on whether $Q$ lies in $\mathcal{Q}^r_h(I,\vecd)$,
$\mathcal{Q}^s_h(I,\vecd)$ or $\mathcal{Q}^{\not=\emptyset}_h(I,\vecd)$.

Given $Q \in \mathcal{Q}^{r}_h(I,\vecd)$,
there exists a unique $j\in [h]$ so that $(I_j, k_1(j), k_2(j)) = (\emptyset, r, 0)$.
Define
\begin{equation}
\label{eq:Qjhat r}
\begin{split}
\widehat Q_j:=
\{(I_1,k_1(1),k_2(1)),\ldots, &(I_{j-1},k_1(j-1),k_2(j-1)),\\
&(I_0,k_1(0)+r,k_2(0)),\ldots, (I_h,k_1(h),k_2(h))\}\in \mathcal{P}_h(I,
\vecd).
\end{split}
\end{equation}
We may take the graph $\Lambda_Q$ and attach an $s$-tail on the vertex $v_0$ to an $r$-tail on the vertex $v_j$ in order to obtain the critical boundary graph $\Gamma_{\widehat Q_j, j, 1, I_0,k_1(0),k_2(0), R_0,S_0}$ introduced
in Notation \ref{not:Gamma mess of subscripts}, where
\begin{itemize}
\item $R_0\in \{0, \ldots, r-1\}$, indicating how many $r$-points
lie after the root and before the node in the cyclic order associated
to the graph; and
\item $S_0=0$ as  there are no $s$-tails attached to $v_j$.
\end{itemize}
Note that the smoothing of $\Gamma_{\widehat Q_j, j, 1, I_0,k_1(0),k_2(0), R_0,S_0}$
agrees with $\Gamma_{\widehat Q_j}$ modulo cyclic orders.
Therefore, for $Q \in \mathcal{Q}_h^{r}(I,\vecd)$, we have that $A^{1}_Q$ is the collection of such critical boundary graphs $\Gamma_{\widehat Q_j, j, 1, I_0,k_1(0),k_2(0),  R_0,S_0}$ above and $A^2_Q = \varnothing$. Note that by the above discussion, the moduli space $\oPM_{\Gamma_{\widehat Q_j}}$ has $r$ boundary strata of the form $\oPM_{\Xi}$, with $\Xi \in A_Q^{1}$,
depending on the choice of $R_0$.

For $Q \in \mathcal{Q}_h^{s}(I,\vecd)$, there exists a unique $j \in [h]$ so that $(I_j, k_1(j), k_2(j)) = (\emptyset, 0, s)$. In this case, we have that $A^{2}_Q$ is the analogous collection of critical boundary graphs $\Gamma_{\widehat Q_j, j, 2, I_0,k_1(0),k_2(0), R_0,S_0}$ as above and $A^1_Q=\varnothing$,
with $R_0=0$, and $S_0 \in \{0, \ldots, s-1\}$.
Similarly as before,
\begin{equation}
\label{eq:Qjhat s}
\begin{split}
\widehat Q_j:=
\{(I_1,k_1(1),k_2(1)),\ldots, &(I_{j-1},k_1(j-1),k_2(j-1)),\\
&(I_0,k_1(0),k_2(0)+s),\ldots, (I_h,k_1(h),k_2(h))\}\in \mathcal{P}_h(I,
\vecd).
\end{split}
\end{equation}
The moduli space $\oPM_{\Gamma_{\widehat Q_j}}$ has $s$ boundary strata
$\oPM_{\Xi}$, with $\Xi \in A_Q^{2}$ depending on the choice of $S_0$.

Lastly, given $Q=\{(I_0,k_1(0)+1,k_2(0)+1),\ldots, (I_h,k_1(h),k_2(h))\}\in\mathcal{Q}^{\neq\emptyset}_h(I,\vecd)$, we may attach the vertex $v_0$ to any other vertex $v_1, \ldots, v_h$. We thus
define, for any $j \in [h]$,
 \begin{equation}\label{defn: hatQj}
\begin{split}
\widehat Q_j:=  \{(I_1,k_1(1),k_2(1)),\ldots, (I_0\cup I_j,k_1(0)+k_1(j),&
k_2(0)+k_2(j)),\\
&\ldots, (I_h,k_1(h),k_2(h))\}\in \mathcal{P}_h(I,\vecd).
\end{split}
 \end{equation}
This keeps all triplets the same from $1$ to $h$ except the $j$th one. There are potentially many ways to attach the vertices $v_0$ and $v_j$, depending on the number of $r$-points and $s$-points on the $j$th disk.
\begin{itemize}
\item If $k_1(j)>0$ then there is an $r$-point on $v_j$ to attach to an $s$-point on $v_0$, hence there is a critical boundary graph $\Gamma_{\widehat Q_j, j, 1, I_0,k_1(0),k_2(0), R_0, S_0}$ where $0\le R_0 \le k_1(j) - 1$ and $0 \le S_0 \le k_2(j)$. There are $k_1(j)(k_2(j)+1)$ such boundaries, forming
the set $A^1_Q$.
\item Analogously, if $k_2(j)>0$, then there is an $s$-point on $v_j$ to attach to an $r$-point on $v_0$, hence there is a critical boundary graph $\Gamma_{\widehat Q_j, j, 2, I_0,k_1(0),k_2(0), R_0, S_0}$ where $0\le R_0\le k_1(j)$ and $0\le S_0\le k_2(j) -1$. There are $(k_1(j) + 1)k_2(j)$ such boundaries,
forming the set $A^2_Q$.
\end{itemize}

Note now that
\begin{equation}\label{all the crits}
\bigcup_{Q \in \mathcal{Q}(I,\vecd), \epsilon \in \{1,2\}} A_Q^\epsilon = \Critrit(I, \vecd).
\end{equation}
The containment of the first set into the second is by our discussion above. The opposite containment is by the discussion in Part (A).

It will be useful to partition the sets $A^{\epsilon}_Q$ as follows.
Define the sets
\begin{equation}\begin{aligned}\label{defn of A^1QjR}
A^1_{Q, j, R_0} &= \{ \Gamma_{\widehat Q_j, j, 1,k_1(0),k_2(0), R_0, S_0} :  0 \leq S_0 \leq k_2(j)\};\\
A^2_{Q, j, S_0} &= \{ \Gamma_{\widehat Q_j, j, 2, k_1(0),k_2(0),R_0, S_0} :  0 \leq R_0 \leq k_1(j)\}.
\end{aligned} \end{equation}
Here, the set $A^1_{Q, j, R_0}$ (resp. $A^2_{Q, j, S_0}$) enumerates all possible choices of $S_0$ (resp. $R_0$) once we fix $Q$, $j$ and $R_0$ (resp. $S_0$).

To summarize:
\begin{itemize}
\item If $Q \in \mathcal{Q}^r_h(I, \vecd)$ and $I_j = \varnothing$, then we have that
\begin{equation}\label{A1Q for Qr}
A^1_Q = \bigcup_{R_0=0}^{r-1} A^1_{Q, j, R_0}; \quad A^2_Q = \varnothing.
\end{equation}
\item If $Q \in \mathcal{Q}^s_h(I, \vecd)$ and $I_j = \varnothing$, then we have that
\begin{equation}\label{A1Q for Qs}
A^1_Q = \varnothing; \quad A^2_Q = \bigcup_{S_0=0}^{s-1} A^2_{Q, j, S_0}.
\end{equation}
\item If $Q \in \mathcal{Q}^{\ne 0}_h(I, \vecd)$, then we have that
\begin{equation}\label{A1Q for Qne0}
A^1_Q = \bigcup_{j = 1}^h \bigcup_{R_0 = 0}^{k_1(j)-1} A^1_{Q, j, R_0}; \quad A^2_Q = \bigcup_{j=1}^h \bigcup_{S_0=0}^{k_2(j)-1} A^2_{Q, j, S_0}.
\end{equation}
\end{itemize}

We lastly reorganize the description of critical boundaries
using the notation $Q^{+r}$, $Q^{+s}$ of \eqref{+r and +s definition}.
We then have
\begin{equation}\label{partition of critrit}
\bigsqcup_{P \in \mathcal{P}(I, \vecd)} \Critrit(P) = \bigsqcup_{Q \in \mathcal{Q}^{\not=\emptyset}(I,\vecd)}
(A_Q^1 \sqcup A_Q^2 \sqcup A_{Q^{+r}}^1 \sqcup A_{Q^{+s}}^2).
\end{equation}

\medskip

{\bf Part (C) {Zeros of homotopies, from boundary to base}.}
Denote by $\oPM_{P, j, 1, I_0, k_1(0), k_2(0), R_0}$ (respectively $\oPM_{P, j, 2, I_0,k_1(0), k_2(0),  S_0}$) the following union of boundary strata of $\oPM_{\Gamma_{P}}$:
\begin{equation}\label{PartDUnions}
\oPM_{P, j, 1, I_0, k_1(0), k_2(0), R_0}:=  \bigcup_{\Xi \in A^1_{Q, j, R_0}} \oPM_{\Xi} \qquad \left(\textrm{resp. }\oPM_{P, j, 2, I_0, k_1(0), k_2(0), S_0}:=
\bigcup_{\Xi \in A^2_{Q, j, S_0}} \oPM_{\Xi}  \right),
\end{equation}
with $Q=Q_{P,j,I_0,k_1(0),k_2(0)}$ given in \eqref{eq:Q mess of subscripts}.

Recall that for $\Gamma_{P,j,1,I_0,k_1(0),k_2(0),R_0,S_0}\in A^1_Q$, we have
$\CB\Gamma_{P,j,1,I_0,k_1(0),k_2(0),R_0,S_0} \prec \Lambda_Q$.
Hence, if
$Q \in  \mathcal{Q}_h^r(I, \vecd) \cup  \mathcal{Q}_h^{\ne 0}(I, \vecd)$, we have a map
$$
F_{P, j, 1, I_0,k_1(0),k_2(0), R_0}:  \oPM_{P, j, 1, I_0,k_1(0),k_2(0), R_0}  \rightarrow  \oPM_{\Lambda_Q},
$$
which is defined to be $F_{P, j, 1, I_0, k_1(0),k_2(0),R_0}|_{\oPM_\Xi} = F_{\Xi}$ for any $\Xi \in A^1_{Q, j, R_0}$.
Here $F_{\Xi}$ is the map defined in Definition \ref{def:base}.
 The map $F_{P, j, 2, I_0,k_1(0),k_2(0), S_0}$ is defined analogously.

We can see that the map $F_{P, j, 1, I_0,k_1(0),k_2(0), R_0}$ is surjective. Indeed, an element of $\oPM_{\Lambda_Q}$ is a disjoint union of
graded $W$-spin disks $\cup_{i=0}^h \Sigma_i$. We may then obtain
an element of the domain of $F_{P, j, 1, I_0,k_1(0),k_2(0),R_0}$ which maps to this given
element of $\oPM_{\Lambda_Q}$ as follows.
In the cyclic order of points on
$\Sigma_j$, we glue the $(R_0+1)^{st}$ $r$-point on $\Sigma_j$ after the root to
some choice of $s$-point on $\Sigma_0$. This, in turn, determines the number $S_0$ of $s$-points that follow the root and precede the glued node.
Then the resulting glued disk  lies in
\[
\oPM_{\Gamma_{P, j, 1, I_0,k_1(0),k_2(0), R_0, S_0}} \subset \oPM_{P, j, 1, I_0,k_1(0),k_2(0), R_0}.
\]
Since we had $k_2(0)+1$ choices of $s$-point on the vertex $v_0$ to attach to the vertex $v_j$, this map is of degree $\pm (k_2(0)+1)$. Analogously,  $F_{P, j, 2, I_0,k_1(0),k_2(0), S_0}$ is surjective with degree $\pm (k_1(0) + 1)$.

We now aim to relate the orientation of the total space of the Witten bundle over $\oPM_{\Gamma_P}$ near the boundary $\oPM_{{P,j,1,I_0,k_1(0),k_2(0),R_0}}$ to the orientation of the Witten bundle over  $\oPM_{\Lambda_Q}$.
To do so, we consider the following diagram:
\begin{equation}\begin{tikzcd}\label{eq:instead_of_product}
\oPM_{{P,j,1,I_0,k_1(0),k_2(0),R_0}}  \arrow{d}{F_{{P,j,1,I_0,k_1(0),k_2(0),R_0}}} \arrow{r}{\iota_{\Gamma_P}} &  \oPM_{\Gamma_P}   \\
\oPM_{\Lambda_Q} &
\end{tikzcd}
\end{equation}

Here, Theorem \ref{thm:or_and_induced boundary}(\ref{part 2: orientation boundary})  relates the orientation on the total space of the Witten bundle over $\oPM_{{P,j,1,I_0,k_1(0),k_2(0),R_0}}$ to that over $\oPM_{\Lambda_Q}$. This is done by inducing the orientation on the total space of the Witten bundle over $\oPM_{{P,j,1,I_0,k_1(0),k_2(0),R_0}}$ by contracting $\iota_{\Gamma_P}^*o_{\Gamma_P}$ with the outward pointing normal for the inclusion $\iota_{\Gamma_P}$. Note that the term $K_2(v_2)$ appearing in Theorem \ref{thm:or_and_induced boundary}(\ref{part 2: orientation boundary}) is defined in Notation \ref{nn:K1 K2}, and here takes the value $k_2(0)$.

We then can conclude that if we consider a homotopy $H^{\Gamma}$ where
\begin{equation}\label{eq:from boundary to base}
H^\Gamma|_{[0,1]\times\oPM_{\Gamma_{P,j,\eps,I_0,k_1(0),k_2(0),R_0,S_0}}}=(
\mathrm{id}_{[0,1]}\times F_{\Gamma_{P,j,\eps,I_0,k_1(0),k_2(0),R_0,S_0}})^*\boxplus_{\Lambda\in\Conn(\CB\Gamma_{P,j,\eps,I_0,k_1(0),k_2(0),R_0,S_0})} H^{\Lambda},\end{equation}
then, by taking union over boundary strata given in Equation~\eqref{PartDUnions}, we see that
\begin{equation} \begin{aligned}\label{eqn for zeros in the homotopy}
\#Z(H^\Gamma|_{[0,1]\times\oPM_{{P,j,1,I_0,k_1(0),k_2(0),R_0}}}) &= (-1)^{k_2(0)} (k_2(0) + 1)\#Z(H^{\Lambda_Q}),\\
\#Z(H^\Gamma|_{[0,1]\times\oPM_{{P,j,2,I_0,k_1(0),k_2(0),S_0}}}) &=  (-1)^{k_2(0)-1}(k_1(0) + 1)\#Z(H^{\Lambda_Q}).
\end{aligned}\end{equation}

\medskip

{\bf Part (D): Number of zeros for a given simple homotopy.} 
Recall that we have chosen the family of transverse homotopies
$H^*$ between the given
families of  transverse multisections $\mathbf{s}_1$ and $\mathbf{s}_2$
to be simple.
By Part (A) of this proof, we know that this homotopy may only vanish in two cases: exchangeable critical boundaries and boundaries corresponding to graphs $\Gamma_{P, j, \epsilon, I_0,k_1(0),k_2(0), R_0,S_0}$. By Lemma~\ref{lem:exchange vanishes}, we know that 
$$
\#Z(H|_{[0,1]\times\partial^\xch\oPM_\Gamma}) = 0.
$$
Thus, from here until the end of the proof we will only need to consider homotopies of the form given in ~\eqref{eq:from boundary to base}. Thus we are in a context where we may use~\eqref{eqn for zeros in the homotopy} above.

We now claim that, in the notation of
\eqref{def:Q minus},
given $Q \in\mathcal{Q}^{r}_{h+1}(I, \vecd) \cup \mathcal{Q}^{s}_{h+1}(I, \vecd)$ where $I_i = \varnothing$,
\begin{equation}\label{comparing with the simple deletion}
\#Z(H^{\Lambda_Q})=-\#Z(H^{\Lambda_{Q^-}}).
\end{equation}
Indeed, recall from Example \ref{ex:emptyboundary} that, for the graphs $\Xi=\Gamma_{0,r,0,1,\emptyset}$ and $\Xi=\Gamma_{0,0,s,1,\emptyset}$, the boundary  $\partial^0\oPM_\Xi=\emptyset$, hence
$\#Z(H^\Xi|_{\{t\}\times\oPM_\Xi})$ is independent of $t$  by Lemma~\ref{lem:zero diff as homotopy}. By Theorem \ref{thm:simple invariants},  the constant invariant associated to this component is
\[\langle\sigma_1^{r}\sigma_{12}\rangle^{o}=\langle\sigma_2^{s}\sigma_{12}\rangle^{o}=-1.\]
Thus $\#Z(H^{\Lambda_Q}) = \#Z(H^{\Xi}|_{\{t\}\times\oPM_{\Xi} })\cdot \#Z(H^{\Lambda_{Q^-}}) = - \#Z(H^{\Lambda_{Q^-}})$, proving \eqref{comparing with the simple deletion}.

We now give closed formulas for combinations of
$\# Z(H^{\Xi})$ for various graphs
$\Xi$ which contribute to changes in the open invariants appearing
in $\mathcal{A}(I,\vecd,\CI^{\ess})$.
For the formulas below, take $\bullet$ to mean the smoothing $\smooth(\Xi)=\smooth_{E(\Xi)}\Xi,$ for any critical boundary graphs $\Xi$ that can be found in Equation~\eqref{all the crits} above.

Consider any $Q=\{(I_0,k_1(0)+1,k_2(0)+1),\ldots, (I_{h+1},k_1(h+1),k_2(h+1))\}\in\mathcal{Q}^{r}_{h+1}(I,\vecd)$. Then  $I_j = \emptyset$ for some $1\le j \le h+1$ and we can use Equations~\eqref{eqn for zeros in the homotopy} and~\eqref{comparing with the simple deletion} to see that

\begin{equation}\begin{aligned}\label{eq:wc3}\sum_{\Xi\in A^1_Q}\#Z(H^\bullet|_{[0,1]\times\oPM_\Xi}) &= \sum_{R_0 = 0}^{r-1} \   \sum_{\Xi\in A^1_{Q, j, R_0}} \#Z(H^\bullet|_{[0,1]\times\oPM_\Xi})
\\
	%&= \sum_{R_0 = 0}^{r-1}  \deg(F_{{P,j,1,I_0,k_1(0),k_2(0),R_0}})\cdot\#Z(H^{\Lambda_Q}) \\
	&= \sum_{R_0 = 0}^{r-1}  (-1)^{k_2(0)}(k_2(0)+1)\cdot\#Z(H^{\Lambda_Q})\\
	%&= r \cdot \deg(F_{{P,j,1,I_0,R_0}})\cdot\#Z(H^{\Lambda_Q})\\
	%&= r(-1)^{k_2(0)}(k_2(0)+1)\#Z(H^{\Lambda_Q})\\
	&=(-1)^{k_2(0)+1}\cdot r(k_2(0)+1)\cdot \#Z(H^{\Lambda_{Q^-}}).\end{aligned}\end{equation}

Similarly, for
any $Q=\{(I_0,k_1(0)+1,k_2(0)+1),\ldots, (I_{h+1},k_1(h+1),k_2(h+1))\}\in\mathcal{Q}^{s}_{h+1}(I,\vecd)$,

\begin{equation}\label{eq:wc4}\sum_{\Xi\in A^2_Q}\#Z(H^\bullet|_{[0,1]\times\oPM_\Xi})=(-1)^{k_2(0)}\cdot s(k_1(0)+1)\cdot\#Z(H^{\Lambda_{Q^-}}).\end{equation}

For any $Q=\{(I_0,k_1(0),k_2(0)),\ldots, (I_h,k_1(h),k_2(h))\}\in\mathcal{Q}^{\neq\emptyset}_h(I,\vecd)$, then the contributions from critical boundary graphs in $A^1_Q$ are
\begin{equation}\begin{aligned}\label{eq:wc1}
\sum_{\Xi\in A^1_Q}\#Z(H^\bullet|_{[0,1]\times\oPM_\Xi}) &= \sum_{j = 1}^h \ \sum_{R_0 = 0}^{k_1(j)-1}  \ \sum_{\Xi \in A^1_{Q, j, R_0}} \#Z(H^\bullet|_{[0,1]\times\oPM_\Xi}) \\
	%&= \sum_{j = 1}^h \ \sum_{R_0 = 0}^{k_1(j)-1}  \ \deg(F_{{P,j,1,I_0,k_1(0),k_2(0),R_0}})\#Z(H^{\Lambda_Q}) \\
	%&= \sum_{j=1}^h k_1(j)\deg(F_{{P,j,1,I_0,R_0}})\#Z(H^{\Lambda_Q}) \\
	&= \sum_{j = 1}^h\ \sum_{R_0 = 0}^{k_1(j)-1}  \left((-1)^{k_2(0)}(k_2(0)+1)\right)\#Z(H^{\Lambda_Q}) \\
	&=(-1)^{k_2(0)}(k_2(0)+1)\left(\sum_{j=1}^h k_1(j)\right)\#Z(H^{\Lambda_Q}).
\end{aligned}\end{equation}

Similarly,  for $A^2_Q$ we have
\begin{equation} \label{eq:wc2}
\sum_{\Xi\in A^2_Q}\#Z(H^\bullet|_{[0,1]\times\oPM_\Xi}) = (-1)^{k_2(0)+1}(k_1(0)+1)\left(\sum_{j=1}^h k_2(j)\right)\#Z(H^{\Lambda_Q}).
\end{equation}

\medskip

{\bf Part (E): Showing that $\mathcal{A}(I,\vecd,\CI^{\ess})$ is independent of $\mathbf{s}$.}  We now show that
\begin{equation}\label{goal for invariants}
\mathcal{A}(I,\vecd, \CI^{\ess_1}) - \mathcal{A}(I,\vecd, \CI^{\ess_2}) = 0.
\end{equation}
for any two canonical families of multisections $\mathbf{s}_1$ and $\mathbf{s}_2$. This is equivalent to showing that $\mathcal{A}(I,\vecd, \CI^{\ess})$ remains invariant under the family of homotopies $H$.
Only the invariants corresponding to graphs $\Gamma_P$ for $P\in \mathcal{P}(I,\vecd)$ where there exists a critical boundary $\Gamma_{P, j, \eps, I_0, k_1(0),k_2(0),R_0, S_0} \in \Critrit(I,\vecd)$ may change. In particular, we can see that the left-hand side of Equation~\eqref{goal for invariants} is of the form
\begin{align}\label{goalsum for invariants}
\begin{split}
\sum_{\substack{\Gamma_{P,j,\eps,I_0,k_1(0),k_2(0),R_0,S_0}\in \Critrit(I,\vecd),\\ P=\{(I_1,k_1(1),k_2(1)),\ldots,(I_h,k_1(h),k_2(h))\}\in  \mathcal{P}(I,\vecd)}}  \frac{1}{h!}&\frac{\Gamma(\frac{1+\sum_{i=1}^{h}k_1(i)}{r}) \Gamma(\frac{1+\sum_{i=1}^h k_2(i)}{s})}{\Gamma(\frac{1+r(I)}{r}) \Gamma(\frac{1+s(I)}{s})}\cdot\\
&\cdot
\#Z(H^{\Gamma_P}|_{[0,1]\times
\oPM_{\Gamma_{P,j,\eps,I_0,k_1(0),k_2(0),R_0,S_0}}}).
\end{split}
\end{align}

We see that this summation breaks down further into summands that correspond to different possible bases of critical boundaries. Take $\overline{\mathcal{Q}}(I,\vecd)$ to be the collection of \[Q=\{(I_0,k_1(0)+1,k_2(0)+1),\ldots,(I_h,k_1(h),k_2(h))\} \in {\mathcal Q}_h(I,\vecd), \qquad h\geq0\] such that
\begin{itemize}
\item $I_j \ne \varnothing $ for all $0 \le j \le h$;
\item the elements $(I_1,k_1(1),k_2(1)),\ldots,(I_h,k_1(h),k_2(h))$ are ordered in a way so that, for any two indices $i$ and $j$, we have that $\text{min}(I_i) < \text{min}(I_j)$ if and only if $i<j$.
\end{itemize}
 In the case $h=0$, we just have the data of a critical graph with internal markings $I$. Note that if $h\geq 1$, then $Q$ is an element of $\mathcal{Q}^{\neq\emptyset}(I,\vecd)$.
Moreover, in this case, if we consider the action of the symmetric group
$S_h$ on $\mathcal{Q}^{\neq\emptyset}_h(I,\vecd)$ given by permuting
the last $h$ elements in the ordered partitions in
$\mathcal{Q}^{\neq\emptyset}_h(I,\vecd)$, we can see that
$\overline{\mathcal{Q}}(I,\vecd)$ contains a unique element of each
orbit of this $S_h$-action.
Thus $\overline{\mathcal{Q}}(I,\vecd)$ has precisely one element
for each graph $\Lambda_Q$, with the order of connected
components not taken into account.

Recall the notation, from \eqref{+r and +s definition},
$Q^{+r} \in \mathcal{Q}^{r}(I,\vecd)$ and $Q^{+s} \in \mathcal{Q}^{s}(I,\vecd)$ for any $Q \in \overline{\mathcal{Q}}(I,\vecd)$.  We see for any $Q \in \overline{\mathcal{Q}}(I,\vecd)$ that, by combining Equations~\eqref{eq:wc3} and~\eqref{eq:wc1},
 \begin{equation}
 \sum_{\Xi\in A^1_Q}\#Z(H^{\smooth_{E(\Xi)}(\Xi)}|_{[0,1]\times\oPM_\Xi})+\frac{\sum_{i=1}^hk_1(i)}{r}\sum_{\Xi\in A^1_{Q^{+r}}}\#Z(H^{\smooth_{E(\Xi)}(\Xi)}|_{[0,1]\times\oPM_\Xi}) = 0,
\end{equation}
and, by combining Equations~\eqref{eq:wc4} and~\eqref{eq:wc2},
 \begin{equation}
 \sum_{\Xi\in A^2_Q}\#Z(H^{\smooth_{E(\Xi)}(\Xi)}|_{[0,1]\times\oPM_\Xi})+\frac{\sum_{i=1}^hk_2(i)}{s}\sum_{\Xi\in A^2_{Q^{+s}}}\#Z(H^{\smooth_{E(\Xi)}(\Xi)}|_{[0,1]\times\oPM_\Xi}) = 0.
\end{equation}
Consequently, using the above two equations and \eqref{eq:wc3}, \eqref{eq:wc4}
applied with $Q=Q^{+r}$ or $Q=Q^{+s}$, we get a new relation
\begin{align}\label{eq:wc_local_cancellation}
\sum_{\Xi\in A^1_Q\cup A^2_Q}\#Z(H^{\smooth_{E(\Xi)}(\Xi)}&|_{[0,1]\times\oPM_\Xi})+\frac{1+k_1(0)+ \sum_{i=1}^hk_1(i)}{r}\sum_{\Xi\in A^1_{Q^{+r}}}\#Z(H^{\smooth_{E(\Xi)}(\Xi)}|_{[0,1]\times\oPM_\Xi})\\\notag&+ \frac{1+k_2(0) + \sum_{i=1}^hk_2(i)}{s}\sum_{\Xi\in A^2_{Q^{+s}}}\#Z(H^{\smooth_{E(\Xi)}(\Xi)}|_{[0,1]\times\oPM_\Xi})  \\
\notag&= ((-1)^{k_2(0)}+(-1)^{k_2(0)+1})(k_1(0)+1)(k_2(0)+1)\#Z(H^{\Lambda_{Q}})=0.
\end{align}

Thus, the difference expressed in Equation~\eqref{goalsum for invariants} simplifies to:
\begin{align*}
&\sum_{\substack{Q\in\overline{\mathcal{Q}}(I,\vecd),\\
Q=\{(I_0,k_1(0),k_2(0)),\ldots,(I_h,k_1(h),k_2(h))\}}}\frac{\Gamma(\frac{1+\sum_{i=0}^{h}k_1(i)}{r}) \Gamma(\frac{1+\sum_{i=0}^h k_2(i)}{s})}{\Gamma(\frac{1+r(I)}{r}) \Gamma(\frac{1+s(I)}{s})}\left(\sum_{\Xi\in A^1_Q\cup A^2_Q}\#Z(H^\bullet|_{[0,1]\times\oPM_\Xi})\right.
\\&\left.+\frac{1+\sum_{i=0}^hk_1(i)}{r}\sum_{\Xi\in A^1_{Q^{+r}}}\#Z(H^\bullet|_{[0,1]\times\oPM_\Xi})+ \frac{1+\sum_{i=0}^hk_2(i)}{s}\sum_{\Xi\in A^1_{Q^{+s}}}\#Z(H^\bullet|_{[0,1]\times\oPM_\Xi})\right)=0,
\end{align*}
where $\bullet =\smooth_{E(\Xi)}\Xi$.
The fact that the expression on the left-hand side above agrees with that of 
\eqref{goalsum for invariants}
follows by: (1) using the partition found in~\eqref{all the crits};
(2) noting that we are summing the $k_\epsilon(i)$ from $0$ to $h$ rather
than $1$ to $h$ because we are summing over possible $Q$ rather than $P$;
(3) the standard relation $\Gamma(x+1)=x\Gamma(x)$. The
stated vanishing uses the relation in~\eqref{eq:wc_local_cancellation}. Therefore, the expression in Notation~\ref{nn:A(J)} is invariant with respect to choice of family of canonical multisections $\mathbf{s}$.
\end{proof}

\noindent We now show that the open FJRW invariants corresponding to two families of canonical multisections are related by the action of an element in the Landau-Ginzburg wall-crossing group.

\begin{thm}\label{thm: group actions only}
Let $\mathbf{s}_0,\mathbf{s}_1$ be two families of canonical multisections of the descendent Witten bundles.
Then there exists an element $g\in G_{A_I}^{r,s}$ such that
\[
W^{\mathbf{s}_1}=g(W^{\mathbf{s}_0}).
\]
\end{thm}

\begin{proof}
By Lemma \ref{lem:partial_homotopy}, there exists a family of simple canonical homotopies $H$ from $\ess_0$ to $\ess_1$ that satisfies the properties stated therein. By item (3) of that lemma,
we have for each non-empty $A\subseteq I$ a time $t_A\in (0,1)$
so that properties (3)(a) and (b) are satisfied. Further, all these times are
distinct. Order the non-empty subsets $A_1,\ldots,A_{2^{|I|}-1}$ of $I$
so that $t_{A_1}<\cdots < t_{A_{2^{|I|}-1}}$. We write $t_i:=t_{A_i}$.

For any descendent
vector $\vecd\in \NN^{A_i}$ and $\Gamma\in \CRIT(A_i,\vecd)$, we define
\begin{equation}\label{defn delta local}
\Delta(\Gamma):= \#Z(H^{\Gamma}).
\end{equation}
Note by property (3)(a), these zeroes occur at time $t_i$. We also remark that the descendent vector $\vecd$ is implicitly used and fixed in the construction of $H^\Gamma$ so we omit it in the notation.

Now set
\begin{align*}
g_i:=
\exp\Bigg( \sum_{\vecd\in \NN^{A_i}} & \sum_{\{(A_i,k_1(0)+1,k_2(0)+1)\}
\in \mathcal{Q}_0(A_i,\vecd)} (-1)^{|A_i|+k_2(0)}\Delta(\Lambda_{Q,0})\cdot\\
&\cdot\Bigg(\prod_{j\in A_i} u_{j,d_j}\Bigg)x^{k_1(0)}y^{k_2(0)}
((k_2(0)+1)x\partial_x - (k_1(0)+1)y\partial_y)\Bigg).
\end{align*}
Note that  $g_i\in G^{r,s}_{A_I}$ by the definition of
$\mathcal{Q}_0(A_i, \vecd)$.
We will ultimately show that the following choice of $g$ gives the desired
result:
\[
g:= g_{2^{|I|}-1}\circ\cdots \circ g_1.
\]
For any $t \in [0,1]$, we write $\mathbf{s}_t$
for the family of canonical sections obtained by
restricting $H$ to time $t$.

We first claim that, for any
\[
t\in \tau:=[0,1]\setminus \{t_i\,|\,1\le i\le 2^{|I|}-1\},
\]
and $\Gamma$ a balanced graph with respect to some descendent
vector $\vecd$,  $\mathbf{s}^{\Gamma}_t|_{\partial\oPM_{\Gamma}}$
is non-vanishing, and hence the degree of the Euler class
\[
Z^{\Gamma}(t):=\int_{\oPM_{\Gamma}} e(E;\ess^{\Gamma}_t|_{U_+
\cup \partial\oPM_{\Gamma}})
\]
is well-defined. 

To show the claim, we  now fix some $J\subseteq I$, $\vecd\in \NN^J$ and
$\Gamma\in \INT(J,\vecd)$. Thus $\Gamma = \Gamma_{P}$ for some $P = \{ (J, k_1(\Gamma), k_2(\Gamma))\} \in \mathcal{P}_1(J, \vecd)$. The analysis of the proof of Theorem \ref{thm:A_mod_invs}
tells us that the zeros of $H^{\Gamma}|_{[0,1]\times\partial\oPM_{\Gamma}}$
occur only on two sorts of codimension one boundary strata. 

The
first type are codimension one exchangeable boundary strata. If $\Lambda
\in \partial^{\xch}\Gamma$ represents a codimension one stratum,
then $\Lambda$ has two vertices $v_0$ and $v_1$, with $v_0$ exchangeable,
so that $I(v_0)=\varnothing$, and thus $I(v_1)=J$. In particular,
by canonicity of $H^{\Gamma}$, $H^{\Gamma}|_{[0,1]\times\oPM_{\Lambda}}$
can vanish only when $H^{v_1}$ vanishes, which occurs at time
$t_{J}$.

The second type are codimension one boundary strata corresponding to graphs $\Lambda\in\partial^0\Gamma$ described as follows. For each $Q\in \mathcal{Q}_1(J,\vecd)$,
we constructed sets $A^1_Q, A^2_Q$ of critical boundary graphs.
All boundary zeroes of $H^{\Gamma}$ on non-exchangeable strata
lie on some strata $\oPM_{\Lambda}\subseteq \oPM_{\Gamma}$ where $\Lambda\in
A^1_Q\cup A^2_Q$ for some $Q\in Q_1(J,\vecd)$ with $\widehat{Q}_1=P$.
In particular, we consider those $Q$ of the form
\begin{equation}
\label{eq:special Q}
Q=\{(I_0, k_1(0)+1,k_2(0)+1), (I_1,k_1(1),k_2(1))\}
\end{equation}
with $I_0\not=\varnothing$, $k_i(0)+k_i(1)=k_i(\Gamma)$,
and $I_0\cup I_1=J$.
Then by Lemma \ref{lem:partial_homotopy}(3)(a) and canonicity of $H^*$,
if $\Lambda\in A^1_Q\cup A^2_Q$, then
$H^{\Lambda}$ only vanishes at time $t_{I_0}$. Thus
we have now seen in particular that ${\bf s}^\Gamma_t|_{\partial\oPM_{\Gamma}}$
is non-vanishing whenever $t\in\tau$, hence showing the claim.

Note that for those $t$ for which $\ess^{\Gamma}_t$ is
transversal to the zero section, we have
\[
Z^{\Gamma}(t)=\#Z(\ess^{\Gamma}_t).
\]
For each $t\in \tau$, we obtain
a potential $W^{\mathbf{s}_t}$, and by Lemma \ref{lem:zero diff as homotopy},
this will only depend on the connected component of $\tau$ containing
$t$ by the claimed non-vanishing of such $\ess^{\Gamma}_t$ on
$\partial\oPM_{\Gamma}$.
For each $i$, choose $t_i^{\pm}$
with $t_{i-1}<t_i^-<t_i<t_i^+<t_{i+1}$ (taking $t_0=0$ and $t_{2^{|I|}}=1$).
Then it will be sufficient to show that
$W^{\mathbf{s}_{t_i^+}}=g_i(W^{\mathbf{s}_{t_i^-}})$.

By Lemma \ref{lem:zero diff as homotopy},
\[
Z^\Gamma(t_i^+)=Z^{\Gamma}(t_i^-)+\#
Z(H^{\Gamma}|_{[t_i^-,t_i^+]\times\partial\oPM_{\Gamma}
}).
\]
Now $\partial\oPM_{\Gamma}=\partial^0\oPM_{\Gamma}\cup
\partial^{\xch}\oPM_{\Gamma}$, and
by Lemma~\ref{lem:exchange vanishes} we have that
\[
\#Z(H^\Gamma|_{[t_i^-,t_i^+]\times\partial^\xch\oPM_\Gamma}) = 0,
\]
hence
\begin{equation}\label{eq:WC_basic}
Z^\Gamma(t_i^+)=Z^\Gamma(t_i^-)+
\#Z(H^\Gamma|_{[t_i^-,t_i^+]\times\partial^0\oPM_\Gamma}).
\end{equation}
Thus we only get contributions to $Z^\Gamma(t_i^+)-Z^\Gamma(t_i^-)$ from
those $Q$ as in \eqref{eq:special Q} with $I_0=A_i$.

Thus, we can use \eqref{eq:wc3}, \eqref{eq:wc4},  \eqref{eq:wc1}, and  \eqref{eq:wc2},
and compute that, with the summation over $Q$
ranging over those $Q$ as in \eqref{eq:special Q}
with $I_0=A_i$,
\begin{align}
\label{eq:Z diff}
\begin{split}
Z^{\Gamma}(t_i^+)-Z^{\Gamma}(t_i^-)= {} &
\#Z(H^{\Gamma}|_{[t_i^-,t_i^+]\times\partial^0\oPM_{\Gamma}})\\
= {} &
\sum_Q\sum_{\Xi\in A^1_Q\cup A^2_Q} \#Z(H^{\Gamma}|_{[t_i^-,t_i^+]\times
\oPM_{\Xi}})\\
= {} &
\sum_Q\bigg((-1)^{k_2(0)}(k_2(0)+1)k_1(1)
+(-1)^{k_2(0)+1}(k_1(0)+1)k_2(1)\bigg)\#Z(H^{\Lambda_Q}).
\end{split}
\end{align}
 
 Using the definition of $\Delta$ in~\eqref{defn delta local} and the fact that
$H^{\Lambda_{Q,1}}$ does not have any zeros near time $t_i$ on the
boundary of $\partial \oPM_{\Lambda_{Q,1}}$ as
$A_i\not\subseteq I(\Lambda_{Q,1})$, we can see that
\begin{equation}
\label{eq:Z diff2}
\#Z(H^{\Lambda_Q})=\#Z(H^{\Lambda_{Q,0}}) \cdot \#Z({\bf s}^{\Lambda_{Q,1}}_{t_i})
=\Delta(\Lambda_{Q,0}) Z^{\Lambda_{Q,1}}(t_i^-),
\end{equation}

We now calculate the coefficient of $x^{k_1(\Gamma)}y^{k_2(\Gamma)}\prod_{j\in
J} u_{j,d_j}$ in $g_i(W^{\mathbf{s}_{t_i^-}})$ and show it coincides
with the corresponding coefficient in $W^{\mathbf{s}_{t_i^+}}$.
We write $g_i=\exp(v_i)$. As the monomial $\prod_{j\in A_i} u_{j,d_j}$ has square zero, $v_i \circ v_i=0$, hence $g_i=\mathrm{Id}+v_i$.
The term with monomial $x^{k_1(\Gamma)} y^{k_2(\Gamma)}
\prod_{j\in J} u_{j,d_j}$ in $W^{\mathbf{s}_{t_i^-}}$ is
\[
(-1)^{|J|-1} Z^{\Gamma}(t_i^-)
x^{k_1(\Gamma)} y^{k_2(\Gamma)}
\prod_{j\in J} u_{j,d_j},
\]
and applying $\mathrm{Id}$ to this gives the same term.
All other terms involving the same monomial arise from applying a summand
of $v_i$ arising from $Q_0:=\{(A_i,k_1(0)+1,k_2(0)+1)\}\in \mathcal{Q}_0(A_i,
\vecd)$ to a term of $W^{\mathbf{s}_{t^-_i}}$ arising from
a balanced (with respect to $\vecd$) graph $\Gamma_{P_1}$ with
\[
P_1= \{(J\setminus A_i, k_1(1),k_2(1))\}
\]
with $k_i(0)+k_i(1)=k_i(\Gamma)$. Thus $Q=Q_0\cup P_1
\in \mathcal{Q}_1(J,\vecd)$ and $\widehat{Q}_1=P$, and all such $Q$ with
$I_0=A_i$ and $\widehat{Q}_1=P$ occur in this way.
Thus the coefficient of $x^{k_1(\Gamma)}y^{k_2(\Gamma)}\prod_{j\in J}
u_{j,d_j}$, after applying $g_i$, becomes, with again the sum over
$Q$ ranging over those $Q$ as in \eqref{eq:special Q} with $I_0=A_i$,
\begin{align*}
&(-1)^{|J|-1} Z^{\Gamma}(t_i^-) +\\
&\quad\quad\quad
+\sum_Q
\Big((-1)^{(|A_i|+k_2(0))+(|J\setminus A_i|-1)}\big(
(k_2(0)+1)k_1(1)
 -(k_1(0)+1)k_2(1)\big)
\Delta(\Lambda_{Q,0})Z^{\Lambda_{Q,1}}(t_i^-) \Big)\\
= {} &
(-1)^{|J|-1}\left( Z^{\Gamma}(t_i^-) +
\sum_Q
\Big((-1)^{k_2(0)}\big( (k_2(0)+1)k_1(1) -(k_1(0)+1)k_2(1)\big)
\Delta(\Lambda_{Q,0})Z^{\Lambda_{Q,1}}(t_i^-) \Big)\right)\\
= {} & (-1)^{|J|-1} Z^{\Gamma}(t_i^+)
\end{align*}
by \eqref{eq:Z diff} and \eqref{eq:Z diff2}. This proves the desired result.
\end{proof}

\subsection{Open topological recursion relations and mirror symmetry: statement
of results}

This subsection states the open topological recursion
relations and proves their immediate corollaries. This is expressed in terms of the quantities
${\mathcal A}(I,\vecd, \CI^{\ess})$, which by Theorem~\ref{thm:A_mod_invs} is known to be independent of $\ess$, so we will now refer to it as just $\mathcal{A}(I, \vecd)$. For ease of notation in stating the topological recursion relations, we take the index set of interior
markings $I$ to coincide with $[\ell]$. For a descendent vector $\vecd=(d_1,\ldots,d_{\ell})$,
we write ${\bf e}_1=(1,0,\ldots,0)$, so that
$\vecd+{\bf e}_1=(d_1+1,d_2,\ldots,d_{\ell})$.

\begin{thm}
\label{thm:open TRR}
Given a canonical family of multisections ${\bf s}$, the following
identities hold. If $|I|\ge 1$, then
\begin{align}\label{eq:calculation_1}
\mathcal{A}&(I,\vecd+\eee_1)=\\
&\notag\quad\quad\sum_{\substack{a\in\{-1,\ldots,r-2\}\\b\in\{-1,\ldots,s-2\}}}\sum_{\substack{A \coprod B = \{2,\ldots,l\}\\ A\not=\varnothing}}\left\langle \tau_0^{(a,b)}\tau_{d_1}^{(a_1,b_1)}\prod_{i \in A}\tau_{d_i}^{(a_i,b_i)}\right\rangle^{\textup{ext}}
\mathcal{A}(B\cup\{z_{a,b}\},\vecd)-
\mathcal{A}(I,\vecd).
\end{align}
If $|I|\ge 2$, then
\begin{equation}\begin{aligned}\label{eq:calculation_2}
\mathcal{A}&(I, \vecd + \eee_1)\\
&=\sum_{\substack{a\in\{-1,\ldots,r-2\}\\b\in\{-1,\ldots,s-2\}}}\sum_{\substack{A \coprod B = \{2,\ldots,l\},2\in B\\ A\not=\varnothing}}\left\langle \tau_0^{(a,b)}\tau_{d_1}^{(a_1,b_1)}\prod_{i \in A}\tau_{d_i}^{(a_i,b_i)}\right\rangle^{\textup{ext}} \mathcal{A}(B\cup\{z_{a,b}\},\vecd).
\end{aligned}\end{equation}
Here $z_{a,b}$ is an additional internal marking with twist
$(r-2-a,s-2-b)$ and no descendent. We extend $\vecd$ in a natural way so that $d_{z_{a,b}} = 0$.
\end{thm}

The next section will be devoted to the proof of this result. Now, we state and prove two key corollaries that establish open mirror symmetry for the Landau-Ginzburg mirror pair.

\begin{cor}\label{thm:mirrorA}
\begin{enumerate}
\item
If $I= \{i\}$ is a singleton, then $\mathcal{A}(I, \vecd) = (-1)^{d_i}$.
\item If $|I| \geq 2$, then
\begin{equation}\label{eq:mirrorEqA}
\mathcal{A}(I,\vecd) =\sum_{\substack{a\in\{-1,\ldots,r-2\}\\b\in\{-1,\ldots,s-2\}}}\sum_{A \coprod B = I\setminus\{1,2\}}\left\langle \tau_0^{(a,b)}\prod_{i \in A\cup\{1,2\}}\tau_{d_i}^{(a_i,b_i)}\right\rangle^{\textup{ext}}
\mathcal{A}\big(B\cup\{z_{a,b}\},\vecd\,\big).
\end{equation}
\item Suppose $|I|\ge 2$.
Let $d(I,\vecd)$ be as defined in Definition \ref{def:d(J)}. Then
\[
\mathcal{A}(I,\vecd)=
\begin{cases}
0 & d(I,\vecd)<0\\
\left\langle \tau_{d(I,\vecd)}^{(r-r(I)-2,s-s(I)-2)}\prod_{i\in I}\tau_{d_i}^{(a_i,b_i)}
\right\rangle^{\textup{ext}}&d(I,\vecd)\ge 0.
\end{cases}
\]
\end{enumerate}
\end{cor}

\begin{proof}
For (1), note that
if $I$ is a singleton $\{i\}$, then we obtain
\[
\mathcal{A}(I,\vecd)=(-1)^{d_i}\mathcal{A}(I,\mathbf{0})=(-1)^{d_i}\langle\tau_0^{(a_1,b_1)}\sigma_1^{a_1}\sigma_2^{b_2}\sigma_{12}\rangle^o=(-1)^{d_i}.
\]
Here the first equality follows from repeated use of
\eqref{eq:calculation_1} as the first summand on the right-hand side of
\eqref{eq:calculation_1} vanishes. Then the second equality comes from
the definition of $\mathcal{A}(I,\mathbf{0})$ and the third
equality is Theorem \ref{thm:simple invariants}.

For (2), we may equate the right-hand sides of \eqref{eq:calculation_1}
and ~\eqref{eq:calculation_2} and solve for $\mathcal{A}(I,\vecd)$ to obtain
\begin{equation}\begin{aligned}\label{eq:calculation_3}
\mathcal{A}(I, \vecd)={} &\sum_{\substack{a\in\{-1,\ldots,r-2\},\\b\in\{-1,\ldots,s-2\}}}\sum_{\substack{A \coprod B = \{2,\ldots,l\}\\ A\neq \emptyset}}\left\langle \tau_0^{(a,b)}\tau_{d_1}^{(a_1,b_1)}\prod_{i \in A}\tau_{d_i}^{(a_i,b_i)}\right\rangle^{\text{ext}}
\mathcal{A}(B\cup\{z_{a,b}\},\vecd)\\
&-\sum_{\substack{a\in\{-1,\ldots,r-2\},\\b\in\{-1,\ldots,s-2\}}}\sum_{\substack{A \coprod B = \{2,\ldots,l\},2\in B \\ A \neq \emptyset}}\left\langle \tau_0^{(a,b)}\tau_{d_1}^{(a_1,b_1)}\prod_{i \in A}\tau_{d_i}^{(a_i,b_i)}\right\rangle^{\text{ext}} \mathcal{A}(B\cup\{z_{a,b}\},\vecd) \\
={} &\sum_{\substack{a\in\{-1,\ldots,r-2\},\\b\in\{-1,\ldots,s-2\}}}\sum_{A \coprod B = \{2,\ldots,l\},2\in A}\left\langle \tau_0^{(a,b)}\tau_{d_1}^{(a_1,b_1)}\prod_{i \in A}\tau_{d_i}^{(a_i,b_i)}\right\rangle^{\text{ext}} \mathcal{A}(B\cup\{z_{a,b}\},\vecd).
\end{aligned}\end{equation}
This gives (2).

Finally, we prove (3). Suppose $I$ is an index set with $|I|\ge 2$
and $1,2\in I$.
Assume by induction the result
for all $I'$ with $2\le |I'|<|I|$.
To prove (3), we aim to apply \eqref{eq:mirrorEqA}, but, to do so effectively, we must establish a relationship between $d(B\cup\{z_{a,b}\},\vecd)$ and $d(I, \vecd)$. Let us consider those sets $B\subseteq I\setminus
\{1,2\}$ for which the contribution in the summation of
\eqref{eq:mirrorEqA} is non-zero,
i.e.,
\begin{equation}
\label{eq:extended nonvanishing}
\left\langle \tau_0^{(a,b)}\prod_{i \in I\setminus B}\tau_{d_i}^{(a_i,b_i)}\right\rangle^{\text{ext}} \not=0.
\end{equation}
By Observation \ref{obs:closed extended}, (1), this requires
\begin{equation}
\label{eq:nonvanishing need}
0=m((I\setminus B)\cup\{(a,b)\},\vecd)+2r+2s=m(I\setminus B,\vecd)+
sa+rb-rs+2r+2s.
\end{equation}
With $B':=B\cup\{z_{a,b}\}$ where $z_{a,b}$ is as in Theorem~\ref{thm:open TRR}, we will now show
that~\eqref{eq:nonvanishing need} implies
\begin{equation}
\label{eq:Bprime I rel}
d(B',\vecd)=d(I,\vecd)-1.
\end{equation}

We note that
\[
m(I\setminus B,\vecd)=m(I,\vecd)-m(B,\vecd)+rs
\]
and
\[
m(B',\vecd)=m(B,\vecd)+r(s-b-2)+s(r-a-2)-rs.
\]
Putting this together, we see that \eqref{eq:nonvanishing need} holds
if and only if
\[
m(B',\vecd)=m(I,\vecd)+rs.
\]
Note further that \eqref{eq:extended nonvanishing} also implies from
\eqref{eq:graded_closed_rank} that
$a+2+\sum_{i\in I\setminus B}a_i \equiv 0 \mod r$, and thus
\begin{equation}\label{r(I) vs r(B')}
r(I)\equiv (a+2)+(r-a-2)+\sum_{i\in I} a_i \equiv
(r-a-2)+\sum_{i\in B} a_i \equiv
r(B')
\mod r,
\end{equation}
giving the equality $r(I)=r(B')$. Similarly, $s(I)=s(B')$.
Thus we conclude that
\[
m(B',\vecd)-(sr(B')+ rs(B')) = m(I,\vecd)-(sr(I)+rs(I))+rs,
\]
showing~\eqref{eq:Bprime I rel}.

Thus, if $B$ is non-empty, so that $|I|>|B'|\ge 2$,
then by the induction hypothesis,~\eqref{eq:Bprime I rel}, and ~\eqref{r(I) vs r(B')},
\[
{\mathcal A}(B',\vecd)=
\left\langle\tau_{d(I,\vecd)-1}^{(r-r(I)-2,s-s(I)-2)}\tau_0^{(r-2-a,s-2-b)}
\prod_{i\in B}\tau_{d_i}^{(a_i,b_i)}\right\rangle^{\text{ext}},
\]
where for convenience we interpret this as $0$ if $d(I,\vecd)-1<0$ and
are using \eqref{eq:Bprime I rel} to write this formula
using $d(I,\vecd)$ instead of $d(B',\vecd)$.

We first consider the case that $d(I,\vecd)<0$, so that $d(B',\vecd)< -1$ if
\eqref{eq:extended nonvanishing} holds.
Then by the induction hypothesis, the only terms in
\eqref{eq:mirrorEqA} which are not obviously zero arise with $B=\varnothing$,
but we claim that in this case, the invariant in
\eqref{eq:extended nonvanishing} vanishes. Indeed, in this case
$m(B',\vecd)=s(r-2-a)+r(s-2-b)=sr(B')+rs(B')$, hence $d(B', \vecd) = -1$, contradicting
 $d(B',\vecd)<-1$.
Thus all terms in \eqref{eq:mirrorEqA} vanish, showing the desired
inductive vanishing of ${\mathcal A}(I,\vecd)$ in this case.

If, on the other hand, $d(I,\vecd)=0$, then again by induction all terms
in the double summation of \eqref{eq:mirrorEqA} vanish except
when $B=\varnothing$. Moreover, a term in the first summation can only be non-zero if
$a=r-r(I)-2, b=s-s(I)-2$
by Observation \ref{obs:closed}.  Thus, after using $\mathcal{A}(\{z_{a,b}\},\vecd)=1$ by
(1), we obtain
\[
\mathcal{A}(I,\vecd)=\left\langle \tau_0^{(r-r(I)-2,s-s(I)-2)}
\prod_{i\in I}\tau_{d_i}^{(a_i,b_i)}\right\rangle^{\text{ext}}
\]
as desired.

Finally, if $d(I,\vecd)>0$,
we write \eqref{eq:mirrorEqA} by splitting the sum according
to whether $B$ is empty or not, use item (1), and the induction hypothesis
to get:
\begin{align*}
{\mathcal A}(I,\vecd)= {} &
\sum_{\substack{a\in\{-1,\ldots,r-2\},\\b\in\{-1,\ldots,s-2\}}}
\left\langle\tau_0^{(a,b)}\prod_{i\in I}\tau_{d_i}^{(a_i,b_i)}
\right\rangle^{\text{ext}}\\
&+
\sum_{\substack{a\in\{-1,\ldots,r-2\},\\b\in\{-1,\ldots,s-2\}}}
\sum_{A \coprod B = I\setminus\{1,2\},B\not=\varnothing}\left\langle \tau_0^{(a,b)}\prod_{i \in A\cup\{1,2\}}\tau_{d_i}^{(a_i,b_i)}\right\rangle^{\text{ext}}
\cdot\\
&\quad\quad\quad\quad\quad\cdot
\left\langle\tau_{d(I,\vecd)-1}^{(r-r(I)-2,s-s(I)-2)}\tau_0^{(r-2-a,s-2-b)}
\prod_{i\in B}\tau_{d_i}^{(a_i,b_i)}\right\rangle^{\text{ext}}.
\end{align*}
By the closed extended topological recursion and Ramond vanishing of
Observation \ref{obs:closed extended}(2),(3), this now agrees with
\[
\sum_{\substack{a\in\{-1,\ldots,r-2\},\\b\in\{-1,\ldots,s-2\}}}
\left\langle\tau_0^{(a,b)}\prod_{i\in I}\tau_{d_i}^{(a_i,b_i)}
\right\rangle^{\text{ext}}
+ \left\langle \tau_{d(I,\vecd)}^{(r-r(I)-2,s-s(I)-2)}\prod_{i\in I}
\tau_{d_i}^{(a_i,b_i)}\right\rangle^{\text{ext}}.
\]
However, a term in the first summation can only be non-zero if
$a=r-r(I)-2, b=s-s(I)-2$ by Observation \ref{obs:closed}. Using this prescribed twist and the definition of $d(I,\vecd)$, one can compute that $m(I \cup \{z\}, \vecd) +2r +2s = -rsd(I, \vecd) < 0$ where $z$ is an internal marking with twist $(a,b)$. Thus, by Observation \ref{obs:closed extended}(1), all the terms in the first summation vanish, giving the result.
\end{proof}

\begin{cor}[Open Mirror Symmetry for 2-dimensional Fermat polynomials]
\label{cor:open MS}
If $\mathbf{s}$ is a canonical symmetric family of multisections bounded by a finite index set $I$, then
\begin{align*}
&\int_{\Xi_{a,b}} e^{W^{\CI,\mathrm{sym}}/\hbar} dx \wedge dy
= {}
\delta_{a,0}\delta_{b,0} + \sum_{d\ge 0} t_{a,b,d}\hbar^{d-1}
\\
{} & +\sum_{l\ge 2}
\sum_{\substack{A\in \mathcal{A}_l\\
d(A)\ge 0}} (-1)^{l}
(-\hbar)^{-d(A)-2}{\delta_{r(A),a}\delta_{s(A),b}\over |\textup{Aut}(A)|}
 \left\langle \tau_{d(A)}^{(r-r(A)-2,s-s(A)-2)}\prod_{(\alpha_i,\beta_i,d_i)
\in A}\tau_{d_i}^{(\alpha_i,\beta_i)}
\right\rangle^{\textup{ext}}
\prod_{(\alpha_i,\beta_i,d_i)\in A} t_{\alpha_j,\beta_j,d_j}.
\end{align*}
 as an equality of Laurent series in the variable $\hbar$ with coefficients in $A_{I, \textrm{sym}}$.

In particular, working modulo $\langle t_{a,b,d} \,|\, 0 \le a \le r-2,
0\le b \le s-2,d>0\rangle$, $dx\wedge dy$ is a primitive form
and the $t_{a,b,0}$ are flat coordinates.
\end{cor}

\begin{proof}
The formula follows from Corollary \ref{cor:symmetric integal}.
Here, for the summation over $l=1$ and $l\ge 2$, we use Corollary
\ref{thm:mirrorA},(1) and (3) respectively
for the value of $\mathcal{A}(A,\vecd)$. Note that for $l=1$,
if $A=\{(a_1,b_1,d_1)\}$, then $r(A)=a_1, s(A)=b_1$, and thus
$d(A)=-d_1-1$. The final statement is then just Corollary \ref{cor:flat coordinates}.
\end{proof}

\subsection{The action of the wall-crossing group on open FJRW invariants}\label{sec:WC}

In this subsection we will prove the following result:

\begin{thm}\label{last containment for chambers}
We have the containment $\ChamberIndices(I)\subseteq \OFJRW(I)$.  That is, every chamber index bounded by $I$ can be realized as the open FJRW invariants for some canonical family of multisections bounded by $I$.
\end{thm}

In particular, we now show that Theorem~\ref{introthm: torsor} will follow from Theorem~\ref{last containment for chambers}.
\begin{cor}\label{chambers are enumerative}
We have 
$$
\OFJRW(I)= \ChamberIndices(I).
$$
Moreover, for a finite set $I\subseteq\Universe$,
$$
\OFJRWSym(I)=\ChamberIndicesSym(I).
$$
In particular, when $I$ is finite, the group $G^{r,s}_{A_I}$ (resp. $G^{r,s}_{A_{I,\mathrm{sym}}}$) acts faithfully and transitively
on $\OFJRW(I)$ (resp. $\OFJRWSym(I)$). 
\end{cor}

\begin{proof} 
Note that the equalities $\OFJRW(I)= \ChamberIndices(I)$ and $\OFJRWSym(I)=\ChamberIndicesSym(I)$ imply the rest of the statement in the corollary. This is because Theorem \ref{thm:G action} shows that $G_{A_I}^{r,s}$ (resp. $G^{r,s}_{A_{I,\mathrm{sym}}}$) acts faithfully and transitively on
$\ChamberIndices(I)$ (resp. $\ChamberIndicesSym(I)$) via the formula
\begin{equation}\label{sec 7 action}
g(W^\nu) = W^{g(\nu)}.
\end{equation}
We remark the formula for the action given here is in line with the relation between open FJRW invariants associated to two canonical families of multisections bounded by $I$ given in  Theorem~\ref{thm: group actions only}. 

 Moreover, note that Corollary~\ref{thm:mirrorA} implies the containments $\OFJRW(I)\subseteq\ChamberIndices(I)$ and
$\OFJRWSym(I)\subseteq \ChamberIndicesSym(I)$. Thus by using Theorem~\ref{last containment for chambers}, we have the equality $\OFJRW(I)= \ChamberIndices(I)$.
Lastly, note that, by Proposition~\ref{symmetrised chamber indices}, the containment $\ChamberIndices(I) \subseteq \OFJRW(I)$ and equality $\ChamberIndicesSym(I)=\ChamberIndices(I)\cap\InvSym(I)$ imply that $\ChamberIndicesSym(I) \subseteq \OFJRWSym(I)$.
\end{proof}

We need a couple of lemmas before we start to prove Theorem~\ref{last containment for chambers}. The first ingredient in the proof is the following lemma:

\begin{lemma}\label{lemma:extension}
Let $\Gamma$ be a smooth connected open graded $W$-spin graph.
Let $\zeta$ be a multisection of $E_{\Gamma}(\vecd)$ defined over $\partial \oPM_{\Gamma}$.
Suppose that, for any $\Lambda\in\partial\Gamma\setminus\partial^+\Gamma$, we can write
\[\zeta|_{\oPM_\Lambda}=F_\Lambda^*\zeta^{\CB\Lambda},\]
where $\zeta^{\CB\Lambda}\in C_m^\infty(\oPM_{\CB\Lambda},E_{\CB\Lambda}(\vecd))$ is a transverse multisection that is strongly positive  with respect to $\CB\Lambda$,  and $\textup{Aut}(\CB\Gamma)$-invariant.
Then one can extend $\zeta$ to a transverse, strongly positive, $\textup{Aut}(\Gamma)$-invariant
multisection of $E_\Gamma(\vecd)$. 
\end{lemma}

\begin{proof}
The lemma is analogous to Lemma 6.5 in \cite{BCT:II}.
The proof is the same.
\end{proof}

\begin{rmk} 
Note that Lemma~\ref{lemma:extension} also holds when $\partial\oPM_{\Gamma}$ is the empty set. In this case, the lemma is restating the existence of a transverse, strongly positive, $\Aut(\Gamma)$-invariant multisection of $E_{\Gamma}(\vecd)$.
\end{rmk}

The second lemma constructs a homotopy between multisections of orbifold bundles that will have a desired number of zeros.

\begin{lemma}\label{lem:changing_winding}
Let $E\to M$ be an oriented orbifold bundle over an oriented orbifold $M$ with corners, such that the generic isotropy of $E \to M$ is trivial and $\rank_\R E=\dim_\R M +1=n$. 
Let $s\in C_m^\infty(M,E)$ be a nowhere vanishing multisection and $c\in\Q$. Then one can construct a multisection $s'$ which agrees with $s$ away from a compact set $K\subset \textup{Interior}(M)$ and a transverse homotopy $H:[0,1]\times M
\to E$, so that
\[H(0,-)=s,H(1,-)=s',~H(t,p)=s(p) \text{ for all }p\in M\setminus K\]
and $\#Z(H)=c$.
\end{lemma}
\begin{proof}
Write $c=a/b$ for some $a\in\Z$ and $b \in \mathbb{N}$.
Choose both a point $x\in M$ with trivial isotropy and an open neighborhood $U$ of
$x$ such that there is a diffeomorphism $\phi:U\rightarrow B$ with $B\subseteq
\R^{n-1}$ being the unit ball. 
We can write the multisection $s|_U=\biguplus_{i=1}^m n_i s_i,$ where $s_i\in C^\infty(U,E)$ and $n_i\in \N$.
Now we can rewrite $s$ in the neighborhood $U$ as
\[s =\biguplus_{i=1}^{b\sum_{i=1}^m n_i} \tilde{s}_j,~~\tilde{s}_j = s_i~\text{ when }b\sum_{q=1}^{i-1}n_q< j\leq b\sum_{q=1}^{i}n_q.\]

We may choose a trivialization $(E|_U\to U)\simeq (\R^n\times U \to U)$ so that $\tilde{s}_1|_U$ is mapped to the constant section $e_1=(1,0,\ldots,0)$. From now on, we use this trivialization to identify the fibers of $E|_U \to U$ with $\R^n$. Denote by $S^{n-1}\subset\R^n$ the unit sphere. Take $g:S^{n-1}\to S^{n-1}$ to be a degree $a(\sum_{i=1}^m n_i)$ map that is chosen so that the lower hemisphere is mapped to $e_1$. Let $\pi:S^{n-1}\to \bar B$ be the projection to the first $n-1$ coordinates and $\pi':B\to S^{n-1}$ be its inverse projection to the upper hemisphere, i.e., $\pi\circ\pi'=id$ and $\text{Image}(\pi')$ has positive last coordinate.
We define $s'$ so that \[s'|_{(M\setminus U)}=s|_{(M\setminus U)},\text{ and }s'|_U = \biguplus_{i=1}^{b\sum_{i=1}^m n_i} \tilde{s}'_j,\]
where
\[\tilde{s}'_j(y) = \begin{cases} g(\pi'(\phi(y))), & \text{ if $j=1$}, \\ \tilde{s}_j, & \text{ otherwise}. \end{cases}\]
The multisection $s'$ glues to give a multisection. Finally, there must exist a transverse homotopy $H$ from $s'$ to $s$ that is constant for all except the first branch, that is constant on the first branch outside of $U$ and satisfies $\#Z(H)=\frac{a\sum_{i=1}^m n_i}{b\sum_{i=1}^m n_i}=c$.\footnote{Here, $\#Z(H)$ is computed using the discussion above Notation A.4 in \cite{BCT:II}.}
\end{proof}

\begin{proof}[Proof of Theorem~\ref{last containment for chambers}]
Fix a chamber index  $\nu := (\CI_{\Gamma,\vecd})\in \ChamberIndices(I)$. We will construct a family of canonical multisections $\mathbf{s}$ of the bundles $E_{\Gamma}(\vecd)$ bounded by $I$ so that
\begin{equation} \label{eqn:chamber realised}
\left\langle \prod_{i\in I(\Gamma)}\tau^{(a_i,b_i)}_{d_i}\sigma_1^{k_1(\Gamma)}\sigma_2^{k_2(\Gamma)}\sigma_{12}^{k_{12}(\Gamma)} \right\rangle^{\mathbf{s},o} = \nu_{\Gamma,\vecd}
\end{equation}
for all graphs $\Gamma$ contained in some $\INT(J,\vecd)$ for some $J \subseteq I$ and $\vecd \in \Z_{\ge 0}^{J}$. We do so via induction on the number $l = |J|$ of internal markings.

First, when $l = 0$, consider any family of canonical multisections $\mathbf{s}$ of $\cW$ bounded by $I$. Such a family exists by Theorem~\ref{prop:int_numbers_exist}. Then by Theorem~\ref{thm:simple invariants}, the condition in ~\eqref{eqn:chamber realised} holds.

Let  $\RS\text{-Graph}(l)$ be the set of all graded open $\RS$-graphs $\Gamma$ where $I(\Gamma) \subset I$ and $|I(\Gamma)| \le l$. For the induction step, assume that we have constructed a family of canonical multisections $(\mathbf{s}^\Gamma)_{\Gamma \in \RS\text{-Graph}(l)}$ so that
\begin{itemize}
\item If $\Gamma\in \INT(I(\Gamma),\vecd)$ for some $\vecd\in \Z_{\ge 0}^{I(\Gamma)}$
 and $|I(\Gamma)| < l$, then ~\eqref{eqn:chamber realised} holds.
\item If $\Gamma \in \RS\text{-Graph}(l)$ then
$$
\mathbf{s}^\Gamma=F_\Gamma^*(\boxplus_{\Xi\in\Conn(\CB\Gamma)}\mathbf{s}^\Xi).
$$
\end{itemize}
Let $J \subseteq I$ be a set of size $l$. We will construct a canonical family of multisections $\mathbf{s}$ bounded by $J$ using induction on a new number
$$
n(\Gamma): = sk_1(\Gamma) + rk_2(\Gamma) + (r+s)k_{12}(\Gamma),
$$
a weighted count of boundary markings on the graph $\Gamma$. Note $n(\Gamma) \ge  0$, so we may vacuously take the base case of this induction to be
$n(\Gamma)=-1$. Inside this induction step, i.e., once we fix the number $n(\Gamma)$, we now perform downward induction on the number $k_{12}(\Gamma)$. Here, we may take the base case of the induction to be $k_{12}(\Gamma) = \lfloor \tfrac{n(\Gamma)}{(r+s)} \rfloor + 1$, which will be vacuous as well.

Next, suppose we have constructed the canonical multisection $\ess$  for any graded graph $\Gamma'$ that satisfies any of the following:
\begin{enumerate}
\item[(i)] $I(\Gamma') \subsetneq I(\Gamma)$,
\item[(ii)] $I(\Gamma')= I(\Gamma)$, and $n(\Gamma')<n(\Gamma)$, or
\item[(iii)] $I(\Gamma')= I(\Gamma)$, $n(\Gamma')=n(\Gamma)$,  and $k_{12}(\Gamma') > k_{12}(\Gamma)$.
\end{enumerate}
We now construct a canonical multisection $\ess^{\Gamma}$ that will be in our canonical family of multisections bounded by $J$. We  do this via a case-by-case analysis in order to respect the canonical property.

If $\Gamma$ is not connected, then each connected component $\Xi$ of $\Gamma$
satisfies either $I(\Xi) \subsetneq I(\Gamma)$ or $n(\Xi) < n(\Gamma)$, hence the  multisection $\mathbf{s}^{\Xi}$ is already defined and we write $\mathbf{s}^\Gamma=\boxplus_{\Xi\in\Conn(\Gamma)}\mathbf{s}^\Xi$.

Analogously, if $\Gamma \neq \CB \Gamma$ and $\Gamma$ does not have an exchangeable vertex, then any connected component $\Xi$ of $\CB\Gamma$ has either the property (i) or (ii) in the inductive hypotheses above. Thus, the multisection $\mathbf{s}^{\Xi}$ is already constructed, so we construct  $\mathbf{s}^\Gamma$ as \begin{equation}\label{eqref:canonical boundary}
\mathbf{s}^\Gamma=F^*_{\Gamma}(\boxplus_{\Xi\in\Conn(\CB\Gamma)}\mathbf{s}^\Xi).
\end{equation} Lastly, if $\Gamma$ has an exchangeable vertex $v$ with incident edge $e$, then $\detach_e \Gamma$ has two connected components, the exchangeable vertex and the rest of the graph $\Gamma'$.
Further $\Gamma'$ has the property (iii) above, so we can construct $\mathbf{s}^{\Gamma}$ as in~\eqref{eqref:canonical boundary} by induction.

Otherwise, we define the multisection $\mathbf{s}^{\Gamma}$ as follows. We start by constructing the multisection at the boundary. First note that, for any $\Lambda\in \partial^0\Gamma \cup \partial^\xch\Gamma$ and $\Gamma = \smooth_E \Lambda$, we have the relation
\begin{equation}\label{n Gamma relation}
n(\Gamma) = \sum_{v \in \mathrm{Vert}(\Lambda)} n(v) - (r+s) |E|.
\end{equation}
Thus, if $\Lambda \in\partial^0\Gamma\cup \partial^\XCH\Gamma$ then $\CB\Lambda$ consists of connected components $\Xi_1, \ldots, \Xi_f$ such that, for each $i$, one of the following holds:
\begin{enumerate}
\item[(i)] $I(\Xi_i)\subsetneq J$,
\item[(ii)] $I(\Xi_i)=J$ and $n(\Xi_i)<n(\Gamma)$, or
\item[(iii)] $I(\Xi_i)=J$, $n(\Xi_i)=n(\Gamma)$ and $k_{12}(\Xi_i)> k_{12}(\Gamma)$.
\end{enumerate}
Thus we can inductively define the transverse multisections $\mathbf{s}^{\Lambda}$ for all $\Lambda  \in \partial^0\Gamma \cup \partial^\XCH\Gamma$ via the formula~\eqref{eqref:canonical boundary}. We then can use Lemma~\ref{lemma:extension} to extend the multisection over the entire moduli space.  By Proposition~\ref{invariants are independent of extension of boundary}, we have that the intersection number
\[\left\langle\prod_{i \in I(\Gamma)}\tau^{(a_i,b_i)}_{d_i}\sigma_1^{k_1(\Gamma)}\sigma_2^{k_2(\Gamma)}\sigma_{12}\right\rangle^{\ess,o}\]
is independent of the choice of the extension $\mathbf{s}^{\Gamma}$ from the boundary.

Consider the corresponding  chamber index $\nu = (\nu_{\Gamma, \vecd})$ given above and the chamber index $\nu^\ess = (\nu^{\ess}_{\Gamma, \vecd})$ corresponding to the family of multisections $\mathbf{s}$ (bounded by $J$) as defined in ~\eqref{eq:CI s def}. For any $\vecd \in \NN^{J}$ and $\Gamma \in \INT(J, \vecd)$,
we now write $$\Delta_{\Gamma,\vecd}(\ess) := \nu^{\ess}_{\Gamma, \vecd}-\nu_{\Gamma,\vecd}.$$  For any $\Gamma\in \INT(\subseteq J,\vecd)$ where $|I(\Gamma) | < l$, we have by induction that $\Delta_{\Gamma,\vecd}(\ess) = 0$. The goal now is to construct a new family of multisections $\mathbf{s}'$ bounded by $J$ so that $\Delta_{\Gamma,\vecd}(\ess')$ vanishes for all $\vecd \in \NN^J$ and $\Gamma \in \INT(J,\vecd)$.

Consider the quantities $\mathcal{A}(J, \vecd, \nu^\ess)$ and $\mathcal{A}(J, \vecd, \nu)$ defined in Notation~\ref{nn:A(J)}. 
 Using that $\nu$ is a chamber index by assumption
and that $\nu^\ess$ is a chamber index by Corollary \ref{thm:mirrorA}(3),
we have  from Corollary~\ref{cor:CIs same result} that
\begin{equation}\label{topContributionInduction}
0 = \mathcal{A}(J, \vecd,\nu^\ess) - \mathcal{A}(J, \vecd,\nu) =
\sum_{\Gamma\in \INT(J,\vecd)} \frac{\Gamma(\tfrac{1+k_1(\Gamma)}{r})\Gamma(\tfrac{1+k_2(\Gamma)}{s})}{\Gamma(\tfrac{1+r(J)}{r})\Gamma(\tfrac{1+s(J)}{s})}\Delta_{\Gamma,\vecd}(\ess).
\end{equation}

Recall the notation for all balanced and critical graphs given in Notation~\ref{Concrete bal and crit graphs with prescribed internals} of balanced graphs and critical graphs with internal markings $J$.   If $N=0$, then $\INT(J,\vecd)=\{\Gamma_0\}$ and the equality $\Delta_{\Gamma_{J,0}}=0$ follows directly from~\eqref{topContributionInduction}. If $N>0$ then we can classify the type of top critical boundary graphs corresponding to each $\Gamma_{J,i}$. Recall from Definition~\ref{top boundary, subsets for Bal and Crit} that a top boundary graph will be a critical boundary graph with two vertices $v_0$ and $v_1$ with a boundary edge $e$ so that
\begin{itemize}
\item the vertex $v_1$ is rooted, balanced and has no internal markings; and
\item the vertex $v_0$ is of the form $\Lambda_{J,p} \in \CRIT(J,\vecd)$ from Notation~\ref{Concrete bal and crit graphs with prescribed internals}.
\end{itemize}
Given a balanced graph $\Gamma_{J,i}$, there are only two possible types of top critical boundaries: (i) when $v_1$ is $\Gamma^W_{0,r,0,1,\emptyset}$  and $v_0$ is $\Lambda_{J,i}$, and (ii) when $v_1$ is $\Gamma_{0,0,s,1,\emptyset}^W$ and $v_0$ is $\Lambda_{J, i+1}$. For $\Gamma_{J,0}$, there are not enough $r$-points to have type (i) and, for $\Gamma_{J,N}$, there are not enough $s$-points to have type (ii). Otherwise, for $1\leq i\leq N-1$, the graph $\Gamma_{J,i+1}$ has precisely two types of top critical boundaries,
such that both have two vertices and one boundary edge and one of its half-edges is a $r$-point and one is a $s$-point. As a quick remark, we note that type (i) (respectively type (ii)) corresponds to critical boundary graphs with $\epsilon =1$ (respectively $\epsilon =2$).

Suppose we change the multisection $\ess^{\Lambda_{J,i}}$ to another multisection $(\ess')^{\Lambda_{J,i}}$, which agrees with $\ess^{\Lambda_{J,i}}$ near $\partial\oCM_{\Lambda_{J,i}}$.  
Given a transverse homotopy $H\in C_m^\infty\big([0,1]\times\oPM_{\Lambda_{J,i}}, \pi^*E_{\Lambda_{J,i}}(\vecd)\big)$ which is constant in time near  $[0,1]\times\partial\oCM_{\Lambda_{J,i}}$ 
we define
\[h_{i,\vecd}: =h_{i,\vecd}(\mathbf{s}^{\Lambda_{J,i}},(\mathbf{s}')^{\Lambda_{J,i}}):=  \int_{[0,1]\times\oPM_{\Lambda_{J,i}}}e(\pi^*E_{\Lambda_{J,i}}(\vecd);H|_V),\]
where $V$ can be taken to be the union of $\{0,1\}\times\oPM_{\Lambda_{J,i}}$ with the locus where $H$ is constant in time.

We claim that the quantity $h_{i,\vecd}$ depends only on the multisections $\mathbf{s}^{\Lambda_{J,i}}$ and $(\mathbf{s}')^{\Lambda_{J,i}}$, but not on the homotopy chosen between them. Choose another homotopy $H'\in C_m^\infty([0,1]\times
\oPM_{\Lambda_{J,i}},\pi^*E_{\Lambda_{J,i}}(\vecd))$
which is the constant in time near $[0,1]\times
\partial\oCM_{\Lambda_{J,i}}$. While $H$ and $H'$ will have different neighborhoods $V$ and $V'$ of $[0,1]\times\partial\oCM_{\Lambda_{J,i}}$ where they are constant in time, we will have that $H|_{V\cap V'}=H'|_{V\cap V'}$. Thus, by an argument analogous to that in the proof of Proposition~\ref{invariants are independent of extension of boundary} we have
\begin{align*}
\int_{[0,1]\times\oPM_{\Lambda_{J,i}}}e(\pi^*E_{\Lambda_{J,i}}(\vecd);H|_V)=
{} & \int_{[0,1]\times\oPM_{\Lambda_{J,i}}}e(\pi^*E_{\Lambda_{J,i}}(\vecd);H|_{V\cap V'})\\
={} &\int_{[0,1]\times\oPM_{\Lambda_{J,i}}}e(\pi^*E_{\Lambda_{J,i}}(\vecd);H'|_{V'}).
\end{align*}

We now construct this new family of canonical multisections $\mathbf{s}'$ bounded by $J$ by taking $\ess$, but then replacing the multisection $\mathbf{s}^{\Lambda_{J,i}}$ for the graphs $\Lambda_{J,i}$ with $\mathbf{s'}^{\Lambda_{J,i}}$. This will require us to change the canonical multisections $\ess^\Gamma$ for all $\Gamma$ such that there exists a boundary strata $\Lambda \in \partial^0\Gamma \cup \partial^\xch \Gamma$ for which $\Lambda_{J,i}$ is a connected component of $\CB \Lambda$. We now see which balanced graphs have this property up to the induction step where we construct the multisection for the graph $\Gamma_{J,i}$.

We can use ~\eqref{n Gamma relation} to see that for 
$$
n(\Lambda_{J,i}) = n(\Gamma_{J,i}) - n(\Gamma^W_{0,r,0,1,\emptyset}) + (r+s) = n(\Gamma_{J,i}) -rs.
$$
Since $n(\Lambda_{J,i})<n(\Gamma_{J,i})$, we are modifying the
sections $\ess^{\Lambda_{J,i}}$ which was already determined at a previous
step of the induction. Take a graded graph $\Gamma$ whose multisection has already been constructed and so that there exists $\Lambda \in \partial^0\Gamma\cup \partial^{\xch}\Gamma$ with $\Lambda_{J,i}$ a connected component of $\CB\Lambda$. Thus $I(\Gamma) \subseteq J$ but if $\Lambda_{J,i}$ is a vertex in the graph $\Gamma$ then $I(\Gamma) = J$. 

Since $I(\Gamma) = J$, all other vertices of $\Lambda$ aside from $\Lambda_{J,i}$ must have no internal markings. Moreover, since $\Gamma$ is graded, all other vertices must be stable and have no untwisted boundary tails. For any vertex $v$ in $\Lambda$ with no internal markings, we must have by combining  \eqref{eq:open_rank1} and \eqref{eq:open_rank2} as in the proof of Proposition~\ref{prop:balanced}(a) that 
\begin{equation}\label{integrality in induction}
r \ | \ k_1(v) + k_{12}(v)-1 \text{ and } s \  |  \ k_2(v) + k_{12}(v) - 1.
\end{equation}
Since we are at the induction step for the graph  $\Gamma_{J,i}$, we know that $n(\Gamma) \le n(\Gamma_{J, i})$. We now can start to order all possible vertices with no internal markings in increasing order with respect to the weighted number $n(v)$. The first two such vertices $v$ that satisfy ~\eqref{integrality in induction} have $n(v) = r+s$ and have the following number of boundary half-edges 
\begin{itemize}
\item  $k_1(v) = k_2(v) = 0$ and $k_{12}(v) = 1$;
\item  $k_1(v) = k_2(v) = 1$ and $k_{12}(v) = 0$.
\end{itemize}
To guarantee stability, any vertex $v$ with this number of boundary half-edges with non-trivial twists will require at least one additional boundary half-edge adjacent to $v$ that is untwisted. Suppose $\Lambda$ has one of these vertices.
Then $\Lambda$ must have additional vertices with no internal markings and a fully twisted half-edge on which we glue the untwisted boundary half-edge. If such an additional vertex $v'$ has $n(v')=r+s$, then we can repeat this argument. Thus, eventually, we must find a vertex $v$ with no internal markings and $n(v)>r+s$. After these two cases, we have the graphs $\Gamma^W_{0,r,0,1,\emptyset}$ and $\Gamma^W_{0,0,s,1,\emptyset}$ in the ordering with respect to $n(v)$, where $n(\Gamma^W_{0,r,0,1,\emptyset}) = n(\Gamma^W_{0,0,s,1,\emptyset}) = rs + (r+s)$. Since $n(\Gamma) \le n(\Gamma_{J,i})$ this means that $\CB\Lambda$ must consist of $\Lambda_{J,i}$, one of the graphs $\Gamma^W_{0,r,0,1,\emptyset}$ and $\Gamma^W_{0,0,s,1,\emptyset}$, and possibly some vertices $v$ with $n(v) =r+s$ by \eqref{n Gamma relation}. However, if we have any vertices $v$ with $n(v) = r+s$, we will then have $n(\Gamma) =n(\Gamma_{J,i})$ with $k_{12}(\Gamma) < k_{12}(\Gamma_{J,i})$. Otherwise, $\Gamma$ coincides with $\Gamma_{J,i-1}$ or $\Gamma_{J,i}$.

Thus we have proved the following: if $\Gamma$ is any graph for which
there exists $\Lambda \in \partial^0\Gamma\cup \partial^{\xch}\Gamma$
with $\Lambda_{J,i}$ a connected component of $\CB\Lambda$, then
either (i) $J\subsetneq I(\Gamma)$;  (ii) $n(\Gamma)>n(\Gamma_{J,i})$;
(iii) $n(\Gamma)=n(\Gamma_{J,i})$ but $k_{12}(\Gamma)<k_{12}(\Gamma_{J,i})$;
or (iv)
$\Gamma=\Gamma_{J,i-1}$ or $\Gamma = \Gamma_{J,i}$. Therefore, changing $\ess^{\Lambda_{J,i}}$
will have no further effect on the inductive construction.
 
 Replace all $\mathbf{s}^{\Lambda_{J,i}}$ with $(\mathbf{s'})^{\Lambda_{J,i}}$, and then apply Lemma~\ref{lemma:extension} to define $(\ess')^{\Gamma}$ at this induction step for those $\Gamma$ with boundary graphs $\Lambda \in \partial^0 \Gamma \cup \partial^\xch\Gamma$ so that $\Lambda_{J,i}$ is a connected component of $\CB\Lambda$. Here, the intersection numbers corresponding to the balanced graphs $\Gamma_{J,i}$ will change. Define
$$
o'_{J,p,\vecd} := \left \langle\prod_{i \in J} \tau^{(a_i,b_i)}_{d_i}\sigma_1^{k_1(\Gamma_{J,p})}\sigma_2^{k_2(\Gamma_{J,p})}\sigma_{12}\right\rangle^{(\mathbf{s}')^{\Gamma_{J,p}},o}.
$$
Using Lemma \ref{lem:zero diff as homotopy} along with ~\eqref{eqn for zeros in the homotopy}, ~\eqref{eq:wc3}, and~\eqref{eq:wc4}, we compute the new invariants $o'_{J,p,\vecd}$ to be as follows:
\begin{itemize}
\item For $p=0$, it will change to
\begin{equation}\label{New f gamma 0}
o'_{J,0,\vecd} = \nu_{\Gamma_{J,0}, \vecd} + \Delta_{\Gamma_{J,0}, \vecd} (\ess)+ (-1)^{k_2(\Lambda_{J,1})}sk_1(\Lambda_{J,1}) h_{1,\vecd};
\end{equation}
\item For  $0<p < N$, it will change to
\begin{equation}\label{New f gamma i}
o'_{J,p,\vecd} =  \nu_{\Gamma_{J,p},\vecd} + \Delta_{\Gamma_{J,p},\vecd}(\ess) + (-1)^{k_2(\Lambda_{J,p})+1} rk_2(\Lambda_{J,p}) h_{p,\vecd}+ (-1)^{k_2(\Lambda_{J,p+1})}sk_1(\Lambda_{J,p+1}) h_{p+1,\vecd}; \text{ and }
\end{equation}
\item For $p=N$, it will change to
\begin{equation}
o'_{J,N,\vecd} = \nu_{\Gamma_{J,N},\vecd} + \Delta_{\Gamma_{J,N},\vecd}(\ess) + (-1)^{k_2(\Lambda_{J,N})+1}rk_2(\Lambda_{J,N}) h_{N,\vecd}.
\end{equation}
\end{itemize}

By Lemma~\ref{lem:changing_winding}, there exists a multisection $(\mathbf{s}')^{\Lambda_{J,1}}$ where $c$ in the lemma is taken to be $h_{1,\vecd}=(-1)^{k_2(\Lambda_{J,1})+1}\frac{\Delta_{\Gamma_{J,0}, \vecd}(\ess)}{sk_1(\Lambda_{J,1})}$. Then we will have that $o'_{J,0,\vecd} = \nu_{\Gamma_{J,0},\vecd}$.

We inductively choose $h_{p,\vecd}$ so that we obtain $o_{J,p,\vecd}' = \nu_{\Gamma_{J,p},\vecd}$ for all $0 \le p < N$. This is done by assuming
given $h_{1,\vecd}, \ldots, h_{p,\vecd}$ so that $o_{J,p-1,\vecd}' =\nu_{\Gamma_{J,p-1},\vecd}$. Then take
$$
h_{p+1, \vecd} = \frac{ - \Delta_{\Gamma_{J,p},\vecd}(\ess) + (-1)^{k_2(\Lambda_{J,p})}rk_2(\Lambda_{J,p})h_{p,\vecd}}{(-1)^{k_2(\Lambda_{J,p+1})} sk_1(\Lambda_{J,p+1})}.
$$
Then the expression for $o_{J,p,\vecd}'$ in Equation~\eqref{New f gamma i} will simplify to $o_{J,p,\vecd}' = \nu_{\Gamma_{J,p},\vecd}$. Again, the multisection $(\mathbf{s}')^{\Lambda_{J,p}}$ is guaranteed to exist by Lemma~\ref{lem:changing_winding}.
Note that this does not change the boundary behavior for these multisections and that these critical boundaries all have internal markings $J$, so they do not affect any previously specified open intersection numbers in the inductive step.

Now we have that $\Delta_{\Gamma_{J,p},\vecd}(\ess)=0$ for $p=0,\ldots, N-1$. Combining this with Equation~\eqref{topContributionInduction}, we obtain that $\Delta_{\Gamma_{J,N},\vecd}(\ess)=0$. We perform the above modifications
in this way for all graphs $\Gamma$ with internal markings $J$ and descendent vectors $\vecd \in \NN^J$ such that $\Gamma$ is balanced with respect
to $\vecd$. It is reassuring to note that the top boundary critical graphs with different descendent vectors will all have independent inductive structures at top boundary strata, so the steps above will not rely on each other. Continue to do the above for all $J \subseteq I$ with $|J| = l$ to complete the induction step. Thus we have found that $\nu^\ess = \nu$, hence $\nu \in \OFJRW(I)$. 
\end{proof}

\section{Proof of the open Topological Recursion Relations}
\label{subsec:proof of open TRR}

We now outline the proof of Theorem \ref{thm:open TRR}.
By Theorem \ref{thm:A_mod_invs}, $\mathcal{A}(I,\vecd)$ is independent of the canonical family used to define it. We prove Theorem~\ref{thm:open TRR} by calculating $\mathcal{A}(I,\vecd+{\bf e}_1)$ in two ways, with respect to a family
of transverse special canonical multisections.

For convenience in this subsection, we assume that $\Gamma$ is a rooted smooth graded $W$-spin graph $\Gamma$ with $I(\Gamma)=[l]$, with $l >0$.  Let
\begin{equation}\label{E1 definition}
E = E_\Gamma(\vecd+{\bf e}_1) \text{ and }E_1 = E_{\Gamma}(\vecd)
\end{equation}
 be the descendent Witten bundles with respect to descendent vectors $\vecd+
{\bf e}_1$ and $\vecd$ respectively. Suppose that $\Gamma$ is balanced with respect to the descendent vector $\vecd+{\bf e}_1$.
Consider the following two  transverse special canonical multisections for $E$ and $E_1$ respectively:
\begin{equation}
\label{eq:ss}
\mathbf{s} = \mathbf{s}^\Gamma = s^\Gamma\oplus\bigoplus_{i\in\left[l\right],~j\in \left[d_i+\delta_{i=1}\right]}s_{ij}^\Gamma
\end{equation}
and
\begin{equation}\label{eq:ss1}
\srest=\srest^\Gamma=s^\Gamma\oplus\bigoplus_{\substack{i\in[l],~j\in[d_i]}}s_{ij}^\Gamma,
\end{equation}
where $s^\Gamma$ is the $\cW$-component of $\mathbf{s}^\Gamma$ and $s_{ij}^\Gamma$ is the projection to the $j^{th}$ $\CL_i$ summand.

Since $\mathbf{s}$ and $\srest$ are canonical, they are strongly positive, so there exists a set $U_+=U_+^\Gamma$ as in \eqref{eq:U_+} on which $s^\Gamma$ does not vanish.
Our goal is to compute $\int_{\oPM_\Gamma} e(E;\mathbf{s}^\Gamma|_{U_+\cap\partial^0\oPM_\Gamma})$ in terms of $\srest$; however, doing this directly is difficult. Instead, given a marked point $w$ other than the first internal marked
point $z_1$, we will look at a specific
transverse multisection, $\srest \oplus t_w$, with a distinguished multisection $t_w\in C^\infty(\oPM_\Gamma,\CL_1)$ that we will define below. This multisection, however, will not be canonical so we must take a homotopy to a canonical multisection. 

A priori, we would like to find the invariant
$$\int_{\oPM_\Gamma}e(E;\mathbf{s}^\Gamma|_{U_+\cap\partial^0\oPM_\Gamma})$$
 by computing $\int_{\oPM_\Gamma}e(E;\mathbf{s}_1^\Gamma\oplus t_w|_{U_+\cap\partial^0\oPM_\Gamma})$, and then finding the difference of the two integrals as  the number of zeros introduced by the homotopy between the two sections.
This naive approach fails, as the integrals become unwieldy.
However, it turns out that we can control the contributions of certain linear combinations of these integrals. A key technique introduced here is that we can use two cleverly chosen special linear combinations of multisections that are multisums corresponding to distinguished multisections of $E$ to prove the recurrence relations of Theorem~\ref{thm:open TRR}.

\medskip

\subsection{The distinguished  multisections $t_w$ of $\CL_1$}

\begin{nn}
\label{nn:pointing at}
Consider a graded graph $\Gamma$ with at least two tails, one of which is internal. We can define the distinguished multisection $ \ttt=\ttt_\XXX \in C^\infty(\oPM_\Gamma,\CL_1)$ mentioned above as follows.  Consider a smooth graded $W$-spin disk $\Sigma$ with an internal marking $z_1$ and additional marking $w$. After identifying $\Sigma$ with the upper half-plane, we set
\begin{equation}\label{eq:t}
\tilde \ttt\left(\Sigma\right) = dz\left.\left(\frac{1}{z - \XXX}-\frac{1}{z-\bar{z}_1}\right)\right|_{z = z_1} \in T_{z_1}^*\Sigma.
\end{equation}
This provides a section of $\CL_1$ over $\CM_{\Gamma}$. Furthermore, this
extends to a section of $\CL_1$ over $\oPM_{\Gamma}$ via exactly
the same argument as given in \cite{PST14}, \S4.3, which we briefly
review. If $\Sigma$ is singular, let $C$ be the underlying complex
curve. Then there is a unique meromorphic differential $\omega$
on the normalization
of $C$ with the following properties. 

First, it has simple poles with residue
$+1$ and $-1$ respectively at $w$ and $\bar z_1$. Second, consider a path in $C$ from $w$ to $\bar z_1$ that passes through the minimal number of irreducible components and nodes. Since $C$ has genus 0, there is a unique choice of nodes and irreducible components for the path to traverse.  For any node that the path traverses through,
the two preimages have at most simple poles, and the residues at these poles
sum to zero. If we normalise at such a node, then one node will be in the connected component associated to $w$ and the other in the connected component containing $\bar z_1$. The residue at the half-node sharing the connected component with $w$ is $-1$ and the residue for the half-node sharing the connected component with $\bar z_1$ is $+1$. For all other irreducible components where the path does not pass through, $\omega$ is holomorphic and hence vanishes. One then defines
$\tilde t(\Sigma)$ to be the value of $\omega$ at $z_1$.

We call $\ttt_\XXX$ the section (of $\CL_1$) \emph{pointing at $\XXX$}.
\end{nn}

\begin{rmk}
\label{rem:degenerate pointing}
One special case of a section $t_w$ occurs when the graded $W$-spin disk
 $\Sigma$ contains a boundary node $n$ and the normalization
of $\Sigma$ at this node $n$ is $\Sigma_1 \sqcup \Sigma_2$ with
$z_1\in\Sigma_1$ and $w\in\Sigma_2$. Write $n_{i}\in\Sigma_i$ for the half-nodes
corresponding to $n$. Then it is easy to see from
the description of Notation \ref{nn:pointing at} in the singular case that
$t_w$ coincides with $t_{n_1}$, the section of $\CL_1$ pointing at
$n_1$.
\end{rmk}

\begin{nn}
\label{not:Gamma ab}
Let $\Gamma$ be a smooth graded $W$-spin graph with at least two tails, one of which is internal. Let $I_1$ be the set of internal tails in the connected component containing the internal tail $1$.

Given the data of a subset $A\subseteq I_1 \setminus\{1\}$ and two numbers $a\in\{-1,\ldots,r-2\}$ and $b\in\{-1,\ldots,s-2\}$, we denote by $\Gamma^{A,a,b}$ the graded graph obtained from $\Gamma$ by replacing the vertex containing internal tail $1$ with another graph with one open vertex $v_o$ and one closed vertex $v_c$ so that:
\begin{itemize}
\item the vertices $v_o$ and $v_c$ are attached with one edge. Here, the half-edge belonging to $v_c$ has twist $(a,b)$ and the half-edge belonging to $v_o$ has twist $(r-2-a,s-2-b)$;
\item the closed vertex $v_c$ has internal tails labeled by $\{1\}\cup A$ with the same twists as the original tails on $\Gamma$;
\item the open vertex $v_o$ has internal tails labeled by $I_1\setminus (A\cup\{1\})$ with the same twists as the original tails on $\Gamma$ and the boundary tails with the same twists as in $\Gamma$.
\end{itemize}
\end{nn}
\noindent Such graphs $\Gamma^{A,a,b}$ correspond to the stable disks seen in Figure~\ref{fig:types of recursion}. We note that $a$ is allowed to take
the value $-1$ instead of $r-1$ as the corresponding half-node will be an
anchor (and similarly for $b$).

Let $\Gamma$ be a rooted smooth graded $W$-spin graph $\Gamma$ with internal markings $I$ with $1 \in I$. As in \eqref{E1 definition}, let $E = E_\Gamma(\vecd+{\bf e}_1)$  and $E_1 = E_{\Gamma}(\vecd)$ be the two descendent Witten bundles with respect to $\vecd + \mathbf{e}_1$ and $\vecd$, where $\vecd \in \Z^I_{\ge 0}$ and ${\bf e}_1 = (\delta_{i1})_{i\in I} \in \Z_{\ge0}^I$. Consider the two special canonical multisections $\ess$ and $\ess_1$ for $E$ and $E_1$ given in~\eqref{eq:ss} and \eqref{eq:ss1}
The first observation is the following, which will be the only place where we use the fact that $\ess_1$ is special canonical rather than canonical.

\begin{lem}\label{lem:closed_contribution}
Suppose $\Gamma = \Gamma_{0,k_1, k_2, k_{12},\{(a_i, b_i)\}_{i\in I}}$ is a smooth connected graded $W$-spin graph with at
least two tails, one of which is internal. Write $\ttts = \srest\oplus\ttt_\XXX$ and $k_\bullet=k_\bullet(\Gamma)$. If $A \subseteq I$, write $z_{A,a,b}$ for an internal marking labelled by the union of labels of the marked points in $A$ with twist $(r-2-a,s-2-b)$. Take $v_{A,a,b}$ to be the graph $ \Gamma_{0,k_1,k_2,k_{12},\{(a_i, b_i)\}_{i \in B\cup\{z_{A,a,b}\}}}$ where $A \sqcup B = I$.
\begin{enumerate}
\item If $\XXX$ is a boundary marking, then the integral $\int_{\oPM_\Gamma} e(E;\ttts|_{U_+\cap\partial^0\oPM_\Gamma})$ equals 
\begin{equation}\label{eq:closed_cont_trr_x1}
\sum_{\substack{A \sqcup B = I \\ 1 \in A \\ a\in\{-1,\ldots,r-2\},\\b\in\{-1,\ldots,s-2\}}}\left\langle \tau_0^{(a,b)}\prod_{i \in A}\tau_{d_i}^{(a_i,b_i)}\right\rangle^{\textup{ext}}
\left\langle \tau_0^{(r-2-a,s-2-b)}\prod_{i\in B}\tau^{(a_i,b_i)}_{d_i}\sigma_1^{k_1}\sigma_2^{k_2}\sigma_{12}^{k_{12}}\right\rangle^{\ess^{v_{A,a,b}},o}.
\end{equation}
\item  If $\XXX$ is an internal marking, say $z_2,$ then the integral $\int_{\oPM_\Gamma} e(E;\ttts)$ equals
\begin{equation}\label{eq:closed_cont_trr_z2}
\sum_{\substack{ A \sqcup B = I \\ 1 \in A, \ 2 \in B \\a\in\{-1,\ldots,r-2\},\\b\in\{-1,\ldots,s-2\}}}\left\langle \tau_0^{(a,b)}\prod_{i \in A}\tau_{d_i}^{(a_i,b_i)}\right\rangle^{\textup{ext}}
\left\langle \tau_0^{(r-2-a,s-2-b)}\prod_{i\in B}\tau^{(a_i,b_i)}_{d_i}\sigma_1^{k_1}\sigma_2^{k_2}\sigma_{12}^{k_{12}}\right\rangle^{\ess^{v_{A,a,b}},o}.
\end{equation}
\end{enumerate}
\end{lem}

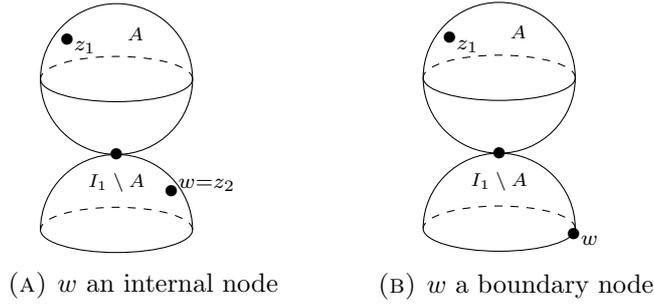
\begin{figure}

\begin{subfigure}{.3\textwidth}
  \centering

\begin{tikzpicture}[scale=0.5]
  \draw (0,0) circle (2cm);
  \draw (-2,0) arc (180:360:2 and 0.6);
  \draw[dashed] (2,0) arc (0:180:2 and 0.6);

  \draw (2,-4) arc (0:180:2);
  \draw (-2,-4) arc (180:360:2 and 0.6);
  \draw[dashed] (2,-4) arc (0:180:2 and 0.6);
\node (c) at (-1,.95) {$\bullet_{z_1}$};
\node (a) at (.5, 1.2) {\tiny $A$};
\node (b) at (0, -2.75) {\tiny $I_1\setminus A$};
\node (c) at (0, -2) {$\bullet$};
\node (d) at (2.2, -2.9) {$\bullet^{w = z_2}$};
\end{tikzpicture}

  \caption{$w$ an internal node}
\end{subfigure}
\begin{subfigure}{.3\textwidth}
  \centering

\begin{tikzpicture}[scale=0.5]
  \draw (0,0) circle (2cm);
  \draw (-2,0) arc (180:360:2 and 0.6);
  \draw[dashed] (2,0) arc (0:180:2 and 0.6);

  \draw (2,-4) arc (0:180:2);
  \draw (-2,-4) arc (180:360:2 and 0.6);
  \draw[dashed] (2,-4) arc (0:180:2 and 0.6);
\node (c) at (-1,.95) {$\bullet_{z_1}$};
\node (a) at (.5, 1.2) {\tiny $A$};
\node (b) at (0, -2.75) {\tiny $I_1\setminus A$};
\node (c) at (0, -2) {$\bullet$};
\node (d) at (2.2,-4.2) {$\bullet_w$};

\end{tikzpicture}

  \caption{$w$ a boundary node}
\end{subfigure}

\caption{Stable disks associated to $\Gamma^{A,a,b}$ for the two choices of $w$ and the boundaries contributing in Lemma~\ref{lem:closed_contribution}}
\label{fig:types of recursion}
\end{figure}

\begin{proof}
The proof of this lemma is completely analogous to the proof of Lemma 4.14 in \cite{BCT:II} (which in turn is inspired by the proof of Lemma 4.7 in \cite{PST14}) and will be only be sketched.

Let ${\CG_{oc}} \subset \partial \Gamma$ be the collection of graphs $\Lambda$ with the following properties:
\begin{itemize}
\item $\Lambda$ has exactly one edge $e$ which connects an open vertex $v^o_\Lambda$ and a closed vertex $v^c_\Lambda$,
\item the internal tail $z_1$ is adjacent to $v^c_{\Lambda}$, i.e.,  $1 \in I(v^c_\Lambda)$, and
\item  if $\XXX$ is an internal marking, say $\XXX=z_2$, then $~2\in I(v^o_\Lambda)$.
\end{itemize}
We refer the reader to the corresponding stable disks in Figure~\ref{fig:types of recursion}.
 By Lemma 4.13 of \cite{BCT:II}, the zero locus of $\ttt_w$ is $\bigcup_{\Lambda \in {\CG_{oc}}} \oPM_{\Lambda}$. By adapting the proof with the number of automorphisms as computed in Observation~\ref{obs:automorphism} in mind, we can see that the zero locus has multiplicity $d$.

Moreover, by Corollary~\ref{thm:or_and_induced internal}, the induced relative orientation on $\cW$ restricted to the zero locus is the one induced from the canonical complex orientation for the closed part, and the canonical relative orientation for the open part. See Lemma 4.13 of \cite{BCT:II} for more details. Note that, by Observation~\ref{obs:automorphism} and
Proposition~\ref{prop:existence_stable}, the map
\[\Detach_e:\oPM_\Lambda \to \oPM_{v^o_\Lambda} \times \oCM_{v^c_\Lambda}\]
has degree $rs/d$.

Following the argument of Lemma 4.14 in \cite{BCT:II}, we can restrict to the zero locus of $t_w$ with the correct multiplicity, and in turn, use the detaching map with the decomposition properties of the special canonical multisection
$\mathbf{s}_1$ (see Definition \ref{def:special canonical Witten descendent}) to see that
\begin{equation}\begin{aligned}
\int_{\oPM_\Gamma} e(E;\ttts|_{U_+\cap\partial^0\oPM_\Gamma}) &= \sum_{\Lambda \in \mathcal{G}_{oc}} \int_{\oPM_\Lambda} d \cdot e(E_1;\srest|_{\oPM_\Lambda}) \\
	&= \sum_{\Lambda \in \mathcal{G}_{oc}} \left(rs\int_{\PM_{v_\Lambda^c}}e(E_1; \srest|_{\oPM_{v^c_{\Lambda}}}) \right)\cdot \int_{\PM_{v_\Lambda^o}} e(E_1; \srest|_{\oPM_{v^o_{\Lambda}}}).
\end{aligned}\end{equation}
Lastly, using the definition of closed extended invariants in \eqref{defn: closed ext W spin invariant}, we obtain \eqref{eq:closed_cont_trr_x1} and \eqref{eq:closed_cont_trr_z2}.
\end{proof}

Note that in Lemma \ref{lem:closed_contribution}(1), when $\XXX$ is a boundary point, the integral is independent of the choice of the specific boundary point, hence item (1) of the lemma will also hold if we replace $\tilde\ttt_\XXX$ by a $\uplus$-sum of similar sections defined using different boundary marked points in the same connected component.

\begin{nn}\label{not: descendent multisections pointing to special points}
We denote by $\ttt_r$ the multisection of $\CL_1 \to \oPM_{\Gamma}$ that is the $\uplus$-sum of multisections pointing to all boundary markings in the connected component of $z_1$ that are $r$-points.
Define similarly $\ttt_s$. Lastly, let $\ttt_\Root$  be the multisection which points to the root of the connected component of $z_1$.
\end{nn}

\begin{nn}\label{not: pullback multisections to create homotopies}
Consider a section $t_w$ of any of the above types. Recall that $\srest$ is a multisection of $E_1$ and we write $\ttts = \srest\oplus \ttt_w$ for a multisection of $E$. Note that $\ttt_w$ need not be canonical, hence $\ttts$ need not be. As a consequence, we will need a homotopy from
$\ttt_w$ to a canonical multisection.
Write
\begin{equation} \label{Decomp Witten}
E = E_1 \oplus \CL_1
\end{equation}
with $E_1$ as in \eqref{E1 definition}. Let $\s_2$ be a canonical multisection of $\CL_1$. Then we can consider the canonical multisection $\s = \srest \boxplus \s_2$ of $E$.

We apply Lemma \ref{lem:hom_trr} with respect to $\mathbf{s}$ and $\mathbf{s}_1\oplus t_w$ to find a canonical family of multisections $\rho^*$
of $\CL_1$, such that, for each balanced $\Gamma$, the homotopy given at time
$u$ by
\begin{equation}\label{trr homotopy for the bullets}
H^\Gamma(t_w)(u,-):= \srest \boxplus \left((1-u)t_w+u(1-u)\rho^\Gamma+u \mathbf{s}_2\right)
\end{equation}
is transverse. Here, as we did in Proposition~\ref{invariants are independent of extension of boundary}, we are viewing multisections $\ess_1, t_w, \ess_2$ with base $\oPM_{\Gamma}$ as multisections with base $[0,1] \times \oPM_{\Gamma}$ via pulling back the bundle under the projection.
\end{nn}

\begin{nn}\label{Graphs for OTRR}
Let $\Gamma$ be a smooth graded $W$-spin graph with an internal tail $1 \in I(\Gamma)$, and let $\Xi_1, \ldots, \Xi_f$ be the connected components of $\Gamma$. Suppose that $\Xi_j$ contains the tail  $1$. Consider graded graphs $\Lambda \in \partial \Gamma$ where
\begin{itemize}
\item $\Lambda$ has connected components $\Lambda_1, \ldots, \Lambda_f$ where $\Lambda_j \in \partial \Xi_j$ and $\Lambda_i = \Xi_i$ for $i \not= j$;
\item $\Lambda_j$ has two open vertices $v_0$ and $v_j$, with one edge $e$ between them;
\end{itemize}
\begin{enumerate}[(a)]
\item Denote by ${\mathcal G}_{WC}$ the collection of
\emph{wall-crossing graphs}, i.e., all $\Lambda$ as above where the following are also satisfied:
\begin{itemize}
\item the vertex $v_0$ contains the internal tail $1$ and is critical
with respect to the descendent vector $\vecd+{\bf e}_1$;
\item the vertex $v_j$ is rooted and balanced with respect to the descendent vector $\vecd+{\bf e}_1$;
\item the half edges $h_0$ and $h_j$ of $e$ must fall into one of the following two cases:
\begin{enumerate}[(i)]
\item $\tw(h_1)= (r-2,0), \alt(h_1) = (1,0), \tw(h_2) = (0,s-2),$ and $\alt(h_2)=(0,1)$;
\item $\tw(h_1)= (0,s-2), \alt(h_1) = (0,1), \tw(h_2) = (r-2,0),$ and $\alt(h_2)=(1,0)$;
\end{enumerate}
\end{itemize}
\item Denote by $\mathcal{G}_{\XCH}$ the set of all $\Lambda$ as above  where the following are satisfied:
\begin{itemize}
\item the vertex $v_0$ is rooted and has the internal tail $1$;
\item the vertex $v_j$ is an exchangeable vertex; and
\item the half-edges satisfy the following:
$$\tw(h_0)= (r-2,s-2), \alt(h_0) = (1,1), \tw(h_j) = (0,0),\text{ and }\alt(h_j)=(0,0).$$
\end{itemize}

\item Denote by ${\mathcal G}_{\Cont}$ the set of all $\Lambda$ as above where the following are satisfied:
\begin{itemize}
\item the vertex $v_0$ is rooted and has the internal tail $1$;
\item  $\CB\Lambda$ is balanced with respect to the descendent vector $\vecd$;
\item the half-edges satisfy the following: $$\tw(h_0)= (0,0), \alt(h_0) = (0,0), \tw(h_j) = (r-2,s-2),\text{ and }\alt(h_j)=(1,1).$$
\end{itemize}
\end{enumerate}
\end{nn}

Recall the definition of a degeneration of a pre-graded $W$-spin graph from Definition~\ref{def:smoothingraph}(4). If we let $\Gamma:= \Gamma_P$ for some $P\in \mathcal{P}_h(I, \vecd+\eee_1)$ (see Notation \ref{P sets of balanced}),
then we denote by
\begin{equation}\begin{aligned}
{\mathcal G}_{\mathrm{WC}}(P) &:= \{ \Lambda \in  {\mathcal G}_{\mathrm{WC}} \ | \ \Lambda \text{ is a degeneration of } \Gamma_P\},\\
{\mathcal G}_{\XCH}(P) &:= \{ \Lambda \in  {\mathcal G}_{\XCH} \ | \ \Lambda \text{ is a degeneration of } \Gamma_P\},\\
{\mathcal G}_{\Cont}(P) &:= \{ \Lambda \in  {\mathcal G}_{\Cont} \ | \ \Lambda \text{ is a degeneration of } \Gamma_P\}.
\end{aligned}\end{equation}

\begin{rmk}\label{rmk: boundary components}
The graphs in the collection ${\mathcal G}_{WC}(P)$ are precisely the subcollection of all graphs of the form $\Gamma_{P,j,\eps,I_0,k_1(0),k_2(0),R_0,S_0}$ described in Part (B) of the proof of Theorem \ref{thm:A_mod_invs}, with the additional constraint that $1\in I_0$.
\end{rmk}

Using this notation, we can provide a new formula for $\left\langle\prod_{i\in {I(\Gamma)}}\tau^{(a_i,b_i)}_{d_i+\delta_{1i}}\sigma_1^{k_1(\Gamma)}\sigma_2^{k_2(\Gamma)}\sigma_{12}^{k_{12}(\Gamma)}\right\rangle^{\mathbf{s}^\Gamma,o}$:

\begin{lemma}\label{lem: homotopy TRR contributions classify}
Let $P\in\mathcal{P}_h(I,\vecd+\eee_1)$. Take $\Gamma := \Gamma_P$ as in \eqref{def: GammaP}.
Let $\srest$ be a canonical multisection of $E_1$ and
$\s_2$ a canonical multisection of $\CL_1$. Take $\s = \srest + \s_2$ and $\ttts_\XXX = \srest + t_\XXX$ and consider the homotopy $H:=H^\Gamma(t)$ between them, as in Equation~\eqref{trr homotopy for the bullets}. Then
\begin{equation}\label{eq:int_num_in_terms_of_trr}
\left\langle\prod_{i\in {I(\Gamma)}}\tau^{(a_i,b_i)}_{d_i+\delta_{1i}}\sigma_1^{k_1(\Gamma)}\sigma_2^{k_2(\Gamma)}\sigma_{12}^{k_{12}(\Gamma)}\right\rangle^{\mathbf{s}^\Gamma,o}=
\int_{\oPM_\Gamma}e(E,\ttts_\XXX)+Z_{WC}^\Gamma(\ttts_\XXX)+Z_{\Cont}^\Gamma(\ttts_\XXX)+Z_{\XCH}^\Gamma(\ttts_\XXX),
\end{equation}
where
\begin{equation}\begin{aligned}\label{all the contributions in the TRR homotopy}
Z^\Gamma_{WC}(\ttts_\XXX)&:=\sum_{\Lambda\in {\mathcal G}_{WC}(P)}\#Z(H|_{[0,1]\times\oPM_\Lambda});\\
Z^\Gamma_{\Cont}(\ttts_\XXX)&:=\sum_{\Lambda\in {\mathcal G}_{\Cont}(P)}\#Z(H|_{[0,1]\times\oPM_\Lambda}); \\
Z^\Gamma_{\XCH}(\ttts_\XXX)&:=\sum_{\Lambda\in {\mathcal G}_{\XCH}(P)}\#Z(H|_{[0,1]\times\oPM_\Lambda}).
\end{aligned}\end{equation}
\end{lemma}

\begin{proof} By Lemma~\ref{lem:hom_trr},
the homotopy $H$ of
\eqref{trr homotopy for the bullets} from $\ttts_w$ to $\ess$ is transverse to zero on each stratum $[0,1]\times\CM^W_\Lambda$ and the projection of $H$ to $E_1$ equals $\srest$ at all times. Moreover, $H$ is non-vanishing on $[0,1]\times
U_+$, where $U_+$ is as in \eqref{eq:U_+}, as $\s_1$ is strongly positive
(see Definition~\ref{def: strongly positive}).
Lemma~\ref{lem:zero diff as homotopy} then gives
\begin{equation}\label{eq:rsH}
\int_{\oPM_\Gamma} e(E;\mathbf{s}) - \int_{\oPM_\Gamma} e(E;\ttts_w) = \# Z(H|_{[0,1]\times(\partial\oPM_{\Gamma})}).
\end{equation}

We will now show that $H$ only vanishes on those boundary strata indexed by graphs that are in the collections ${\mathcal G}_{WC}$, ${\mathcal G}_{\Cont}$, and ${\mathcal G}_{\XCH}$, yielding \eqref{eq:int_num_in_terms_of_trr}. For $\Lambda\in\partial^0\Gamma$ with $\dim\CM^W_\Lambda\leq\dim\CM^W_\Gamma-2$,
$H$ will not vanish as it is a transverse homotopy. If instead $\Lambda
\in\partial^0\Gamma$ is irrelevant, then $\s_1$ is non-vanishing
on $\CM^W_{\Lambda}$ as $\s_1$ is strongly positive,
and hence $H$ is non-vanishing on $\CM^W_{\Lambda}$.
We thus may assume that $\dim\CM^W_\Lambda=\dim\CM^W_\Gamma-1$ and $\Lambda$ is relevant.

We can then describe $\Lambda$. Write $\Xi_1, \ldots, \Xi_f$ for the connected components of $\Lambda$.
Using the assumption that $\dim\CM^W_\Lambda=\dim\CM^W_\Gamma-1$, we see
$\Lambda$ must have a single boundary edge $e$ which connects open vertices
$v_1$ and $v_2$. Thus we may assume that $\Xi_1$ contains the edge $e$
and vertices $v_1,v_2$, and the remaining connected components $\Xi_2,\ldots,
\Xi_f$ are rooted smooth open graded graphs.
Then one can  decompose
\begin{equation}\label{eqn: Witten bundle for conn components}
E_1 = (E_1)_{v_1}\boxplus{(E_1)_{v_2}}\boxplus \Big((E_1)_2\boxplus\cdots
\boxplus (E_1)_f\Big),\quad\quad E = E_{v_1}\boxplus{E_{v_2}}\boxplus \Big(E_2\boxplus\cdots\boxplus E_f\Big),
\end{equation}
where $(E_1)_{v_i}=E_{v_i}(\vecd),~E_{v_i}=E_{v_i}(\vecd+{\bf e}_1)$ and  $(E_1)_{i}=E_{\Xi_i}(\vecd),~E_{i}=E_{\Xi_i}(\vecd+{\bf e}_1)$.

We next observe that for $H$ to vanish on this stratum, the internal tail
$1$ must be contained in $\Xi_1$. Indeed, if not, the component of
$H$ in $E_{v_1}\boxplus E_{v_2}$ is constant, coinciding with the
$E_{v_1}\boxplus E_{v_2}$ component of $\s$, and hence by transversality
has no zeros. Thus for $i\neq 1$, the bundle $E_i$ is the same as $(E_1)_i$.

As in Part (A) of the proof of Theorem~\ref{thm:A_mod_invs}, there are four cases which we must discuss, depending on the different potential twists for the half-edges in the first connected component. Without loss of generality take $v_1$ to be the rooted vertex of
$\Xi_1$ with corresponding boundary half-edge $h_1$ and $v_2$ to be the additional vertex with half-edge $h_2$. We can have:
\begin{enumerate}[(i)]
\item $\tw(h_1)= (r-2,0), \alt(h_1) = (1,0), \tw(h_2) = (0,s-2),$ and $\alt(h_2)=(0,1)$;
\item $\tw(h_1)= (0,s-2), \alt(h_1) = (0,1), \tw(h_2) = (r-2,0),$ and $\alt(h_2)=(1,0)$;
\item $\tw(h_1)= (r-2,s-2), \alt(h_1) = (1,1), \tw(h_2) = (0,0),$ and $\alt(h_2)=(0,0)$;
\item $\tw(h_1)= (0,0), \alt(h_1) = (0,0), \tw(h_2) = (r-2,s-2),$ and $\alt(h_2)=(1,1)$.
\end{enumerate}

Recall from Equation~\eqref{eq:open_rank2} that
the rank of the Witten bundle associated to a given vertex $v$
satisfies
\begin{equation}
\label{eq:rank parity}
\rank E_v \equiv k_1(v) + k_2(v) \pmod 2.
\end{equation}
On the other hand,
\begin{equation}
\label{eq:dim parity}
\dim \CM^W_v \equiv 1+\#\hbox{(boundary half-edges incident to $v$)}\pmod 2.
\end{equation}

Note that since $H$ agrees
with $\s$ on $E_i$, $2\le i\le f$ and $\s$ is transversal,
in order for $H$ to vanish we need that
\begin{equation}
\label{eq:rank E f}
\rank E_i \le \dim\CM^W_{\Xi_i},\quad 2\le i \le f.
\end{equation}
Since $\rank E=\dim \CM_{\Lambda}^W + 1$ by the fact that $\Gamma$ is
balanced with respect to $\vecd+\mathbf{e}_1$, we must also have
\begin{equation}
\label{eq:rank dim_2}
\rank E_{v_1}+\rank E_{v_2} +\sum_{i=2}^f \rank E_i = 1+
\dim\CM_{v_1}^W+\dim \CM_{v_2}^W+\sum_{i=2}^f \dim \CM_{\Xi_i}^W.
\end{equation}

We now carry out a more careful analysis of ranks and dimensions
for the vertices $v_1$ and $v_2$ in the four cases to determine when
$H$ may vanish.
\begin{enumerate}
\item[(i) \& (ii)] Cases (i) and (ii) are similar to one another. 
By \eqref{eq:rank parity} and \eqref{eq:dim parity}, we have
\begin{align}
\label{eq:parity relation}
\begin{split}
\rank E_{v_1}\equiv {} & k_1(v_1)+k_2(v_1)\equiv k_1(v_1)+k_2(v_1)+k_{12}(v_1)+1
\equiv \dim \CM^W_{v_1}
\pmod 2;\\
\rank E_{v_2}\equiv{} & k_1(v_2)+k_2(v_2)\equiv k_1(v_2)+k_2(v_2)+1-1\equiv \dim \CM^W_{v_2}-1
\pmod 2.
\end{split}
\end{align}

If $1 \in I(v_1)$, then the homotopy $H$ will also be constant in time for the
$v_2$ component.
Thus, in order for $H$ to vanish, we need by transversality that
$\rank E_{v_2}\leq \dim\CM^W_{v_2}$.
By \eqref{eq:parity relation}, we know then that $\rank E_{v_2}< \dim\CM^W_{v_2}$, but the parities of $\rank E_{v_1}$ and $\dim \CM^W_{v_1}$ agree.
Thus by \eqref{eq:rank dim_2},
we then must have $\rank E_{v_1}> \dim\CM^W_{v_1}$, and hence
by parity,
$\rank E_{v_1} \geq \dim \CM^W_{v_1} + 2$. By transversality, the $v_1$-component of the homotopy $H|_{[0,1]\times\CM^W_{\Lambda}}$ cannot vanish, hence $H|_{[0,1]\times\CM^W_{\Lambda}}$ does not vanish.

When $1\in I(v_2)$ and $H$ vanishes, we claim that the graph $\Lambda$ is a critical boundary graph with respect to the descendent vector $\vecd+{\bf e}_1$.
Indeed, note that now $H$ is constant
on the $v_1$ component, so necessarily $\rank E_{v_1}\le \dim\CM_{v_1}^W$. Thus by \eqref{eq:rank dim_2}, $\rank E_{v_2}\ge
\dim \CM_{v_2}^W+1$. If the inequality is strict, then $H$ will
non-vanishing on $\CM_{\Lambda}^W$ by transversality. Thus in this
case we have $\rank E_i=\dim\CM_{\Xi_i}^W$ for $2\le i \le f$,
$\rank E_{v_1}=\dim \CM_{v_1}^W$ and $\rank E_{v_2}=\dim \CM_{v_2}^W + 1$,
i.e., $\Lambda$ is a critical boundary.
We denote the collection of graphs which parameterize such boundaries by ${\mathcal G}_{WC}$, as in Notation~\ref{Graphs for OTRR}.

Case (i) here corresponds to the case where $\eps = 1$ and case (ii) corresponds to the case where $\eps = 2$ in Remark \ref{rmk: boundary components}.

\item[(iii)] In case (iii), when we pass to the base, we will forget the half-edge $h_2$. As in \eqref{eq:strata dim} and \eqref{eq:base dim},
we must have $\dim\CM^W_{\CB\Lambda}=\dim\CM^W_\Gamma-2+\sigma$, where
$\sigma=1$ if $v_2$ becomes partially stable after forgetting $h_2$
and $\sigma=0$ otherwise.

Suppose $\sigma=1$. Then, as discussed in Case (iii) of Part (A) of the proof of Theorem~\ref{thm:A_mod_invs}, since $\Lambda$ is assumed to be relevant, the vertex
$v_2$ must be exchangeable.  In this case, $\Lambda$ is one of the
the exchangeable graphs in the set ${\mathcal G}_\XCH$, as defined in Notation~\ref{Graphs for OTRR}.

Now suppose $\sigma=0$. We have two subcases.
\begin{enumerate}
\item[(a)]
$1 \in I(v_1)$. 
Recall the map $F_\Lambda: \oPM_\Lambda^W \to \oPM_{\CB\Lambda}^W$.
Here by definition, $F_{\Lambda}=
\text{For}_{\text{non-alt}}\circ\Detach_{E^0(\Lambda)}$. We write
$\Xi=\detach_{E^0(\Lambda)}\Lambda$ and $F=\Detach_{E^0(\Lambda)}$. There is a canonical identification of the bundles $\CL_1$ on the moduli spaces $\oPM_\Lambda$ and $\oPM_{\CB\Lambda}$ under the map $F_{\Lambda}$ by Observation~\ref{isomorphism forgetting non-alt for descendents}(i).
Note that
$F$ is a finite covering map. Indeed, given two $W$-spin disks $\Sigma_1,\Sigma_2$
representing elements of $\oPM_{v_1}$ and $\oPM_{v_2}$ respectively, 
we may obtain an element of $\oPM_{\Lambda}$ in two different ways,
by attaching the unique boundary point of $\Sigma_2$ with twist $(0,0)$
to one of the two boundary points of $\Sigma_1$ with twist $(r-2,s-2)$.
Recall $t_{\times}$ is the section of $\CL_1$ pointing towards the root.
Then $t_{\times}$ is not the pull-back of a section of $\CL_1$ on
$\oPM_{\Xi}$ as on $\oPM_{\Xi}$ we cannot in general
distinguish between the root
and the boundary point corresponding to $h_1$. Similarly, if $t_{h_1}$
denotes the section on $\oPM_{\Lambda}$ pointing from $z_1$ to the 
half-node corresponding to $h_1$, $t_{h_1}$ is not the pull-back of a 
section from $\oPM_{\Xi}$. Lastly, $t_r, t_s$ or $t_w$ for
$w\in I(v_2)$
may contain $t_{h_1}$ as a weighted summand, see Remark 
\ref{rem:degenerate pointing}. Thus it is possible that none of these
sections are pulled back from $\oPM_{\Xi}$. On the other hand, recalling
the definition of pushforward from Notation \ref{not:pushforward},
it is easy to see that $F_*t_{\times} = F_*t_{h_1} = t_{(r,s)}$,
the $\uplus$-sum of multisections pointing to all boundary markings
which are fully twisted. Further, $t_{(r,s)}$ is then pulled back from
$\oPM_{\CB\Lambda}$. Hence in any event, $F_*t_w$ is pulled back from
$\oPM_{\CB\Lambda}$. Now, the remaining multisections used to construct
the homotopy $H$ in Lemma~\ref{lem:hom_trr} are canonical, so we have that 
$F_*H$ is also pulled back from $[0,1]\times \oPM_{\CB\Lambda}$. 
Since $\dim ([0,1]\times\oPM^W_{\CB \Lambda})  = \dim \oPM^W_{\Gamma} -1$, we have by transversality that $F_* H$ does not vanish, and hence 
$H$ itself does not vanish.
 Thus we do not get any contributions to the number of zeroes of the homotopy from this case.
\item[(b)] $1 \in I(v_2)$.
Let $k_{\varnothing}(v_i)$ denote the number of boundary half-edges where $(\tw_i, \alt_i) = (0,0)$ for all $i$. We use  \eqref{eq:rank parity} and \eqref{eq:dim parity} to see that
\begin{align}
\label{eq:parity relation_case iii}
\begin{split}
\rank E_{v_1}\equiv {} k_1(v_1) + k_2(v_1) \equiv {} & k_1(v_1)+k_2(v_1)+k_{12}(v_1) + 1 +1 \\
\equiv {} &
 \dim \CM^W_{v_1}+1
\pmod 2;\\
\rank E_{v_2} \equiv{}   k_1(v_2) + k_2(v_2) \equiv {} & k_1(v_2) + k_2(v_2) + k_{\varnothing}(v_2) + 1\\
 \equiv {} & \dim \CM^W_{v_2}
\pmod 2.
\end{split}
\end{align}

Combining the facts that $H$ is constant on the $v_1$ component with the parities computed in~\eqref{eq:parity relation_case iii}, we must have $\rank E_{v_1}< \dim\CM_{v_1}^W$ if $H|_{[0,1]\times\CM_\Lambda}$ is to vanish. Thus, by applying~\eqref{eq:rank E f} and~\eqref{eq:rank dim_2},  we have that $\rank E_{v_2} > \dim \CM_{v_2}^W$. In turn, \eqref{eq:parity relation_case iii} implies that
 $$
 \rank E_{v_2} \ge \dim \CM_{v_2}^W + 2.
 $$
Thus, by transversality, the $v_2$-component of the homotopy $H|_{[0,1]\times
\CM^W_{\Lambda}}$ cannot vanish, hence $H|_{[0,1]\times\CM^W_{\Lambda}}$ does not vanish.
\end{enumerate}

\item[(iv)] We again can compute using  \eqref{eq:rank parity} and \eqref{eq:dim parity} that in case (iv)
\begin{align}
\label{eq:parity relation_case iv}
\begin{split}
\rank E_{v_1}\equiv {} &
 \dim \CM^W_{v_1}+1
\pmod 2;\\
\rank E_{v_2} \equiv{} &  \dim \CM^W_{v_2}
\pmod 2.
\end{split}
\end{align}
Note that the parities in case (iv) are the same as in \eqref{eq:parity relation_case iii}. Thus, if $1 \in I(v_2)$, we can apply~\eqref{eq:parity relation_case iv}, ~\eqref{eq:rank E f} and~\eqref{eq:rank dim_2} as before to imply that
$$
\rank E_{v_2} \ge \dim \CM_{v_2}^W + 2.
$$
This again implies by transversality that the $v_2$-component of the homotopy $H|_{[0,1]\times\CM^W_{\Lambda}}$ cannot vanish, hence $H|_{[0,1]\times\CM^W_{\Lambda}}$ does not vanish.

If $1 \in I(v_1)$, then it is possible that $\CB\Lambda$ is balanced with respect to $\vecd$. Indeed, note that the base $\CB\Lambda$  forgets the half-edge $h_1$. Then the rank of $(E_1)_{v_1}=E_{v_1}(\vecd)$ has the same parity as $\CM^W_{\CB v_1}$, hence one can have that $\rank (E_1)_{v_1}=\dim
\CM^W_{\CB v_1}$, and so
$$
\quad \qquad \rank E_{v_1} = \dim \CM^W_{\CB v_1} + 2, \quad \rank E_{v_2} =  \dim \CM^W_{ v_2}, \text{ and } \rank E_i = \dim\CM^W_{\Xi_i},\quad 2\le i \le f.
$$
These graphs make up the collection
${\mathcal G}_{\Cont}$. 
\end{enumerate}

Using the analysis of contributing boundaries and \eqref{eq:rsH}, we
conclude the lemma.
\end{proof}

The bulk of the analysis involves treating each of the summands
on the right-hand side of \eqref{eq:int_num_in_terms_of_trr}
separately. We begin with the easy case:

\begin{lemma}\label{lem:xch_vanishing_invariants}
$Z^\Gamma_{\XCH}(\ttts_\XXX)=0$.
\end{lemma}

\begin{proof}
The argument is similar to the corresponding argument in Lemma~\ref{lem:exchange vanishes}.  Recall that the homotopy $H^\Gamma(t_w)$
of \eqref{trr homotopy for the bullets}
depends only on the multisections $\srest, \rho, \mathbf{s}_2$, and $t_w$. For any $\Lambda \in \mathcal{G}_{\XCH}(P)$, we have the corresponding graph $\Lambda' := \xch(\Lambda) \in \mathcal{G}_{\XCH}(P)$. By Observation~\ref{obs:exchangeable_and_base}, we know that the canonical multisections $\srest, \rho$, and $\mathbf{s}_2$ are the same on $\Lambda$ and $\Lambda'$. Moreover, the multisection $t_w$ is independent of the cyclic ordering on the exchangeable graph, so we have that $\XCH^* t_w = t_w$. Thus we have that
\[|\#Z(H|_{[0,1]\times\oPM_{\Lambda}})|=|\#Z(H|_{[0,1]\times
\oPM_{\Lambda'}})|.\]
Recalling that, by Proposition~\ref{prop: orient exchange}, the map $\XCH: \oPM_\Lambda \to \oPM_{\Lambda'}$ is orientation reversing, we conclude that
\[\#Z(H|_{[0,1]\times\oPM_{\Lambda}})=-\#Z(H|_{[0,1]\times\oPM_{\Lambda'}}).\]
Summing over all $\Lambda\in \mathcal{G}_{\XCH}(P)$ implies the result.\end{proof}

We now restrict our attention to computing $Z^\Gamma_{WC}(\ttts_\XXX)$. This requires understanding the boundary graphs $\Lambda\in \mathcal{G}_{WC}(P)$ for a fixed $P \in \mathcal{P}_h(I, \vecd+ \eee_1)$. As said in Remark~\ref{rmk: boundary components}, the graphs in the collection ${\mathcal G}_{WC}(P)$ are precisely the subcollection of all graphs of the form $\Gamma_{P,j,\eps,I_0,k_1(0),k_2(0),R_0,S_0}$ in Notation \ref{not:Gamma mess of subscripts}, with the additional constraint that $1\in I_0$. Recall that such a graph $\Gamma_{P,j,\eps,I_0,k_1(0),k_2(0),R_0,S_0} \in \mathcal{G}_{WC}(P)$  will have the property that 
$$\CB\Gamma_{P,j,\eps,I_0,k_1(0),k_2(0),R_0,S_0} \prec \Lambda_Q$$
 for some
\begin{equation}\label{redefining Q in trr}
Q = \{(I_0,k_1(0),k_2(0)),\ldots, (I_h,k_1(h),k_2(h))\} \in \mathcal{Q}_h(I, \vecd+\eee_1),
\end{equation}
where $1 \in I_0$ and $\widehat{Q}_j = P$ where $\widehat Q_j$ is as in~\eqref{defn: hatQj}. We can then classify all graphs in ${\mathcal G}_{WC}(P)$ as
\begin{equation} \label{eqn:allWCContribsToP}
{\mathcal G}_{WC}(P)  = \bigcup_{\substack{ Q \in \mathcal{Q}(I, \vecd+\eee_1)  \\ \text{where $1\in I_0$ and} \\ P= \widehat{Q}_j \text{ for some $j$}}} \left(\bigcup_{R_0=0}^{k_1(j) - 1} A^1_{Q, j, R_0} \right) \cup \left( \bigcup_{S_0=1}^{k_2(j)-1} A^2_{Q, j, S_0}\right).
\end{equation}
with notation as in \eqref{defn of A^1QjR}.

\begin{obs}\label{obs:restriction_of_t_to_WC}
Suppose $\Lambda=\Gamma_{P,j,\eps,I_0,k_1(0),k_2(0),R_0,S_0}\in \mathcal{G}_{WC}$ where $\CB\Lambda = \Lambda_Q$ with $Q$ as in~\eqref{redefining Q in trr}. This implies that the $j$th connected component in the critical boundary graph consists of two open vertices $v_0$ and $v_j$ sharing a boundary edge so that
\begin{itemize}
\item $1 \in I(v_0)$;
\item $v_0$ is a critical vertex;
\item $v_j$ is a balanced vertex;
\item $h_0$ is the boundary half-edge in vertex $v_0$;
\item $h_j$ is the boundary half-edge in vertex $v_j$.
\end{itemize}
By definition, if $\eps = 1$ then $h_0$ is an $s$-point and if $\eps = 2$, then $h_0$ is an $r$-point. Consider the multisections of $\CL_1\to\oCM^W_{v_0}$ given by $\ttt_r$, $\ttt_s$, $\ttt_\Root$ and $\ttt_{z_2}$ from Notation~\ref{not: descendent multisections pointing to special points}. By using the canonical identification of $\CL_1\to\oCM^W_{v_0}$ and $(\CL_1\to\oCM^W_\Gamma)|_{{\oCM^W_\Lambda}}$ and Definition \ref{def:union_of_sections} and Remark~\ref{rem:degenerate pointing}, we compute that
\begin{equation}\begin{aligned}
\ttt_r|_{\oPM_\Lambda}  &=  (k_1(0) + 1)\ttt^{v_0}_r \uplus (k_1(j) -1) \ttt^{v_0}_{h_0}; \\
\ttt_s|_{\oPM_\Lambda} &=  (k_2(0) + 1) \ttt^{v_0}_s \uplus (k_2(j) -1) \ttt^{v_0}_{h_0}; \\
\ttt_\Root|_{\oPM_\Lambda} &= \ttt^{v_0}_{h_0};\\
\ttt_{z_2}|_{\oPM_\Lambda}&=\begin{cases} \ttt^{v_0}_{z_2}, &  \text{if $2\in I(v_0)$;}\\  \ttt^{v_0}_{h_0}, &\text{if $2 \in I(v_j)$.}\end{cases}
\end{aligned}
\end{equation}
\end{obs}

While almost all multisections that factor into the construction in a homotopy $H(t_\bullet)$ are canonical, the multisection $t_{h_0}$ on the boundary is not, but we can prove the following technical lemma that deals with this issue. 

\begin{lem}\label{th0 technical annoyance}
Suppose $P \in \mathcal{P}(I, \vecd)$ and $\Lambda:=\Gamma_{P,j,\eps,I_0,k_1(0),k_2(0),R_0,S_0}\in \mathcal{G}_{WC}(P)$ so that $\CB\Lambda = \Lambda_Q$ for some $Q \in \mathcal{Q}(I, \vecd)$. Take the multisection $t_{h_0}^{\Lambda}$ pointing at the half-node corresponding to the half-edge $h_0$ on the graph $\Lambda$. Then
$$
\# Z(H(t^{\Lambda}_{h_0})|_{[0,1]\times\oPM_\Lambda} ) =\begin{cases}  \#Z(H (t_s^{v_0})|_{[0,1]\times\oPM_\Lambda}) & \text{ if $\eps = 1$}, \\  \#Z(H(t_r^{v_0})|_{[0,1]\times\oPM_\Lambda} ) & \text{ if $\eps = 2$}.\end{cases}
$$
\end{lem}

\begin{proof}
Without loss of generality, we assume that the graph $\Lambda$ is connected. If it were disconnected, then the homotopy as defined in \eqref{trr homotopy for the bullets} restricted to the components not containing the tail $1$ is constant so the number of zeroes is fixed and we reduce to the connected case.

Now consider a fixed connected component $\CM$ of $\oPM_{\Lambda}$. Recall from \S\ref{or:background in rank one}, since we are in the unlabeled case, there can be multiple cyclic orderings of the boundary tails that correspond to the same connected component. We say two cyclic orderings are equivalent if they correspond to the same connected component in $\oPM_{\Lambda}$. Note the data encoded in each equivalence class $[\hat\pi]$ of cyclic orderings may be viewed as the
pattern of $r$- and $s$-points following the root in the anti-clockwise
direction on the boundary of the stable disk.

Any cyclic ordering $\hat\pi$ in the equivalence class $[\hat\pi]$ induces a cyclic ordering $\hat\pi_0$ on all
half-edges attached to $v_0$ in $\Lambda$.
This further induces an ordering $\pi_0$ on the boundary half-edges of the vertex $v_0$ by putting the half-edge $h_0$ at the beginning of the ordering (as seen in Notation~\ref{nn: boundary orientation case} and Figure~\ref{fig:boundary edge orient}). Analogously as above, we say two orderings on $v_0$ are equivalent if they differ by a permutation of the half-edges of $v_0$ preserving $r$- points and $s$-points. 

Let us assume that $\eps = 1$ and take the group $G:= \Z / (k_2(0)+1)\Z$. It acts on the set
$$
\Conn(\oPM_{\Lambda}):= \{ \CM \ | \ \CM \text{ a connected component of $\oPM_{\Lambda}$}\}
$$
as follows. Given a connected component $\CM$ of $\oPM_{\Lambda}$, there is an equivalence class of orderings $[\pi_0]$ as above.
Take $n\in\{0,\ldots,k_2(0)\}$ representing an element of $\Z/(k_2(0) + 1)\Z$.
Let $s(n)$ be the positive integer such that $\pi_0(s(n)+1)$ is the $n$th
$s$-point in the ordering $\pi_0$, not counting the initial half-node.
In particular, we take $s(0)=0$, i.e., the half-node is in the $0^{th}$ location in the ordering.
There exists a connected component $n \cdot \CM$ of $\oPM_{\Lambda}$ where the analogous equivalence class of orderings $[\pi_n]$ on the boundary half edges of $v_0$ has the property that
$$
\tw(\pi_0(i+s(n))) = \tw(\pi_n(i))
$$
for any $i$ (viewed modulo $k_2(0) + 1$) and representatives $\pi_0$ and $\pi_n$ of $[\pi_0]$ and $[\pi_n]$. Geometrically, in the unlabeled case, this can be viewed as changing which $s$-point on the disk corresponding to $v_0$ is the half-node by `rotating to the $n$th $s$-point found anti-clockwise.' Note that there is a canonical isomorphism
$$
\rho_n: \CM \to {n \cdot \CM}
$$
 given by changing the half-node on the disk corresponding to the vertex $v_0$.

As orbits of the group action partition the set of connected components, we treat each orbit of connected components of $\oPM_\Lambda$ separately to prove the lemma. Fix an orbit $O(\CM)$ of $\oPM_{\Lambda}$. The orbit has cardinality $\#O(\CM)$ where $\#O(\CM) | (k_2(0)+1)$. If $\#O(\CM) \ne k_2(0)+1$, then this means that there is a rotational symmetry in the cyclic ordering $\hat\pi_0$ that preserves the pattern of $r$- and $s$-points on the vertex $v_0$. Take the minimal representative $\tau \in \{1, \ldots, k_2(0)+1\}$ of a generator of the stabilizer subgroup $G_\CM \subseteq \Z/(k_2(0)+1)\Z$ for $\CM$. This means
$$
\tw(\pi_0(i+s(n))) = \tw(\pi_n(i)) \text{ for all $i$} \Longleftrightarrow n = k\tau.
$$

Take the map to the base   given in Definition~\ref{def:base} 
$$
F_{\Lambda}: \oPM_\Lambda \to \oPM_{\Lambda_Q}.
$$
As in Observation~\ref{isomorphism forgetting non-alt for descendents}(i), there is a canonical identification for the cotangent bundle $\CL_1$ on the moduli spaces above.
Recall that the map $F_\Lambda$ is finite with degree $(-1)^{k_2(0)}(k_2(0)+1)$.

Consider the multisection $t_{h_0}^{v_0}$ on $\CM$. Once we detach this edge, the half-edge $h_0$ will become indistinguishable from the tails that are $s(k\tau) \, (\text{mod } k_1(0)+k_2(0) + 2)$ places away from $h_0$ for any $k \in \Z$. Define the following subset of boundary marked points $$T(\CM) = \{ x_{s(k\tau)} \ | \  0 \le k < \tfrac{k_2(0) + 1}{\tau} \}.$$ Thus we have that
$$
(F_{\Lambda})_* (t_{h_0}^{v_0}) = \biguplus_{x \in T(\CM)} t_{x},
$$
with the pushforward as defined in Notation \ref{not:pushforward}.
Using this fact and that $\rho, \mathbf{s}_1$, and $\mathbf{s}_2$ are pulled back from the base, we can see that the pushforward of $H(t_{h_0}^{v_0})$ on this connected component of the boundary stratum $\oPM_{\Lambda}$ is
\begin{equation}\label{pushforward of th0}
(F_{\Lambda})_* ( H(t_{h_0}^{v_0})|_{\CM})= H(\uplus_{x\in T(\CM)} t_{x})|_{F_{\Lambda}( \CM)}.
\end{equation}
 Note that we also have that
$$
F_{\Lambda}^*  \big( H(\uplus_{x\in T(\CM)} t_{x})|_{F_{\Lambda}( \CM)}\big) = H(\uplus_{x \in T(\CM)} t_{x})|_{\CM}.
$$
 This means that the number of zeros of the homotopy $ H(\uplus_{x \in T(\CM)} t_{x})|_{ \CM}$ will not depend on which half-edge of $v_0$ corresponds to the half-node, i.e., for all $n \in G$ we have
\begin{equation}\label{action leaves the tail pointing invariant}
H(\uplus_{x\in T(\CM)} t_{x})|_{ \CM} = \rho_n^*( H(\uplus_{x \in T(\CM)} t_x)|_{(n \cdot \CM)}).
\end{equation}
By applying the pullback $F_\Lambda^*$ to Equation \eqref{pushforward of th0} and taking zeros, we have that
\begin{equation}\label{def Z CM}
Z(\CM) := \# Z(H(t_{h_0}^{v_0})|_{ \CM} ) = \#Z(H(\uplus_{x\in T(\CM)} t_{x})|_{ \CM}).
\end{equation}
Note that we have a natural bijection between the $s$-points on $v_0$ and the set $\bigcup_{\CM' \in O(\CM)} T(\CM') $. Using this bijection and the fact that $|T(\CM)| = |T(\CM')|$ for all $\CM' \in O(\CM)$, we have that
\begin{equation}\label{partitioning pointing towards s}
\biguplus_{\CM' \in O(\CM)} \biguplus_{x \in T(\CM')} t_{x} = t_s.
\end{equation}
We then see from ~\eqref{action leaves the tail pointing invariant}, \eqref{def Z CM}, and~\eqref{partitioning pointing towards s} that
\begin{align*}
\sum_{\CM' \in O(\CM)} Z(\CM')  & = \sum_{\CM' \in O(\CM)}\#Z(H(\uplus_{x\in T(\CM')} t_{x})|_{ \CM'})\\
	&= \sum_{\CM' \in O(\CM)} \sum_{\CM'' \in O(\CM)} \frac{1}{|O(\CM)|}\#Z(H(\uplus_{x\in T(\CM')} t_{x})|_{ \CM''})\\
	&=\sum_{\CM''\in O(\CM)} \#Z(H(t_{s})|_{ \CM''}). \\
\end{align*}
We can perform the same computation for all orbits of the action on the set $\Conn(\oPM_{\Lambda})$ to prove the $\eps=1$ case.   The case where $\eps = 2$ is the same.
\end{proof}

\noindent Using Lemma~\ref{th0 technical annoyance}, we now give a closed form expression for the contribution of $\mathcal{G}_{WC}$ for $H(t_{\bullet})$. 

\begin{lemma}\label{lem:WC_cont_for_sections}
Let $P\in\mathcal{P}_h(I,\vecd)$. Take $\Gamma := \Gamma_P$ as in \eqref{def: GammaP}. Let $t_\bullet$ be one of the multisections of $\CL_1 \to \oPM_\Gamma$ defined in Notation~\ref{not: descendent multisections pointing to special points} or $t_{z_2}$ for an internal marking $z_2$. Take the corresponding homotopy $H^\Gamma(t_\bullet)$ as defined in~\eqref{trr homotopy for the bullets}.  Then
\begin{equation}\begin{aligned}
Z^\Gamma_{WC}(\ttts_\bullet)&= \sum_{\substack{ Q \in \mathcal{Q}(I, \vecd+\eee_1)  \\ Q=  \{(I_0,k_1(0),k_2(0)),\ldots, (I_h,k_1(h),k_2(h))\} \\  \text{where $1\in I_0$ and} \\ P= \widehat{Q}_j \text{ for some $j$}}} (-1)^{k_2(0)}\prod_{i=1}^h  \left\langle\prod_{\ell\in I_i}\tau^{(a_\ell,b_\ell)}_{d_\ell}\sigma_1^{k_1(i)}\sigma_2^{k_2(i)}\sigma_{12}\right\rangle^{(\mathbf{s}^{\Gamma_{0,k_1(i),k_2(i),1,I_i}}),o} \\
& \cdot \Big( (k_2(0) + 1)k_1(j)\alpha^{\Gamma_{0,k_1(0)+1, k_2(0)+1,0,I_0}}_{\bullet, 1}- (k_1(0) + 1)k_2(j)\alpha^{\Gamma_{0,k_1(0)+1, k_2(0)+1,0,I_0}}_{\bullet, 2} \Big),
\end{aligned} \end{equation}
where $\alpha^{\Gamma_{0,k_1(0)+1, k_2(0)+1,0,I_0}}_{\bullet, \eps}$
is defined as follows depending on $\eps$ and $\bullet\in \{r, s, \Root, z_2\}$. Taking $\Xi := \Gamma_{0,k_1(0)+1, k_2(0)+1,0,I_0}$, we have:

\begin{equation}\begin{aligned}
\alpha^{\Xi}_{r,1} &=\frac{k_1(0) + 1}{{k_1(0) + k_1(j)}}\#Z(H^{\Xi}(\ttt^{\Xi}_r))+\frac{ k_1(j) -1}{{k_1(0) + k_1(j)}}\#Z(H^{\Xi}(\ttt^{\Xi}_{s})) \\
\alpha^{\Xi}_{s,1} &=\#Z(H^{\Xi}(\ttt^{\Xi}_{s})) \\
\alpha^{\Xi}_{\Root,1} &=\#Z(H^{\Xi}(\ttt^{\Xi}_{s})) \\
\alpha^{\Xi}_{z_2,1} &=\begin{cases} \#Z(H^{\Xi}(\ttt^{\Xi}_{z_2})) & \text{ if $2 \in I_0$} \\
 \#Z(H^{\Xi}(\ttt^{\Xi}_{s})) & \text{ if $2 \in I_j$}\end{cases} \\
\alpha^{\Xi}_{r,2} &=\#Z(H^{\Xi}(\ttt^{\Xi}_{r})) \\
\alpha^{\Xi}_{s,2} &=\frac{ k_2(j) -1}{{k_2(0)+ k_2(j)}}\#Z(H^{\Xi}(\ttt^{\Xi}_r)) +\frac{ k_2(0) + 1}{{k_2(0)+k_2(j)}}\#Z(H^{\Xi}(\ttt^{\Xi}_{s})) \\
\alpha^{\Xi}_{\Root,2} &=\#Z(H^{\Xi}(\ttt^{\Xi}_{r})) \\
\alpha^{\Xi}_{z_2,2} &=\begin{cases} \#Z(H^{\Xi}(\ttt^{\Xi}_{z_2})) & \text{ if $2 \in I_0$} \\
 \#Z(H^{\Xi}(\ttt^{\Xi}_{r})) & \text{ if $2 \in I_j$}\end{cases}
\end{aligned}\end{equation}
\end{lemma}

\begin{proof}[Proof of Lemma~\ref{lem:WC_cont_for_sections}]

The proof is a direct computation using the definition of $H^{\Gamma_P}(t_\bullet)$.
We start by recalling that by using \eqref{eqn:allWCContribsToP} we have that
\begin{equation}\begin{aligned}\label{first pass WC contrib}
Z^\Gamma_{WC}(\ttts_\XXX)&:=\sum_{\Lambda\in {\mathcal G}_{WC}(P)}\#Z(H(t_w)|_{[0,1]\times\oPM_\Lambda}) \\
	= &\sum_{\substack{ Q \in \mathcal{Q}(I, \vecd+\eee_1)  \\Q=  \{(I_j,k_1(j),k_2(j))\} \\ \text{where $1\in I_0$ and} \\ P= \widehat{Q}_j \text{ for some $j$}}} \left(\sum_{R_0=0}^{k_1(j) - 1} \sum_{\Xi:=\Gamma_{P, j, 1, I_0, k_1(0), k_2(0), R_0, S_0} \in A^1_{Q, j, R_0}} \#Z(H(t_w)|_{[0,1]\times\oPM_{\Xi}}) \right. \\ &\quad\qquad\qquad\qquad + \left. \sum_{S_0=1}^{k_2(j)-1} \sum_{\Xi:=\Gamma_{P, j, 2, I_0, k_1(0), k_2(0), R_0, S_0} \in A^2_{Q, j, S_0}}\#Z(H(t_w)|_{[0,1]\times\oPM_{\Xi}}) \right).
\end{aligned} \end{equation}
Recall that for any $\Gamma_{P, j, 1, I_0, k_1(0), k_2(0), R_0, S_0} \in A^1_{Q, j, R_0}$ or $\Gamma_{P, j, 2, I_0,  k_1(0), k_2(0),R_0, S_0} \in A^2_{Q, j, S_0}$, we have that
\[
\CB \Gamma_{P, j, \epsilon, I_0, k_1(0), k_2(0), R_0, S_0}\prec \Lambda_Q,
\]
where $\eps\in \{1,2\}$ and  $\Lambda_Q$ is the graph given in Notation \ref{nn:Q notation}. 

Let $\Lambda_\eps := \Gamma_{P, j, \eps, I_0,  k_1(0), k_2(0),R_0, S_0}$. %and $\Lambda_2 := \Gamma_{P, j, 2, I_0,  k_1(0), k_2(0),R_0, S_0}$. 
 While $\ess_1, \ess_2,$ and $\rho$ can all be pulled back from the base, $t_w$ may not be and hence the homotopy $H^{\Gamma_P}(t_\bullet)|_{\Lambda_i}$ given in~\eqref{trr homotopy for the bullets} may not be pulled back from the base in the summations in~\eqref{first pass WC contrib}. 
However, we can use Lemma~\ref{th0 technical annoyance} to see that the quantities in ~\eqref{first pass WC contrib} are equal to those that can be pulled back from the base as follows: 

\begin{equation}\begin{aligned}\label{th0 conclusion equation}
\#Z(H(t_r)|_{[0,1]\times\oPM_{\Lambda_1}}) &= \frac{k_1(0) + 1}{k_1(0)+k_1(j)} \#Z(H(\ttt^{v_0}_r)|_{[0,1]\times\oPM_{\Lambda_1}})\\ &\quad + \frac{k_1(j) -1}{k_1(0)+k_1(j)}  \#Z(H(\ttt^{v_0}_s)|_{[0,1]\times \oPM_{\Lambda_1}}); \\
\#Z(H(t_r)|_{[0,1]\times\oPM_{\Lambda_2}}) &=  \#Z(H(\ttt^{v_0}_r)|_{[0,1]\times\oPM_{\Lambda_2}}); \\
\#Z(H(t_s)|_{[0,1]\times\oPM_{\Lambda_1}}) &= \#Z(H(\ttt^{v_0}_s)|_{[0,1]\times\oPM_{\Lambda_1}}); \\
\#Z(H(t_s)|_{[0,1]\times\oPM_{\Lambda_2}}) &=\frac{ k_2(j) -1}{{k_2(0)+k_2(j)}} \#Z(H(\ttt^{v_0}_r)|_{[0,1]\times\oPM_{\Lambda_2}}) \\ & \quad +\frac{ k_2(0) + 1}{{k_2(0)+k_2(j)}}\#Z(H(\ttt^{v_0}_s)|_{[0,1]\times\oPM_{\Lambda_2}});\\
\#Z(H(t_\Root)|_{[0,1]\times\oPM_{\Lambda_1}}) &= \#Z(H(\ttt^{v_0}_s)|_{[0,1]\times \oPM_{\Lambda_1}}); \\
\#Z(H(t_\Root)|_{[0,1]\times\oPM_{\Lambda_2}}) &= \#Z(H(\ttt^{v_0}_r)|_{[0,1]\times\oPM_{\Lambda_2}}); \\
\#Z(H(t_{z_2})|_{[0,1]\times\oPM_{\Lambda_1}}) &= \begin{cases} \#Z(H^{\Xi}(\ttt^{\Xi}_{z_2})|_{[0,1]\times\oPM_{\Lambda_1}})& \text{ if $2 \in I_0$}, \\
\#Z(H(\ttt^{v_0}_s)|_{[0,1]\times\oPM_{\Lambda_1}})& \text{ if $2 \in I_j$};\end{cases} \\
\#Z(H(t_{z_2})|_{[0,1]\times\oPM_{\Lambda_2}}) &= \begin{cases} \#Z(H^{\Xi}(\ttt^{\Xi}_{z_2})|_{[0,1]\times\oPM_{\Lambda_2}}) & \text{ if $2 \in I_0$}, \\
  \#Z(H(\ttt^{v_0}_r)|_{[0,1]\times\oPM_{\Lambda_2}}) & \text{ if $2 \in I_j$}.\end{cases} \\
\end{aligned}\end{equation}

Thus we have reduced the computation of~\eqref{first pass WC contrib} to the homotopies $H(\ttt^{v_0}_r), H(\ttt^{v_0}_s),$ and $H(\ttt^{v_0}_{z_2})$, which all can be pulled back from the base.
Take $H$ to be any of these three homotopies. Recall that the homotopy $H$ is constant in time on the connected components which do not contain the internal tail $1$, that is, the graphs $\Lambda_{Q, \ell} : = \Gamma_{0, k_1(\ell), k_2(\ell),1,(a_i,b_i)_{i \in I_\ell}}$ for $1\le \ell \le h$. We have the decomposition
$$
[0,1]\times \oPM_{\Lambda_Q} = \prod_{i=1}^h \oPM_{\Lambda_{Q, i}} \times
([0,1]\times \oPM_{\Lambda_{Q, 0}}).
$$
In turn, we may decompose the homotopy $H=\left(\boxplus_{ i=1}^h H_i\right) \boxplus H^{v_0}$ with $H_i={\bf s}^{\Lambda_{Q, i}}$ and $$H^{v_0}(t_\bullet)(u,-)=F_{v_0}^*(\mathbf{s}_1^{ v_0}) \boxplus \left( (1-u)t^{v_0}_\bullet+u(1-u)\rho^{v_0}+u\mathbf{s}_2^{ v_0}\right).$$

By an analogous argument to that used to obtain Equation~\eqref{eqn for zeros in the homotopy} and using~\eqref{eqn for zeros in the homotopy}, we have that
\begin{align}\label{being able to use degree in GWC}
\begin{split}
\sum_{\Gamma_{P, j, 1, I_0,  k_1(0), k_2(0),R_0, S_0} \in A^1_{Q, j, R_0}} & \#Z(H|_{[0,1]\times\oPM_{\Gamma_{P, j, 1, I_0, k_1(0),k_2(0),R_0, S_0}}})\\
 = & (-1)^{k_2(0)}(k_2(0)+1) \cdot \#Z(H^{v_0}) \cdot \prod_{i=1}^h \#Z({\bf s}^{\Lambda_{Q, j}});\\
\sum_{\Gamma_{P, j, 2, I_0, k_1(0),k_2(0),R_0, S_0} \in A^2_{Q, j, S_0}}&\#Z(H|_{[0,1]\times\oPM_{\Gamma_{P, j, 2, I_0, k_1(0), k_2(0), R_0, S_0}}})  \\
=  &(-1)^{k_2(0)+1}(k_1(0)+1) \cdot \#Z(H^{v_0}) \cdot \prod_{i=1}^h \#Z({\bf s}^{\Lambda_{Q, j}}).
\end{split}\end{align}

Note that
\begin{equation}\label{Oj rest}
 \prod_{i=1}^h \#Z({\bf s}^{\Lambda_{Q, j}}) = \prod_{i=1}^h \left\langle\prod_{\ell\in I_i}\tau^{(a_\ell,b_\ell)}_{d_\ell}\sigma_1^{k_1(i)}\sigma_2^{k_2(i)}\sigma_{12}\right\rangle^{\mathbf{s}^{\Gamma_{0,k_1(i),k_2(i),1,I_i}},o}
\end{equation}
Moreover, the quantities in~\eqref{being able to use degree in GWC} do not depend on $R_0$ or $S_0$ respectively. Using Observation~\ref{obs:restriction_of_t_to_WC} and \eqref{th0 conclusion equation}, we then also have that
\begin{equation}\label{alpha deduction}
\#Z(H^{v_0}(\ttt^{v_0}_{\bullet})) =\begin{cases}  \alpha^{\Gamma_{0,k_1(0)+1, k_2(0)+1,0,I_0}}_{\bullet, 1} & \text{ if $\Lambda \in A^1_{Q, j, R_0}$} \\
\alpha^{\Gamma_{0,k_1(0)+1, k_2(0)+1,0,I_0}}_{\bullet, 2} & \text{ if $\Lambda \in A^2_{Q, j, S_0}$} \end{cases}
\end{equation}

By applying Equations~\eqref{being able to use degree in GWC},~\eqref{Oj rest}, and ~\eqref{alpha deduction} to simplify~\eqref{first pass WC contrib}, we obtain the claim.
\end{proof}

As we can see above, the expression for $\#Z_{WC}^\Gamma(\ttts_\XXX)$ is complicated for an arbitrary $\ttts_\XXX$.  We instead find specific families of non-trivial linear combinations of $\#Z_{WC}^\Gamma(\ttts_\XXX)$ for different graphs $\Gamma$ and points $\XXX$ which vanish.

\begin{nn}
Let $Q= \{(I_0,k_1(0)+1,k_2(0)+1),\ldots, (I_h,k_1(h),k_2(h))\} \in \mathcal{Q}^{\ne 0}_h (I, \vecd+\eee_1)$, defined in Notation \ref{nn:Q notation}.
There are $j+2$ different elements of $\mathcal{P}(I, \vecd+\eee_1)$ which we can naturally build from $Q$. First, recall from \eqref{defn: hatQj} that
$$
 \widehat{Q}_j :=  \{(I_1,k_1(1),k_2(1)),\ldots, (I_0\cup I_j,k_1(0)+k_1(j),k_2(0)+k_2(j)),\ldots, (I_h,k_1(h),k_2(h))\} \in \mathcal{P}_h(I,\vecd+\eee_1)\\
$$
for all $1 \le j \le h$. Second, we define the following:
\begin{equation}
\label{eq:hat Q r}
\begin{aligned}
 \widehat{Q}^{+r} &:= \{(I_1,k_1(1),k_2(1)),\ldots, (I_h,k_1(h),k_2(h)), (I_0, k_1(0)+r, k_2(0))\} \in \mathcal{P}_{h+1}(I,\vecd+\eee_1);\\
  \widehat{Q}^{+s} &:= \{(I_1,k_1(1),k_2(1)),\ldots, (I_h,k_1(h),k_2(h)), (I_0, k_1(0), k_2(0)+s)\} \in \mathcal{P}_{h+1}(I,\vecd+\eee_1).
  \end{aligned}\end{equation}
 Note that $ \widehat{Q}^{+r} = \widehat{Q^{+r}}_{h+1}$ and  $ \widehat{Q}^{+s} = \widehat{Q^{+s}}_{h+1}$, where $Q^{+r}$ and $Q^{+s}$ are as defined in \eqref{+r and +s definition}.
\end{nn}

\subsection{Special linear combinations of the distinguished multisections}

\begin{definition}\label{4tuplesQ}
Let $\{(a^P, b^P, c^P, d^P)\}_{P\in \mathcal{P}(I,\vecd+\eee_1)}$ be a family of 4-tuples of rational numbers so that the 4-tuples do not depend on the order of the elements of $P$ and $a^P + b^P + c^P + d^P = 1$.
\begin{enumerate}
\item The family of 4-tuples satisfies \emph{Property {\bf Q1}} if the family
further satisfies the following conditions:
\begin{enumerate}
\item For all $P$, $d^P = 0$.
\item For any $h\ge 0$ and  $Q= \{(I_0,k_1(0)+1,k_2(0)+1),\ldots, (I_h,k_1(h),k_2(h))\}  \in \mathcal{Q}^{\ne 0}_h (I, \vecd)$ with $1 \in I_0$, we have that
\begin{equation}\label{property Q1}
\frac{1+\sum_{i=0}^hk_1(i)}{k_1(0)+r}a^{\widehat{Q}^{+r}} +\frac{1+\sum_{i=0}^hk_2(i)}{k_2(0)+s}b^{\widehat{Q}^{+s}} = 1 + \sum_{j=1}^h\left(\frac{k_1(j)}{k_1(0)+k_1(j)}a^{\widehat{Q}_j} +\frac{k_2(j)}{k_2(0)+k_2(j)}b^{\widehat{Q}_j}\right).
\end{equation}
\end{enumerate}
\item The family of 4-tuples satisfies \emph{Property {\bf Q2}} if the family
further satisfies the following conditions:
\begin{enumerate}
\item If $P$ satisfies $\{1, 2\} \subseteq I_j$ for some $j$ then
$d^P=1$ and $a^P=b^P=c^P=0$. Otherwise $d^P=0$.
\item For any $h\ge 0$ and  $Q = \{(I_0,k_1(0)+1,k_2(0)+1),\ldots, (I_h,k_1(h),k_2(h))\}  \in \mathcal{Q}^{\ne 0}_h (I, \vecd)$ where $1 \in I_0$ and $2 \notin I_0$, we have that
\begin{equation}\label{property Q2}
\frac{1+\sum_{i=0}^hk_1(i)}{k_1(0)+r}a^{\widehat{Q}^{+r}} +\frac{1+\sum_{i=0}^hk_2(i)}{k_2(0)+s}b^{\widehat{Q}^{+s}} = 1 + \sum_{\substack{j=1 \\ 2 \notin I_j}}^h\left(\frac{k_1(j)}{k_1(0)+k_1(j)}a^{\widehat{Q}_j} +\frac{k_2(j)}{k_2(0)+k_2(j)}b^{\widehat{Q}_j}\right).
\end{equation}
\end{enumerate}
\end{enumerate}
\end{definition}

\begin{rmk}
Note that, in the case $h=0$ in Definition~\ref{4tuplesQ} and in the proofs below, we remark that we will often make use of the summation from $j=1$ to $h$ being empty and equal to zero.
\end{rmk}

\begin{lemma}\label{existence of 4tuples}
 There exists two families of 4-tuples $\{(a^P, b^P, c^P, d^P)\}_{P\in \mathcal{P}(I,\vecd+\eee_1)}$ of rational numbers, one of which satisfies Property {\bf Q1} and another which satisfies Property {\bf Q2}.
\end{lemma}

\begin{proof}
We first prove the case of Property {\bf Q1} inductively on $h$.
The inductive hypothesis is: there are choices of $a^P,b^P$ for all
$P\in \mathcal{P}_{\le h+1}(I, \vecd)$, with $a^P,b^P$ not depending
on the order of elements $P$, such that \eqref{property Q1} holds
for all $Q \in {\mathcal Q}^{\not=\emptyset}_{\le h}(I, \vecd)$.

 The case where $h=-1$ is vacuously true. Now suppose the inductive hypothesis is true for $h$. In doing so, we have determined the constants $a^P, b^P$ for all $P \in \mathcal{P}_{\le h+1}(I, \vecd)$. Given $Q \in \mathcal{Q}^{\not=\emptyset}_{h+1}(I, \vecd)$, the right-hand side of its corresponding equation~\eqref{property Q1} is then determined as $\widehat Q_j \in \mathcal{P}_{h+1}(I, \vecd)$ for all $j$. Now note that, given $P \in \mathcal{P}_{h+2}(I, \vecd)$, there is at most one $Q \in \mathcal{Q}^{\not=\emptyset}_{h+1}(I, \vecd)$ so that $a^P$ is in the equation ~\eqref{property Q1}, namely $P = \widehat{Q}^{+r}$. Similarly, $b^P$  appears in~\eqref{property Q1} for some $Q \in \mathcal{Q}^{\not=\emptyset}_{h+1}(I, \vecd)$ only if $P = \widehat{Q}^{+s}$. Thus each variable $a^P, b^P$ appears in at most one equation in the system
\eqref{property Q1} (where we run over all $Q$).
This determines the values of $a^P, b^P$ for those elements of
$\mathcal{P}_{h+2}(I, \vecd)$ of the form $\widehat Q^{+r}$ or
$\widehat Q^{+s}$. We remark here that if we permute the order of the tuples in $Q$, the summation on the right-hand side of \eqref{property Q1} will remain invariant by the induction hypothesis, hence the values of $a^P$ and $b^P$ for $P =\widehat Q^{+r}, \widehat Q^{+s}$ can be chosen to not depend on the order of its tuples. We may then choose arbitrary values for all other $a^P, b^P$
for $P\in \mathcal{P}_{h+2}(I, \vecd)$ subject to the constraint of
being independent of ordering. Further, we define $c^P=1-a^P-b^P$.
Thus there exists such a family that has Property {\bf Q1}.

An analogous inductive argument using \eqref{property Q2} yields a family
satisfying Property {\bf Q2}.
\end{proof}

We now have the following lemma:

\begin{lemma}\label{lem:cancelation_eqs}
Let $\{(a^P, b^P, c^P, d^P)\}_{P\in \mathcal{P}(I,\vecd)}$ be a family of 4-tuples of rational numbers that satisfies either property {\bf Q1} or {\bf Q2}. Then
\begin{align}\label{eq:cancellation_wc_trr}
\sum_{h=1}^{|I|}\frac{1}{h!}&\sum_{\substack{P\in\mathcal{P}_h(I,\vecd),\\
P=\{(I_1,k_1(1),k_2(1)),\ldots,(I_h,k_1(h),k_2(h))\}}}\frac{\Gamma(\frac{1+\sum_{i=1}^hk_1(i)}{r}) \Gamma(\frac{1+\sum_{i=1}^h k_2(i)}{s})}{\Gamma(\frac{1+r(I)}{r}) \Gamma(\frac{1+s(I)}{s})}\cdot\\
\notag&\quad\quad\cdot
\left(a^PZ_{WC}^{\Gamma_P}(\ttts_r)+b^PZ_{WC}^{\Gamma_P}(\ttts_s)+c^PZ_{WC}^{\Gamma_P}(\ttts_\Root)+d^PZ^{\Gamma_P}_{WC}(\ttts_{z_2})\right)=0.
\end{align}
\end{lemma}

\begin{proof}
First, we recall the collection $\Critrit(P)$
of \eqref{eq:crit def}.
Second, we recall the collection $\overline{\mathcal{Q}}(I,\vecd)$
defined in Part (E) of the proof of Theorem~\ref{thm:A_mod_invs} of ordered sets
\[Q=\{(I_0,k_1(0)+1,k_2(0)+1),\ldots,(I_h,k_1(h),k_2(h))\}\in
\mathcal{Q}_h(I,\vecd),  \ h\geq0\]
  so that
\begin{itemize}
\item $I_j \ne \varnothing $ for all $0 \le j \le h$;
\item the elements $(I_1,k_1(1),k_2(1)),\ldots,(I_h,k_1(h),k_2(h))$ are ordered in a way so that, for any two indices $i$ and $j$, we have that $\text{min}(I_i) < \text{min}(I_j)$ if and only if $i<j$.
\end{itemize}
As seen in~\eqref{all the contributions in the TRR homotopy}, the quantity  $Z^{\Gamma_P}_{WC}(H^{\Gamma_P}(t_\bullet))$ only sums over $ \mathcal{G}_{WC}(P)$, which are those graphs $\Lambda \in  \Critrit(P)$ where $1 \in I_0$. If $Q\in \overline{\mathcal Q}(I,\vecd)\cap {\mathcal Q}_0(I,\vecd)$,
then $\Lambda_Q$ is not the base of any critical boundary graph.
If $h>0$ then $Q \in \mathcal{Q}^{\neq\emptyset}(I,\vecd)$ so $\widehat{Q}_j$ is defined for all $1\le j\le h$.

Given $Q\in\overline{\mathcal{Q}}(I,\vecd)
,\bullet\in\{\Root,r,s,z_2\}$, we will write
\[
O_i(Q):=
\left\langle\prod_{\ell\in I_i}\tau^{(a_\ell,b_\ell)}_{d_\ell}\sigma_1^{k_1(i)}\sigma_2^{k_2(i)}\sigma_{12}\right\rangle^{\mathbf{s}^{\Gamma_{0,k_1(i),k_2(i),1,I_i}},o}
\]
and
\[X_{Q,\bullet}:=
\#Z(H^{\Gamma_{0,k_1(0)+1,k_2(0)+1,0,I_0}}(\ttt^{\Gamma_{0,k_1(0)+1,k_2(0)+1,0,I_0}}_{\bullet}))\prod_{i=1}^h O_i(Q).
\]
We note that the quantity $X_{Q, \bullet}$ is invariant with respect to permuting the last $h$ triplets of $Q$.

For any $Q\in \overline{\mathcal{Q}}_h(I,\vecd)$, we
recall the definition of $Q^{+r}$ and $Q^{+s}$ from \eqref{+r and +s definition}.
We define $\overline{\mathcal{Q}}^r_{h+1}(I,\vecd)$ and $\overline{\mathcal{Q}}^s_{h+1}(I,\vecd)$ to be the collections of tuples that can be written in the form $Q^{+r}$ and $Q^{+s}$
for $Q\in\overline{\mathcal{Q}}_h(I,\vecd)$. Take $$\widetilde{\mathcal{Q}}_h(I,\vecd) : = \overline{\mathcal{Q}}_h(I,\vecd) \cup \overline{\mathcal{Q}}^r_h(I,\vecd) \cup \overline{\mathcal{Q}}^s_h(I,\vecd).$$
Define $\widetilde{\mathcal{Q}}(I,\vecd) : = \bigcup_{h} \widetilde{\mathcal{Q}}_h(I,\vecd)$.
As in ~\eqref{comparing with the simple deletion}, we have that
\begin{equation}\label{eqn:rel Q and Q pluses TRR}
X_{Q, \bullet} = -X_{Q^{+r}, \bullet} = - X_{Q^{+s}, \bullet}.
\end{equation}

We consider the quantities $X_{Q,\bullet}$ for all $Q\in\overline{\mathcal{Q}}(I,\vecd)$ and $\bullet\in\{\Root,r,s,z_2\}$ as independent linear variables. By Lemma~\ref{lem:WC_cont_for_sections} and Equation~\eqref{eqn:rel Q and Q pluses TRR}, we can write any of $Z^{\Gamma}_{WC}(H^\Gamma(t_r))$, $Z^{\Gamma}_{WC}(H^\Gamma(t_s))$, $Z^{\Gamma}_{WC}(H^\Gamma(t_\Root))$, or $Z^{\Gamma}_{WC}(H^\Gamma(t_{z_2}))$, as a linear combination of the variables $X_{Q, \bullet}$. Thus, we can write the left hand side of Equation~\eqref{eq:cancellation_wc_trr} as a linear combination of the variables $X_{Q, \bullet}$.

We now compute the coefficient of $X_{Q, \bullet}$ in Equation~\eqref{eq:cancellation_wc_trr} for all  $Q\in\overline{\mathcal{Q}}(I,\vecd)$ and $\bullet\in\{\Root,r,s,z_2\}$.
Using Lemma~\ref{lem:WC_cont_for_sections}, we can see that the left hand side of~\eqref{eq:cancellation_wc_trr} can be expanded to be:
\begin{equation}\begin{aligned}
\label{simplified WC TRR first}
\sum_{h=1}^{|I|}&\sum_{P\in\mathcal{P}_h(I,\vecd+\eee_1)} \sum_{\substack{ Q \in \mathcal{Q}_h(I, \vecd+\eee_1)  \\ Q=  \{(I_0,k_1(0),k_2(0)),\ldots, (I_h,k_1(h),k_2(h))\} \\  \text{where $1\in I_0$ and} \\ P= \widehat{Q}_j \text{ for some $j$}}}
 \frac{1}{h!} \frac{\Gamma(\frac{1+\sum_{i=0}^hk_1(i)}{r})\Gamma(\frac{1+\sum_{i=0}^h k_2(i)}{s})}{\Gamma(\frac{1+r(I)}{r}) \Gamma(\frac{1+s(I)}{s})}\cdot
\prod_{i=1}^h O_i(Q) \cdot \\
&(-1)^{k_2(0)} \left( a^P    \Big( (k_2(0) + 1)k_1(j)\alpha^{\Gamma_{0,k_1(0)+1, k_2(0)+1,0,I_0}}_{r, 1}- (k_1(0) + 1)k_2(j)\alpha^{\Gamma_{0,k_1(0)+1, k_2(0)+1,0,I_0}}_{r, 2} \Big)\right. \\
& +  b^P  \Big( (k_2(0) + 1)k_1(j)\alpha^{\Gamma_{0,k_1(0)+1, k_2(0)+1,0,I_0}}_{s, 1}- (k_1(0) + 1)k_2(j)\alpha^{\Gamma_{0,k_1(0)+1, k_2(0)+1,0,I_0}}_{s, 2} \Big)  \\
& +  c^P   \Big( (k_2(0) + 1)k_1(j)\alpha^{\Gamma_{0,k_1(0)+1, k_2(0)+1,0,I_0}}_{\Root, 1}- (k_1(0) + 1)k_2(j)\alpha^{\Gamma_{0,k_1(0)+1, k_2(0)+1,0,I_0}}_{\Root, 2} \Big)  \\
& + d^P \left. \Big( (k_2(0) + 1)k_1(j)\alpha^{\Gamma_{0,k_1(0)+1, k_2(0)+1,0,I_0}}_{z_2, 1}- (k_1(0) + 1)k_2(j)\alpha^{\Gamma_{0,k_1(0)+1, k_2(0)+1,0,I_0}}_{z_2, 2} \Big) \right).
\end{aligned}\end{equation}

We can change the summations over $\mathcal{P}_h(I,\vecd+\eee_1)$ and special subsets of $\mathcal{Q}_h(I, \vecd+\eee_1)$ to a summation over $\widetilde{Q}_h(I, \vecd+\eee_1)$ using the partition given in Equation~\eqref{partition of critrit} and the fact that there are $h!$ elements of $\mathcal{Q}_h (I, \vecd)$ represented by a given element in $\widetilde{\mathcal{Q}}_h(I, \vecd + \eee_1)$. We may then rewrite the summation
given in~\eqref{simplified WC TRR first} as:
\begin{equation}\begin{aligned}\label{simplified WC TRR}
&\sum_{\substack{ Q \in \widetilde{\mathcal{Q}}(I, \vecd+\eee_1)  \\ Q=  \{(I_0,k_1(0),k_2(0)),\ldots, (I_h,k_1(h),k_2(h))\} \\  \text{where $1\in I_0$ } }} WC(Q)
\end{aligned}\end{equation}
where $WC(Q)$ is defined as follows.
\begin{itemize}
\item
If $Q \in \overline{Q}_h(I, \vecd+\eee_1)$, then
\begin{equation}\begin{aligned}
WC(Q) := &(-1)^{k_2(0)} \sum_{j=1}^h
 \frac{\Gamma(\frac{1+\sum_{i=0}^hk_1(i)}{r})\Gamma(\frac{1+\sum_{i=0}^h k_2(i)}{s})}{\Gamma(\frac{1+r(I)}{r}) \Gamma(\frac{1+s(I)}{s})}\cdot
\prod_{i=1}^h  O_i(Q) \cdot\\
& \left(  (k_2(0) + 1)k_1(j)\left(a^{\widehat{Q}_j}  \alpha^{\Gamma_{0,k_1(0)+1, k_2(0)+1,0,I_0}}_{r, 1} + b^{\widehat{Q}_j}  \alpha^{\Gamma_{0,k_1(0)+1, k_2(0)+1,0,I_0}}_{s, 1}+ \right.\right.
\\ &\qquad \left. \left. c^{\widehat{Q}_j}  \alpha^{\Gamma_{0,k_1(0)+1, k_2(0)+1,0,I_0}}_{\Root, 1}+ d^{\widehat{Q}_j}  \alpha^{\Gamma_{0,k_1(0)+1, k_2(0)+1,0,I_0}}_{z_2, 1}\right)\right. \\
&-   (k_1(0) + 1)k_2(j)\left(a^{\widehat{Q}_j}  \alpha^{\Gamma_{0,k_1(0)+1, k_2(0)+1,0,I_0}}_{r, 2} + b^{\widehat{Q}_j}  \alpha^{\Gamma_{0,k_1(0)+1, k_2(0)+1,0,I_0}}_{s, 2}+ \right.
\\ &\qquad\left. \left. c^{\widehat{Q}_j}  \alpha^{\Gamma_{0,k_1(0)+1, k_2(0)+1,0,I_0}}_{\Root, 2}+ d^{\widehat{Q}_j}  \alpha^{\Gamma_{0,k_1(0)+1, k_2(0)+1,0,I_0}}_{z_2, 2}\right) \right).
\end{aligned} \end{equation}
\item
If $Q^{+r} \in \overline{Q}_h^{r}(I, \vecd+\eee_1)$, then we can use the facts that $k_1(h) = r$, $k_2(h) = 0$, $O_h(Q^{+r})=\langle \sigma_1^r\sigma_{12} \rangle = -1$ and see that
\begin{equation}\begin{aligned}
WC(Q^{+r}) := &- (-1)^{k_2(0)}
 \frac{\Gamma(\frac{1+\sum_{i=0}^hk_1(i)}{r})\Gamma(\frac{1+\sum_{i=0}^h k_2(i)}{s})}{\Gamma(\frac{1+r(I)}{r}) \Gamma(\frac{1+s(I)}{s})}\cdot \prod_{i=1}^{h-1}
O_i(Q) \cdot \\
& (k_2(0) + 1)r\left( a^{\widehat Q^{+r}}   \alpha^{\Gamma_{0,k_1(0)+1, k_2(0)+1,0,I_0}}_{r, 1}\right. +  b^{\widehat Q^{+r}}  \alpha^{\Gamma_{0,k_1(0)+1, k_2(0)+1,0,I_0}}_{s, 1}  \\
& \qquad\qquad\qquad +  c^{\widehat Q^{+r}}    \alpha^{\Gamma_{0,k_1(0)+1, k_2(0)+1,0,I_0}}_{\Root, 1}   + d^{\widehat Q^{+r}} \left. \alpha^{\Gamma_{0,k_1(0)+1, k_2(0)+1,0,I_0}}_{z_2, 1} \right).
\end{aligned} \end{equation}
\item
If $Q^{+s} \in \overline{Q}_h^{s}(I, \vecd+\eee_1)$, then we can use the facts that $k_1(h) = 0$, $k_2(h) = s$, $O_h(Q^{+s})=\langle \sigma_2^s\sigma_{12} \rangle = -1$ and see that
\begin{equation}\begin{aligned}
WC(Q^{+s}) := &
(-1)^{k_2(0)} \frac{\Gamma(\frac{1+\sum_{i=0}^hk_1(i)}{r})\Gamma(\frac{1+\sum_{i=0}^h k_2(i)}{s})}{\Gamma(\frac{1+r(I)}{r}) \Gamma(\frac{1+s(I)}{s})}\cdot \prod_{i=1}^{h-1} O_i(Q)
\cdot \\
&(k_1(0) + 1)s \left( a^{\widehat Q^{+s}}  \alpha^{\Gamma_{0,k_1(0)+1, k_2(0)+1,0,I_0}}_{r, 2} \right. +  b^{\widehat Q^{+s}}   \alpha^{\Gamma_{0,k_1(0)+1, k_2(0)+1,0,I_0}}_{s, 2}   \\
& \qquad \qquad \qquad +  c^{\widehat Q^{+s}}   \alpha^{\Gamma_{0,k_1(0)+1, k_2(0)+1,0,I_0}}_{\Root, 2} + d^{\widehat Q^{+s}} \left.\alpha^{\Gamma_{0,k_1(0)+1, k_2(0)+1,0,I_0}}_{z_2, 2}  \right).
\end{aligned} \end{equation}
\end{itemize}

Using the definition of $X_{Q, \bullet}$ and $\alpha^{\Gamma_{0,k_1(0)+1, k_2(0)+1,0,I_0}}_{\bullet, \eps}$ given in Lemma~\ref{lem:WC_cont_for_sections}, we can see that
\begin{equation}\begin{aligned}
WC(Q) =&(-1)^{k_2(0)}\sum_{j=1}^h \frac{\Gamma(\frac{1+\sum_{i=0}^hk_1(i)}{r})\Gamma(\frac{1+\sum_{i=0}^h k_2(i)}{s})}{\Gamma(\frac{1+r(I)}{r}) \Gamma(\frac{1+s(I)}{s})}\cdot \\ &  \left[ (k_2(0)+1) k_1(j)\cdot \left[ a^{\widehat Q_j} \left( \frac{k_1(0)+1}{k_1(0) + k_1(j)} X_{Q, r} + \frac{k_1(j) - 1}{k_1(0) + k_1(j)} X_{Q,s}\right) + b^{\widehat Q_j} X_{Q, s} +\right.\right. \\
&\qquad \qquad \qquad \left. c^{\widehat{Q}_j} X_{Q, s} + \begin{cases} d^{\widehat{Q}_j} X_{Q, z_2} & \text{ if $2 \in I_0$} \\ d^{\widehat{Q}_j} X_{ Q, s} & \text{ if $2 \in I_j$} \\ 0 & \text{ otherwise} \end{cases} \right] \\
&- (k_1(0)+1) k_2(j)\cdot  \left[ a^{\widehat Q_j} X_{Q, r} + b^{\widehat Q_j} \left(\frac{k_2(j) -1}{k_2(0) + k_2(j)} X_{Q, r} + \frac{k_2(0) +1}{k_2(0) + k_2(j)} X_{Q, s}\right) \right. \\
&\qquad \qquad \qquad\left.\left. + c^{\widehat{Q}_j} X_{Q, r} + \begin{cases} d^{\widehat{Q}_j} X_{Q, z_2} & \text{ if $2 \in I_0$} \\ d^{\widehat{Q}_j} X_{ Q, r} & \text{ if $2 \in I_j$} \\ 0 & \text{ otherwise}\end{cases} \right]\right].
 \end{aligned}\end{equation}
 Moreover, we have:
\begin{equation}\begin{aligned}
WC&(Q^{+r}) = -(-1)^{k_2(0)} \frac{\Gamma(\frac{1+r+\sum_{i=0}^hk_1(i)}{r})\Gamma(\frac{1+\sum_{i=0}^h k_2(i)}{s})}{\Gamma(\frac{1+r(I)}{r}) \Gamma(\frac{1+s(I)}{s})}\cdot  (k_2(0)+1) r \cdot \\& \left[ a^{\widehat Q^{+r}} \left( \frac{k_1(0)+1}{k_1(0) + r} X_{Q, r} + \frac{r - 1}{k_1(0) +r } X_{Q,s}\right) + b^{\widehat Q^{+r}} X_{Q, s} + c^{\widehat Q^{+r}} X_{Q, s} + \begin{cases} d^{\widehat Q^{+r}} X_{Q, z_2} & \text{ if $2 \in I_0$} \\ 0 & \text{ otherwise} \end{cases} \right].
\end{aligned}\end{equation}
Lastly, we similarly have
\begin{equation}\begin{aligned}
WC&(Q^{+s}) = (-1)^{k_2(0)} \frac{\Gamma(\frac{1+\sum_{i=0}^hk_1(i)}{r})\Gamma(\frac{1+s+ \sum_{i=0}^h k_2(i)}{s})}{\Gamma(\frac{1+r(I)}{r}) \Gamma(\frac{1+s(I)}{s})}\cdot  (k_1(0)+1) s \cdot \\ & \left[ a^{\widehat Q^{+s}} X_{Q, r} + b^{\widehat Q^{+s}} \left(\frac{s -1}{k_2(0)+s} X_{Q, r} + \frac{k_2(0) +1}{ k_2(0)+s} X_{Q, s}\right) + c^{\widehat Q^{+s}} X_{Q, r} + \begin{cases} d^{\widehat Q^{+s}} X_{Q, z_2} & \text{ if $2 \in I_0$} \\ 0 & \text{ otherwise} \end{cases} \right].
\end{aligned}\end{equation}

We can then see that the coefficient of
$X_{Q,r}$ in \eqref{eq:cancellation_wc_trr} is proportional to,\footnote{After dividing by $(-1)^{k_2(0)}\frac{\Gamma\left(\frac{1+\sum_{i=0}^hk_1(i)}{r}\right) \Gamma\left(\frac{1+\sum_{i=0}^h k_2(i)}{s}\right)}{\Gamma\left(\frac{1+r(I)}{r}\right) \Gamma\left(\frac{1+s(I)}{s}\right)}$. Note that the numerators for the $\Gamma$ functions in the $WC(Q^{+r})$ and $WC(Q^{+s})$ cases are different and have an extra factor.}
\begin{align}\label{eq:Q_coef_WC_r}
\sum_{j=1}^h &\left(
\left((k_2(0)+1)k_1(j) \frac{k_1(0)+1}{k_1(0)+k_1(j)} -(k_1(0)+1)k_2(j) \right)a^{\widehat{Q}_j}
-(k_1(0)+1)k_2(j) \frac{k_2(j)-1}{k_2(0)+k_2(j)}b^{\widehat{Q}_j}\right.\\\notag&\left.
-(k_1(0)+1)k_2(j)c^{\widehat{Q}_j}
-\mathbf{1}_{2\in I_j}(k_1(0)+1)k_2(j)d^{\widehat{Q}_j}
\right)
\\\notag&-(k_2(0)+1)\frac{k_1(0)+1}{k_1(0)+r}(1+\sum_{i=0}^hk_1(i)) a^{\widehat{Q}^{+r}}
\\\notag&+(k_1(0)+1)(1+\sum_{i=0}^hk_2(i)) \left(a^{\widehat{Q}^{+s}}
+\frac{s-1}{k_2(0)+s}b^{\widehat{Q}^{+s}}
+c^{\widehat{Q}^{+s}}\right).
\end{align}

Similarly, the coefficient of $X_{Q,s}$ is proportional to
\begin{align}\label{eq:Q_coef_WC_s}
\sum_{j=1}^h &\left(
(k_2(0)+1)k_1(j) \frac{k_1(j)-1}{k_1(0)+k_1(j)}a^{\widehat{Q}_j}
+\left( (k_2(0)+1)k_1(j) -(k_1(0)+1)k_2(j)\frac{k_2(0)+1}{k_2(0)+k_2(j)} \right)b^{\widehat{Q}_j}\right.\\\notag&\left.
+(k_2(0)+1)k_1(j)c^{\widehat{Q}_j}
+\mathbf{1}_{2\in I_j}(k_2(0)+1)k_1(j)d^{\widehat{Q}_j}
\right)\\
\notag&-
 (k_2(0)+1)(1+\sum_{i=0}^hk_1(i)) \left(\frac{r-1}{k_1(0)+r}a^{\widehat{Q}^{+r}}
+b^{\widehat{Q}^{+r}}
+c^{\widehat{Q}^{+r}}\right)\\
\notag&+(k_1(0)+1)\frac{k_2(0)+1}{k_2(0)+s}(1+\sum_{i=0}^hk_2(i)) b^{\widehat{Q}^{+s}}.
\end{align}

Note that the coefficient for $X_{Q, \times}$ is always $0$. The coefficient of $X_{Q,z_2}$ is $0$ if $2\notin I_0$ and otherwise it is proportional to
\begin{align}\label{eq:Q_coef_WC_d}
\sum_{j=1}^h &\Big((k_2(0)+1)k_1(j)-(k_1(0)+1)k_2(j)
\Big)d^{\widehat{Q}_j}\\
\notag&-(k_2(0)+1)(1+\sum_{i=0}^hk_1(i))d^{\widehat{Q}^{+r}}
+(k_1(0)+1)(1+\sum_{i=0}^hk_2(i))d^{\widehat{Q}^{+s}}.
\end{align}

Consider a family $\{(a^P, b^P, c^P, d^P)\}_{P\in \mathcal{P}(I,\vecd)}$ that satisfies Property {\bf Q1} (which we know exists by Lemma~\ref{existence of 4tuples}). Then $d^P = 0$ and $a^P+b^P+c^P=1$ for all $P$. Thus we have the constraint  $c^P = 1- a^P - b^P$ for all $P$, so we may substitute this into \eqref{eq:Q_coef_WC_r} above and simplify. After factoring out $(k_1(0)+1)(k_2(0)+1)$, we can compute that the vanishing of \eqref{eq:Q_coef_WC_r} is equivalent to the equation
\begin{equation}\label{eq:Q_coef_WC_r_simplified}
\frac{1+\sum_{i=0}^hk_1(i)}{k_1(0)+r}a^{\widehat{Q}^{+r}} +\frac{1+\sum_{i=0}^hk_2(i)}{k_2(0)+s}b^{\widehat{Q}^{+s}} = 1 + \sum_{j=1}^h\left(\frac{k_1(j)}{k_1(0)+k_1(j)}a^{\widehat{Q}_j} +\frac{k_2(j)}{k_2(0)+k_2(j)}b^{\widehat{Q}_j}\right).
\end{equation}
But this is exactly Equation~\eqref{property Q1} so \eqref{eq:Q_coef_WC_r} vanishes. One can additionally check that after the same substitution, the vanishing of ~\eqref{eq:Q_coef_WC_s} is also equivalent to satisfying the equation ~\eqref{eq:Q_coef_WC_r_simplified}. Lastly, note that the coefficient of $X_{Q, z_2}$ immediately vanishes when $d^P =0$ for all $P$. In summary, the coefficients of $X_{Q, \bullet}$ vanish for all $Q \in \overline{\mathcal Q}(I, \vecd)$ if $\{(a^P, b^P, c^P, d^P)\}_{P\in \mathcal{P}(I,\vecd)}$ satisfies Property {\bf Q1}, proving Item (1) of the Lemma.

Next, consider  a family $\{(a^P, b^P, c^P, d^P)\}_{P\in \mathcal{P}(I,\vecd)}$ that satisfies Property {\bf Q2}. Given $Q \in \overline{\mathcal{Q}}(I, \vecd)$, note that we must have that $1 \in I_0$ and $2 \in I_0 \cup I_1$ due to the convention of the ordering of the tuples.  If $2 \in I_0$, then $d^{\widehat{Q}_j} = 1$ and $a^{\widehat{Q}_j}=b^{\widehat{Q}_j}=c^{\widehat{Q}_j}=0$ for all $j$. In turn, we have that $d^{\widehat{Q}^{+r}} = d^{\widehat{Q}^{+s}} =1$ and $a^{\widehat{Q}^{+r}}=b^{\widehat{Q}^{+r}}=c^{\widehat{Q}^{+r}}=a^{\widehat{Q}^{+s}}=b^{\widehat{Q}^{+s}}=c^{\widehat{Q}^{+s}}=0$. Note that when $2 \in I_0$, there are no summands with $d^{\widehat Q_j}$ in Equations~\eqref{eq:Q_coef_WC_r} or~\eqref{eq:Q_coef_WC_s}, so the corresponding coefficients of $X_{Q, r}$ and $X_{Q, s}$ vanish. Also, one can compute that the quantity in~\eqref{eq:Q_coef_WC_d} vanishes when $d^{\widehat Q_j} = d^{\widehat{Q}^{+r}} = d^{\widehat{Q}^{+s}} =1$.

If $2 \in I_1$, then the coefficient of $X_{Q, z_2}$ vanishes, as stated above. In this case, we can see that $d^{\widehat{Q}_1} = 1$ and $a^{\widehat{Q}_1}=b^{\widehat{Q}_1}=c^{\widehat{Q}_1}=0$. On the other hand, $d^{\widehat{Q}_j} = 0$ for all $1< j \le h$ and $d^{\widehat{Q}^{+r}} = d^{\widehat{Q}^{+s}} =0$. We then can substitute using constraint that $c^P = 1- a^P - b^P$ for all $P \in \{\widehat{Q}_2, \ldots, \widehat{Q}_h, \widehat{Q}^{+r}, \widehat{Q}^{+s}\}$ to simplify \eqref{eq:Q_coef_WC_r} in terms of $a^P$ and $b^P$. After doing so and factoring out $(k_1(0)+1)(k_2(0)+1)$, we can compute that the vanishing of \eqref{eq:Q_coef_WC_r} is equivalent to the equation
\begin{equation}\label{eq:Q_coef_WC_r_simplified2}
\frac{1+\sum_{i=0}^hk_1(i)}{k_1(0)+r}a^{\widehat{Q}^{+r}} +\frac{1+\sum_{i=0}^hk_2(i)}{k_2(0)+s}b^{\widehat{Q}^{+s}} = 1 + \sum_{j=2}^h\left(\frac{k_1(j)}{k_1(0)+k_1(j)}a^{\widehat{Q}_j} +\frac{k_2(j)}{k_2(0)+k_2(j)}b^{\widehat{Q}_j}\right).
\end{equation}
This is exactly Equation~\eqref{property Q2}, so \eqref{eq:Q_coef_WC_r} vanishes. One can additionally check that after the same substitution, that Equation~\eqref{eq:Q_coef_WC_s} vanishes.
\end{proof}

We now turn our attention to $\Lambda\in \mathcal{G}_{\Cont}(P)$.
\begin{nn}
Let $\mathcal{P}_{Cont, h}(I,\vecd+\eee_1)$ be the collection of all tuples
$$\{(I_0,k_1(0),k_2(0)),\ldots,(I_h,k_1(h),k_2(h))\}$$
such that
\begin{itemize}
\item $1 \in I_0$
\item $I_j$ are disjoint,
\item  $\sqcup_{j=0}^h I_j = I$ with at most one empty set, and
\item the graphs $\Gamma_{0, k_1(j), k_2(j), 1, I_j}$ are balanced with respect to $\vecd$ for all $j$.
\end{itemize}
We write $\mathcal{P}_{Cont,h}^{\not=\emptyset}(I,\vecd+\eee_1)$
to be the set of those tuples in $\mathcal{P}_{Cont, h}(I,\vecd+\eee_1)$
with all $I_j\not=\emptyset$.

Consider a tuple $P^{\Cont} =\{(I_0,k_1(0),k_2(0)),\ldots,(I_h,k_1(h),k_2(h))\}\in \mathcal{P}_{Cont, h}(I,\vecd+\eee_1)$. If for all $i \neq  j$, $I_i$ is nonempty, then we can define the element $\widehat{P}^{\Cont}_j  \in \mathcal{P}_h(I, \vecd + \eee_1)$ given by
\[\widehat{P}_j^{\Cont}:=\{(I_1,k_1(1),k_2(1)),\ldots,(I_0\cup I_j,k_1(0)+k_1(j),k_2(0)+k_2(j))\ldots,(I_h,k_1(h),k_2(h))\},\]
for any $1 \le j \le h$.
In the case where $P^{\Cont} \in \mathcal{P}_{Cont, h}(I,\vecd+\eee_1)$ has the property that all $I_j$ are nonempty, we may define,
similar to \eqref{+r and +s definition},
two new elements in $\mathcal{P}_{Cont, h+1}(I,\vecd+\eee_1)$
\begin{equation}\begin{aligned}\label{PContr and PConts}
P^{Cont,+r}&=\{(I_0,k_1(0),k_2(0)),\ldots,(I_h,k_1(h),k_2(h)), (\emptyset, r, 0)\}; \\
P^{Cont,+s}&=\{(I_0,k_1(0),k_2(0)),\ldots,(I_h,k_1(h),k_2(h)), (\emptyset, 0, s)\}.
\end{aligned}\end{equation}
Similar to \eqref{eq:hat Q r}, we define the
elements in $\mathcal{P}_{h+1}(I,\vecd+\eee_1)$
$$
\widehat{P}^{Cont,+r} = (\widehat{P^{Cont,+r}})_{h+1}; \qquad \widehat{P}^{Cont,+s} = (\widehat{P^{Cont,+s}})_{h+1}.
$$

For $P^{\Cont} \in \mathcal{P}_{Cont, h}(I,\vecd+\eee_1)$, there is a corresponding graph
\begin{equation}\label{def: GammaPCont}
\Gamma_{P^{\Cont}}:=\bigsqcup_{j=0}^h \Gamma_{0,k_1(j),k_2(j),1, \{(a_i,b_i)\}_{i\in I_j}}.
\end{equation}
\end{nn}

Let $\Gamma = \Gamma_P$ with $P \in \mathcal{P}_h(I, \vecd+\eee_1)$ be a smooth graded balanced graph with an internal tail $1 \in I(\Gamma)$. Suppose $\Lambda \in \mathcal{G}_{\Cont}(P)$ as in Notation~\ref{Graphs for OTRR}. That is,
the graph $\Lambda$ has connected components $\Xi_1, \ldots, \Xi_h$ and there is some $j$ such that $\Xi_j$ consists of a graph with two open vertices $v_0$ and $v_j$ with one edge between them so that:
\begin{itemize}
\item The vertex $v_0$ is rooted and has the internal tail $1$.
\item  $\CB\Lambda$ is balanced with respect to the descendent vector $\vecd$.
\item The half-edges $h_0$ and $h_j$ corresponding to the edge $e$ satisfy the following: $$\tw(h_0)= (0,0), \alt(h_0) = (0,0), \tw(h_j) = (r-2,s-2),\text{ and }\alt(h_j)=(1,1).$$
\end{itemize}
Set $I_0 := I(v_0)$ and $I_j := I(v_j)$. Note that $\CB v_0 = \Gamma_{0, k_1(v_0), k_2(v_0), 1, \{(a_i, b_i)\}_{i \in I_0}}$ and $\CB v_j = \Gamma_{0, k_1(v_j), k_2(v_j), 1, \{(a_i, b_i)\}_{i \in I_j}}$.  Moreover, the graph $\Xi_j$ is a degeneration of the smooth vertex graph $\Gamma_{0, k_1(v_0) + k_1(v_j), k_2(v_0) + k_2(v_j), 1, \{(a_i, b_i)\}_{i \in I_0\cup I_j}}$, hence $\CB \Lambda$ corresponds to some tuple $P^{\Cont} \in \mathcal{P}_{Cont, h}(I,\vecd+\eee_1)$ and $P = \widehat{P}_j^{\Cont}$.

Note that there are precisely $(k_1(v_0) + 1)(k_2(v_0) + 1)$ graphs in $\mathcal{G}_{\Cont}(P)$ that have the base $\CB\Lambda$ and are a degeneration of $\Gamma$. Indeed, the graphs in $\mathcal{G}_{\Cont}(P)$ with such properties are completely determined by two numbers:
\begin{itemize}
\item $R_0$  the number of $r$-points in $v_0$ that follow its root and precede the half-edge with respect to the cyclic order and
\item $S_0$ the number of $s$-points in $v_0$ that follow its root and precede the half-edge with respect to the cyclic order.\footnote{We alert the reader that in this definition of
$R_0,S_0$, the role of $v_0$ and $v_j$ are the opposite of that used
for critical boundary graphs of $\mathcal{G}_{WC}(P)$. The reason for this is that we
prefer to keep the internal tail $1$ attached to $v_0$ for consistency.}
\end{itemize}
We will denote the corresponding boundary graph by $\Gamma_{P, j, I_0, k_1(0), k_2(0), R_0, S_0}^{\Cont}$, where $0 \le R_0 \le k_1(v_0)$ and $0 \le S_0 \le k_2(v_0)$.

The fiber $\Phi_p$ of the map
\begin{equation}\label{eq:F_p}\bigsqcup_{R_0, S_0}F_{\Gamma_{P, j, I_0,  k_1(0), k_2(0),R_0, S_0}^{\Cont}}:\bigsqcup_{R_0, S_0}\oPM_{\Gamma_{P, j, I_0,  k_1(0), k_2(0),R_0, S_0}^{\Cont}}\to\oPM_{\CB\Lambda}\end{equation}
over a point $p\in\CM_{\CB\Lambda}$ is the union of
$k_1(v_0) + k_2(v_0)+1$ closed intervals.
Indeed, the additional information of the location of the
forgotten non-twisted half-node requires a choice of a point between boundary markings.
The boundary $\partial \Phi_p$ corresponds to the stable disks obtained by the forgotten half-node approaching a boundary marking. We pair endpoints corresponding to stable disks where one is obtained by the node approaching the marking from the right and the other one is obtained by letting the node approach the same marking from the left. Denote by $\hat \Phi_p$ the quotient space of $\Phi_p$ obtained by identifying the boundary points which are paired.
Thus, $\hat \Phi_p$ is homeomorphic to $S^1$ and is canonically identified with the boundary of the connected component indexed by $v_0$
of the stable disk $\Sigma_p$ corresponding to $p\in \oPM_{\CB\Lambda}$.

 Observations 4.9 and 4.10 in \cite{PST14} imply the following observation.
\begin{obs}\label{obs:trivial_on_fiber}\label{obs:canonical_const}
Note that when we forget the half-edge on $v_0$, the disk remains stable hence does not contract. By Observation~\ref{isomorphism forgetting non-alt for descendents}(ii), the restriction of the tautological line $\CL_i|_{\hat{\Phi}_p}$ is canonically trivial as
this line bundle is pulled back from a point of $\oPM_{\CB v_0}$.
With respect to this trivialization a canonical multisection of $\CL_i|_{\Phi_p}$ as well as $\ttt_r^{v_0},\ttt_s^{v_0},\ttt_\Root^{v_0}$, and $\ttt_{z_2}^{v_0}$ (if $z_2\in I_0$) restricted to $\Phi_p$ are constant multivalued functions.
Indeed, in all these cases, these sections are pulled back from $\oPM_{\CB v_0}$.
\end{obs}
On the other hand, Observation 4.11 and Lemma 4.12 in \cite {PST14} imply that:
\begin{obs}\label{obs:cont_on_fiber}
The section $\ttt^{v_0}_{h_0}|_{\Phi_p}$ descends to a continuous function $\theta_p : \hat \Phi_p \to \C^\times$. The winding number of $\theta_p$ is $-1$.
\end{obs}

The next observation, analogous to Observation \ref{obs:restriction_of_t_to_WC}, follows again from direct inspection.

\begin{obs}\label{obs:restriction_of_t_to_Cont}
Let $P=\{(I_i,k_1(i),k_2(i))\}\in {\mathcal P}_h(I,\vecd+{\bf e}_1)$ with
$1\in I_j$ for some $j$.
Let $\Lambda=\Gamma_{P, j, I_0, k_1(0), k_2(0), R_0, S_0}^{\Cont}\in \mathcal{G}_{\Cont}(P)$
and $\Gamma=\Gamma_P$.
Consider the multisections of $\CL_1\to\oCM^W_{v_0}$ given by $\ttt_r$, $\ttt_s$, $\ttt_\Root$ and $\ttt_{z_2}$ from Notation~\ref{not: descendent multisections pointing to special points}.
 Then, by using the canonical identification of
Observation \ref{isomorphism forgetting non-alt for descendents},(i) of the pullback of $\CL_1\to\oCM^W_{v_0}$ to $\oCM^W_{\Lambda}$ with $(\CL_1\to\oCM^W_\Gamma)|_{\oCM^W_\Lambda}$, we have
\begin{enumerate}
\item\label{it:rest_WC_r}
$\ttt_r|_{\oPM_\Lambda} = k_1(v_0)\ttt^{v_0}_r \uplus k_1(v_j)\ttt^{v_0}_{h_0}$.
\item\label{it:rest_WC_s}
$\ttt_s|_{\oPM_\Lambda} = k_2(v_0)\ttt^{v_0}_s \uplus k_2(v_j)\ttt^{v_0}_{h_0}$.
\item\label{it:rest_WC_root}
$\ttt_\Root|_{\oPM_\Lambda}=\ttt^{v_0}_{\Root}$.
\item\label{it:rest_WC_z2}
If $2\in I_j$,
$\ttt_{z_2}|_{\oPM_\Lambda}=\begin{cases} \ttt^{v_0}_{z_2} & \text{if $2\in I_0$,} \\    \ttt^{v_0}_{h_0} & \text{if $2 \in I_j \setminus I_0$}.\end{cases}$
\end{enumerate}
\end{obs}

 We now compute the contributions of this case when we take the homotopy to a canonical multisection.
\begin{lemma}\label{lem:TRR_cont_for_sections}
Let $P \in\mathcal{P}_h(I,\vecd+\eee_1)$, and
write $\Gamma := \Gamma_P$ as in \eqref{def: GammaP}. Let $t_\bullet$ be one of the multisections of $\CL_1 \to \oPM_\Gamma$ defined in Notation~\ref{not: descendent multisections pointing to special points}. Take the corresponding homotopy $H^\Gamma(t_\bullet)$ as defined in~\eqref{trr homotopy for the bullets}.  Then \[Z^{\Gamma_P}_{\Cont}(\ttts_\bullet)=
\sum_{\substack{P^{\Cont}\in\mathcal{P}_{Cont,h}(I,\vecd+\eee_1),\\ P^{\Cont} =\{(I_i,k_1(i),k_2(i))\}_{i=0}^h\\  \text{where $\widehat{P}^{\Cont}_j=P$}\\\text{ for some $j\in [h]$ and $1 \in I_0$}}}\left(\prod_{i=0}^h \left\langle\prod_{j\in I_i}\tau_{d_j}^{(a_j,b_j)}\sigma_1^{k_1(i)}\sigma_2^{k_2(i)}\sigma_{12}\right\rangle^{\mathbf{s}^{\Gamma_{0,k_1(i),k_2(i),1,I_i}},o}\right)
\beta_{P^{\Cont}}(\ttts_{\bullet}),\]
where
\begin{equation}\label{eqn:ContMultisection}
\beta_{P^{\Cont}}(\ttts_{\bullet}):=
\begin{cases}
\frac{k_1(j)}{k_1(0)+k_1(j)} &\text{if~}\bullet=r,\\
\frac{k_2(j)}{k_2(0)+k_2(j)} &\text{if~}\bullet=s,\\
0 &\text{if }\bullet=\Root, \\
0 &\text{if }\bullet=z_2,~\text{and }2\in I_0,\\
1 &\text{if }\bullet = z_2,~\text{and }2\in I_j\setminus I_0.
\end{cases}
\end{equation}
\end{lemma}
\begin{proof}
Fix $\Gamma=\Gamma_P$. First recall the definition of $Z^\Gamma_{\Cont}(\ttts_\XXX)$ and the classification of graphs in $\mathcal{G}_{\Cont}(P)$. Note that, if $\Lambda= \Gamma_{P, j, I_0,  k_1(0), k_2(0),R_0, S_0}^{\Cont}  \in \mathcal{G}_{\Cont}(P)$, then $\CB\Lambda=\Gamma_{P^{\Cont}}$
as defined in \eqref{def: GammaPCont}
where $P^{\Cont} =\{(I_i,k_1(i),k_2(i))\}_{i=0}^h \in \mathcal{P}_{Cont,h}(I,\vecd+\eee_1)$. Then we have that

\begin{equation}\begin{aligned}
Z^\Gamma_{\Cont}(\ttts_\XXX)&:=\sum_{\Lambda\in {\mathcal G}_{\Cont}(P)}\#Z(H|_{[0,1]\times\oPM_\Lambda})\\
	&= \sum_{\substack{P^{\Cont}\in\mathcal{P}_{Cont,h}(I,\vecd+\eee_1),\\ P^{\Cont} =\{(I_i,k_1(i),k_2(i))\}_{i=0}^h\\  \text{where $\widehat{P}^{\Cont}_j=P$}\\\text{ for some $j\in [h]$}}} \sum_{R_0, S_0} \#Z(H|_{[0,1]\times\oPM_{\Gamma_{P, j, I_0,  k_1(0), k_2(0),R_0, S_0}^{\Cont}} }).
\end{aligned}\end{equation}

Since $\srest$ is canonical, we have that
\begin{equation}\label{eq:ZeroesBase}\# Z\left(\srest^{\CB\Lambda}\right) = \prod_{i=0}^h\langle\prod_{j\in I_i}\tau_{d_j}^{(a_j,b_j)}\sigma_1^{k_1(i)}\sigma_2^{k_2(i)}\sigma_{12}\rangle^{\mathbf{s}^{\Gamma_{0,k_1(i),k_2(i),1,I_i}},o}.
\end{equation}
Write
\begin{equation}\label{eq:ZeroesBase2}
\# Z\left(\srest^{\CB\Lambda}\right) = \sum_{p \in Z\left(\srest^{\CB\Lambda}\right)} \eps_p,
\end{equation}
where $\eps_p$ is the weight of the zero $p$ defined in \eqref{def: weight}. Recall that since $H$ is of the form~\eqref{trr homotopy for the bullets}, the homotopy decomposes as $H = \pi^*\srest \oplus H_2$, where
$$H_2(u) := \left((1-u)t_\bullet+u(1-u)\rho^\Gamma+u \mathbf{s}_2\right)\in C^\infty(\CL_1).$$
We can use this decomposition, along with the definition~\eqref{eq:F_p} of $\Phi_p$ and the orientation of the boundary given in Theorem \ref{thm:or_and_induced boundary}(1),
to see that, for $\Lambda \in \mathcal{G}_{\Cont}(P)$:

\begin{align}\label{eq:HG}
&\sum_{R_0, S_0} \# Z\left(H|_{[0,1]\times\oPM_{\Gamma_{P, j, I_0,  k_1(0), k_2(0),R_0, S_0}^{\Cont}}}\right) = \sum_{R_0, S_0}\# \left(Z\left(\pi^* F_{\Gamma_{P, j, I_0, k_1(0), k_2(0), R_0, S_0}^{\Cont}}^* \srest^{\CB\Lambda}\right) \cap Z(H_2)\right) \\
\notag&\quad\quad = \# \left(\left(\bigcup_{R_0, S_0}Z\left(\pi^* F_{\Gamma_{P, j, I_0,  k_1(0), k_2(0),R_0, S_0}^{\Cont}}^* \srest^{\CB\Lambda}\right)\right) \cap Z(H_2)\right)= \sum_{p \in Z\left(\srest^{\CB\Lambda}\right)}\eps_p\cdot \# Z\left(H_2|_{[0,1]\times\Phi_p }\right).
\end{align}

Note that the multisection $H_2|_{[0,1]\times\Phi_p}$ descends to a multisection $\hat H_{2,p}$ on the quotient space ${[0,1]\times\hat \Phi_p }$. Indeed, by Observation~\ref{obs:canonical_const}, the multisections $\rho$, $\mathbf{s}_2$, $\ttt_r^{v_0},\ttt_s^{v_0},\ttt_\Root^{v_0}$, and $\ttt_{z_2}^{v_0}$ are constant for a fixed time and thus descend. Further,  Observation~\ref{obs:cont_on_fiber} implies that $t_{h_0}^{v_0}$ descends. We remark that as $H_2$ is transverse on all boundary strata by Lemma~\ref{lem:hom_trr}, we have that it does not vanish at the endpoints of the intervals in $\Phi_p$. By Observation~\ref{obs:restriction_of_t_to_Cont}, $t_\bullet$ is a multisum of the above multisections hence $H_2$  descends as the components $\ttt_{\bullet}$,
$\rho$ and $\mathbf{s}_2$ descend.  Thus,  $\hat H_{2,p}$ is a multisection of the trivial rank 1 $\C$-bundle on
$[0,1]\times\hat\Phi_p$ and $\hat H_{2,p}$ can be viewed as a
$\C$-multivalued function on $[0,1]\times\hat\Phi_p$.
As a consequence,
\begin{equation}\label{eq:HhH}
\#Z\left(H_2|_{[0,1]\times\Phi_p }\right) = \#Z\left(\hat H_{2,p}\right).
\end{equation}
By Observation~\ref{obs:restriction_of_t_to_Cont}, we have the following
case by case description of $\hat H_{2,p}|_{\{0\}\times \hat \Phi_p}$:
\begin{enumerate}
\item if $\bullet = r$, then $\hat H_{2,p}|_{\{0\}\times\hat \Phi_p} = k_1(0) \ttt_r \uplus k_1(j) \theta_p$;
\item if $\bullet = s$, then $\hat H_{2,p}|_{\{0\}\times\hat \Phi_p} = k_2(0) \ttt_s \uplus k_2(j) \theta_p$;
\item if $\bullet = \Root$, then $\hat H_{2,p}|_{\{0\}\times\hat \Phi_p} = \ttt_\Root$;
\item if $\bullet = z_2$ and $2 \in I_0$, then  $\hat H_{2,p}|_{\{0\}\times\hat \Phi_p} = \ttt_{z_2}$; and
\item if $\bullet = z_2$ and $2 \in I_j\setminus I_0$, then $\hat H_{2,p}|_{\{0\}\times\hat \Phi_p} = \theta_p$.
\end{enumerate}

Therefore, by Observations~\ref{obs:canonical_const}
and~\ref{obs:cont_on_fiber}, we have
\begin{equation}\label{eq:Hwinding}
\#Z\left(\hat H_{2,p}\right) = \beta_{P^{\Cont}}(\ttts_\bullet).
\end{equation}
where $\beta_{P^{\Cont}}(\ttts_\bullet)$ is as in~\eqref{eqn:ContMultisection}. Putting together equations \eqref{eq:ZeroesBase},~\eqref{eq:ZeroesBase2},~\eqref{eq:HG},~\eqref{eq:HhH} and~\eqref{eq:Hwinding}, we obtain the statement of the lemma.
\end{proof}

As in the case for the wall-crossing contributions given by the $Z_{WC}^{\Gamma_P}(\ttts_\bullet)$, computing a closed form expression for $Z_{\Cont}^{\Gamma_P}(\ttts_\bullet)$ is hard, but the special linear combinations in Definition~\ref{4tuplesQ} will be meaningful.

\begin{lemma}\label{PCont Properties}
Let $\{(a^P, b^P, c^P, d^P)\}_{P\in \mathcal{P}(I,\vecd+\eee_1)}$ be a family of 4-tuples of rational numbers so that the 4-tuples do not depend on the order of the elements of $P$ and $a^P + b^P + c^P + d^P = 1$.
\begin{enumerate}
\item Suppose that $\{(a^P, b^P, c^P, d^P)\}_{P\in \mathcal{P}(I,\vecd+\eee_1)}$ satisfies Property {\bf Q1}. Then for any $h\ge 0$ and  $P^{\Cont} =  \{(I_0,k_1(0),k_2(0)),\ldots, (I_h,k_1(h),k_2(h))\}  \in\mathcal{P}_{Cont,h}(I,\vecd+\eee_1)$ with $1 \in I_0$, we have that
\begin{align}\label{property P1}
\begin{split}
\frac{1+\sum_{i=0}^hk_1(i)}{r+k_1(0)}a^{\widehat{P}^{Cont,+r}}+&\frac{1+\sum_{i=0}^hk_2(i)}{s+k_2(0)}b^{\widehat{P}^{Cont,+s}}\\
&=1+\sum_{j=1}^h\left(\frac{k_1(j)}{k_1(0)+k_1(j)}a^{\widehat{P}^{\Cont}_j}+  \frac{k_2(j)}{k_2(0)+k_2(j)}b^{\widehat{P}^{\Cont}_j}\right).
\end{split}
\end{align}
\item Suppose that $\{(a^P, b^P, c^P, d^P)\}_{P\in \mathcal{P}(I,\vecd+\eee_1)}$ satisfies Property {\bf Q2}. Then for any $h\ge 0$ and  $P^{\Cont}=  \{(I_0,k_1(0)+1,k_2(0)+1),\ldots, (I_h,k_1(h),k_2(h))\}  \in \mathcal{P}_{Cont,h}(I,\vecd+\eee_1)$ with $1 \in I_0$ and $2 \notin I_0$, we have that
\begin{align}\label{property P2}
\begin{split}
\frac{1+\sum_{i=0}^hk_1(i)}{r+k_1(0)}a^{\widehat{P}^{Cont,+r}}&+\frac{1+\sum_{i=0}^hk_2(i)}{s+k_2(0)}b^{\widehat{P}^{Cont,+s}}\\
&=1+ \sum_{\substack{j=1 \\ 2 \notin I_j}}^h\left(\frac{k_1(j)}{k_1(0)+k_1(j)}a^{\widehat{P}^{\Cont}_j}+  \frac{k_2(j)}{k_2(0)+k_2(j)}b^{\widehat{P}^{\Cont}_j}\right).
\end{split}
\end{align}
\end{enumerate}
\end{lemma}
\begin{proof}
Notice that
$$
P^{\Cont}:=  \{(I_0,k_1(0),k_2(0)),\ldots, (I_h,k_1(h),k_2(h))\}  \in \mathcal{P}_{Cont,h}(I,\vecd+\eee_1)
$$
if and only if
$$
Q :=  \{(I_0,k_1(0)+1,k_2(0)+1),\ldots, (I_h,k_1(h),k_2(h))\}  \in \mathcal{Q}_h(I,\vecd+\eee_1).
$$
Moreover, we note that $\widehat Q_j = \widehat{P}^{\Cont}_j$ and, if $Q \in \mathcal{Q}^{\ne0}_h(I,\vecd+\eee_1)$, then $\widehat{Q}^{+r} = \widehat{P}^{Cont, +r}$ and $\widehat{Q}^{+s} = \widehat{P}^{Cont, +s}$. Note that
$d^{\widehat{Q}_j}=0$ if and only if $d^{\widehat{P}^{\Cont}_j}=0$.
 It is then straightforward to see that \eqref{property Q1} is equivalent to \eqref{property P1} and \eqref{property Q2} is equivalent to \eqref{property P2}.
\end{proof}

\begin{lemma}\label{lem:contribution_eqs}
Let $\{(a^P, b^P, c^P, d^P)\}_{P\in \mathcal{P}(I,\vecd+\eee_1)}$ be a family of 4-tuples of rational numbers so that the 4-tuples do not depend on the order of the elements of $P$ and $a^P + b^P + c^P + d^P = 1$.
\begin{enumerate}
\item
Suppose that $\{(a^P, b^P, c^P, d^P)\}_{P\in \mathcal{P}(I,\vecd+\eee_1)}$ satisfies
Property {\bf Q1}. Then
\begin{align}\label{eq:contribution_eq1_trr}
\sum_{h=1}^{|I|}\frac{1}{h!}&\sum_{\substack{P\in\mathcal{P}_h(I,\vecd+\eee_1),\\
P=\{(I_1,k_1(1),k_2(1)),\ldots,(I_h,k_1(h),k_2(h))\}}}\frac{\Gamma(\frac{1+\sum_{i=1}^hk_1(i)}{r}) \Gamma(\frac{1+\sum_{i=1}^h k_2(i)}{s})}{\Gamma(\frac{1+r(I)}{r}) \Gamma(\frac{1+s(I)}{s})}\cdot\\
\notag&\quad\quad\cdot
\left(a^PZ_{\Cont}^{\Gamma_P}(\ttts_r)+b^PZ_{\Cont}^{\Gamma_P}(\ttts_s)+c^PZ_{\Cont}^{\Gamma_P}(\ttts_\Root)+d^PZ_{\Cont}^{\Gamma_P}(\ttts_{z_2})\right)=-\mathcal{A}(I,\vecd).
\end{align}
\item
Suppose that $\{(a^P, b^P, c^P, d^P)\}_{P\in \mathcal{P}(I,\vecd+\eee_1)}$ satisfies
Property {\bf Q2}. Then
\begin{align}\label{eq:contribution_eq2_trr}
\sum_{h=1}^{|I|}\frac{1}{h!}&\sum_{\substack{P\in\mathcal{P}_h(I,\vecd+\eee_1),\\
P=\{(I_1,k_1(1),k_2(1)),\ldots,(I_h,k_1(h),k_2(h))\}}}\frac{\Gamma(\frac{1+\sum_{i=1}^hk_1(i)}{r}) \Gamma(\frac{1+\sum_{i=1}^h k_2(i)}{s})}{\Gamma(\frac{1+r(I)}{r}) \Gamma(\frac{1+s(I)}{s})}\cdot\\
\notag&\quad\quad\cdot
\left(a^PZ_{\Cont}^{\Gamma_P}(\ttts_r)+b^PZ_{\Cont}^{\Gamma_P}(\ttts_s)+c^PZ_{\Cont}^{\Gamma_P}(\ttts_\Root)+d^PZ_{\Cont}^{\Gamma_P}(\ttts_{z_2})\right)=0.
\end{align}
\end{enumerate}
\end{lemma}
\begin{proof}
For $h\ge 0$, let
$\overline{\mathcal{P}}_{Cont,h}(I,\vecd+\eee_1)$ be the collection of elements \[P^{\Cont}=\{(I_0,k_1(0),k_2(0)),\ldots,(I_h,k_1(h),k_2(h))\}\] so that
\begin{itemize}
\item $I_j \ne \varnothing $ for all $0 \le j  \le h$;
\item the elements $(I_0,k_1(0),k_2(0)),\ldots,(I_h,k_1(h),k_2(h))$ are ordered in a way so that, for any two indices $i$ and $j$, we have that $\text{min}(I_i) < \text{min}(I_j)$ if and only if $i<j$;
\item  the graph $\Gamma_{P^{\Cont}}$ from~\eqref{def: GammaPCont} is balanced with respect to the descendent vector $\vecd$.
\end{itemize}
This implies that in particular  $1\in I_0$ and any $P^{\Cont} \in \overline{\mathcal{P}}_{Cont,h}(I,\vecd+\eee_1)$ with $h>0$ belongs to $\mathcal{P}_{\Cont}(I,\vecd+\eee_1)$.
 In the case $h=0$, we just have the data of a graph balanced with respect to
$\vecd$ with internal markings $I$. Note that if $h\geq 1$, then $P^{\Cont}\in \overline{\mathcal{P}}_{Cont,h}(I,\vecd+\eee_1)$ is an element of $\mathcal{P}^{\neq\emptyset}_{Cont,h}
(I,\vecd+\eee_1)$.
Moreover, in this case, if we consider the action of the symmetric group
$S_h$ on $\mathcal{P}_{Cont,h}^{\neq\emptyset}(I,\vecd+\eee_1)$ given by permuting
the last $h$ elements in the ordered partitions in
$\mathcal{P}^{\neq\emptyset}_{Cont,h}(I,\vecd+\eee_1)$, we can see that
$\overline{\mathcal{P}}_{Cont,h}(I,\vecd+\eee_1)$ contains a unique element of each
orbit of this $S_h$-action.
Thus $\overline{\mathcal{P}}_{Cont,h}(I,\vecd+\eee_1)$ has precisely one element
for each graph $\Gamma_{P^{\Cont}}$, with the order of connected
components not taken into account.

We define $ \overline{\mathcal{P}}_{Cont,h+1}^r(I,\vecd+\eee_1)$ and $ \overline{\mathcal{P}}_{Cont,h+1}^s(I,\vecd+\eee_1)$ to be the collections of tuples that can be written in the form $P^{Cont,+r}$ and $P^{Cont,+s}$, respectively, as
in~\eqref{PContr and PConts} for some $P^{\Cont} \in  \overline{\mathcal{P}}_{Cont,h}(I,\vecd+\eee_1)$. Take $$\widetilde{\mathcal{P}}_{Cont,h}(I,\vecd+\eee_1) : = \overline{\mathcal{P}}_{Cont,h}(I,\vecd+\eee_1) \cup  \overline{\mathcal{P}}_{Cont,h}^r(I,\vecd+\eee_1) \cup \overline{\mathcal{P}}_{Cont,h}^s(I,\vecd+\eee_1).$$
Define $\widetilde{\mathcal{P}}_{\Cont}(I,\vecd+\eee_1) : = \bigcup_{h} \widetilde{\mathcal{P}}_{Cont,h}(I,\vecd+\eee_1)$.

Set
\[
O_i(P^{\Cont})=\left\langle \prod_{j\in I_i} \tau_{d_j}^{(a_j,b_j)}
\sigma_1^{k_1(i)}\sigma_2^{k_2(i)}\sigma_{12}
\right\rangle^{\mathbf{s}^{\Gamma_{0,k_1(i),k_2(i),1,I_i},o}}.
\]
Using Lemma~\ref{lem:TRR_cont_for_sections} and the fact that $P^{\Cont}$ and $\widehat{P}^{\Cont}_j$ have the same number of $r$-points and $s$-points, the left-hand side of both ~\eqref{eq:contribution_eq1_trr} and ~\eqref{eq:contribution_eq2_trr} can be expanded to:
\begin{equation}\begin{aligned}\label{expanded Cont TRR}
\sum_{h=1}^{|I|}&\sum_{P\in\mathcal{P}_h(I,\vecd+\eee_1)} \sum_{\substack{P^{\Cont}\in\mathcal{P}_{Cont,h}(I,\vecd+\eee_1),\\ P^{\Cont} =\{(I_i,k_1(i),k_2(i))\}_{i=0}^h\\  \text{where $\widehat{P}^{\Cont}_j=P$}\\\text{ for some $j\in [h]$ and $1 \in I_0$}}}
 \frac{1}{h!} \frac{\Gamma(\frac{1+\sum_{i=0}^hk_1(i)}{r})\Gamma(\frac{1+\sum_{i=0}^h k_2(i)}{s})}{\Gamma(\frac{1+r(I)}{r}) \Gamma(\frac{1+s(I)}{s})}\cdot
\left(\prod_{i=0}^hO_i(P^{\Cont})\right) \cdot \\
&\left(\frac{k_1(j)}{k_1(0)+k_1(j)}a^{\widehat{P}^{\Cont}_j} + \frac{k_2(j)}{k_2(0)+k_2(j)} b^{\widehat{P}^{\Cont}_j} + \begin{cases} d^{\widehat{P}^{\Cont}_j}=1 & \text{ if $2 \in I_j$ and {\bf Q2} holds} \\ 0 & \text{otherwise}  \end{cases}\right).
\end{aligned}\end{equation}

We can change the summations over $\mathcal{P}_h(I,\vecd+\eee_1)$ and special subsets of $\mathcal{P}_{Cont,h}(I, \vecd+\eee_1)$ to a summation over $\widetilde{\mathcal{P}}_{\Cont}(I, \vecd+\eee_1)$. Given a $P^{\Cont}= \{(I_i,k_1(i),k_2(i))\}_{i=0}^h \in \widetilde{\mathcal{P}}_{\Cont}(I, \vecd+\eee_1)$, we define the quantity:
$$
Y_{P^{\Cont}} := \frac{\Gamma(\frac{1+\sum_{i=0}^hk_1(i)}{r})\Gamma(\frac{1+\sum_{i=0}^h k_2(i)}{s})}{\Gamma(\frac{1+r(I)}{r}) \Gamma(\frac{1+s(I)}{s})} \left(\prod_{i=0}^hO_i(P^{\Cont})\right).
$$
Using the fact that there are $h!$ elements of $\mathcal{P}_{Cont,h} (I, \vecd+\eee_1)$ in an $S_h$-orbit of a given element of $\mathcal{P}_{Cont,h}^{\not=\emptyset}(I,\vecd+\eee_1)$, we can rewrite the summation given
in~\eqref{expanded Cont TRR} as:

\begin{equation}\label{Cont TRR as a case by case}
\sum_{P\in\widetilde{\mathcal{P}}_{\Cont}(I,\vecd+\eee_1)} C(P),
\end{equation}
where $C(P)$ can be computed as follows from~\eqref{expanded Cont TRR}.
\begin{itemize}
\item If $P = P^{\Cont} = \{(I_i,k_1(i),k_2(i))\}_{i=0}^h \in  \overline{\mathcal{P}}_{Cont,h}(I,\vecd+\eee_1)$, then
$$
C(P) = Y_{P^{\Cont}}\sum_{j=1}^h \left(\frac{k_1(j)}{k_1(0)+k_1(j)}a^{\widehat{P}^{\Cont}_j} + \frac{k_2(j)}{k_2(0)+k_2(j)} b^{\widehat{P}^{\Cont}_j} + \begin{cases} 1 & \text{ if $2 \in I_j$ and {\bf Q2} holds} \\ 0 & \text{otherwise}\end{cases} \right).
$$

\item If $P = P^{Cont, +r} = \{(I_0,k_1(0),k_2(0)),\ldots,(I_h,k_1(h),k_2(h)), (\emptyset, r, 0)\} \in \overline{\mathcal{P}}_{Cont,h+1}^r(I,\vecd+\eee_1)$, then using that $\langle \sigma_1^r \sigma_{12}\rangle=-1$ and
$\Gamma(x+1)=x\Gamma(x)$, we have that
\begin{align*}
C(P) &= \frac{\Gamma(\frac{1+r+ \sum_{i=0}^hk_1(i)}{r})\Gamma(\frac{1+\sum_{i=0}^h k_2(i)}{s})}{\Gamma(\frac{1+r(I)}{r}) \Gamma(\frac{1+s(I)}{s})} \left(\prod_{i=0}^hO_i(P^{\Cont})\right) \cdot \langle \sigma_1^r \sigma_{12}\rangle \cdot \left(\frac{r}{r+k_1(0)} a^{\widehat{P}^{Cont, +r}}\right)\\
	&= \frac{1+\sum_{i=0}^h k_1(i)}{r} \cdot (-1) \cdot \frac{r}{r+k_1(0)} a^{\widehat{P}^{Cont, +r}}  Y_{P^{\Cont}} \\ &= - \frac{1+\sum_{i=0}^h k_1(i)}{r+k_1(0)}a^{\widehat{P}^{Cont, +r}}  Y_{P^{\Cont}}.
\end{align*}
\item If $P= P^{Cont, +s} \in\overline{\mathcal{P}}_{Cont,h+1}^s(I,\vecd)$, then we can analogously compute that
\begin{equation*}
C(P) = - \frac{1+\sum_{i=0}^h k_2(i)}{s+k_2(0)}b^{\widehat{P}^{Cont, +s}}  Y_{P^{\Cont}}.
\end{equation*}
\end{itemize}

View the quantities $Y_{P^{\Cont}}$ for all $ P^{\Cont} \in  \overline{\mathcal{P}}_{Cont,h}(I,\vecd+\eee_1)$ as independent linear variables. Then we can write  the left-hand side of both ~\eqref{eq:contribution_eq1_trr} and ~\eqref{eq:contribution_eq2_trr} as a linear combination of the variables $Y_{P^{\Cont}}$.

Consider a family of $4$-tuples satisfying Property {\bf Q1},
so that in particular
we fix $d^P = 0$ for all $P \in \mathcal{P}(I,\vecd+\eee_1)$. Then the coefficient of $Y_{P^{\Cont}}$ in
\eqref{expanded Cont TRR}
when Property {\bf Q1} holds is, by Lemma~\ref{PCont Properties}(1),
\begin{align*}
\sum_{j=1}^h \Big(\frac{k_1(j)}{k_1(0)+k_1(j)}a^{\widehat{P}^{\Cont}_j} +
& \frac{k_2(j)}{k_2(0)+k_2(j)} b^{\widehat{P}^{\Cont}_j}\Big)
\\
& -  \frac{1+\sum_{i=0}^h k_1(i)}{r+k_1(0)}a^{\widehat{P}^{Cont, +r}} - \frac{1+\sum_{i=0}^h k_2(i)}{s+k_2(0)}b^{\widehat{P}^{Cont, +s}} = - 1.
\end{align*}
Thus we can simplify the left-hand of \eqref{eq:contribution_eq1_trr} to be
\begin{align*}
&\sum_{P^{\Cont} \in \overline{\mathcal P}_{\Cont}(I, \vecd+\eee_1)} -Y_{P^{\Cont}} \\
&\qquad=  \sum_{h\ge0} \sum_{P \in \mathcal{P}_{h+1}(I, \vecd)} \frac{-1}{(h+1)!}  \frac{\Gamma(\frac{1+\sum_{i=0}^hk_1(i)}{r})\Gamma(\frac{1+\sum_{i=0}^h k_2(i)}{s})}{\Gamma(\frac{1+r(I)}{r}) \Gamma(\frac{1+s(I)}{s})} \prod_{i=0}^h \langle\prod_{j\in I_i}\tau_{d_j}^{(a_j,b_j)}\sigma_1^{k_1(i)}\sigma_2^{k_2(i)}\sigma_{12}\rangle^{\mathbf{s}^{\Gamma_{0,k_1(i),k_2(i),1,I_i}},o} \\
	&\qquad = -\mathcal{A}(I, \vecd).
\end{align*}
Note here we are replacing the sum over elements of
$\overline{\mathcal{P}}_{Cont,h}(I,\vecd+\eee_1)$ with the set
$\mathcal{P}_{h+1}(I,\vecd)$ and dividing by $(h+1)!$, as the set
of $S_{h+1}$-orbits of $\mathcal{P}_{h+1}(I,\vecd)$ is in one-to-one
correspondence with $\overline{\mathcal{P}}_{Cont,h}(I,\vecd+\eee_1)$. Moreover, to fit with the convention of the indexing of $\overline{\mathcal{P}}_{Cont,h}(I,\vecd+\eee_1)$ going from $0$ to $h$, we temporarily shift the indexing of tuples in $\mathcal{P}_{h+1}(I,\vecd)$ down by one.

Next consider a family of 4-tuples satisfying Property {\bf Q2}. So, if $P$ satisfies $\{1, 2\} \subseteq I_j$ for some $j$ then $d^P=1$ and $a^P=b^P=c^P=0$ and otherwise $d^P=0$. Then the coefficient of $Y_{P^{\Cont}}$  in the left-hand side of \eqref{eq:contribution_eq2_trr} is
\begin{align*}
\sum_{\substack{j=1 \\ 2 \notin I_j}}^h&\left(\frac{k_1(j)}{k_1(0)+k_1(j)}a^{\widehat{P}^{\Cont}_j} + \frac{k_2(j)}{k_2(0)+k_2(j)}b^{\widehat{P}^{\Cont}_j}\right) +\sum_{j=1}^h \mathbf{1}_{2\in I_j} \\ &\qquad -\frac{1+\sum_{i=0}^hk_1(i)}{r+k_1(0)}a^{\widehat{P}^{Cont, +r}}-\frac{1+\sum_{i=0}^hk_2(i)}{s+k_2(0)}b^{\widehat{P}^{Cont, +s}}.
\end{align*}
Note that if $2\in I_0$ then $\{1,2\}\subseteq I_0$, and hence
$d^{\widehat{P}^{\Cont}_j}=d^{\widehat{P}^{Cont,+r}}=d^{\widehat{P}^{Cont,+s}}=1$
and the corresponding $a,b,c$ terms are zero. Thus the above
expression is automatically zero.
Otherwise, we assume $2\not\in I_0$.
Since Property {\bf Q2} holds, we have by Lemma~~\ref{PCont Properties}(2) that
\begin{align*}
\sum_{\substack{j=1 \\ 2 \notin I_j}}^h&\left(\frac{k_1(j)}{k_1(0)+k_1(j)}a^{\widehat{P}^{\Cont}_j} + \frac{k_2(j)}{k_2(0)+k_2(j)}b^{\widehat{P}^{\Cont}_j}\right) +\sum_{j=1}^h \mathbf{1}_{2\in I_j}\\ &\qquad -\frac{1+\sum_{i=0}^hk_1(i)}{r+k_1(0)}a^{\widehat{P}^{Cont, +r}}-\frac{1+\sum_{i=0}^hk_2(i)}{s+k_2(0)}b^{\widehat{P}^{Cont, +r}} = -1 + \sum_{j\neq 0} \mathbf{1}_{2\in I_j} = 0.
\end{align*}
Hence the coefficient of $Y_{P^{\Cont}}$ vanishes for all $P^{\Cont}$ and the expression in \eqref{eq:contribution_eq2_trr} vanishes.
\end{proof}

With all these computations in hand, we can take a deep breath, put this together, and prove Theorem~\ref{thm:mirrorA}.

 \begin{proof}[Proof of Theorem~\ref{thm:mirrorA}]
Suppose $\{(a^P,b^P, c^P, d^P)\}_{P\in\mathcal{P}(I,\vecd+\eee_1)}$ is a family of $4$-tuples as constructed in Lemma \ref{existence of 4tuples}. Take  a special canonical family $\ess$ of multisections bounded by $I$.  By Theorem~\ref{thm:A_mod_invs}, this choice will not affect the quantities of $\mathcal{A}(I, \vecd + \eee_1)$ or $\mathcal{A}(I, \vecd)$.

We define the quantity
$$
\#Z^{\Gamma_P}(\ttts_{a,b,c,d}):= a^P \int_{\oPM_{\Gamma_P}}e(E,\ttts_r) + b^P \int_{\oPM_{\Gamma_P}}e(E,\ttts_s) + c^P \int_{\oPM_{\Gamma_P}}e(E,\ttts_\Root)+ d^P \int_{\oPM_{\Gamma_P}}e(E,\ttts_{z_2})
$$
We then take the four homotopies $H^\Gamma(t_\bullet)$ from $\ttts_\bullet$ to $\ess$. We define analogously $\#Z^{\Gamma_P}(H^\Gamma_{a,b,c,d})$, $Z_{WC}^{\Gamma_P}(\ttts_{a,b,c,d})$, $Z_{\Cont}^{\Gamma_P}(\ttts_{a,b,c,d})$, and $Z_{\XCH}^{\Gamma_P}(\ttts_{a,b,c,d})$ as the appropriate weighted sums.

As usual, we write
\begin{equation}
\Gamma_{P,j} := \Gamma_{0, k_1(j), k_2(j), 1, \{(a_i, b_i)\}_{i \in I_j}}; \qquad O_j(P):= \left\langle\prod_{i\in {I_j}}\tau^{(a_i,b_i)}_{d_i}\sigma_1^{k_1(j)}\sigma_2^{k_2(j)}\sigma_{12}\right\rangle^{\mathbf{s}^{\Gamma_{P,j}},o}
\end{equation}
for any $P = \{(I_j,k_1(j),k_2(j))\}\in \mathcal{P}(I, \vecd+\eee_1)$.
We compute the quantity  $\mathcal{A}(I, \vecd + \eee_1)$ given by
\begin{equation}\label{A(I, d+e)}
\sum_{h=1}^{|I|}\frac{1}{h!}\sum_{\substack{P \in \mathcal{P}(I, \vecd+ \eee_1) \\ P = \{(I_i, k_1(i), k_2(i))\}_{i=1}^h}}
\frac{\Gamma(\frac{1+\sum_{i=1}^hk_1(i)}{r}) \Gamma(\frac{1+\sum_{i=1}^h k_2(i)}{s})}{\Gamma(\frac{1+r(I)}{r}) \Gamma(\frac{1+s(I)}{s})}
\left\langle\prod_{i\in {I(\Gamma_P)}}\tau^{(a_i,b_i)}_{d_i+\delta_{1i}}\sigma_1^{k_1(\Gamma_P)}\sigma_2^{k_2(\Gamma_P)}\sigma_{12}^{k_{12}(\Gamma_P)}\right\rangle^{\mathbf{s}^{\Gamma_P},o}
\end{equation}
by applying Lemma~\ref{lem: homotopy TRR contributions classify} in the context of the quantity $\#Z^\Gamma(H_{a,b,c,d})$. 
In this case, ~\eqref{eq:int_num_in_terms_of_trr} becomes
\begin{equation}\begin{aligned}\label{homotopy reminder}
\left\langle\prod_{i\in {I(\Gamma_P)}}\tau^{(a_i,b_i)}_{d_i+\delta_{1i}}\sigma_1^{k_1(\Gamma_P)}\sigma_2^{k_2(\Gamma_P)}\sigma_{12}^{k_{12}(\Gamma_P)}\right\rangle^{\mathbf{s}^{\Gamma_P},o} & \\
=\#Z^{\Gamma_P}(\ttts_{a,b,c,d})+Z_{WC}^{\Gamma_P}(\ttts_{a,b,c,d})&+Z_{\Cont}^{\Gamma_P}(\ttts_{a,b,c,d})+Z_{\XCH}^{\Gamma_P}(\ttts_{a,b,c,d}).
\end{aligned}\end{equation}

We now consider each term of the right hand side of ~\eqref{homotopy reminder} when considered as a linear combination given in~\eqref{A(I, d+e)}.
Take a family of 4-tuples $\{(a^P, b^P, c^P, d^P)\}$ that satisfies Property {\bf Q1}.
In this case, since we know that $d^P=0$, the relevant multisections are $\ttts_r$, $\ttts_s$ and $\ttts_\Root$ and these all point towards a boundary marking.
Let $P=\{(I_j,k_1(j),k_2(j))\}\in \mathcal{P}(I, \vecd+\eee_1)$ with $1\in I_1$.

As in Lemma \ref{lem:closed_contribution}, given the graph $\Gamma_{P,1}$ and $A \subseteq I_1$, write $z_{A,a,b}$ for an internal marking labelled by the union of labels of the marked points in $A$ with twist $(r-2-a,s-2-b)$ and write $v(P,1)_{A, a, b}$ for the graph $ \Gamma_{0,k_1,k_2,k_{12},\{(a_i, b_i)\}_{i \in (I_1\setminus A)\cup\{z_{A,a,b}\}}}$. Since $\ess$ is special canonical so is $\ess_1$, hence by Lemma \ref{lem:closed_contribution} and
Notation \ref{not:Gamma ab} we have
\begin{align*}
\#&Z^{\Gamma_{P,1}}(\ttts_{a,b,c,d}) \\ &= \sum_{\substack{A\sqcup B = I_1 \\ 1 \in A, \ |A| \ge 2 \\ a\in\{-1,\ldots,r-2\},\\b\in\{-1,\ldots,s-2\}}}\left\langle \tau_0^{(a,b)}\prod_{i \in A}\tau_{d_i}^{(a_i,b_i)}\right\rangle^{\text{ext}}
\left\langle \tau_0^{(r-2-a,s-2-b)}\prod_{i\in B}\tau^{(a_i,b_i)}_{d_i}\sigma_1^{k_1(1)}\sigma_2^{k_2(1)}\sigma_{12}\right\rangle^{s^{v(P,1)_{A, a, b}},o}
\end{align*}
Thus
\begin{equation}\begin{aligned}\label{first t abcd eqn}
\sum_{h=1}^{|I|}&\frac{1}{h!} \sum_{\substack{P \in \mathcal{P}_h(I, \vecd+ \eee_1) \\ P = \{(I_i, k_1(i), k_2(i))\}_{i=1}^h}}
\frac{\Gamma(\frac{1+\sum_{i=1}^hk_1(i)}{r}) \Gamma(\frac{1+\sum_{i=1}^h k_2(i)}{s})}{\Gamma(\frac{1+r(I)}{r}) \Gamma(\frac{1+s(I)}{s})}\cdot \#Z^{\Gamma_P}(\ttts_{a,b,c,d})  \\
=&\sum_{h=1}^{|I|}\frac{1}{(h-1)!}\sum_{\substack{P \in \mathcal{P}_h(I, \vecd+ \eee_1) \\ P = \{(I_i, k_1(i), k_2(i))\}_{i=1}^h\\ 1 \in I_1}}
\frac{\Gamma(\frac{1+\sum_{i=1}^hk_1(i)}{r}) \Gamma(\frac{1+\sum_{i=1}^h k_2(i)}{s})}{\Gamma(\frac{1+r(I)}{r}) \Gamma(\frac{1+s(I)}{s})}\cdot \prod_{j=2}^h O_j(P) \\
&  \left(\sum_{\substack{A\sqcup B = I_1 \\ 1 \in A, \ |A| \ge 2 \\ a\in\{-1,\ldots,r-2\},\\b\in\{-1,\ldots,s-2\}}}\left\langle \tau_0^{(a,b)}\prod_{i \in A}\tau_{d_i}^{(a_i,b_i)}\right\rangle^{\text{ext}}
\left\langle \tau_0^{(r-2-a,s-2-b)}\prod_{i\in B}\tau^{(a_i,b_i)}_{d_i}\sigma_1^{k_1(1)}\sigma_2^{k_2(1)}\sigma_{12}\right\rangle^{s^{v(P,1)_{A, a, b}},o}\right) \\
&= \sum_{\substack{a\in\{-1,\ldots,r-2\}\\b\in\{-1,\ldots,s-2\}}}\sum_{\substack{ A \sqcup B = I \\ 1 \in A, \ |A| \ge 2}}\left\langle \tau_0^{(a,b)}\prod_{i \in A}\tau_{d_i}^{(a_i,b_i)}\right\rangle^{\text{ext}} \cdot \\
&\qquad \left( \sum_{h=1}^{|B|+1}\frac{1}{(h-1)!}\sum_{\substack{P \in \mathcal{P}_h(B \cup \{z_{a,b}\}, \vecd) \\ P = \{(I_i, k_1(i), k_2(i))\}_{i=1}^h \\ z_{a,b} \in I_1}}
\frac{\Gamma(\frac{1+\sum_{i=1}^hk_1(i)}{r}) \Gamma(\frac{1+\sum_{i=1}^h k_2(i)}{s})}{\Gamma(\frac{1+r(I)}{r}) \Gamma(\frac{1+s(I)}{s})} \prod_{j=2}^h O_j(P)\right. \\
&\qquad\qquad\qquad \left.\left\langle \tau_0^{(r-2-a,s-2-b)}\prod_{i\in I_1\setminus\{z_{a,b}\}}\tau^{(a_i,b_i)}_{d_i}\sigma_1^{k_1(1)}\sigma_2^{k_2(1)}\sigma_{12}\right\rangle^{s^{\Gamma_{0,k_1(1), k_2(1), 1, \{(a_i, b_i)\}_{i \in I_1 \cup \{z_{a,b}\}}}},o}\right) \\
&= \sum_{\substack{a\in\{-1,\ldots,r-2\}\\b\in\{-1,\ldots,s-2\}}}\sum_{\substack{A \sqcup B = I \\ 1 \in A, \ |A| \ge 2}}\left\langle \tau_0^{(a,b)}\prod_{i \in A}\tau_{d_i}^{(a_i,b_i)}\right\rangle^{\text{ext}} \mathcal{A}(B \cup \{z_{a,b}\}, \vecd).
\end{aligned}\end{equation}
By Lemma~\ref{lem:cancelation_eqs}, we know that if $\{(a^P,b^P, c^P, d^P)\}_{P\in\mathcal{P}(I,\vecd+\eee_1)}$ satisfies Property {\bf Q1}, then
\begin{equation}\label{WC recursion equation}
\sum_{h=1}^{|I|}\frac{1}{h!}\sum_{\substack{P \in \mathcal{P}_h(I, \vecd+ \eee_1) \\ P = \{(I_i, k_1(i), k_2(i))\}_{i=1}^h}}
\frac{\Gamma(\frac{1+\sum_{i=1}^hk_1(i)}{r}) \Gamma(\frac{1+\sum_{i=1}^h k_2(i)}{s})}{\Gamma(\frac{1+r(I)}{r}) \Gamma(\frac{1+s(I)}{s})}Z_{WC}^{\Gamma_P}(\ttts_{a,b,c,d} ) = 0.
\end{equation}
By Lemma ~\ref{lem:xch_vanishing_invariants}, we know that for all $P \in \mathcal{P}_h(I, \vecd + \eee_1)$,
\begin{equation}\label{exchange vanish in OTRR}
Z_{\XCH}^{\Gamma_P}(\ttts_{a,b,c,d}) = 0.
\end{equation}
Lastly, since $\{(a^P,b^P, c^P, d^P)\}_{P\in\mathcal{P}(I,\vecd+\eee_1)}$
satisfies Property {\bf Q1}, we have that by Lemma~\ref{lem:contribution_eqs},
\begin{equation}\label{recurse with Q1}
\sum_{h=1}^{|I|}\frac{1}{h!}\sum_{\substack{P \in \mathcal{P}(I, \vecd+ \eee_1) \\ P = \{(I_i, k_1(i), k_2(i))\}_{i=0}^h}}
\frac{\Gamma(\frac{1+\sum_{i=1}^hk_1(i)}{r}) \Gamma(\frac{1+\sum_{i=1}^h k_2(i)}{s})}{\Gamma(\frac{1+r(I)}{r}) \Gamma(\frac{1+s(I)}{s})}Z_{\Cont}^{\Gamma_P}(\ttts_{a,b,c,d} ) = - \mathcal{A}(I, \vecd).
\end{equation}

By substituting into~\eqref{A(I, d+e)} the relation given in \eqref{homotopy reminder}  and then simplifying using \eqref{first t abcd eqn}, \eqref{WC recursion equation}, \eqref{exchange vanish in OTRR}, and ~\eqref{recurse with Q1},
we obtain the open topological recursion relation \eqref{eq:calculation_1}.

We next turn to the case where $|I|\ge 2$.
Take a family of 4-tuples $\{(a^P, b^P, c^P, d^P)\}$ that satisfies Property {\bf Q2}. In this case, the computation for $\#Z^{\Gamma_{P,1}}(\ttts_{a,b,c,d})$ for $\Gamma = \Gamma_P$ has two cases, depending on whether or not the markings $1$ and $2$ are in the same connected component.
Let $P=\{(I_j,k_1(j),k_2(j))\}\in \mathcal{P}(I, \vecd+\eee_1)$
with $1\in I_1$.
If $2 \notin I_1$ then $d^P = 0$ and the relevant multisections $\ttts_\bullet$ point to only boundary markings. We then have that
\begin{align*}
\#&Z^{\Gamma_{P,1}}(\ttts_{a,b,c,d}) \\ &= \sum_{\substack{A\sqcup B = I_1 \\ 1 \in A, |A| \ge 2 \\ a\in\{-1,\ldots,r-2\},\\b\in\{-1,\ldots,s-2\}}}\left\langle \tau_0^{(a,b)}\prod_{i \in A}\tau_{d_i}^{(a_i,b_i)}\right\rangle^{\text{ext}}
\left\langle \tau_0^{(r-2-a,s-2-b)}\prod_{i\in B}\tau^{(a_i,b_i)}_{d_i}\sigma_1^{k_1(1)}\sigma_2^{k_2(1)}\sigma_{12}\right\rangle^{s^{v(P,1)_{A, a, b}},o}
\end{align*}
Otherwise, if $2 \in I_1$ then $d^P=1$ and $a^P=b^P=c^P=0$ so the relevant multisection is $\ttts_{z_2}$ which points towards the internal marking $z_2$, hence we have that
\begin{align*}
\#&Z^{\Gamma_{P,1}}(\ttts_{a,b,c,d}) \\ &=
\sum_{\substack{ A\sqcup B = I_1 \\ 1\in A, 2 \in B  \\ |A| \ge 2 \\a\in\{-1,\ldots,r-2\},\\b\in\{-1,\ldots,s-2\}}}\left\langle \tau_0^{(a,b)}\prod_{i \in A}\tau_{d_i}^{(a_i,b_i)}\right\rangle^{\text{ext}}
\left\langle \tau_0^{(r-2-a,s-2-b)}\prod_{i\in B}\tau^{(a_i,b_i)}_{d_i}\sigma_1^{k_1(1)}\sigma_2^{k_2(1)}\sigma_{12}\right\rangle^{s^{v(P,1)_{A, a, b}},o}.
\end{align*}

Note in both cases $2$ is never an internal marking of a closed vertex
and we will get a contribution whenever $2$ is an internal marking of an
open vertex.
Through a similar resummation as before in~\eqref{first t abcd eqn} we get that

\begin{equation}\begin{aligned}\label{second t abcd eqn}
\sum_{h=1}^{|I|}\frac{1}{h!}&\sum_{\substack{P \in \mathcal{P}_h(I, \vecd+ \eee_1) \\ P = \{(I_i, k_1(i), k_2(i))\}_{i=1}^h\\ 1 \in I_1}}
\frac{\Gamma(\frac{1+\sum_{i=1}^hk_1(i)}{r}) \Gamma(\frac{1+\sum_{i=1}^h k_2(i)}{s})}{\Gamma(\frac{1+r(I)}{r}) \Gamma(\frac{1+s(I)}{s})}\cdot \#Z^{\Gamma_P}(\ttts_{a,b,c,d}) \\
&=\sum_{\substack{a\in\{-1,\ldots,r-2\}\\b\in\{-1,\ldots,s-2\}}} \ \sum_{\substack{A \coprod B =I\\ 1 \in A, \ 2\in B \\ |A| \ge 2}}\left\langle \tau_0^{(a,b)}\prod_{i \in A}\tau_{d_i}^{(a_i,b_i)}\right\rangle^{\text{ext}} \mathcal{A}(B\cup\{z_{a,b}\},\vecd).
\end{aligned}\end{equation}

Lastly, recall that when $\{(a^P,b^P, c^P, d^P)\}_{P\in\mathcal{P}(I,\vecd+\eee_1)}$ satisfies Property {\bf Q2}, we have by Lemma \ref{lem:cancelation_eqs}
that \eqref{WC recursion equation} still holds. Further,
Lemma~\ref{lem:contribution_eqs} implies  that
\begin{equation}\label{recurse with Q2}
\sum_{h=1}^{|I|}\frac{1}{h!}\sum_{\substack{P \in \mathcal{P}_h(I, \vecd+ \eee_1) \\ P = \{(I_i, k_1(i), k_2(i))\}_{i=1}^h}}
\frac{\Gamma(\frac{1+\sum_{i=1}^hk_1(i)}{r}) \Gamma(\frac{1+\sum_{i=1}^h k_2(i)}{s})}{\Gamma(\frac{1+r(I)}{r}) \Gamma(\frac{1+s(I)}{s})}Z_{\Cont}^{\Gamma_P}(\ttts_{a,b,c,d} ) = 0.
\end{equation}
Concluding this computation, by substituting into~\eqref{A(I, d+e)} the relation given in \eqref{homotopy reminder}  and then simplifying using \eqref{second t abcd eqn}, \eqref{WC recursion equation}, \eqref{exchange vanish in OTRR}, and ~\eqref{recurse with Q2}, we obtain the open topological recursion relation
\eqref{eq:calculation_2}.
\end{proof}

\bibliography{biblio_fermat}
\bibliographystyle{amsalpha}
\end{document}